\documentclass[11pt,a4paper,reqno]{amsart}
\usepackage[applemac]{inputenc}
\usepackage[T1]{fontenc}
\usepackage{amsmath}
\usepackage{amsthm}
\usepackage{amsfonts}
\usepackage{amssymb}
\usepackage{graphicx}
\usepackage{color}
\usepackage{xcolor}
\usepackage{amsbsy}
\usepackage{mathrsfs}
\usepackage{bbm}
\usepackage{hyperref}
\usepackage{mathtools}
\usepackage{graphicx}
\usepackage{caption}
\usepackage{subcaption}
\newtheorem{theorem}{Theorem}[section]
\newtheorem{lemma}[theorem]{Lemma}

\newtheorem{proposition}[theorem]{Proposition}

\newtheorem{definition}[theorem]{Definition}
\theoremstyle{remark}
\newtheorem{remark}[theorem]{\it \bf{Remark}\/}

\numberwithin{equation}{section}
\catcode`@=11
\def\section{\@startsection{section}{1}%
  \z@{1.5\linespacing\@plus\linespacing}{.5\linespacing}%
  {\normalfont\bfseries\large\centering}}
\catcode`@=12
\newcommand{\be}{\begin{equation}}
	\newcommand{\ee}{\end{equation}}
\newcommand{\bea}{\begin{eqnarray}}
	\newcommand{\eea}{\end{eqnarray}}
\newcommand{\bee}{\begin{eqnarray*}}
	\newcommand{\eee}{\end{eqnarray*}}

\def\pa{\partial}

\def\Ai{\mathbf{Ai}}
\def\Bi{\mathbf{Bi}}
\def\CC{\mathbb{C}}
\def\NN{\mathbb{N}}
\def\RR{\mathbb{R}}
\def\ZZ{\mathbb{Z}}

\def\SS{\mathcal{S}}

\def\ti{\tilde I}
\def\tk{\tilde K}

\def\calC{{\mathcal C}}
\def\calB{{\mathcal B}}
\def\calA{{\mathcal A}}

\def\calH{{\mathcal H}}
\def\calI{{\mathcal I}}

\def\calK{{\mathcal K}}
\def\calL{{\mathcal L}}

\def\calO{{\mathcal O}}

\def\calR{{\mathcal R}}

\def\calT{{\mathcal T}}

\def\calV{{\mathcal V}}
\def\calW{{\mathcal W}}

\def\calZ{{\mathcal Z}}

\def\calV{{\mathcal V}}

\catcode`@=11
\def\supess{\mathop{\operator@font Sup\,ess}}
\catcode`@=12

\def\CC{\mathbb{C}}

\def\HH{\mathbb{H}}

\def\NN{\mathbb{N}}

\def\PP{\mathbb{P}}

\def\RR{\mathbb{R}}
\def\SS{\mathbb{S}}

\def\XX{\mathbb{X}}

\def\ZZ{\mathbb{Z}}

\def\PP{\mathbb{P}}

\def\ZZ{\mathbb{Z}}

\def\a{\alpha}
\def\frakc{\mathfrak{c}}
\def\frakE{\mathfrak{E}}
\def\e{\varepsilon}

\def\bar#1{{\overline #1}}

\def\R2+{\RR ^2_+}

\def\tb{\tilde b}

\def\sgn{{\rm sgn}}

\def\pa{\partial}

\def\lim{\mathop{\rm lim}}

\def\sup{\mathop{\rm sup}}
\def\exp{{\rm exp}}
\def\l{\lambda}

\def\log{{\rm log}}

\def\rpa{\mathring{\pa}}

\def\pa{\partial}

\def\pa{\partial}
\def\la{\langle}

\def\ra{\rangle}

\def\Wfr{\mathfrak{W}}
\begin{document}

\title[]{Mode stability for self-similar blowup of slightly supercritical NLS: I. low-energy spectrum}

% \author[Z. Li]{Zexing Li}
% \address{University of Cambridge}
% \email{zl486@cam.ac.uk}

\author[Z. Li]{Zexing Li}
\address{Laboratoire Analyse, G\'eom\'etrie et Mod\'elisation,
CY Cergy Paris Universit\'e,
2 avenue Adolphe Chauvin, 95300, Pontoise, France}
\email{zexing.li@cyu.fr}

\maketitle

\begin{abstract}
    We consider self-similar blowup for (NLS) $i\pa_t u + \Delta u + u|u|^{p-1} = 0$ in $d \ge 1$ and slightly mass-supercritical range $0 < s_c := \frac d2 - \frac{2}{p-1} \ll 1$. The existence and stability of such dynamics \cite{MR2729284} and construction of suitable profiles \cite{MR4250747} lead to the question of asymptotic stability. Based on our previous work \cite{li2023stability}, this nonlinear problem is reduced to linear mode stability of the matrix linearized operator 
    $$\mathcal{H}_b = \left( \begin{array}{cc} \Delta_b -1 &  \\ &  -\Delta_{-b} +1\end{array} \right) -ibs_c + \left( \begin{array}{cc} W_{1,b} & W_{2,b}  \\ -\overline{W_{2,b}} & -W_{1,b}   \end{array} \right),$$
	$0 < b = b(s_c, d) \ll 1$ and
	$\Delta_b = \Delta + ib\left(\frac d2 + x\cdot\nabla\right)$. 
    In this work, we prove mode stability for the low-energy spectrum of $\calH_b$ in $d \ge 1$ as a perturbation of the linearized operator around ground state for mass-critical NLS. The main difficulty of this spectral bifurcation problem arises from the non-self-adjoint, relatively unbounded and high-dimensional nature, for which we exploit the Jost function argument from \cite{MR1852922}, qualitative WKB analysis generalized from \cite{MR4250747}, matched asymptotics method and uniform estimates for high spherical classes based on special functions.
\end{abstract}

\tableofcontents

\section{Introduction}

\subsection{Setting and background}
We consider the nonlinear Schr\"odinger equation
\be \left\{\begin{array}{l} i\pa_t u+\Delta u+u|u|^{p-1}=0\\ u_{|t=0}=u_0\end{array}\right. \tag{NLS} \label{eqNLS}\ee 
for spatial dimension $d \ge 1$ and $p > 1$. The equation is invariant under a $(2d+2)$-dim symmetry group of spatial translation, phase rotation, Galilean transform and scaling: if $u = u(t, x)$ is a solution, then so is
\be
  \tilde u(t, x) := \l^{\frac{2}{p-1}} u(\l^2 t, \l x - \l^2 t v - x_0) e^{i\left( \frac{\l x\cdot v}{2} - \frac{\l^2 |v|^2}{4}t + \gamma_0 \right)}\label{eqsymmetry}
\ee
with $x_0, v \in \RR^d$, $\l > 0$ and $\gamma_0 \in \RR$. 
The scaling symmetry
allows us to compute the critical Sobolev norm
\[ \| u_\l (t, \cdot) \|_{\dot{H}^{s_c}} = \| u (\l^2 t, \cdot) \|_{\dot{H}^{s_c}},\quad s_c := \frac{d}{2} - \frac{2}{p-1}.\]
In this paper, we are interested in the asymptotic stability of a singularity formation mechanism for nonlinearities slightly above the mass conservation law:
\be
\label{ranges}
0<s_c\ll 1.
\ee

\mbox{}

\noindent \textit{Type I blowup.} We say a solution blows up self-similarily, or of type I, if it saturates the self-similar law
\be \| u(t)\|_{\dot{H}^{\sigma}} \sim (T-t)^{-\frac{\sigma - s_c}{2}},\quad s_c < \sigma \le 1  \label{eqscalingbis}.
\ee
In the supercritical range \eqref{ranges}, type I blowup is conjectured to exist and generate stable blow up dynamics, \cite{MR1696311,MR1487664}. More precisely, the self-similar renormalization 
$$ u(t, x) = \frac{e^{i\gamma(t)}}{(\l(t))^{\frac{2}{p-1}}} v\left(\tau,y\right)$$
with the explicit law\footnote{Here $b$ is a constant independent of $t$.} 
\be
\label{eqlaw}
\lambda(t) = \sqrt{2b(T-t)},\quad \gamma(t) = \tau(t), \ \  \tau(t) = -\frac{1}{2b}\ln (T- t), \ \ y=\frac{x}{\lambda(t)}
\ee
maps \eqref{eqNLS} onto the renormalized flow 
\be
\label{enoivnevneiovnovneonvi}
i\pa_\tau v+\Delta v -v+ib\left(\frac{2}{p-1}v+y\cdot\nabla v\right)+v|v|^{p-1}=0.
\ee In the range \eqref{ranges}, admissible self-similar profiles are expected to exist from numerical grounds, \cite{mclaughlin1986focusing,lemesurier1988focusing,landman1991stability,MR1487664}. Moreover, Merle and Rapha\"el have obtained an a priori lower bound for critical norm of blowup \cite{MR2427005}, seemingly indicating that any one-bubble concentration blowup when $s_c \in (0, 1)$ should be type I. 

Besides, Type II solutions which blow up strictly faster than \eqref{eqscalingbis} have been constructed in the supercritical case through the derivation of "ring" solutions which concentrate on a sphere \cite{MR2370365, MR2248831,MR3161317,MR3324912}.

\mbox{}

\noindent\textit{Description and stability of self-similar blowup.}
The existence and stability of such self-similar dynamics was first obtained in the work of Merle-Rapha\"el-Szeftel \cite{MR2729284} for $0<s_c\ll 1$ using a bifurcation argument from the analysis of log-log blowup in critical case $s_c=0$. Nevertheless, the sharp description and asymptotic stability of the singularity, as observed numerically \cite{MR1487664}, are still missing. 

To resolve this question, the first step is to construct admissible self-similar profiles, which was done by Bahri-Martel-Rapha\"el \cite{MR4250747} in the range $0<s_c\ll1$.\footnote{We record their conclusions, namely the existence and estimates of the self-similar profile, in Proposition \ref{propQbasymp}.} Let $Q$ be the \textit{ground state} of mass-critical NLS, namely the positive radial $H^1$ solution\footnote{The existence and uniqueness of ground state are well-known (see for example \cite{MR1696311}, \cite{MR2233925} and \cite{MR2002047}). Some further properties will be discussed in Subsection \ref{sec21}.} of 
\be \Delta Q - Q + Q^{\frac 4d + 1} = 0\quad {\rm on\,\,} \RR^d. \label{eqgroundstatemasscritical} \ee
They found $0 < b = b(d, s_c) \ll 1$ and stationary solutions $Q_b$ to the elliptic equation
\be \Delta Q_b-Q_b+ib\left(\frac{2}{p-1}Q_b+y \cdot \nabla Q_b\right)+Q_b|Q_b|^{p-1}=0 \label{eqselfsimilar} \ee
by bifurcating from $Q$, see also \cite{MR3311593} for the case of (gKdV). 

Later, a finite-codimensional asymptotic stability result was obtained \cite{li2023stability}. The core of the proof is, following Beceanu's strategy \cite{beceanu2011new}, a Strichartz estimate for the linearized operator of \eqref{enoivnevneiovnovneonvi} around $Q_b$
\be \calH_b = \left( \begin{array}{cc} \Delta_b -1 &  \\ &  -\Delta_{-b} +1\end{array} \right) -ibs_c + \left( \begin{array}{cc} W_{1, b} & W_{2, b}  \\ -\overline{W_{2, b}} & -W_{1, b}   \end{array} \right),\label{eqdefcalH} \ee
	with 
	\be  \Delta_b = \Delta + ib \Lambda_0, \quad \Lambda_0 := \frac d2 + x\cdot\nabla \label{eqdefDeltabLambda0} \ee
    \be  W_{1, b} = \frac{p+1}{2}|Q_b|^{p-1}, \quad W_{2, b} = \frac{p-1}{2} |Q_b|^{p-3}Q_b^2. \label{eqdefW1bW2b} \ee 
Based on \cite{li2023stability}, the missing piece towards asymptotic stability of $Q_b$ is a sharp spectral analysis to count exactly the unstable directions, or rather, to derive \textit{mode stability} of this linearized operator. 

% \mbox{}

Noticing that when $s_c \to 0$, the profile of B-M-R \cite{MR4250747} satisfies $b(s_c) \to 0$ and $Q_b \to Q$ in $\dot H^1$, the operator $\calH_b$ can be viewed as a bifurcation from $\calH_0$, the linearized operator around the ground state $Q$ in mass-critical NLS
\bea
 \calH_0 = \left(\begin{array}{cc}
     \Delta - 1 &  \\
      & -\Delta + 1
 \end{array}  \right) + \left(\begin{array}{cc}
     W_1 & W_2 \\
     -W_2 & -W_1
 \end{array}  \right) \label{eqdefH0} \\
 W_1 = \frac{p_0+1}{2} Q^{p_0-1},\quad W_2 = \frac{p_0-1}{2} Q^{p_0-1}. \label{eqdefW1W20}
\eea

\mbox{}

\noindent \textit{Mode stability in bifurcation regime.} Briefly, mode stability means that all unstable modes of the linearized operator arise from symmetries of the equation, so that they are removable by choosing a correct reference frame during the nonlinear evolution. It can be viewed as a \textit{uniqueness} result for eigenvalues and eigenfunctions. 
This spectral problem has gone through extensive study these years: see \cite{MR3986939,MR4685953,arXiv:2501.07073,MR4126325} for parabolic problems and \cite{MR3623242,MR3475668,zbMATH07928637,MR3537340,arXiv:2503.02632,zbMATH07801696} for wave-type problems.

% To our knowledge, there have been no such results 

% To our knowledge, mode stability for self-similar profiles have only been considered for parabolic or wave-type problems, for which we provide an incomplete list of literature here \cite{MR3986939,MR4685953,arXiv:2501.07073,MR4126325,MR3623242,MR3475668,zbMATH07928637,MR3537340,arXiv:2503.02632,zbMATH07801696}. 

Notably, the self-similar profile with its linearized operator in self-similar variable is a \textit{bifurcation branch} from the ground state with its linearized operator in the original variable. In particular, the bifurcation parameter $b$ becomes a small factor for the coupled scaling generator $\Lambda_0$ in \eqref{eqdefcalH}. Most literature for mode stability listed above, on the contrary, investigates one particular self-similar solution. Despite the lack of a limiting target, one could use some numerical help in that scenario. For bifurcation problems, due to the degeneracy as $b \to 0$, it usually involves constructing the bifurcated eigenmodes and analytically computing the asymptotics of degenerate eigenvalues to identify whether they are stable.

Therefore, our problem also closely relates to the spectral analysis for Type II blowup, where the blowup profile concentrates slower, leading to a similar smallness of coupled scaling generator. We refer the readers to \cite{MR4012340,MR4073868,MR4400117,arXiv:2409.05363} for parabolic problems and \cite{MR4630477} for nonlinear wave equation, and Perelman's construction of log-log blowup for 1D mass-critical NLS \cite{MR1852922}. 

We also mention another sequence of results discussing spectrum of linearized operator around NLS ground state for $s_c \approx 0$ \cite{MR1334139,MR2219305,chang2008spectra}. This can be viewed as another bifurcation branch from $\calH_0$ by relatively compact perturbation, where more tools are applicable. 

To our knowledge, the only work involving spectral analysis for Schr\"odinger-type linearized operator with scaling generator is Perelman's pioneering work \cite{MR1852922}, from which we borrowed one crucial idea - the Jost function argument (see \textit{Comment 3}).

\subsection{Main result}

In this paper, we will perform the mode stability analysis for low-energy spectrum of $\calH_b$, namely those spectral points with small absolute values. The complete mode stability result will be established in the second paper \cite{Limodestability2}. 

% {\color{red} State the result as existence and uniqueness to describe my bifurcation result to the maximum. See how to state uniqueness in the best way.

% Also comment on the non-existence of bifurcated modes above, and conjectures the existence near $\pm 1 + inb$ and cite Perelman.}

% {\color{red} In this paper, only stress the bifurcation of eigenstates problem, instead of the mode stability. Mention the construction of Type II problem.}

% We formulate the spectral assumption on $\calH_0$:

% \begin{ass}[Characterization of spectrum of $\calH_0$]\label{assspecH0}
%     Fix $d \ge 1$. Let $\calH_0$ as closed operator on $(L^2(\RR^d))^2$. Then $\sigma_{\rm ess}(\calH_0) = (-\infty, -1] \cup [1, \infty)$ and  $\sigma_{\rm disc}(\calH_0) = \{ 0 \}$ with the generalized null space of dimension $2d+4$.\footnote{See \eqref{eqnullspaceHL2} for the eigenmodes and eigenrelations.} 
% \end{ass}

% {\color{red} Comment that the assumption is only used in the Section 7 in the proof of uniqueness.}

% Then our main result is that this picture for small spectral can be bifurcated to $0 < b \ll 1$.

% {\color{red} Bad name - "small spectrum". Probably say low energy spectrum.} 

\begin{theorem}[Mode stability of $\calH_b$ for low-energy spectrum]\label{thmmodestabsmallspec}
    For any $d \ge 1$, there exist $0 < s_c^*(d) \ll 1$ and $0 < \delta = \delta(d) \ll 1$, such that for $s_c \in (0, s_c^*(d))$ and $b, Q_b$ constructed as in \cite{MR4250747}, there exists $0 < \epsilon^*(s_c) \ll 1$ such that for any $0 < \sigma - s_c \le \epsilon^*(s_c)$, the discrete spectrum of $\calH_b$ satisfies 
       \be
      \sigma_{\rm disc}\left( \calH_b \big|_{(\dot H^\sigma(\RR^d))^2} \right) \cap \left\{ z\in\CC: \Im z < b(\sigma - s_c), |z| \le \delta \right\} = \{ 0, -bi, -2bi \}, \label{eqcharacspecHb}
    \ee
     with the corresponding Riesz projections satisfying
    \be
    \dim {\rm Ran} P_0 = 1,\quad \dim {\rm Ran} P_{-bi} = d, \quad \dim {\rm Ran} P_{-2bi} = 1. \label{eqRieszchar}
    \ee
\end{theorem}

\mbox{}

\noindent \textit{Comments on Theorem \ref{thmmodestabsmallspec}}

\mbox{}

\noindent \textit{1. On the unstable spectrum.} The eigenfunctions related to eigenvalues $0, -bi, -2bi$ are generated by phase rotation, spatial translation or scaling symmetry respectively (see \eqref{eqdefxi01b}, \eqref{eqeigencalHb1} and \eqref{eqeigencalHb3} for their explicit formula). We refer to them as unstable modes of $e^{it\calH_b}$ since the essential spectrum of $\calH_b$ is $\sigma_{\rm ess}(\calH_b \big|_{(\dot H^\sigma)^2}) = \RR + ib(\sigma -s_c)$ (see \cite[Proposition 4.5]{li2023stability}).

\mbox{}

\noindent \textit{2. Structure of low-energy spectrum.}

Notice that $Q$ and $Q_b$ are radial symmetric and thereafter $\calH_0$ and $\calH_b$ preserve spherical harmonics decomposition (see \eqref{eqsphedecomp}), we can restrict the eigenequation in each spherical class. 

Classical results (see Proposition \ref{propspecH0} and the references therein) show that $\calH_0$ has $(2d+4)$-dim generalized eigenspace at $0$. Here $2d$ directions belong to the first spherical class, generated by spatial translation and Galilean transform; while the other $4$ eigenmodes belong to the radial class, with 3 generated by phase rotation, rescaling and pseudo-conformal transform (see \cite[P.35]{MR1696311}) plus one exotic mode. 

For $0 < s_c \ll 1$, the $(2d+2)$-dim symmetry group \eqref{eqsymmetry} results in $(2d+2)$ explicit eigenmodes for $\calH_b$ (see \eqref{eqeigencalHb1} and \eqref{eqeigencalHb3}). In this work, in addition to Theorem \ref{thmmodestabsmallspec}, we verify the following statements (proven at the end of Section \ref{sec72}):
\begin{enumerate}
    \item \textit{Existence of bifurcated eigenmodes:} There exist two eigen pairs $(i\l_{j, b}, Z_{j, b}) \in i\RR_+ \times C^\infty_{rad}(\RR^d; \CC^2)$ for $j \in \{ 2, 3 \}$ such that $(\calH_b - i\l_{j, b}) Z_{j, b} = 0$,
    and
    \bee
      \l_{2, b} = 2b(1 + o_{s_c \to 0}(1)),\quad \l_{3, b} = \frac{4\pi \kappa_Q^2}{\int_0^\infty Q^2 r^{d+1} dr} b^{-3}e^{-\frac{\pi}{b}}  (1 + o_{s_c \to 0}(1)); \\
     \left| \pa_r^N Z_{j, b}(r) \right| \lesssim_{b, N} r^{-\frac d2 + \frac{\l_{j, b}}{b} + s_c - N},\quad \forall\,\,N \ge 0, \,\,r \ge 4b^{-1}. 
    \eee
    where $\kappa_Q = \lim_{r\to \infty} Q r^{\frac{d-1}{2}}e^r > 0$.
    \item \textit{Uniqueness of eigenmodes:} There exists no other eigen pairs $(i\l, Z) \in \{z \in \CC: |z| < \delta, \Im z < 10 b\} \times \left( C^\infty \cap \left( \dot H^{11}_{0} \oplus\dot H^{11}_{1} \right)\right)$ of $\calH_b$ except $(i\l_{j, b}, Z_{j, b})_{j = 2, 3}$ and the $(2d+2)$ pairs from \eqref{eqeigencalHb1} and \eqref{eqeigencalHb3}, and these eigenvalues admit no generalized eigenfunctions.\footnote{The upper bound $\Im \l \le 10b$ here can be replaced by $I_0b$ for any fixed $I_0 > 0$. Moreover, we have non-existence of eigenvalues for higher spherical classes for $\{ |z| \le \delta, \Im z < \frac 12b\}$, and it can also be extended to $I_0 b$ with some additional efforts (see Remark \ref{rmkextfundh}).} Here $\dot H^\sigma_l$ is $\dot H^\sigma$ restricted to $l$-th spherical classes, see \eqref{eqdefsphclass}.
\end{enumerate}
	\begin{figure}[h!]
		\centering
		\begin{subfigure}[b]{0.45\linewidth}
			\includegraphics[width=\linewidth]{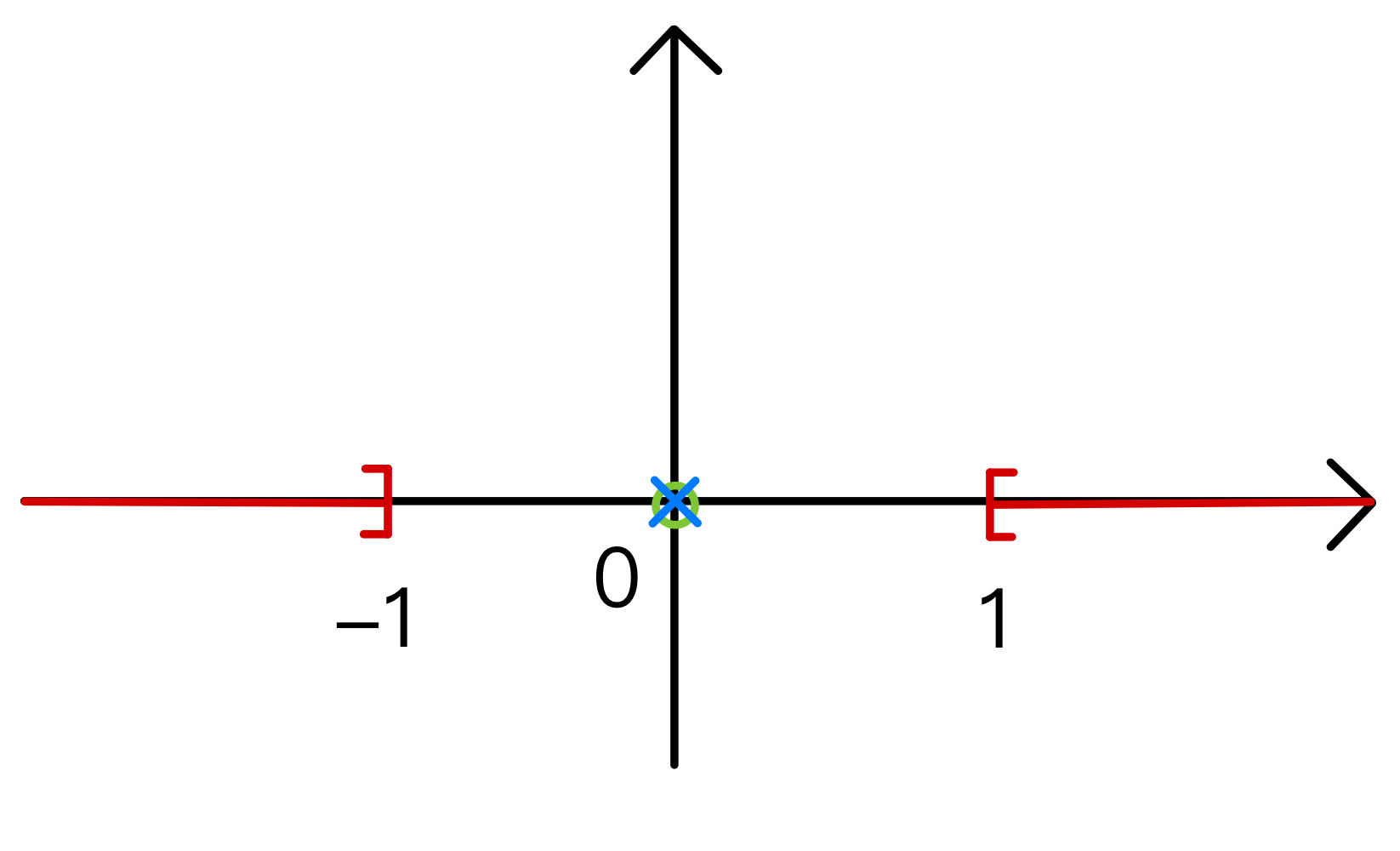}
			\caption{$\sigma(\calH_0 \big|_{(L^2)^2})$}
		\end{subfigure}
		\begin{subfigure}[b]{0.45\linewidth}
			\includegraphics[width=\linewidth]{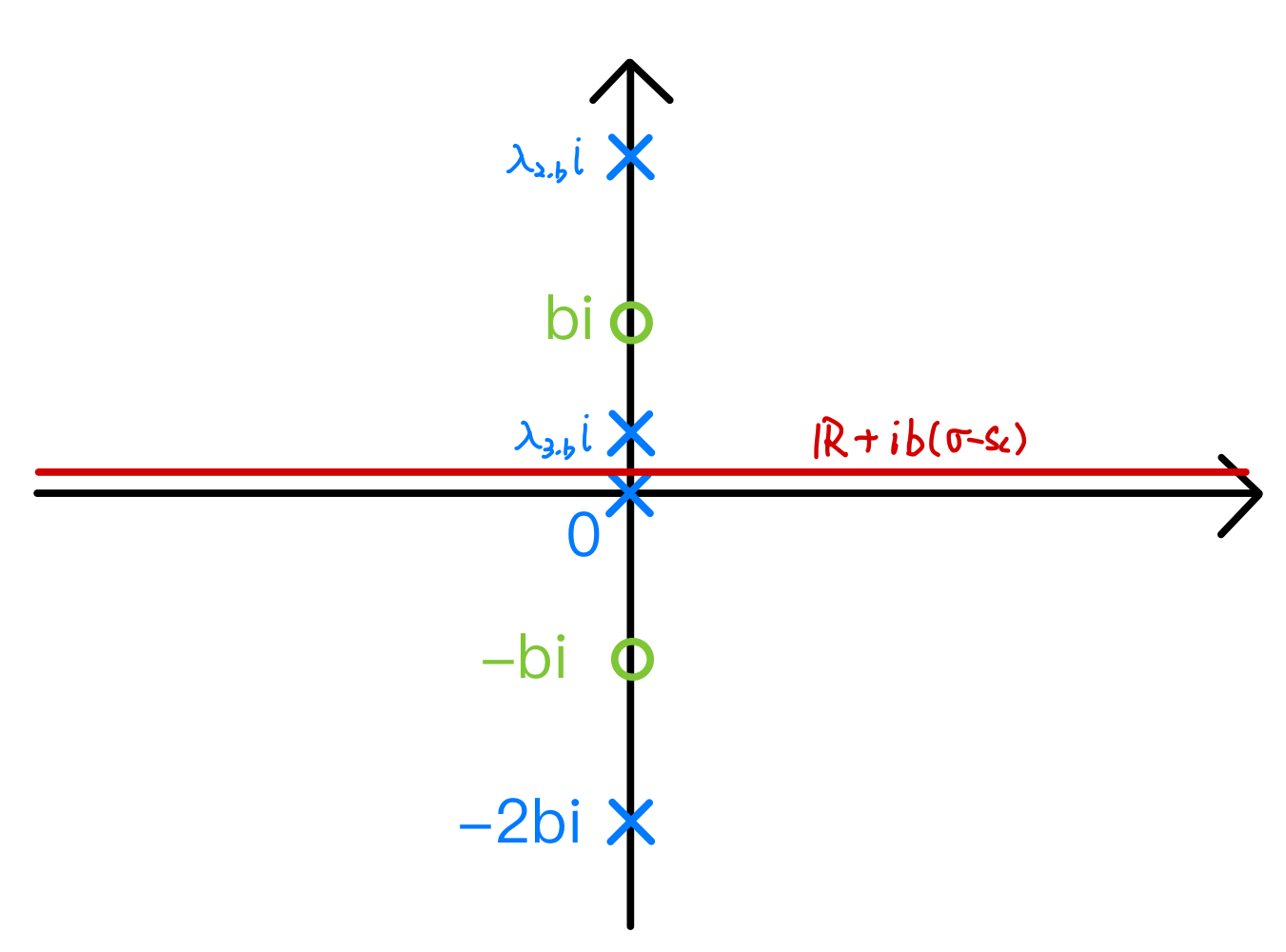}
			\caption{$\sigma(\calH_b \big|_{(\dot H^\sigma)^2})$}
		\end{subfigure}
		% \begin{subfigure}[b]{0.3\linewidth}
		% 	\includegraphics[width=\linewidth]{Admissible3.png}
		% 	\caption{$d\ge 3$}
		% \end{subfigure}
		%	\begin{subfigure}[b]{0.4\linewidth}
		%		\includegraphics[width=\linewidth]{coffee.jpg}
		%		\caption{More coffee.}
		%	\end{subfigure}
		\caption{Spectrum of $\calH_0$ and $\calH_b$ near the origin: red line for $\sigma_{\rm ess}$, blue cross for eigenpairs in radial class and green circle for eigenpairs in first spherical classes.}
		\label{fig:coffee}
	\end{figure}

In particular, we stress that there do not exist quantized stable eigenvalues near the origin, which is very different from the spectrum of linearized operator in critical/supercritical blowup regime \cite{MR4012340,MR4073868,MR4400117,arXiv:2409.05363,MR3435273}. However, we expect quantized eigenvalues would appear near $\pm 1 + i \RR_+$. For example, in 1D radial case, they should be close to $\pm 1 + ib(\frac 32 + 2n)$ for $n \ge 0$ as indicated in Perelman's work \cite[P. 653]{MR1852922}. This feature is credited to the energy-subcritical nature.

\mbox{}

\noindent\textit{3. Difficulties and novelties of Theorem \ref{thmmodestabsmallspec}.} We first summarize the main challenges of this spectral bifurcation problem.
%Noticing that the operator $\calH_b$ reduces to the non-self-adjoint linearized operator around ground state $Q$ \eqref{eqdefH0} when $b = 0$, and that $ib\Lambda_0$ is not relatively bounded by $\Delta - 1$, we view Theorem \ref{thmmodestab} as a spectral bifurcation problem for relatively unbounded perturbation of a non-self-adjoint operator. More specifically, we summarizes the main difficulties:
\begin{enumerate}
    \item \textit{Non-self-adjoint matrix operator}: We have no help from Sturm-Liouville theory or coercivity for self-adjoint operator as in the parabolic cases \cite{MR4012340,MR4073868,MR4400117,arXiv:2409.05363,MR3986939}. It is possible to encounter embedded eigenvalues and resonances, and spectrum appearing in $\CC$.
    \item \textit{Relatively unbounded perturbation}\footnote{The perturbation $ib\Lambda$ in $\calH_b$ is posed on the low frequency when formulated as a matrix Schr\"odinger operator (see \eqref{eqdefHHb}), which seems different from a semi-classical bifurcation.}: Since $ib\Lambda_0$ is not relatively bounded by $1 - \Delta$, resolvent and thus Riesz projection are no longer bounded perturbation as shown in the drastic change of essential spectrum. Moreover, the fundamental solutions of $\calH_b - \l$ exhibits $O(\frac{\l}{b})$-order polynomial behavior near infinity (see Proposition \ref{propextfund}), making it difficult to make sense of fundamental solutions (let alone resolvent) on any $b$-independent contour.
    \item \textit{Multiscale problem}: As in other Type II spectral analysis, the bifurcation parameter $b \ll 1$ determines the soliton scale $r \lesssim 1$ and the self-similar scale $r \gtrsim b^{-1}$. However, the mass-critical nature introduced another turning point $r \approx \frac 2b$ distinguishing the elliptic behavior and hyperbolic behavior, which adding to the complexity of matching asymptotics. 
    \item \textit{Polynomial decay of potential:} The potential terms $W_{1, b}$, $W_{2, b}$ \eqref{eqdefW1bW2b} in $\calH_b$ behave like $r^{-2}$ near infinity, unlike $W_1$, $W_2$ \eqref{eqdefW1W20} in $\calH_0$ enjoying exponential decay. This further poses a technical difficulty for constructing solutions with growing polynomial tails (for example $Z_{2,b}$). In particular, this forbits the direct power series expansion in $b$ used in Perelman's work \cite{MR1852922}.
    \item \textit{High dimensions}: Uniform analysis in high spherical harmonic classes is necessary. We stress that with the spherical class $l \gg 1$, the angular momentum term $(l+\frac{d-1}{2})^2r^{-2}$ will compete with the scaling generator $ib\Lambda_0$, hence shifting the turning point and creating another scale for analysis in high spherical classes. Moreover, coercivity argument as in parabolic cases seems impossible due to (1) and (2). 
\end{enumerate}

Among these difficulties, (1) and (2) cause essential trouble for uniqueness, while (3)-(5) are mainly technical difficulties for construction. Fortunately, Perelman \cite{MR1852922} tackled (1)-(2) via the introduction of \textit{Jost function}, and (3) can be resolved with the now standard matching asymptotics strategy (for example \cite{MR4073868,MR4400117}), with \textit{quantitative WKB analysis} to deal with the turning point as in \cite{MR4250747} (see also \cite{zbMATH00204189}). We brief our strategies based on these ideas plus treatment of (4)-(5). 
\begin{enumerate}
    \item \textit{Jost function argument and counting zeros of degenerate analytic function families}: Restricted to a fixed spherical harmonic class, the eigen problem can be reduced to the matching of admissible exterior solutions with interior solutions. Their Wronskian as a function of $\l \in \CC$ is referred to as Jost function, whose zeros and their multiplicities determine the existence of eigenvalue and dimension of generalized eigenspace. We prove its continuity w.r.t. $b \ge 0$, and introduce an elementary argument to show the uniqueness of zeros in degenerate regions $\{  |\l| \le \delta_0, \Im \l \le 10b \}$ without contour integral (Lemma \ref{lemuniqzero}), which might have its own interest. 

    The application of Jost function is classical for self-adjoint operators (see \cite[Chap. XI, 8.E]{MR0529429}), while for non-self-adjoint operators, its property as an indicator of generalized eigenspace was claimed in \cite{MR1334139,MR1852922} and we provide a direct and elementary proof as Lemma \ref{lemvanishJW} in our case. This property without multiplicity was also built in \cite[Lemma 5.27]{MR2219305}.  
    \item \textit{Quantitative WKB method with $\CC$-argument and exterior resolvent in admissible oscillation class}: We generalize the quantitative WKB construction in \cite{MR4250747} by using $\CC$-argument Airy function (Proposition \ref{propWKB}), and designing the inversion operator by exploiting the algebraic structure of the quadratic oscillation class (Lemma \ref{leminvtildeHext} and Appendix \ref{appA1}). We remark that the requirement to analyze eigenvalues with $O(b)$-size imaginary part makes $\RR$-argument Airy function insufficient to capture accurate asymptotics.
    \item \textit{Uniform estimates in high spherical classes}: For high spherical classes, we show the potential from spherical Laplacian will dominate the potential $W_{1,b}$, $W_{2,b}$ in the operator, leading to interior and exterior admissible solutions asymptotically free and mismatching. We employ the Tur\'an type estimates of modified Bessel functions \cite{MR2685149} in the interior region (Lemma \ref{lemmodBessel} (4)) and construct new WKB approximate solutions (Section \ref{sec42}) in the exterior region. 
\end{enumerate}

% We expect these ideas are robust and helpful for other spectral problems. In particular, an application for the operator \eqref{eqdeftildecalHb} with $1 \le d \le 10$ should lead to the nonexistence of quantized eigenvalues near the origin, indicating a drast difference for Type II singularity formation between mass-critical NLS and the known scenario in (super)critical models \cite{MR3435273,MR4438587,MR4073868}.  

% \mbox{}

% \noindent\textit{4. Comparison with Perelman's work \cite{MR1852922}.} In \cite{MR1852922}, Perelman has considered the following linear operator
% \[ H(a) = \left(\begin{array}{cc}
%     -\Delta + 1 - \frac{a r^2}{4} &  \\
%       &  \Delta - 1 + \frac{a r^2}{4}
%  \end{array}  \right) + \left(\begin{array}{cc}
%      -3\tilde \varphi_a^{4} & -2 \tilde \varphi_a^{4}\\
%      2 \tilde \varphi_a^{4} & 3\tilde \varphi_a^{4}
%  \end{array}  \right) 
%  \]
%  in $1D$. 
%  Here $\tilde \varphi_a$ is a positive and exponentially decreasing function, solving the approximate self-similar profile equation $(-\Delta + 1 - \frac{ar^2}{4} \theta(\sqrt a r))\varphi - \varphi^5 = 0$ with $\theta$ is a smooth cutoff within $B_{2-\delta}$. Taking $a = b^2$ and noticing that $\Delta_b = e^{-i\frac{b|x|^2}{4}} \circ (\Delta + \frac{b^2 r^2}{4})\circ e^{i\frac{b|x|^2}{4}}$, this operator $H(a)$ is similar to our bifurcated operator $\calH_b$ \eqref{eqdefcalH}. 

\subsection{Outline of proof}\label{sec13}

Let us set up and sketch the proof. 

\mbox{}

\underline{\textbf{Step 0. Setup: Reduction to ODE problem.}}
The target of this paper is to consider the existence and uniquenss of $Z \in \dot{H}^\sigma$ and $\l \in \CC$ such that
\be (\calH_b - \l)Z = 0 \label{eqeigenHb} \ee
where $0 < b(s_c) \ll 1$. Since $Q_b$ is radial symmetric and thereafter $\calH_b$ preserves spherical harmonics decomposition (see \eqref{eqsphedecomp}), we can restrict \eqref{eqeigenHb} to each spherical class and consider 
\be (\calH_{b, l, d} - \l) Z_{l,m} = 0 \label{eqcalHld} \ee
with $l \ge 0, 1 \le m \le N_{d, l}$, the radial vector-valued function $Z_{l, m}$, 
% $\in (\dot H^\sigma_l(\RR^d))^2$
and 
\[   \calH_{b, l, d} = 
\left[ \left( \pa_r^2 + \frac{d-1}{r}\pa_r - 1 - \frac{l(l+d-2)}{r^2}\right)\sigma_3 + ib\left(\frac d2 - s_c + r\pa_r \right) + \left( \begin{array}{cc} W_{1, b} & W_{2, b}  \\ -\overline{W_{2, b}} & -W_{1, b}   \end{array} \right)  \right]
\]
Let
\be \Phi_{l,m} = \left( \begin{array}{c} \Phi_{l,m}^1 \\ \Phi_{l,m}^2 \end{array} \right) 
=  r^{\frac{d-1}{2}}
\left( \begin{array}{c}  e^{i\frac{br^2}{4}} Z_{l,m}^1 \\ e^{-i\frac{br^2}{4}} Z_{l,m}^2 \end{array} \right) \label{eqZPhi}
\ee
and 
\be E_\pm = 1 \pm (\l + ibs_c), \label{eqdefEpm} \ee
then the radial vector-valued function $\Phi_{l,m}$ satisfies 
\be
  (\HH_{b, \nu} - \l)\Phi = 0 \label{eqnu}
\ee
with 
\be
\HH_{b, \nu} = \left(\pa_r^2 - 1 + \frac{b^2 r^2}{4} - \frac{\nu^2 - \frac 14}{r^2}\right) \sigma_3 + \left( \begin{array}{cc} -ibs_c + W_{1, b} &e^{i\frac{br^2}{2}} W_{2, b}  \\ -e^{-i\frac{br^2}{2}}\overline{W_{2, b}} & -ibs_c -W_{1, b}   \end{array} \right). \label{eqdefHHb}
\ee
and\footnote{Notice that the equation only depends on $\nu^2$, we will merely consider $\nu \ge 0$ in the constructions.}
\be \nu = \nu (l,d) = l + \frac{d-2}{2}  \in \frac 12 \NN_{\ge 0} \cup \left\{ -\frac 12 \right\}
% =\left|\begin{array}{ll}
    % l + \frac{d-2}{2}  \in \frac 12 \NN_{\ge 0} & {\rm for\,\,} d \ge 2, l \ge0,\\
    % \frac 12 & {\rm for\,\,} d = 1, l \in \{ 0, 1\},
% \end{array}\right.
\label{eqdefnu} \ee
such that $\nu^2 - \frac 14 = \frac{(d-1)(d-3)}{4} + l(l+d-2)$. 
When $b = 0$, the above definition \eqref{eqeigenHb} becomes $\calH_0$ in \eqref{eqdefH0} restricted to spherical class $l$ (depending on $\nu, d$), namely 
\be
\HH_{0, \nu} = \left(\pa_r^2 - 1 - \frac{\nu^2 - \frac 14}{r^2}\right) \left( \begin{array}{cc}  1 & \\ & -1 \end{array} \right) + \left( \begin{array}{cc}  W_{1} & W_2  \\ -W_{2} &  -W_{1}   \end{array} \right).\label{eqdefHH0}
\ee
We remark that $\HH_{b, \nu}$ for $0 \le b \le b_0$ also depends on the dimension $d$ from the potentials, but we will usually omit the subscript $d$ for simplicity. We also remark that the spherical harmonics decomposition in $d=1$ case becomes odd-even decomposition, with spherical classes $l \in \{0, 1\}$ related to $\nu \in \{ \pm \frac 12\} $. 

%  {\color{red} Make sure the notation works for $d=1$. For example, what about $\nu$ here?}

\mbox{}

\underline{\textbf{Step 1. Construction of (approximate) fundamental solutions}}

In Part \ref{part1}, we focus on designing different scalar approximate linear operators for $\calH_b$ near origin or infinity, so that we can obtain scalar approximate fundamental solution, scalar inversion operator and thereafter construct the fundamental solutions for the original system \eqref{eqnu}. Here we want to emphasize the route map and therefore will structure this outline differently from the paper's order. 

\mbox{}

\underline{Step 1.1. Scalar approximation schemes.}

We consider the following $2 \times 2$ scalar approximation regime to $\calH_b$ (or $\calH_b - \l$). The parameter $E$ below will always be close to $1$. 

\mbox{}

\noindent(1) Interior region\footnote{The specific choices of interior/exterior regions will only be clarified later until we construct fundamental solutions (Step 1.2) or bifurcated eigenfunctions (Step 2.1) through matching asymptotics. Here the region of each regime is taken to be maximal where the constructions and estimates hold.}: 
     \begin{enumerate}
         \item[(1-a)] \textit{Soliton approximation regime} (Subsection \ref{sec31}): In the radial class $l = 0$, we consider the anti-diagonalization of $\calH_b + ibs_c$ (neglecting the diagonal part $\pm \Im (e^{i\frac{br^2}{4}} W_{2, b})$):
         \be 
  L_{\pm, b} = -\Delta - \frac{b^2 r^2}{4} + 1 - W_{1, b} \mp \Re \left( e^{i\frac{br^2}{2}} W_{2, b}\right). \tag{\eqref{eqdefLpmb}}
\ee
        Therefore, it is well-approximated by that of $\calH_0$ (namely $L_\pm$ \eqref{eqdefLpm}) when $r \ll b^{-\frac 12}$. 
         \item[(1-b)] \textit{Free approximation regime} (Subsection \ref{sec32}): In any spherical classes $l \ge 0$, we drop the matrix profile potential and growing potential $\frac{b^2 r^2}{4}$ to consider 
         \be 
  L_{\nu, E} = \pa_r^2 - E - \frac{\nu^2 - \frac 14}{r^2},  \tag{\eqref{eqdefLnuE}}
\ee
whose fundamental solutions are given by modified Bessel functions (see \eqref{eqdeftildeIK}). 
     \end{enumerate}
    (2) Exterior region\footnote{Note that our exterior regions here include both sides of the turning point for $H_{b,E}$ or $H_{b, E, \nu}$ operators, (which equals roughly the union of the intermediate and exterior regions in \cite{MR4250747}).}:
         \begin{enumerate}
         \item[(2-a)] \textit{Low spherical classes} (Subsection \ref{sec41}): We drop the  matrix profile potential and angular momentum potential $\frac{\nu^2 -\frac 14}{r^2}$ to consider (see \eqref{eqscalar})
         \[ H_{b, E} = \pa_r^2 - E + \frac{b^2 r^2}{4}. \]
         Through a change of variable suggested in \cite{MR109898} (also used in \cite{MR4250747}), we can construct approximate fundamental solutions (Definition \ref{defWKBappsolu}) using $\CC$-valued Airy function to a perturbed operator $\tilde H_{b, E}$ \eqref{eqdeftildeHbE}. 
         \item[(2-b)] \textit{High spherical classes} (Subsection \ref{sec42}): Merely dropping the matrix profile potential, we consider (see \eqref{eqscalarh})
         \[ H_{b, E, \nu} = \pa_r^2 - E - \frac{\nu^2 - \frac 14}{r^2} + \frac{b^2 r^2}{4}. \]
         A similar change of variable implies the construction of approximate fundamental solutions (Definition \ref{defWKBappsoluh}) to another perturbed operator $\tilde H_{b, E, \nu}$ \eqref{eqdeftildeHbEh}. We stress that the perturbation is bounded uniformly in $\nu$ (see \eqref{eqbddhh}). 
     \end{enumerate}

\mbox{}

For each scheme, the scalar operator under consideration ($L_{\pm, b}, L_{\nu, E}, \tilde H_{b, E}, \tilde H_{b, E, \nu}$) is a Schr\"odinger operator, for which we discuss two aspects: 
\begin{itemize}
\item Find the fundamental solutions and identify their asymptotics;
\item Construct the inversion operators through Duhamel formula.
\end{itemize}
% We stress that the region of each regime is taken to be maximal where these aims can be achieved, and we choose particular subregions in the matching construction later. In particular, the union of the effective regions for (1-a) and (2-a) covers $[0, \infty)$, so that we can construct the bifurcated eigenfuctions in Step 2.1 below (Section \ref{sec6}).
Besides, a direct application of (2-a) is the refined asymptotics of $Q_b$ near infinity (Proposition \ref{propQbasympref}).

% {\color{red} No need to complain too much here. Just explain what we did.}

Those two interior schemes are easy to analyze, thanks to the abundant knowledge of $L_\pm$ (Proposition \ref{lemLpm}) for (1-a) and of modified Bessel functions (Lemma \ref{lemmodBessel}) for (1-b). 

The exterior schemes are more delicate. For the first task, the Airy function will take values from a trajectory in the $\CC$-plane which requires a careful analysis. 
%In particular, in lack of description of asymptotics through elementary functions, we will exploit some monotonicity properties (Lemma \ref{lemWKBeta}(1), Lemma \ref{lemWKBetah} (1)) for later estimates. 
For the second task, we will construct inversion operators on both sides of the turning point using different branches of fundamental solutions. Moreover, for the region near infinity, the $e^{\pm i\frac{br^2}{4}}$ oscillation (see \eqref{eqpsibderiv2}) and $r^{\pm \frac{\Im E}{b}}$ polynomial growth/decay of fundamental solutions motivate us to define the corresponding function space (Definition \ref{defdiffopspace}) and rewrite the Duhamel formula (Lemma \ref{leminvtildeHext}), so as to single out the branch with correct oscillation and integrate over polynomial growing functions with quadratic oscillation.\footnote{We only did this for the low spherical class regime (2)(a) for simplicity. We believe similar construction holds for regime (2)(b), which would lead to non-existence of eigenvalues for $|\Im \l| \le Nb$ in high spherical classes.}

\mbox{}

\underline{Step 1.2. Fundamental solutions for the original system \eqref{eqnu}.}

Next, we will apply the scalar approximate scheme (1-b) and (2-a)-(2-b) to construct 4 linear independent vector fundamental solutions of the original system \eqref{eqnu} in the interior (Proposition \ref{propintfund}) and exterior region (Section \ref{sec5}) respectively. Note that they are all free approximation schemes (namely without the profile potential). We discuss the low/high spherical class cases separately.

\mbox{}

(1) \textit{Low spherical class} (Proposition \ref{propintfund} for interior, Proposition \ref{propextfund} and Proposition \ref{propextfundin} for exterior): Fixing a $\nu_0 > 0$, for $\nu \le \nu_0$, $0 \le b \ll 1$ and $|\l| \ll 1, \Im \l \lesssim b$, we determine an $1 \ll x_* \le b^{-\frac 12}$ independent of $b$ and construct fundamental solutions of \eqref{eqnu} on $[0, x_*]$ and $[x_*, \infty)$ respectively. We mainly examine the following two properties of these fundamental solutions:
  \begin{itemize}
  \item \textit{Free asymptotics near 0 or $\infty$}: We verify they are determined by the corresponding scalar approximate fundamental solutions. In particular, the asymptotics distinguishes two admissible branches, for which the corresponding original variable $Z$ \eqref{eqZPhi} is smooth near $0$ or has no quadratic oscillation near infinity. 
  \item \textit{Continuity w.r.t. $b$ of admissible branches}: The admissible branches and its $\pa_r$-derivative evaluated at $x_*$ are continuous w.r.t. $b$ at $b = 0$. Moreover, such continuity holds for its $\pa_\l^k$-derivatives with $0 \le k \le 4$. 
  \end{itemize}

  % For the interior case (Proposition \ref{propintfund}), since the scalar approximate scheme (1-b) is $b$-independent, we can easily construct fundamental solutions near the origin by contraction principle, show continuity at $b = 0$ and further propagate this smallness to $[0, x_*]$ where both $\frac{b^2 r^2}{4}$ and the difference of potentials $|W_{j, b} - W_{j}|$ remain small. 

% We hereby explain the length of the proof of Proposition \ref{propextfund}, although the result seems well-expected:
% \begin{itemize}
%     \item Due to vector-valued nature of $\Phi$, the turning point for its two scalar equation differs as $\Re \l \neq 0$, so we have to introduce a connection region $I_{con}$ and track carefully the $e^{\pm O(\delta)b^{-1}}$ cancellation. To sum up, we will construct on three regions and match twice on the endpoints of $I_{con}$.
%     \item The continuity w.r.t. $b$ is not apparent, since the construction for $b>0$ uses very different inversion operator and matching asymptotics from $b = 0$ case. So we have to evaluate all the coefficient matrices of matching to normalize the solution branch.
%     \item We need to estimate up to $(k+1)$-th derivative w.r.t. $\l$ for identifying $k$ zeros of the Jost function later, forcing us to redo estimates of solution and coefficients on each region again. One conceptual obstacle is that the turning points are not analytic, which is resolved by introducing local analytic branches to differentiate (see Step 3(1) of the proof of Proposition \ref{propextfund}).
% \end{itemize}

  With the $b$-independent scalar approximate scheme (1-b), we can construct the interior solutions near origin via contraction and propagate the continuity estimate to $x_*$. However, due to the turning point nature for (2-a) and the continuity w.r.t. $b$ holding merely for $r \ll b^{-1}$ (see Proposition \ref{propWKB} (5) and \eqref{eqetaconv}), the exterior construction requires more delicate contraction in different zones, matching asymptotics and normalization. The non-analyticity of turning points (Lemma \eqref{lemWKBeta}) causes extra complexity to estimate its $\pa_\l^k$-derivatives.

\mbox{}

  (2) \textit{High spherical class} (Proposition \ref{propintfund} for interior, Proposition \ref{propextfundh} and Proposition \ref{propextfundin} for exterior): For $d \ge 2$, we show that there exists $\nu_0(d) \gg 1$ (independent of $b$), such that for $\nu \ge \nu_0(d)$, $0 < b \ll 1$ and $|\l| \ll 1, \Im \l \le \frac 12 b$, we can construct fundamental solutions of \eqref{eqnu} on $[0, b^{-\frac 12}]$ and $[b^{-\frac 12}, \infty)$ with the following two properties: 
  \begin{itemize}
      \item \textit{Free asymptotics near 0 or $\infty$}.
      \item \textit{Almost free asymptotics at matching point for admissible branches}: The admissible branches and its $\pa_r$-derivative evaluated at $b^{-\frac 12}$ behaves like the corresponding scalar approximate fundamental solutions. In particular, the interior/exterior ones exhibit exponential growth/decay.  
  \end{itemize}
  
The main task is to derive uniform estimates for scalar approximate fundamental solutions in high spherical classes so as to beat the effect of profile potential. For interior case (Proposition \ref{propintfund} (4)), we use the Tur\'an type estimates for modified Bessel functions (Lemma \ref{lemmodBessel}(4)); for exterior case (Proposition \ref{propextfundh}), we rely on detailed analysis of the approximate fundamental solutions (Subsection \ref{sec42}).  

\mbox{}

\underline{\textbf{Step 2. Existence and uniqueness of bifurcated eigenmodes.}}

In Part \ref{part2}, we apply the results in Part \ref{part1} to construct the bifurcated eigenpairs and prove their uniqueness. The core arguments are \textit{matching asymptotics} and \textit{Jost function argument} respectively. 

\mbox{}

\underline{Step 2.1. Construction of radial bifurcated eigenmodes.} 

In Section \ref{sec6}, we construct two bifurcated radial eigenpairs of $\calH_b$ (Proposition \ref{propbifeigen}). We will set up the asymptotic expansion ansatz with the coefficient depending nonlinearly on $\l$, and solve for the residual term and $\l$ by matching asymptotics at $x_* \sim  |\log b|$. The expansion exploits the bifurcation of exotic mode for $\calH_0$, which will be constructed beforehand (Lemma \ref{lemrhob}). 

The construction of interior and exterior solutions are based on the scalar approximate scheme (1-a) and (2-a) in Step 1.1 respectively. To extract the leading order of $\l$ (and determine its instability), we need the sharp asymptotics of $\Re Q_b$, $\Im Q_b$ on $[b^{-\frac 12}, 2b^{-1}]$ and the oscillation of $Q_b$ on $[2b^{-1}(1 + o(1)), \infty)$ from Proposition \ref{propQbasympref}. Besides, we mention that the matching process here is a two-scale matching, due to the exponential nature of asymptotics. 

\mbox{}

\underline{Step 2.2. Uniqueness of bifurcated eigenmodes.} 

In Section \ref{sec7}, we conclude the proof of Theorem \ref{thmmodestabsmallspec} using Jost function. 

For uniqueness in low spherical classes, we first define the Jost function $F_b(\l)$ in Lemma \ref{lemJostWronskian} via the fundamental solutions for \eqref{eqnu}, showing its continuity w.r.t. $b$ at $b=0$ up to 4 order of $\pa_\l$-derivatives. Since zeros of $F_b(\cdot)$ correspond to the eigenmodes of $\calH_b$ counting multiplicity (Lemma \ref{lemvanishJW}), the uniqueness of eigenmodes of $\calH_0$ (Proposition \ref{propspecH0}) can be translated into uniqueness of zeros of $F_0(\cdot)$, which propagates to $F_b(\cdot)$ for $0 < b \ll 1$ and implies the uniqueness of eigenmodes for $\calH_b$. One particular difficulty is that $F_b(\l)$ is only defined on degenerate regions $\{|\l| \ll 1, \Im \l \lesssim b \}$ for $b > 0$, forbidding using $b$-independent contour integral to show continuity of zeros. Hence we employ a complex analysis argument (Lemma \ref{lemuniqzero}) via Lagrangian interpolation polynomial, which requires the existence of all bifurcated zeros from Step 2.1.

For high enough spherical classes (independent of $b$), the almost free asymptotics of fundamental solutions from Step 1.2 indicates the mismatching of admissible branches, resulting in the non-existence of eigenmodes.

\subsection{Notations}

We use the following traditional notations: 
\begin{itemize}
    \item Japanese bracket $\la x \ra = (1 + |x|^2)^\frac 12$.
    \item The scalar Wronskian for scalar or vector-valued functions as 
\be
\calW[f, g] = 
  f\cdot g' - f' \cdot g,\quad {\rm where}\,\,\,f, g: (0,\infty) \to \CC^N,\,\, N \ge 1.
 \label{eqdefWronskian}
\ee 
\item Ball of radius $\delta$ in Banach space $X$ is denoted as $B_\delta^X = \{ x \in X: \| x \|_X < \delta \}$. The superscript may be omitted if $X = \RR^d$ and  no ambiguity occurs.
\item Pauli matrix $\sigma_3 = \left( \begin{array}{cc}
    1 &  \\
     & -1
\end{array} \right)$. 
\item Spherical harmonics decomposition in $\dot H^\sigma(\RR^d)$:
For $d \ge 2$, let $\{Y_{d, l, m}\}_{l \ge 0, 1 \le m \le N_{d, l}}$ be the spherical harmonics, namely $-\Delta_{\mathbb{S}^{d-1}} Y_{d, l, m} = l(l+d-2)Y_{d, l, m}$. In particular, the numbers of spherical harmonics at the first two classes satisfy $N_{d, 0} = 1$, $N_{d, 1} = d$. 
For $0 \le \sigma < \frac d2$, we have the spherical class decomposition for $f \in \dot H^\sigma(\RR^d)$ as
\be
  f(r\omega) =  \sum_{l = 0}^\infty \sum_{m = 1}^{N_{d, l}} f_{l, m}(r) Y_{d, l, m}(\omega),\quad {\rm with\,\,}  f_{l, m}(|\cdot|) Y_{d, l, m}(\cdot/|\cdot|) \in \dot H^\sigma_l(\RR^d).\label{eqsphedecomp}
\ee
Here the spherical class $\dot H^\sigma_l(\RR^d)$ is defined as 
\be  \dot H^\sigma_l(\RR^d) = \left\{  f \in \dot H^\sigma(\RR^d): \exists f_{l,m} \in L^1_{loc, rad}(\RR^d) \,\,{\rm s.t.}\,\, f = \sum_{m = 1}^{N_{d, l}} f_{l, m}(r) Y_{d, l, m}(\omega). \right\} \label{eqdefsphclass} \ee
so that $\dot H^\sigma (\RR^d) = \bigoplus_{l \ge 0} \dot H^\sigma_l(\RR^d)$.

For $d = 1$, $0 \le \sigma < \frac 12$, we denote $\dot H^\sigma_0 (\RR) := \dot H^\sigma_{\rm even}(\RR)$, $\dot H^\sigma_1(\RR) := \dot H^\sigma_{\rm odd}(\RR)$. Correspondingly, we write $Y_{1, 0, 1}(x) = 1$, $Y_{1, 1, 1} = \sgn(x)$ for $x \in \RR$. 
% \item Spherical class $\dot H^\sigma (\RR^d) = \bigoplus_{l \ge 0} \dot H^\sigma_l(\RR^d)$ for $d \ge 2$ and $0 \le \sigma < \frac d2$. Moreover, for $d = 1$, we denote the even-odd decomposition as $\dot H^\sigma(\RR) = \dot H^\sigma_0(\RR) \oplus \dot H^\sigma_1(\RR)$. {\color{red} Repeat from below?}
\item Weighted $C^k$ spaces: $\| f \|_{\dot C^k_\rho} = \| |\nabla^k f|/\rho\|_{C^0}$, 
  $\| f \|_{C^k_{\rho}} = \sum_{j = 0}^k \| f\|_{\dot C^j_\rho}$, and complex analytic function space $C^{\omega_\CC}(\Omega)$. Here $\nabla^k f = (\pa^{\vec \a} f)_{\substack{\vec \a \in \NN^d\\ \sum_i \a_i=k}}$ is understood as a vector-valued function.
  \item Product of Banach spaces: $\| (f, g) \|_{X_1 \times \a X_2} := \max \{ \| f \|_{X_1}, \a \| g \|_{X_2}\} $.
\item Kronecker notation $\delta_{a, b} = \left| \begin{array}{ll}
   1  &  a = b \\
   0  &  a \neq b
\end{array}\right.$.
\item $(s)_+ := \max\{s , 0 \} $.
\item Smooth radial cut-off function: Let 
\be \chi(r) = \left| \begin{array}{ll}
    1 & r \le 1 \\
    0 & r \ge \frac 32
\end{array}\right. \in  C^\infty_{c, rad}(\RR^d),\quad \chi_R = \chi(R^{-1}\cdot). \label{eqdefchiR} \ee
\end{itemize}

The following are special notations and conventions: 
\begin{itemize}
\item For parameters: We assume throughout this paper that $0 < s_c \ll 1$ so that $b = b(s_c)$ and $Q_b$ from Proposition \ref{propQbasymp} exists. For simplicity, we will also usually omit the dependence of $d$ for everything, and the dependence of $s_c$ for $b$ considering the self-similar profile $Q_b$ and $W_{1, b}$, $W_{2, b}$, $\calH_b$. 
\item Regions in $\CC$: Interior truncated region for $\delta, b, N > 0$  (used in Section \ref{sec5}, \ref{sec7})
\bee
 \Omega_{\delta, N;b} := \{z \in \CC: |z| < \delta, \Im z < b N \}, 
\eee
% and its interior as $\mathring \Omega_{\delta, N; b}$; {\color{red} Change notation to be $\Omega$ open and $\bar \Omega$ closed?}
and conic region (used in Subsection \ref{sec42})
\be \calC_{\theta_0}:= \{z = \rho e^{i\theta} \in \CC: \rho > 0, |\theta| \le \theta_0  \}.\label{eqdefcalCtheta} \ee
\item Auxiliary functions:  $S_b(r) = \int_{\min\{ r, \frac 2b \}}^{\frac 2b} \left( 1 - \frac{b^2 s^2}{4} \right)^\frac 12$,  $S_{b, \a}(r) = S_{b}(\frac{r}{\sqrt\a}) \a$ for $\a \in \RR$, and $\omega_{b, E}^\pm$ as \eqref{eqomegapm}. 
\item Auxiliary function spaces: $Z_{\pm, \a; x_*}, \tilde Z_{\pm, \a; x_*}$ \eqref{eqdefZpm} in the interior region, and $X^{\a, N, \pm}_{r_1;b, E}, X^{\a, N, \pm}_{r_0, r_1;b, E}$ \eqref{eqdefXBanach} in the exterior region.
\end{itemize}

\mbox{}

\textbf{Acknowledgments.} The author is supported by the ERC advanced grant SWAT and ERC starting grant project FloWAS (Grant agreement No.101117820). He sincerely thanks his supervisor Pierre Rapha\"el for his constant encouragement and support. The author is also grateful to Charles Collot, Yvan Martel, and Cl\'ement Mouhot for valuable suggestions, and to Zhiyan Sun for helpful discussions.

\part{Preparation} \label{part1}

\section{Preliminary of spectral properties}\label{sec2}

\subsection{Linearization around the ground state}
\label{sec21}
% Consider the semilinear NLS
% \be
% i \partial_{t} u+\Delta u=-|u|^{p} u,\label{eqNLS}\tag{NLS}
% \ee
% with $x \in \RR^d$, $d \ge 1$ and $p < 2^*$ for $2^* = \left| \begin{array}{ll} \infty & d= 1, 2 \\ \frac{4}{d-2} + 1 & d \ge 3 \end{array}\right.$.
% % \be d \le 10,\quad p \in [ \frac 4d + 1, 2^*)  \quad {\rm for}\, 2^* = \left| \begin{array}{ll} \infty & d= 1, 2 \\ \frac{4}{d-2} + 1 & d \ge 3 \end{array}\right.. \label{eqprange}
% % \ee  

\mbox{}

\underline{Ground state}

For each $p < 2^*$, it is well-known (see for example \cite{MR1696311,MR2233925,MR2002047}) that there uniquely exists a positive radial $H^1$ solution $Q$ for 
\[ \Delta Q - Q + Q^p = 0 \] 
called the ground state. By further standard ODE analysis, we have the following asymptotics and estimates.
\begin{lemma}[Asymptotics and estimates for ground state]\label{lemQasymp}
    For $d \ge 1$ and $p < 2^*$, there exists $\kappa_Q > 0$ and $c_Q \in \RR$, such that for $k = 0, 1$, 
\be
\pa_r^k\left( r^{\frac{d-1}{2}} Q(r)\right) =  (-1)^k \kappa_Q e^{-r} \left(1 + c_Q r^{-1} + O(r^{-2}) \right),\quad \forall\, r \ge 1,
 % \sum_{k =0}^1 \left|\pa_r^k Q(r) - (-1)^k \kappa_Q r^{-\frac{d-1}{2}} e^{-r} \left(1 + \left( c_Q + \delta_{k,1} \frac{d-1}{2} \right) r^{-1} \right) \right| \lesssim r^{-\frac{d+3}{2}} e^{-r}, 
 \label{eqQdecay1}
 \ee
 and 
 \be
 Q(x) \sim \la x \ra^{-\frac{d-1}{2}} e^{-|x|},\quad  \left| \nabla^n Q(x) \right| \lesssim_n \la x \ra^{-\frac{d-1}{2}} e^{-|x|},\quad \forall\, x \in \RR^d,\,\,n\ge 1. \label{eqsolitondecay}
\ee
\end{lemma}

We also need the following difference estimate for slightly mass-supercritical ground states $Q_p$. 

\begin{lemma}[Continuation of ground state]\label{lemcontgs}
    For $d \ge 1$,  $\left| p - p_0 \right| \ll 1$ with $p_0 = \frac 4d + 1$, we have 
    \be \left\| Q_p - Q_{p_0}\right\|_{H^2 \cap W^{2,\infty}} \lesssim_d \left| p - p_0  \right|. \label{eqcontinuitysoliton} \ee
\end{lemma} 

For completeness, we present the proof of these two lemmas in Appendix \ref{appA1}.

\mbox{}

\underline{Spectral properties.}

Now we consider $\calH_0$ and let $Q$ be the ground state for the mass-critical case $p = p_0 = \frac 4d + 1$.  Recall the linearized operator around $Q$ as \eqref{eqdefH0}-\eqref{eqdefW1W20} and define the related scalar operators
\be L_\pm = -\Delta + 1 - W_1 \mp W_2, \quad \Lambda_0 = \frac d2 + x\cdot \nabla.  \label{eqdefLpm} \ee
We record the generalized kernel of $\calH_0$ from \cite[Section 2.1]{chang2008spectra}. 
Let 
\be Q_1 = \Lambda_0 Q, \quad \rho =  L_+^{-1} (|x|^2 Q). 
\label{eqdefQ1Q2rho} \ee
Here the invertibility of $L_+$ on $L^2_{rad}$ comes from \cite[Section 2.1]{chang2008spectra}.
We also compute using $L_+ Q_1 = - 2Q$ (see \cite[Section 2.1]{chang2008spectra}) that
\be
 \left( \rho, Q\right)_{L^2(\RR^d)} = \left(|x|^2 Q, -\frac 12\Lambda_0 Q\right)_{L^2(\RR^d)} = \frac{1}{2}\| xQ\|_{L^2(\RR^d)}^2.
 \label{eqrhobQ}
\ee
% A standard compactness analysis with Weyl's criteron implies
% \be  \sigma_{ess}(\calH_0 \big|_{L^2(\RR^d)}) = (-\infty, -1] \cup [1, \infty). \label{eqspecessH0}  \ee

Now we summarize the knowledge of $\sigma(\calH_0)$ in the following proposition.
\begin{proposition}[Spectrum of $\calH_0$]\label{propspecH0}
  For $d \ge 1$, let $\calH_0$ from \eqref{eqdefH0} be the closed operator on $(L^2(\RR^d))^2$. 
  \begin{enumerate}
  \item Essential spectrum:  $\sigma_{\rm ess}(\calH_0) = (-\infty, -1] \cup [1, \infty)$.
  \item Generalized nullspace: Define the eigenmodes 
\be
\left| \begin{array}{l}
\xi_0 = i \left(\begin{array}{c} Q \\ -Q \end{array}\right), \quad
\xi_1 = \frac 12 \left(\begin{array}{c} Q_1 \\ Q_1 \end{array}\right),\\
\xi_2 = -\frac{i}{8} \left(\begin{array}{c} |x|^2 Q \\ -|x|^2 Q \end{array}\right), \quad
\xi_3 = \frac 18 \left(\begin{array}{c} \rho \\ \rho \end{array}\right). \\
\zeta_{0, j} = \left(\begin{array}{c} \pa_j Q \label{eqxi2xi3} \\ \pa_j Q \end{array}\right) \quad \zeta_{1, j} = -\frac i2 \left(\begin{array}{c} x_j Q \\ - x_j Q \end{array}\right),\quad 1 \le j \le d.
\end{array}\right.
\ee
They satisfy the algebraic relations \bee \calH_0 \xi_0 = 0, && \calH_0 \xi_{k} = -i\xi_k,\quad k = 1, 2, 3; \\ 
 \calH_0 \zeta_{0, j} = 0, && \calH_0 \zeta_1 = -i \zeta_0,\quad 1 \le j \le d.
\eee
and the generalized nullspace of $\calH_0$ is of dimension $2d+4$, generated by these vectors
\be N_g (\calH_0) = {\rm span} \{ \xi_0, \xi_1, \xi_2, \xi_3, \zeta_{0, j}, \zeta_{1, j}  \}_{1 \le j \le d}. \label{eqnullspaceHL2}\ee
  \end{enumerate}
\end{proposition}

The proof of (1) is standard application of Weyl's essential spectral theorem (see for example \cite[Theorem XIII.14]{reed1972methods4}), and (2) is from \cite{MR0783974} and \cite{kwong1989uniqueness} (also see \cite[Section 2.1]{chang2008spectra}).

\subsection{Linearization around self-similar profile}\label{sec22} 

\mbox{}

\underline{Existence and asymptotics of $Q_b$.}

\begin{proposition}[Existence and asymptotics of $Q_b$, \cite{MR4250747}]\label{propQbasymp}
For any $d \ge 1$, there exists $s_c^{(0)}(d) \ll 1$ such that for $0 < s_c \le s_c^{(0)}(d)$, there exists $b = b(s_c, d) > 0$ with
\be s_c \sim b^{-1} e^{-\frac{\pi}{b}}, \label{eqbasymp}\ee
and $Q_b \in C^2(\RR^d) \cap \dot H^1(\RR^d)$ solving \eqref{eqselfsimilar} such that for $k = 0, 1$, 
\begin{align}
 &|\pa_r^k Q_b(r)| \sim (br)^{-k} r^{-\frac d2+s_c} b^{-\frac 12}e^{-\frac{\pi}{2b}} , \quad  r \ge b^{-2}; \label{eqQbasymp1}\\
% |\pa_r Q_b(r)| &\lesssim& (br)^{-1} r^{-\frac d2+s_c} b^{-\frac 12}e^{-\frac{\pi}{2b}} ,\quad r \ge b^{-2};\label{eqQbasymp7} \\
  &|\pa_r^k Q_b(r)| \left| \begin{array}{ll}
      \sim r^{-\frac{d-1}{2}} b^{-\frac 16}e^{-\frac{\pi}{2b} + S_b(r)} \left\la b^{-\frac 23} (4-b^2r^2)\right\ra^{-\frac 14},& r \in [b^{-\frac 12},  b^{-2}], k=0;  \\
       \lesssim \la br\ra r^{-\frac{d-1}{2}} b^{-\frac 16}e^{-\frac{\pi}{2b} + S_b(r)} \left\la b^{-\frac 23} (4-b^2r^2)\right\ra^{-\frac 14},& r \in [b^{-\frac 12},  b^{-2}], k= 1;
  \end{array}\right.\label{eqQbasymp2}\\
  &|\pa_r^k (\Re P_b - Q)(r)| \lesssim \la r \ra^{-\frac{d-1}{2}} e^{-r}\left(b^\frac 13 + b^\frac 16 e^{-2(b^{-\frac 12} - r)}\right), \quad r \le b^{-\frac 12}; \label{eqQbasymp5} \\
  &|\pa_r^k \Im P_b(r)| \lesssim bs_c \la r \ra^{-\frac{d-1}{2}} e^{r},\quad  r \le b^{-\frac 12}; \label{eqQbasymp6}
\end{align}
and for $k \ge 2$, 
\bea
  |\pa_r^k Q_b (r)| &\lesssim_k& \left| \begin{array}{ll} 
  \la r \ra^{-\frac{d-1}{2}}e^{-r} &  r \le b^{-\frac12} \\
  r^{-\frac{d-1}{2}} b^{-\frac 16}e^{-\frac{\pi}{2b} + S_b(r)} \left\la b^{-\frac 23} (2-br)\right\ra^{-\frac 14} & r \in [b^{-\frac 12},  4b^{-1}]
  \end{array} \right. \label{eqQbasymp7} \\
  |\pa_r^k (\Re P_b-  Q)(r)| &\lesssim_k& b^\frac 13 \la r \ra^{-\frac{d-1}{2}} e^{-r}, \quad r \le b^{-\frac 13}; \label{eqQbasympint1}\\
  |\pa_r^k \Im P_b(r)| &\lesssim_k&  b s_c \la r \ra^{-\frac{d-1}{2}} e^{r}, \quad r \le b^{-\frac 13}. \label{eqQbasympint2}
\eea
  Here $P_b = Q_b e^{i\frac{br^2}{4}}$ and $Q$ is the mass-critical ground state \eqref{eqgroundstatemasscritical}, $S_b(r) = \int_{\min\{ r, \frac 2b \}}^{\frac 2b} \left( 1 - \frac{b^2 s^2}{4} \right)^\frac 12$. 
% \bea \label{eqQbasymp7}
%  |\pa_r Q_b(r)| \lesssim b^{-1} r^{-1} |Q_b(r)|,\quad |\pa_r^2 Q_b(r)| \lesssim |Q_b(r)| \quad r \ge b^{-2}.
% \eea
\end{proposition}

This proposition is a corollary from the main theorem and its proof in \cite{MR4250747}. We write its proof in Appendix \ref{appQb}. Besides, some refined asymptotics will be proven later in Proposition \ref{propQbasympref}, where we prove the higher derivative estimates in the exterior region and the lower bound of $|Q_b|$ on $[b^{-\frac 12}, b^{-2}]$ in particular. 

%The estimates \eqref{eqQbasymp1}-\eqref{eqQbasymp6} are proven in  \cite{MR4250747} through the construction process by contraction. In particular, the derivative estimates follow the Duhamel structure for $k = 1$ and the profile equation \eqref{eqselfsimilar} for $k = 2$; the substitution of $Q_p$ in \cite{MR4250747} to $Q = Q_{p_0}$ in \eqref{eqQbasymp5} is due to \eqref{eqcontinuitysoliton} and $p-p_0 \sim s_c \ll b^{\frac d2}e^{-b^{-\frac 12}}$. The higher derivative estimate  for $r \le 4b^{-1}$ \eqref{eqQbasymp7} easily follows from differentiating the profile equation \eqref{eqselfsimilar}. 

\mbox{}

\underline{Spectrum of $\calH_b$.} 

Let $\calH_b$ be the linearized operator \eqref{eqdefcalH}. Inspired by the eigenmodes of $\calH_0$ \eqref{eqxi2xi3}, we define  
\be \label{eqdefxi01b}
\left| \begin{array}{l}
\xi_{0,b} = i \left(\begin{array}{c} Q_b \\ -\bar{Q}_b \end{array}\right), \quad
\xi_{1,b} = \frac 12 \left(\begin{array}{c} \Lambda Q_b \\ \overline{\Lambda Q_b} \end{array}\right), \\
\xi_{2,b} = -\frac{i}{8} \left(\begin{array}{c} |x|^2 Q_b \\ -|x|^2 \bar{Q}_b \end{array}\right), \quad
\eta_b = \left(\begin{array}{c} Q_b \\ \bar{Q}_b \end{array}\right), \\
\zeta_{0, j, b} = \left(\begin{array}{c} \pa_j Q_b \\ \overline{\pa_j Q_b} \end{array}\right),\quad \zeta_{1,j,b} = - \frac i2 \left(\begin{array}{c} x_j  Q_b \\ -x_j \bar{Q}_b \end{array}\right),\quad 1 \le j \le d. 
\end{array}\right.
\ee
Then by differentiating the self-similar profile equation \eqref{eqselfsimilar}, we obtain the algebraic relations
\be
\begin{split}
    \calH_b \xi_{0, b} = 0,\quad \calH_b \xi_{1,b} = -2bi\xi_{1,b} - i\xi_{0,b},\quad \calH_b \xi_{2,b} = -i\xi_{1,b} + 2bi\xi_{2,b} - i\frac{s_c}{2}\eta_b, \\
\calH_b \zeta_{0,j,b} = -i \zeta_{0, j, b},\quad \calH_b \zeta_{1, j, b} = - i \zeta_{0, j, b} + i b \zeta_{1, j, b}.
\end{split} \label{eqQbalgrel}
\ee
thus 
\bea
\calH_b \xi_{0, b} = 0,\quad (\calH_b + 2bi) \left( \xi_{0,b} + 2b\xi_{1,b} \right) = 0,\label{eqeigencalHb1} \\
(\calH_b - 2bi)\left( \xi_{0,b} - 2b\xi_{1,b} + 8b^2 \xi_{2,b} \right) = -i 4b^2 s_c \eta_b,\label{eqeigencalHb2}\\  
(\calH_b + bi)\zeta_{0, j, b} = 0,\quad (\calH_b - bi) \left( \zeta_{0,j,b} - 2b\zeta_{1,j,b} \right) = 0.\label{eqeigencalHb3}
\eea

\section{Interior ODE analysis}

\subsection{Scalar case: soliton approximation regime}
\label{sec31}
In this subsection, we consider the scalar operators $L_\pm$ from \eqref{eqdefLpm} and $L_{\pm, b}$ (defined later in \eqref{eqdefLpmb}), which 
involve soliton potential and are
related to the radial case $l = 0$ of $\HH_{b, \nu}$. 
Their fundamental solutions and inversion will be discussed, which will be used to construct radial bifurcated eigenmodes in Section \ref{sec6}. 

% {\color{red} Discuss more carefully in the introduction: they are related to  $r^{-\frac{d-1}{2}} \HH_{b, \nu = \frac{d-2}{2}} r^{\frac{d-1}{2}}$ , and involve the soliton behavior so serves as relatively accurate approximation.}

We will use the Banach space $Z_{\pm, \a; x_*} \subset C^0_{rad}(B_{x_*}^{\RR_d})$ and $\tilde Z_{\pm, \a; x_*} \subset C^1_{rad}(B_{x_*}^{\RR_d})$ for $\a \in \RR$ and $x_* \in (0, \infty]$ that 
\be \| f \|_{Z_{\pm, \a; x_*}} := \left\|f  \right\|_{C^0_{e^{\pm r} \la r \ra^{\a - \frac{d-1}{2}} } (B_{x_*}) },\quad \| f \|_{\tilde Z_{\pm, \a; x_*}} := \left\|f  \right\|_{C^1_{e^{\pm r} \la r \ra^{\a - \frac{d-1}{2}} } (B_{x_*}) }   \label{eqdefZpm} \ee

The first lemma regards $L_\pm$, which is based on \cite[Lemma 4.1]{MR4250747} and whose proof will be given in Appendix \ref{appA1}. 

\begin{lemma}[Properties of $L_\pm$]   \label{lemLpm}
Let $d \ge 1$ and $p_0 = \frac 4d + 1$. $Q$ be the ground state of \eqref{eqNLS} and $\kappa_{Q}$ from \eqref{eqQdecay1}.

    \noindent (1) Fundamental solutions: There exist radial functions $A: [0,\infty) \to \RR$, $D: (0, \infty) \to \RR$ and 
    \be
      \frakE (r) = Q(r) \int_1^r Q(s)^{-2} s^{-(d-1)} ds \label{eqdefE}
    \ee
    such that 
    \[ L_- Q = L_- \frakE = 0, \quad L_+ A = L_+ D = 0, \quad r \in (0, \infty). \]
    There exists $\kappa_A \neq 0$ and $c_A, c_D, c_\frakE \in \RR$ for $ k= 0, 1$
    \bea
    \left| \begin{array}{l}
        \pa_r^k \left( r^{\frac {d-1}{2}} A(r) \right) = \kappa_A e^r \left( 1 + c_Ar^{-1} + O(r^{-2})  \right),\\
      \pa_r^k \left( r^{\frac {d-1}{2}} D(r) \right) = (-1)^{k+1} (2\kappa_A)^{-1}  e^{-r} \left( 1 + c_D r^{-1} + O(r^{-2})\right), \\
      \pa_r^k \left( r^{\frac {d-1}{2}} \frakE(r) \right) = (2\kappa_Q)^{-1} e^r \left( 1 + c_\frakE r^{-1} + O(r^{-2})  \right),
    \end{array}\right.\quad \forall\, r\ge 1, 
    \eea
    and 
    \be
    \left| \begin{array}{l}
        \pa_r^k A(r), Q(0)^{-1} \cdot \pa_r^k Q(r) = \delta_{k, 0} + O(r^{2-k}) \\
        D(r), Q(0) \cdot \frakE(r) = \left| \begin{array}{cc}
            (2-d)^{-1} r^{2-d}(1 + O(r)) & d \ge 3 \\
            \log r(1+O(r)) & d = 2\\
            r(1 + O(r)) & d = 1
        \end{array}\right. \\
        D'(r), Q(0) \cdot \frakE'(r) = r^{-d+1} (1 + O(r)), 
    \end{array}\right.\quad \forall\, r \in (0, 1].
    \ee
    and the Wronskian 
    \be 
      AD' - DA' = Q \frakE' - \frakE Q' = r^{-(d-1)},\quad \forall\, r >0.
    \ee

    \noindent (2) Inversion operators: For $x_* \in [1, \infty]$, define
     \be
     \begin{split} L_{+; x_*}^{-1} f =& -A \int_r^{x_*} f D s^{d-1} ds - D\int_0^r fA s^{d-1} ds, \\
  L_-^{-1} f =& -Q \int_0^r \left[\int_0^s f(\tau) Q(\tau) \tau^{d-1} d\tau \right] \frac{ds}{Q^2(s) s^{d-1}}. \end{split} \label{eqdefLpminv}\ee
  In particular, denote $L_{+; \infty}^{-1} =: L_+^{-1}$. Then they are inversions of $L_\pm$ respectively, and satisfy\footnote{Here $|\SS^{d-1}| = \frac{2\pi^{d/2}}{\Gamma(d/2)}$ is the area of $(d-1)$-dimensional unit sphere. In particular, $|\SS^{0}| = 2$.}
  \begin{align}
   &\| L_{+; x_*}^{-1} f \|_{\tilde Z_{-, \a+1;x_*}} \lesssim_\a \| f \|_{Z_{-, \a;x_*}}, \,\,
   \| L_- ^{-1} f \|_{\tilde Z_{+, \a+1;x_*}} \lesssim_\a \| f \|_{Z_{+, \a;x_*}},\,\, \forall \,\a \ge 0\label{eqLpmest1}\\
   &\| L_{+; x_*}^{-1} f \|_{\tilde Z_{\pm, 0; x_*} } \lesssim_\a \| f \|_{Z_{\pm, \a; x_*}}, \ \| L_- ^{-1} f \|_{\tilde Z_{+, 0; x_*}} \lesssim_\a \| f \|_{Z_{+, \a; x_*}},\,\, \forall\, \a \le  -2,\label{eqLpmest2}\\
   &\left\| L_-^{-1} f + |\SS^{d-1}|^{-1}(f, Q)_{L^2(B_{x_*}^{\RR^d})} \tilde{\frakE} \right\|_{\tilde Z_{-, \a+1; x_*}} \lesssim_\a \| f \|_{Z_{-, \a;x_*}},\quad \forall\, \a \in \RR \label{eqLpmest3}
  \end{align}
  where $(f, Q)_{L^2(B_{x_*}^{\RR^d})} = |\SS^{d-1}| \int_0^{x_*} fQ\tau^{d-1}d\tau$, $\tilde{\frakE} = \frakE (1 - \chi)$ with cut-off $\chi$ from \eqref{eqdefchiR}, and the constants are independent of $x_* \ge 1$.
\end{lemma}

\mbox{}

Next, we define the scalar operators with $b > 0$ as 
\be 
  L_{\pm, b} = -\Delta - \frac{b^2 r^2}{4} + 1 - W_{1, b} \mp \Re \left( e^{i\frac{br^2}{2}} W_{2, b}\right) \label{eqdefLpmb}
\ee
and the difference potential as
\be
  V_{\pm, b} = L_{\pm, b} - L_{\pm} = -\frac{b^2 r^2}{4} - \left( W_{1,b} - W_1 \right) \mp \left( \Re \left( e^{i\frac{br^2}{2}} W_{2,b}\right) - W_2 \right). \label{eqdefVpmb}
\ee

\begin{lemma} \label{lemLpmb} For $d \ge 1$, there exists $s_{c;{\rm int}}^{(1)} > 0$ such that for $0 < s_c <  s_{c;{\rm int}}^{(1)}$ and $1 \le x_* \le b^{-\frac 13}$ with $b = b(s_c, d) \ll 1$, we can define 
\be L_{+, b; x_*}^{-1} = \left(1 + L_{+;x_*}^{-1} V_{+,b}\right)^{-1} L_{+;x_*}^{-1},\quad L_{-, b}^{-1} = \left(1 + L_{-}^{-1} V_{-,b}\right)^{-1} L_{-}^{-1},   \label{eqdefLpmb} \ee
and the following bounds hold
 \begin{align}
  & \left| V_{\pm, b} (r) \right| \lesssim b^\frac 43 + b^\frac 13 \la r \ra^{-\frac{(p-1)(d-1)}{2}} e^{-(p-1)r},\quad \forall \, r \le b^{-\frac 13}, \label{eqestVpmb} \\
   &\| L_{+,b; x_*}^{-1} f \|_{\tilde Z_{-, \a+1;x_*}} \lesssim \| f \|_{Z_{-, \a;x_*}}, \,\,
   \| L_{-,b} ^{-1} f \|_{\tilde Z_{+, \a+1;x_*}} \lesssim \| f \|_{Z_{+, \a;x_*}},\,\, \forall \,\a \in [0, 10] \label{eqLpmbest1}\\
   &\| L_{+,b; x_*}^{-1} f \|_{\tilde Z_{\pm, 0; x_*} } \lesssim \| f \|_{Z_{\pm, \a; x_*}}, \ \| L_{-,b} ^{-1} f \|_{\tilde Z_{+, 0; x_*}} \lesssim \| f \|_{Z_{+, \a; x_*}},\,\, \forall\, \a \in [-10, -2],\label{eqLpmbest2}\\
   &\left\| L_{-,b}^{-1} f + |\SS^{d-1}|^{-1} (f - V_{-,b} L_{-,b}^{-1} f, Q)_{L^2(B_{x_*}^{\RR^d})} \tilde{\frakE} \right\|_{\tilde Z_{-, \a+1; x_*}} \lesssim \| f \|_{Z_{-, \a;x_*}},\quad \forall\, |\a| \le 10, \label{eqLpmbest3}
  \end{align}
  where the constant is independent of $x_* \in [1, b^{-\frac 13}]$, and $L_{+;x_*}^{-1}, \,L_-^{-1},\, \tilde{\frakE}$, $(\cdot,\cdot)_{L^2(B_{x_*}^{\RR^d})}$ are from Lemma \ref{lemLpm}. Besides,
there exists a function $A_b: [0, b^{-\frac 13}] \to \CC$ such that $L_{+,b} A_b = 0$ and  
\be A_b = A + \epsilon, \quad \| \epsilon \|_{\tilde Z_{+,0;b^{-\frac 16}}} \lesssim b^\frac 13 \label{eqestAb} \ee
with $A$ from Lemma \ref{lemLpm}.
\end{lemma}

\begin{proof} We first require $s_{c;{\rm int}}^{(1)} < s_c^{(0)}$ from Proposition \ref{propQbasymp}, and might further shrink to ensure smallness of $b$. 
Then the estimate \eqref{eqestVpmb} follows Proposition \ref{propQbasymp} 
\begin{align}
\left| V_{\pm, b} (r) \right| &\le \frac{b^2 r^2}{4} +  \left| \frac{p+1}{2} |P_b|^{p-1} - \frac{p_0+1}{2} Q^{p_0 -1}\right| +  \left| \frac{p-1}{2} |P_b|^{p-3}\Re (P_b^2) - \frac{p_0-1}{2} Q^{p_0 -1}\right| \nonumber \\
&\lesssim b^\frac 43 + b^\frac 13 \la r \ra^{-\frac{(p-1)(d-1)}{2}} e^{-(p-1)r} \lesssim \la r \ra^{-2} b^\frac 13,\qquad \forall \, r \le b^{-\frac 13}. \nonumber
\end{align}
where we used $p - p_* \sim s_c \ll b^{10}$. Therefore with \eqref{eqLpmest1}-\eqref{eqLpmest2}, we have 
\be
\| L_{+;x_*}^{-1} V_{+, b} \|_{\calL \left( \tilde Z_{-, \a; x_*} \right) \cap \calL \left( \tilde Z_{+, 0; x_*} \right)} + \| L_{-}^{-1} V_{-, b} \|_{\calL \left( \tilde Z_{+, \a; x_*} \right)} \lesssim b^\frac 13,\quad \forall\,\a \in [0,10]. \label{eqnonoq}
\ee
This implies the well-definedness and estimates \eqref{eqLpmbest1}-\eqref{eqLpmbest2} for $L_{+,b;x_*}^{-1}$ and $L_{-,b}^{-1}$ when $b \ll 1$; 
and \eqref{eqLpmbest3} follows from observing that 
\bee
 u = L_{-,b}^{-1} f \quad \Rightarrow \quad L_- u = f - V_{-,b} u \quad \Rightarrow \quad u = L_-^{-1} \left( f - V_{-,b} u \right)
\eee
and applying \eqref{eqLpmest3}. Finally, since $A$ satisfies $L_+ A = 0$, we choose
$\epsilon = L_{+, b; b^{-\frac 13}}^{-1} \left( - V_{+, b} A\right)$, and the boundedness follows \eqref{eqnonoq} and $\| A \|_{\tilde Z_{+, 0; \infty}} \lesssim 1$ from Lemma \ref{lemLpm}. 
\end{proof}

\subsection{Scalar and system case: free approximation regime}\label{sec32}

Here we consider the scalar operator
\be 
  L_{\nu, E} = \pa_r^2 - E - \frac{\nu^2 - \frac 14}{r^2},  \label{eqdefLnuE}
\ee
by dropping $b$-dependence and soliton potentials from $\HH_{b, \nu}$. Their fundamental solutions can be constructed from modified Bessel functions $I_\nu$, $K_\nu$ with $\nu \ge 0$ (see for example \cite[Chapter 10]{MR2723248})
\bea 
 I_\nu(z) &=& \frac{(z/2)^\nu }{\pi^\frac 12 \Gamma (\nu + \frac 12)} \int_{-1}^1 (1-t^2
)^{\nu - \frac 12} e^{-zt} dt, \label{eqInu}\\
 K_\nu(z) &=& \frac{\pi^\frac 12 (z/2)^\nu }{ \Gamma (\nu + \frac 12)} \int_{1}^\infty (t^2-1
)^{\nu - \frac 12} e^{-zt} dt.\label{eqKnu}\eea
Both functions solve 
\[ z^2 \frac{d^2 w}{dz^2} + z \frac{dw}{dz} - (z^2 + \nu^2) w =0. \]
Thereafter, if we define 
\be \tilde{I}_{\nu, E}(z) = (\sqrt E z)^\frac 12 I_\nu( \sqrt E z),\quad  \tilde K_{\nu, E}(z) = (\sqrt E z)^\frac 12 K_\nu(\sqrt E z), \label{eqdeftildeIK} \ee
then they will satisfy
\be L_{\nu, E} \tilde{\mathscr{Z}}_{\nu, E} = 0 \label{eqLnuEeq} \ee
where $\tilde{\mathscr{Z}}$ denotes $\tilde{I}$ or $\tilde K$. 

We first record some properties of $I_\nu$ and $K_\nu$ with particular emphasis on qualitative estimates when $\nu$ large. Then we will construct fundamental solutions of \eqref{eqnu} near the origin and evaluate the large $\nu$ asymptotics up to $r = b^{-\frac 12}$.

\begin{lemma}[Properties of modified Bessel functions] \label{lemmodBessel}
For $\nu \ge 0$, the modified Bessel functions $I_\nu(z)$, $K_\nu(z)$ are analytic functions in $\{ z \in \CC - \{0\}: |{\rm arg} z| \le \frac \pi {4} \}$ with the following properties:
    \begin{enumerate}
        \item Asymptotics: for $|{\rm arg} z| \le \frac \pi 4 $, 
        \bee
        I_\nu(z) &=& \left\{\begin{array}{ll} \frac{z^\nu}{2^\nu \Gamma(\nu + 1)} + O(z^{\nu + 2}), & |z|\le 1; \\
        (2\pi z)^{-\frac 12} e^z \left( 1+ O(z^{-1}\right),& |z| \ge 1;
        \end{array}\right. \\
        K_\nu(z) &=& \left\{\begin{array}{ll} 2^{\nu - 1} \Gamma(\nu) z^{-\nu} (1 + o(1)), & |z|\le 1,\,\,\nu > 0; \\
        2^{\nu - 1} \Gamma(\nu) z^{-\nu}( 1 - \frac{z^2}{4(\nu-1)}  + o(z^2)), & |z|\le 1,\,\,\nu > 2; \\
        -\ln z + \ln 2 - \gamma + O(z), & |z| \le 1, \,\,\nu = 0;\\
        \pi^\frac 12 (2z)^{-\frac 12} e^{-z} \left( 1+ O(z^{-1}\right),& |z| \ge 1;
        \end{array}\right. 
        \eee
        where $\gamma$ is the Euler's constant, and the constants of residuals $O(...), o(...)$ may depend on $\nu$. Moreover, when $\nu = \frac 12$, 
        \be I_{\frac 12}(z) = (2\pi z)^{-\frac 12}(e^z - e^{-z}),\quad K_{\frac 12}(z) = \pi^\frac 12 (2z)^{-\frac 12} e^{-z}. \label{eqexplicitIK12}
        \ee
        \item Wronskian: 
        \be \mathcal{W}(I_\nu, K_\nu) = -z^{-1}. \label{eqWronskiIK}
        \ee
        \item Derivatives: let $\mathscr{Z}_\nu$ denote $I_\nu$ or $e^{\nu \pi i} K_\nu$,
        \be \mathscr{Z}_\nu'(z) =  \mathscr{Z}_{\nu - 1}(z) - \frac{\nu}{z}  \mathscr{Z}_\nu (z) =  \mathscr{Z}_{\nu + 1}(z) + \frac{\nu}{z}  \mathscr{Z}_\nu (z). \label{eqBesselderiv} \ee
        \item Tur\'an type estimate: for $|{\rm arg} z| \le \frac \pi {16} $,
        \bea \frac{I_\nu'(z)}{I_\nu(z)} = \sqrt{1 + \frac{\nu^2}{z^2} - \varphi_\nu(z)},\quad  |\varphi_\nu(z)| \le \frac{1}{\nu + 1},&\quad& \forall \nu \ge 0; \label{eqTuranI}\\
        \frac{K_\nu'(z)}{K_\nu(z)} = -\sqrt{1 + \frac{\nu^2}{z^2} - \phi_\nu(z)},\quad|\phi_\nu(z)| \le \frac{1}{\nu - 1},&\quad& \forall \nu \ge 3.  \label{eqTuranK}
        \eea
        Moreover, $|\varphi_\nu(x e^{i\theta})|$ and $|\phi_\nu(x e^{i\theta})|$ are strictly decreasing w.r.t. $x > 0$ when $|\theta| \le \frac \pi {16}$.
        \item Monotonicity and product estimates: For $|\theta| \le \frac{\pi}{16}$ and $\nu \ge 3$, the map $r \mapsto |I_\nu(e^{i\theta} r)|$ is increasing, and $r \mapsto |K_\nu(e^{i\theta} r)|$ is decreasing on $r \in (0, \infty)$. Moreover, we have
        \be
          |I_\nu(z) K_\nu(z)| \lesssim |z^2 + \nu^2|^{-\frac 12},\quad |\arg z | \le \frac{\pi}{16} 
          \label{eqIKproduct2} 
        \ee
        where the constant is independent of $\nu \ge 3$, and 
        \be \frac{I_\nu(z)}{I_\nu(z')} = \left(\frac{z}{z'}\right)^\nu e^{\eta_\nu(z, z')\cdot(z-z')},\quad {\rm with}\,\,|\eta_\nu(z, z')| \le 2, \label{eqIratio2}
        \ee
        for $|\arg z|, |\arg z'| \le \frac{\pi}{16}$ and $\nu \ge 3$. 
        % {\color{purple}
        % \item {\color{red} No need this whole part.} Real argument case: when $z = x > 0$, $I_\nu(x)$ is positive and increasing and $K_\nu(x)$ is positive and decreasing; they act as controlling functions
        % \[ |I_\nu(z)| \le |z/ \Re z|^\nu I_\nu(\Re z),\quad |K_\nu(z)| \le |z/ \Re z|^\nu  K_\nu(\Re z). \] 
        % Morevoer, we have the following bounds:
        % \bea I_\nu(x) K_\nu(x) \lesssim \min\{x^{-1}, \nu^{-1}\}; \label{eqIKproduct} \\
        % \left(\frac{x}{y}\right)^\nu < \frac{I_\nu(x)}{I_\nu(y)} < e^{x-y} \left(\frac{x}{y}\right)^\nu,\quad x > y > 0. \label{eqIratio}
        % \eea
        % where the constant is independent of $\nu \ge 0$. {\color{red} Actually \eqref{eqIratio} can also be implied by the proof of \eqref{eqIratio2} using Turan inequatities with $\varphi_\nu(x) > 0$. }
        % }
    \end{enumerate}
\end{lemma}
\begin{proof} (1), (2), (3) are standard from \cite[Chapter 10]{MR2723248}. The refined asymptotics for $K_\nu$ when $\nu > 2$ follows from \eqref{eqBesselderiv}, and the explicit formula \eqref{eqexplicitIK12} can be verified by the equation \eqref{eqLnuEeq} and their limit asymptotics.

(4) The proof is a slight generalization of \cite[Theorem 2.1, Theorem 3.1]{MR2685149}. 

For \eqref{eqTuranI}, using argument as \cite[Theorem 2.1]{MR2685149} with the complex-valued Mittag-Leffler expansion for $|\arg z| \le \frac \pi 4$ \cite[10.21.15, 10.27.6]{MR2723248}, we first obtain
 \cite[(2.2)]{MR2685149}
 \[ \left(1 + \frac{\nu^2}{z^2} \right) I_\nu(z)^2 - (I_\nu'(z))^2 = I_\nu(z)^2 \sum_{n \ge 1} \frac{4j_{\nu,n}^2}{(z^2 + j_{\nu, n}^2)^2},\quad |\arg z| \le \frac\pi {16}, \]
  where $j_{\nu, n} > 0$ are positive zeros of Bessel function $J_\nu(z)$. The representation formula in \eqref{eqTuranI} follows by setting $\varphi_\nu(z) = \sum_{n \ge 1} \frac{4j_{\nu,n}^2}{(z^2 + j_{\nu, n}^2)^2}$, where the sign of square root is taken according to the asymptotics of $I_\nu(z)$ near $z= 0$. The Rayleigh formula $\sum_{n \ge 1} 4j_{\nu,n}^{-2} = (\nu + 1)^{-1}$ \cite[15.51]{MR0010746} implies $\varphi_{\nu}(0) = \frac{1}{\nu+1}$, and obviously $\lim_{|z|\to \infty}\varphi_\nu(z) = 0$ with $|\arg z| \le \frac{\pi}{16}$. Finally, the uniform boundedness of $|\varphi_{\nu}(z)|$ and monotonicity follows the computation 
  \bee
  \pa_x |\varphi_\nu(xe^{i\theta})|^2 = -8x \Re \left[ e^{2i\theta} \left( \sum_{n \ge 1} \frac{4j_{\nu,n}^2}{(x^2 e^{2i\theta} + j_{\nu, n}^2)^3} \right)\left(\sum_{n \ge 1} \frac{4j_{\nu,n}^2}{(x^2 e^{-2i\theta} + j_{\nu, n}^2)^2} \right) \right] \le 0,
  \eee
  using $|\theta| \le \frac{\pi}{16}$ and $j_{\nu, n} > 0$. 

 For \eqref{eqTuranK}, using argument as \cite[Theorem 3.1]{MR2685149} with the  integral identity \cite[(3.5)]{MR2685149} extended to complex-valued case \cite[(1.4)]{MR448480}, we have
 \bee
   \phi_\nu(z) := 1 + \frac{\nu^2}{z^2} -\left( \frac{K_\nu'(z)}{K_\nu(z)}\right)^2 = -\frac{4}{\pi^2} \int_0^\infty \frac{(z^2 + t^2 +1 )\gamma(t)}{(z^2 + t^2)^2} dt,\quad |\arg z| \le \frac{\pi}{20}, \,\, \nu \ge 3,
 \eee
 where $\gamma(t) = \frac{t^{-1}}{J_\nu^2(t) + Y_\nu^2(t)} > 0$. From the asymptotics of $K_\nu(z)$ and $K'_\nu(z)$ using (1) and (3), we see $\lim_{|z|\to \infty} \phi_\nu(z) = 0$ and  $\lim_{|z|\to 0} \phi_\nu(z) = -(\nu - 1)^{-1}$ with $|\arg z| \le \frac{\pi}{16}$. Now the bound and monotonicity follows 
 \bee
  &&\pa_x|\phi_{\nu}(xe^{i\theta})|^2 \\
  &=& -\frac{64}{\pi^4} \Re \left[ \int_0^\infty \frac{(x^2 e^{-2i\theta} + t^2 + 1)\gamma(t) dt }{(x^2 e^{-2i\theta} + t^2 )^2} \cdot \int_0^\infty \frac{xe^{2i\theta} (x^2 e^{2i\theta} + t^2 + 2)\gamma(t) dt }{(x^2 e^{2i\theta} + t^2 )^3}   \right] \le 0
 \eee
 When $\theta \in[0, \frac{\pi}{16}]$, the non-positivity comes from $\arg \left( \frac{x^2 e^{-2i\theta} + t^2 + 1}{(x^2 e^{-2i\theta} + t^2 )^2} \right) \in [0,4\theta]$ and $\arg \left( \frac{xe^{2i\theta} (x^2 e^{2i\theta} + t^2 + 2)}{(x^2 e^{2i\theta} + t^2 )^3} \right) \in [-4\theta, 2\theta]$ and $|8\theta| \le \frac \pi 2$, and similarly when $\theta \in [-\frac{\pi}{16}, 0]$. 

 (5) This is a corollary of the Tur\'an inequalities. For the monotoniocity of $|I_\nu(e^{i\theta} r)|$, we apply \eqref{eqTuranI} to compute 
 \bea
  \frac 12\pa_r |I_\nu(e^{i\theta} r)|^2 &=& \Re \left[ e^{i\theta} I_\nu'(e^{i\theta} r) \overline{I_\nu(e^{i\theta} r)} \right] \nonumber \\
  &=& |I_\nu(e^{i\theta} r)|^2 \cdot \Re\left[  \sqrt{e^{2i\theta} + \nu^2  r^{-2} + e^{2i\theta}\varphi_\nu (e^{i\theta}r) }\right] > 0 \label{eqmonoInu}
 \eea
 with $|\varphi_\nu(e^{i\theta} r)| \le \frac 12$ and $\Re e^{2i\theta} > \frac{\sqrt 2}{2}$. Similarly we have the monotonic decreasing of $r \mapsto |K_\nu(e^{i\theta}r)|$ using \eqref{eqTuranK}. For \eqref{eqIKproduct2}, we further exploit the Wronskian \eqref{eqWronskiIK} to see
 \bee
    I_\nu(z) K_\nu(z) \left[\sqrt{1 + \frac{\nu^2}{z^2} - \varphi_\nu(z)} + \sqrt{1 + \frac{\nu^2}{z^2} - \phi_\nu(z)}  \right] = - \calW(I_\nu, K_\nu) = z^{-1},
 \eee
 which implies \eqref{eqIKproduct2} with the elementary estimate $|\sqrt{1+\nu^2 z^{-2}} - \sqrt{1 + \nu^2 z^{-2} + \e}| \le |\e| \cdot |1 + \nu^2 z^{-2}|^{-\frac 12}$ for $|\e| \le \frac 12$. 
 
 For \eqref{eqIratio2}, we integrate \eqref{eqTuranI} subtracted by $\frac{\nu}{z}$ along the segment from $z$ to $z'$, which yields the identity \eqref{eqIratio2} with 
 \[ \eta_{\nu}(z, z')\cdot (z-z') = \int_{z'}^z \left( \sqrt{1 + \frac{\nu^2}{\omega^2} - \varphi_\nu(\omega)} - \frac{\nu}{\omega} \right) d\omega \]
 The bound of $\eta_\nu$ follows 
 \[ \left|\sqrt{1 + \frac{\nu^2}{\omega^2} - \varphi_\nu(\omega)} - \frac{\nu}{\omega} \right| = \frac{|1-\varphi_\nu(\omega)|}{|\sqrt{1 + \frac{\nu^2}{\omega^2} - \varphi_\nu(\omega)} + \frac{\nu}{\omega}|} \le |1-\varphi_\nu(\omega)|^\frac 12 \le 2\]
 thanks to $|\arg (1-\varphi_\nu(\omega))|, |\arg(\frac{\nu}{\omega})| \le \frac \pi 4$ under the assumption. 

% {\color{purple}
% (6) The positivity and monotonicity in (4) are standard from \cite[Chapter 10]{MR2723248}. 
% The controlling property in (4) follows their integral definitions \eqref{eqInu}, \eqref{eqKnu}. 
% For \eqref{eqIKproduct} in (4), we use \cite[10.32.16]{MR2723248}
% \[ I_\nu(x) K_\nu(x) = \int_0^\infty J_0(2x \sinh t) e^{-2\nu t} dt,\quad \nu \ge 0. \]
% The direct bound $|J_0(x)| \le 1$ \cite[10.9.1]{MR0167642} implies $I_\nu(x) K_\nu(x) \lesssim \nu^{-1}$ when $\nu > 0$; and for $O(x^{-1})$ bound with $\nu \ge 0$, we further exploit
% \cite[9.2.1]{MR0167642}
% \[ J_0(y) = \left(\frac{2}{\pi y}\right)^{\frac 12} \cos\left(y - \frac 14 \pi\right) + O(y^{-\frac 32}),\quad y \ge 1, \]
% to decompose $ I_\nu(x) K_\nu(x) = I + II + III$, where
% \bee
%    I &=& \int_0^{x^{-1}} J_0(2x\sinh t) e^{-2\nu t} dt \\
%   \left(\frac{2}{\pi}\right)^{-\frac 12}  II &=& \int_{x^{-1}}^\infty ( x\sinh t)^{-\frac 12}  \cos\left(x\sinh t - \frac 14 \pi\right) e^{-2\nu t} dt \\
%    &=& - \left[(x \sinh t)^{-\frac 12} (x \cosh t)^{-1} e^{-2\nu t} \sin \left(x\sinh t - \frac 14 \pi\right) \right]\bigg|_{t = x^{-1}} \\
%    &&- \int_{x^{-1}}^\infty x^{-\frac 32} \pa_t \left( (\sqrt{\sinh t} \cosh t)^{-1} e^{-2\nu t} \right)  \sin \left(x\sinh t - \frac 14 \pi\right) dt  \\
%    III &=& \int_{x^{-1}}^\infty O((x \sinh t)^{-\frac 32}) e^{-2\nu t}
% \eee
% can be controlled by $O(x^{-1})$ independent of $\nu \ge 0$. 
% Finally, \eqref{eqIratio} is from \cite[(1.5)]{MR1094928}.
% }
 
\end{proof}

\begin{proposition}[Construction of interior fundamental solution] \label{propintfund} For $d \ge 1$, there exist $s_{c;{\rm int}}^{(2)}(d) > 0$, $b_{\rm int} > 0$ and $\delta_{\rm int} > 0$ such that for any $0 < s_c \le s_{c;{\rm int}}^{(2)}$ and $b = b(s_c, d), \l, \nu$ satisfying
\[  0 \le b \le b_{\rm int}, \quad |\l| \le \delta_{\rm int}, \quad \nu \ge 0, \] 
there exist four fundamental solutions $\Psi_{j;b, \l,\nu}$ for $j = 1, 2, 3, 4$ with $r \in (0, \infty)$ solving \eqref{eqnu}
 satisfying
 \begin{enumerate}
     \item Asymptotics at $0$: when $r \to 0$,
     \begin{align} \pa_r^k \Psi_{j;b,\l,\nu}(r) &= \left(\frac{\nu + \frac 12}{r} \right)^k r^{\nu + \frac 12} \vec e_j
     + O(r^{\nu + \frac 52- k}), \quad j = 1, 2,\,\, k = 0, 1,\,\, \nu \ge 0, \label{eqPsiasympest1}\\
     \pa_r^k \Psi_{j;b,\l,\nu}(r)& = (2\nu)^{-1} \left(\frac{\frac 12 -\nu}{r} \right)^k r^{ \frac 12 -\nu} \vec e_{j-2}
     + o(r^{-\nu + \frac 12 - k}), \quad j = 3, 4,\,\,k= 0, 1, \,\,\nu > 0, \label{eqPsiasympest2}\\
     \pa_r^k \Psi_{j;b,\l,0}(r) &= \left| \begin{array}{ll}
     \left(- \ln (\sqrt{1- (-1)^j \l} r) + \ln 2 - \gamma \right)  r^\frac 12 \vec  e_{j-2} + O(r^{\frac 32}) & k = 0,\\
     \frac 12 \left(- 2 - \ln (\sqrt{1- (-1)^j \l} r) + \ln 2 - \gamma \right) r^{-\frac 12}\vec e_{j-2} + O(r^{\frac 12}) & k = 1,
     \end{array}\right. \,\, j = 3, 4. \label{eqPsiasympest3}
     \end{align}
     where $\vec e_1 = (1, 0)^\top$, $\vec e_2 = (0, 1)^\top$ are the basis of $\CC^2$ and $\gamma$ is the Euler's constant. In particular, $\{ \Psi_{j;b,\l,\nu} \}_{1 \le j \le 4}$ are linear independent functions. Moreover, when $\nu = \frac 12$, we can refine \eqref{eqPsiasympest2} to be
     \be
      \pa_r^k \Psi_{j;b,\l,\frac 12}(r) = \pa_r^k (1-r)\vec e_{j-2} + O(r^{2-k}),\quad j = 3,4,\,\, k = 0, 1.  \label{eqPsiasympest4}
     \ee
     \item Analyticity w.r.t. $\l$: for any $r > 0$ and $1 \le j \le 4$, $\Psi_{j;b,\l,\nu}(r)$ and $\pa_r \Psi_{j;b,\l,\nu}(r)$ are analytic w.r.t. $\l$, and $\{ \pa_\l^n \Psi_{j;b,\l,\nu} \}_{\substack{1 \le j \le 4, n \ge 0}}$ are linear independent functions. For $\nu \ge 0$ and $j = 1, 2$, we have the estimate
     \be 
     \pa_r^k \pa_\l^n \Psi_{j;b,\l,\nu}(r) = O(r^{\nu + \frac 52- k}), \quad {\rm as} \,\, r \to 0\,\, {\rm for}\,\, k = 0, 1,\,\, n \ge 1. \label{eqPsipalest}
     \ee
     \item Continuity w.r.t. $b$ at $b = 0$: 
     For any $n \ge 0$, $r_* > 0$ and $1 \le j \le 4$, supposing additionally $b \le \min \{ b_{\rm int}, r_*^{-2}\}$, we have that
     \be
        \sum_{k = 0}^1 \sup_{|\l| \le \delta_0} \big|\pa_r^k \pa_\l^n \Psi_{j;b,\l,\nu}(r_*) - \pa_r^k \pa_\l^n \Psi_{j;0,\l,\nu}(r_*)\big| \lesssim_{ n, \nu} b^{\frac 16}.\label{eqPsicontb}
     \ee 
     \item Almost free asymptotics at high angular momentum: for any $d \ge 2$, there exists $\nu_{\rm int}(d)$ such that when $\nu \ge \nu_{\rm int}(d)$, we have two fundamental solutions $\check \Psi_{j;b,\l,\nu}$ for $j = 1, 2$ such that 
     \be {\rm span}\{ \check \Psi_{j;b,\l,\nu}\}_{j = 1}^2 = {\rm span}\{\Psi_{j;b,\l,\nu}\}_{j = 1}^2, \label{eqlindepcheckPsi} \ee
     and are evaluated at $x_* = b^{-\frac 12}$ as 
     \bea 
     &&\left| \begin{array}{l}
      \check \Psi_{1;b,\l,\nu}(x_*) = \left( \vec e_1 + O_\RR^2 (b^\frac 12 + \nu^{-1}) \right) \ti_{\nu, 1+\l}(x_*)\\
      \pa_r \check \Psi_{1;b,\l,\nu}(x_*) = \left( \sqrt{1 + \l + b\nu^2} \vec e_1 + O_\RR^2 (b^\frac 12 + \nu^{-1}) \right) \ti_{\nu, 1+\l}(x_*)
     \end{array} \right.
      \label{eqinthighnu1}
\\
&&\left| \begin{array}{l}
      \check \Psi_{2;b,\l,\nu}(x_*) = \left( \vec e_2 + O_\RR^2 (b^\frac 12 + \nu^{-1}) \right) \ti_{\nu, 1-\l}(x_*)\\
      \pa_r \check \Psi_{2;b,\l,\nu}(x_*) = \left( \sqrt{1 - \l + b \nu^2} \vec e_2 + O_\RR^2 (b^\frac 12 + \nu^{-1}) \right) \ti_{\nu, 1-\l}(x_*)
     \end{array} \right.
      \label{eqinthighnu2}
%  \tilde \Psi_{2;b,\l,\nu}(r) &=&  \left(\begin{array}{c}
%       O(b^\frac 12 + \nu^{-1})  \left|\frac{1-\l}{1+\l}\right|^\frac \nu 2 \ti_{\nu, 1+\l}(r) \\ 
%       \left( 1+ O(b^\frac 12 + \nu^{-1})\right) \ti_{\nu, 1-\l}(r)
% \end{array}\right), \label{eqinthighnu2} \\
%  \pa_r \tilde  \Psi_{1;b,\l,\nu}(r) &=&  \left(\begin{array}{c}
%      \left( 1+ O(b^\frac 12 + \nu^{-1})\right) \sqrt{1+\l+ \frac{\nu^2}{r^2}} \ti_{\nu, 1+\l}(r)  \\ 
%       O(b^\frac 12 + \nu^{-1})\left|\frac{1+\l}{1-\l}\right|^\frac \nu 2  \sqrt{1-\l+ \frac{\nu^2}{r^2}}  \ti_{\nu, 1-\l}(r)
% \end{array}\right),  \label{eqinthighnu3}\\
%  \pa_r \tilde  \Psi_{2;b,\l,\nu}(r) &=&  \left(\begin{array}{c}
%      O(b^\frac 12 + \nu^{-1})\left|\frac{1-\l}{1+\l}\right|^\frac \nu 2  \sqrt{1+\l+ \frac{\nu^2}{r^2}} \ti_{\nu, 1+\l}(r)  \\ 
%       \left( 1 + O(b^\frac 12 + \nu^{-1}) \right) \sqrt{1-\l+ \frac{\nu^2}{r^2}}  \ti_{\nu, 1-\l}(r)
% \end{array}\right),   \label{eqinthighnu4}
\eea
where
% \[ \tilde \Psi_{1;b,\l,\nu} = \frac{(1 + \l)^{\frac \nu 2 + \frac 14}}{2^\nu \Gamma(\nu + 1)} \Psi_{1;b,\l,\nu}, \quad 
% \tilde \Psi_{2;b,\l,\nu} = \frac{(1 - \l)^{\frac \nu 2 + \frac 14}}{2^\nu \Gamma(\nu + 1)} \Psi_{2;b,\l,\nu},\]
% and 
the bounds of $O(b^\frac 12 + \nu^{-1})$ are uniform for $\l, \nu, b$ in the above range. 
\end{enumerate}
\end{proposition}
\begin{remark}[Interior solutions for 1D even case]\label{rmkchoice1Dint}
    For the case $d = 1$, $l = 0$, we have $\nu = -\frac 12$ and hence the equation \eqref{eqnu} is the same for $d = 1, l=1$ case where $\nu = \frac 12$ with $\{\Psi_{j;b,\l,\frac 12}\}_{1 \le j \le 4}$ as the fundamental solution family. Given the asymptotics \eqref{eqPsiasympest1}, we choose  
\be
\left| \begin{array}{ll}
    \Psi_{j;b,\l,-\frac 12} = \Psi_{j;b,\l,\frac 12} + \Psi_{j+2;b,\l,\frac 12} & j = 1, 2; \\
    \Psi_{j;b,\l,-\frac 12} = \Psi_{j;b,\l,\frac 12} & j = 3, 4;
\end{array}\right.
\ee
so that
\be  \Psi_{j;b,\l,-\frac 12}(0) = \vec e_j,\quad \pa_r \Psi_{j;b,\l,-\frac 12}(0) = \vec 0,\quad {\rm for}\,\, j = 1, 2. \label{eq1Devenbdry}\ee
\end{remark}
\begin{proof} We choose $\delta_{\rm int} = \frac{\min\{ 4/d, 1\}}{100} < \frac{\min\{ p-1, 1\}}{100}$, and assume $b_{\rm int} \le \frac{1}{100}$ to be determined in the proof of (4). Finally, we let $s_{c;{\rm int}}^{(2)} \le s_c^{(0)}$ from Proposition \ref{propQbasymp} small enough such that $b(s_c, d) \le b_{\rm int}$ for all $s_c \in (0, s_{c;{\rm int}}^{(2)})$. 

\underline{1. Inversion operators of $L_{\nu, E}$.}

Consider  $\ti_{\nu, E}$, $\tk_{\nu, E}$ from \eqref{eqdeftildeIK}, with $|E - 1| \le 50^{-1} \ll 1$. The Wronskian of $I_\nu, K_\nu$ \eqref{eqWronskiIK} implies
\[ \calW(\ti_{\nu, E}, \tk_{\nu, E}) = E.\]
So we define the inversion operators of $L_{\nu, E}$ through Duhamel formula as 
\bea T^{int, (0)}_{\nu, E} f &=& \ti_{\nu, E} \int_0^r \tk_{\nu, E} f E^{-1} ds - \tk_{\nu, E} \int_0^r \ti_{\nu, E} f E^{-1} ds,\\
\tilde T^{int, (0)}_{r_0;\nu, E} f &=& -\ti_{\nu, E} \int_r^{r_0} \tk_{\nu, E} f E^{-1} ds - \tk_{\nu, E} \int_0^r \ti_{\nu, E} f E^{-1} ds,
\eea
with $r_0 > 0$, and the derivatives w.r.t. $E$
\bee
   T^{int, (k)}_{\nu, E} f = \int_0^r \pa_E^k \left(  \ti_{\nu, E}(r) \tk_{\nu, E}(s) E^{-1}\right) f(s) ds - \int_0^r \pa_E^k \left(  \tk_{\nu, E}(r) \ti_{\nu, E}(s) E^{-1}\right) f(s) ds, \\
   \tilde T^{int, (k)}_{r_0; \nu, E} f = -\int_r^{r_0} \pa_E^k \left(  \ti_{\nu, E}(r) \tk_{\nu, E}(s) E^{-1}\right) f(s) ds - \int_0^r \pa_E^k \left(  \tk_{\nu, E}(r) \ti_{\nu, E}(s) E^{-1}\right) f(s) ds.
\eee
When taking $\pa_r$, the integrand is invariant
\be \begin{split}
 \pa_r T^{int, (k)}_{\nu, E} f &= \int_0^r \pa_r \pa_E^k \left(  \ti_{\nu, E}(r) \tk_{\nu, E}(s) E^{-1}\right) f ds - \int_0^r \pa_r \pa_E^k \left(  \tk_{\nu, E}(r) \ti_{\nu, E}(s) E^{-1}\right) f ds, \\
 \pa_r \tilde T^{int, (k)}_{r_0; \nu, E} f &=  -\int_r^{r_0}\pa_r \pa_E^k \left(  \ti_{\nu, E}(r) \tk_{\nu, E}(s) E^{-1}\right) f ds - \int_0^r \pa_r \pa_E^k \left(  \tk_{\nu, E}(r) \ti_{\nu, E}(s) E^{-1}\right) f ds. \end{split} \label{eqintegrandinv}
\ee
With $|E-1| < \frac 1{50}$, we can bound $\pa_E^n \tilde{\mathscr{Z}}_{\nu, E}$ from \eqref{eqdeftildeIK} and Lemma \ref{lemmodBessel} (1), (3)
\bee 
  |\pa_E^n \tilde{\mathscr{Z}}_{\nu, E}(r)| &\lesssim_n& \sum_{k=0}^n r^k \left| \left(\pa_z^k \tilde{\mathscr{Z}}_{\nu, E}\right) (\sqrt E r)\right| \lesssim_{n, \nu} \sum_{k=0}^n r^k \left|\tilde{\mathscr{Z}}_{\nu + k, E}(r)\right| \\
  &\lesssim_{n, \nu}& \left| \begin{array}{ll}
      r^{\nu + \frac 12} & \mathscr{Z} = I,\,\, \nu \ge 0, \\
      r^{-\nu+ \frac 12 } & \mathscr{Z} = K,\,\, \nu > 0, \\
      r^\frac 12|\ln r| & \mathscr{Z} = K,\,\,\nu = 0,
  \end{array}\right.\quad {\rm for}\,\, r \le 1,\,\, n \ge 0.
\eee
This implies the boundedness on $[0, r_0]$ with $r_0 \le 1$, 
\begin{align}
&\left| \begin{array}{ll}
\| T^{int,(k)}_{\nu, E}  \|_{\calL \left( Y^+_{\nu, r_0} \to \tilde Y^+_{\nu, r_0} \right)} + \|\tilde T^{int, (k)}_{r_0; \nu, E}  \|_{\calL \left( Y^-_{\nu, r_0} \to \tilde Y^-_{\nu, r_0}\right) } \lesssim_{\nu, k} r_0, & \nu > 0, \\
 \| T^{int,(k)}_{0, E}  \|_{\calL \left( Y^+_{0, r_0} \to \tilde Y^+_{0, r_0} \right)} + \|\tilde T^{int,(k)}_{r_0; 0, E} \|_{\calL \left( Y^-_{0, r_0} \to \tilde Y^-_{0, r_0}\right) } \lesssim_{k} r_0; &
\end{array}\right.
 \label{eqTintbdd1} \\
& \left| \begin{array}{ll}
\| T^{int,(k)}_{\nu, E}  \|_{\calL \left( Y^+_{\nu, r_0} \to Y^+_{\nu+2, r_0} \right)} \lesssim_{\nu, k} 1, & \nu \ge 0 \\
  \|\tilde T^{int, (k)}_{r_0; \nu, E}  \|_{\calL \left( Y^-_{\nu, r_0} \to Y^-_{\nu - 2, r_0}\right) } + \|\tilde T^{int, (k)}_{r_0; 1, E}  \|_{\calL \left( Y^-_{1, r_0} \to Y^-_{ - \frac 12, r_0}\right) } \lesssim_{\nu, k} 1,& \nu > 1, \\
  \| T^{int,(k)}_{\nu, E} \|_{\calL \left( Y^-_{\nu, r_0} \to \tilde Y^-_{\nu - 2, r_0}\right) } + \| T^{int,(k)}_{0, E} \|_{\calL \left( Y^-_{0, r_0} \to \tilde Y^-_{- \frac 32, r_0}\right) } \lesssim_{\nu, k} 1, & \nu \in (0, 1).
 \end{array}\right.\label{eqTintbdd2}
\end{align}
where the Banach spaces for $\nu \in \RR$ are defined as 
\bee
  \|f \|_{Y^\pm_{\nu, r_0}} &:=& \left| \begin{array}{ll}
     \| f \|_{ C^0_{r^{\frac 12}|\ln r| }([0, r_0]) }  & (\pm, \nu) = (-, 0)  \\
       \| f \|_{ C^0_{r^{\pm \nu + \frac 12} }([0, r_0]) } & {\rm otherwise}
  \end{array}\right.\\
  \|f \|_{\tilde Y^\pm_{\nu, r_0}} &:=& \left| \begin{array}{ll}
     \| f \|_{ C^0_{r^{\frac 12}|\ln r| }([0, r_0]) } + \| f' \|_{C^0_{r^{-\frac 12}|\ln r|}([0, r_0]) }  & (\pm, \nu) = (-, 0)  \\
       \| f \|_{ C^0_{r^{\pm \nu + \frac 12} }} + \| f' \|_{ C^0_{r^{\pm \nu - \frac 12} }([0, r_0]) } & {\rm otherwise}
  \end{array}\right.
\eee
Besides, it is apparent from the Duhamel definition that $L_{\nu, E} T^{int, (0)}_{\nu, E}f = f$, $L_{\nu, E}\tilde  T^{int,(0)}_{r_0;\nu, E} g = g$ for regular $f, g$ in the above space respectively.

\mbox{}

\underline{2. Proof of (1), (2) and (3).}

(1) Now we rewrite \eqref{eqnu} as 
\bee
 \left| \begin{array}{l}
   L_{\nu, 1+\l} \Phi^1 = -\left( \frac{b^2 r^2}{4} + W_{1, b} - ibs_c \right)\Phi^1 -   e^{i\frac{br^2}{2}} W_{2,b} \Phi^2  \\
    L_{\nu, 1-\l} \Phi^2 = -\left( \frac{b^2 r^2}{4} + W_{1, b} + ibs_c \right)\Phi^2 -   e^{-i\frac{br^2}{2}} \overline{W_{2,b}} \Phi^1  
 \end{array}\right.
\eee
and invert $L_{\nu, E}$ to define first the normalized solutions $\tilde \Psi_j$ as 
\be
  \left| \begin{array}{l}
    \tilde \Psi^1_j = S^{int}_j - T^{int,(0)}_{\nu, 1+\l} \left[ \left( \frac {b^2 r^2}{4} + W_{1, b} -  ibs_c\right)\tilde  \Psi^1_j +  e^{i\frac{br^2}{2}} W_{2,b} \tilde  \Psi^2_j \right] \\
    \tilde  \Psi^2_j = R^{int}_j - T^{int,(0)}_{\nu, 1-\l} \left[ \left( \frac {b^2 r^2}{4} + W_{1, b} + ibs_c \right)\tilde  \Psi^2_j +  e^{-i\frac{br^2}{2}}\overline{W_{2,b}} \tilde \Psi^1_j \right]
  \end{array}\right.\label{eqIint1}
  \ee
  for $ j =1, 2$ with $\nu \ge 0$, or $j = 3, 4$ with $\nu \in [0, 1)$; and 
  \be
  \left| \begin{array}{l}
    \tilde \Psi^1_j = S^{int}_j - \tilde T^{int}_{r_0; \nu, 1+\l} \left[ \left( \frac {b^2 r^2}{4} + W_{1, b} -  ibs_c\right)\tilde \Psi^1_j +  e^{i\frac{br^2}{2}} W_{2,b} \tilde \Psi^2_j \right] \\
    \tilde \Psi^2_j = R^{int}_j - \tilde T^{int}_{r_0; \nu, 1-\l} \left[ \left( \frac {b^2 r^2}{4} + W_{1, b} +  ibs_c \right)\tilde \Psi^2_j +  e^{-i\frac{br^2}{2}}\overline{W_{2,b}} \tilde \Psi^1_j \right]
  \end{array}\right.\label{eqIint2}
\ee
for $j = 3, 4$ with $\nu \ge 1$. Here $r_0 \le 1$ is small enough to be determined, and the source terms are 
\be
  S^{int}_j = \left| \begin{array}{ll}
      \ti_{\nu, 1+\l} & j = 1, \\
      \tk_{\nu, 1+\l} & j = 3, \\
      0 & j = 2, 4,
  \end{array}\right. \quad 
  R^{int}_j = \left| \begin{array}{ll}
      \ti_{\nu, 1-\l} & j = 2, \\
      \tk_{\nu, 1-\l} & j = 4, \\
      0 & j = 1, 3.
  \end{array}\right. \label{eqdefSRint}
\ee
We also define the target fundamental solution families as 
\be \Psi_j = c_{j;\l,\nu}^{-1} \tilde \Psi_j,\quad {\rm for }\,\, j = 1, 2, 3, 4,\label{eqnormalPsi}
\ee
with normalizing coefficients
\be c_{j;\l,\nu} = \left| \begin{array}{ll}
   \frac{(1 - (-1)^j \l)^{\frac \nu 2 + \frac 14}}{2^\nu \Gamma(\nu + 1)}  &  j = 1, 2; \\
    c_{j-2;\l,\nu}^{-1} & j = 3, 4.
\end{array}\right. \label{eqc12blnu} \ee

% Recall that $\nu \in \frac 12 \NN$ as \eqref{eqdefnu} and $|E_\pm - 1| \le 2\delta_0 \ll 1$ from \eqref{eqdefEpm}. 

Recall from Proposition \ref{propQbasymp} that 
\be |W_{k,b}| \lesssim \la r \ra^{-\frac{(d-1)(p-1)}{2}} e^{-(p-1)r},\quad {\rm for\,\,}r \le b^{-\frac 12}, \label{eqpotbdd1} \ee
with the constant independent of $0 \le b \le b_{\rm int}$. 
The boundedness of $T^{int, (0)}_{\nu, E}$ and $\tilde T^{int}_{\nu, E}$ \eqref{eqTintbdd1} now implies that for $r_0(\nu) \ll 1$, the linear part in RHS of \eqref{eqIint1} and \eqref{eqIint2} are contraction mappings in the corresponding topology $Y^\pm_{\nu, r_0(\nu)}$. With the stronger boundedness in $\tilde Y^{\pm}_{\nu, r_0(\nu)}$, we obtain the existence of $\tilde \Psi_j$ satisfying
\bee
  \| \tilde \Psi_{j;b,\l,\nu}\|_{(\tilde Y^+_{\nu, r_0(\nu)})^2}
   + \| \tilde \Psi_{j+2;b,\l,\nu}\|_{(\tilde Y^-_{\nu, r_0(\nu)})^2} \lesssim_\nu 1,\quad j = 1, 2.
\eee
Now the asymptotics of $\ti_{\nu, 1+\l}$, $\tk_{\nu, 1-\l}$ given by Lemma \ref{lemmodBessel} (1) and \eqref{eqdeftildeIK}, the normalizing coefficients \eqref{eqc12blnu}, the smoothing estimate of Duhamel terms \eqref{eqTintbdd2} plus the above boundedness implies the asymptotics of $\Phi_j$ and $\pa_r \Phi_j$ \eqref{eqPsiasympest1}-\eqref{eqPsiasympest3}. The refinement \eqref{eqPsiasympest4} follows from the refined asymptotics of $K_{\frac 12}$ from the explicit formula \eqref{eqexplicitIK12}.

\mbox{}

(2) We first consider $j = 1, 2$. Differentiate \eqref{eqIint1} w.r.t. $\l$, 
\be
 \left| \begin{array}{l}
    \pa_\l^n \tilde \Psi^1_j =  \pa_\l^n  S^{int}_j - \sum_{j=0}^n \binom{n}{j} T^{int,(j)}_{\nu, 1+\l} \left[ \left( \frac {b^2 r^2}{4} + W_{1, b} -  ibs_c \right) \pa_\l^{n-j}\tilde \Psi^1_j +  e^{i\frac{br^2}{2}} W_{2,b}\pa_\l^{n-j} \tilde \Psi^2_j \right] \\
    \pa_\l^n\tilde \Psi^2_j =  \pa_\l^n  R^{int}_j - \sum_{j=0}^n (-1)^j \binom{n}{j} T^{int,(j)}_{\nu, 1-\l} \left[ \left( \frac {b^2 r^2}{4} + W_{1, b}+ ibs_c \right) \pa_\l^{n-j} \tilde \Psi^2_j +  e^{-i\frac{br^2}{2}}\overline{W_{2,b}} \pa_\l^{n-j}\tilde \Psi^1_j \right]
  \end{array}\right.
  \label{eqIint3}
\ee
Notice that the linear operator acting on $(\pa_\l^n \tilde \Psi^1_j, \pa_\l^n \tilde \Psi^2_j)$ are the same as $n = 0$ case, so by inverting the contraction part and applying the boundedness of $T^{int,(j)}_{\nu, E}$ \eqref{eqTintbdd1}, we can inductively show  
\be \| \pa_\l^n \tilde \Psi_{j;b,\l,\nu}\|_{\left(\tilde Y^+_{\nu, r_0(\nu)}\right)^2} \lesssim_{\nu, n} 1, \quad j = 1, 2,\,\,\nu \ge 0. 
\label{eqPsipalEest}
\ee
The vanishing of $T^{int,(k)}_{\nu, E}$ at $r = 0$ indicates  
\bee
 \pa_r^k \pa_\l^n \tilde \Psi_{j;b,\l,\nu} = \pa_r^k \pa_\l^n \left( \begin{array}{c}
      S^{int}_{j;b,\l,\nu} \\
      R^{int}_{j;b,\l,\nu} 
 \end{array} \right) + O(r^{\nu + \frac 52 - k}) 
 = \pa_r^k \pa_\l^n \left( c_{j;\l,\nu} r^{\nu + \frac 12} \right)\vec e_j + O(r^{\nu + \frac 52 - k})
\eee
where the second equivalence used $\pa_z^n I_\nu(z) = \frac{\nu (\nu - 1)\cdots (\nu - n + 1)}{z^n} I_\nu(z) ( 1+ O(z^2))$ as $z \to 0$ from \eqref{eqBesselderiv}, namely the leading order behaves in the same way as $\frac{z^{\nu}}{2^\nu \Gamma(\nu+ 1)}$ when taking derivatives. This implies \eqref{eqPsipalest} by Leibniz rule.

Moreover, since the potentials are bounded in any finite interval away from $r = 0$, all four solutions $\Psi_j$ can be uniquely extended to $r \in (0,\infty)$,  and the solution still satisfies \eqref{eqIint1}, \eqref{eqIint2} and hence \eqref{eqIint3}. In particular, on $[0, b_{\rm int}^{-\frac 12}]$, the potentials are uniformly bounded w.r.t. $b, \l$ for fixed $\nu$, together with the uniformly bounded initial data at $r = r_0(\nu)$ \eqref{eqPsipalEest}, we have uniform boundedness of $\Psi_{j;b,\l,\nu}$ on $[r_0(\nu), b_{\rm int}^{-\frac 12}]$
\be \| \pa_\l^n \Psi_{j;b,\l,\nu}\|_{\left(C^0([r_0(\nu), b_{\rm int}^{-\frac 12}])\right)^2} \lesssim_{\nu, n, b_{\rm int}} 1, \quad {\rm for\,\,} j = 1, 2,\,\,\nu \ge 0,\,\,n \ge 0. 
\label{eqPsipalEest2}
\ee

Next, for $j = 3, 4$, the existence and boundedness of $\pa_\l^n \tilde \Psi_{j;b,\l,\nu}$ in $\tilde Y^-_{\nu, r_0(\nu)}$ for any $n \ge 0$ and extension to $r \in (0,\infty)$ follow the same argument by applying boundedness of $\tilde T^{int, (k)}_{r_0(\nu);\nu,E}$ \eqref{eqTintbdd1}. In particular, $\Psi_{j;b,\l,\nu}$ are analytic w.r.t. $\l$ for $j = 1, 2, 3, 4$. 

Finally, differentiate \eqref{eqnu} implies $(\HH_{b,\nu} - \l)\pa_\l^n \Psi_{j;b,\l,\nu} = n\pa_\l^{n-1}\Psi_{j;b,\l,\nu}$ for $n \ge 1$ and $1 \le j \le 4$. Hence the linear independence of $\{ \pa_\l^n \Psi_{j;b,\l,\nu} \}_{n\ge 0, 1 \le j \le 4}$ follows that of $\{ \Psi_{j;b,\l,\nu} \}_{1 \le j \le 4}$.

\mbox{}

(3) By the analyticity of $c_{j;\l,\nu}$ \eqref{eqc12blnu}, it suffices to prove \eqref{eqPsicontb} for the normalized solution $\tilde \Psi_{j;b,\l,\nu}$, $1 \le j \le 4$ as \eqref{eqnormalPsi}. We consider for instance $j = 1$. Take difference of \eqref{eqIint3} and the $b=0$ case,
\be
\left| \begin{array}{l}
    \pa_\l^n \left(\tilde \Psi^1_{1;b} - \tilde \Psi^1_{1;0}\right) =  - \sum_{j=0}^n \binom{n}{j} T^{int,(j)}_{\nu, 1+\l} \bigg[ W_{1} \pa_\l^{n-j}\left( \tilde \Psi^1_{1;b} - \tilde \Psi^1_{1;0} \right)  + W_2 \pa_\l^{n-j} \left( \tilde \Psi^2_{1;b} - \tilde \Psi^2_{1;0} \right) \\
    \qquad \qquad + \left( \frac{b^2r^2}{4}  + W_{1, b} - W_1 - ibs_c \right)\tilde \Psi^1_{1;b}
    +  \left(e^{i\frac{br^2}{2}} W_{2,b} - W_2\right) \tilde \Psi^2_{1;b} \bigg] \\
    \pa_\l^n \left(\tilde \Psi^2_{1;b} - \tilde \Psi^2_{1;0}\right) = - \sum_{j=0}^n (-1)^j \binom{n}{j} T^{int,(j)}_{\nu, 1-\l} \bigg[ W_{1} \pa_\l^{n-j}\left( \tilde \Psi^2_{1;b} - \tilde \Psi^2_{1;0} \right)  + W_2 \pa_\l^{n-j} \left( \tilde \Psi^1_{1;b} - \tilde \Psi^1_{1;0} \right) \\
    \qquad \qquad + \left( \frac{b^2r^2}{4}  + W_{1, b} - W_1 + ibs_c \right)\tilde \Psi^2_{1;b}
    +  \left(e^{-i\frac{br^2}{2}} \overline{W_{2,b}} - W_2\right) \tilde \Psi^1_{1;b} \bigg]
  \end{array}\right.
  \label{eqIint4}
\ee
Recall smallness from \eqref{eqQbasymp5}-\eqref{eqQbasymp6} that for $0 < b \le b_{\rm int}$,
\be
  \sup_{r \in[0,b^{-\frac 12}]} \left[ \left| \frac{b^2r^2}{4}  + W_{1, b} - W_1 \pm ibs_c \right|  + \left|e^{i\frac{br^2}{2}} W_{2,b} - W_2\right| \right]  \le b^{\frac 16} C_{b_{\rm int}}.  \label{eqintpotest}
\ee
Then on $[0, r_0(\nu)]$, similar to the proof of \eqref{eqPsipalEest}, we invert the contraction linear operator on $\pa_\l^n \left(\tilde \Psi^k_{1;b} - \tilde \Psi^k_{1;0}\right)$, apply \eqref{eqintpotest} and \eqref{eqPsipalEest}, and induct on $n\ge 0$ to see
\be
 \left \| \pa_\l^n  \left( \tilde \Psi_{j;b,\l,\nu} - \tilde \Psi_{j;0,\l,\nu} \right) \right\|_{\left(\tilde Y^+_{\nu, r_0(\nu)}\right)^2} \lesssim_{\nu, n, b_{\rm int}} b^\frac 16, \quad j = 1, 2,\,\,\nu \ge 0. 
\label{eqPsipalEest3}
\ee
On $[r_0(\nu), r_*]$ with $b \le \min \{ b_{\rm int}, r_*^{-2}\}$, we take the absolute value in \eqref{eqIint4}, and brutally estimate using the $C^0$-boundedness of $\pa_E^n \ti_{\nu, 1+\l}$, $\pa_E^n \tk_{\nu, 1+\l}$ on this region and \eqref{eqintpotest}
\bee
 \left| \pa_\l^n \left(\tilde \Psi_{1;b} - \tilde \Psi_{1;0}\right) \right|
 \le C_{\nu, n, b_{\rm int}} \int_{r_0(\nu)}^r \sum_{m=0}^n \left[ \left| \pa_\l^m \left(\tilde \Psi_{1;b} - \tilde \Psi_{1;0}\right) \right| + b^\frac 16 C_{b_{\rm int}}   \left| \pa_\l^m \tilde \Psi_{1;b} \right| \right] ds
\eee
Now summing up the estimate for $0 \le n \le N$ and invoking \eqref{eqPsipalEest}, \eqref{eqPsipalEest2} and \eqref{eqPsipalEest3}, we have 
\bee
  \sum_{n=0}^N \left| \pa_\l^n \left(\tilde \Psi_{1;b} - \tilde \Psi_{1;0}\right) \right|
 \lesssim_{\nu, N, b_{\rm int}} b^\frac 16 + \int_{r_0(\nu)}^r  \sum_{n=0}^N \left| \pa_\l^n \left(\tilde \Psi_{1;b} - \tilde \Psi_{1;0}\right) \right| ds,
\eee
which inductively yields 
\be
  \| \pa_\l^N \left(\tilde \Psi_{1;b} - \tilde \Psi_{1;0}\right) \|_{\left( C^0([r_0(\nu), r_*]) \right)^2} \lesssim_{\nu, N, b_{\rm int}} b^{\frac 16}, \quad j = 1, 2, \,\,\nu \ge 0,\,\,N \ge 0 \label{eqPsipalEest4}
\ee
via Gronwall's inequality. The $C^1$-control follows taking $\pa_r$ in \eqref{eqIint4}, and use the above $C^0$-control with the structure \eqref{eqintegrandinv}. This $C^1$-control and \eqref{eqPsipalEest4} imply \eqref{eqPsicontb}. 

\mbox{}

\underline{3. Proof of (4).}

Let $x_* = b^{-\frac 12}$. Notice that $|\l| \le \delta_{\rm int} \le 100^{-1}$ indicates $|\arg (1\pm\l)| \le \frac{\pi}{16}$. We also assume $\nu \ge 3$ in this part. 

Define $\check \Psi_{j;b,\l,\nu}$ for $j = 1, 2$
as 
\be
  \left| \begin{array}{l}
    \check \Psi^1_j = S^{int}_j - \tilde T^{int}_{x_*;\nu, 1+\l} \left[ \left( \frac {b^2 r^2}{4} + W_{1, b} -  ibs_c\right)\check  \Psi^1_j +  e^{i\frac{br^2}{2}} W_{2,b} \check  \Psi^2_j \right] \\
    \check  \Psi^2_j = R^{int}_j - \tilde T^{int}_{x_*;\nu, 1-\l} \left[ \left( \frac {b^2 r^2}{4} + W_{1, b} + ibs_c \right)\check  \Psi^2_j +  e^{-i\frac{br^2}{2}}\overline{W_{2,b}} \check \Psi^1_j \right],
  \end{array}\right.\label{eqIint5} 
  \ee
  where the source terms are from \eqref{eqdefSRint}.
We will show these linear maps for $j= 1, 2$ are
linear contractions in respectively 
\be 
\XX^{\rm int}_j := \left| \begin{array}{ll}
    C^0_{\ti_{\nu, 1+\l}}([0, x_*]) \times (\frac{1+\l}{1-\l})^{\frac \nu 2} C^0_{\ti_{\nu, 1-\l} \cdot e^{-5\delta_{\rm int} r}}([0, x_*])  & j = 1; \\
    (\frac{1-\l}{1+\l})^{\frac \nu 2} C^0_{\ti_{\nu, 1+\l} \cdot e^{-5\delta_{\rm int} r}}([0, x_*]) \times C^0_{\ti_{\nu, 1-\l}}([0, x_*])  & j = 2;
\end{array}\right. \label{eqdefXXintj}
\ee 
 with $\nu \ge \nu_{\rm int}$ large enough and $b \le b_{\rm int}$ small enough. For simplicity, we only prove the $j = 1$ case, namely \eqref{eqinthighnu1}.

Firstly, with monotonic increasing of $I_{\nu, 1+\l}$ (Lemma \ref{lemmodBessel} (5)), \eqref{eqIKproduct2} and \eqref{eqpotbdd1}, we estimate
\bee
   && \left| \tilde T^{int}_{x_*;\nu, 1+\l} \left[ \left( \frac {b^2 r^2}{4} + W_{1, b} -  ibs_c\right)\check \Psi^1_1 \right]\right|/  \| \check \Psi^1_1 \|_{C^0_{\ti_{\nu, 1+\l}}([0, b^{-\frac 12}]) } \\
   &\lesssim& |\ti_{\nu, 1+\l}(r)|  \int_r^{x_*} \left| \tk_{\nu, 1+\l}(s) \ti_{\nu, 1+\l}(s) \cdot \left( b^2 \la s\ra^2 + e^{-(p-1)s} \right)  \right| ds  \\
   &+& |\tk_{\nu, 1+\l}(r)| \int_0^r \left| \ti_{\nu, 1+\l}(s)^2 \left( b^2 \la s\ra^2 + e^{-(p-1)s} \right)\right| ds \\
   &\lesssim&  |\ti_{\nu, 1+\l}(r)|  \left(  \int_r^{x_*} (b^2 \la s\ra^2 + \frac {se^{-(p-1)s}}{\nu} ) ds + \frac 1r \int_0^r b^2 \la s\ra^2 s ds + \frac{1}{\nu} \int_0^r s e^{-(p-1)s} ds \right)  \\
   &\lesssim& (b^\frac 12 + \nu^{-1})|\ti_{\nu, 1+\l}(r)|.
   \eee
   Notice that $|I_{\nu, E_1} I_{\nu, E_2}| e^{-c\delta_{\rm int} r}$ with $0 \le c \le 50$ is monotonically increasing since $\delta_{\rm int} \le 100^{-1}$ and $\pa_r |I_{\nu, E}(r)|\ge \frac 12  |I_{\nu, E}(r)|$ from a similar computation as \eqref{eqmonoInu} using \eqref{eqTuranI}. We can similarly obtain
   \[  \left\| \tilde T^{int}_{x_*;\nu, 1-\l} \circ \left( \frac {b^2 r^2}{4} + W_{1, b} +  ibs_c\right)\right\|_{\calL\left(C^0_{\ti_{\nu, 1-\l} e^{-5\delta_{\rm int} r}}([0, b^{-\frac 12}])\right) }  \lesssim b^\frac 12 + \nu^{-1}.\]

For the crossing term, we recall \eqref{eqIratio2} to define 
\be \tilde \eta_{\nu}(r) := \eta_\nu(\sqrt{1+\l} r, \sqrt{1-\l}r) \cdot (\sqrt{1+\l} - \sqrt{1-\l}).
\ee
The bound $|\eta_{\nu}| \le 2$ yields
\be 
|\tilde \eta_{\nu}(r)| \le 4|\l| \le 4\delta_{\rm int} \le \frac{p-1}{4}. \label{eqesttildeetanu}
\ee
Therefore,
\bee
   &&\left| T^{int}_{x_*;\nu, 1-\l} \left( e^{-i\frac{br^2}{2}}\overline{W_{2, b}} \check \Psi^1_1 \right) \right| / \| \check \Psi^1_1 \|_{C^0_{\ti_{\nu, 1+\l}}([0, b^{-\frac 12}])}\\
   &\lesssim& |\ti_{\nu, 1-\l}| \left[ \int_r^{x_*} \left| \tk_{\nu, 1-\l} \ti_{\nu, 1+\l} \cdot e^{-(p-1)s}\right| ds  +  |K_{\nu, 1-\l} I_{\nu, 1+\l} e^{-10\delta_{\rm int}r}| \int_0^r s e^{-\frac{p-1}{2}s}ds\right] \\
   &\lesssim& |\ti_{\nu, 1-\l}| \bigg[ \int_r^{x_*}  \left( \frac{\sqrt{ 1+\l}s }{\sqrt{ 1-\l} s}\right)^{\nu + \frac 12}\frac{s}{\nu}  e^{-(p-1 - \tilde \eta_{\nu}(s)) s} ds  \\
   &&\qquad + \left( \frac{\sqrt{ 1+\l}r }{\sqrt{ 1-\l} r}\right)^{\nu + \frac 12} e^{-(10\delta_{\rm int}- \tilde \eta_{\nu}(r))r}\frac 1\nu \int_0^r s e^{-\frac{p-1}{2}s}ds \bigg] \\
   &\lesssim& \nu^{-1} \left( \frac{ 1+\l}{ 1-\l}\right)^{\frac \nu 2} \cdot e^{-5\delta_{\rm int} r} |\ti_{\nu, 1-\l}(r)|.
   % &\lesssim& |\ti_{\nu, 1+\l}| \int_r^{x_*} \left| \tk_{\nu, 1+\l} \ti_{\nu, 1+\l} \left( \frac{\sqrt{ 1-\l}s }{\sqrt{ 1+\l} s}\right)^{\nu + \frac 12} e^{O(\l)s} e^{-(p-1)s} \right| ds \\
   % &+& |\tk_{\nu, 1-\l}(r)| \cdot |\ti_{\nu, 1-\l}(r) \ti_{\nu, 1+\l}(r)| e^{-10\delta_{\rm int} r}\int_0^r e^{-\frac{p-1}{2}s} ds\\
   % &\lesssim& (b^\frac 12 + \nu^{-1})|\ti_{\nu, 1+\l}(r)|.
\eee
We stress that the constants above are dependent on $d$ (for boundedness of $W_{j,b}$) but independent of $\l, \nu, b$. Lastly, the crossing term in the first equation of \eqref{eqIint5} is easier to bound which we omit. 

Consequently, \eqref{eqIint5} defines a contraction mapping in $\XX^{\rm int}_1$ as \eqref{eqdefXXintj}. Since \eqref{eqIratio2} and \eqref{eqesttildeetanu} indicates that $\left( \frac{1+\l}{1-\l}\right)^{\frac \nu 2} \ti_{\nu, 1-\l}(r) e^{-5\delta_{\rm int}r} \lesssim \ti_{\nu, 1+\l}(r)$, the asymptotics of $\Phi_{1;b,\l,\nu}$ in \eqref{eqinthighnu1} follows when $b_{\rm int}$ small enough and $\nu_{\rm int}$ large enough. The asymptotics of derivative follows the Duhamel structure of \eqref{eqIint5} and Tur\'an inequality \eqref{eqTuranI} so that for $\tilde{\mathscr{Z}} = \tilde I$ or $\tilde K$, 
\[  
\pa_r \ti_{\nu, E}(r) = \left[ \pm (2r)^{-1} + \sqrt{E + \frac{\nu^2}{r^2} - O(\nu^{-1})  } \right] \ti_{\nu, E}(r) = \sqrt{E + \frac{\nu^2}{r^2}} \ti_{\nu, E}(r) (1 + O(\nu^{-1})),
\]
and similarly for $\tk_{\nu, E}$. That conclude the proof of $j = 1$ case \eqref{eqinthighnu1}, and the $j = 2$ case \eqref{eqinthighnu2} can be obtained in the same way. 

Finally, $\check \Psi_{j;b,\l,\nu} \subset  {\rm span}\{\Psi_{j;b,\l,\nu}\}_{j = 1}^2$ thanks to the singular, linear independent asymptotics of $\Psi_{3;b,\l,\nu}, \Psi_{4;b,\l,\nu}$ as $r \to 0$ \eqref{eqPsiasympest2}. The boundary values \eqref{eqinthighnu1}-\eqref{eqinthighnu2} indicate their linear independence, leading to \eqref{eqlindepcheckPsi}.  
\end{proof}

% \subsection{System case: fundamental solutions}

\section{Exterior ODE analysis: Scalar case}\label{sec4}

\subsection{Low spherical class}\label{sec41}
In this subsection, we consider the scalar equation 
\be \left( \pa_r^2 - E + \frac{b^2 r^2}{4}\right) \psi = 0, \label{eqscalar} \ee
with $|E - 1| \ll 1$. It serves as approximation of \eqref{eqnu} when $r \gg 1$ so that we can drop the potentials from $Q_b$ and angular momentum when $\nu$ is not too large. We will construct approximate fundamental solutions for \eqref{eqscalar} using Airy function with $\CC$-variable, which presents a refined version of the WKB analysis in Perelman's work \cite[Appendix 4]{MR1852922}. Then we can construct the inversion operator in suitable function spaces, and prove refined asymptotics of $Q_b$ in Proposition \ref{propQbasympref} as an application. 

\subsubsection{WKB approximate solution}

Recall the Airy function $\Ai: \CC \to \CC$ is an entire function defined by
\be \Ai(z) = \frac{\sqrt{3}}{2\pi} \int_0^\infty \exp\left(-\frac{t^3}{3} - \frac{z^3}{3t^3} \right)dt. \label{eqdefAi} \ee
We also define the other Airy function $\mathbf{Bi}$ as \cite[9.2.10]{MR2723248}
\be \mathbf{Bi}(z) = e^{-\frac{\pi i}{6}} \Ai \left(z e^{-\frac{2 \pi i}{3}} \right) + e^{\frac{\pi i}{6}} \Ai \left(z e^{\frac{2 \pi i}{3}} \right).  \label{eqdefBi} \ee
\begin{lemma}[Properties of Airy function] \label{lemAiry1} Let $\zeta = \frac 23 z^\frac 32$.
\begin{enumerate}
    \item Asymptotics: for $|z| \ge 1$, 
    \bee
      \Ai(z) &=& \frac{e^{-\zeta}}{2\sqrt{\pi}z^{\frac 14}}\left(1 + O(\zeta^{-1})\right),\quad |\arg z| \le \frac 56\pi; \\
      \Ai(-z) &=& \frac{1}{\sqrt{\pi} z^{\frac 14}} \left( \cos(\zeta - \frac 14 \pi) + \frac{5}{72 \zeta}\sin(\zeta - \frac 14\pi) \right)\left(1 + O(\zeta^{-1})\right),\quad |\arg z| \le \frac 13\pi; \\
    \Ai'(z) &=& - \frac{z^{\frac 14}e^{-\zeta}}{2\sqrt{\pi}}\left(1 + O(\zeta^{-1})\right),\quad |\arg z| \le \frac 56\pi; \\
      \Ai'(-z) &=& \frac{z^{\frac 14}}{\sqrt{\pi}} \left( \sin(\zeta - \frac 14 \pi) - \frac{5}{72 \zeta}\cos(\zeta - \frac 14\pi) \right)\left(1 + O(\zeta^{-1})\right),\quad |\arg z| \le \frac 13\pi.
    \eee
    \item Solutions of Airy equation, connection formula and Wronskian: $\Ai(z)$, $\Ai(e^{\mp \frac 23\pi i}z)$ forms the $2$-dim solution space of Airy equation $\frac{d^2}{dz^2}w = zw$. They have the following relations:
    \bee
      \Ai(z) + e^{-\frac{2\pi i}{3}}\Ai(e^{- \frac 23\pi i}z) + e^{\frac{2\pi i}{3}} \Ai(e^{\frac 23\pi i}z) &=& 0 \\
      \mathcal{W}(\Ai, \Ai(\cdot e^{\mp \frac{2\pi i}{3}})) = \frac{e^{\pm i\pi/6}}{2\pi},\quad \mathcal{W}(\Ai(\cdot e^{-\frac{2\pi i}{3}}) ,  \Ai(\cdot e^{ \frac{2\pi i}{3}})) &=& \frac{1}{2\pi i}.
    \eee
    \item Identity of higher derivatives: there exists polynomial $P_n(z)$, $Q_n(z)$ of degree
    \[ \deg P_n = \left| \begin{array}{ll}
        k & n = 2k, \\
        k-1 & n = 2k+1,
    \end{array}\right.\quad 
    \deg Q_n = \left| \begin{array}{ll}
        k-2 & n = 2k, \\
        k & n = 2k+1,
    \end{array}\right.\quad \forall\, n \ge 0,\]
    such that 
    \be \pa_z^n \Ai(z) = P_n(z) \Ai(z) + Q_n(z) \Ai'(z),\quad \forall\,n\ge 0.  \label{eqAiryderivident} \ee
     \item Refined estimate for higher derivatives: for $|z| \ge 1$, $|\arg z| \le \frac 56\pi$ and for all $n \ge 0$, we have
    \bea \left| \left( \pa_z + z^\frac 12\right)^n \pa_z^k \Ai(z) \right| \lesssim_n |z|^{-n} |\pa_z^k \Ai(z)|,\quad k = 0, 1;    \label{eqAiryderiv} \\
    \left| \pa_z^n \left( 2\sqrt{\pi} z^{\frac 14} e^{\frac 23 z^\frac 32} \Ai(z) \right) - \delta_{n, 0} \right| \lesssim_n | z |^{-\frac 32 - n} . \label{eqAiryderiv2} \eea
\end{enumerate}
\end{lemma}
\begin{proof}
    (1) and (2) are standard in \cite[Chapter 9]{MR2723248}. For (3), the equation $\Ai''(z) = z\Ai(z)$, and initially $P_0= 1$, $Q_0 = 0$ imply \eqref{eqAiryderivident} with the recurrence formula 
    \[ P_{n+1} = P_n' + z Q_n,\quad Q_{n+1} = Q_n' + P_n,\quad \forall\, n \ge 0. \]
    The degree of $P_n$, $Q_n$ follows from induction.
    
    For (4), we first recall the representation of $\Ai$ by confluent hypergeometric function (Kummer's function) \cite[9.6.21]{MR2723248}
    \be\Ai(z) = 3^{-\frac 16} \pi^{-\frac 12} \zeta^{\frac 23} e^{-\zeta} U\left(\frac 56, \frac 53, 2\zeta\right) \label{eqAiKummer1}\ee
    and the properties of $U$ \cite[13.2.6, 13.7.3]{MR2723248}
    \be
    \begin{split} 
    &|U(a, b,z) - z^{-a}| \lesssim_{a, b} z^{-a-1}, \qquad |z| \ge 1,\,\, |{\rm arg} z| \le \frac 54\pi;\\
    &\pa_z^n U(a, b,z) = (-1)^n \left[\prod_{k=0}^{n-1} (a+k)\right] U(a+n, b+n, z),\qquad n \ge 1.  \label{eqAiKummer3}
    \end{split}
    \ee

    Noticing that $\pa_z + z^{\frac 12} = e^{-\zeta} \pa_z e^\zeta$, 
    we get 
    \[ \left| \left(e^{-\zeta} \pa_\zeta e^{\zeta} \right)^n \Ai(z) \right| \lesssim_n |\zeta|^{-n} |\Ai(z)|, \]
    namely
    \[ \left| \pa_\zeta^n \left( e^\zeta \Ai (z) \right)\right| \lesssim_n |z|^{-\frac 32 n} |e^{\zeta} \Ai(z)|.  \]
    Recall the Fa\'a di Bruno's formula (chain-rule for higher derivatives):
    \be \pa_z^n F(\zeta(z)) = \sum_{\sum_{k=1}^n km_k = n } \frac{n!}{m_1! m_2! \cdots m_n!} \cdot \pa_\zeta^{m_1+\cdots +m_n} F (\zeta(z)) \cdot \prod_{j=1}^n \left(\pa_z^j \zeta(z) \right)^{m_j}. \label{eqFaadiBruno}\ee
    The estimate \eqref{eqAiryderiv} for $k = 0$ follows by taking $F(\zeta) = e^\zeta \Ai(z(\zeta)) = 3^{-\frac 16} \pi^{-\frac 12} \zeta^{\frac 23} U \left( \frac 56, \frac 53, 2\zeta \right)$ and using $|\pa_z^m \zeta| \lesssim |z|^{\frac 32 - m}$. The case $ k = 1$ comes similarly using \cite[9.6.22]{MR2723248} 
    \be \Ai'(z) = -3^{-\frac 16} \pi^{-\frac 12} \zeta^{\frac 43} e^{-\zeta} U\left(\frac 76, \frac 73, 2\zeta\right). \label{eqAiKummer2}\ee
    
    For \eqref{eqAiryderiv2}, using Kummer's function $U$, we rewrite 
    $2\sqrt{\pi} z^{\frac 14} e^{\frac 23 z^\frac 32} \Ai(z) = (2\zeta)^\frac 56 U\left(\frac 76, \frac 73, 2\zeta\right)$. From the derivative formula, we know $\left|\pa_z^n U(a, b, z) - \pa_z^n z^{-a}\right| \lesssim_{a, b, n} |z|^{-a-n-1}$. So with this cancellation at leading order, a similar computation with \eqref{eqFaadiBruno} implies \eqref{eqAiryderiv2}.
\end{proof}
%    \end{enumerate}
% \end{lemma}
% Now we derive WKB solution of \eqref{eqscalar} and its correction term. 
%\mbox{}
Now we are in place to define the WKB approximate solutions to \eqref{eqscalar}.

\begin{definition}[WKB approximate solutions]\label{defWKBappsolu}
  Let $b > 0$ and $|E-1| \le \frac 12$. Define the complex arguments $s$ and $\zeta$ by
  \bea  s &=& \frac{br}{2\sqrt{E}}, \label{eqsdef} \\
   \zeta \left(\frac{d\zeta}{ds}\right)^2 &=& s^2 - 1,\quad \zeta(1) = 0, \label{eqzetadef}
  \eea
  and the parameter 
  \be \mu = e^{i\frac{\pi}{4}}\left(\frac{2E}{b} \right)^{\frac 12}. \label{eqmudef}\ee
Then we define the WKB approximate solutions to be 
\be \begin{split} 
\psi_1^{b, E} = (\pa_s \zeta)^{-\frac 12}  \Ai(e^{-\frac{2\pi i}{3}}\mu^\frac 43\zeta),&\quad 
\psi_2^{b, E} = (\pa_s \zeta)^{-\frac 12} \Ai(e^{ \frac{2\pi i}{3}}\mu^\frac 43 \zeta), \\
\psi_3^{b, E} = (\pa_s \zeta)^{-\frac 12} \Ai(\mu^\frac 43 \zeta),&\quad 
\psi_4^{b, E} = \frac 12 (\pa_s \zeta)^{-\frac 12} \mathbf{Bi}(e^{ \frac{2\pi i}{3}}\mu^\frac 43 \zeta)
\end{split} \label{eqpsidef} \ee
and the correction function\footnote{The second part of the formula is a brute force computation using \eqref{eqzetadef}, similar to \eqref{eqPsiF} below.}
\be h_{b, E}(r) = -\frac{b^2}{4E} \zeta_s^{\frac 32} \pa_\zeta^2 (\zeta_s^\frac 12) = -\frac{b^2}{4E} \left[ \frac{5(s^2 - 1)}{16\zeta^3} - \frac{3s^2 + 2} {4(s^2 - 1)^2}\right ]. \label{eqcorrection}\ee
\end{definition}

\begin{remark}\mbox{}\label{rmkWKB}
\begin{enumerate}
\item (Well-definedness) We will prove below in Lemma \ref{lemzeta32} that \eqref{eqzetadef} determines $\zeta = \zeta(s)$ as an analytic function on $\CC - (\infty, -1]$ and $\pa_s \zeta(s) \neq 0$. Therefore $\psi_j^{b, E}$ are well-defined $C^\infty_{loc}(\RR_+ \to \CC)$ functions. Here thanks to \eqref{eqargzetas}, the branch of $\zeta_s^{-\frac 12}$ is chosen to be the analytic extension of 
\be s \mapsto \zeta_s^{-\frac 12}: (0, \infty) \mapsto (0, \infty). \label{eqchoicebranchzetas} \ee
% {\color{purple} which is equivalent to taking
% \be \arg(s^2 -1), \arg \zeta \in (\a -\pi, \a + \pi]\quad {\rm for\,\,} s \in \CC (-\infty, -1], \label{eqbranch14}\ee
% with arbitrary $\a \in \RR$, 
 % since these two formulations coincide on $s \in (1, \infty)$. }
\item (Derivation) The idea comes from \cite{MR109898}. Under the variable $s$,  \eqref{eqscalar} becomes
\[\left[  \pa_s^2 + \frac{4E^2}{b^2} (s^2 - 1) \right] \psi = 0.  \]
With a further change of argument to $\zeta$ and variable to $\Upsilon$ as
\[ \Upsilon = \zeta_s^{\frac 12} \psi,  \]
the equation \eqref{eqscalar} turns into 
\be \pa_\zeta^2 \Upsilon - \mu^4 \zeta \Upsilon = F(\zeta)\Upsilon \label{eqPsi}\ee
where a direct computation gives the correction potential
\be F(\zeta) = \zeta_s^{-\frac 12} \pa_\zeta^2 (\zeta_s^\frac 12) = \frac{5}{16\zeta^2} - \frac{3s^2 + 2 }{4(s^2 - 1)^3} \zeta. \label{eqPsiF}\ee
So $\zeta_s^\frac 12 \psi_j^{b, E}$ for $j = 1, 2, 3, 4$ and their linear superposition will solve the homogeneous equation of \eqref{eqPsi}.
\end{enumerate}
\end{remark}

We first show the analyticity and asymptotics of $\zeta$.
\begin{lemma}\label{lemzeta32}
Let $\zeta$ defined as in \eqref{eqzetadef} and $s \in \CC$. Then $\zeta = \zeta(s)$ is analytic in $\CC - (-\infty, -1]$ and satisfies 
% on $\{ s: \Re s > -1 \} - [-1, 1] $
        \be \frac 23 \zeta^\frac 32 = \int_1^s (w^2 - 1)^\frac 12 dw 
        % = \frac 12 s(s^2 - 1)^\frac 12 - \frac 12 \ln \left(s + (s^2 - 1)^\frac 12\right)
\label{eqetaident} 
\ee
where the integrand takes the branch
\be w \mapsto (w^2 - 1)^\frac 12: \left| \begin{array}{ll} \CC - (-\infty, 1] \to  \CC - (-\infty, 0], & {\rm where\,\,}\arg (w\pm 1) \in (-\pi, \pi),  \\
(-1, 1] \to i\RR_{\ge 0},& {\rm where\,\,} \arg (w-1) = \pi, \,\,\arg(w+1) = 0
\end{array} \right.
% \quad \arg \left( w^2 - 1 \right)\in (-\pi, \pi),
\label{eqbranchsqrt}\ee
with $\arg (s-1) \in (-\pi, \pi]$, 
and the fractional power of LHS takes $\arg \zeta \in (-\pi, \pi]$. 
% \[ (\rho e^{i\theta})^\frac 32 := \rho^\frac 32 e^{i\frac 32 \theta},\quad \theta \in (-\pi, \pi]. \]
Moreover, the following properties hold
\begin{enumerate}
    \item  $\zeta$ has only one zero at $s = 1$ and $\zeta_s$ is non-vanishing in $\CC - (-\infty, -1]$.
    Moreover, we have 
    \bea
    \sgn(\Im \zeta(s)) = \sgn (\Im s),&& {\rm when}\,\, \Re s > 0; \label{eqImzetasgn} \\
    \sgn(\zeta(s)) = \sgn(s-1),\quad \zeta_s (s) > 0,&& {\rm when} \,\, s \in (-1, \infty).
      \label{eqargzetas} 
    \eea
    % and $\zeta_s(s) > 0$  for $s \in (-1, \infty)$. 
    % \item {\color{purple} There exists $\delta > 0$ such that for $s \in \{s \in \CC: |s + 1| \le \delta \} - (-\infty, -1]$, 
    % \be |\zeta| \sim 1,\quad |\zeta_s| \sim |s+1|^\frac 12, \quad |\zeta_{ss}| \sim |s+1|^{-\frac 12}. \label{eqzetaasymp0}  \ee}
    \item For $s \in \CC - (-\infty, 1]$, the integration \eqref{eqetaident} can be computed as
    \be \label{eqetaident2}
    \frac 23 \zeta^\frac 32 = \frac 12 s(s^2 - 1)^\frac 12 - \frac 12 \ln \left(s + (s^2 - 1)^\frac 12\right).
    \ee
    Here the square root of $s^2 - 1$ takes the branch as \eqref{eqbranchsqrt}, and the logarithm function takes the principal branch $\ln: \{|z|>0, \arg z \in (-\pi, \pi)\} \to \RR + \left(-\pi,  \pi \right)$.
    \item  Asymptotics in the cone near $\RR$: for $|\arg s| \le \frac \pi 4$, we have
        \be 
        \zeta(s)= \begin{cases}-\left(\frac{3 \pi}{8}\right)^{\frac{2}{3}}\left(1-\frac{8}{3 \pi} s+O\left(s^{2}\right)\right) & |s| \leq \frac{1}{2},\\ 2^{\frac{1}{3}}(s-1)(1+O(s-1)) & \frac{1}{2}<|s| \leq \frac{3}{2}, \\ \left(\frac{3}{4}\right)^{\frac{2}{3}} s^{\frac{4}{3}}\left(1+O\left(s^{-2} \ln s\right)\right) & |s|>\frac{3}{2},\end{cases} \label{eqzetaasymp}
        \ee
\end{enumerate} 
    \end{lemma}
\begin{proof} The formula \eqref{eqetaident} follows \eqref{eqzetadef}, and indicates that $\zeta^\frac 32(s)$ has only two branching point $s = \pm 1$ in $\CC$. Note that $w \mapsto (w^2 - 1)^\frac 12 = (w-1)^\frac 12 \cdot (w+1)^\frac 12$ maps $\CC - (-\infty, 1]$ analytically to $\CC - (-\infty, 0]$ with branch taken as \eqref{eqbranchsqrt}. After integraion, we have $\zeta^\frac 32 \neq 0$ and $|\arg \zeta^{\frac 32}| < 2 \pi$ for $s\in \CC - (-\infty, 1]$.

Thus we can define $\zeta: \CC - (-\infty, 1] \to \CC - \{ 0 \}$ analyticially by choosing the branch of $z^\frac 23$ correspondingly. To prove the analyticity of $\zeta$ in $\CC - (-\infty, -1]$, it suffices to check that $\zeta$ is analytic near $s=1$ for analytic extension. Indeed, we have $\zeta'(1) = 2^\frac 13$ since 
\be
  \zeta = \left( \frac 32 \int_1^s (w^2 - 1)^\frac 12 dw\right)^\frac 23 = \left( \frac 32 \int_0^{s-1} h^\frac 12 (2+h)^\frac 12 dh \right)^\frac 23 = 2^\frac 13 (s-1) + O(s-1)^2. \label{eqzetaderiv}
\ee
We then choose to extend \eqref{eqetaident} analytically to $s \in (-1, 1]$ from below, leading to the choice of branches for $(w^2 - 1)^\frac 12$ and $\zeta^\frac 32$. 

Next, we prove the properties of $\zeta$.

(1) The uniqueness of zero for $\zeta$ and non-vanishing of $\zeta_s$ are obvious from \eqref{eqzetadef}. For \eqref{eqImzetasgn}, we notice that 
${\rm sgn}(\Im (w^2 - 1)) = {\rm sgn}\Im s$ when $\Re w > 0$, so the argument of $\zeta$ is evaluated from \eqref{eqetaident} and the choice of $\arg (s-1)$. This also implies the sign of $\zeta$ in \eqref{eqargzetas} when $s > 0$, and the sign of $\zeta_s$ follows \eqref{eqzetadef} for $s \neq 1$ and that $\zeta'(1) = 2^\frac 13$ above. 
% {\color{purple} The behavior of $\zeta$ when $s$ close to $-1$ is given by 
% \[ \frac 23 \zeta^\frac 32(s) = \int_1^{s} i(1 - w^2)^\frac 12 dw = -i\frac \pi 2 + O(|s-1|^\frac 32). \]
% And $\zeta_s = \left(\frac{s^2 - 1}{\zeta}\right)^\frac 12$ indicates $|\zeta_s| \sim |s+1|^{\frac 12}$. Taking one more derivative of \eqref{eqzetadef}, we obtain 
% \[ \zeta_s^3 + 2 \zeta_{ss}\zeta_s \zeta = 2s \quad \Rightarrow \quad \zeta_{ss} = \frac{2s - \zeta_s^3}{2\zeta_s \zeta} \quad \Rightarrow \quad |\zeta_{ss}| \sim |s+1|^{-\frac 12}.\]}

(2) The formula \eqref{eqetaident2} for $s \in (1,\infty)$ is a direct computation. Noticing that $s \in \CC - (-\infty, 1]$ implies $|\arg(s^2 - 1)^\frac 12| < \pi$ and $\left|\arg\left(s + (s^2 - 1)^\frac 12\right) \right| < \pi$, the formula in this complex region is a direct analytic extension.

(3) The asymptotics near $s = 1$ and near infinity easily come from \eqref{eqzetaderiv} and \eqref{eqetaident2} respectively. Lastly, we compute 
\bee
  \int_1^s (w^2 - 1)^\frac 12 dw = -\int_0^1 i(1 -w^2)^\frac 12 dw + \int_0^s i (1 -w^2)^\frac 12 dw = -i\frac \pi 4 + i s + O(s^3)
\eee
with leading order term $\arg(-i\frac \pi 4) = \frac 32\pi$ according to \eqref{eqbranchsqrt}. Then inverting the $\frac 32$ power yields the asymptotics of $\zeta$ near $s=0$. 
 \end{proof}

Before proving the properties of the WKB solutions, we first define two auxiliary functions for describing the asymptotics. 

\begin{definition}[Auxiliary functions for WKB approximate solutions]\label{defWKBaux}
    Let $b > 0$ and $|E-1| \le \frac 12$ and $s, \mu, \zeta$ be as Definition \ref{defWKBappsolu}. We define the $\CC$-valued function $\eta = \eta_{b, E}: (0, \infty) \to \CC$ as 
% \be \eta  = \frac 23 \mu^2 \zeta^{\frac 32} \ee
% using \eqref{eqetaident}.
\be \eta_{b, E}(r)  = \frac 23 \left(\mu^\frac 43 \zeta(s(r)) e^{-\frac{2\pi i}{3}}\right)^{\frac 32},\quad {\rm for\,\,} \arg \left(\mu^\frac 43 \zeta e^{-\frac{2\pi i}{3}} \right) \in (-\pi, \pi],\label{eqetadef}\ee
and the weight functions $\omega_{b,E}^\pm: (0, \infty) \to (0,\infty)$ as 
  \be
 \omega^\pm_{b, E}(r) =  \left\la b^{-\frac 23} (b^2r^2 - 4E) \right\ra^{-\frac 14} e^{\pm \Re \eta_{b, E}(r)}. \label{eqomegapm}
\ee
We also define the center interval as\footnote{Note that $I_c^{b, E}$ is a closed interval, a point or empty. The last two cases happen when $|\Im E| \gg b^\frac 23$.}
\be I_c^{b, E} = \left\{ r \in (0, \infty):  |br-2\sqrt{E}| \le b^{\frac 23} \right\}. \label{eqIcdef}
 \ee
\end{definition}

They will be used to describe the asymptotics of the WKB solution $\psi_j^{b, E}$ \eqref{eqpsidef}. We stress that $\eta_{b, E}$ may have at most one discontinuity due to the jump of branches (where $\Re \eta_{b, E}$, $|\eta_{b, E}|$ are still continuous), as characterized in the lemma below.

\begin{lemma}\label{lemWKBeta}
There exists $0 < \delta_0  \ll 1$ such that for $0 < b \le \frac 14$ and $|E-1|\le \delta_0$, we have the following properties for $\eta_{b, E}$. 
\begin{enumerate}
   \item Turning point and monotonicity of $\Re \eta$: there exists a unique $r^* = r^*_{b, E} \in (0, \infty)$ such that
   \bea\Im \left(\mu^\frac 43 \zeta(r) e^{-\frac{2\pi i}{3}} \right) \left\{ \begin{array}{ll} 
        < 0 & r > r^* \\
        = 0 & r = r^* \\
        > 0 & r < r^* 
        \end{array}\right.\quad {\rm if}\,\, \Im E \ge 0;
        \label{eqsignImzeta} \\
        \Im \left(-\mu^\frac 43 \zeta(r) \right) \left\{ \begin{array}{ll} 
        < 0 & r > r^* \\
        = 0 & r = r^* \\
        > 0 & r < r^* 
        \end{array}\right. \quad {\rm if}\,\, \Im E \le 0;
        \label{eqsignImzeta2}
        \eea
        % For $\Im E > 0$, there also exists a unique $r^+ = r^+_{b, E}$ such that $\zeta \in (-\infty, 0)$. 
    Then $\Re \eta (r^* - 0) = \Re \eta(r^* + 0) = 0$ and $E \mapsto r^*_{b, E}$ is continuous and satisfies the following monotonicity w.r.t. $\Re E$ 
     \be r^*_{b, E_1} < r^*_{b, E_2},\quad {\rm if\,\,} \Re E_1 < \Re E_2 \,\,{\rm and}\,\, \Im E_1 = \Im E_2. 
    \label{eqr*mono} 
    \ee
    Moreover, it satisfies the following bounds
    % \bea
    % \max\left\{ \left| r^+_{b, E} - \frac 2b \right|, \left| r^*_{b, E} - \frac 2b \right| \right\} \le \frac{1}{2b},&& \forall \,\, |E-1| \le \delta_0;  \label{eqr*bE1} \\
    % r^+_{b, E}, r^*_{b, E} \in I_c^{b, E} && \forall \,\,|E-1| \le \delta_0,\,\,{\rm and} \,\, |\Im E| \le b I_0. \label{eqr*bE2}
    % \eea
    \be
     % \left| r^*_{b, E} - \frac 2b \right|\le \frac{100\delta_0}{b},&& \forall \,\, |E-1| \le \delta_0;  \label{eqr*bE1} \\
     \left| r^*_{b, E} - \frac{2\sqrt{\Re E}}{b} \right| \le  \frac{|\Im E|}{b},\quad   \forall \,\,|E-1| \le \delta_0. \label{eqr*bE1}
    \ee
    On different side of $r^*$, we have formula of $\eta$
    \be
      \eta = \left| \begin{array}{ll}
        -i \frac{2E}{b} \int_1^{s(r)} (w^2 - 1)^\frac 12 dw    & r > r^*,\,\, {\rm or}\,\,r < r^* \,\,{\rm and}\,\,\Im E \le 0, \\
         i \frac{2E}{b} \int_1^{s(r)} (w^2 - 1)^\frac 12 dw    & r < r^*\,\,{\rm and}\,\,\Im E > 0,
      \end{array} \right. \label{eqformulaeta}
    \ee
    where $\omega \mapsto (\omega^2 - 1)^\frac 12$ takes the branch as in \eqref{eqbranchsqrt}.
    We also have the formula of $\pa_r \Re \eta$ 
    \be 
     \pa_r \left(\Re \eta(r) \right) = \Im \left( \frac{b^2 r^2}{4} - E \right)^\frac 12,\quad {\rm where}\,\, \left\{ \begin{array}{ll} 
       \arg \left( \frac{b^2 r^2}{4} - E \right) \in (0, 2\pi) & r < r^*, \\
       \arg \left( \frac{b^2 r^2}{4} - E \right) \in (-\pi, \pi) & r > r^*,
     \end{array}\right. \label{eqReetaderiv} 
     % {\rm sgn}(\pa_r \Re \eta) &= \left\{ \begin{array}{ll} 
     %   1 & r < r^*, \\
     %   - {\rm sgn}(\Im E) & r > r^*.
     % \end{array}\right. \label{eqReetamono}
    \ee
    and monotonicity of $\Re \eta$
    \be {\rm sgn}(\pa_r \Re \eta) = \left\{ \begin{array}{ll} 
       1 & r < r^*, \\
       - {\rm sgn}(\Im E) & r > r^*.
     \end{array}\right.
\label{eqReetamono}
    \ee
    \item Asymptotics of $\eta$ and $\Re \eta$: we have
    \bea
           &&|\eta|\left\{\begin{array}{ll}
           \sim b r^2 & r \ge \frac 4b,\\
           \sim b^{-1} \left| br - 2 \sqrt E \right|^{\frac 32} & r \in \left[0,\frac 4b \right] - I_c^{b, E}  \\
               \lesssim 1 & r\in I_c^{b, E}
           \end{array} \right. \label{eqetaabs} \\
            &&\Re \eta = -\frac{\Im E\cdot \left(\ln (br) + O(1)\right)}{b} + \frac{\Im (E \ln E)}{2b},\quad r \ge \frac 4b; \label{eqReetaext} \\
          &&\eta= -\frac{\pi E}{2b} + \left(\sqrt E + O_\CC(b)\right)r,\quad 0 \le r \le b^{-\frac 12}.  \label{eqetaconv} 
           \eea
Moreover, if in addition $|\Im E| \le bI_0$ for some $0 < I_0 < (10b^{\frac 13})^{-1}$, we have refined asymptotics of $\Re \eta$
    \bea
           \Re \eta = \left\{\begin{array}{ll}
               -\frac{\Im E}{b} \ln (br) + O(1) & r \in  [r^*_{b, E},\infty) - I_c^{b, E} \\
               -S_{b, \Re E}(r) + O(b |br-2\sqrt {\Re E}|^{-\frac 12}) & r\in [0, r^*_{b, E}] - I_c^{b, E}
           \end{array} \right. \label{eqetaRe}  
           % \Im \eta &=& O(|br - 2\sqrt{\Re E}|^\frac 12) \qquad r \in [0, r^*_{b, E}] - I_c^{b, E}. \label{eqetaIm}
           \eea
           where $S_b(r) = \int_{\min\{ r, \frac 2b \}}^{\frac 2b} \left( 1 - \frac{b^2 s^2}{4} \right)^\frac 12$,  $S_{b, \a}(r) = S_{b}(\frac{r}{\sqrt\a}) \a$ for $\a \in \RR$, and the bound of $O(...)$ terms in \eqref{eqetaRe} depend on $I_0$. 
\end{enumerate}
\end{lemma}
\begin{remark}
    From Definition \ref{defWKBappsolu} and Definition \ref{defWKBaux}, we have the scaling law for $b > 0$:
    \be \eta_{b, E}(r) = b^{-1} \eta_{1, E}(br),\quad r^*_{b, E} = b^{-1}r^*_{1, E}. 
    \ee
    However, we choose not to apply it in our proof to coordinate with the case of high spherical classes (Proposition \ref{propWKBh}) which does not have such a simple scaling law. 
\end{remark}

\begin{proof}[Proof of Lemma \ref{lemWKBeta}]\mbox{}
    
\underline{\textbf{1. Proof of (1).}} 

\underline{1.1. Existence and uniqueness of $r^*_{b, E}$.} We consider the following three cases.

\mbox{}

(a) For $E \in \RR$, \eqref{eqargzetas} implies $r^*_{b, E} = \frac{2\sqrt E}{b}$. The sign of $\Im \left( \mu^\frac 43 \zeta(s(r)) e^{-\frac{2\pi i}{3}} \right)$ and $\Im \left(-\mu^\frac 43 \zeta(s(r))\right)$ follows $\arg \left(\mu^\frac 43 e^{-\frac {2 \pi i}{3}} \right)  = -\frac \pi 3$ and the monotonicity of $\zeta$. 

% For $E \notin \RR$ and $|E-1| \le \delta_0$, \eqref{eqetaident} indicates that $\pm \Im \zeta(r) > 0$ for all $r > 0$ if $\pm \Im E < 0$. We now prove the case of $\Im E < 0$ and the case for $\Im E > 0$ follows similarly.

\mbox{}

(b) For $\Im E < 0$, \eqref{eqImzetasgn} and $s = \frac{br}{2\sqrt {E}}$ implies $\arg\zeta(s(r)) \in (0,\pi)$. Recall that $\arg \left(-\mu^\frac 43  \right) = -\frac{2\pi}3 + \frac 23 \arg E$, we see $\Im \left(-\mu^\frac 43 \zeta(r)\right) = 0$ is equivalent to $\arg \zeta(r) = \frac {2\pi}{3} - \frac 23 \arg E$. 

Note that $\arg (s(r)-1)$ decays monotonically from $\pi$ to $-\frac 12 \arg E$ when $r$ increases from $0$ to $\infty$, yielding the unique existence of $r_L$ and $r_R$ (depending on $b, E$) such that 
\be  \arg (s(r_L) - 1) = \frac {5\pi}6, \quad \arg (s(r_R) - 1) = \frac \pi 6;\quad \arg(s(r)-1) \in \left[ \frac \pi 6, \frac{5\pi}{6} \right]\,\, \forall r \in [r_L, r_R]. \label{eqrLrR}\ee
Since $|E-1| \le \delta_0 \ll 1$ implies $\arg E = O(\delta_0)$, so we also have 
\[ \arg\left(r_L - \frac{ \sqrt{E}}{2b}\right) = \frac{5\pi}{6} + O(\delta_0),\quad  \arg\left(r_R - \frac{\sqrt{E}}{2b}\right) = \frac{\pi}{6} + O(\delta_0),  \] 
and hence 
\be  r_{\sigma_\pm} - \frac{2\Re \sqrt E}{b}
% &=&  (\sqrt 3 + O(\delta_1)) \frac{\Im \sqrt E}{2b} 
= \pm\left(\frac{1}{\sqrt 3} + O(\delta_0)\right) \frac{2\Im \sqrt{ E}}{b} = \pm\left(\frac{1}{\sqrt 3} + O(\delta_0)\right)\frac{\Im E}{b},\label{eqrLrRasymp}
\ee
where $\sigma_+ = L$ and $\sigma_- = R$, using $\Im \sqrt{E} = (\frac 12 + O(\delta_0)) \Im E$.
%$$|\Re (E^{-\frac 12}) - (\Re E)^{-\frac 12}| \le 10 \delta_0, \quad |\Im  (E^{-\frac 12})| \le 10 \delta_0$$
%when $|E-1| \le \delta_0 \ll 1$. So with $\Re (s-1) = \frac{br}{2}\Re (E^{-\frac 12}) - 1$, $\Im (s-1) = \frac{br}{2}\Im (E^{-\frac 12})$, we have
%\be [r_L,  r_R] \subset \left[\frac{2 - 100\delta_0}b, \frac{2 + 100\delta_0}b\right], \label{eqrLrRsest1}
%\ee
%so 
%\be |\Im E| \sim \Im (s(r) - 1) \gtrsim |s(r) - 1|,\quad \forall\,\,r \in [r_L, r_R]. \label{eqrLrRsest2}\ee

Now we claim the following two ingredients: a rough estimate of $\arg \zeta$
\be
   \arg \zeta(s) = \arg (s-1) + O(\delta_0), \quad \forall\,  r \ge 0,
  %  \left\{ \begin{array}{ll}
  %     \arg (s-1) + O(\delta_0^\frac 12) & |s| \le \delta_0^{-\frac 12},\\
  %     O(\delta_0^\frac 12) & |s| \ge \delta_0^{-\frac 12}.
  % \end{array}\right. 
  \label{eqzetaarg}
\ee
 and the monotonicity of $\arg \zeta$ on $[r_L, r_R]$
 \be \pa_r \arg \zeta(r) < 0,\quad r \in [r_L, r_R].\label{eqzetaargmono} \ee 
 Together with the definition of $r_L, r_R$, \eqref{eqzetaarg} implies $r^* \in [r_L, r_R]$, and \eqref{eqzetaargmono} confirms the existence and uniqueness when $\delta_0 \ll 1$. 

 \textit{Proof of \eqref{eqzetaarg}.} Notice that $\arg s = -\frac 12\arg E = O(\delta_0)$, so $\{ ts + (1-t): t \in [0, 1] \} \subset \{ z: |\arg z| \lesssim \delta_0 \}$ and thus $|\arg (t(s-1) + 2)| \lesssim \delta_0$ for $t \in [0, 1]$. Then 
 \be
\frac 23 \zeta^\frac 32 = \int_0^{s-1} h^\frac 12 (h+2)^\frac 12 dh = e^{i\frac 32 \arg(s-1)} \int_0^{|s-1|} t^\frac 12 \left((s-1)\cdot \frac{t}{|s-1|} + 2\right)^\frac 12 dt \label{eqetaident3}
\ee
 implies \eqref{eqzetaarg}.

%  when $|s| \le \delta_0^{-\frac12}$, we obtain $|\Im s| \lesssim \delta_0^{\frac 12}$ and thus $|\arg (t e^{i\arg(s-1)} + 2)| \lesssim \delta_0^\frac 12$ for all $t \in [0, |s-1|]$. Writting \eqref{eqetaident} as
% \be
% \frac 23 \zeta^\frac 32 = \int_0^{s-1} h^\frac 12 (h+2)^\frac 12 dh = e^{i\frac 32 \arg(s-1)} \int_0^{|s-1|} t^\frac 12 (t e^{i\arg(s-1)} + 2)^\frac 12 dt \label{eqetaident3}
% \ee
% yields the asymptotics for $|s| \le \delta_0^{-\frac 12}$. When $|s| \ge \delta_0^{-\frac 12} \gg 1$, the $O(\delta_0^\frac 12)$ bound easily follows estimating $|\eta^\frac 32|$ and $|\Im \eta^\frac 32|$ from the explicit formula \eqref{eqetaident2}.

\textit{Proof of \eqref{eqzetaargmono}.} 
This is a direct computation with \eqref{eqzetaderiv}
\bee
 \pa_r \arg \zeta(r)  &=& \pa_r \Im \ln \zeta(r) = \Im \left[ \frac{\zeta_s(r)}{\zeta(r)} \pa_r s(r) \right] \\
 &=&  \Im \left[ (s-1)^{-1} \left( 1 + O_\CC(s-1)\right) \frac{b}{2 \sqrt E} \right] < 0,\quad r\in [r_L, r_R],
\eee
where in the last step we used $\Im (s-1) \sim  |s-1|$ from \eqref{eqrLrR}. 

\mbox{}

(c) For $\Im E > 0$, \eqref{eqImzetasgn} implies $\arg \zeta(s(r)) \in (-\pi, 0)$, and thus $\Im (\mu^\frac 43 \zeta(r) e^{-\frac{2 \pi i}{3}}) = 0 \Leftrightarrow \arg \zeta(r) = -\frac {2\pi} 3 - \frac 23 \arg E$. The proof is similar with $\arg(s(r_L) - 1) = -\frac{5\pi}{6}$, $\arg (s(r_R) - 1) = -\frac \pi 6$, and hence is omitted. In particular, \eqref{eqzetaarg} still holds. 
% so we only present a sketch. After verifying \eqref{eqzetaarg}, we denote $r_L$ and $r_R$ such that 
% \[ \arg (s(r_L) - 1) = -\frac \pi 2 , \quad \arg (s(r_R) - 1) = -\frac \pi 6;\quad \arg(s(r)-1) \in \left[ -\frac \pi 2, - \frac \pi 6 \right]\,\, \forall r \in [r_L, r_R],\]
% and check \eqref{eqrLrRsest} and $\pa_r \arg \zeta(r) > 0$ on $[r_L, r_R]$. Then $r^*_{b, E} \in [r_L, r_R]$ exists uniquely from the monotonicity and \eqref{eqzetaarg}. 

\mbox{}

\underline{1.2. Vanishing of $\Re \eta$ and estimates of $r^*_{b, E}$.}
The vanishing of $\Re \eta$ at $r^* \pm 0$.  follows that $|\arg(\eta)| \to  \frac 32\pi$ or $\frac 12 \pi$ when $r \to r_* \pm 0$. As for the estimate of $r^*_{b, E}$ \eqref{eqr*bE1}, it is self-evident by $r^*_{b, E} = \frac{2\sqrt{E}}{b}$ when $E \in \RR$; when $E \notin \RR$, it easily follow $r^*_{b, E} \in [r_L, r_R]$ and \eqref{eqrLrRasymp} from Step 1.1.

\mbox{}

\underline{1.3. The monotonicity \eqref{eqr*mono} and continuity of $r^*_{b, E}$.} They are conclusions  of implicit function theorem. Specifically, Let $F_b^+(r, \Re E, \Im E) := \Im \left( \mu^\frac 43 \zeta e^{-\frac{2\pi i}{3}} \right)$ which is apparently $C^\infty$ w.r.t. $r$, $\Re E$ and $\Im E$. Then $r^*_{b, E}$ is uniquely determined by $F_b^+(r^*_{b, E}, \Re E, \Im E) = 0$ when $\Im E \ge 0$. Because of \eqref{eqr*bE1}, we compute for $r \in [\frac{2-2\delta_0}{b}, \frac{2+2\delta_0}{b}]$ that
\bee
  \pa_r F_b^+(r, \Re E, \Im E) &=& \Im \left( \mu^\frac 43 \zeta_s \frac{b}{2\sqrt{E}} e^{-\frac{2\pi i}3}\right) = \Im \left( e^{-\frac{\pi i}{3}} E^\frac 16 b^\frac 13 (1 + O(\delta_0)) \right) < 0 \\
  \pa_{\Re E} F_b^+(r, \Re E, \Im E) &=& \Im \left( \frac 23 E^{-1} \mu^\frac 43 \zeta e^{-\frac{2\pi i}3} +   \mu^\frac 43 \zeta_s \left(-\frac{br}{4E^{\frac 32}} \right) e^{-\frac{2\pi i}3}  \right)\\
  &=& \Im \left( -e^{-\frac{\pi i}{3}} \frac{b^\frac 13 r}{2E^\frac 56} \left( 1 + O(\delta_0) \right) \right) > 0.
\eee
where we used $|s-1| \lesssim \delta_0$, $\zeta_s = 2^\frac 13 + O(\delta_0)$ from \eqref{eqzetaasymp}. By implicit function theorem, the non-degeneracy of $\pa_r F_b^+$ and differentiability of $F_b^+$ imply that $(\Re E, \Im E) \to r^*_{b, E}$ is differentiable for $\Im E > 0$ and continuous for $\Im E \ge 0$; and the sign above indicates that $\pa_{\Re E} r^*_{b, E} > 0$, which is the monotonicity \eqref{eqr*mono} for $\Im E \ge 0$. The result in the region $\Im E \le 0$ follows by considering $F_b^-(r, \Re E, \Im E) := \Im \left( -\mu^\frac 43 \zeta \right)$.

\mbox{}

\underline{1.4. Formula and monotonicity of $\Re \eta$ \eqref{eqformulaeta}-\eqref{eqReetamono}.}

(a) For $r > r^*$, \eqref{eqsignImzeta}-\eqref{eqsignImzeta2} imply the choice of branch $\arg (\mu^\frac 43 \zeta e^{-\frac{2\pi i}{3}}) \in (-\pi, \frac \pi 3)$ and $\zeta \in (-\frac {2\pi} 3 - \frac 23 \arg E, \frac {2\pi}{3} - \frac 23 \arg E) \subset (-\pi, \pi)$. 
So direct application of \eqref{eqetaident} gives 
\be \eta = \frac 23 e^{-\pi i} \mu^2 \zeta^\frac 32 = 
-i \frac{2E}{b} \int_1^{s(r)} (w^2 - 1)^\frac 12 dw  \label{eqeta00}
% = -i\int_{\frac{2\sqrt{E}}{b}}^r \left(\frac{b^2 \tau^2}{4} - E\right)^\frac 12 d\tau  
\ee
where the branch of square root is taken as \eqref{eqbranchsqrt}, 
and thus $|\arg(s-1)| \le \frac {5\pi}{6}$ from \eqref{eqrLrR} and $|\arg(s+1)| \lesssim \delta_0$ from $|E-1| \le \delta_0$. Hence $|\arg(s^2 - 1)^\frac 12| \le \frac{5\pi}{12} + O(\delta_0)$, and thus  
\bee
 \pa_r \Re \eta = \Re \left[  \frac{ds}{dr}\cdot (-i)\frac{2E}{b} (s^2 - 1)^\frac 12 \right] =  \Im \left( \frac{b^2 r^2}{4} - E \right)^\frac 12 \,\,\, \Rightarrow \,\,\,  {\rm sgn}(\pa_r \Re \eta) = -{\rm sgn} (\Im E).
\eee
where  $ \arg (\frac{b^2 r^2}{4} - E) = \arg(s^2 - 1) + O(\delta_0)$ so satisfies \eqref{eqReetaderiv}.
% \[ \pa_r \Re \eta = - \Im \left( \frac{b^2 r^2}{4} - E \right)^\frac 12 \quad \Rightarrow \quad  {\rm sgn}(\pa_r \Re \eta) = {\rm sgn} (\Im E). \]

(b) For $r < r^*$, we claim that with the branch of $\zeta$ such that $\arg (\mu^\frac 43 \zeta e^{-\frac{2\pi i}3}) \in (-\pi, \pi)$, we have 
\bee
 \arg \zeta  \left| \begin{array}{ll}
     \in \left( \frac {2\pi} 3 - \frac 23 \arg E, \pi \right) &  \Im E < 0, \\
     = \pi & \Im E  = 0, \\
     \in (\pi, \frac {4\pi}{3} - \frac 23 \arg E) & \Im E > 0,
 \end{array}\right.\quad {\rm for\,\,} r < r^*.
\eee
This follows from \eqref{eqImzetasgn} and \eqref{eqsignImzeta}, \eqref{eqargzetas}, \eqref{eqImzetasgn} and \eqref{eqsignImzeta2} for the cases $\Im E > 0$, $\Im E = 0$, $\Im E < 0$ respectively. 

% For $\Im E < 0$, \eqref{eqImzetasgn} and \eqref{eqsignImzeta2} indicate $\arg \zeta \in \left( \frac {\pi} 3 - \frac 23 \arg E, \pi \right)$ and $\arg (\mu^\frac 43 \zeta e^{-\frac{2\pi i}3}) \in (\frac{\pi}{3}, \frac{2\pi}{3} + \frac 23 \arg E) \subset (-\pi, \pi)$. For $\Im E = 0$, from \eqref{eqargzetas} we see $\arg \zeta = \pi$ and $\arg  (\mu^\frac 43 \zeta e^{-\frac{2\pi i}3}) = \frac{2\pi}{3} \in (-\pi, \pi)$. On the other hand, for $\Im E > 0$, \eqref{eqImzetasgn} and \eqref{eqsignImzeta} yield $\arg (\mu^\frac 43 \zeta e^{-\frac{2\pi i}3}) \in (\frac{\pi}{3}, \frac{2\pi}{3} + \frac 23 \arg E) \subset (-\pi, \pi)$ with $\arg \zeta \in (\pi, \frac {4\pi}{3} - \frac 23 \arg E)$.

% $\Im(\mu^\frac 43 \zeta) > 0$ indicates the choice of branch $\arg (\mu^\frac 43 \zeta e^{-\frac{2\pi i}3}) \in (0, \pi)$ for \eqref{eqetadef},  and hence $\zeta \in \left( \frac {\pi} 3 - \frac 23 \arg E, \frac {4\pi}{3} - \frac 23 \arg E \right)$. More specifically, with the sign of $\Im \eta$ depending on ${\rm sgn}(\Im E)$, 
% \bee
%  \arg \zeta  \left| \begin{array}{ll}
%      \in \left( \frac {\pi} 3 - \frac 23 \arg E, \pi \right) &  \Im E < 0, \\
%      = \pi & \Im E  = 0, \\
%      \in (\pi, \frac {4\pi}{3} - \frac 23 \arg E) & \Im E > 0,
%  \end{array}\right.\quad {\rm for\,\,} r < r^*.
% \eee
% % when $\Im E < 0$, with $\Im \zeta > 0$ for any $r \ge 0$, we have $\arg \zeta \in \left( \frac {\pi} 3 - \frac 23 \arg E, \pi \right)$ for $r \le r^*$; similarly, $\arg \zeta = \pi$ for $\Im E = 0$ and $\arg \zeta \in (\pi, \frac {4\pi}{3} - \frac 23 \arg E)$ when $\Im E > 0$. 

Now for $\Im E \le 0$, we can apply \eqref{eqetaident} to obtain \eqref{eqeta00} in the same way. Again with \eqref{eqrLrR}, $\Im E \le 0$ and \eqref{eqbranchsqrt}, in \eqref{eqeta00} we take branches $\arg (s-1) \in [\frac \pi 6, \pi]$, $0 \le \arg(s+1) \lesssim \delta_0$ and thus $ \arg (\frac{b^2 r^2}{4} - E) = \arg(s^2 - 1) + O(\delta_0) \in [ \frac{\pi}{6}, \pi ] + O(\delta_0)$, which implies
$
\pa_r \Re \eta = \Im \left( \frac{b^2 r^2}{4} - E \right)^\frac 12 > 0.$ 

For $\Im E > 0$, we apply \eqref{eqetaident} with $\zeta e^{-2\pi i}$ to compute
\be
  \eta = \frac 23 \left( e^{-\frac{2\pi i}{3}} \mu^\frac 43 \zeta \right)^\frac 32 = \mu^2 e^{-\pi i} \cdot \frac 23 \left(\zeta e^{-2 \pi i}\right)^\frac 32 e^{3\pi i} = i \frac{2E}{b} \int_1^{s(r)} (w^2 - 1)^\frac 12 dw  \label{eqeta01}
\ee
where the branch of square root is again as \eqref{eqbranchsqrt}. Then $\pa_r \Re \eta = - \Im \left( \frac{b^2 r^2}{4} - E \right)^\frac 12 > 0$ where  $\arg \left( \frac{b^2 r^2}{4} - E \right) \in [-\pi, -\frac \pi{6} ] + O(\delta_0)$. This is equivalent to take $\pa_r \Re \eta = \Im \left( \frac{b^2 r^2}{4} - E \right)^\frac 12$ with the branch $\arg \left( \frac{b^2 r^2}{4} - E \right) \in [\pi, \frac {11\pi}{6} ] + O(\delta_0)$, which implies the formulation in \eqref{eqReetaderiv}.

\mbox{}

\underline{\textbf{2. Proof of (2).}} The asymptotics of $|\eta|$ \eqref{eqetaabs} is a simple consequence of asymptotics of $\zeta$ \eqref{eqzetaasymp} and non-vanishing of $\zeta$ on $\{ |\arg s| \le \frac \pi 4, s \neq 1 \}$ from Lemma \ref{lemzeta32} (1). 

For \eqref{eqReetaext}, when $r > r^*$, we compute using \eqref{eqformulaeta} and the expansion \eqref{eqetaident2}
\bee
  \Re \eta &=& \Re \left[  -i \frac{2E}b \left( \frac 12 s(s^2 - 1)^\frac 12 - \frac 12 \ln \left(s + (s^2 - 1)^\frac 12\right) \right)  \right] \\
  &=&  \Im \left[ \frac 14 r \left(b^2 r^2 - 4E \right)^\frac 12 - \frac{E}{b} \left( \ln \left( br + (b^2 r^2 - 4E)^\frac 12 \right) -  \ln (2\sqrt{E}) \right) \right] 
\eee
Since $\Im \left(b^2 r^2 - 4E \right)^\frac 12 = O_\RR \left( \frac{\Im E}{(b^2 r^2 - 4 \Re E)^\frac 12} \right)$ and 
\bee 
\ln \left( br + (b^2 r^2 - 4E)^\frac 12 \right) &=&  \ln \left| br + (b^2 r^2 - 4E)^\frac 12\right| + i   \arctan \frac{\Im (b^2 r^2 - 4E)^\frac 12 }{br + \Re  (b^2 r^2 - 4E)^\frac 12} \\
&=&
\ln (br) + O_\CC(1)
\eee
supposing $br \ge 2 \sqrt{\Re E}$, 
we have
\bee
 \Re \eta = -\frac{\Im E(\ln (br) + O_\CC (1)) }{b} +O_\RR \left( \frac{r \Im E}{(b^2 r^2 - 4 \Re E)^\frac 12}  \right) + \frac{\Im \left( E \ln (2\sqrt E) \right) }{b}.
\eee
Thus when $r \ge \frac 4b > 2 \sqrt{\Re E}$, we have  $\frac{r}{(b^2 r^2 - 4\Re E)^\frac 12} \lesssim b^{-1}$ and hence \eqref{eqReetaext}. 

For \eqref{eqetaconv}, with $r \le b^{-\frac 12} \ll r^*$, we use \eqref{eqformulaeta} to compute similarly as \eqref{eqzetaasymp}
\bea
 \eta &=& \left| \begin{array}{ll}
     -i\frac{2E}b\left( -\int_0^1 i (1-w^2)^\frac 12dw + \int_0^{s(r)} i (1-w^2)^\frac 12 dw \right)  & \Im E \le 0 \\
     i\frac{2E}b\left( -\int_0^1 (-i)(1-w^2)^\frac 12dw + \int_0^{s(r)} (-i)(1-w^2)^\frac 12 dw \right)  & \Im E > 0
 \end{array}\right.\nonumber  \\
 &=& -\frac{\pi E}{2b} + \frac{2E}{b} (s + O_\CC (s^3)) \label{eqetacomputerleb-1}
\eea
which implies \eqref{eqetaconv} with $|s| \lesssim b^\frac 12$.

Lastly, we prove \eqref{eqetaRe}, the refined asymptotics of $\Re \eta$, when $|\Im E| \le bI_0$. When $r \ge \frac 4b$, it follows \eqref{eqReetaext}, in particular we have $|\Re \eta(\frac 4b)| \lesssim 1$. Hence the boundary value $\Re \eta (r^*+0) = 0$ and the monotonicity \eqref{eqReetamono} implies the estimate on $r \in [r^*, \frac 4b]$. For $r \in [0, r^*] - I_c^{b, E}$, we again use \eqref{eqformulaeta} for $\Im E \le 0$ to derive 
\bee
  \eta &=&  -i \frac{2E}{b} \left( -\int_s^1 i(1-w^2)^\frac 12 dw \right) \\
  &=& -\frac{2E}{b} \left[ \int_{\Re s}^1 (1-t^2)^\frac 12 dt - i\int_{0}^{\Im s} \left( 1 - (\Re s + i\tau)^2 \right)^\frac 12 d\tau  \right] \\
  &=& -\left( \frac{2\Re E}{b} + i O(1) \right) \bigg\{ \int_{\frac{br}{2\sqrt{\Re E}}}^1 (1-t^2)^\frac 12 dt + O(b^2) \\
  && \quad - i \int_0^{O(b)}\left[ \left(1 - \frac{b^2 r^2}{4\Re E} + O(b^2)  \right)^\frac 12 + O_\CC\left( \frac{br \tau}{(4\Re E - b^2 r^2 + O(b^2))^\frac 12} \right) \right] d\tau \bigg\}  \\
  &=& -\frac{2\Re E}{b} \int_{\frac{br}{2\sqrt{\Re E}}}^1 (1-t^2)^\frac 12 dt + O\left( b(2\sqrt{\Re E} - br)^{-\frac 12} \right) + iO\left((2\sqrt{\Re E} - br)^{\frac 12}  \right)
\eee
where we used $\Re \sqrt E = \sqrt{\Re E} + O((\Im E)^2)$, $\Im \sqrt E = \frac 12 \Im E (1 + O(\delta_0))= O(b)$ and that $\frac 12 b^{\frac 23} \le 4\Re E - b^2 r^2 \sim |4E - b^2 r^2| \sim |2\sqrt{E} - br| \lesssim 1$ according to \eqref{eqIcdef} and that $I_0 \le (10b^\frac 13)^{-1}$. That concludes \eqref{eqetaRe} for $\Im E \le 0$. The case of $\Im E > 0$ follows similarly.
\end{proof}

% {\color{purple} This branch of $\zeta^\frac 32$ is chosen according to take different branch from \eqref{eqetaident} sometimes, but this definition simplifies description of the asymptotics of $\psi_j$.
 % We note that when $r > -\frac 2b$, we have $s = \frac{br}{2\sqrt{E}} = \frac{br}{2}\left(1 - \frac i2 b\nu - \frac 38 b^2 \nu^2 + O(b^3\nu^3) \right)$ satisfying $\Re s > -1$. }
 % We stress that $\eta$ is not analytic as a function of $s$ in $\{\Re s > -1 \}$.

%
%Define the derivatives (for $r \ge 4b^{-1}$)
%\bea D_\pm = e^{\mp\eta} \pa_r e^{\pm\eta} = \pa_r \pm i \left(\frac{b^2 r^2}{4} - E\right)^\frac 12, && D_{\pm \pm} =  \pa_r \pm i \left(b^2 r^2 - 4E\right)^\frac 12, \\
%\mathring{D}_\pm = e^{\mp i \frac{br^2}{4}} \pa_r e^{\pm i \frac{br^2}{4}} = \pa_r \pm i\frac{br}{2}, && \mathring{D}_{\pm \pm} = \pa_r \pm i br.
%\eea

Next, we prove properties for these WKB solutions.

\begin{proposition}[Linear WKB solution with correction]\label{propWKB} Let $0 < \delta_0 \ll 1$ from Lemma \ref{lemWKBeta}. For any $I_0 > 0$ there exists $b_0 = b_0(I_0) \ll 1$ such that for any $0 < b \le b_0$ and $E \in \{ z \in \CC: |z-1| \le \delta_0, \Im z \le bI_0 \}$, the following statements hold. Here $s, \zeta, \mu, \psi_j^{b, E}, h_{b, E}$ and $\eta_{b, E}, \omega_{b, E}^\pm, I_c^{b, E}$ are from Definition \ref{defWKBappsolu} and Definition \ref{defWKBaux} respectively.

% Define the center interval as\footnote{Note that $I_c^{b, E}$ may be a point or empty when $|\Im E| \gg b^\frac 23$.}
% \[ I_c^{b, E} = \left\{ r \in (0, \infty):  |br-2\sqrt{E}| \le M_0 b^{\frac 23} \right\} \subset \left[ \frac 3{2b}, \frac 5{2b} \right],
%  \]
% and in particular, for $|\Im E| \le bI_0$, 
%  \be I_c^{b, E} =\frac {2 \sqrt{\Re E}}b  \left[ -M_0 b^{-\frac 13} + O(1), M_0 b^{-\frac 13} + O(1) \right]. \label{eqIc2} \ee

% {\color{blue}
% and intervals are 
% \bee I_1^{b, E} = \left[-\frac {2\Re E}b + b^{-\frac 13}, \frac {2\Re E}b - M_0 b^{-\frac 13} \right], &&
% I_2^{b, E} = \left[\frac {2\Re E}b - M_0 b^{-\frac 13}, \frac {2\Re E}b + M_0 b^{-\frac 13} \right], \\
% I_3^{b, E} = \left[\frac {2\Re E}b + M_0 b^{-\frac 13}, \frac 4b \right], && 
% I_4^{b, E} = \left[ \frac 4b , \infty \right). \eee}
% We might omit the superscript for simplicity. 

    \begin{enumerate}
        \item Corrected equation: $\psi_j^{b, E}$ for $j = 1, 2, 3, 4$ solves 
        \[ \left( \pa_r^2 - E + \frac{b^2 r^2}{4}\right) \psi_j^{b, E} = h_{b, E} \psi_j^{b, E} \]
        with correction $h_{b, E}$ satisfying the estimates for all $n \ge 1$, 
        % \footnote{\color{blue} The largeness of $h$ near $r =- \frac 2b$ can be refined by connecting with the branch of Airy starting from $r = -\frac 2b$.}
        \be |h_{b, E}(r)| \lesssim \min\{ b^2, r^{-2} \},\quad r \ge 0; \quad |\pa_r^n h_{b, E}(r)| \lesssim_{n} r^{-2-n},\quad r \ge \frac 4b. \label{eqbddh}\ee
        \item Connection formula:
\be \psi_1^{b, E} + e^{-\frac{2\pi i}{3}}\psi_2^{b, E} +  e^{\frac{2\pi i}{3}}\psi_3^{b, E} = 0,\quad 2\psi_4^{b, E}(r) =  e^{\frac{\pi i}{6}} \psi_1^{b, E} +  e^{-\frac{\pi i}{6}} \psi_3^{b, E}. \label{eqconnect}
\ee
\item Wronskian:
        \bea
        \mathcal{W}(\psi_4^{b, E}, \psi_2^{b, E}) = \frac{b^\frac 13 E^\frac 16 }{2^\frac 43 \pi},\quad  \mathcal{W}(\psi_1^{b, E}, \psi_3^{b, E}) = \frac{-i b^\frac 13 E^\frac 16}{2^\frac 43 \pi}. \label{eqWronski2}
        \eea
        % {\color{purple} \[  \mathcal{W}(\psi_1^{b, E}, \psi_2^{b, E}) = \frac{b^\frac 13 E^\frac 16 e^{-\frac{\pi i}{6}}}{2^\frac 43 \pi}.\]}
%        \bea |h(r)| &\lesssim& \max\{b^{-1}, r\}^{-2},\quad \forall r \ge 0; \\
%        |\pa_r^k h(r)| &\lesssim_k&  r^{-2-k},\quad \forall r \ge 4 b^{-1}. \eea
\item $\RR$-valued components for $E = 1$: $\psi_2^{b, 1}(r), \psi_4^{b, 1}(r), h_{b, 1}(r)\in \RR$ for $r \le \frac 2b$. Moreover, $\psi_1^{b, 1}$ is non-vanishing on $\RR$. 
    \item Asymptotics and bounds of $\psi_j^{b, E}$ and $(\psi_j^{b, E})'$: for $|E-1| \le \delta_0$ and $\Im E \le bI_0$, we have 
    % \footnote{Here the notation $f \sim g$ with $f$ complex-valued and $g$ positive means $|f| \sim g$.}
        \bea 
         \psi_1^{b, E}(r)&=& \begin{cases}
            \frac{e^{\frac{\pi i}{6}}e^{-\eta}}{2\sqrt{\pi}\mu^\frac 13 (s^2 - 1 )^\frac 14 } (1 + O(\eta^{-1})) 
            % \sim b^\frac 16|b^2 r^2 - 4E|^{-\frac 14}(br)^{\frac{\Im E}{b}}, 
            &r \in  [r^*_{b, E},\infty) - I_c^{b, E}, \\
             \frac{e^{-\frac{\pi i}{12} }e^{-\eta} }{2\sqrt{\pi}\mu^\frac 13 (1-s^2)^\frac 14 } (1 + O(\eta^{-1})) 
             % \sim b^\frac 16 |b^2 r^2 - 4E|^{-\frac 14} e^{S_{b, E}(r)} , 
             &  r \in  [0, r^*_{b, E}] - I_c^{b, E}, \\
            O (1), & r \in I_c^{b, E};
        \end{cases}\label{eqWKBasymp2} \\
        \psi_2^{b, E}(r) &=& \begin{cases}
         \frac{e^{\frac{\pi i}{12} }e^{\eta} }{2\sqrt{\pi}\mu^\frac 13 (1-s^2)^\frac 14 } (1 + O(\eta^{-1})) 
           % % O(b^\frac 16 |b^2 r^2 - 4E|^{-\frac 14} \max_\pm \{e^{\pm \Re \eta}\} ), & r \in  [r^*_{b, E},\infty) - I_c^{b, E},\\
           %   % \sim b^\frac 16 |b^2 r^2 - 4E|^{-\frac 14} e^{-S_{b, E}(r)} ,
             & r\in [0, r^*_{b, E}] - I_c^{b, E}, \\
            O(1), & r \in I_c^{b, E};
        \end{cases}\label{eqWKBasymp3} \\ 
       \psi_3^{b, E}(r)&=& \begin{cases}
            \frac{e^{\eta}}{2\sqrt{\pi}\mu^\frac 13 (s^2 - 1 )^\frac 14 } (1 + O(\eta^{-1})) 
            % \sim b^\frac 16 |b^2 r^2 - 4E|^{-\frac 14} (br)^{-\frac{\Im E}{b}},
            & r \in [4b^{-1}, \infty),\\
            O(\mu^{-\frac 13}|s^2 - 1|^{-\frac 14}e^\eta) 
            % \sim b^\frac 16 |b^2 r^2 - 4E|^{-\frac 14} (br)^{-\frac{\Im E}{b}},
            & r \in [r^*_{b, E}, 4b^{-1}] - I_c^{b, E},\\
             % \frac{e^{\frac{\pi i}{4}}e^{-\eta}}{2\sqrt{\pi}\mu^\frac 13 (1-s^2)^\frac 14 } (1 + O(\eta^{-1}))
             %% \sim b^\frac 16 |b^2 r^2 - 4E|^{-\frac 14} e^{S_{b, E}(r)} , 
             % & r \in [0, r^*_{b, E}] - I_c^{b, E},\\
            O(1), & r \in I_c^{b, E};
        \end{cases}\label{eqWKBasymp1} 
        % \psi_4^{b, E}(r)&=& \begin{cases}
        %     \frac{e^{\frac{\pi i}{12} }e^{-\eta} }{2\sqrt{\pi}\mu^\frac 13 (1-s^2)^\frac 14 } (1 + O(\eta^{-1}))
        %      % \sim b^\frac 16 |b^2 r^2 - 4E|^{-\frac 14} e^{S_{b, E}(r)} , 
        %      &  r \in  [0, r^*_{b, E}] - I_c^{b, E}, \\
        %     O (1), & r \in I_c^{b, E};
        % \end{cases}\label{eqWKBasymp4}
        \eea
        and
        \bea
         (\psi_1^{b, E})'(r)&=& \begin{cases}
            \frac{i\sqrt{E} (s^2 - 1)^\frac 14 e^{\frac{\pi i}{6}}e^{-\eta}}{2\sqrt{\pi}\mu^\frac 13} (1 + O(\eta^{-1})) 
            % \sim b^\frac 16 |b^2 r^2 - 4E|^{\frac 14}(br)^{\frac{\Im E}{b}}, 
            & r \in [r^*_{b, E},\infty) - I_c^{b, E},\\
           \frac{- \sqrt{E}(1-s^2)^\frac 14 e^{-\frac{\pi i}{12}}e^{-\eta}}{2\sqrt{\pi}\mu^\frac 13 } (1 + O(\eta^{-1})) 
           % \sim b^\frac 16 |b^2 r^2 - 4E|^{\frac 14} e^{S_{b, E}(r)} , 
           &r\in [0, r^*_{b, E}] - I_c^{b, E}, \\
            O(b^\frac 13), & r \in I_c^{b, E};
        \end{cases}\label{eqWKBasymp2d}\\
        (\psi_2^{b, E})'(r)&=& \begin{cases}
         % O\left( \frac{|1-s^2|^\frac 14 e^{\eta}}{\mu^\frac 13 } \right)
           % % O(b^\frac 16 |b^2 r^2 - 4E|^{\frac 14}\max_\pm \{e^{\pm \Re \eta}\} ), & r \in [r^*_{b, E},\infty) - I_c^{b, E},\\
           \frac{\sqrt{E} (1-s^2)^\frac 14 e^{\frac{\pi i}{12}}e^{\eta}}{2\sqrt{\pi}\mu^\frac 13 } (1 + O(\eta^{-1}))
           % % \sim b^\frac 16 |b^2 r^2 - 4E|^{\frac 14} e^{-S_{b, E}(r)} , 
           & r\in [0, r^*_{b, E}] - I_c^{b, E},\\
            O(b^\frac 13), & r \in I_c^{b, E};
        \end{cases}\label{eqWKBasymp3d}\\
          (\psi_3^{b, E})'(r)&=& \begin{cases}
            \frac{-i\sqrt{E}(s^2 - 1)^\frac 14 e^{\eta}}{2\sqrt{\pi}\mu^\frac 13} (1 + O(\eta^{-1})) 
            % \sim b^\frac 16 |b^2 r^2 - 4E|^{\frac 14}(br)^{-\frac{\Im E}{b}}, 
            & r \in [4b^{-1}, \infty),\\
            O(\mu^{-\frac 13}|s^2 - 1|^{\frac 14} e^{\eta} )
            % \sim b^\frac 16 |b^2 r^2 - 4E|^{\frac 14}(br)^{-\frac{\Im E}{b}}, 
            & r \in [r^*_{b, E}, 4b^{-1}] - I_c^{b, E},\\
             % \frac{- \sqrt{E}(1-s^2)^\frac 14 e^{\frac{\pi i}{4}}e^{-\eta}}{2\sqrt{\pi}\mu^\frac 13 } (1 + O(\eta^{-1}))
             % % \sim b^\frac 16 |b^2 r^2 - 4E|^{\frac 14} e^{S_{b, E}(r)} , 
             % & r\in [0, r^*_{b, E}] - I_c^{b, E},\\
            O(b^\frac 13), & r \in I_c^{b, E};
        \end{cases}\label{eqWKBasymp1d} 
%                \psi^a_j(r)= \begin{cases}
%            a^\frac 16 |s^2-1|^{\frac 14}, & r \ge \frac 2a + M a^{-\frac 13},\\
%            a^\frac 16 |s^2-1|^{\frac 14} e^{\nu_j S_b(r)} , & r  \le \frac 2a - M a^{-\frac 13 }, \\
%            a^\frac 13, & | r - \frac 2a | \le M a^{-\frac 13 };
%        \end{cases}
        %  (\psi_4^{b, E})'(r)&=& \begin{cases}
        %    % O(b^\frac 16 |b^2 r^2 - 4E|^{\frac 14}\max_\pm \{e^{\pm \Re \eta}\} ), & r \in [r^*_{b, E},\infty) - I_c^{b, E},\\
        %    \frac{-\sqrt{E} (1-s^2)^\frac 14 e^{\frac{\pi i}{12}}e^{-\eta}}{2\sqrt{\pi}\mu^\frac 13 } (1 + O(\eta^{-1}))
        %    % \sim b^\frac 16 |b^2 r^2 - 4E|^{\frac 14} e^{-S_{b, E}(r)} , 
        %    & r\in [0, r^*_{b, E}] - I_c^{b, E},\\
        %     O(b^\frac 13), & r \in I_c^{b, E};
        % \end{cases}\label{eqWKBasymp4d} 
        \eea
        Here we choose $\arg (s^2 - 1) \in (-\pi, \pi]$ when $r > r^*_{b, E}$ and $\arg (1 - s^2) \in (-\pi, \pi]$ when $r < r^*_{b, E}$. If additionally $\Im E \ge - bI_0$, then we have 
        \be \pa_r^k \psi_4^{b, E} =  e^{\frac{\pi i}{6}} \pa_r^k \psi_1^{b, E}(1 + O(\eta^{-1})) \quad {\rm for}\,\, r \in [0, r^*_{b, E}] - I_c^{b, E}, \,\, k \in \{ 0, 1\}. \label{eqWKBasymp4L} \ee
\item Derivatives estimates in the exterior region for $\psi_1^{b, E}$, $\psi_3^{b, E}$: Let $j_+ =3$ and $j_- = 1$. For any $n \ge 1$, 
        \bea
        |D_{\pm; b, E}^n \psi_{j_\pm}^{b, E}(r)| &\lesssim_n& \left| \begin{array}{ll}
             r^{-n} |\psi_{j_\pm}^{b, E}(r)|  &r \ge \frac 4b, \\
             b^n |s-1|^{-n} |\psi_{j_\pm}^{b, E}(r)| & r\in \left[r^*, \frac 4b \right] - I_c^{b, E},
        \end{array}\right. 
        % \sim_E b^\frac 16 r^{-n} (br)^{-\frac 12 \mp \frac{\Im E}{b}},
         \label{eqpsibderiv1} \\
        \left|\left(\pa_r \pm \frac{ibr}{2}\right)^n\psi_{j_\pm}^{b, E}(r)\right| &\lesssim_n& (br)^{-n} |\psi_{j_\pm}^{b, E}(r)| 
        % \sim_E b^\frac 16 (br)^{-n-\frac 12 \mp \frac{\Im E}{b}},
        \quad r \ge \frac 4b.\label{eqpsibderiv2}
        \eea
        where $D_{\pm; b, E} = \pa_r \pm i \left( \frac{b^2 r^2}{4} - E\right)^\frac 12$.
        \item  Derivatives with respect to $E$: $\psi_j^{b, E}(r)$ for $j = 1, 2, 3, 4$, $r > 0$ are analytic with respect to $E$. For any $N \ge 0$, we have the following estimates.
        
        \noindent(a) Boundedness: for $k = 0, 1$, 
        \bea
           \left|\pa_r^k \pa_E^N \psi_1^{b, E}(r)\right| &\lesssim_N& \left\{ \begin{array}{ll}
               (br^2)^N  (br)^k \omega^-_{b, E} & r \ge r^*_{b, E}  \\
               b^{-N} \omega^-_{b, E} & r \le r^*_{b, E}
           \end{array} \right. \label{eqpsibEderiv1} \\
           \left|\pa_r^k \pa_E^N \psi_2^{b, E}(r)\right| &\lesssim_N& b^{-N} \omega^+_{b, E}, \quad r \le r^*_{b, E} \label{eqpsibEderiv2}\\
           \left|\pa_r^k \pa_E^N \psi_3^{b, E}(r)\right| &\lesssim_N& (br^2)^N  (br)^k \omega^+_{b, E},\quad r \ge r^*_{b, E} \label{eqpsibEderiv3} \\
            \left|\pa_r^k \pa_E^N \psi_4^{b, E}(r)\right| &\lesssim_N& b^{-N} \omega^-_{b, E}, \quad r \le r^*_{b, E} \label{eqpsibEderiv4}
        \eea
        % {\color{purple} Moreover, for $k \ge 0, 1$ and $j_+ = 3$, $j_-= 1$, we have improved estimates
        % \be  \left|\pa_r^k \pa_E^N \psi_{j_\pm}^{b, E}(r)\right| \lesssim_N \left(b^{-1}\ln (br)\right)^N (br)^k \omega^\pm_{b, E},\quad r \ge\frac 4b. \label{eqpsibEderiv11} \ee}
        
        \noindent (b) Improved bounds and asymptotics in the interior region: Define the normalization coefficients
\be \kappa_{b, E}^\pm = \left( \frac{e^{\frac{\pi i}{12}} e^{\mp \frac{\pi E}{2b}} }{2\sqrt{\pi} \mu_{b, E}^\frac 13 } \right)^{-1}  = \left( \frac{b^\frac 16 e^{\mp \frac{\pi E}{2b}} }{2^\frac 76 \pi^\frac 12 E^\frac 16} \right)^{-1},  \label{eqnormalcoeff}
    \ee
for $l_- = 4$, $l_+ = 2$, and $k = 0, 1$, we have 
\begin{align}
  \left| \pa_r^k \pa_E^N \left( \kappa_{b, E}^\pm \psi_{l_\pm}^{b, E} \right) \right| &\lesssim_N  r^N |\kappa_{b, E}^\pm| \omega_{b, E}^\pm, \quad r \le r^*_{b, E}; \label{eqpsibEderiv4} \\
   \pa_r^k \pa_E^N \left( \kappa_{b, E}^\pm \psi_{l_\pm}^{b, E} \right) &= \pa_r^k \pa_E^N \left( e^{\pm \sqrt{E} r}\right)+ O_\CC (b^\frac 12 \la r\ra^N e^{\pm \sqrt E r}),\quad r \le b^{-\frac 12}.\label{eqpsibEderiv5}
\end{align}

% {\color{purple} Original version: Define the normalization coefficients
% \be \kappa_{b, E}^\pm = \left( \frac{e^{\pm \frac{\pi i}{12}} e^{\mp \frac{\pi E}{2b}} }{2\sqrt{\pi} \mu_{b, E}^\frac 13 } \right)^{-1}  = \left( \frac{b^\frac 16 e^{-\frac{\pi i}{12} \pm \frac{\pi i}{12}}}{2^\frac 76 \pi^\frac 12 E^\frac 16} e^{\mp \frac{\pi E}{2b}}\right)^{-1},  \label{eqnormalcoeff}
%     \ee
% for $l_- = 1$, $l_+ = 2$, and $k = 0, 1$, we have 
% \begin{align}
%   \left| \pa_E^N \left( \kappa_{b, E}^\pm \psi_{l_\pm}^{b, E} \right) \right| &\lesssim_N  r^N |\kappa_{b, E}^\pm| \omega_{b, E}^\pm, \quad r \le r^*_{b, E}; \label{eqpsibEderiv4} \\
%    \pa_r^k \pa_E^N \left( \kappa_{b, E}^\pm \psi_{l_\pm}^{b, E} \right) &= \pa_r^k \pa_E^N \left( e^{\pm \sqrt{E} r}\right)+ O_\CC (b r^N e^{\pm \sqrt E r}),\quad r \le b^{-\frac 12}.\label{eqpsibEderiv5}
% \end{align}}
        
        \noindent (c) Derivative estimates in the exterior region: Let $j_+ =3$ and $j_- = 1$. For any $n \ge 1$,
        \bea
        |D_{\pm; b, E}^n \pa_E^N \psi_{j_\pm}^{b, E}(r)| &\lesssim_{n, N}& (br^2)^N r^{-n} \omega_{b, E}^\pm,\quad r \ge \frac 4b. \label{eqpsibEderiv6}
        \eea

        \noindent (d) Boundedness of correction:
         \be \begin{split}
         |\pa_E^N h_{b, E}(r)| &\lesssim_N \min\{ b^2, r^{-2} \},\quad r \ge 0; \\ |\pa_r^n \pa_E^N h_{b, E}(r)| &\lesssim_{n, N} r^{-2-n},\quad r \ge \frac 4b, \,\,n \ge 1. \label{eqbddh2}\end{split}\ee

\end{enumerate}
\end{proposition}

% {\color{purple}

% The $\pa_E$ derivative estimates are not sharp - each $\pa_E$ should only bring about $\la \ln (br) \ra$ growth rather than $(br)$ growth. To see this, we notice that 
% \[ \eta = -i\frac 2b \int_{\sqrt E}^{\frac{br}{2}} (z^2 - E)^\frac 12 dz  \quad \Rightarrow \quad \pa_E\eta = i\frac 1b \int_{\sqrt E}^{\frac{br}{2}} (z^2 - E)^{-\frac 12} dz = O(\la \ln(br) \ra). \]
% And each $\pa_E \eta$ appears from $\pa_E e^{\eta}$. This can probably be made rigorous using the Kummer's function representation of $\Ai$: $\Ai(z) = 3^{-\frac 16} \pi^{-\frac 12} \zeta^{\frac 23} e^{-\zeta} U\left(\frac 56, \frac 53, 2\zeta\right)$ where $\zeta = \frac 23 z^\frac 32$ (used in Lemma 4.1). 

% As a corollary, $\pa_E^n \tilde \calT^{ext}_{\infty;b,E}$ is always integrable for just $r^{-2}$ additional decay.
% }

\begin{proof} 
% Above all, we notice that Lemma \ref{lemzeta32} (1) implies $ \left(\frac{s^2 - 1}{\zeta}\right)^{-\frac 14} = \zeta_s^{-\frac 12}$ is analytic with respect to $s \in \CC - (-\infty, -1]$, so $\psi_j^{b, E}$ are well-defined $C^\infty_{loc}(\RR_+ \to \CC)$ function. 
% Note that $\zeta_s > 0$ for $s > -1$, we choose the branch to be analytic extension of 
% $$s \mapsto \zeta_s^{-\frac 12}: (-1, \infty) \mapsto (0, \infty),$$
% {\color{red} I don't know what the following means and whether it's necessary:} which is equivalent to taking
% \be \arg(s^2 -1), \arg \zeta \in (\a -\pi, \a + \pi]\quad {\rm for\,\,} s \in \CC - (-\infty, -1], \label{eqbranch14}\ee
% with arbitrary $\a \in \RR$, 
%  since these two formulations coincide on $s \in (1, \infty)$. 
% The estimate of $I_c^{b, E}$ in general and \eqref{eqIc2} as $|\Im E| \le bI_0$ are straightforward.
To begin with, we require 
\be b_0 \le (10 (I_0 + 1))^{-3} \label{eqrequestb0} 
\ee
to ensure \eqref{eqetaRe}. We will further shrink $b_0$ to ensure $\arg \zeta$ in the correct range for $r \notin I_c^{b, E}$ during the proof of (5). 

\mbox{}

\underline{\textbf{1. Proof of (1)-(4).}} For (1), the equation follows from the derivation of the equation in Remark \ref{rmkWKB} (2). For \eqref{eqbddh}, we first compute from \eqref{eqcorrection} that
\be
 h_{b, E} = \frac{b^2}{4E} \left(-\frac 34 \zeta_s^{-2} \zeta_{ss}^2 + \frac 12 \zeta_s^{-1} \zeta_{sss} \right). \label{eqcorrection2}
\ee
Notice that from the definition of $\zeta$ \eqref{eqzetadef}, and analyticity w.r.t. $s$ and non-vanishing of $\zeta_s$ from Lemma \ref{lemzeta32}, we have for any $k \ge 1$, 
\[ |\zeta_s| \sim 1,\quad |\pa_s^k \zeta| \lesssim_k 1,\quad s \in \{ x+iy \in \CC: x  \in [0, 4], |y| \le 1 \}, \]
which implies the estimate in $[0, \frac 4b]$ of \eqref{eqbddh}. 
For $r \ge \frac 4b$, we have $|s| \ge \frac 32$, $|\arg s| \ll 1$ with $|E - 1| \ll 1$. Thus we can compute $|\zeta^3| \sim |s|^4$ and $|\pa_s^k (\zeta^{-3})| \lesssim_k |s|^{-4-k}$ for $k \ge 1$ from \eqref{eqetaident}. Applying this to the second formulation in \eqref{eqcorrection} yields \eqref{eqbddh} for $r \ge \frac 4b$. 

(2) and (3) directly follow the connection formula and Wronskian for Airy functions in Lemma \ref{lemAiry1} (2) and \eqref{eqdefBi}. 

For (4), when $E = 1$ and $r \in [0, 2b^{-1}]$, we have $s(r) \le 1$, $\zeta(r) \le 0$ and $e^{\frac{2\pi i}{3}} \mu^\frac 43 \zeta \ge 0$, $e^{-\frac{2\pi i}{3}} \mu^\frac 43 \zeta = \overline{ \mu^\frac 43 \zeta}$. So $h_{b, 1}(r) \in \RR$ from its definition \eqref{eqcorrection}. Noticing that $\overline{\Ai(z)} = \Ai(\bar z)$ from the integral definition \eqref{eqdefAi}, combined with \eqref{eqdefBi}, we have
$\psi_2^{b, 1}(r) \in \RR$ and $\psi_4^{b, 1}(r) = \Re(e^\frac{\pi i}{6}\psi_1^{b, 1}(r)) \in \RR$ for $r \le \frac 2b$. Besides, (2) indicates 
\[ e^{\frac{\pi i}{6}}\psi_1^{b, 1} = \psi_4^{b, 1} + \frac i2 \psi_2^{b, 1} = \frac 12 \zeta_s^{-\frac 12} \left( \mathbf{Bi} + i \Ai\right)\left(e^{\frac{2\pi i}{3}} \mu^\frac 43 \zeta\right), \]
Therefore, the non-vanishing of $\psi_1^{b, 1}$ follows that of $|\zeta_s^{-1}|$ from Lemma \ref{lemzeta32} and of $\Ai^2 + \Bi^2$ on $\RR$ from \cite[9.8(iii), 9.9(i)]{MR2723248}.

\mbox{}

\underline{\textbf{2. Proof of (5).}} Now we prove asymptotics of $\psi_j^{b, E}$, $(\psi_j^{b, E})'$ for $j = 1, 2, 3$ via $\eta$. Notice that $\psi_4^{b, E} = e^{\frac{\pi i }{6}} \psi_1^{b, E} - \frac i2 \psi_2^{b, E}$ by \eqref{eqconnect}. The asymptotics of $\pa_r^k \psi_4^{b, E}$ \eqref{eqWKBasymp4L} when $|\Im E| \le bI_0$ follows from those of $\pa_r^k\psi_1^{b, E}$ and $\pa_r^k\psi_2^{b, E}$ thanks to $\Re \eta < 0$ for $r < r^*_{b, E}$ from \eqref{eqReetamono} and the smallness $|e^{\eta}| \lesssim |(\Re \eta)^{-1}| \sim |\eta^{-1}|$  for $r \in [0, r^*_{b, E}] - I_c^{b, E}$ using \eqref{eqetaabs} and \eqref{eqetaRe}.

Before specific computation of asymptotics, we record the basic bounds of $\arg (\mu^\frac 43 e^{-\frac{2\pi i}{3}} \zeta)$ using  \eqref{eqImzetasgn}, \eqref{eqsignImzeta} and \eqref{eqsignImzeta2}: 
\bea
  \arg \left(\mu^\frac 43 \zeta e^{-\frac{2\pi i}3} \right) \left| \begin{array}{ll}
     \in \left( -\frac \pi 3 + \frac 23 \arg E, \frac \pi 3\right)  &  \Im E < 0 \\
     = -\frac \pi 3  &  \Im E = 0 \\
     \in \left(-\pi, -\frac \pi 3 + \frac 23 \arg E\right) & \Im E > 0
  \end{array} \right.\quad r > r^*; \label{eqrangeargself1}\\
  \arg \left(\mu^\frac 43 \zeta e^{-\frac{2\pi i}3} \right) \left| \begin{array}{ll}
     \in \left(\frac \pi 3, \frac{2\pi}3 + \frac 23 \arg E \right)  &  \Im E < 0 \\
     = \frac {2\pi} 3  &  \Im E = 0 \\
     \in \left( \frac{2\pi}{3} + \frac 23\arg E, \pi \right) & \Im E > 0
  \end{array} \right.\quad r < r^*.\label{eqrangeargself2}
\eea
Each interval is the intersection of two intervals given by \eqref{eqImzetasgn}, and one of \eqref{eqsignImzeta}, \eqref{eqsignImzeta2}. For example, for $r > r^*$ and $\Im E < 0$, it is $\left[(0, \pi) -\frac{\pi}{3} + \frac 23 \arg E \right] \cap \left[ (-\pi, 0) + \frac{\pi}{3}  \right] = \left( -\frac \pi 3 + \frac 23 \arg E, \frac \pi 3\right)$. 

% {\color{purple}
% Therefore, the restriction for $\Im E < bI_0$ and $r \notin I_c^{b, E}$ can be viewed purely for getting rid of the region $\arg(\mu^\frac 43 \zeta e^{-\frac{2\pi i}{3}}) \simeq \pi$, where the exponential approximation of $\Ai$ doesn't work.
% }

\mbox{}

\underline{2.1. Asymptotics of $\psi_j^{b, E}$.} 

\textit{(a) On $I_c^{b, E}$.} To begin with, the boundedness of $\psi_j^{b, E}$ with $r \in I_c^{b, E}$ follows easily from boundedness of $|\mu^\frac 43\zeta|$ from \eqref{eqetaabs} and boundedness of $\zeta_s^{-\frac 12}$ from Lemma \ref{lemzeta32} (1). 

\textit{(b) On $r \in [r^*, \infty) - I_c^{b, E}$.} We claim that
\be
  \arg \left(\mu^\frac 43 \zeta e^{-\frac{2\pi i}3} \right) \left| \begin{array}{ll}
     \in \left( -\frac \pi 3 + \frac 23 \arg E, \frac \pi 3\right)  &  \Im E < 0 \\
     = -\frac \pi 3  &  \Im E = 0 \\
     \in \left( -\frac{2\pi}{3}, -\frac \pi 3 + \frac 23 \arg E\right) & 0 < \Im E \le bI_0,
  \end{array} \right.\quad r \in [r^*, \infty) - I_c^{b, E}. \label{eqzetaarg1}
\ee
Compared with \eqref{eqrangeargself1}, it suffices to check the lower bound for $\Im E \in (0, bI_0]$. Indeed, we notice that $r^* \in I_c^{b, E}$ by \eqref{eqr*bE1} when $b \ll 1$, and then $|\arg (s-1)| = \frac{|\Im (E^{-\frac 12})|}{br/2 - \Re (E^{-\frac 12})}  \lesssim_{I_0} b^\frac 23$ for $r \in [r^*,\infty) - I_c^{b, E}$.  Therefore with $b \le b_0(I_0) \ll 1$, \eqref{eqzetaarg} implies the lower bound above. In particular, we are taking $\arg \zeta \in (-\pi, \pi)$.

Now \eqref{eqzetaarg1} and Lemma \ref{lemAiry1} implies the asymptotics of $\psi_1^{b, E}$
\bee
\psi_1^{b, E} = \left( \frac{s^2 - 1}{\zeta}\right)^{-\frac 14} \frac{e^{-\eta}}{2 \sqrt{\pi}( \mu^\frac 43 \zeta e^{-\frac{2\pi i}3})^\frac 14} (1 + O(\eta^{-1})) = \frac{e^{\frac {\pi i}{6}}e^{-\eta}}{2\sqrt{\pi}\mu^\frac 13 (s^2 - 1)^{\frac 14}}(1 + O(\eta^{-1})). 
\eee
Here we are taking $\arg (s^2 - 1) \in (-\pi, \pi)$ from the choice of branch of $\zeta_s = \left(\frac{s^2 - 1}{\zeta}\right)^\frac 12$ in \eqref{eqchoicebranchzetas}. 

% \mbox{}

Next, we prove the estimate for $\psi_3^{b, E}$ \eqref{eqWKBasymp1}. When $r \ge \frac 4b$, since $|\arg (s-1)| \lesssim |\Im E| \lesssim \delta_0$,  \eqref{eqzetaarg} implies that $\arg(\mu^\frac 43 \zeta) \in (\frac \pi 6, \frac \pi 2)$ and $\arg(\mu^\frac 43 \zeta e^{-\frac{2\pi i}{3}}) \in (-\frac \pi 2, -\frac \pi 6)$ with $\delta_0 \ll 1$. Therefore Lemma \ref{lemAiry1} (1) yields the asymptotics of $\psi_3^{b, E}$ noticing that $\frac 23 (\mu^\frac 43 \zeta)^\frac 32 = -\eta$. 
When $r \in [r^*_{b, E}, \frac 4b] - I_c^{b, E}$, we consider two scenarios. (a) If $|\Im E| \le bI_0$, then $|\arg (s-1)| \lesssim_{I_0} b^\frac 23 \ll 1$, so the asymptotics (and thereafter the bound in \eqref{eqWKBasymp1}) follows as in the discussion above. (b) If $\Im E < -bI_0$, then $\arg(\mu^\frac 43 \zeta) \in \left( \frac \pi 3 + \frac 23 \arg E, \pi\right)$ from \eqref{eqrangeargself1}. Noticing that $|\Im \left[\frac 23(-\mu^\frac 43 \zeta)^\frac 32\right]| = |\Re \eta| = \Re \eta$ when $\Im E < 0$ by \eqref{eqReetamono}, the estimate in this scenario again follows Lemma \ref{lemAiry1} (1) plus the elementary bound $|\cos z|, |\sin z| \le e^{|\Im z|}$.

% And with \eqref{eqsignImzeta} and \eqref{eqsignImzeta2} implying $\arg (\mu^\frac 43 \zeta) \in [0, \frac{2\pi}3]$, Lemma \ref{lemAiry1} also yield the asymptotics of $\psi_3^{b, E}$ noticing that $\frac 23 (\mu^\frac 43 \zeta)^\frac 32 = -\eta$ since $\arg(\mu^\frac 43 \zeta), \arg(\mu^\frac 43\zeta e^{-\frac{2\pi i}{3}}) \in (-\frac{5\pi}{6}, \frac{5\pi}6)$. 

\textit{(c) On $r \in [0, r^*]- I_c^{b, E}$.} We claim that
\be
  \arg \left(\mu^\frac 43 \zeta e^{-\frac{2\pi i}3} \right) \left| \begin{array}{ll}
     \in \left(\frac \pi 3, \frac{2\pi}3 + \frac 23 \arg E \right)  &  \Im E < 0 \\
     = \frac {2\pi} 3  &  \Im E = 0 \\
     \in \left( \frac{2\pi}{3} + \frac 23\arg E, \frac {5\pi}{6} \right) & 0 < \Im E \le bI_0,
  \end{array} \right.\quad r \in [0, r^*] - I_c^{b, E}. \label{eqzetaarg2}
\ee
In particular, $\arg \zeta \in (\frac \pi 2, \pi]$ when $\Im E \le 0$ and $\arg \zeta \in (\pi, \frac 76\pi)$ when $\Im E > 0$. 
 This again follows \eqref{eqrangeargself2} plus the sharper upper bound when $\Im E \in (0, bI_0]$ using \eqref{eqzetaarg}, $|\arg (-(s-1))|\lesssim_{I_0} b^\frac 23$ and $b \le b_0(I_0) \ll 1$. 

Next, we claim that Lemma \ref{lemAiry1} implies the asymptotics of $\psi_1^{b, E}$ as in the computation for $r \in [r^*, \infty) - I_c^{b, E}$. The only difference is that we now choose $\arg \zeta \in (0, 2\pi]$ and hence also $\arg (s^2 - 1) \in (0, 2\pi]$ as required by \eqref{eqchoicebranchzetas}. Then 
\[ (s^2 - 1)^{-\frac 14} = \left( (1-s^2) e^{\pi i} \right)^{-\frac 14} = (1-s^2)^{-\frac 14} e^{-\frac {\pi i} 4}, \quad \arg(1-s^2) \in (-\pi, \pi]\]
implies the correct branch of $1-s^2$ and the phase in the numerator. 

For $\psi_2^{b, E}$, by analyticity of $\Ai$, we write $\psi_2^{b, E} = \left(\frac{s^2 -1}{\zeta} \right)^{-\frac 14} \Ai (\mu^\frac 43\zeta e^{-\frac{4\pi i}3})$. So with the same branch of $\zeta$ as \eqref{eqzetaarg2}, we have $\arg \zeta \in (\frac \pi 2, \frac 76\pi) \subset (0, 2\pi]$, $\arg \left(\mu^\frac 43 \zeta e^{-\frac{4\pi i}3} \right) \in (-\frac {\pi}{6}, \frac \pi 2) \subset (-\pi, \pi]$. Thus $\frac 23 \left(\mu^\frac 43\zeta e^{-\frac{4\pi i}3} \right) = -\eta$ which yields the asymptotic formula. 

\mbox{}

\underline{2.2. Asymptotics of $(\psi_j^{b, E})'$.}

Considering $\psi_1^{b, E}$, we compute
\[(\psi_1^{b, E})' = \pa_r \left[ \left( \frac{s^2 - 1}{\zeta}\right)^{-\frac 14} \right] \Ai (\mu^\frac 43 \zeta e^{-\frac{2\pi i}3}) + \left( \frac{s^2 - 1}{\zeta}\right)^{-\frac 14} \mu^\frac 43 \pa_r \zeta e^{-\frac{2\pi i}3} \Ai' (\mu^\frac 43 \zeta e^{-\frac{2\pi i}3}).
\]
We will show the second term provides the leading order in all three ranges of $r$. For $r \in I_c^{b, E}$, the first term is of size $O(b)$ by $\pa_r s(r) = O(b)$, and the second term is bounded by $O(\mu^\frac 43 \pa_r s(r)) = O(b^\frac 13)$. For $r \notin I_c^{b, E}$, from the discussion about phase \eqref{eqzetaarg1} and \eqref{eqzetaarg2}, we apply Lemma \ref{lemAiry1} to see
\[ \Ai' \left(\mu^\frac 43 \zeta e^{-\frac{2\pi i}3}\right) = -\left(\mu^\frac 43 \zeta e^{-\frac{2\pi i}3}\right)^\frac 12 \Ai \left(\mu^\frac 43 \zeta e^{-\frac{2\pi i}3}\right) (1 + O(\eta^{-1})) \]
with $\arg \left(\mu^\frac 43 \zeta e^{-\frac{2\pi i}3}\right) \in (-\frac 56\pi, \frac 56\pi)$. So 
\bee
  (\psi_1^{b, E})' &=& \psi_1^{b, E} \left[ -\mu^\frac 43 \pa_r \zeta e^{-\frac{2\pi i}{3}}  \left( \mu^\frac 43 \zeta e^{-\frac{2\pi i}{3}} \right)^\frac 12 (1 + O(\eta^{-1})) - \frac 12 \zeta_s^{-1} \zeta_{ss} \pa_r s (1 + O(\eta^{-1}))   \right] \\
  &=& \psi_1^{b, E} \mu^2 (s^2 - 1)^\frac 12 \pa_r s \left[ 1 + O_\CC (\eta^{-1}) + O_\CC \left( \frac{\zeta_s^{-1} \zeta_{ss}}{\mu^2 (s^2 - 1)^{\frac 12}} \right) \right]
\eee
Hence to verity the asymptotic formula of $(\psi_1^{b, E})'$, it suffices to prove
\be \left| \frac{\zeta_s^{-1} \zeta_{ss}}{\mu^2 (s^2 - 1)^{\frac 12}} \right| \lesssim |\eta^{-1}|,\quad  r \notin I_c^{b, E}. \label{eqzetasszetas} \ee
Indeed, when $r \le \frac 4b$, we use $\zeta_{ss} \zeta_s^{-1} = O(1)$ from Lemma \ref{lemzeta32}. For $r \ge \frac 4b$, we compute using \eqref{eqzetadef} that
\bee
 \frac{\zeta_{ss}}{\zeta_s} = \zeta_s^{-1} \pa_s\left[\left( \frac{s^2 - 1}{\zeta}\right)^\frac 12 \right] = \frac{2s \zeta^\frac 32 - (s^2 - 1)^\frac 32}{2(s^2 - 1) \zeta^\frac 32} = O(s^{-1})
\eee
where in the last step we used \eqref{eqzetaasymp}. These estimates and \eqref{eqetaabs} yield the bound \eqref{eqzetasszetas} and conclude the proof for  $(\psi_1^{b, E})'$ case. 

The cases of $(\psi_2^{b, E})'$ and $(\psi_3^{b, E})'$ are similar and thus omitted. 

\mbox{}

\underline{\textbf{3. Proof of (6).}} Note that \eqref{eqpsibderiv2} can be derived from \eqref{eqpsibderiv1} 
% \be  |D_{\eta,\pm}^n \psi_{j_\pm}^b(r)| \lesssim_n r^{-n} |\psi_{j_\pm}^b(r)|,\quad \forall\,\, r \in I_4. \label{eqpsibderiv3} \ee
% with 
% \[ D_{\eta, \pm} = e^{\mp \eta} \pa_r e^{\pm \eta} = \pa_r \pm i\left( \frac{b^2 r^2}{4} - E \right)^\frac 12, \]
using
\be 
% \left| \pa_r^k \left[ \left( \frac{b^2 r^2}{4} - E \right)^\frac 12 - \left( \frac{b^2 r^2}{4} - 1 \right)^\frac 12 \right] \right| \lesssim_k r^{-k-1},\quad 
\left| \pa_r^k \left[ \left( \frac{b^2 r^2}{4} - E \right)^\frac 12 - \frac{br}{2} \right] \right| \lesssim_k (br)^{-k-1},\quad r \ge \frac 4b,\quad k \ge 0. \label{eqphasematcheibx2} \ee
Now we focus on \eqref{eqpsibderiv1}. 

We first claim the estimate of $\pa_s^k \zeta$ for $r \ge 0$ that
\be |\pa_s \zeta| \sim \la br \ra^{\frac 13};\quad  |\pa_s^k \zeta| \lesssim_k \la br \ra^{\frac 43 - k},\quad \forall \,\,k\ge 0.\label{eqzetaderivest} \ee
Indeed, when $r \le \frac 4b$, we have $|s| \le 8$, so the upper bound of $\pa^k_s \zeta$ for all $k \ge 0$ follows from the analyticity of $\zeta$, and the non-vanishing of $\pa_s\zeta$ implies its lower bound. For $r \ge \frac 4b$, the asymptotics of $\pa_s \zeta = \left(\frac{s^2 - 1}{\zeta} \right)^\frac 12$ follows \eqref{eqzetaasymp}, and we estimate $\pa_s^k \zeta$ inductively using
\[ \pa_s^k \zeta = \pa_s^{k-1} \zeta_s = \pa_s^{k-1} \left[ (s^2 - 1)^\frac 12 \zeta^{-\frac 12} \right].\]
 
Next, let 
 $\xi_{j_-} = e^{-\frac{2\pi i}{3}}\mu^\frac 43 \zeta$, $\xi_{j_+} = \mu^\frac 43 \zeta$. 
 We compute using \eqref{eqzetadef} and Fa\'a di Bruno's formula \eqref{eqFaadiBruno}
 \bee\pa_r &=& \pa_r (\xi_{j_\pm}) \pa_{\xi_{j_\pm}} = \pm i \sqrt{E} (s^2 - 1)^\frac 12 \xi_{j_{\pm}}^{-\frac 12} \pa_{\xi_{j_\pm}},\\
 \pa_r^n &=& \sum_{\sum_{k=1}^n km_k = n} \frac{n!}{m_1! m_2! \cdots m_n!} \left(\prod_{i=1}^n \left(\pa_r^i \xi_{j_\pm}\right)^{m_i}\right) \pa_{\xi_{j_\pm}}^{m_1 + \cdots + m_n} \eee
 and therefore 
 \bee
  D_{\pm; b, E} &=& \pm i \sqrt{E} (s^2 - 1)^\frac 12 \xi_{j_{\pm}}^{-\frac 12}  \left(\pa_{\xi_{j_\pm}} + \xi_{j_\pm}^{\frac 12}\right)\\
 D_{\pm; b, E}^n &=& e^{-\frac 23 \xi_{j_\pm}^\frac 32} \circ \left(\pm i \sqrt E (s^2 - 1)^\frac 12 \xi_{j_{\pm}}^{-\frac 12} \pa_{\xi_{j_\pm}}\right)^n \circ e^{\frac 23 \xi_{j_\pm}^\frac 32} = e^{-\frac 23 \xi_{j_\pm}^\frac 32} \circ \pa_r^n  \circ e^{\frac 23 \xi_{j_\pm}^\frac 32} \\
 &=& e^{-\frac 23 \xi_{j_\pm}^\frac 32} \circ \left[\sum_{\sum_{k=1}^n km_k = n} \frac{n!}{m_1! m_2! \cdots m_n!} \left(\prod_{i=1}^n \left(\pa_r^i \xi_{j_\pm}\right)^{m_i}\right) \pa_{\xi_{j_\pm}}^{m_1 + \cdots + m_n} \right] \circ e^{\frac 23 \xi_{j_\pm}^\frac 32} \\
 &=& \sum_{\sum_{k=1}^n km_k = n} \frac{n!}{m_1! m_2! \cdots m_n!} \left(\prod_{i=1}^n \left(\pa_r^i \xi_{j_\pm}\right)^{m_i}\right) \left( \pa_{\xi_{j_\pm}} + \xi_{j_\pm}^\frac 12\right)^{\sum_{k=1}^n m_k}.
 \eee
 From \eqref{eqzetaderivest}, we have $|\pa_r^i \xi_{j_\pm}| \lesssim_i \mu^\frac 43 |\pa_s^i \zeta| b^i \lesssim_i b^{\frac 23} r^{\frac 43 - i}$ for any $i \ge 1$, $r \ge b^{-1}$. Hence with \eqref{eqAiryderiv} and \eqref{eqzetaasymp},
 \bee
 \left| D_{\pm; b, E}^n \Ai (\xi_{j_\pm}) \right| &\lesssim_n&  \sum_{\sum_{k=1}^n km_k = n} \left(b^\frac 32 r^\frac 43 \right)^{\sum_k m_k} r^{-\sum_k km_k} |\xi_{j_\pm}|^{-\sum_k m_k} |\Ai(\xi_{j_\pm})| \\
 &\lesssim_n& r^{-n} |\Ai(\xi_{j_\pm})| \sum_{n' = 1}^n \left((br)^{\frac 43} \zeta^{-1}\right)^{n'} \\
 &\lesssim_n& |\Ai(\xi_{j_\pm})| \cdot \left|\begin{array}{ll} 
 b^n |s-1|^{-n} & r \in \left[r^*, \frac 4b \right] - I_c^{b, E},\\
 r^{-n} & r \ge \frac 4b.
 \end{array}\right. 
 \eee
 % So when $D_{\pm; b, E}$ hits $\Ai(\xi_{j,\pm})$, using \eqref{eqAiryderiv}, we gain decay
 % \be
 % \left| (s^2 - 1)^\frac 12 \xi_{j_{\pm}}^{-\frac 32}\right| = \left|(s^2 - 1)^\frac 12 \eta^{-1} \right| \sim \left|\begin{array}{ll} 
 % b|s-1|^{-1} & r \in \left[r^*, \frac 4b \right] - I_c^{b, E} \\
 % r^{-1} & r \ge \frac 4b.
 % \end{array}\right. \label{eqDpmEdecay}
 % \ee
 % The last asymptotics exploits \eqref{eqzetaasymp}. 
Finally, by Leibniz rule
\bee
D_{\pm; b, E}^k \psi_{j_\pm}^{b, E} = \sum_{l=0}^k \binom{k}{l} D_{\pm; b, E}^l \left(\Ai (\xi_{j_\pm})\right) \pa_r^{k-l} \left( \zeta_s^{-\frac 12}\right) 
\eee
and $|\pa_r^n (\zeta_s^{-\frac 12})| \lesssim_n |\zeta_s^{-\frac 12}| \cdot r^{-n}$ by \eqref{eqzetaderivest}, we obtain the desired estimate \eqref{eqpsibderiv1}.

\mbox{}

\underline{\textbf{4. Proof of (7).}} For simplicity, for estimates in (a)-(c), we only prove the case of $\psi_1^{b, E}$  (for (b), replacing $\psi_4^{b, E}$ by $ e^\frac{\pi i}{6}\psi_1^{b, E}$ thanks to \eqref{eqconnect}). The proof for $\psi_2^{b, E}$ and $\psi_3^{b, E}$ follows similarly, and that for $\psi_4^{b, E}$ comes from the connection formula \eqref{eqconnect} and smallness of $|e^{\Re \eta}|\le 1$ on $r \in [0,  r^*_{b, E}]$ from \eqref{eqReetamono}. Also, we will apply the Fa\'a di Bruno's formula \eqref{eqFaadiBruno} without citing.

Denote $\xi(r) = e^{-\frac{2\pi i}{3}} \mu^\frac 43 \zeta(r)$ as $\xi_1$ in the above proof of (6). Then for $|E - 1| \le \delta_0 \ll 1$ and $\Im E \le b I_0$, from the argument estimates \eqref{eqzetaarg1}-\eqref{eqzetaarg2}, the definition of $I_c^{b, E}$ \eqref{eqIcdef} and \eqref{eqzetaasymp}, we have 
\be \xi(r) \in B_{C_0} \cup \left\{ z \in \CC: |\arg z| \le \frac 56\pi \right\},\quad \forall \, r > 0  \label{eqargxi}\ee
for a constant $C_0 > 0$ independent of $b, r, E$.

% We first notice that the weight function $\omega_{b, E}^\pm$ is chosen such that the $N=0$ case of \eqref{eqpsibEderiv1}-\eqref{eqpsibEderiv3} hold from the asymptotics in (6). 

\mbox{}

\underline{4.1. Derivative estimates for regular functions.} 

We claim for $n \ge 0$, $\a > 0$, $k = 0, 1$ that
\bea\left|\pa_r^k \pa_E^n (\zeta_s)^{-\a}\right| \lesssim_{\a,n} b^k \la s \ra^{-\frac \a 3-k}, \quad \left|\pa_r^k \pa_E^n \xi \right| \lesssim_n b^{-\frac 23+k} \la s \ra^{\frac 43-k},\quad  \forall \, r > 0;\label{eqbddpalzetaxi} \\
\left| \pa_E^n \zeta_{ss} \right| + \left| \pa_E^n \zeta_{sss} \right| \lesssim_n 1,\quad \forall\, r \le \frac 4b;\label{eqbddpalzetaxi5}\\
\left| \pa_E^n \left( e^{-\frac{2E}{b} \int_0^s  (1-w^2)^\frac 12 dw } \right) \right| \lesssim_{n} r^n \left| e^{-\frac{2E}{b} \int_0^s  (1-w^2)^\frac 12 dw } \right|,\quad \forall \, r \le r^*_{b, E}, \label{eqbddpalzetaxi4}
\eea
and for $n \ge 0$, $k = 0, 1$, $\a \notin \NN_{\ge 0}$,  
\be
\left| \begin{array}{l}
\left|\pa_r^k \pa_E^n (1 - s^2)^{\a} - \delta_{n, 0}\delta_{k, 0} \right| \lesssim_{\a, n, k} |s|^2 b^k,\\
% \forall\, r\le b^{-1};
\left| \pa_r^k \pa_E^n\left(e^{\frac{2E}{b} \int_0^{s} [1 - (1-w^2)^\frac 12] dw } \right) - \delta_{n, 0}\delta_{k, 0}\right| \lesssim_{n,k} b^{-1+k}|s|^3,
% \\
% \left| \pa_r^k \pa_E^n \xi \right|\lesssim_{n, k} b^{-\frac 23+k} ,
\end{array}\right.\quad \forall\, r\le b^{-1};
 \label{eqbddpalzetaxi3}
\ee
and for $n \ge 0$, $k \ge 0$, 
\be
\left|\pa_r^k \pa_E^n (\zeta_s)^{-\frac 12}\right| \lesssim_{n,k}  |s|^{-\frac 16}r^{-k}, \,\,
\left|\pa_r^k \pa_E^n (\zeta^{-3}) \right| \lesssim_{n,k} |s|^{-4}r^{-k},
\,\,
\left|\pa_r^k \pa_E^n \xi \right| \lesssim_{n,k} b^{-\frac 23} | s |^{\frac 43}r^{-k},\quad \forall \, r \ge \frac 4b.\label{eqbddpalzetaxi2}\ee
% {\color{purple} and for $n \ge 1$, 
% \be
% |\pa_E^{n} \eta_{b, E}| \lesssim_{n, k} b^{-1} \left( \ln(br)\right)^{\delta_{n, 1}},\quad  |\pa_r \pa_E^{n} \eta_{b, E}| \lesssim_{n, k} (br)^{1-2n},\quad \forall\,\, r \ge \frac 4b.\label{eqbddpalzetaxi6}
% \ee}

Indeed, we recall \eqref{eqzetaderivest} and use $| \pa_E^m s| \lesssim_m |s|$ $\forall m \ge 0$ to control for $n, m_0 \ge 0$,
\bee
 \pa_E^n (\pa_s^{m_0} \zeta) = \sum_{\sum_{k=1}^n km_k = n} C_{n, \vec m}
\pa_s^{\sum_{k=0}^n m_k}\zeta \prod_{j = 1}^k (\pa_E^j s)^{m_j} = O\left(\la s \ra^{\frac 43 - m_0}\right), \quad \forall\, r \ge 0.
\eee
The estimate of $\pa_r^k \pa_E^n (\zeta_s)^{-\a}$ with $k = 0, 1$ follows similarly using $|\pa_s^m (\zeta_s)^{-\a}| \lesssim_{\a, m} \la s \ra^{-\frac \a 3 - m}$ from \eqref{eqzetaderivest}. Noticing that $\xi = e^{-\frac{2\pi i}{3}} \mu^\frac 43 \zeta$ and $\pa_r \xi = e^{-\frac{2\pi i}{3}} \mu^\frac 43 \zeta_s \frac{b}{2\sqrt{E}}$, they imply \eqref{eqbddpalzetaxi} and \eqref{eqbddpalzetaxi5}. And \eqref{eqbddpalzetaxi2} follows by further using $|\pa_r^k \pa_E^n s| \lesssim b r^{1-k}$ for all $k, n \ge 0$. 

% {\color{purple}For the estimate of $\eta$ \eqref{eqbddpalzetaxi6}, we first apply the formula \eqref{eqeta00} to rewrite for $r \ge \frac 4b$
% \bee
%  \eta_{b, E}(4b^{-1}) &=& \left( \eta_{b, E}(4b^{-1})  - \eta_{b, E}(3b^{-1}) \right) + \eta_{b, E}(3b^{-1})  \\
%  &=& -i\frac{2E}{b} \int_{\frac3{2\sqrt E}}^{\frac{br}{2\sqrt E}} (w^2 - 1)^\frac 12 dw  - i\frac{2E}{b} \int^{\frac3{2\sqrt E}}_1 (w^2 - 1)^\frac 12 dw \\
%  &=&  -i\frac{2}{b} \int_{\frac3{2}}^{\frac{br}{2}} (w^2 - E)^\frac 12 dw  - i\frac{2E}{b} \int^{\frac3{2\sqrt E}}_1 (w^2 - 1)^\frac 12 dw.
% \eee
% Obviously, $\pa_E^n \eta_{b, E}(3b^{-1}) = O(b^{-1})$, and $\pa_r \pa_E^n \eta_{b, E}(3b^{-1}) = 0$. For the first term, we compute for $n \ge 1$ that
% \bee
%   \pa_E^n \left( \eta_{b, E}(4b^{-1})  - \eta_{b, E}(3b^{-1}) \right) = -i\frac 2b \left[\prod_{k=0}^{n-1} (-\frac 12 + k)\right] \int_{\frac3{2}}^{\frac{br}{2}} (w^2 - E)^{\frac 12- n} dw.
%   \eee
%   That easily implies the desired estimates in \eqref{eqbddpalzetaxi6}. }

Next, for the first estimate in \eqref{eqbddpalzetaxi3}, we compute for $n \ge 1$, and $\a \notin \NN_{\ge 0}$, 
\bee
  \pa_E^n (1-s^2)^{\a} =  \sum_{\sum_{k=1}^n km_k = n} C_{n, \vec m} \left[ \pa_z^{\sum_{k=1}^n m_k} ((1-z)^{\a}) \right] \bigg|_{z = s^2} \prod_{k =1}^{n} ( \pa_E^k (s^2) )^{m_k} = O(s^2),\quad r \le b^{-1},
\eee
and similarly for $\pa_r \pa_E^n (1-s^2)^\a = -\frac{\a b^2 r}{2} \pa_E^n \left( (1-s^2)^{\a-1} E^{-1} \right)$. 

For \eqref{eqbddpalzetaxi4}, we exploit $\pa_E^n e^{f(E)} = e^{f(E)}  \sum_{\sum_{k=1}^n km_k = n} C_{n, \vec m} \prod_{k =1}^{n} (\pa_E^k f)^{m_k}$, and compute for $r \le r^*$ 
% \bee
%   \pa_E^n \left( e^{-\frac{2E}{b} \int_0^s  (1-w^2)^\frac 12 dw } \right) = e^{-\frac{2E}{b} \int_0^s  (1-w^2)^\frac 12 dw }  \sum_{\sum_{k=1}^n km_k = n} C_{n, \vec m} \prod_{k =1}^{n} \left[\pa_E^k (-\frac{2E}{b} \int_0^s  (1-w^2)^\frac 12 dw )\right]^{m_k}
% \eee
% and 
\bee
 \pa_E^k \int_0^s (1-w^2)^\frac 12 dw = \left| \begin{array}{ll}
      O(br)  & k = 0, \\
      -\frac{br}{4} \pa_E^{k-1} \left( E^{-\frac 32} (1-s^2)^\frac 12 \right) = O(br) & k \ge 1,
 \end{array}\right.
\eee
where we used $ \pa_E^n (1-s^2)^{\frac 12} = \delta_{n,0} + O(s^2)$.
For the second estimate of \eqref{eqbddpalzetaxi3}, we compute for $r \le b^{-1}$,  
\bee
 \pa_E^k \int_0^s (1 - (1-w^2)^\frac 12) dw = \left| \begin{array}{ll}
      \int_0^s O(w^2) dw = O(s^3) & k = 0, \\
      -\frac{br}{4} \pa_E^{k-1} \left( E^{-\frac 32} (1 - (1-s^2)^\frac 12) \right) = O(s^3) & k \ge 1.
 \end{array}\right.
\eee

\mbox{}

\underline{4.2. Proof of (a) and (c).} The $N=0$ case is included in the asymptotics in (7), so we only consider $N \ge 1$. 
 Apply \eqref{eqAiryderivident} to compute
 \bea
\pa_E^N \psi_1^{b, E} = \sum_{n = 0}^N \sum_{\sum_{k=1}^n k m_k = n} \binom{N}{n} n! \pa_E^{N-n}   \left(\zeta_s \right)^{-\frac 12} \pa_\xi^{\sum_k {m_k}} \Ai(\xi) \prod_{k=1}^n (\pa_E^k \xi)^{m_k} (m_k!)^{-1}\nonumber \\
=  \sum_{n = 0}^N \sum_{\sum_{k=1}^n k m_k = n} C_{N,n,\vec m} \pa_E^{N-n}  \left(\zeta_s \right)^{-\frac 12}  \left( P_{\sum m_k} \Ai + Q_{\sum m_k} \Ai'  \right)(\xi) \prod_{k=1}^n (\pa_E^k \xi)^{m_k} \label{eqpaENpsi1bE}
 \eea
 where $C_{N, n, \vec m} > 0$ are universal constants.  
Recall from Lemma \ref{lemAiry1}(1) and the range of $\xi$ \eqref{eqargxi}, that $|\pa_\xi^k \Ai(\xi)| \lesssim \left|e^{-\frac 23 \xi^{\frac 32}}\right| \la \xi \ra^{-\frac 14 + \frac k2}$ with $k = 0, 1$ for $r > 0$. Thus combined with \eqref{eqbddpalzetaxi} and degree of $P_n$, $Q_n$ from Lemma \ref{lemAiry1}(3), we have 
\bee
  \left| \pa_E^N \psi_1^{b, E}\right| \lesssim_N \la s \ra^{-\frac 16} \left(b^{-\frac 32} \la s \ra^{\frac 43} \right)^{\frac 32 N - \frac 14}   \left|e^{-\frac 23 \xi^{\frac 32}}\right| \sim b^{-N} \la br \ra^{2N} \omega_{b, E}^-,
\eee
where for the second inequality, we used $\left|e^{-\frac 23 \xi^{\frac 32}}\right| = e^{-\Re \eta_{b, E}}$ when $|\arg \xi| \le \frac 56\pi$, and $\la \xi \ra^{-\frac 14} \la s \ra^{-\frac 16} \sim \la b^{-\frac 32} (s^2 - 1)\ra^{-\frac 14}$ from \eqref{eqzetaasymp}. This verifies the $k = 0$ case of \eqref{eqpsibEderiv1}. The $k = 1$ case follows the observation that applying $\pa_r$ to $\pa_E^N \psi_1^{b, E}$ creates at most $\la \xi \ra^{\frac 12} \pa_r \xi = O(\la s \ra)$ growth when hitting the $\pa_\xi^{\sum_k m_k} \Ai$ term.

% {\color{purple} For the improved estimate \eqref{eqpsibEderiv11}, we exploit the Kummer's function formulation of $\Ai$ \eqref{eqAiKummer1}. Hence with \eqref{eqbddpalzetaxi6} and \eqref{eqAiKummer3}, we compute 
% \bee
%  &&\pa_E^n \Ai(\xi) = \pa_E^n \left(3^{-\frac 16} \pi^{-\frac 12} \eta^{\frac 23} e^{-\eta} U\left(\frac 56, \frac 53, 2\eta\right) \right) \\
%  &=& 3^{-\frac 16} \pi^{-\frac 12} \sum_{\sum_{k=1}^n km_k = n } C_{n, \vec m} \cdot \pa_\eta^{\sum_k m_k} \left[ \eta^{\frac 23} e^{-\eta} U\left(\frac 56, \frac 53, 2\eta\right)\right] 
%  \cdot \prod_{j=1}^n \left(\pa_E^j \eta \right)^{m_j} \\
%   &=& O\left( \eta^{-\frac 16} e^{-\eta} \left( b^{-1} \ln(br) \right)^n \right) = O\left( \Ai(\xi)  \left( b^{-1} \ln(br) \right)^n\right),\quad {\rm for }\,\, r \ge \frac 4b
% \eee}

For (c), we apply $D_{-;b,E}^n$ to the above formula for $\pa_E^N \psi_1^{b, E}$ \eqref{eqpaENpsi1bE} and attribute the derivatives through Leibniz rule similarly as in (6). We gain $r^{-1}$ when the derivative hits $\Ai(\xi)$ or $\Ai'(\xi)$ as $D_{-;b,E}$ (noticing that \eqref{eqAiryderiv} holds for both $\Ai$ and $\Ai'$), and gain $r^{-1}$ as well when the derivative hits $\pa_E^k \xi$, $P_{\sum m_k}(\xi)$, $Q_{\sum m_k}(\xi)$ or $ \pa_E^{N-n}  \left(\zeta_s \right)^{-\frac 12}$ due to \eqref{eqbddpalzetaxi2}.

\mbox{}

\underline{4.3. Proof of (b).} The first estimate \eqref{eqpsibEderiv4} for $r \in [b^{-1}, r^*_{b, E}]$ follows from \eqref{eqpsibEderiv1} and $|\pa_E^n \kappa_{b, E}^\pm| \lesssim b^{-n} |\kappa_{b, E}^\pm|$. Now we consider \eqref{eqpsibEderiv4} for $r \le b^{-1}$ and \eqref{eqpsibEderiv5}. Recall that \eqref{eqzetaarg} implies $\arg \xi(r) = \frac{2\pi}{3} + O(\delta_0^\frac 12)$ and thus $\arg \zeta, \arg (s^2 - 1) = \pi + O(\delta_0^\frac 12)$. We compute $\psi_1^{b, E} =  \Ai(\xi) \left( \frac{e^{\pi i} ( 1-  s^2)}{e^{\frac{2\pi i}{3}} \mu^{-\frac 43} \xi}  \right)^{-\frac 14} = \Ai(\xi) \xi^{\frac 14} (1- s^2)^{-\frac 14} e^{-\frac{\pi i}{12}} \mu^\frac 13$ and $\frac 23 \xi^\frac 32 = \eta = -\frac{\pi E}{2b} + \frac{2E}{b} \int_0^s (1-w^2)^\frac 12 dw$ similar to \eqref{eqetacomputerleb-1}, leading to that
\bea
e^{\frac{\pi i}{6}}\kappa_{b, E}^- \psi_1^{b, E} &=& \left( 2 \sqrt{\pi} \xi^\frac 14 e^{\frac 23 \xi^\frac 32} \Ai(\xi) \right) (1-s^2)^{-\frac 14} e^{-\frac{2E}{b} \int_0^s  (1-w^2)^\frac 12 dw } \label{eqnormpsi1rep}\\
&=& e^{-\sqrt{E}r} \left( 2 \sqrt{\pi} \xi^\frac 14 e^{\frac 23 \xi^\frac 32} \Ai(\xi) \right) (1-s^2)^{-\frac 14} e^{\frac{2E}{b} \int_0^s [1 - (1-w^2)^\frac 12 ] dw }. \label{eqnormpsi1rep2}
\eea
% using $\frac 23 \xi^{\frac 32} = \frac{2E}{b} \int_1^s (1-w^2)^\frac 12 ds = -\frac{\pi E}{2b} + \sqrt{E}r  + \frac{2E}{b} \int_0^s[(1-w^2)^\frac 12 - 1] ds$. 
% {\color{purple} We used $\psi_1^{b, E} =  \Ai(\xi) \left( \frac{e^{\pi i} ( 1-  s^2)}{e^{\frac{2\pi i}{3}} \mu^{-\frac 43} \xi}  \right)^{-\frac 14} = \Ai(\xi) \xi^{\frac 14} (1- s^2)^{-\frac 14} e^{-\frac{\pi i}{6}} \left(\frac{2E}{b} \right)^\frac 16$ and $-\frac{\pi E}{2b} - \eta_{b, E} = \frac{2E}{b} \int_0^s [1 - (1-w^2)^\frac 12 ] dw$.}

From \eqref{eqbddpalzetaxi} and \eqref{eqAiryderiv2}, we have for $n \ge 0$ and $k= 0, 1$, 
\bee
  &&\pa_r^k \pa_E^n \left( 2 \sqrt{\pi} \xi^\frac 14 e^{\frac 23 \xi^\frac 32} \Ai(\xi) \right)\\
  &=& \sum_{\sum_{j=1}^n jm_j = n} C_{n, \vec m} \pa_r^k \left[ \pa_\xi^{\sum_{m_j}} \left( 2 \sqrt{\pi} \xi^\frac 14 e^{\frac 23 \xi^\frac 32} \Ai \right) \prod_{j=1}^n (\pa_E^j \xi)^{m_j}\right]\\
  &=& \delta_{n,0} \delta_{k, 0} + O(\xi^{-\frac 32} b^{k})  = \delta_{n,0}\delta_{k, 0} + O(b^{1+k}),\qquad \forall\, r \le b^{-1}. 
\eee
Then taking $\pa_E^n$ on \eqref{eqnormpsi1rep} and applying \eqref{eqbddpalzetaxi4} yields \eqref{eqpsibEderiv4} with $\psi_4^{b, E}$ replaced by $e^{\frac{\pi i}{6}} \psi_1^{b, E}$, since each derivative $\pa_E$ creates at most $r^1$ growth no matter it hits which term in \eqref{eqnormpsi1rep}. 

Next, we take $\pa_E^N$ on \eqref{eqnormpsi1rep2} and  apply \eqref{eqbddpalzetaxi3} to obtain 
\bee
  \pa_E^N \left(e^{\frac {\pi i}{6}} \kappa_{b, E}^- \psi_1^{b, E} \right) &=& \pa_E^N  e^{-\sqrt{E}r}  \cdot (1 + O(b)) \cdot (1 + O(s^2)) \cdot (1 + O(b^{-1}s^3)) \\
  &+& \sum_{n = 1}^{N} \pa_E^{N-n} e^{-\sqrt{E}r} \cdot \left(   O(b) + O(s^2)+ O(b^{-1}s^3) \right) \\
  &=&  \pa_E^N  e^{-\sqrt{E}r} \cdot (1 + O(b^\frac 12)),\quad r \le b^{-\frac 12}. 
\eee
which is \eqref{eqpsibEderiv5} with $k=0$. For the $k=1$ case, we take $\pa_r \pa_E^N$ to \eqref{eqnormpsi1rep2}. The control of $\pa_r^k \pa_E^n \left( 2 \sqrt{\pi} \xi^\frac 14 e^{\frac 23 \xi^\frac 32} \Ai(\xi) \right)$ and \eqref{eqbddpalzetaxi4} show that when $\pa_r$ does not hit $\pa_r^{N-n} e^{-\sqrt{E} r}$ it creates $b^1$ smallness. So the above estimate yields \eqref{eqpsibEderiv5} with $k=1$. 

\mbox{}

\underline{4.4. Proof of (d).} For $r \le \frac 4b$, the $b^2$ smallness follows taking $\pa_E^n$ on the formula for $h_{b, E}$ \eqref{eqcorrection2} and applying estimates \eqref{eqbddpalzetaxi5}, \eqref{eqbddpalzetaxi2}. For $r \ge \frac 4b$, from the formula \eqref{eqcorrection}, it suffices to show 
$$ \left| \pa_r^n \pa_E^N \left( \frac{s^2 - 1}{\zeta^3}\right)\right| + \left| \pa_r^n \pa_E^N \left( \frac{3 s^2 + 2}{(s^2 - 1)^2}\right)\right| \lesssim_{n, N} b^{-2}r^{-2-n} ,\quad \forall\, r \ge \frac 4b,  $$ 
which follows \eqref{eqbddpalzetaxi2}, $|\pa_r^n \pa_E^N (s^2)| \lesssim_{n, N} b^2 r^{2-n}$ and $\left|\pa_z^m \left(\frac{3z+2}{(z-1)^2} \right)\right| \lesssim_m |z|^{-1-m}$ for $|z| \ge 4|E|^{-1} \ge 2$. 

\end{proof}

\subsubsection{Inversion of corrected scalar operator}

For $b > 0$, $|E- 1| \le \frac 12$, we define the scalar differential operator with correction in \eqref{eqcorrection} as
\be 
 \tilde H_{b, E} =  \pa_r^2 - E + \frac{b^2 r^2}{4} -  h_{b, E}. \label{eqdeftildeHbE}
\ee 
And we will construct its inversion on the exterior and middle region.

\mbox{}

We start from exterior region $[r^*_{b, E}, \infty)$. We introduce the following definitions to characterize functions with polynomial growth or decay and quadratic phase oscillation near infinity. 

\begin{definition}[Differential operators and Banach spaces for exterior region] \label{defdiffopspace}Let $b > 0$ and $|E-1| \le \frac 12$. We define the phase function
\be \phi_{b, E} := \frac{E}{2b} \int_2^{\frac{br}{\sqrt E}} \sqrt{\tau^2 - 4} d\tau,\quad  \Rightarrow \quad \pa_r \phi_{b, E} = \left(\frac{b^2 r^2}{4} - E \right)^\frac 12, \label{eqdefphi} \ee
    the differential operators (same as in Proposition \ref{propWKB} (6))
\be
\begin{split}
D_{\pm k; b, E} = \pa_r \pm k i \pa_r \phi_{b, E}  = e^{\mp k i \phi_{b, E}} \pa_r e^{\pm k i\phi_{b, E}},\quad k \ge 1.
%  D_{\pm; b, E} &= \pa_r \pm i \pa_r \phi_{b, E}  = e^{\mp i \phi_{b, E}} \pa_r e^{\pm i\phi_{b, E}},\\
%  D_{\pm \pm; b, E} &= \pa_r \pm 2i \pa_r \phi_{b, E}  = e^{\mp 2i \phi_{b, E}} \pa_r e^{\pm 2i\phi_{b, E}}, 
 \end{split}
\label{eqDpm}
\ee
In particular, we denote
\[D_{\pm;b, E} = D_{\pm 1;b ,E},\quad D_{\pm\pm;b, E} = D_{\pm 2;b ,E}.  \]
We also define 
%the weight function on $r > 0$ (same as in Proposition \ref{propWKB} (6))
%\be
% \omega^\pm_{b, E}(r) =  \left\la b^{-\frac 23} (b^2r^2 - 4E) \right\ra^{-\frac 14} e^{\pm \Re \eta_{b, E}(r)}, \tag{\ref{eqomegapm}}
%\ee
%with $\eta_{b, E}$ from \eqref{eqetadef}, 
 the Banach spaces $X_{r_1; b, E}^{\a, N, \pm} \subset  C^N([r_1, \infty))$ and $X_{r_0, r_1; b, E}^{\a, N, \pm} \subset C^0([r_0, \infty)) \cap C^N ([r_1, \infty))$ with norms
\be
\begin{split}
% X_{r_1; b, E}^{\a, N, \pm} \subset  C^N([r_1, \infty)),\quad  X_{r_0, r_1; b, E}^{\a, N, \pm} \subset C^0([r_0, \infty)) \cap C^N ([r_1, \infty)), \\
\| f \|_{X_{r_1; b, E}^{\a, N, \pm}} &= \sum_{k = 0}^N \sup_{r \ge r_1} \left( (r^\a \omega_{b, E}^\pm)^{-1} |r - 2\sqrt E b^{-1}|^{k} |D_{\pm; b, E}^k f|\right);    \\
\| f \|_{X_{r_0, r_1; b, E}^{\a, N, \pm}} &= \left\| (r^\a \omega_{b, E}^\pm)^{-1} f \right \|_{C^0_{[r_0, r_1]}} + \| f \|_{X_{r_1; b, E}^{\a, N, \pm}}.
\end{split}\label{eqdefXBanach}
% \begin{split}
%  X_{r_1; b, E}^{\a, N, \pm} &= \left\{ f \in C^N([r_1, \infty)) : \| f \|_{X_{r_1; b, E}^{\a, N, \pm}} = \sum_{k = 0}^N \sup_{r \ge r_1} \left( \frac{r^{-\a}|r - 2\sqrt E b^{-1}|^{k} |D_{\pm; b, E}^k f|}{\omega_{b, E}^\pm}\right) < \infty \right\}; \\
% X_{r_0, r_1; b, E}^{\a, N, \pm} &= \bigg\{ f \in C^0([r_0, \infty)) \cap C^N ([r_1, \infty)): \\
% &\qquad \| f \|_{X_{r_0, r_1; b, E}^{\a, N, \pm}} = \left\| \frac{f}{r^\a \omega_{b, E}^\pm}\right \|_{C^0_{[r_0, r_1]}} + \| f \|_{X_{r_1; b, E}^{\a, N, \pm}} < \infty \bigg\}.
% \end{split}\label{eqdefXBanach}
\ee
where $r_1 > \frac{2\Re \sqrt E}{b}$, $r_0 \in (0, r_1)$, $N \in \NN$, $\a \in \RR$, and $\omega_{b, E}^\pm$ is from \eqref{eqomegapm}.
\end{definition}

\begin{remark}
%     \mbox{}

%     \begin{enumerate}
%         \item If $r_1 \ge \frac 4b$ and $|E-1| \le \frac 12$, we have 
% \[ \| f \|_{X_{r_1; b, E}^{\a, N, \pm}} \sim_{N, \a} \sum_{k = 0}^N \sup_{r \ge r_1} \left(\frac{r^{-\a + k} |D_{\pm; b, E}^k f|}{\omega_{b, E}^\pm}\right). \]
% \item
From Proposition \ref{propWKB} (5)-(6), we see for any $I_0 > 0$, $0 < b \le b_0(I_0)$, $|E-1|\le \delta_0$, $\Im E \le bI_0$, $N>1$, with $r_0 = r^*_{b, E} - b^{-\frac 13}$, $r_1 = r^*_{b, E} + 2 b^{-\frac 13} \in [r_0, \frac 4b] - I_c^{b, E}$,   
\be
 \left\| \psi_1^{b, E} \right\|_{X^{0, N, -}_{r_0, r_1;b, E}} \sim_N 1, \quad \left\| \psi_3^{b, E} \right\|_{X^{0, N, +}_{r_0, r_1;b, E}} \sim_N 1. \label{eqestpsi13Xspace}
 % \| P_b\|_{X^{-\frac 12+s_c, N}_{\frac 4b}} \lesssim_N b^{-\frac 12}e^{-\frac{\pi}{2b}}
\ee
    % \end{enumerate}
\end{remark}

Now we construct the inversion of $\tilde H_{b, E}$ in $[r^*_{b, E} - b^{-\frac 13}, \infty)$.

\begin{lemma}[Inversion of $\tilde H_{b, E}$ on $[r^*_{b, E} - b^{-\frac 13}, \infty)$]\label{leminvtildeHext}

For any $I_0 > 0$, let $0 < b \le b_0$, $|E-1|\le \delta_0$ and $\Im E \le b I_0$ with $\delta_0 \ll 1$ and $b_0 = b_0(I_0) \ll 1$ from Proposition \ref{propWKB}. 
For $\a \in \RR, N \in \ZZ_{\ge 0}$ satisfying
\be \a \ge -2, \quad
\max \left\{\frac{2\Im E}{b}, 0\right\} + \max \{ \a, 0\} + \frac 12 < 2 N, \label{eqcondinv} \ee
and any $r_0 \in \left[ r^*_{b, E} - b^{-\frac 13}, \frac 4b\right]$, $r_1 \in \left[\max \{r_0, r^*_{b, E} + 2b^{-\frac 13}\}, \frac 4b\right]$,
we define the inversion operator of $\tilde H_{b, E}$ for $ f \in X^{\a, N, -}_{r_0, r_1 ;b, E}$ as 
\be \label{eqdeftildeHbnu}
  \tilde T_{r_0; b, E}^{ext} f =: \left| \begin{array}{ll}
      \psi_3^{b, E} \calI^-_{N; b, E} [f](r) + \psi_1^{b, E} \int_{r_0}^r \psi_3^{b, E} f \frac{ds}{W_{31}}  & r \ge r_1 \\
      \psi_3^{b, E} \left(\calI^-_{N; b, E} [f]\left(r_1\right) + \int^{r_1}_r \psi_1^{b, E} f \frac{ds}{W_{31}} \right) + \psi_1^{b, E} \int_{r_0}^r \psi_3^{b, E} f \frac{ds}{W_{31}}  & r \in \left[r_0, r_1 \right]
  \end{array}\right.
\ee
where $W_{31} = \calW (\psi_3^{b, E}, \psi_1^{b, E}) = i\frac{b^\frac 13 E^\frac 16}{2^\frac 43 \pi}$ from \eqref{eqWronski2} and the integral operator being
\be
\begin{split}
  \calI^-_{N; b, E} [f](r) &= \int_r^\infty \left(D_{--} \circ (-2i\phi')^{-1}\right)^N (\psi_1^{b, E} f) \frac{ds}{W_{31}} \\
  &+ \sum_{l = 0}^{N-1} \left(D_{--} \circ (-2i\phi')^{-1}\right)^l (\psi_1^{b, E} f) \cdot (-2i\phi' W_{31})^{-1}.
  \end{split} \label{eqdefcalIn}
\ee
Then the following properties hold. We stress that the bounds below in (3), (4) are dependent of $I_0, \a, N$, but independent of $b, E, r_0, r_1$. 
\begin{enumerate}
    \item {\rm Well-definedness of $\calI_N^-$}: for $f \in  X^{\a, N, -}_{r_0, r_1;b, E}$, the integral in $\calI^-_N$ converges absolutely; and the definition of $\calI^-_N$ is independent of $N$ in the following sense: if $N',N \in \ZZ_{\ge 0}$, $f \in X^{\a, \max\{N, N'\}, -}_{r_0, r_1;b, E}$, then
    \be
      \max \left\{\frac{2\Im E}{b}, 0\right\} + \max \{ \a, 0\}+ \frac 12 <2 \min \{ N, N'\} \,\, \Rightarrow \,\,  \calI^-_{N; b, E}[f] = \calI^-_{N'; b, E}[f]. \label{eqindepN}
    \ee
    \item {\rm Well-definedness and inversion property of $\tilde T^{ext}_{r_0;b, E}$}: For any fixed $r_1^* \ge \max \{r_0, r^*_{b, E}+ 2b^{-\frac 13}\}$, $\tilde T_{r_0; b, E}^{ext} f$ can also be uniformly written for $r \ge r_0$ as
    \be \label{eqdeftildeHbnu2}
     \tilde T_{r_0; b, E}^{ext} f =  \psi_3^{b, E} \left(\calI^-_{N; b, E} [f]\left(r_1^*\right) + \int^{r_1^*}_r \psi_1^{b, E} f \frac{ds}{W_{31}} \right) + \psi_1^{b, E} \int_{r_0}^r \psi_3^{b, E} f \frac{ds}{W_{31}}.
     \ee
    Hence the definition of $\tilde T_{r_1; b, E}^{ext}$ is independent of $r_1$, and it is an inverse of $\tilde H_{b, E}$ in the sense that
    $$\tilde H_{b, E} \tilde T_{r_0; b, E}^{ext} f = f. $$
    \item {\rm Boundedness}:
    \be
    % \begin{split}
 \| \tilde T_{r_0; b, E}^{ext} f \|_{X^{\a+2, N, -}_{r_0, r_1;b, E}} \lesssim b \| f \|_{X^{\a, N, -}_{r_0, r_1;b, E}},\quad 
  \| \tilde T_{r_0; b, E}^{ext} f \|_{X^{\a+2, N+1, -}_{r_0, r_1;b, E}} \lesssim  \| f \|_{X^{\a, N, -}_{r_0, r_1;b, E}}
    % \end{split}
      % \| \tilde T_{r_1; b, E}^{ext} \|_{\calL\left( X^{\a, N, -}_{r^*_{b, E}, \frac 4b;b, E}\right)} \lesssim b^{-1},\quad \| \tilde T_{r_1; b, E}^{ext} \|_{ X^{\a, N, -}_{r^*_{b, E}, \frac 4b;b, E} \to X^{\a+2, N+1, -}_{r^*_{b, E}, \frac 4b;b, E}} \lesssim 1.
      \label{eqbddHbnuinvext}
    \ee
    For $r_0' \in [r_0, \frac 4b]$, we also have 
    \be 
    \| \tilde T_{r_0; b, E}^{ext} (\mathbbm{1}_{[r_0, r_0']} f) \|_{C^0_{\omega_{b, E}^\pm}([r_0, r_0'])} \lesssim b^{-1} \| f \|_{C^0_{\omega_{b, E}^\pm}([r_0, r_0'])},
\label{eqbddHbnuinvext2}
    \ee
    \item {\rm Evaluation at $r_0 \in [r_0, r_1]$}: for $k = 0, 1$, 
 \be
      \pa_r^k \tilde T_{b, E}^{ext} f \left(r \right) = \beta (r) \pa_r^k \psi_1^{b, E} (r) + \gamma(r)\pa_r^k \psi_3^{b, E} (r),\label{eqasymptildeHbnu}
    \ee
    where
    \bee
    \gamma(r) &=& \calI^-_N [f]\left(r_1\right) +  \int_{r}^{r_1} \psi_1^{b, E} f \frac{ds}{W_{31}} \\
    &=& e^{-2\Re \eta_{b, E}(r_1)}\|f \|_{X^{\a, N, -}_{r_0, r_1;b, E}}O( b^{-\a - \frac 12} |r_1 - 2\sqrt E b^{-1}|^{\frac 12}) \\
    \beta(r) &=& \int_{r_0}^{r} \psi_3^{b, E} f \frac{ds}{W_{31}} = \|f \|_{X^{\a, N, -}_{r_0, r_1;b, E}}O( b^{-\a - \frac 12} |r_1 - 2\sqrt E b^{-1}|^{\frac 12}) 
    % \label{eqasymptildeHbnu2}
    \eee
    
%     \item \textit{Boundary value at $r^*$}: {\color{blue} This whole part is only required for construction of bifurcated eigenmodes.} When $N \ge 2$, we have 
%     \be
%       \tilde T_{b, E}^{ext} f \left( \frac 4b \right) = \beta \psi_3^{b, E} \left( \frac 4b \right),\quad  (\tilde T_{b, E}^{ext} f)' \left( \frac 4b \right) = \beta (\psi_3^{b, E})' \left( \frac 4b \right)  \label{eqasymptildeHbnu}
%     \ee
% where\footnote{\color{blue}The boundedness of $\tilde T_{b, E}^{ext}$ on $I_{ext}$ loses $b^{-1}$, while the asymptotics at $r = \frac 4b$ does not lose, since $\calI^-_N$ is the good part.}{\color{blue} Suffices to have estimate, no need asymptotics.}
%   \be \beta = \calI^-_N [f]\left(\frac 4b \right) + \int_{\frac 2b - M b^{-\frac 13}}^{\frac 4b} \psi_1^{b, E} f \frac{ds}{W_{31}} 
%   {\color{blue}= \| f \|_{X^{\a, N}_{\frac 4b; \frac 2b - M b^{-\frac 13}}} O \left( b^{-\a - 1 - \frac 16} \right) } \label{eqasymptildeHbnu2}
%    \ee
  
\end{enumerate}
\end{lemma}

\begin{proof} 

(1) We apply \eqref{eqpsibderiv1}, \eqref{eqpsibEderiv1} to compute for $0 \le n \le N$ and $r \ge r_1$,  
\bee
  \left|D_{--}^n (\psi_1^{b, E} f)\right| \le \sum_{k = 0}^n \binom{n}{k} |D_-^k \psi_1^{b, E}| \cdot |D_-^{n-k} f | \lesssim_n r^\a |r - 2\sqrt E b^{-1}|^{- n}(\omega_{b, E}^-)^2 \| f \|_{ X^{\a, N, -}_{r_1;b, E}}. 
  % &\lesssim_{b, E}& r^{-1+2\frac{\Im E}{b} + \a - N} \| f \|_{ X^{\a, N, -}_{r^*, \frac 4b}} \in L^1\left(\left[ \frac 4b, \infty \right) \right)
\eee
Further with Lemma \ref{lemgnk}, we compute
\bea
&& \left| \left(D_{--} \circ (-2i\phi')^{-1}\right)^N (\psi_1^{b, E} f) \right| \le \sum_{k = 0}^N |g_{N, k}| \cdot |D_{--}^k (\psi_1^{b, E} f)| \nonumber \\
 &\lesssim_N& b^{-N} r^{\a-\frac N2} |r-2\sqrt{E}b^{-1}|^{-\frac 32N} (w_{b, E}^-)^2  \| f \|_{ X^{\a, N, -}_{r_1;b, E}}.  \label{eqestD--N}
 % \nonumber\\
 % &\le&  \sum_{k = 0}^N |g_{N, k}| \sum_{j = 0}^k \binom{k}{j} |D_{-}^j \psi_1^{b, E}| \cdot |D_-^{k-j} f| \nonumber  \\ 
 % &\lesssim_N& \sum_{k = 0}^N b^{-N} r^{k-2N} \cdot \sum_{j = 0}^k b^\frac 13 r^{-j} (br)^{-1} |e^{-2\eta}| \cdot r^{\a - (k-j)} \| f \|_{ X^{\a, N, -}_{r^*, \frac 4b}} \nonumber\\
 % &\lesssim& b^{-N- \frac 23 } |e^{-2\eta}| r^{-1 -2N + \a} \| f \|_{ X^{\a, N, -}_{r^*, \frac 4b}} 
\eea
Since $\omega_{b, E}^- \lesssim b^{-\frac 13} r^{-\frac 12} e^{-\Re \eta_{b, E}(r)} \lesssim_{b, E} r^{-\frac 12 + \frac{\Im E}{b}}$ for $r \ge \frac 4b$ by from \eqref{eqetaRe}, we see $\eqref{eqestD--N} \in L^1\left(\left[ r_1, \infty \right) \right)$ by the numerology \eqref{eqcondinv}. 
Thus $\calI^-_N[f]$ is well-defined.

To show \eqref{eqindepN}, it suffices to prove for $N$ satisfying \eqref{eqcondinv}, we have $\calI^-_{N+1}[f] = \calI^-_{N}[f]$. Indeed, this follows 
\bee
    &&  \int_r^\infty \left(D_{--} \circ (2i\phi')^{-1}\right)^N \left(\psi_1^{b, E} f\right)ds\\
   &=& - \int_r^\infty \left(D_{--} \circ (2i\phi')^{-1}\right)^{N+1} \left(\psi_1^{b, E} f\right)ds + \int_r^\infty \pa_r \left[ (2i\phi')^{-1}\left(D_{--} \circ (2i\phi')^{-1}\right)^{N} \left(\psi_1^{b, E} f\right)\right]ds \\
   &=&- \int_r^\infty \left(D_{--} \circ (2i\phi')^{-1}\right)^{N+1} \left(\psi_1^{b, E} f\right)ds - \left(D_{--} \circ (2i\phi')^{-1}\right)^{N} \left(\psi_1^{b, E} f\right)\cdot (2i\phi')^{-1}.
   \eee

\mbox{}

(2) For \eqref{eqdeftildeHbnu2}, it suffices to notice that 
\bee
   W_{31} \pa_r  \calI^-_N[f](r) 
   &=& \left(D_{--} \circ (-2i\phi')^{-1}\right)^N \left( \psi_1^{b, E} f \right) \nonumber \\
   &+& \sum_{l = 0}^{N-1} \left[ \left(D_{--} \circ (-2i\phi')^{-1}\right) - 1 \right] \left(D_{--} \circ (-2i\phi')^{-1}\right)^l \left( \varphi_1^{b, E} f \right) \nonumber \\
   &=&-\varphi_1^{b, E} f 
 \eee

 \mbox{}

 (3) We first prove \eqref{eqbddHbnuinvext}. It involves three steps.   
 
 \textit{Step 1. Computation and estimates of $D_{--}^k \calI^-_N[f](r)$ for $r \ge r_1$.}

 First notice that when $r \in [r_1, \infty) \subset [r^*_{b, E}, \infty) - I_c^{b, E}$, we have 
 \bea \omega_{b, E}^\pm \sim b^{-\frac 13} r^{-\frac 14} |r - 2\sqrt{E}b^{-1}|^{-\frac 14} e^{\pm \Re \eta_{b, E}(r)}, \label{eqIcoutomega} \\
|r-2\sqrt E b^{-1}|^{3} \gtrsim b^{-1}. \label{eqIcout2}
 \eea

For notational simplicity, let $G = \psi_1^{b, E} f W_{31}^{-1}$. Define 
\be I_n[f] = \left| \begin{array}{ll}
     \int_r^\infty (D_{--} \circ (-2i\phi')^{-1})^{N} G ds +   (-2i\phi')^{-1} \sum_{l = n}^{N-1} (D_{--} \circ  (-2i\phi')^{-1})^{l} G & 0 \le n \le N-1, \\
    \int_r^\infty (D_{--} \circ  (-2i\phi')^{-1})^{N} G ds  & n = N,
\end{array} \right. \label{eqDNIn} \ee
or equivalently,
\bee
  I_n[f] &=& \left| \begin{array}{ll}
     \calI^-_N[f] -  (-2i\phi')^{-1} \sum_{l = 0}^{n-1}(D_{--} \circ  (-2i\phi')^{-1})^{l}G  & 1 \le n \le N, \\
      \calI^-_N[f] & n = 0.
  \end{array}\right.
\eee
% And we define 
% \[ I_{N+1} = (ir)^{-1} D_{--} I_N = (ir)^{-1} (D_{--} \circ (ir)^{-1})^{N} (v_2 f r^2 W^{-1}) + \int_r^\infty (D_{--} \circ (is)^{-1})^{N} (v_2 f s^2 W^{-1}) ds \]

For $0 \le n \le N-1$, using the integration by parts formula \eqref{eqD++commutator1} and \eqref{eqindepN}, we compute
\bea
 D_{--} I_n[f] 
 &=& (-2i\phi') \int_r^\infty (D_{--} \circ (-2i\phi')^{-1})^{N+1} G ds + \sum_{l = n}^{N-1} (D_{--} \circ (-2i\phi')^{-1})^{l+1} G  \nonumber \\
 % &&-\sum_{l=0}^{n-1}  (D_{--} \circ (-2i\phi')^{-1})^{l+1} G \nonumber \\
  &=&  (-2i\phi') \left( \calI^-_{N+1} [f] -  (-2i\phi')^{-1} \sum_{l = 0}^{n}(D_{--} \circ (-2i\phi')^{-1})^{l} G \right) \nonumber  \\
  &=& -2i\phi' I_{n+1}[f],\label{eqD--Initer}
\eea
where in the last equivalence we used $\calI_{N+1}^-[f] = \calI_N^-[f]$ which holds for $f \in X^{\a, N+1, -}_{r_0, r_1;b, E}$, and the case $f \in X^{\a, N, -}_{r_0, r_1;b, E}$ follows by a standard limiting argument. 
This shows that $\{ I_n[f]\}_{0 \le n \le N}$ satisfies the algebraic recurrence relation \eqref{eqalgrec}. Hence for $1 \le n \le N$, with $I_n[f]$ from \eqref{eqDNIn}, Lemma \ref{lemfnk} indicates
\be
  D_{--}^n \calI^-_N[f] = \sum_{l = 1}^{n} f_{n, l} I_l[f], \quad 1 \le n \le N.  \label{eqD++DN}
\ee
Now we can estimate for $1 \le n\le N-1$ and $r \ge r_1$ that
\bee 
  &&|D_{--}^n \calI^-_N[f](r)|\\
 &\lesssim& \sum_{l=1}^n b^l r^{\frac l2} |r-2\sqrt{E}b^{-1}|^{\frac{3l}{2}-n} \bigg[ \int_r^\infty b^{-N-1} \tau^{\a-\frac N2-\frac 12} |\tau-2\sqrt{E}b^{-1}|^{-\frac 32N-\frac 12}  e^{-2 \Re \eta_{b, E}(\tau)} d\tau \\
 && \qquad +\sum_{k=l}^{N-1} b^{-k-2} r^{\a-\frac k2 - 1}  |r-2\sqrt{E}b^{-1}|^{-\frac 32k-1} e^{-2 \Re \eta_{b, E}(r)}\bigg]\| f \|_{ X^{\a, N, -}_{r_1;b, E}} \\
 &\lesssim&  e^{-2\Re \eta_{b, E}(r)}\cdot  \sum_{l= 1}^n \Big[ b^{-(N-l-1)-2} r^{\a-\frac{(N-l-1)}2-1} |r-2\sqrt{E}b^{-1}|^{- \frac 32(N-l-1)-n-1} \\
 &&\qquad + \sum_{k=l}^{N-1}  b^{-(k-l)-2}r^{\a-\frac{k-l}{2}-1}|r-2\sqrt{E}b^{-1}|^{-\frac 32(k-l)-n-1} \Big]\| f \|_{ X^{\a, N, -}_{r_1;b, E}} \\
 &\lesssim& b^{-2} r^{\a-1}|r-2\sqrt{E}b^{-1}|^{- n - 1} e^{-2\Re \eta_{b, E}(r)} \| f \|_{ X^{\a, N, -}_{r_1;b, E}} 
\eee
where the first inequality follows \eqref{eqD++DN}, \eqref{eqestD--N}, \eqref{eqestfnk} and \eqref{eqIcoutomega}, and the last follows \eqref{eqIcout2}. For the second inequality, 
the control of $e^{-2\Re \eta_{b, E}(s)}$ comes from its monotonicity \eqref{eqReetamono} for $\Im E < 0$ or $|e^{-2\eta(s)}| \sim (br)^{2\frac{\Im E}{b}}$ from \eqref{eqetaRe} for $0 \le \Im E \le bI_0$ with $b \le b_0$ small enough \eqref{eqrequestb0}, and then we integrate with the admissible numerology \eqref{eqcondinv}. Similarly, for $n=0$ we directly have 
\[ |\calI_N^-[f](r)| \lesssim b^{-2} r^{\a-1} |r-2\sqrt{E}b^{-1}|^{- 1} e^{-2\Re \eta_{b, E}(r)} \| f \|_{ X^{\a, N, -}_{r_1;b, E}}. \]
And for $n = N$, the dominant term is given by the integration, so 
\bee
|D_{--}^N \calI_N^-[f](r)| \lesssim b^{-1} r^{\a-\frac 12} |r-2\sqrt{E}b^{-1}|^{ - N + \frac 12} e^{-2\Re \eta_{b, E}(r)} \| f \|_{ X^{\a, N, -}_{r_1;b, E}}.
\eee
To sum up, 
\be
|D_{--}^n \calI^-_N[f](r)| \lesssim 
    b^{-1} r^{\a -\frac 12} |r - 2\sqrt E b^{-1}|^{-n+\frac 12} e^{-2\Re\eta_{b,E}(r)}| \cdot \| f \|_{ X^{\a, N, -}_{r_1;b, E}},\quad  0 \le n \le N,\,\, r \ge r_1.
\label{eqestD--IN}
\ee

 For $n = N+1$, we apply \eqref{eqD++DN} and \eqref{eqD--Initer} to compute
\bee
  D_{--}^{N+1} \calI^-_N[f] &=& D_{--} \sum_{l = 1}^{N} f_{N, l} I_l[f] = \sum_{l = 1}^N \left(\pa_r f_{N, l} I_l[f] + f_{N, l} D_{--} I_l[f] \right) \\
  &=& \sum_{l = 1}^N \pa_r f_{N, l} I_l[f] + (-2i\phi') \sum_{l = 1}^{N-1} f_{N, l} I_{l+1}[f] \\
  &+& f_{N, N} \left( - (D_{--} \circ  (-2i\phi')^{-1})^{N} G - 2i\phi' \int_r^\infty (D_{--} \circ  (-2i\phi')^{-1})^{N} G ds \right)
\eee
Then similar estimates lead to 
\be
|D_{--}^{N+1} \calI^-_N[f](r)| \lesssim 
 r^\a |r-2\sqrt E b^{-1}|^{-(N+1)+2} e^{-2\Re\eta_{b, E}(r)} \cdot \| f \|_{ X^{\a, N, -}_{r_1;b, E}},\quad r \ge r_1. 
\label{eqestD--IN+1}
\ee
where the worst decay is given by the last term $f_{N, N}\cdot(-2i\phi')I_N[f]$.

\mbox{}

\textit{Step 2. Estimates for $r \ge r_1$.} By Leibniz rule,
\bee
 D_{-}^k (\tilde T_{b, E}^{ext} f) = \sum_{j =0}^k \binom{k}{j} \left( D_+^{k-j} \psi_3^{b, E} \cdot D_{--}^{j} \calI^-_N[f] + D_-^{k-j} \psi_1^{b, E} \cdot \pa_r^{j} \int_{\frac 4b}^r \psi_3^{b, E} f \frac{ds}{W_{31}} \right) 
\eee
The last integral is bounded by 
\bee
 \left|\int_{r_1}^r \psi_3^{b, E} f \frac{ds}{W_{31}}\right| &\lesssim& b^{-1} \int_{r_1}^r s^{\a-\frac 12} |s-2\sqrt E b^{-1}|^{-\frac 12} ds  \| f \|_{ X^{\a, N, -}_{r_1;b, E}} \\
 &\lesssim& b r^{\a + \frac 32} |r - 2 \sqrt E b^{-1}|^{\frac 12 } \| f \|_{ X^{\a, N, -}_{r_1;b, E}} \\
 % &\lesssim&
 % \left| \begin{array}{ll}
 %    b^{-1} r^\a   \| f \|_{ X^{\a, N, -}_{r^*, \frac 4b}}   &  \a > \frac 34 \\
 %    r^{\a + 1} \| f \|_{ X^{\a, N, -}_{r^*, \frac 4b}}   & \a \in [-\frac 14, \frac 34] \\
 %    b^{-\a - 1} \| f \|_{ X^{\a, N, -}_{r^*, \frac 4b}}  & \a \in [-2, -\frac 14)
 % \end{array}\right.
 % \lesssim b r^{\a + 2} \| f \|_{ X^{\a, N, -}_{r^*, \frac 4b}} \\
  \left|\pa_r^{j}\int_{r_1}^r \psi_3^{b, E} f \frac{ds}{W_{31}}\right| &\lesssim& 
  \sum_{m =0}^{j-1} |D_+^{m}\psi_3^{b, E}| |D_-^{j-1-m} f| b^{-\frac 13} \\
 & \lesssim & b^{-1} r^{\a-\frac 12} |r - 2\sqrt E b^{-1}|^{-j+ \frac 12} \| f \|_{ X^{\a, N, -}_{r_1;b, E}},\quad 1 \le j \le N+1.
  % \left| \begin{array}{ll}
  %     \int_{\frac 4b}^r | \psi_3^{b, E}| |f| b^{-\frac 13} ds  & j=0 \\
  %     \sum_{m =0}^{j-1} |D_+^{m}\psi_3^{b, E}| |D_-^{j-1-m} f| b^{-\frac 13} & j \ge 1
  % \end{array}\right.  \\
  % &\lesssim &
  %    \left| \begin{array}{ll}
  %       b^{-1} r^{-j+\a}  \| f \|_{ X^{\a, N, -}_{r^*, \frac 4b}}  & 0 \le j \le N, \\
  %        r^{-N+1+\a} \| f \|_{ X^{\a, N, -}_{r^*, \frac 4b}} & j = N+1,
  %    \end{array}\right.  \qquad r \ge \frac 4b,  
\eee
where we used $\a \ge -2$ and the cancellation $e^{\Re\eta_{b, E}} \cdot e^{-\Re\eta_{b, E}} = 1$. Now \eqref{eqpsibderiv1}, \eqref{eqestD--IN}, \eqref{eqestD--IN+1} and the above estimate yields 
\be
\left| D_{-}^k (\tilde T_{r_1; b, E}^{ext} f) \right| \lesssim \left| \begin{array}{ll}
b r^{\a + \frac 32} |r - 2 \sqrt E b^{-1}|^{-k + \frac 12} \omega_{b, E}^-(r) \| f \|_{ X^{\a, N, -}_{r_1;b, E}} & 0 \le k \le N, \\
     % b r^{-k+\a+2} \omega_{b, E}^{-}(r) \| f \|_{ X^{\a, N, -}_{r^*, \frac 4b}}  & 0 \le k \le N, \\
     r^\a |r - 2 \sqrt E b^{-1}|^{-k + 2} \omega_{b, E}^-(r) \| f \|_{ X^{\a, N, -}_{r_1;b, E}} & k = {N+1},
\end{array}\right.\,\,  r\ge r_1.\label{eqtildeTbEbdd1}
\ee

\mbox{}

\textit{Step 3. Estimates for $r \in [r_0, r_1] \subset [r^*_{b, E} - b^{-\frac 13}, \frac 4b]$.} 

We first estimate the integral coefficients in \eqref{eqdeftildeHbnu}
\bee
  && \left|\int_r^{r_1} \psi_1^{b, E}f \frac{ds}{W_{31}}\right| \lesssim  \int_r^{r_1} b^{-1} s^{\a - \frac 12} |s - 2\sqrt{E}b^{-1}|^{-\frac 12}  e^{-2\Re\eta_{b, E}(s)} ds \cdot \| f \|_{ X^{\a, N, -}_{r_0, r_1; b, E}}\\
  % &\lesssim& b^{-\a} |e^{-2\eta(r)}| \| f \|_{ X^{\a, N, -}_{r^*, \frac 4b}} \cdot  \int_r^\frac 4b \left| b^2 s^2 - 4E \right|^{-\frac 12} ds \\
  &\lesssim&  b^{-\a-\frac 12} |r_1 - 2\sqrt{E}b^{-1}|^{\frac 12}  e^{-2\Re\eta_{b, E}(r)} \| f \|_{ X^{\a, N, -}_{r_0, r_1; b, E}} 
\eee
where for $e^{-2\Re\eta_{b, E}(s)}$ term, we again exploited the monotonicity \eqref{eqReetamono} for $\Im E < 0$ and $|e^{-2\eta(s)}| \sim (br)^{2\frac{\Im E}{b}} \sim 1$ on $[r^*_{b, E} - b^{-\frac 13}, \frac 4b]$ from \eqref{eqetaRe} for $0 \le \Im E \le bI_0$. Similarly, for $r \in [r_0, r_1]$, 
\bee
   \left| \int_{r_0}^{r} \psi_3^{b, E}f \frac{ds}{W_{31}}\right| \lesssim  \int_{r^*_{b, E}}^{r_1} 
   \left| \psi_3^{b, E}f W_{31}^{-1}\right| ds  \lesssim b^{-\a-\frac 12} |r_1 - 2\sqrt{E}b^{-1}|^{\frac 12} \| f \|_{ X^{\a, N, -}_{r_0, r_1; b, E}}.
\eee
Recall \eqref{eqestD--IN} yields $\left| \calI^-_N[f]\left( \frac 4b \right) \right| \lesssim b^{-\a-\frac 12} |r_1 - 2\sqrt{E}b^{-1}|^{\frac 12}  e^{-2\Re\eta_{b, E}(r_1)} \cdot \| f \|_{ X^{\a, N, -}_{r_0, r_1; b, E}}$. Plugging these estimates into \eqref{eqdeftildeHbnu}, we obtain
\be
  \left| \tilde T_{r_1; b, E}^{ext} f \right| \lesssim b^{-\a-\frac 12} |r_1 - 2\sqrt{E}b^{-1}|^{\frac 12}  \omega_{b, E}^{-}(r) \| f \|_{ X^{\a, N, -}_{r_0, r_1; b, E}}, \quad r \in [r_0, r_1]. \label{eqtildeTbEbdd2}
\ee
The boundedness \eqref{eqbddHbnuinvext} follows \eqref{eqtildeTbEbdd1} and \eqref{eqtildeTbEbdd2}. 

\mbox{}

For \eqref{eqbddHbnuinvext2}, notice that by picking $r_1 = \frac 4b \ge r_0'$, the inversion operator becomes the common Duhamel formula 
\[ \tilde T_{r_0; b, E}^{ext} (\mathbbm{1}_{[r_0, r_0']} f) = \psi_3^{b, E} \int_r^{r_0'} \psi_1^{b, E} f \frac{ds}{W_{31}} + \psi_1^{b, E} \int_{r_0}^r \psi_3^{b, E} f \frac{ds}{W_{31}}. \]
Then as in Step 3 above, we can exploit the monotonicity or smallness of $\Re \eta_{b, E}$ to obtain the estimates 
\bee
  \left|\int_r^{r_0'} \psi_1^{b, E}f \frac{ds}{W_{31}}\right| &\lesssim& b^{-1}  e^{(-1 \pm 1) \Re \eta_{b, E}(r)}\|f \|_{C^0_{\omega^\pm_{b, E}} ([r_0, r_0'])} \\ 
  \left|\int_{r_0}^{r} \psi_3^{b, E}f \frac{ds}{W_{31}}\right| &\lesssim& b^{-1} e^{(1 \pm 1) \Re \eta_{b, E}(r)}\|f \|_{C^0_{\omega^\pm_{b, E}} ([r_0, r_0'])},
\eee
which leads to \eqref{eqbddHbnuinvext2}.

\mbox{}

(4) The formulation \eqref{eqdeftildeHbnu2} yields \eqref{eqasymptildeHbnu}, and the estimates of $\beta$, $\gamma$  follows the estimates in (3). 
\end{proof}

\mbox{}

Next, we consider the intermediate region $[x_*, r^*_{b, E} + b^{-\frac 13}]$ for $1 \ll x_* \ll b^{-1}$. 
% Recall the weight function $\omega^\pm_{b, E}$ from \eqref{eqomegapm} is also well-defined on $r \in (0, r^*_{b, E})$, such that
% \be |\psi_1^{b, E}| \lesssim \omega^-_{b, E},\quad  |\psi_2^{b, E}| \lesssim \omega^+_{b, E} \ee
% according to Proposition \ref{propWKB}. 

% Now we write the inversion of $\tilde H_{b, E}$ on $r \le r^*_{b, E}$. 

\begin{lemma}[Inversion of $\tilde H_{b, E}$ for $r \le r^*_{b, E} + b^{-\frac 13}$] \label{leminvtildeHmid}
For any $I_0 > 0$, let $0 < b \le b_0(I_0) \ll 1$ and $|E-1| \le \delta_0$, $|\Im E| \le b I_0$ with $\delta_0 \ll 1$ and $b_0 = b_0(I_0) \ll 1$ from Proposition \ref{propWKB}. For $[x_*, x^*] \subset [10, r^*_{b, E} + b^{-\frac 13}]$, define the inversions of $\tilde H_{b, E}$ on $[x_*, x^*]$ as\footnote{Here $G$ and $D$ indicate growing and decaying respectively.} 
\be\begin{split} 
% \tilde T^{mid, H}_{x_*, x^*; b, E} g = -\psi_4^{b, E} \int^r_{x_*} \psi_2^{b, E} g \frac{ds}{W_{42}} + \psi_2^{b, E} \int_{x_*}^r \psi_4^{b, E} g \frac{ds}{W_{42}} \\
\tilde T^{mid, G}_{x_*, x^*; b, E} g = -\psi_4^{b, E} \int^r_{x_*} \psi_2^{b, E} g \frac{ds}{W_{42}} - \psi_2^{b, E} \int^{x^*}_r \psi_4^{b, E} g \frac{ds}{W_{42}}   \\
 \tilde T^{mid, D}_{x_*, x^*; b, E} g = \psi_4^{b, E} \int_r^{x^*} \psi_2^{b, E} g \frac{ds}{W_{42}} - \psi_2^{b, E} \int^{x^*}_r \psi_4^{b, E} g \frac{ds}{W_{42}} 
 \end{split} \label{eqinvHmid}
\ee
where $W_{42} = \calW(\psi_4^{b, E}, \psi_2^{b, E}) =  \frac{b^\frac 13 E^\frac 16 }{2^\frac 43 \pi}$ from \eqref{eqWronski2}. Then for $g \in C^0([x_*, x^*])$, we have
$ \tilde H_{b, E} \tilde T^{mid, G}_{x_*, x^*; b, E} g = \tilde H_{b, E} \tilde T^{mid, D}_{x_*, x^*; b, E} g = g$,
and the following properties hold. We stress that the bounds in (2)-(4) are independent of $b, E, x_*, x^*$. 
\begin{enumerate}
    \item $\RR$-valued property: When $E = 1$ and $x^* \le r^*_{b, 1} = \frac 2b$, both operators map $\RR$-valued function to $\RR$-valued function. 
    \item Boundedness: Recall the weight function $w_{b, E}^\pm$ from \eqref{eqomegapm}.
    We have for $ \a  \in [0, 1]$,
    % \footnote{Recall that $\omega^+_{b, E} e^{-2 \Re \eta_{b, E}} = \omega^-_{b, E}$. }
    $\beta > 0$ and $0 \le k \le \frac{x_*}{10}$ that
    \bea
      \left\| \tilde T^{mid, G}_{x_*, x^*; b, E} f \right\|_{C^1_{\omega^+_{b, E} e^{-\a \Re \eta_{b, E}} } ([x_*, x^*])} &\lesssim& x_*^{-1}  \| f \|_{C^0_{\omega^+_{b, E} e^{-\a \Re\eta_{b, E}} r^{-2}} ([x_*, x^*])}\label{eqtildeTmidGest1} \\
      \left\| \tilde T^{mid, G}_{x_*, x^*; b, E} f \right\|_{C^1_{\omega^+_{b, E} } ([x_*, x^*])} &\lesssim_{k}&  b^{-k-1} \| f \|_{C^0_{\omega^+_{b, E} r^k} ([x_*, x^*])}\label{eqtildeTmidGest4}  \\ 
      \left\| \tilde T^{mid, G}_{x_*, x^*; b, E} f \right\|_{C^1_{\omega^-_{b, E}  r^k } ([x_*, x^*])} &\lesssim_k& x_*^{-1} \| f \|_{C^0_{\omega^-_{b, E} r^{k-2}} ([x_*, x^*])} \label{eqtildeTmidGest2} \\
      \left\| \tilde T^{mid, G}_{x_*, x^*; b, E} f \right\|_{C^1_{\omega^-_{b, E}  r^{k+2}} ([x_*, x^*])} &\lesssim_k& \| f \|_{C^0_{\omega^-_{b, E} r^{k+1}} ([x_*, x^*])} \label{eqtildeTmidGest3} \\
      % \left\| \tilde T^{mid, G}_{x_*, x^*; b, E} f \right\|_{C^0_{\omega^+_{b, E}e^{-\a \eta_{b, E}} r^{-1}} ([x_*, x^*])} &\lesssim& \| f \|_{C^0_{\omega^+_{b, E}e^{-\a r} r^{-2}} ([x_*, x^*])} \\
      \left\| \tilde T^{mid, D}_{x_*, x^*; b, E} f \right\|_{C^1_{\omega^-_{b, E} e^{-\beta r }r^k } ([x_*, x^*])} &\lesssim_{\beta, k}& x_*^{-1} \| f \|_{C^0_{\omega^-_{b, E} e^{-\beta r} r^{k-2}} ([x_*, x^*])}  \label{eqtildeTmidDest} 
    \eea
%     {\color{blue}No need this: and for $0 \le k \le k_0 := \frac{x_*}{10}$,
%     \bee
% \left\| \tilde T^{mid}_{x_*, x^*; b, E} f \right\|_{C^0_{w_* r^{k+1}} (I_{mid})} \lesssim \| f \|_{C^0_{w_* r^k} (I_{mid})} 
%     \eee}
    \item Boundary values: There exists $\gamma_2$, $\gamma_4$ such that for $m = 0, 1$, 
    \bee
     \pa_r^m \tilde  T^{mid, G}_{x_*, x^*; b, E} f  (x^*) = \gamma_2 \pa_r^m \psi_4^{b, E} (x^*),\quad
     \pa_r^m \tilde T^{mid, G}_{x_*, x^*; b, E} f \left(x_* \right) = \gamma_4 \pa_r^m \psi_2^{b, E} \left(x_* \right),\\
     \pa_r^m \tilde  T^{mid, D}_{x_*, x^*; b, E} f  (x^*) = 0, \quad \pa_r^m \tilde  T^{mid, D}_{x_*, x^*; b, E} f  (x_*) = \gamma_4  \pa_r^m \psi_2^{b, E} (x_*) - \gamma_2 \pa_r^m \psi_4^{b, E} (x_*)
    \eee
    where 
    \bee
     \gamma_{j} = -\int_{x_*}^{x^*} \psi_j^{b, E} f \frac{ds}{W_{42}},\quad j = 4, 2.
    \eee
    satisfies estimates for $\a \in [0, 1], k \ge 0$,
    \bea
      % |\gamma_1[f]| \lesssim 
       |\gamma_2[f]| &\lesssim& \left| \begin{array}{l}
           x_*^{-1} \| f \|_{C^0_{\omega^-_{b, E} r^{-2} }([x_*, x^*])}    \\
           b^{-k-1} \| f \|_{C^0_{\omega_{b, E}^- r^k}([x_*, x^*])}   \\
           (x^*)^{-1} e^{(2-\a) \Re \eta_{b, E}(x^*) } \| f \|_{ C^0_{\omega^+_{b, E} e^{-\a \Re \eta_{b, E}} r^{-2}}([x_*, x^*])}  \\
           b^{-k-1} e^{2\Re \eta_{b, E}(x^*)} \| f \|_{C^0_{\omega_{b, E}^+ r^k}([x_*, x^*])}  
       \end{array}\right. \label{eqeta2est1}\\
        \left| \gamma_{4}[f]\right| &\lesssim& \left| \begin{array}{l}
          % x_*^{-1} e^{-2\Re \eta_{b, E}(x_*)}\| f \|_{C^0_{\omega^-_{b, E} r^{-2} }([x_*, x^*])}, \\
          x_*^{k-1} e^{-2\Re \eta_{b, E}(x_*)}\| f \|_{C^0_{\omega^-_{b, E} r^{k-2} }([x_*, x^*])}, \\
          x_*^{-1}\| f \|_{ C^0_{\omega^+_{b, E} r^{-2}}([x_*, x^*])}  \\
          b^{-1-k}   \| f \|_{C^0_{\omega_{b, E}^+ r^k}([x_*, x^*])}  
      \end{array} \right.  \label{eqeta2est2}
    \eea

\item Improved boundary estimates: If $x_* \le b^{-1}$, we have 
\be  |\gamma_4[f]| \lesssim x_*^{k-2} e^{-2\Re \eta_{b, E}(x_*)}\| f \|_{C^0_{\omega^-_{b, E} r^{k-2} }([x_*, x^*])} \quad {\rm for}\,\,|k| \le \frac{x_*}{10};\label{eqeta2est3}
\ee
if $E = 1 + iIb$ with $I \in \RR$, we have
\be 
  |\gamma_2[f]| \lesssim_{|I|} b^{-k-\frac 23} e^{2\Re \eta_{b, E}(x^*)} \| f \|_{C^0_{\omega_{b, E}^+ r^k}([x_*, x^*])} \quad {\rm for}\,\,|k| \le \frac{x_*}{10}. \label{eqeta2est4}
\ee
\end{enumerate}

\end{lemma}

\begin{proof} (1) The $\RR$-valued property follows that of $\psi_2^{b, 1}$ and $\psi_4^{b, 1}$ from Proposition \ref{propWKB} (4). 

For (2)-(3), 
the Duhamel definition of $\tilde T^{mid, G}_{x_*, x^*; b, E}$, $\tilde T^{mid, D}_{x_*, x^*; b, E}$ \eqref{eqinvHmid} clearly implies the inversion property and the formula for boundary values. Thus we only need to verify the estimates \eqref{eqtildeTmidGest1}-\eqref{eqtildeTmidDest} and \eqref{eqeta2est1}-\eqref{eqeta2est2}.
From Proposition \ref{propWKB}, we have 
\bee
 | \psi_4^{b, E}| + |\pa_r \psi_4^{b, E}|  \lesssim \omega^-_{b, E},\quad  |\psi_2^{b, E}| +  |\pa_r \psi_2^{b, E}| \lesssim \omega^+_{b, E}, \quad r\in [0, r^*_{b, E} + b^{-\frac 13}];   
\eee 
and $|W_{42}| \sim b^{\frac 13}$. So those estimates are reduced to the following integral bounds with $r \in [x_*, x^*] \subset [1,  r^*_{b, E} + b^{-\frac 13}]$, $k \ge 0$, $\beta > 0$ and $\a \in [0, 1]$
\bea
 \int_{x_*}^{r} (\omega_{b, E}^+)^2  e^{-\a \Re \eta_{b, E}} s^{k-2}
  b^{-\frac 13} ds &\lesssim_k& 
  e^{(2-\a) \Re \eta_{b, E}(r)} r^{k-1} 
\label{eqintest11}\\
 \int_{r}^{x^*} \omega_{b, E}^- \omega_{b, E}^+  e^{-\a \Re \eta_{b, E}} s^{-2}
  b^{-\frac 13} ds &\lesssim_k& r^{-1} e^{-\a \Re \eta_{b, E}} 
  \label{eqintest12}\\
%    \int_{x_*}^{r} (\omega_{b, E}^+)^2 s^{k}
%   b^{-\frac 13} ds &\lesssim_k& 
%    r^{k+1}  e^{2\Re \eta_{b, E}(r)}
% \label{eqintest113}\\
   \int_{r}^{x^*} \omega_{b, E}^- \omega_{b, E}^+ s^k
  b^{-\frac 13} ds &\lesssim_k& b^{-k-1} 
  \label{eqintest112}\\
  \int_{x_*}^{r} \omega_{b, E}^- \omega_{b, E}^+ s^{k-2} b^{-\frac 13} ds &\lesssim_k& \left| \begin{array}{ll}
     x_*^{-1} r^k  & k \ge 0 \\
     r^{k-1}  & k \ge 3
  \end{array} \right.  \label{eqintest114}\\
    % \int_{x_*}^{r} \omega_{b, E}^- \omega_{b, E}^+ s^{k+1} b^{-\frac 13} ds &\lesssim_k& r^{k+2} \label{eqintest115}\\
  \int_r^{x^*}  (\omega_{b, E}^-)^2  s^{k-2}
  b^{-\frac 13} ds &\lesssim_k&  r^{k-1}e^{-2\Re \eta_{b, E}(r)}
  \label{eqintest14} \\
  \int_{r}^{x^*} \omega_{b, E}^+ \omega_{b, E}^-  e^{-\beta s}  s^{k-2}
  b^{-\frac 13} ds &\lesssim_{\beta, k}& e^{-\beta r}  r^{k-2} \label{eqintest15} \\
  \int_{r}^{x^*} (\omega_{b, E}^-)^2 e^{-\beta s}  s^{k-2}
  b^{-\frac 13} ds &\lesssim_{\beta, k}& r^{k-1} e^{-2\Re \eta_{b, E}(r)} e^{-\beta r}.  \label{eqintest16} 
\eea
The estimates \eqref{eqintest12}-\eqref{eqintest114}, \eqref{eqintest15}-\eqref{eqintest16} are simple consequences of the three facts: (a) $e^{-\Re\eta_{b, E}}$ is monotonically decreasing on $[0, r^*_{b, E}]$ by \eqref{eqReetamono} and is $O(1)$ on $[r^*_{b, E}, r^*_{b, E} + b^{-\frac 13}]$, (b) $|b^2 r^2 - 4E|^{-\frac 12}$ is integrable near $r^*_{b, E}$, and (c) (for \eqref{eqintest15}-\eqref{eqintest16}) $\int_r^\infty e^{-\beta s}s^{k-2} \lesssim_{\beta, k} e^{-\beta r}r^{k-2}$. For \eqref{eqintest11}, we also exploit
\bea
\pa_r \left( e^{(2-\a)\Re \eta_{b, E}} r^{k-2} \right) &=& \left((2-\a)\Im (b^2 r^2/4 - E)^\frac 12 + (k-2)r^{-1}\right)  e^{(2-\a) \Re \eta_{b, E}} r^{k-2}\nonumber \\
&\sim_k& e^{(2-\a) \Re \eta_{b, E}} r^{k-2},\quad r \in [x_*, \frac{3}{2b}]
\label{eqderivetarkest}
\eea
from \eqref{eqReetaderiv}, which implies \eqref{eqintest11} for $r \le \frac{3}{2b}$ and thereafter for $r \ge \frac{3}{2b}$ combined with the facts (a), (b) above. And \eqref{eqintest14} similarly $\pa_r \left(e^{-2\Re \eta_{b, E}} r^{k-2}\right) \sim -e^{-2\Re \eta_{b, E}} r^{k-2}$ for $0 \le k \le \frac{x_*}{10}$, $r \in [x_*, \frac{3b}{2}^{-1}]$, where we require $k \le \frac{x_*}{10}$ so that $(k-2)r^{-1} \le |\Im (b^2 r^2/4 - E)^\frac 12|$.

(4) For \eqref{eqeta2est3}, 
 the integral on $[x_*, \frac{3}{2b}]$ is controlled as in the proof of \eqref{eqintest14} above, and we control $r \in [\frac{3}{2b}, r^*_{b, E} + b^{-\frac 13}]$ using \eqref{eqintest14} and $e^{-2\left(\Re \eta_{b, E}(b^{-1}) - \Re \eta_{b, E}(\frac{3}{2b})\right)} \gg b^{-1}$.

  For \eqref{eqeta2est4}, first notice that \eqref{eqomegapm} indicates that $ \omega_{b, 1+iIb}^\pm \sim_{|I|} \omega_{b, 1}^\pm$ and that $\Re \eta_{b, 1}(r) = - S_b(r)$. Then compute
  \bea
 &&\left(  (2-br)^{-1} e^{-2S_b(r)} r^k \right)' \nonumber \\
 &=& (2-br)^{-1} e^{-2S_b(r)} r^k \left[ b(2-br)^{-1} + k r^{-1} + (4 - b^2 r^2)^\frac 12 \right]\nonumber \\
 &=&   \left( (2+br)^\frac 12 + b (2-br)^{-\frac 32} + kr^{-1}(2-br)^{-\frac 12}\right) (2-br)^{-\frac 12} e^{-2S_b(r)}r^k \nonumber\\
 &\sim& (2-br)^{-\frac 12} e^{-2S_b(r)}r^k
 \sim  \left( \int^r_{x_*} (\omega_{b, 1}^+)^2 s^k b^{-\frac 13}ds \right)',\quad r \in \left[x_*, \frac 2b - b^{-\frac 13}\right], \nonumber
\eea
where for the last line we used 
$$ \sup_{r \in \left[x_*, 2b^{-1} - b^{-\frac 13} \right]}\left( k r^{-1} (2-br)^{-\frac 12} \right) \lesssim |k|\max\left\{ x_*^{-1}, b^{\frac 23} \right\} \le 10^{-1}$$ for $|k| \le k_0 = \frac{x_*}{10}$. Thus with $E = 1+iIb$, 
\[ \int_{x_*}^r (\omega_{b, E}^+)^2 s^k b^{-\frac 13}ds \lesssim (2-br)^{-1} e^{-2S_b(r)} r^k \lesssim r^{k+1} b^\frac 13 e^{2\Re \eta_{b, E}(r)},\quad r \in \left[x_*, 2b^{-1} - b^{-\frac 13}\right].  \]
The boundedness on $[2b^{-1} - b^{-\frac 13}, r^*_{b, E} + b^{-\frac 13}]$ is trivial, and that concludes \eqref{eqeta2est4}.

\end{proof}

\subsubsection{Refined asymptotics of $Q_b$}

As a first application of these inversions, we redo the exterior construction of $Q_b$ in \cite{MR4250747} using our approximate fundamental solutions $\psi_j^{b, E}$ to obtain refined asymptotics as below. The proof 
% is similar to that of Lemma \ref{lemrhob} Step 2, and 
will be presented in Appendix \ref{appQb}. 

\begin{proposition}[Refined asymptotics of $Q_b$]\label{propQbasympref}
% {\color{purple}  $\| r^{\frac{d-1}{2}} P_b \|_{X^{s_c - b^{-1}\upsilon, N, -}_{r^*_{b, E}, \frac{2+\sqrt b}{b}; b, 1 + i \upsilon } } \lesssim_N b^{-\frac 16}e^{-\frac{\pi}{2b}}$. }

% For any $0 < s_c \ll 1$, there exists $b > 0$ with
% \be s_c \sim b^{-1} e^{-\frac{\pi}{b}}, \label{eqbasymp}\ee
% and $Q_b$ such that 

% \bea
%  |Q_b(r)| &\sim& r^{-\frac d2+s_c} b^{-\frac 12}e^{-\frac{\pi}{2b}} ,\quad {\color{blue}r \ge \frac 4b} ;\label{eqQbasymp1}\\
%   |Q_b(r)| &\sim& r^{-\frac{d-1}{2}} b^{-\frac 16}e^{-\frac{\pi}{2b} + S_b(r)} \left\la b^{-\frac 23} (2-br)\right\ra^{-\frac 14},\quad r \in [b^{-\frac 12},  b^{-2}];\label{eqQbasymp2}\\
%   |\Re P_b(r) - Q(r)| &\lesssim& \la r \ra^{-\frac{d-1}{2}} e^{-r}\left(b^\frac 13 + b^\frac 16 e^{-2(b^{-\frac 12} - r)}\right), \quad r \le b^{-\frac 12}; \label{eqQbasymp5} \\
%   |\Im P_b(r)| &\sim& bs_c \la r \ra^{-\frac{d-1}{2}} e^{r},\quad M \le  r \le b^{-\frac 12} \label{eqQbasymp6}; \\
%   |\Im P_b(r)| &\lesssim& bs_c \la r \ra^{-\frac{d-1}{2}} e^{r},\quad r \le M;
%   \label{eqQbasymp61}
% \eea
%   where $M \gg 1$ is independent of $s_c$, and $S_b(r) = \int_{\min\{r, \frac 2b \}}^{\frac 2b} (1 - \frac{b^2 s^2}{4})^\frac 12 ds$.
% {\color{blue} useless?: Moreover,
% \bea \label{eqQbasymp7}
%  |Q_b(r)| = C_b r^{-\frac d2 + s_c} (1 + O(b^{-3}r^{-2})) ,\quad 
%  |Q_b'(r)| \lesssim b^{-1} r^{-1} |Q_b(r)|,\quad r \ge b^{-2}.
% \eea
% with constant $C_b \in [ \frac 12 b^{-\frac 12}  e^{-\frac{\pi}{2b}}, 2 b^{-\frac 12}  e^{-\frac{\pi}{2b}}]$.}

For $d \ge 1$, there exists $s_c^{(0)'}(d) \in (0, s_c^{(0)}(d))$ with $s_c^{(0)}(d) \ll 1$ from Proposition \ref{propQbasymp}, such that for $s_c \in (0, s_c^{(0)'}(d))$, the profile $Q_b$ further satisfies

(1) Derivative estimates in the exterior and intermediate region: for $n \ge 0$,
% \footnote{\color{purple}The first estimate \eqref{eqQbasymp8} also improves \eqref{eqQbasymp2} on $[4b^{-1}, b^{-2}]$.}
 \begin{align}
    |\pa_r^n Q_b(r)| & \lesssim_n  (br)^{-n} \omega_{b, E}^- r^{-\frac{d-1}{2}} b^{-\frac 16}e^{-\frac{\pi}{2b}},\quad r \ge 4b^{-1}. \label{eqQbasymp8}\\
    \left| D_{-;b, 1}^n (e^{i\frac{br^2}{4}} Q_b(r))\right| &\lesssim_n  |r - 2b^{-1}|^{-n} \omega_{b, E}^- r^{-\frac{d-1}{2}} b^{-\frac 16}e^{-\frac{\pi}{2b}},\quad r \ge 2b^{-1} + b^{-\frac 12}, \label{eqQbasymp9}  \\
    \left| \pa_r^n |Q_b(r)|^2 \right| & \lesssim_n  |r - 2b^{-1}|^{-n}  ( \omega_{b, E}^- r^{-\frac{d-1}{2}} b^{-\frac 16}e^{-\frac{\pi}{2b}})^2,\quad r \ge 2b^{-1} + b^{-\frac 12}. \label{eqUabs2}
    \end{align}
% and for $k = 0, 1, 2$, 
% \bea
% |\pa_r^k \Re P_b(r)| \lesssim r^{-\frac{d-1}{2}} b^{-\frac 16} e^{-\frac \pi{2b} + S_b(r)} \left\la b^{-\frac 23} (2-br) \right\ra^{-\frac 14},\quad r \in [10|\log b|, 2b^{-1}]; \\
% |\pa_r^k \Im P_b(r)| \lesssim r^{-\frac{d-1}{2}} b^{-\frac 16} e^{-\frac \pi{2b} - S_b(r)} \left\la b^{-\frac 23} (2-br) \right\ra^{-\frac 14} ,\quad r \in [10|\log b|, 2b^{-1}]. 
% \eea
where $E = 1 + ibs_c$, $\omega_{b, E}^-(r)$ is from \eqref{eqomegapm}. 

(2) Asymptotics for $r \ge b^{-\frac 12}$: there exist\footnote{Here the $\varrho_b$ is different from the $\rho$ in \cite[Theorem 1]{MR4250747}  by a factor $2^\frac 76 \pi^\frac 12 (1 + o_{s_c \to 0}(1))$.}
\be  
  \varrho_b = 2^\frac 76 \pi^\frac 12 \kappa_Q b^{-\frac 12} e^{-\frac{\pi}{2b}}(1 + o_{s_c\to 0}(1)), \label{eqvarrhob}
\ee
and $|\theta_b| \lesssim b$ such that
\be
P_b(r) =
 \left| \begin{array}{ll}
 \varrho_b b^\frac 13 e^{\frac{\pi i}{6} + i\theta_b} \left[  \psi^{b, E}_1 (1 + O_\CC (b^{-1} r^{-2})) +  \psi^{b, E}_3 O_\CC (b^{-1+2s_c}r^{-2+2s_c}) \right] & r \ge \frac 2b \\
 \varrho_b b^{\frac 13} \left[   \psi_4^{b, 1} (1 + O_\RR (r^{-1})) + i   \psi_2^{b, 1} (\frac 12 + O_\RR(r^{-1})) \right] & r \in \left[ 2b^{-\frac 12}, \frac 2b \right]   \\
  \varrho_b b^{\frac 13} \left[   \psi_4^{b, 1} (1 + O_\RR (r^{-1})) + i   \psi_2^{b, 1} (\frac 12 + O_\RR(b^\frac 14)) \right] & r \in \left[b^{-\frac 12},  2b^{-\frac 12} \right]
  \end{array} \right. \label{eqQbsharp}
\ee
% In particular, we can extend the upper bound of $Q_b$ in \eqref{eqQbasymp2} to be
% \be 
%    |Q_b(r)| \sim r^{-\frac{d-1}{2}} b^{-\frac 16}e^{-\frac{\pi}{2b} + S_b(r)} \left\la b^{-\frac 23} (2-br)\right\ra^{-\frac 14},\quad r \in [b^{-\frac 12},  b^{-2}]. \label{eqQbasymp21}
% \ee

\end{proposition}

\begin{remark}
    This proposition verifies two features of $Q_b$: the exterior oscillation (even $o(b^{-1})$ close to the turning point) and the sharp asymptotics for real and imaginary parts in the intermediate range (clear from the connection formula \eqref{eqconnect}). They are essential in computing the bifurcated eigenvalue in Section \ref{sec6}. 
\end{remark}

\subsection{High spherical class}\label{sec42}

When it comes to high spherical classes $\nu = l + \frac{d-2}{2} \gg 1$, the approximate solution should take $\nu$ into account. 
We will look for approximate solution of 
\be \left( \pa_r^2 - E + \frac{b^2 r^2}{4} - \frac{\nu^2 - \frac 14}{r^2}  \right) \psi  = 0 \label{eqscalarh} \ee
with $\nu \ge 1, |E-1| \ll 1, 0 < b \ll 1$. The construction will be similar to the case when $\nu = 0$ through the Airy function with $\CC$-variable.

In this subsection, we will use the notation $\calC_{\theta}$ in \eqref{eqdefcalCtheta} for sector regions in $\CC$.

\mbox{}

\subsubsection{WKB approximate solutions}

\begin{definition}[WKB approximate solutions for high spherical classes]\label{defWKBappsoluh}
  Let $b > 0$, $|E-1| \le \frac 18$, $\nu \ge 1$. Define the complex arguments $s$ and $\zeta$ by
  \bea  s &=& \frac{br}{2\sqrt{E}}, \label{eqsdefh} \\
  \zeta \left(\frac{d\zeta}{ds}\right)^2 &=& s^2 - 1 - \a s^{-2},\quad \zeta(s_0) = 0, \label{eqzetadefh}
  \eea
  and where $\a, s_0(\a) \in \CC$ and extra auxiliary parameters $\mu, t_\pm$ are defined by 
  \bea \a = -\frac{\nu^2 - \frac 14}{\mu^4} = \frac{b^2 (4 \nu^2 - 1)}{16E^2},&&  \mu = e^{i\frac{\pi}{4}}\left(\frac{2E}{b} \right)^{\frac 12}, \label{eqmudefh}\\
   t_\pm = \frac{1 \pm \sqrt{1 + 4\a}}{2}, && s_0 = \sqrt{t_+}. \label{eqs0def1}
  \eea
Here $|\arg \a| \le \frac \pi 4$ and $\Re \a > 0$ from the range of $E, \nu, b$, so the branch of square roots in \eqref{eqs0def1} are chosen as $|\arg(1+4\a) |, |\arg t_+| \le \frac \pi 2$. 
% so $|\arg \a| \le \frac \pi 4$ and $\Re \a > 0$ from the range of $E, \nu, b$, and $s_0 = s_0(\a) \in \CC$ by
% \be t_\pm = \frac{1 \pm \sqrt{1 + 4\a}}{2},\quad s_0 = \sqrt{ t_+} \quad {\rm where}\,\,|\arg(1+4\a) |, |\arg t_+| \le \frac \pi 2.
% \label{eqs0def1}
% \ee
Then we define the WKB approximate solutions to be 
\be \begin{split} 
\psi_1^{b, E, \nu} = (\pa_s \zeta)^{-\frac 12}  \Ai(e^{-\frac{2\pi i}{3}}\mu^\frac 43\zeta),&\quad 
\psi_2^{b, E, \nu} = (\pa_s \zeta)^{-\frac 12} \Ai(e^{ \frac{2\pi i}{3}}\mu^\frac 43 \zeta), \\
\psi_3^{b, E, \nu} = (\pa_s \zeta)^{-\frac 12} \Ai(\mu^\frac 43 \zeta),&\quad 
\psi_4^{b, E, \nu} = \frac 12 (\pa_s \zeta)^{-\frac 12} \mathbf{Bi}(e^{ \frac{2\pi i}{3}}\mu^\frac 43 \zeta)
\end{split} \label{eqpsidefh} \ee
and the correction function\footnote{Again, the second part of the formula follows \eqref{eqzetadef} and is similar to \eqref{eqPsiF} below.}
\be h_{b, E, \nu}(r) = -\frac{b^2}{4E} \zeta_s^{\frac 32} \pa_\zeta^2 (\zeta_s^\frac 12) = -\frac{b^2}{4E} \left[ \frac{5(s^2 - 1 - \a s^{-2})}{16\zeta^3} - \frac{3s^2 + 2 + 18\a s^{-2} - 6 \a s^{-4} - \a^2 s^{-6}}{4(s^2 - 1 - \a s^{-2})^2} \right]. \label{eqcorrectionh}\ee
\end{definition}

\begin{remark}\mbox{}\label{rmkWKBh}
\begin{enumerate}
\item (Well-definedness) Notice that with $\a$ defined as \eqref{eqmudefh}, $\arg \a = - 2 \arg E = O(\delta)$ for $\nu \ge 1$, $b > 0$ and $|E-1| \le \delta\ll 1$.  We will prove below in Lemma \ref{lemzeta32h} that when $|E-1| \ll 1$, \eqref{eqzetadef} determines $\zeta = \zeta(s)$ as an analytic function on the sector region $\calC_\frac \pi 4$ and $\pa_s \zeta(s) \neq 0$. Therefore $\psi_j^{b, E}$ are well-defined $C^\infty_{loc}(\RR_+ \to \CC)$ functions, with the branch of $\zeta_s^{-\frac 12}$ determined by $|\arg \zeta'(s_0)| \le \frac \pi 4$ from \eqref{eqzetas0h}. 
\item (Derivation) We modify the construction of Definition \ref{defWKBappsolu}. Under the variable $s$,  \eqref{eqscalarh} becomes
\[\left[  \pa_s^2 + \frac{4E^2}{b^2} (s^2 - 1 - \a s^{-2}) \right] \psi = 0.  \]
This inspires the construction of $\zeta$ as \eqref{eqzetadefh}. 
With a further change of argument to $\zeta$ and variable to $\Upsilon = \zeta_s^{\frac 12} \psi$, the equation \eqref{eqscalar} turns into 
\be \pa_\zeta^2 \Upsilon - \mu^4 \zeta \Upsilon = F(\zeta)\Upsilon \label{eqPsih}\ee
where $F(\zeta) = \zeta_s^{-\frac 12} \pa_\zeta^2 (\zeta_s^\frac 12)$ like \eqref{eqPsiF}. So $\zeta_s^\frac 12 \psi_j^{b, E, \nu}$ for $j = 1, 2, 3, 4$ and their linear superposition will solve the homogeneous equation of \eqref{eqPsih}.
\end{enumerate}
\end{remark}

% Take 
% \be  s = \frac{br}{2\sqrt{E}}, \label{eqsdefh}\ee
% then \eqref{eqscalarh} becomes
% \[\left[  \pa_s^2 + \frac{4E^2}{b^2} (s^2 - 1) - \frac{\nu^2 - \frac 14}{s^2} \right] \psi = 0.  \]
% Next, again motivated by \cite{MR109898}, we define the new complex argument $\zeta$ as
% \bea \zeta \left(\frac{d\zeta}{ds}\right)^2 = s^2 - 1 - \a s^{-2},\quad \zeta(s_0) = 0, \label{eqzetadefh}  \eea
% where
% \be \a = -\frac{\nu^2 - \frac 14}{\mu^4} = \frac{b^2 (4 \nu^2 - 1)}{16E^2},\quad  \mu = e^{i\frac{\pi}{4}}\left(\frac{2E}{b} \right)^{\frac 12}. \label{eqmudefh}\ee
% and 
% \be t_\pm = \frac{1 \pm \sqrt{1 + 4\a}}{2},\quad s_0 = \sqrt{ t_+} \quad {\rm where}\,\,|\arg(1+4\a) |, |\arg t_+| \le \frac \pi 2,
% \label{eqs0def1}
% \ee
% due to $\Re \a > 0$ when $\nu \ge 1$.
% In particular, 
% \be s_0^2 - 1 - \a s_0^{-2} = 0,\quad s^2 - 1 - \a s^{-2} = s^{-2}( s^2 - t_+)(s^2 - t_-). \label{eqsdecomp} \ee 

% {\color{red} Notation of $\calC_{\theta_0}$. Also uniformize the range of parameters?}

We first show $\zeta$ with parameter $\a$ is a well-defined analytic function for $\a \in \calC_{\theta_0}$.

\begin{lemma}\label{lemzeta32h}
There exists a constant $0 < \theta_0 \ll 1$ such that for any $\a \in \calC_{\theta_0}$, let $\zeta$, $\a$, $t_\pm$ and $s_0$ defined as in \eqref{eqzetadefh}-\eqref{eqs0def1}. Then $\zeta = \zeta(s)$ is analytic in  $\calC_{\frac \pi 4}$,
% {\color{purple}$$ \Omega_m := \left| \begin{array}{ll}
%    \{s = e^{i\theta} \rho: |\theta| \le \theta_0, \rho > 0\}  & \nu \ge 1, \\
%    \{s = e^{i\theta} \rho: |\theta| \le \theta_0, \rho > b\}  & \nu=0,
% \end{array}\right.$$}
and it satisfies\footnote{This integral can also be computed explicitly, but we do not need that.}
        \be \zeta = \left( \frac 32\int_{s_0}^s (w^2 - 1 - \a w^{-2})^\frac 12 dw \right)^\frac 23
        % = \frac 12 s(s^2 - 1)^\frac 12 - \frac 12 \ln \left(s + (s^2 - 1)^\frac 12\right)
\label{eqetaidenth} 
\ee
where the branch of RHS is uniquely determined by 
\be \zeta'(s_0) = (2(t_+ - t_-)s_0^{-1})^{\frac 13} \in \calC_{\frac \pi 4}.\label{eqzetas0h} \ee
Moreover, on $\calC_\frac{\pi}{4}$, $\zeta$ has only one zero at $s = s_0$ and $\zeta_s$ is non-vanishing; on $\calC_{\theta_0}$, the following estimates hold
\be
  |\zeta(s)^\frac 32|\sim_{\theta_0} \left| \begin{array}{ll} 
  \ln\left|\frac{s_0}{s} \right| |\a|^\frac 12 + \la \a \ra^{\frac 12} & |s| \le \frac 12 |s_0|,\\
  |s_0|^{\frac 12}|s - s_0|^\frac 32 & \frac 12 |s_0| \le |s| \le \frac 32|s_0|, \\
  |s|^2 & |s| \ge \frac 32 |s_0|. 
  \end{array}\right.  \label{eqzetaasymph}
\ee
where the constant is independent of $\a$. 
\end{lemma}
\begin{proof}  
Integrating \eqref{eqzetadefh} from $s_0$ to $s \in \calC_{\frac \pi 4}$, we obtain
\[ \frac 23 \zeta^\frac 32 = \int_{s_0}^s (w^2 - 1 - \a w^{-2})^\frac 12 dw. \]
Notice that with $|\arg \a| \le \theta_0 \ll 1$, we have 
\[ |\arg t_+| \lesssim |\arg \a| \le \theta_0,\quad |\arg(-t_-)| = |\arg(\a/t_+)| \lesssim \theta_0.\]
Notice that 
\bee s_0^2 - 1 - \a s_0^{-2} = 0,\quad s^2 - 1 - \a s^{-2} = s^{-2}( s^2 - t_+)(s^2 - t_-), \eee 
and among the four roots $ \pm \sqrt{t_+}, \pm \sqrt{t_-}$, only $s_0 = \sqrt {t_+} \in \calC_{\frac \pi 4}$. Thus $\zeta^\frac 32(s)$ has only one branching point $s = s_0$ in $\calC_{\frac\pi 4}$, and similar to Lemma \ref{lemzeta32}, the analyticity of $\zeta$ in $\calC_{\frac \pi 4}$ reduces to its analyticity near $s_0$. This is a direct computation
\bea
  \zeta &=& \left( \frac 32 \int_{s_0}^s (w^2 - 1-\a w^{-2})^\frac 12 dw\right)^\frac 23 \nonumber \\
  &=& \left( \frac 32 \int_0^{s-s_0} h^\frac 12 (2s_0 + h)^\frac 12 (s_0 + h)^{-1} (t_+ - t_- + 2hs_0 + h^2)^{\frac 12} dh \right)^\frac 23 \label{eqzetaexpandh} \\
  &=& \left( \frac{2(t_+ - t_-)}{s_0} \right)^\frac 13 (s-s_0) + O(s-s_0)^2.\nonumber
\eea
Since obviously $|\arg((t_+ - t_-)s_0^{-1})|\lesssim \delta_2 \ll 1$, we can choose the natural branch such that $\zeta_s(s_0) \in \calC_{\frac \pi 4}$.
The non-vanishing of $\zeta$ and $\zeta_s$ in $\calC_\frac \pi 4$ is obvious from \eqref{eqzetadefh}. 

For the estimate \eqref{eqzetaasymph}, we first notice that 
    \bee
     |t_+| \sim |t_+ - t_-| \sim \la \a \ra^{\frac 12}, \quad |t_-| \sim \a \la \a \ra^{-\frac 12}, \quad |s_0| \sim \la \a \ra^\frac 14.
    \eee
We take $\theta_0 \ll 1$ such that $100 \theta_0 \ll \frac 12$, then 
 on $|s - s_0| \lesssim 100 \theta_0 |s_0| \ll \frac 12|s_0|$, the estimate follows the expansion of $\zeta$ at $s_0$ above. When $|s| \le (1-50\theta_0) |s_0|$ (or $|s| \ge (1+50\theta_0)|s_0|$) with $s \in \calC_{\theta_0}$, noticing that $|\arg (s_0 - s)| \le 10^{-1}$ (or $|\arg(s-s_0)| \le 10^{-1}$ respectively), we can write 
\bee
  |\zeta(s)^\frac 32| \sim \left| \int_{s_0}^s \left| w^2 - 1 - \a w^{-2} \right|^{\frac 12}  dw \right|.
\eee
So \eqref{eqzetaasymph} on these regions will follow from integrating the estimate for any $\theta \le \theta_0$
\bea
  &&|s^2 - 1 - \a s^{-2}|= |s|^{-2}|s^2 - t_+||s^2 - t_-| \nonumber \\
  &\sim_{\theta}& \left| \begin{array}{ll}
  |s|^{-2}|s_0|^2 \left( |s|^2 + |\a| \la \a\ra^{-\frac 12} \right) \sim |s_0|^2 + |\a s^{-2}| & |s| \le (1-50\theta) |s_0| \\
  |s|^2 &  |s| \ge (1+50\theta) |s_0|.
  \end{array}\right. \label{eqsintasymph}
\eea

 \end{proof}

The following definition and lemma are the counterparts of Definition \ref{defWKBaux} and Lemma \ref{lemWKBeta} in high spherical classes.

\begin{definition}[Auxiliary functions for WKB approximate solutions for high spherical classes]\label{defWKBauxh}
    Let $b > 0$, $\nu \ge 1$ and $|E-1| \le \frac 1{8}$ and $s, s_0, \mu, \zeta$ be as Definition \ref{defWKBappsoluh}. We define the $\CC$-valued function $\eta = \eta_{b, E, \nu}: (0, \infty) \to \CC$ as 
% \be \eta  = \frac 23 \mu^2 \zeta^{\frac 32} \ee
% using \eqref{eqetaident}.
\be \eta_{b, E, \nu}(r)  = \frac 23 \left(\mu^\frac 43 \zeta(s(r)) e^{-\frac{2\pi i}{3}}\right)^{\frac 32},\quad {\rm for\,\,} \arg \left(\mu^\frac 43 \zeta e^{-\frac{2\pi i}{3}} \right) \in (-\pi, \pi],\label{eqetadefh}\ee
and the weight functions $\omega_{b, E, \nu}^\pm: (0, \infty) \to (0,\infty)$ as 
  \be
 \omega^\pm_{b, E, \nu}(r) =  \left\la b^{-\frac 23} (b^2 r^2 - 4E - (4\nu^2 - 1)Er^{-2}) \right\ra^{-\frac 14} e^{\pm \Re \eta_{b, E, \nu}(r)}. \label{eqomegapmh}
\ee
We also define the center interval as
\be I_c^{b, E, \nu} = \left\{ r \in (0, \infty):  |br - 2\sqrt E s_0| \le 4|s_0|^{-\frac 13} b^{\frac 23} \right\}. \label{eqIcdefh}
 \ee
\end{definition}

\begin{lemma}\label{lemWKBetah}
There exists $0 < \delta_2' \ll 1$ such that for $0 < b \le \frac 14$, $\nu \ge 1$, and $|E-1|\le \delta_2'$, we have the following properties for $\eta_{b, E}$. 
\begin{enumerate}
    \item Turning point and monotonicity of $\Re \eta$: there exists a unique $r^* = r^*_{b, E, \nu} > 0$ such that 
    \bea\Im \left(\mu^\frac 43 \zeta(r) e^{-\frac{2\pi i}{3}} \right) \left\{ \begin{array}{ll} 
        < 0 & r > r^* \\
        = 0 & r = r^* \\
        > 0 & r < r^* 
        \end{array}\right.\quad {\rm if}\,\, \Im E \ge 0;
        \label{eqsignImzetah} \\
        \Im \left(-\mu^\frac 43 \zeta(r) \right) \left\{ \begin{array}{ll} 
        < 0 & r > r^* \\
        = 0 & r = r^* \\
        > 0 & r < r^* 
        \end{array}\right. \quad {\rm if}\,\, \Im E \le 0;
        \label{eqsignImzeta2h}
        \eea
    Then $\Re \eta (r^* - 0) = \Re \eta(r^* + 0) = 0$. Moreover, it satisfies the following bounds
    \bea
     \left| r^*_{b, E, \nu} - \frac{2\Re (\sqrt{ E}s_0) }b \right|\le \frac{|\Im E|}{b|s_0|}, \quad \forall \,\, |E-1| \le \delta_2';  \label{eqr*bE1h} 
    \eea
    and monotonicity
    \be r^*_{b, E_1, \nu} < r^*_{b, E_2, \nu},\quad {\rm if\,\,} \Re E_1 < \Re E_2 \,\,{\rm and}\,\, \Im E_1 = \Im E_2. 
    \label{eqr*monoh} 
    \ee
    % On different side of $r^*$,
    % we have formula of $\eta$
    % \be
    %   \eta = \left| \begin{array}{ll}
    %     -i \frac{2E}{b} \int_1^{s(r)} (w^2 - 1)^\frac 12 dw    & r > r^*,\,\, {\rm or}\,\,r < r^* \,\,{\rm and}\,\,\Im E \le 0, \\
    %      i \frac{2E}{b} \int_1^{s(r)} (w^2 - 1)^\frac 12 dw    & r < r^*\,\,{\rm and}\,\,\Im E > 0,
    %   \end{array} \right. \label{eqformulaeta}
    % \ee
    % where $\omega$ takes the branch as in \eqref{eqbranchsqrt}.
    We also have the formula of $\pa_r \Re \eta$ 
    \be 
     \pa_r \left(\Re \eta(r) \right) = \Im \left(  \frac{b^2 r^2}{4} - E - \frac{\nu^2 - \frac 14}{ r^2}  \right)^\frac 12,\label{eqReetaderivh} 
    \ee
    where $\arg \left(  \frac{b^2 r^2}{4} - E - \frac{\nu^2 - \frac 14}{ r^2}  \right) \in \left\{ \begin{array}{ll} 
        (0, 2\pi) & r < r^*, \\
        (-\pi, \pi) & r > r^*,
     \end{array}\right.$ 
    and monotonicity of $\Re \eta$
    \be {\rm sgn}(\pa_r \Re \eta) = \left\{ \begin{array}{ll} 
       1 & r < r^*, \\
       - {\rm sgn}(\Im E) & r > r^*.
     \end{array}\right.
\label{eqReetamonoh}
    \ee
    \item Asymptotics and estimates of $\eta$ and $\Re \eta$: for any $|E-1| \le \delta_2'$, 
    % {\color{red} Need asymptotics of $|\eta|$ (easy), of $\Re \eta$ when $|\Im E| \lesssim b$ for the whole range, and of $\eta$ near $b^{-\frac 12}$ for all $E$. (\eqref{eqReetaext} is not required even in the previous case. For $\Im E \ll -b$ we can apply monotonicity to get $|e^{\eta}| \le 1$ which is enough.)}
    \bea
           &&|\eta|\left\{\begin{array}{ll}
           \sim b r^2 & r \ge \frac{4|s_0|}{b},\\
           \sim b^{-1} |s_0|^\frac 12 \left| br - 2 \sqrt E s_0 \right|^{\frac 32} & r \in \left[\frac{|s_0|}b,\frac{4|s_0|}b \right] - I_c^{b, E, \nu}  \\
           \sim b^{-1} \left( \ln\left|\frac{s_0}{br}\right| |\a|^\frac 12 + \la \a \ra^{\frac 12} \right) & r \in \left[0, \frac{|s_0|}b\right] \\
               \lesssim 1 & r\in I_c^{b, E, \nu}
           \end{array} \right. \label{eqetaabsh}  \\
          % &&\eta= -\frac{\pi E}{2b} + \left(\sqrt E + O_\CC(b)\right)r,\quad 0 \le r \le b^{-\frac 12}.  \label{eqetaconvh} 
&& 
\begin{split}
\Re \eta = -\frac{\Im E}{b} \bigg[&\ln (br) - \frac 14 \ln\left( (\Re E)^2 + b^2(\nu^2 - 1/4)/4 \right) + O\left(|s_0|^2(br)^{-2}\right)\\
& + O\left(|s_0|^{-1}|\Im E|^\frac 12\right)   \bigg],\qquad {\rm for\,\,} r \ge \frac{2\Re (\sqrt E s_0)}{b} + 6 \frac{|\Im E|}{b|s_0|}  \end{split} \label{eqetaReh}  \\
&& \pa_r \Re \eta \ge \frac 12|s_0|,\qquad {\rm for}\,\, r \le \frac{\sqrt 2 |\sqrt E s_0|}{b}. \label{eqetaRehlowerbdd}
% \\
          % &&\left| \Re \eta +\frac{\Im E}{b} \ln (br)\right| \lesssim_{b, E, \nu} 1  \quad r \ge \frac{4|s_0|}{b}, \label{eqetaReh1}
           \eea
% Moreover, when $|\Im E| \le bI_0$ for some $I_0 > 0$, there exists $b_2'(I_0) \ll 1$ such that when $0 < b \le b_2'(I_0$, we have the refined asymptotics of $\Re \eta$
%     % \be
%     %        \Re \eta =    -\frac{\Im E}{b} \left[\ln (br) - \frac 14 \ln\left( 1 + b^2(\nu^2 - 1/4)/4 \right) \right] + O(1) \quad r \in [r^*_{b, E, \nu},\infty) - I_c^{b, E, \nu}
%     %        % \left\{\begin{array}{ll}
%     %        %     -\frac{\Im E}{b} \ln (br) + O(1) & r \in  [r^*_{b, E, \nu},\infty) - I_c^{b, E, \nu} \\
%     %        %     -S_{b, \Re E}(r) + O(b |br-2\sqrt {\Re E}|^{-\frac 12}) & r\in [0, r^*_{b, E, \nu}] - I_c^{b, E, \nu}
%     %        % \end{array} \right. 
%     %        \label{eqetaReh} 
%     %        \ee
%            where $S_b(r) = \int_{\min\{ r, \frac 2b \}}^{\frac 2b} \left( 1 - \frac{b^2 s^2}{4} \right)^\frac 12$,  $S_{b, \a}(r) = S_{b}(\frac{r}{\sqrt\a}) \a$ for $\a \in \RR$, and the bound of $O(1)$ term is independent of $b, \nu, E$ and only depends on $I_0$. 
\end{enumerate}
\end{lemma}

\begin{proof} To begin with, we choose $\delta_2' \ll \theta_0$ from Lemma \ref{lemzeta32h}, so that $|\arg E| \ll \theta_0$ and thus $\a \in \calC_{\theta_0}$ from \eqref{eqmudefh}, and then the conclusion of Lemma \ref{lemzeta32h} is valid. 

    \underline{\textbf{1. Proof of (1).}} When $E \in \RR$, we have $s_0 > 0$ and monotonicity of $r \mapsto s(r)$, $r \mapsto \zeta(s(r))$. Thus $r^*_{b, E, \nu} = \frac{2\sqrt E s_0}{b} = b^{-1} \left( E + \sqrt{E^2 + b^2(\nu - \frac 14)} \right)^\frac 12$ and all the properties follow directly. For simplicity, we only consider the case $\Im E < 0$ below, and the case $\Im E > 0$ can be derived similarly.

    \underline{1.1. Existence, uniqueness, monotonicity, and estimate of $r^*_{b, E, \nu}$.} 

    % Notice that
    % $ E(s^2 - t_+) = \frac{b^2 r^2}{4} - \frac{E + \sqrt{E^2 + b^2 (\nu^2 - \frac 14) } }{2}$ has fixed positive imaginary part of size $O(\arg E)$ and strictly increasing real part with respect to $r$, so $r \mapsto \arg( s(r)^2 - t_+)$ is monotonically decreasing from $\pi - \arg t_+$ to $-\arg E$ when $r$ goes from $0$ to $\infty$. 
    
    Since $0 < \arg t_+ <  -\arg E \lesssim \delta_0$ by its definition \eqref{eqs0def1}, we see $r \mapsto \arg(s(r) - s_0)$ is decreasing from $\pi - \frac 12\arg t_+$ to $-\frac 12 \arg E$ as $r$ goes from $0$ to $\infty$. Combined with that $|\arg s^{-2}|, |\arg (s^2 - t_-)|, |\arg (s + s_0)| \lesssim \delta_2$, the integration representation \eqref{eqzetas0h}, the root decomposition of $s^4 - \a s^2 - 1$ and the choice of branch \eqref{eqzetas0h} indicates 
    \be \arg \zeta \in (O(\delta_2), \pi + O(\delta_2)),\quad  {\rm for}\,\, r > 0,\label{eqzetaargrangeh} \ee
    which also falls in the admissible branch for \eqref{eqetadefh}. Thus $\Im \left(-\mu^\frac 43 \zeta(r) \right) = 0$ is equivalent to $\arg \zeta(r) = \frac {2\pi}{3} - \frac 23 \arg E$.
    
    The monotonicity of $\arg(s(r) - s_0)$ also implies that there exist $r_L$ and $r_R$ (depending on $b, E, \nu$) such that $r_L < r_R$ and 
\be  \arg (s(r_L) - s_0) = \frac {5\pi}6, \quad \arg (s(r_R) - s_0) = \frac \pi 6;\quad \arg(s(r) - s_0) \in \left[ \frac \pi 6, \frac{5\pi}{6} \right]\,\, \forall r \in [r_L, r_R]. \label{eqrLrRh}\ee
As in the proof of (1) for Lemma \ref{lemWKBeta}, the existence and uniqueness of $r^*_{b, E, \nu}$ plus \eqref{eqsignImzetah} and \eqref{eqr*bE1h} are reduced to 
\bea
&& \max \left\{ \frac{2\Re (\sqrt E s_0)}{b} - r_L, r_R - \frac{2\Re (\sqrt E s_0)}{b} \right\} < 
 \left(\frac{1}{\sqrt 3} + O(\delta_0)\right)\frac{|\Im E|}{b|s_0|},\label{eqrLrRasymph} \\
   &&\arg \zeta(s) =\arg (s-s_0) + O(\delta_0),\quad \forall \, r > 0, \label{eqzetaargh} \\
   && \pa_r \arg \zeta(r) < 0,\quad r \in [r_L, r_R].\label{eqzetaargmonoh}
\eea
The proof is similar to Lemma \ref{lemWKBeta} (1). We only mention that 
for \eqref{eqrLrRasymph}, we exploit 
\bea \Im (Et_+) &=& \Im \frac{E + \sqrt{E^2 + b^2 (\nu^2 - \frac 14) } }{2} \in (\Im E, 0), \label{eqImEt+1} \\
\Im \sqrt{Et_+} &=& \frac{\Im (Et_+)}{2|t_+|^\frac 12}(1 + O(\Im(Et_+)^2 |t_+|^{-1})) \in (|s_0|^{-1}\Im E, 0).\nonumber
\eea

% Next, we can derive the counterpart of \eqref{eqzetaarg} and \eqref{eqzetaargmono}
% \bea
% &&  r_{\sigma_\pm} - \frac{2\Re \sqrt E}{b}
% = \pm\left(\frac{1}{\sqrt 3} + O(\delta_0)\right)\frac{\Im E}{b},\label{eqrLrRasymp} \\
%    &&\arg \zeta(s) =\arg (s^2-s_0^2) + O(\delta_0),\quad \forall \, r > 0, \label{eqzetaargh} \\
%    && \pa_r \arg \zeta(r) < 0,\quad r \in [r_L, r_R].\label{eqzetaargmonoh}
% \eea
% using $|\arg s_0|, |\arg s| \lesssim \delta_0$ and  $|\arg(-t_-)| = |\arg (\a / t_+)| \lesssim \delta_0$ for \eqref{eqzetaargh} and \eqref{eqzetas0h} for \eqref{eqzetaargmono}. 
% Since  $\Im \left(\mu^\frac 43 \zeta(r) e^{-\frac {2 \pi i}{3}} \right) = 0$ is equivalent to $\arg \zeta(r) = \frac {\pi}{3} - \frac 23 \arg E$, they imply the existence of $r^* \in [r_L, r_R]$ and uniqueness on $(0,\infty)$. 

The monotonicity \eqref{eqr*monoh} can be proven by implicit function theorem with similar computation of derivative.

\underline{1.2. Vanishing, formula, and monotonicity of $\Re \eta$ \eqref{eqReetaderivh}-\eqref{eqReetamonoh}.} 

When $\Im E < 0$, the vanishing of $\Re \eta$ follows from $\arg (\mu^\frac 43 \zeta e^{-\frac{2\pi i}{3}})(r^*_{b, E}) = \frac \pi 3$. 

From \eqref{eqsignImzeta2h} and \eqref{eqzetaargh}, we have 
\[ \arg \left( \mu^\frac 43 \zeta (r)  e^{-\frac{2\pi i}{3}} \right) \in \left|\begin{array}{ll}  \left(-\frac{\pi }{3} + O(\delta_0), \frac \pi 3 \right) & r > r^* \\
\left(\frac \pi 3,\frac{2\pi }{3} + O(\delta_0)\right) & r < r^*.
\end{array} \right.   \]
% \eqref{eqzetaargmonoh}, \eqref{eqzetaargrangeh} and the definition of $r^*_{b, E,\nu}$, we immediately have for $\Im E < 0$ that
% \[ \arg \left( \mu^\frac 43 \zeta (r)  e^{-\frac{2\pi i}{3}} \right) \in \left|\begin{array}{ll}  \left(-\frac{\pi i}{3} + O(\delta_0), 0 \right) & r > r^* \\
% \left(0,\frac{2\pi i}{3} + O(\delta_0)\right) & r < r^*.
% \end{array} \right.   \]
So applying \eqref{eqetaidenth}, 
\be \eta = e^{-\frac{\pi i}{2}}\frac{2E}{b} \int_{s_0}^{s(r)} \left(w^2 - 1 - \a w^{-2} \right)^\frac 12 dw,\quad {\rm where}\,\, \arg \eta \in \left|\begin{array}{ll} \left(-\frac {\pi}2 + O(\delta_0), 0\right) & r > r^* \\
\left(0, \pi + O(\delta_0) \right) & r < r^*.
\end{array} \right.  \label{eqetainth} \ee
Therefore, when $r > r^*$, we take the branch $|\arg(w^2 - 1 - \a w^{-2})| \le \frac{5\pi}{6} $; and when $r \to 0$, we take $\arg(w^2 - 1 - \a w^{-2}) \in (\frac \pi 6, \frac {11\pi}{6})$. Then with 
\[ \pa_r \Re \eta = \Re \left[ e^{-\frac{\pi i}{2}} \sqrt E \left( \frac{b^2 r^2}{4E} - 1 - \frac{\nu^2 - \frac 14}{E r^2} \right)^\frac 12 \right] \]
the choice of branch implies \eqref{eqReetaderivh} and \eqref{eqReetamonoh}.

\mbox{}

\underline{\textbf{2. Proof of (2).}} The asymptotics of $|\eta|$ \eqref{eqetaabsh} follows directly from \eqref{eqzetaasymph}.

\emph{Proof of \eqref{eqetaReh}.} Firstly, we carefully evaluate the integral in \eqref{eqetainth} to be
\be
 \eta = -i \frac 2b \int_{\sqrt E s_0}^\frac{br}{2} \left( \left(w^4 - Ew^2 - \frac{b^2(4\nu^2 -1)}{16} \right) w^{-2} \right)^\frac 12 dw. \label{eqwquadtilde}
\ee
Denote $\tilde t_\pm = \frac{\Re E \pm \sqrt{(\Re E)^2 + b^2(\nu^2 - 1/4)/4}}{2}$ and $\tilde s_0 = \sqrt{\tilde t_+}$. Noticing that when $|E-1| \le \delta_2' \ll 1$, 
\be
  |\sqrt E s_0 - \tilde s_0| = \left|\frac{E t_+ - \tilde t_+}{\sqrt E s_0 + \tilde s_0}\right| \le \frac{|\Im E|}{2|s_0|(1 + O(\delta_2'))} \le \frac{|\Im E|}{|s_0|},\label{eqdifEs0ts0} \ee
  and $w^4 - Ew^2 - \frac{b^2(4\nu^2 -1)}{16}  = (w - \sqrt E s_0)(w + \sqrt E s_0) (w^2 - E t_-)$, we first have
\begin{align*}
&\left| \int_{\sqrt E s_0}^{\tilde s_0 +2\frac{|\Im E|}{|s_0|}} \left( (w^4 - Ew^2 - \frac{b^2(4\nu^2 -1)}{16} ) w^{-2} \right)^\frac 12 \right| \nonumber\\
\lesssim& \left| \int_{\sqrt E s_0}^{\tilde s_0 +2\frac{|\Im E|}{|s_0|}} |w - \sqrt E s_0|^\frac 12 dw \right| \cdot |s_0|^\frac 12 
\lesssim \frac{|\Im E|^\frac 32}{|s_0|}.
\end{align*}
Next, 
we notice that $r^*_{b, E, \nu} \le \frac{2\Re (\sqrt E s_0)}{b} + \frac{|\Im E|}{|s_0|} \le \frac2b \left(\tilde s_0 + 2\frac{|\Im E|}{|s_0|}\right)$ from \eqref{eqr*bE1h} and \eqref{eqdifEs0ts0}. When $r \ge \frac2b \left(\Re(\sqrt E s_0) + 3\frac{|\Im E|}{|s_0|}\right) \ge \frac2b \left(\tilde s_0 + 2\frac{|\Im E|}{|s_0|}\right)$, with $\Re \left( w^4 - E w^2 - \frac{b^2(4\nu^2 -1)}{16} \right)= (w^2 - \tilde t_+)(w^2 - \tilde t_-)$,  $\Im (w^4 - Ew^2 - \frac{b^2(4\nu^2 -1)}{16}) = -\Im E w^2$ for $w \ge r^*_{b, E, \nu}$, we compute
\bee
  &&\Im \int_{\tilde s_0 + 2\frac{|\Im E|}{|s_0|}}^{\frac{br}{2}}  \left( \left(w^4 - Ew^2 - \frac{b^2(4\nu^2 -1)}{16} \right) w^{-2} \right)^\frac 12 dw \\
  &=& \int_{\tilde s_0 + 2\frac{|\Im E|}{|s_0|}}^{\frac{br}{2}} \left[ \frac{-\frac 12 \Im E \cdot w}{(w^2 - \tilde t_+)^\frac 12(w^2 - \tilde t_-)^\frac 12} + O\left( \frac{|\Im E \cdot w^2|^3 w^{-1}}{(w^2 - \tilde t_+)^\frac 52(w^2 - \tilde t_-)^\frac 52} \right) dw \right] =: I + II
\eee
For $I$, we have explicit integration 
\bee
  I &=& - \frac 14 \Im E \int_{\tilde s_0 + 2\frac{|\Im E|}{|s_0|}}^{\frac{b^2 r^2}{4}} \left((u - \tilde t_+)(u - \tilde t_-)\right)^{-\frac 12} du \\
  &=& - \frac{\Im E}{2} \ln \left[ \frac{\sqrt{u - \tilde t_+} + \sqrt{ u -\tilde t_-}}{\sqrt{\tilde t_+ - \tilde t_-}} \right] \Bigg|_{u = (\tilde s_0 + 2\frac{|\Im E|}{|s_0|})^2}^{\frac{b^2 r^2}{4}} \\
  &=& - \frac{\Im E}{2}  \ln \left( \frac{\sqrt{\frac{b^2 r^2}{4} - \tilde t_+} + \sqrt{\frac{b^2 r^2}{4} - \tilde t_-}}{ \sqrt{2\frac{|\Im E|}{|s_0|} \left( 2\tilde s_0 + 2\frac{|\Im E|}{|s_0|} \right) } + \sqrt{\tilde t_+ - \tilde t_- + 2\frac{|\Im E|}{|s_0|} \left( 2\tilde s_0 + 2\frac{|\Im E|}{|s_0|} \right)}} \right)\\
  &=& - \frac{\Im E}{2} \bigg[\ln (br) - \frac 14 \ln\left( (\Re E)^2 + b^2(\nu^2 - 1/4)/4 \right) + O\left(|s_0|^2(br)^{-2}\right)
 + O\left(|s_0|^{-1}|\Im E|^\frac 12\right)   \bigg]
\eee
where we used $\tilde t_+ = \tilde s_0^2$ and $\tilde t_+ - \tilde t_- = 
  \left( (\Re E)^2 + b^2(\nu^2 - 1/4)/4 \right)^\frac 12 \sim |s_0|^2$. 
For $II$, we estimate
\bee
 |II| \lesssim |\Im E|^3 \int_{\tilde s_0 +  2\frac{|\Im E|}{|s_0|} }^{\frac{br}{2}} (w^2 - \tilde t_+)^{-\frac 52} dw \lesssim |\Im E|^3\cdot |s_0|^{-\frac 52} \left( \frac{|\Im E|}{|s_0|}\right)^{-\frac 32} = \frac{|\Im E|^\frac 32}{|s_0|}.
\eee
These estimates together yield \eqref{eqetaReh}. 

\textit{Proof of \eqref{eqetaRehlowerbdd}.}  Let $r_0 = \sqrt 2 |\sqrt E s_0| b^{-1}$ and $r = ar_0$ for $a \in [0, 1]$. Observe that $E t_+ = |Et_+| e^{iO(|\Im E|)}$ from \eqref{eqImEt+1}, and that $E t_- = -\frac{b^2(\nu^2 - \frac 14)}{16}(E t_+)^{-1}$. Since $r \le r_0 \le r^*_{b, E, \nu}$ from \eqref{eqr*bE1h}, we compute $\pa_r \Re \eta_{b, E, \nu}$ using \eqref{eqReetaderivh} 
\begin{align*}
 \pa_r \Re \eta_{b, E, \nu}(r) =& \Re \left(   E + \frac{\nu^2 - \frac 14}{ r^2} - \frac{b^2 r^2}{4} \right)^\frac 12 = \Re \left[ \left( Et_+ - \frac{b^2 r^2}{4} \right)^\frac 12 \left( 1 - \frac{4Et_-}{b^2 r^2}\right)^\frac 12 \right] \\
 =& \Re \left[ |Et_+|^{\frac 12} \left(e^{iO(\Im E)} - \frac{a^2}2\right)^\frac 12 \left( 1 + \frac{b^2(\nu^2 - \frac 14)}{16 |Et_+|^2} e^{iO(\Im E)}    \right)^\frac 12 \right] \\
 \ge& |s_0| \left(1 - \frac{a^2}{2}\right)^\frac 12 (1 + O(\delta_2'))
\end{align*}
Hence \eqref{eqetaRehlowerbdd} follows $a \le 1$ and taking $\delta_2'$ small enough. 
% For \eqref{eqetaReh1}, we decompose the integral to be on $[{\sqrt E s_0},{\frac{3|s_0|}{2}}]$ and $[\frac{3|s_0|}{2}, \frac{br}{2}]$, the first one provides a constant dependent on $b, E, \nu$, and the second can be evaluated similarly as above, creating leading order term $-\frac{\Im E}{2}\ln (br)$ and a bounded error term.
\end{proof}

\begin{proposition}[Linear WKB solution with correction for high spherical classes]\label{propWKBh} 
There exists $0 < \delta_2 \ll 1$ and $b_2(I_0) \ll 1$ depending on any fixed $I_0 > 0$, such that for any $0 < b < b_2(I_0)$, $E \in \{ z \in \CC: |z-1| \le \delta_2, \Im z \le bI_0 \}$, and $\nu \ge 1$, the following statements hold. Here $s, \zeta, \mu, \a, s_0, \psi_j^{b, E, \nu}, h_{b, E, \nu}$ and $\eta_{b, E, \nu}, \omega_{b, E, \nu}^\pm, I_c^{b, E, \nu}$ are from Definition \ref{defWKBappsoluh} and Definition \ref{defWKBauxh} respectively.

% Define the center interval as
% \[ I_c^{b, E, \nu} = \left\{ r \in (0, \infty):  | s - s_0| \le M_0 |s_0|^{-\frac 13} b^{\frac 23} \right\} \subset \left[ \frac{3|s_0|}{4b}, \frac{5|s_0|}{4b} \right];
%  \]
%  for $|\Im E| \le bI_0$, with $\tilde t_\pm = \frac{\Re E \pm \sqrt{(\Re E)^2 + b^2(\nu^2 - 1/4)/4}}{2}$ and $\tilde s_0 = \sqrt{\tilde t_+}$, we have 
%  \be I_c^{b, E, \nu} =\frac {2\tilde s_0}b + \left[ - \frac{2\sqrt{\Re E} M_0}{ (b|s_0|)^{\frac 13} }+ O(1)  , \frac{2\sqrt{\Re E} M_0}{ (b|s_0|)^{\frac 13}} + O(1)  \right]. \label{eqIc2h} \ee

    \begin{enumerate}
        \item Corrected equation: $\psi_j^{b, E}$ for $j = 1, 2, 3, 4$ solves 
        \[ \left( \pa_r^2 - E + \frac{b^2 r^2}{4} - \frac{\nu^2 - \frac 14}{r^2}\right) \psi_j^{b, E, \nu} = h_{b, E, \nu} \psi_j^{b, E, \nu} \]
        with correction $h_{b, E, \nu}$   
        % \be h_{b, E, \nu}(r) = -\frac{b^2}{4E} \zeta_s^{\frac 32} \pa_\zeta^2 (\zeta_s^\frac 12) = -\frac{b^2}{4E} \left[ \frac{5(s^2 - 1 - \a s^{-2})}{16\zeta^3} - \frac{3s^2 + 2 + 18\a s^{-2} - 6 \a s^{-4} - \a^2 s^{-6}}{4(s^2 - 1 - \a s^{-2})^2} \right]. \label{eqcorrectionh}\ee
        satisfying the estimate
        \be |h_{b, E, \nu}(r)| \lesssim r^{-2},\quad r > 0 \label{eqbddhh}\ee
        where the constant is independent of $b, E, \nu$. 
        \item Connection formula:
\be \psi_1^{b, E, \nu} + e^{-\frac{2\pi i}{3}}\psi_2^{b, E, \nu} +  e^{\frac{2\pi i}{3}}\psi_3^{b, E, \nu} = 0,\quad 2\psi_4^{b, E, \nu}(r) =  e^{\frac{\pi i}{6}} \psi_1^{b, E, \nu} +  e^{-\frac{\pi i}{6}} \psi_3^{b, E, \nu}. \label{eqconnecth}
\ee
\item Wronskian:
        \bea
        \mathcal{W}(\psi_4^{b, E, \nu}, \psi_2^{b, E, \nu}) = \frac{b^\frac 13 E^\frac 16 }{2^\frac 43 \pi},\quad  \mathcal{W}(\psi_1^{b, E, \nu}, \psi_3^{b, E, \nu}) = \frac{-i b^\frac 13 E^\frac 16}{2^\frac 43 \pi}. \label{eqWronski2h}
        \eea
    \item Asymptotics and bounds of $\psi_j^{b, E, \nu}$ and $(\psi_j^{b, E, \nu})'$: for $|E-1| \le \delta_0$ and $\Im E \le bI_0$, we have 
    % \footnote{Here the notation $f \sim g$ with $f$ complex-valued and $g$ positive means $|f| \sim g$.}
        \bee
         \psi_1^{b, E, \nu}(r)&=& \begin{cases}
            \frac{e^{\frac{\pi i}{6}}e^{-\eta}}{2\sqrt{\pi}\mu^\frac 13 (s^2 - 1 - \a s^{-2} )^\frac 14 } (1 + O(\eta^{-1})) 
            % \sim b^\frac 16|b^2 r^2 - 4E|^{-\frac 14}(br)^{\frac{\Im E}{b}}, 
            &r \in  [r^*_{b, E, \nu},\infty) - I_c^{b, E, \nu}, \\
             \frac{e^{-\frac{\pi i}{12} }e^{-\eta} }{2\sqrt{\pi}\mu^\frac 13 (1-s^2+\a s^{-2})^\frac 14 } (1 + O(\eta^{-1})) 
             % \sim b^\frac 16 |b^2 r^2 - 4E|^{-\frac 14} e^{S_{b, E, \nu}(r)} , 
             &  r \in  [0, r^*_{b, E, \nu}] - I_c^{b, E, \nu}, \\
            O (1), & r \in I_c^{b, E, \nu};
        \end{cases}\\
        % \label{eqWKBasymp2} \\
        \psi_2^{b, E, \nu}(r) &=& \begin{cases}
           % O(b^\frac 16 |b^2 r^2 - 4E|^{-\frac 14} \max_\pm \{e^{\pm \Re \eta}\} ), & r \in  [r^*_{b, E, \nu},\infty) - I_c^{b, E, \nu},\\
             \frac{e^{\frac{\pi i}{12} }e^{\eta} }{2\sqrt{\pi}\mu^\frac 13 (1-s^2+\a s^{-2})^\frac 14 } (1 + O(\eta^{-1})) 
             % \sim b^\frac 16 |b^2 r^2 - 4E|^{-\frac 14} e^{-S_{b, E, \nu}(r)} ,
             & r\in [0, r^*_{b, E, \nu}] - I_c^{b, E, \nu}, \\
            O(1), & r \in I_c^{b, E, \nu};
        \end{cases}\\
        % \label{eqWKBasymp3} \\ 
       \psi_3^{b, E, \nu}(r)&=& \begin{cases}
            \frac{e^{\eta}}{2\sqrt{\pi}\mu^\frac 13 (s^2 - 1 - \a s^{-2} )^\frac 14 } (1 + O(\eta^{-1})) 
            % \sim b^\frac 16 |b^2 r^2 - 4E|^{-\frac 14} (br)^{-\frac{\Im E}{b}},
            & r \in [4|s_0|b^{-1}, \infty),\\
            O\left( \mu^{-\frac 13} e^\eta |s^2 - 1 - \a s^{-2}|^{-\frac 14} \right)
            % \sim b^\frac 16 |b^2 r^2 - 4E|^{-\frac 14} (br)^{-\frac{\Im E}{b}},
            & r \in [r^*_{b, E, \nu}, 4|s_0|b^{-1}] - I_c^{b, E, \nu},\\
             % \frac{e^{\frac{\pi i}{4}}e^{-\eta}}{2\sqrt{\pi}\mu^\frac 13 (1-s^2+\a s^{-2})^\frac 14 } (1 + O(\eta^{-1}))
             %% \sim b^\frac 16 |b^2 r^2 - 4E|^{-\frac 14} e^{S_{b, E, \nu}(r)} , 
             % & r \in [0, r^*_{b, E, \nu}] - I_c^{b, E, \nu},\\
            O(1), & r \in I_c^{b, E, \nu};
        \end{cases}
        % \psi_4^{b, E, \nu}(r)&=& \begin{cases}
        %     \frac{e^{\frac{\pi i}{12} }e^{-\eta} }{2\sqrt{\pi}\mu^\frac 13 (1-s^2+\a s^{-2})^\frac 14 } (1 + O(\eta^{-1}))
        %      % \sim b^\frac 16 |b^2 r^2 - 4E|^{-\frac 14} e^{S_{b, E, \nu}(r)} , 
        %      &  r \in  [0, r^*_{b, E, \nu}] - I_c^{b, E, \nu}, \\
        %     O (1), & r \in I_c^{b, E, \nu};
        % \end{cases}
        % % \label{eqWKBasymp4}
        \eee
        and
        \bee
         (\psi_1^{b, E, \nu})'(r)&=& \begin{cases}
            \frac{i\sqrt{E} (s^2 - 1 - \a s^{-2})^\frac 14 e^{\frac{\pi i}{6}}e^{-\eta}}{2\sqrt{\pi}\mu^\frac 13} (1 + O(\eta^{-1})) 
            % \sim b^\frac 16 |b^2 r^2 - 4E|^{\frac 14}(br)^{\frac{\Im E}{b}}, 
            & r \in [r^*_{b, E, \nu},\infty) - I_c^{b, E, \nu},\\
           \frac{- \sqrt{E}(1-s^2+\a s^{-2})^\frac 14 e^{-\frac{\pi i}{12}}e^{-\eta}}{2\sqrt{\pi}\mu^\frac 13 } (1 + O(\eta^{-1})+ O(\nu^{-1})) 
           % \sim b^\frac 16 |b^2 r^2 - 4E|^{\frac 14} e^{S_{b, E, \nu}(r)} , 
           &r\in [0, r^*_{b, E, \nu}] - I_c^{b, E, \nu}, \\
            O(b^\frac 13), & r \in I_c^{b, E, \nu};
        \end{cases}\\
        % \label{eqWKBasymp2d}\\
        (\psi_2^{b, E, \nu})'(r)&=& \begin{cases}
           % O(b^\frac 16 |b^2 r^2 - 4E|^{\frac 14}\max_\pm \{e^{\pm \Re \eta}\} ), & r \in [r^*_{b, E, \nu},\infty) - I_c^{b, E, \nu},\\
           \frac{\sqrt{E} (1-s^2+\a s^{-2})^\frac 14 e^{\frac{\pi i}{12}}e^{\eta}}{2\sqrt{\pi}\mu^\frac 13 } (1 + O(\eta^{-1})+ O(\nu^{-1}))
           % \sim b^\frac 16 |b^2 r^2 - 4E|^{\frac 14} e^{-S_{b, E, \nu}(r)} , 
           & r\in [0, r^*_{b, E, \nu}] - I_c^{b, E, \nu},\\
            O(b^\frac 13), & r \in I_c^{b, E, \nu};
        \end{cases}\\
        % \label{eqWKBasymp3d}\\
          (\psi_3^{b, E, \nu})'(r)&=& \begin{cases}
            \frac{-i\sqrt{E}(s^2 - 1 - \a s^{-2})^\frac 14 e^{\eta}}{2\sqrt{\pi}\mu^\frac 13} (1 + O(\eta^{-1})) 
            % \sim b^\frac 16 |b^2 r^2 - 4E|^{\frac 14}(br)^{-\frac{\Im E}{b}}, 
            & r \in [4|s_0|b^{-1}, \infty) - I_c^{b, E, \nu},\\
             O\left( \mu^{-\frac 13}e^\eta |s^2 - 1- \a s^{-2}|^{\frac 14} \right) 
            % \sim b^\frac 16 |b^2 r^2 - 4E|^{\frac 14}(br)^{-\frac{\Im E}{b}}, 
            & r \in [r^*_{b, E, \nu}, 4|s_0|b^{-1}],\\
             % \frac{- \sqrt{E}(1-s^2+\a s^{-2})^\frac 14 e^{\frac{\pi i}{4}}e^{-\eta}}{2\sqrt{\pi}\mu^\frac 13 } (1 + O(\eta^{-1}))
             % % \sim b^\frac 16 |b^2 r^2 - 4E|^{\frac 14} e^{S_{b, E, \nu}(r)} , 
             % & r\in [0, r^*_{b, E, \nu}] - I_c^{b, E, \nu},\\
            O(b^\frac 13), & r \in I_c^{b, E, \nu}.
        \end{cases}
        % \label{eqWKBasymp1d} \\
        %  (\psi_4^{b, E, \nu})'(r)&=& \begin{cases}
        %    % O(b^\frac 16 |b^2 r^2 - 4E|^{\frac 14}\max_\pm \{e^{\pm \Re \eta}\} ), & r \in [r^*_{b, E, \nu},\infty) - I_c^{b, E, \nu},\\
        %    \frac{-\sqrt{E} (1-s^2+\a s^{-2})^\frac 14 e^{\frac{\pi i}{12}}e^{-\eta}}{2\sqrt{\pi}\mu^\frac 13 } (1 + O(\eta^{-1})+ O(\nu^{-1}))
        %    % \sim b^\frac 16 |b^2 r^2 - 4E|^{\frac 14} e^{-S_{b, E, \nu}(r)} , 
        %    & r\in [0, r^*_{b, E, \nu}] - I_c^{b, E, \nu},\\
        %     O(b^\frac 13), & r \in I_c^{b, E, \nu};
        % \end{cases}
        % % \label{eqWKBasymp4d} 
        \eee
        Here we choose $\arg (s^2 - 1 - \a s^{-2}) \in (-\pi, \pi]$ when $r > r^*_{b, E, \nu}$ and $\arg (1 - s^2 + \a s^{-2}) \in (-\pi, \pi]$ when $r < r^*_{b, E, \nu}$. If additionally $\Im E \ge - bI_0$, then we have 
        \bee \pa_r^k \psi_4^{b, E, \nu} =  e^{\frac{\pi i}{6}} \pa_r^k \psi_1^{b, E}(1 + O(\eta^{-1}) + O(\nu^{-1})) \quad {\rm for}\,\, r \in [0, r^*_{b, E}] - I_c^{b, E}, \,\, k \in \{ 0, 1\}. 
        % \label{eqWKBasymp4Lh}
        \eee
\item Derivatives estimates in the exterior region: Let $j_+ =3$ and $j_- = 1$, then
\be 
  \sup_{r \ge 4|s_0|b^{-1}}  \left|r^{\mp \frac{\Im E}{b} + 1} \left(\pa_r \pm \frac{ibr}{2}\right) \psi_{j_\pm}^{b, E, \nu} \right| < \infty.\label{eqpsibderivh}
\ee
        \end{enumerate}
        \end{proposition}

\begin{proof} Recall $\theta_0$, $\delta_2'$ from Lemma \ref{lemzeta32h} and Lemma \ref{lemWKBetah} respectively. To begin with, we assume $\delta_2 \le \delta_2' \ll \theta_0$ and $b_2 \le \frac 14$. 

% We take $\delta_0 \ll 1$ such that $|\arg \a| \le \theta_0$ from Lemma \ref{lemzeta32h}. Then
%     the proof is similar in strategy to that of Proposition \ref{propWKB}. We only focus on their difference. To begin with, the charaterization of $I_c^{b, E, \nu}$ \eqref{eqIc2h} when $|\Im E| \lesssim b$ follows from the computation
%     \be
%   \sqrt E s_0 - \tilde s_0 = \frac{E t_+ - \tilde t_+}{\sqrt E s_0 + \tilde s_0} = O_\CC(\Im E|s_0|^{-1}).\label{eqtildes0est}
% \ee

\mbox{}

    \textbf{Proof of (1)-(3).} The derivation of equation and form of correction $h_{b, E, \nu}$ are from Remark \ref{rmkWKBh} (2), and (2)-(3) comes from properties of Airy functions in Lemma \ref{lemAiry1} (2). It suffices to check the bound \eqref{eqbddhh}. To begin with, notice that 
    \bee
     |t_+| \sim |t_+ - t_-| \sim \la \a \ra^{\frac 12}, \quad |t_-| \sim \a \la \a \ra^{-\frac 12}, \quad |s_0| \sim \la \a \ra^\frac 14.
    \eee
    
    Firstly, we claim when $\delta_2$ small enough, for $|s - s_0| \le 100 \delta_2 |s_0|$, we have
    \be|\pa_s \zeta(s)| \sim |s_0|^{\frac 13},\quad |\pa_s^k \zeta(s)| \lesssim_k |s_0|^{\frac 43 - k}\,\,\, {\rm for}\,\, k  \in \{2, 3\}.\label{eqzetaderivesth} \ee
    This will immediately imply $|h_{b, E, \nu}(s)| \lesssim b^2 |s_0|^{-2} \sim r^{-2}$ from the same formula as \eqref{eqcorrection2}.
    Indeed, for $|s-s_0| \le \frac 12 |s_0|$, the differential equation of $\zeta$ \eqref{eqzetadefh} and asymptotics \eqref{eqzetaasymph} yield uniform estimates $|\pa_s^k \zeta| \lesssim |s_0|^{\frac 43 - k}$ for $k \in \{ 1, 2, 3\}$, and first part of \eqref{eqzetaderivesth} follows from the non-vanishing of $\zeta_s(s_0)$ \eqref{eqzetas0h} and taking $\delta_2$ small enough. 
    
    Next, when $|s| \le (1 -50 \delta_2) |s_0|$, we apply \eqref{eqsintasymph} (since $\delta_2 \ll \theta_0$) and \eqref{eqzetaasymph} to \eqref{eqcorrectionh} to see
    \bee
     b^{-2}|h_{b, E, \nu}(r)| &\lesssim& \frac{|s_0|^2 + |\a s^{-2}|}{\la \a \ra} + \frac{\la s\ra^2 + |\a s^{-2}| + \la \a s^{-2} \ra \cdot | \a s^{-4} | }{(|s_0|^2 + |\a s^{-2}|)^2} \\
     &\lesssim& |s|^{-2} \left( \frac{|s_0|^2 |s|^{2} + |\a|}{\la \a \ra} + \frac{|s|^2 \la s \ra^2 + |\a| + \la \a s^{-2}\ra |\a s^{-2}| }{\la \a \ra + |\a s^{-2}|^2} \right) \lesssim |s|^{-2}
    \eee
    Similarly, for $|s| \ge (1 + 50 \delta_2) |s_0|$, using $|\a s^{-2}| \le |\a s_0^{-2}| \lesssim |s_0|^2 \le |s|^2$,
    \bee
    b^{-2}|h_{b, E, \nu}(r)| &\lesssim& \frac{|s|^2}{|s|^4} = |s|^{-2}.
    \eee
    With $|s| \sim br$, these bounds conclude \eqref{eqbddhh}. 
  
    \mbox{}

\textbf{Proof of (4).} The asymptotics of $\psi^{b, E, \nu}_j$ are proven in the same way as Proposition \ref{propWKB} (5), from the asymptotics of Airy function, definition of $r^*_{b, E, \nu}$ and \eqref{eqzetaargh}. For those of $(\psi^{b, E, \nu}_j)'$, similarly, it suffices to check the derivative hitting on $\zeta_s^{-\frac 12}$ only gives lower order terms, namely
\be \left| \frac{\zeta_s^{-1} \zeta_{ss}}{\mu^2 (s^2 - 1 - \a s^{-2})^{\frac 12}} \right| \lesssim \left| \begin{array}{ll} |\eta^{-1}| & r \in [r^*_{b, E, \nu},\infty) - I_c^{b, E, \nu},  \\
|\eta^{-1}| + \nu^{-1} & r \in [0, r^*_{b, E, \nu}] - I_c^{b, E, \nu}, 
\end{array}\right. \label{eqzetasszetash} \ee
which is the counterpart of \eqref{eqzetasszetas}. When $|s - s_0| \le 100 \delta_2 |s_0|$, we use $\zeta_s^{-1} \zeta_{ss} = O_\CC(|s_0|^{-1})$ from  \eqref{eqzetaderivesth} and $|(s^2 - 1 - \a s^{-2})^\frac 12| \sim |s_0|^\frac 12 |s - s_0|^\frac 12$. When $|s| \le (1-50\delta_2)|s_0|$ or $|s| \ge (1+50\delta_2)|s_0|$, we compute using \eqref{eqzetadefh},
\bee
  \frac{\zeta_{ss}}{\zeta_s} = \zeta_s^{-1} \pa_s\left[\left( \frac{s^2 - 1 - \a s^{-2}}{\zeta}\right)^\frac 12 \right] = \frac{s (1 + \a s^{-4})}{s^2 - 1 - \a s^{-2}} -  \frac{(s^2 - 1 - \a s^{-2})^\frac 12}{2 \zeta^\frac 32},
\eee
so
\bee
 \left| \frac{\zeta_s^{-1} \zeta_{ss}}{\mu^2 (s^2 - 1 - \a s^{-2})^{\frac 12}} \right| \lesssim  b \left| \frac{s (1 + \a s^{-4})}{(s^2 - 1 - \a s^{-2})^\frac 32} \right| + |\eta^{-1}|.
 \eee
The first term can be evaluated through \eqref{eqsintasymph}, 
 \bee
 \left| \frac{s (1 + \a s^{-4})}{(s^2 - 1 - \a s^{-2})^\frac 32} \right| \lesssim
 \left| \begin{array}{ll}
       \frac{|s| + |\a s^{-3}|}{(|s_0|^2 + |\a s^{-2}|)^\frac 32} \lesssim |\a|^{-\frac 12} \sim (b\nu)^{-1} &  |s| \le (1-50\delta_2)|s_0| \\ 
       \frac{|s|}{|s|^3} = |s|^{-2} & |s| \ge (1+50\delta_2)|s_0|
  \end{array}\right. 
\eee
The above estimates combined together imply \eqref{eqzetasszetash}. 

\mbox{}

\textbf{Proof of (5).} This is a direct consequence of the asymptotics in (4) and \eqref{eqetaReh}. 
    
\end{proof}

\subsubsection{Inversion of corrected scalar operator}

For $b > 0$, $|E-1| \le \frac 18$ and $\nu \ge 1$, we define the scalar differential operator in high spherical classes with correction from \eqref{eqcorrectionh} as 
\be 
 \tilde H_{b, E, \nu} =  \pa_r^2 - E + \frac{b^2 r^2}{4} - \frac{\nu^2 - \frac 14}{r^2} -  h_{b, E, \nu}. \label{eqdeftildeHbEh}
\ee 
Now we construct the scalar inversion operator as in Lemma \ref{leminvtildeHext} and Lemma \ref{leminvtildeHmid}. We only consider $\Im E \le \frac 12b$, so we can work with the simple Duhamel formula and in weighted $C^0$ space. 

\begin{lemma}[Inversion of $\tilde H_{b, E, \nu}$] \label{leminvtildeHh}
    Fix $I_0 = \frac 12$. Let $0 < b \le b_2(I_0)$, $E \in \{ z \in \CC: |E-1| \le \delta_2, \Im z \le bI_0\}$ and $\nu \ge 1$, with $b_2(I_0), \delta_2$ from Proposition \ref{propWKBh}. We define the inversion operators of $\tilde \calH_{b, E, \nu}$ \eqref{eqdeftildeHbEh} in the exterior region $[r^*_{b, E, \nu}, \infty)$ with $r_0 \in [r^*_{b, E, \nu}, \infty]$ as 
    \be
    \tilde T_{r_0; b, E, \nu}^{ext} f = \psi_3^{b, E, \nu} \int_r^\infty \psi_1^{b, E, \nu} f \frac{ds}{W_{31}}  + \psi_1^{b, E, \nu} \int_{r_0}^r \psi_3^{b, E, \nu} f \frac{ds}{W_{31}} \label{eqinvHexth}
    \ee
    and in the middle region $[x_*, x^*] \subset [10, r^*_{b, E, \nu}]$ as
\be\begin{split} 
\tilde T^{mid, G}_{x_*, x^*; b, E, \nu} g = -\psi_4^{b, E, \nu} \int^r_{x_*} \psi_2^{b, E, \nu} g \frac{ds}{W_{42}} - \psi_2^{b, E, \nu} \int^{x^*}_r \psi_4^{b, E, \nu} g \frac{ds}{W_{42}}   \\
 \tilde T^{mid, D}_{x_*, x^*; b, E, \nu} g = \psi_4^{b, E, \nu} \int_r^{x^*} \psi_2^{b, E, \nu} g \frac{ds}{W_{42}} - \psi_2^{b, E, \nu} \int^{x^*}_r \psi_4^{b, E, \nu} g \frac{ds}{W_{42}} 
 \end{split} \label{eqinvHmidh}
\ee
where the Wronskians are $W_{31} = \calW(\psi_3^{b, E, \nu}, \psi_1^{b, E, \nu}) = \frac{ib\frac 13 E^\frac 16}{2^\frac 43 \pi}$, $W_{42} = \calW(\psi_4^{b, E, \nu}, \psi_2^{b, E, \nu}) = \frac{b\frac 13 E^\frac 16}{2^\frac 43 \pi}$ from \eqref{eqWronski2h}. They satisfy the following properties: 
\begin{enumerate}
    \item Boundedness: for $r_0' \in [r_0, 4|s_0|b^{-1}]$, $\a \in [0, \frac 12]$ and $\beta \ge 0$, 
    \begin{align}
      &\left\| \tilde T_{\infty; b, E, \nu}^{ext} f \right\|_{C^0_{\omega_{b, E, \nu}^- r^{-2}}([r_0, \infty)) \cap \dot C^1_{\omega_{b, E, \nu}^- (br)^{-1}r^{-2}} ([\frac{4|s_0|}{b},\infty))} 
      \lesssim b^{-1} \| f \|_{C^0_{\omega_{b, E, \nu}^- r^{-2}}([r_0, \infty))} \label{eqtildeHextesth1}\\
      &\left\| \tilde T_{r_0; b, E, \nu}^{ext} (\mathbbm{1}_{[r_0, r_0']} f) \right\|_{C^0_{\omega_{b, E, \nu}^\pm}([r_0, r_0'])} \lesssim b|s_0|^{-2} \| f \|_{C^0_{\omega_{b, E, \nu}^\pm r^{-2}}([r_0, r_0'])} \label{eqtildeHextesth2}\\
       &\left\| \tilde T^{mid, G}_{x_*, x^*; b, E, \nu} f \right\|_{C^1_{\omega^\pm_{b, E, \nu}} ([x_*, x^*])} \lesssim   (x_* |s_0|)^{-1}  \| f \|_{C^0_{\omega^\pm_{b, E, \nu} r^{-2}} ([x_*, x^*])} \label{eqtildeTmidGesth} \\
       & \left\| \tilde T^{mid, G}_{x_*, x^*; b, E, \nu} f \right\|_{C^1_{\omega^+_{b, E, \nu} e^{-\frac{\a r}{2|s_0|}}} ([x_*, x^*])} \lesssim   (x_* |s_0|)^{-1}  \| f \|_{C^0_{\omega^+_{b, E, \nu} e^{-\frac{\a r}{|s_0|}} r^{-2}} ([x_*, x^*])} \label{eqtildeTmidGesth2} \\
      &\left\| \tilde T^{mid, D}_{x_*, x^*; b, E} f \right\|_{C^1_{\omega^-_{b, E, \nu} e^{-\frac{\beta r}{|s_0|}} } ([x_*, x^*])} \lesssim (x_* |s_0|)^{-1}  \| f \|_{C^0_{\omega^-_{b, E, \nu} e^{-\frac{\beta r}{|s_0|}} r^{-2}} ([x_*, x^*])}  \label{eqtildeTmidDesth} 
    \end{align}
    Here $\| f \|_{\dot C^1_\omega(I)} := \| \nabla f \|_{C^0_\omega(I)}$.
    \item Boundary values: for $r_0, r_0'$ satisfying $r^*_{b, E, \nu} \le r_0 \le r_0' \le 4|s_0|b^{-1}]$ and $k \in \{ 0, 1\}$, there exist functionals $\beta_1, \beta_3, \gamma_2, \gamma_4$ such that
    \begin{align*}
      &\pa_r^k \tilde T^{ext}_{\infty; b, E, \nu} f(r_0) = \beta_3 \pa_r^k \psi_3^{b, E, \nu}(r_0) - \beta_1 \pa_r^k \psi_1^{b, E, \nu}(r_0),\\
      &\pa_r^k \tilde T^{ext}_{r_0; b, E, \nu} f(r_0) = \beta_3 \pa_r^k \psi_3^{b, E, \nu}(r_0),\quad \pa_r^k \tilde T^{ext}_{r_0; b, E, \nu} (\mathbbm{1}_{[r_0, r_0']} f)(r_0') = \beta_1 \pa_r^k \psi_1^{b, E, \nu}(r_0'),\\
      &\pa_r^k \tilde  T^{mid, G}_{x_*, x^*; b, E, \nu} f  (x^*) = \gamma_2 \pa_r^k \psi_4^{b, E, \nu} (x^*),\quad
     \pa_r^k \tilde T^{mid, G}_{x_*, x^*; b, E, \nu} f \left(x_* \right) = \gamma_4 \pa_r^k \psi_2^{b, E, \nu} \left(x_* \right),\\
     &\pa_r^k \tilde  T^{mid, D}_{x_*, x^*; b, E, \nu} f  (x^*) = 0, \quad \pa_r^k \tilde  T^{mid, D}_{x_*, x^*; b, E, \nu} f  (x_*) = \gamma_4  \pa_r^k \psi_2^{b, E, \nu} (x_*) - \gamma_2 \pa_r^k \psi_4^{b, E, \nu} (x_*)
    \end{align*}
    and they satisfy the following estimates with $\a \in [0, 1]$, $\beta > 0$
    \bee
      |\beta_3[f]| &\lesssim& b |s_0|^{-2} e^{(-1\pm 1)\Re \eta_{b, E, \nu}(r_0)} \| f \|_{C^0_{\omega_{b, E, \nu}^\pm  r^{-2}} ([r_0, \infty))} \\
      |\beta_1[f]| &\lesssim& b |s_0|^{-2} \| f \|_{C^0_{\omega_{b, E, \nu}^-  r^{-2}} ([r_0, \infty))} \\
      |\beta_1[\mathbbm{1}_{[r_0, r_0']} f]| &\lesssim& b |s_0|^{-2} e^{(1 \pm 1)\Re \eta_{b, E, \nu}(r_0')} \| f \|_{C^0_{\omega_{b, E, \nu}^\pm  r^{-2}} ([r_0, r_0'])} \\
      |\gamma_2[f]| &\lesssim& \left| \begin{array}{l}
           (x_*|s_0|)^{-1} e^{-\frac{\beta x_*}{|s_0|}} \| f \|_{C^0_{\omega_{b, E, \nu}^- e^{-\frac{\beta r}{|s_0|}} r^{-2}} ([x_*, x^*])} \\
           (x^*|s_0|)^{-1} e^{2\Re \eta_{b, E, \nu}(x^*)} e^{-\frac{\a x^*}{2|s_0|}}  \| f \|_{C^0_{\omega_{b, E, \nu}^+ e^{-\frac{\a r}{|s_0|}}  r^{-2}} ([x_*, x^*])}
      \end{array}\right.\\
    |\gamma_4[f]| &\lesssim& \left| \begin{array}{l}
           (x_*|s_0|)^{-1} e^{-2\Re \eta_{b, E, \nu}(x_*)} e^{-\frac{\beta x_*}{|s_0|}} \| f \|_{C^0_{\omega_{b, E, \nu}^- e^{-\frac{\beta r}{|s_0|}} r^{-2}} ([x_*, x^*])} \\
           (x_*|s_0|)^{-1} e^{-\frac{\a x^*}{2|s_0|}}  \| f \|_{C^0_{\omega_{b, E, \nu}^+ e^{-\frac{\a r}{|s_0|}}  r^{-2}} ([x_*, x^*])} 
      \end{array}\right.
    \eee
\end{enumerate}
\end{lemma}

\begin{proof}
Notice that the polynomial part of $\omega_{b, E, \nu}^\pm$ is bounded for $r \ge 0$ by 
    \be \left\la b^{-\frac 23} (b^2 r^2 - 4E - (4\nu^2 - 1)Er^{-2}) \right\ra^{-\frac 14} \lesssim b^{-\frac 13}  \left|r + \frac{2\sqrt E s_0}{b}\right|^{-\frac 14} \left|r-\frac{2\sqrt E s_0}{b}\right|^{-\frac 14}. \label{eqpolyweightcontrolh} \ee
    
    Then the proof of \eqref{eqtildeHextesth1}-\eqref{eqtildeHextesth2} is similar to that of \eqref{eqbddHbnuinvext2}. We remark that (1) as $\Im E \in (0, \frac b2]$, we have  $e^{- 2\Re \eta_{b, E, \nu}} \sim \left| \frac{br}{s_0} \right|^{\frac{2\Im E}{b}} \lesssim \left| \frac{br}{s_0} \right|$, so that $\int_r^\infty(\omega_{b, E, \nu}^-)^2 r'^{-2} b^{-\frac 13} dr' \lesssim  b^{-1} r^{-2} e^{-2\Re \eta_{b, E, \nu}(r)}$ holds without further integration by parts thanks to the slow growth; (2) in \eqref{eqtildeHextesth1} the integrability of the second integral of $\tilde \calT^{ext}_{\infty; b, E, \nu}$ follow from the extra $r^{-2}$ decay. 
    
    The proof of \eqref{eqtildeTmidGesth} and \eqref{eqtildeTmidDesth} mimic that of \eqref{eqtildeTmidGest1} with $\a = 0$, \eqref{eqtildeTmidGest2} with $k = 0$ and \eqref{eqtildeTmidDest} with $k = 0$, which mainly reduces to the monotonicity of $e^{\pm \Re \eta_{b, E, \nu}}$ from \eqref{eqReetamonoh} and $e^{-\beta r |s_0|^{-1}}$ to take them out of the integral. 
    
    For \eqref{eqtildeTmidGesth2}, we instead use that $r \mapsto e^{2\Re \eta_{b, E, \nu}}e^{-\frac{\a}{|s_0|} \min \{ r, r^*_{b, E, \nu}/2 \} }$ with $\a \in [0, \frac 12]$ is non-decreasing on $[x_*, x^*] \subset [0, r^*_{b, E, \nu}]$ thanks to \eqref{eqReetamonoh}, \eqref{eqetaRehlowerbdd}, and $|s_0| > 1$. This weight turns into RHS since $e^{-\frac{\a}{|s_0|} \min \{ r, r^*_{b, E, \nu}/2 \} } \le e^{-\frac{\a r}{2|s_0|}}$ on $[x_*, x^*]$. 
    
    Meanwhile, we also obtain the boundary value estimates, except the second estimate for $\gamma_2[f]$. To obtain the bound with $(x^*)^{-1}$ factor, we exploit the non-decreasing of $r \mapsto r^{-4}  e^{2\Re \eta_{b, E, \nu}}e^{-\frac{\a r}{|s_0|}}$ on $[2, \frac{r^*_{b, E, \nu}}2]$ and $r \mapsto e^{2\Re \eta_{b, E, \nu}}$ on $[\frac{r^*_{b, E, \nu}}2, r^*_{b, E, \nu}]$ from \eqref{eqReetamonoh} and \eqref{eqetaRehlowerbdd} and partition the integration into two parts, like the proof of \eqref{eqintest11}.
\end{proof}

\section{Exterior ODE analysis: System case}\label{sec5}
 In this section, we construct exterior fundamental solution to the original system \eqref{eqnu} with the scalar approximate solutions and inversion operators in Section \ref{sec4}.

\subsection{Admissible solutions for low spherical classes}

Firstly, we construct the admissible branches of fundamental solutions for finite spherical classes based on Section \ref{sec41}. 

\begin{proposition}[Construction of exterior admissible solution] 
\label{propextfund} For $d \ge 1$, $I_0, \nu_0, K_0 > 0$, there exists $s_{c;{\rm ext}}^{(1)} > 0$, $b_1 > 0$, $\delta_1 > 0$ and $x_* \gg 1$ such that for $0 < s_c \le s_{c;{\rm ext}}^{(1)}(d, I_0, \nu_0, K_0)$, $b = b(d, s_c)$ from Proposition \ref{propQbasymp} and $\l, \nu$ satisfying
    \be 0 \le b \le b_1, \quad  \l \in \Omega_{\delta_1; I_0, b}, \quad \nu \in \frac 12 \NN \cap [0, \nu_0], \label{eqrangeblnu} \ee
    there exist two smooth functions $\Phi_{j;b, \l, \nu}: (0, \infty) \to \CC^2$ with $j = 1, 2$ solving the equation \eqref{eqnu} with parameters in the above range. They are analytic w.r.t. $\l$, and the following statements hold.
    \begin{enumerate}
    \item Asymptotic behavior at infinity for $b = 0$: For $0 \le k \le K_0$, $N = 0, 1$ and $j = 1, 2$, 
    \bea
      \sup_{r \ge 1} \left[\left| r^{-k} e^{\sqrt{1+\l}r} \pa_\l^k \pa_r^N \Phi^1_{j;0,\l,\nu}  \right| + \left| r^{-k} e^{\sqrt{1-\l}r} \pa_\l^k \pa_r^N \Phi^2_{j;0,\l,\nu}  \right|\right] < \infty, \label{eqb0extbdd}
      % \\
      % \sup_{r \ge 1} \left[\left| e^{-\sqrt{1+\l}r}  \pa_r^N \Phi^1_{j;0,\l,\nu}  \right| + \left|  e^{-\sqrt{1-\l}r} \pa_r^N \Phi^2_{j;0,\l,\nu}  \right|\right] < \infty,\quad j = 3, 4.
    \eea
    with non-degeneracy that for $N = 0, 1$, and any $(c_1,c_2) \in \CC^2 - \{ \vec 0\}$,
    \be
    \limsup_{r \to \infty}  \left[ \left| e^{\sqrt{1+\l}r} \pa_r^N \sum_{j=1}^2  ( c_j \Phi^1_{j;0,\l,\nu} )\right| +  \left|  e^{\sqrt{1-\l}r}  \pa_r^N \sum_{j=1}^2  ( c_j \Phi^2_{j;b,\l,\nu}) \right|  \right]  > 0. 
  % |\Phi^1_{1;0,\l,\nu}| \sim |e^{-\sqrt{1+\l}r}|,\quad
  % |\Phi^2_{2;0,\l,\nu}| \sim |e^{-\sqrt{1-\l}r}|.
  \label{eqnondegb0}
\ee
Moreover, $\{\pa_\l^n \Phi_{j;0,\l,\nu}\}_{\substack{ n \ge 0 \\ 1 \le j \le 2}}$ are linearly independent.
        \item Asymptotic behavior at infinity and linear independence for $ b > 0$: For $b>0$, $\Phi_{j;b,\l,\nu}$ for $j = 1, 2$ satisfy for $0 \le k \le K_0$ and $N \ge 0$ that 
        \bea 
          \sup_{r \ge \frac 4b} \left[ \left|\pa_\l^k \pa_r^N ( e^{-\frac{ibr^2}{4}} \Phi^1_{j;b,\l,\nu}) \right| +  \left|\pa_\l^k \pa_r^N ( e^{\frac{ibr^2}{4}} \Phi^2_{j;b,\l,\nu}) \right|  \right] r^{\frac 12 -\frac{\Im \l}{b} - 2k - s_c+N} < \infty,
          \label{eqadmosc}
        \eea
        % while $\Phi_{3;b,\l,\nu}$, $\Phi_{4;b,\l,\nu}$ satisfy for $N \ge 0$ that 
        % \bea
        % \begin{split}
        %   \sup_{r \ge \frac 4b} \left[ \left| \pa_r^N ( e^{\frac{ibr^2}{4}} \Phi^1_{3;b,\l,\nu}) \right| +  \left|  \pa_r^N ( e^{\frac{3ibr^2}{4}} \Phi^2_{3;b,\l,\nu}) \right| \right] r^{\frac 12 +\frac{\Im \l}{b} - s_c+N} < \infty, \\
        %   \sup_{r \ge \frac 4b} \left[ \left| \pa_r^N ( e^{-\frac{3ibr^2}{4}} \Phi^1_{4;b,\l,\nu}) \right| +  \left|  \pa_r^N ( e^{-\frac{ibr^2}{4}} \Phi^2_{4;b,\l,\nu}) \right| \right] r^{\frac 12 +\frac{\Im \l}{b} - s_c+N} < \infty,
        %   \end{split}
        %   \label{eqnonadmosc}
        % \eea
        Moreover, $\{\pa_\l^n \Phi_{j;b,\l,\nu}\}_{\substack{ n \ge 0 \\ 1 \le j \le 2}}$ are linearly independent and satisfies for $N = 0, 1$ that, 
        % Moreover, they are non-degenerate in the following sense
        % that for $N = 0, 1$, 
        \bea
        % % \limsup_{r \to \infty}  \left[ \left|  \pa_r^N \sum_{j=1}^2  ( e^{-\frac{ibr^2}{4}} c_j \Phi^1_{j;b,\l,\nu}) \right| +  \left|  \pa_r^N \sum_{j=1}^2  ( e^{\frac{ibr^2}{4}} c_j \Phi^2_{j;b,\l,\nu}) \right|  \right]  r^{\frac 12 -\frac{\Im \l}{b} - s_c+N} > 0. 
          |\Phi^1_{1;b,\l,\nu}| \sim_{b, E} |\psi_1^{b, 1+\l+ibs_c}|, \quad  |\Phi^2_{2;b,\l,\nu}| \sim_{b, E} |\psi_1^{b, 1-\bar \l+ibs_c}|, \quad r \ge \frac 4b.
          \label{eqadmnondeg}
        %   % |\Phi^1_{3;b,\l,\nu}| \sim |\psi_3^{b, 1+\l+ibs_c}|, \quad  |\Phi^2_{4;b,\l,\nu}| \sim |\psi_3^{b, 1-\bar \l+ibs_c}|, \quad r \ge r_{0;b,\nu,\delta_1} \label{eqnonadmnondeg}
        \eea
        % with $\psi_1^{b, E}$ from Definition \ref{defWKBappsolu}. 
        % with $r_{0;b,\nu,\delta_1} \ge \frac 4b$.
        % And $\{ \pa_\l^n\Phi_{j;b,\l,\nu} \}_{n \ge 0, 1 \le j \le 4}$ are linear independent functions. 
        % {\color{red} Include in proof: In particular, $\{ \pa_\l^n\Phi_{j;b,\l,\nu} \}_{n \ge 0, 1 \le j \le 2}$ are linear independent functions. }
        \item Boundary value at $x_*$: With $\nu$ in the above range fixed, define the admissible boundary value map
        \be
        \begin{split}
        f_{x_*;\nu} : \{ (b, \l) \in [0, b_1] \times B_{\delta_1}^\CC: \l \in \Omega_{\delta_1; I_0, b} \,\, {\rm if}\,\,b > 0 \} \to (\CC^2)^4,  \\
        (b, \l) \mapsto \left( \Phi_{1;b, \l, \nu}(x_*), \pa_r \Phi_{1;b, \l, \nu}(x_*), \Phi_{2;b, \l, \nu}(x_*), \pa_r \Phi_{2;b, \l, \nu}(x_*)  \right). 
        \end{split}
        \ee
        Then $f_{x_*;\nu}$ is analytic w.r.t. $\l$ and satisfies the following quantitative continuity w.r.t. $b$: for $0 \le k \le K_0$,  
        \bea 
        \sup_{\substack{b \in [0, b_1], \\ \l \in \Omega_{\delta_1;I_0,b} }}| \pa_\l^k f_{x_*;\nu}| \lesssim_{K_0} e^{-\sqrt{1-\delta_1} x_*} x_*^{k},  \label{eqbdrymapest}\\
        \sup_{\l \in \Omega_{\delta_1;I_0,b}}\left|\pa_\l^k f_{x_*;\nu}(b, \l) - \pa_\l^k f_{x_*;\nu}(0, \l) \right| \lesssim_{K_0} b^\frac 16 e^{-\sqrt{1-\delta_1} x_*} x_*^{k},\quad 0 < b \le b_1.
        \label{eqbdrymapasymp}
        \eea
    \end{enumerate}
\end{proposition}

\begin{proof} 
Within the proof, to simplify notation, we use the vector form of fundamental solution $\vec \Phi$ to \eqref{eqnu} and WKB approximate solutions $\vec \psi_j^{b, E}$ ($j = 1, 2, 3$) in Proposition \ref{propWKB} as 
 \be \vec\Phi (r) = \left( \Phi^{1}, \pa_r \Phi^{1}, \Phi^{2}, \pa_r \Phi^{2} \right)^\top (r) \in \CC^4, \quad
\vec\psi_j^{b, E}(r) = \left(
    \psi_j^{b, E}, \pa_r \psi_j^{b, E}
\right)^\top (r) \in \CC^2. \label{eqvectorform}
\ee

Fix a $d \ge 1$. We first restrict our parameters to satisfy
\be s_{c;{\rm ext}}^{(1)} \le s_c^{(0)'},\quad \delta_1 \le \min \left\{\frac 12 \delta_0, (100d)^{-1}\right\},\quad b_1 \le \min \left\{b_0(I_0), \delta_1^\frac 32\right\}, \label{eqdelta1restrict} \ee
where $s_c^{(0)'}$ is from Proposition \ref{propQbasympref} and $b_0(I_0)$ and $\delta_0$ are from Proposition \ref{propWKB}. We will determine $\delta_1$ in Step 2 (1) as \eqref{eqdelta1choice}, and take $b_1$ small enough and $x_*$ large enough according to restrictions in Step 1, Step 2, and Step 3. Finally $s_{c;{\rm ext}}^{(1)}$ will be further shrunk so that the parameter of self-similar profile (see Proposition \ref{propQbasymp}) satisfies $b(d, s_c) \le b_1$ for all $s_c \le s_{c;{\rm ext}}^{(1)}$.

\mbox{}

\underline{Step 1. Fundamental solutions for $b = 0$ case.}

This part is standard, see a similar setting in \cite[Lemma 5.2, Lemma 5.5]{MR2219305} for example. We first define two inversion operators of $\pa_r^2 - E$ as
\be
\begin{split}
T_{x_*; E}^{(0), G} f &= - \int_{x_*}^r  \frac{e^{-\sqrt E (r-s)}}{2\sqrt{E}} f ds - \int_r^\infty \frac{e^{-\sqrt E (s-r)}}{2\sqrt{E}} f ds,   \\
T_{x_*; E}^{(0), D} f &=  \int_r^\infty \frac{e^{-\sqrt E (r-s)}}{2\sqrt{E}} f ds - \int_r^\infty \frac{e^{-\sqrt E (s-r)}}{2\sqrt{E}} f ds,
\end{split}  \label{eqdefTb0inv}
\ee
and their derivatives with respect to $E$ for $k \ge 1$
\bee T_{x_*; E}^{(k), G} f &=& - \int_{x_*}^r \pa^k_E \left( \frac{e^{-\sqrt E (r-s)}}{2\sqrt{E}}\right) f ds - \int_r^\infty \pa^k_E \left( \frac{e^{-\sqrt E (s-r)}}{2\sqrt{E}}\right) f ds,   \\
T_{x_*; E}^{(k), D} f &=&  \int_r^\infty \pa^k_E \left( \frac{e^{-\sqrt E (r-s)}}{2\sqrt{E}}\right) f ds - \int_r^\infty \pa^k_E \left( \frac{e^{-\sqrt E (s-r)}}{2\sqrt{E}}\right) f ds.
\eee
 Notice that 
\be 
\left| \pa_E^k e^{-\sqrt E r} \right| \lesssim_k r^k  e^{-\sqrt E r},  \label{eqexpdiff}
\ee
and that from the cancellation $\pa_E^k  \left( \frac{e^{-\sqrt E (r-s)} - e^{-\sqrt{E}(s-r)} }{2\sqrt{E}}\right)\big|_{r = s} = 0$, we have 
\bee
  \pa_r T_{x_*; E}^{(k), G} f &=&  \int_{x_*}^r \pa^k_E \left( \frac{e^{-\sqrt E (r-s)}}{2}\right) f ds - \int_r^\infty \pa^k_E \left( \frac{e^{-\sqrt E (s-r)}}{2}\right) f ds,   \\
  \pa_r T_{x_*; E}^{(k), D} f &=&  -\int_r^\infty \pa^k_E \left( \frac{e^{-\sqrt E (r-s)}}{2}\right) f ds - \int_r^\infty \pa^k_E \left( \frac{e^{-\sqrt E (s-r)}}{2}\right) f ds.
\eee
Fix $\a > 0$, these operators enjoy the boundedness: for $k, m \ge 0$, 
\be\begin{split}
  \left \| T_{x_*; E}^{(k), G}  f\right\|_{C^1_{r^{m+k} e^{-\sqrt E r}}([x_*, \infty))} &\lesssim_{k, m} x_*^{-1} \| f \|_{C^0_{r^{m-2} e^{-\sqrt E r}}([x_*, \infty)) }, \\
  \left \| T_{x_*; E}^{(k), D}  f\right\|_{C^1_{r^{m+k} e^{-(\sqrt E + \a) r}}([x_*, \infty))} &\lesssim_{k, m, \a} x_*^{-1} \| f \|_{C^0_{r^{m-2} e^{-(\sqrt E + \a) r}}([x_*, \infty)) },
  \end{split}\label{eqbddTk}
\ee
where the constant is independent of $x_* \ge 1$. 

\mbox{}

Now we invert $\HH_{0,\nu}$ in \eqref{eqnu} and define $\Phi_{j ;0, \l, \nu} = \left( \begin{array}{c}
     \Phi^1_j  \\
     \Phi^2_j 
\end{array} \right)$ as the solution of the integral equation 
\bea
\left| \begin{array}{l}
    \Phi^1_1 = S_1 + T_{x_*; 1 + \l}^{(0), G} F^+_1
    \\
    \Phi^2_1 = R_1 + T_{x_*; 1-\l}^{(0), D} \overline{F^-_1}
\end{array}\right. && 
\left| \begin{array}{l}
    \Phi^1_2 = S_2 + T_{x_*; 1 + \l}^{(0), D} F^+_2
    \\
    \Phi^2_2 = R_2 + T_{x_*; 1-\l}^{(0), G} \overline{F^-_2}
\end{array}\right. \label{eqb0fund}\\
% \left| \begin{array}{l}
%     \Phi^1_j = S_j + T_{x_*; 1+\l}^{(0), G} F^+_j
%     \\
%     \Phi^2_j = R_j + T_{x_*; 1-\l}^{(0), G} \overline{F^-_j}
% \end{array}\right.&& \quad j = 3, 4.
% \label{eqb0fund2}
\eea
% where $T_{x_*; E}^{(0), G}$ and $T_{x_*; E}^{(0), D}$ are the growing and decaying inversion operator of $\pa_r^2 - E$
% \bee T_{x_*; E}^{(0), G} f &=& - \int_{x_*}^r  \frac{e^{-\sqrt E (r-s)}}{2\sqrt{E}} f ds - \int_r^\infty \frac{e^{-\sqrt E (s-r)}}{2\sqrt{E}} f ds,   \\
% T_{x_*; E}^{(0), D} f &=&  \int_r^\infty \frac{e^{-\sqrt E (r-s)}}{2\sqrt{E}} f ds - \int_r^\infty \frac{e^{-\sqrt E (s-r)}}{2\sqrt{E}} f ds, 
% \eee
with the source terms
\bea
  \left( \begin{array}{c} 
  S_1 \\ R_1
  \end{array} \right) = 
  \left( \begin{array}{c} 
  e^{-\sqrt{1+\l} r}  \\ 0
  \end{array} \right),
  &&
  \left( \begin{array}{c} 
  S_2 \\ R_2
  \end{array} \right) = 
  \left( \begin{array}{c} 
  0 \\ e^{-\sqrt{1-\l} r}   
  \end{array} \right),
  \label{eqHH0solu}
  % \\
  %  \left( \begin{array}{c} 
  % S_3 \\ R_3
  % \end{array} \right) = 
  % \left( \begin{array}{c} 
  % e^{\sqrt{1+\l} r}  \\ 0
  % \end{array} \right),
  % &&
  % \left( \begin{array}{c} 
  % S_4 \\ R_4
  % \end{array} \right) = 
  % \left( \begin{array}{c} 
  % 0 \\ e^{\sqrt{1-\l} r}   
  % \end{array} \right).\nonumber
\eea
and 
\be F^+_j = \left(\frac{\nu^2 - \frac 14}{r^2} - W_1 \right) \Phi^1_j - W_2 \Phi^2_j,\quad F^-_j =\left (\frac{\nu^2 - \frac 14}{r^2} - W_1\right ) \overline{\Phi^2_j} - W_2 \overline{\Phi^1_j}. \label{eqb0Fpm} \ee
Thanks to the exponential decay $|W_1| + |W_2| \lesssim \la r\ra^{-\frac{d-1}{2}(p-1)} e^{-(p-1)r}$ and $r^{-2}$ decay, the boundedness of inversion operators above with $\a = \frac{p-1}{8} > 0$ then implies the existence and uniqueness of $\Phi_{j;0, \l, \nu}$ satisfying
\be
\begin{split}
 \| \Phi_{1;0, \l, \nu}^1 - e^{-\sqrt{1+\l}r} \|_{C^1_{e^{-\sqrt{1+\l}{r}}}([x_*, \infty)) } + 
 \| \Phi_{1;0, \l, \nu}^2 \|_{C^1_{e^{-(\sqrt{1-\l} + \frac{p-1}{8}){r}}}([x_*, \infty)) } &\lesssim x_*^{-1},\\
\| \Phi_{2;0, \l, \nu}^1 \|_{C^1_{e^{-(\sqrt{1+\l}+\frac{p-1}{8} ){r}}}([x_*, \infty)) } + \| \Phi_{2;0, \l, \nu}^2 - e^{-\sqrt{1-\l}r} \|_{C^1_{e^{-\sqrt{1-\l}{r}}}([x_*, \infty)) } &\lesssim x_*^{-1},
\end{split} \label{eqestPhib0} 
\ee
via contraction mapping principle in the corresponding space. Here we used $\frac{p-1}{8} \ge (2d)^{-1} \gg 2\delta_1 \ge |\sqrt{1 + \l} - \sqrt{1-\l} |$ from $s_c= \frac d2 - \frac{2}{p-1} > 0$ and \eqref{eqdelta1restrict} to absorb the off-diagonal term, and we need $x_* \gg 1$ depends on $\nu_0$ to ensure smallness of the potential. The faster decay of $\Phi^2_{1;0,\l,\nu}$ and $\Phi^1_{2;0,\l,\nu}$ also implies the non-degeneracy \eqref{eqnondegb0}.

For analyticity, we differentiate \eqref{eqb0fund} and \eqref{eqHH0solu} w.r.t. $\l$ for $1 \le N \le K_0$ times to obtain
\be \begin{split}
\left| \begin{array}{l}
    \pa_\l^N \Phi^1_1 = \pa_\l^N S_1 + \sum_{k=0}^N \binom{N}{k} T_{x_*; 1 + \l}^{(k), G} \pa_\l^{N-k} F^+_1
    \\
    \pa_\l^N \Phi^2_1 = \pa_\l^N R_1 + \sum_{k=0}^N (-1)^k \binom{N}{k} T_{x_*; 1 - \l}^{(k), D} \overline{\pa_{\bar\l}^{N-k} F^-_1}
\end{array}\right. \\
\left| \begin{array}{l}
    \pa_\l^N \Phi^1_2 = \pa_\l^N S_2 + \sum_{k=0}^N \binom{N}{k} T_{x_*; 1 + \l}^{(k), D} \pa_\l^{N-k} F^+_2
    \\
    \pa_\l^N \Phi^2_2 = \pa_\l^N R_2 + \sum_{k=0}^N (-1)^k \binom{N}{k} T_{x_*; 1 - \l}^{(k), G} \overline{\pa_{\bar\l}^{N-k} F^-_2}
\end{array}\right.\end{split} \label{eqb0fundN}
\ee
Again from \eqref{eqexpdiff} and boundedness of resolvents, we can show inductively that for $x_* \gg 1$ depending on $K_0$ and $\nu_0$ 
\be
\begin{split}
 \| \pa_\l^N \Phi_{1;0, \l, \nu}^1 \|_{C^1_{r^N e^{-\sqrt{1+\l}{r}}}([x_*, \infty)) } + \| \pa_\l^N \Phi_{1;0, \l, \nu}^2 \|_{C^1_{r^N e^{-(\sqrt{1-\l} + \frac{p-1}{8}){r}}}([x_*, \infty)) } \lesssim& 1,\\
\| \pa_\l^N \Phi_{2;0, \l, \nu}^1 \|_{C^1_{r^N e^{-(\sqrt{1+\l}+\frac{p-1}{8} ){r}}}([x_*, \infty)) } + \| \pa_\l^N \Phi_{2;0, \l, \nu}^2 \|_{C^1_{r^N e^{-\sqrt{1-\l}{r}}}([x_*, \infty)) } \lesssim& 1.
\end{split} \label{eqestPhib0N} 
\ee

\mbox{}\\

\underline{Step 2. Fundamental solution for $0 < b \le b_1$.}

This step aims to construct admissible fundamental solutions on $[x_*, \infty)$ for $0 < b \le b_1$ case with $b_1 \ll 1$ small enough and any $x_* \gg 1$ large enough depending on $\nu_0$. We might further shrink $b_1$ and increase $x_*$ in Step 3. From the local well-posedness of ODE \eqref{eqnu}, the solution can then be extended to $r \in (0, \infty)$. 

We will apply notations from Definition \ref{defWKBappsolu}, Definition \ref{defWKBaux} and Definition \ref{defdiffopspace} without citing them repeatedly. 

To begin with, we take complex conjugation on the second scalar equation of \eqref{eqnu} and rewrite the system as 
\be
\left\{ \begin{array}{l}
 \tilde H_{b, E_+} \phi = \left( h_{b, E_+} + \frac{\nu^2 - \frac 14}{r^2} - W_{1, b}\right) \phi - e^{i\frac{br^2}2} W_{2, b} \bar\varphi \\
 \tilde H_{b, \bar E_-} \varphi = \left( h_{b, \bar E_-} + \frac{\nu^2 - \frac 14}{r^2} - W_{1, b}\right) \varphi - e^{i\frac{br^2}2} W_{2, b} \bar\phi
 \end{array}\right. \label{eqsystemphivarphi}
\ee
where $E_\pm = 1 \pm (\l + ibs_c)$, $\tilde H_{b, E}$ is from \eqref{eqdeftildeHbE} and 
\be \Phi = \left( \begin{array}{c}
     \Phi^1 \\
     \Phi^2 
\end{array} \right) =: \left( \begin{array}{c}
     \phi \\
     \bar \varphi 
\end{array} \right). \label{eqPhitophi} \ee 
Notice that this system satisfies conjugate linearity inherited from linearity of \eqref{eqnu}: if $(\phi_j, \varphi_j)$ for $j = 1, 2$ solve \eqref{eqsystemphivarphi}, then so is $(z_1 \phi_1 + z_2 \phi_2, \overline{z_1} \varphi_1 + \overline{z_2} \varphi_2)$ for any $z_1, z_2 \in \CC$.
Let 
\[ r_1 = r^*_{b, E_+}, \quad r_2 = r^*_{b, \bar E_-}. \]
Without loss of generality, we assume 
\be  \Re \l \ge 0, \ee
which implies by \eqref{eqr*mono} that
\[  r_1 \ge r_2.\]
% since the construction for $\Re \l \le 0$ comes in exactly the same way. In particular, the construction on $\Re \l = 0$ is the same. 
The construction for $\Re \l \le 0$ is exactly symmetric and hence omitted. 

\mbox{}

\textit{Step 2.1. Set up of construction.}

We will construct $2 + 4 + 4$ solutions to \eqref{eqsystemphivarphi} on three regions
\[ I_{ext} = \left[r_1, \infty \right),\quad I_{con} = [r_2, r_1],\quad I_{mid} = [x_*, r_2] \]
respectively, where $x_* \gg 1$ is specified later. In each case, we will consider the following system
\be
   \left( \begin{array}{c} 
  \phi^{\Box}_j \\ \varphi^{\Box}_j
  \end{array} \right) = 
  \left( \begin{array}{c} 
  S^{\Box}_j \\ R^{\Box}_j
  \end{array} \right) 
  + 
\calT^\Box_j
   \left( \begin{array}{c} 
  \phi^{\Box}_j \\ \varphi^{\Box}_j
  \end{array} \right), \quad \calT^\Box_j := \left( \begin{array}{cc} 
  \calA^\Box_j & \\ & \calB^\Box_j
  \end{array} \right)
   \circ \calV, \label{eqI0} 
\ee
where the parameters range from 
\[ \Box \in \{ext, con, mid\},\quad j \in \left|\begin{array}{ll}\{ 1, 2\} & \Box = ext, \\ \{ 1, 2, 3, 4\} & \Box \in \{con, mid\}, \end{array}\right.\]
$\calA_j^\Box$, $\calB_j^\Box$ are the scalar inversion operators from  Lemma \ref{leminvtildeHext} and \ref{leminvtildeHmid}, $S_j^\Box$, $R_j^\Box$ are kernel of $\tilde H_{b, E_+}$, $\tilde H_{b, \bar E_-}$ respectively, 
, and the potential operator is
\be \calV   \left( \begin{array}{c} 
  \phi \\ \varphi
  \end{array} \right) := \left( \begin{array}{c} 
  \left( h_{b, E_+} + \frac{\nu^2 - \frac 14}{r^2}  - W_{1, b}\right) \phi - e^{i\frac{br^2}2} W_{2, b} \bar\varphi \\ 
  \left( h_{b, \bar E_-} + \frac{\nu^2 - \frac 14}{r^2} - W_{1, b}\right) \varphi - e^{i\frac{br^2}2} W_{2, b} \bar\phi
  \end{array} \right). \ee
  For notational simplicity, we denote 
  \be 
    \left( \begin{array}{c} 
  F^{\Box, +}_j \\ F^{\Box,-}_j
  \end{array} \right) := \calV \circ  \left( \begin{array}{c} 
  \phi^{\Box}_j \\ \varphi^{\Box}_j
  \end{array} \right). \label{eqdefFjBoxpm}
  \ee
Thereafter, solutions of \eqref{eqI0} will solve \eqref{eqsystemphivarphi}. We will construct solutions by finding suitable $\calA_j^\Box$, $\calB_j^\Box$, $S_j^\Box$, $R_j^\Box$ in each case and showing that \eqref{eqI0} is a linear contraction in some Banach space. Then match them to obtain $\Phi_{1;b,\l,\nu}$ and $\Phi_{2;b,\l,\nu}$. 

Finally, we mention the following estimates in preparation: 
\begin{enumerate}
    \item For parameters: \[ |E_\pm - 1| \le 2 \delta_1,\qquad r_1, r_2 \in \left[\frac{2-6\delta_1}{b}, \frac{2+6\delta_1}{b}\right],
\]
This follows from \eqref{eqdelta1restrict} and \eqref{eqr*bE1}.
\item For potentials: with $E \in \{E_+, \bar E_-\}$, we have
\bea 
 \left|h_{b, E} + \frac{\nu^2 - \frac 14}{r^2} - W_{1, b}\right| \lesssim_\nu r^{-2}, \quad |W_{2, b}| \lesssim e^{-\frac{p-1}{2}\min\{ r, 2b^{-1} \} } \la r \ra^{-2},\quad {\rm for}\,\, r > 0;\label{eqpotentialest} \\
  \left| \pa_r^n \left(h_{b, E} + \frac{\nu^2 - \frac 14}{r^2} - W_{1, b} \right) \right| \lesssim_{n, \nu} r^{-2-n},\quad {\rm for}\,\, r \ge 4b^{-1}, \,\,n \ge 1;\label{eqpotentialest2}\\
  \left| D_{--;b,1}^n (e^{i\frac{br^2}{2}} W_{2, b}) \right| \lesssim_n e^{-\frac{p-1}{b}} r^{-2-n},\quad {\rm for}\,\, r \ge 4b^{-1}, \,\,n \ge 1.\label{eqpotentialest3}
\eea
They follow from \eqref{eqbddh}, Proposition \ref{propQbasymp} and Proposition \ref{propQbasympref} (1). In particular, for \eqref{eqpotentialest2} and \eqref{eqpotentialest3}, we exploited the Fa\'a di Bruno's formula \eqref{eqFaadiBruno} for estimating $\pa_r^n(|Q_b|^{\a})$ and the Leibniz rule to compute
\[  D_{--;b, 1}^n \left( e^{i\frac{br^2}{2}}W_{2,b}\right) =  \sum_{j, k, l}
\binom{n}{j \, k\, l} \pa_r^j |Q_b|^{p-3}\cdot D_{-; b, 1}^{k} (e^{i\frac{br^2}{4}} Q_b)\cdot D_{-; b, 1}^{l} (e^{i\frac{br^2}{4}} Q_b).  \] 
We emphasize that the constants in \eqref{eqpotentialest}-\eqref{eqpotentialest3} are independent of $s_c, b, \l$. 
\end{enumerate}

% Note that we can set $b_1 \ll 1$ so that $x_* \le b^{-\frac 12}$.

\mbox{}

\textit{Step 2.2 Comparison estimates of weight functions and choice of $\delta_1$.}

We claim for any $|\l| \le \delta_1$, $0 < b \le b_1$ with $\delta_1, b_1$ satisfying \eqref{eqrequestb0}, there exists a universal constant $C_1 > 0$ such that
\bea
  % \max\left\{ \frac{\omega_{b, E_+}^\pm}{\omega_{b, \bar E_-}^\pm}, \frac{\omega_{b, \bar E_-}^\pm}{\omega_{b, E_+}^\pm}\right\}  \le e^{C\frac{\delta_1}{b} },\quad r\ge \frac 2b. \label{eqcomparisonomega} \\
 \frac 12  e^{-C_1\delta_1 \min\{r, b^{-1}\} } \le \frac{\omega_{b, \bar E_-}^\pm}{\omega_{b, E_+}^\pm}e^{\mp \frac{\pi(E_+ - \bar E_-)}{2b}} 
  \le 2e^{C_1\delta_1 \min\{r, b^{-1}\} },&& \forall\, r\ge 0, \label{eqcomparisonomega}\\
   \sup_{\substack{|r-\frac 2b| \le 6\delta_1b^{-1} \\ |E-1| \le 2 \delta_1} } \{ \omega_{b, E}^\pm(r), (\omega_{b, E}^\pm(r))^{-1}, e^{\pm\Re \eta_{b, E}(r)} \} \le e^{\frac{C_1\delta_1}{b}}. &&
   % \forall\, |E-1| \le 2\delta_1. 
   \label{eqcomparisonomega2}
\eea
% and the monotonicity 
% \be 
% {\rm sgn}\left( \pa_r \left( \Re \eta_{b, E_+} - \Re \eta_{b, \bar E_-}\right)\right) = {\rm sgn}(E_+ - \bar E_-) = {\rm sgn}(\Re \l) ,\quad r < \min\{ r_1, r_2\}.  \label{eqcomparisonmono}
% \ee

Indeed, 
for \eqref{eqcomparisonomega}, we first notice that 
\bee
  \sup_{\left| r - \frac 2b \right| \ge \frac{6\delta_1}{b}} \left| \frac{\la b^{-\frac 23} (b^2r^2 - 4\bar E_-) \ra^{\pm \frac 14}}{\la b^{-\frac 23} (b^2r^2 - 4E_+) \ra^{\pm \frac 14}} \right| \le   1 + C\delta_1, \quad \sup_{\substack{\left| r - \frac 2b \right| \le \frac{6\delta_1}{b} \\ |E-1| \le 2\delta_1  } } \la b^{-\frac 23} (b^2r^2 - 4E) \ra^{\frac 14} \le e^{\frac{C \delta_1}{b^\frac 23}},
\eee
with some $C > 0$. Here we have used $|r - \frac 2b| \ge 6 \delta_1 b^{-1} \Rightarrow |r - \frac {2\sqrt E}{b}| \ge 2 \delta_1 b^{-1} \ge 2 b^{-\frac 13}$ for $E \in \{ E_+, \bar E_-\}$ since $\delta_1 \le b^{-\frac 23}$ from \eqref{eqrequestb0}. So we can control this weight in $\frac{\omega_{b, \bar E_-}^\pm}{\omega_{b, E_+}^\pm}$ and reduce \eqref{eqcomparisonomega} to 
\be \left|\Re \eta_{b, E_+} - \Re \eta_{b, \bar E_-} + \frac{\pi(E_+ - \bar E_-)}{2b}\right| \lesssim \delta_1 \min\{ r, b^{-1}\}, \quad r \ge 0. \label{eqomegadiff} \ee
For $r \ge \frac 4b$, we use \eqref{eqReetaext} and that $\Im E_+ = \Im \bar E_- = \Im \l + bs_c$. For $r \in [\frac {2+6\delta_1}b, \frac 4b]$, we notice that $|\Re \eta_{b, E}(r)| \lesssim \delta_1 b^{-1}$ at $r = \frac 4b$ from \eqref{eqReetaext}; this estimate holds for $r \in [r^*_{b, E}, \frac 4b] \supset [\frac {2+6\delta_1}b, \frac 4b]$ thanks to the monotonicity \eqref{eqReetamono} and estimate of $r^*_{b, E}$ \eqref{eqr*bE1}. For $|r - \frac 2b| \le 6\delta_1 b^{-1}$, we use \eqref{eqetaabs} to obtain $|\eta_{b, E}(r)| \lesssim \delta_1 b^{-1}$ for $|E-1| \le 2\delta_1$. 
Finally, for $r \le \frac{2 - 6\delta_1}{b}$, we have $r \le \min \{ r^*_{b, E_+}, r^*_{b, \bar E_-} \}$ due to \eqref{eqr*bE1} and  $\Im E_+ = \Im \bar E_-$, so we will compute the difference via \eqref{eqformulaeta}: if $\Im E \le 0$,
\bee
  &&\left|\eta_{b, E_+}(r) - \eta_{b, \bar E_-}(r) + \frac{\pi(E_+ - \bar E_-)}{2b}\right| \\
  &=& \left| -i \frac{2E_+}{b} \int_1^{\frac{br}{2\sqrt{E_+}}} (w^2 - 1)^\frac 12 dw + i\frac{2\bar E_-}{b} \int_1^{\frac{br}{2\sqrt{\bar E_-}}} (w^2 - 1)^\frac 12 dw + \frac{\pi(E_+ - \bar E_-)}{2b} \right| \qquad \\
  &\le& \left|-i \frac{2E_+}{b} \int_0^{\frac{br}{2\sqrt{E_+}}} (w^2 - 1)^\frac 12 dw \right| + \left|i \frac{2 \bar E_-}{b} \int_0^{\frac{br}{2\sqrt{\bar E_-}}} (w^2 - 1)^\frac 12 dw \right| \lesssim \delta_1 r
\eee
where all the integrand takes the branch $ \arg(w^2 - 1) \in (0, 2\pi)$, and we have used $\int_0^1 (1- w^2)^\frac 12 dw = \frac{\pi}{4}$. The case $\Im E < 0$ comes similarly. That concludes the proof of \eqref{eqomegadiff} in all regions, and hence the proof of \eqref{eqcomparisonomega}.

The estimate \eqref{eqcomparisonomega2} was already proven above during the discussion on $|r-\frac 2b| \le 6\delta_1 b^{-1}$. Thus we have completed the proof of \eqref{eqcomparisonomega}-\eqref{eqcomparisonomega2}. 

% For \eqref{eqcomparisonmono}, we recall \eqref{eqReetaderiv} to compute
% \bee 
% \pa_r \left( \Re \eta_{b, E_+} - \Re \eta_{b, \bar E_-}\right) &=& \Im \left[ \left( \frac{b^2 r^2}{4} - E_+\right)^\frac 12 - \left( \frac{b^2 r^2}{4} - \bar E_-\right)^\frac 12 \right] \\
% &=& \Im \left[ \frac{\bar E_- - E_+}{\left( \frac{b^2 r^2}{4} - E_+\right)^\frac 12 + \left( \frac{b^2 r^2}{4} - \bar E_-\right)^\frac 12}\right]. 
% \eee
% The sign \eqref{eqcomparisonmono} now follows $\bar E_- - E_+ \in \RR$ and $\Im \left( \frac{b^2 r^2}{4} - E\right)^\frac 12 > 0$ for $E = E_+, \bar E_-$ and $r < r_2$ according to \eqref{eqReetamono}. That concludes the proof of claim. 

\mbox{}

Thereafter, we can choose $\delta_1$ as
\be
\delta_1 = \min \left\{ \frac 12 \delta_0, (100d)^{-1}, (64C_1 d)^{-1} \right\} \label{eqdelta1choice}
\ee
so that $p-1 \ge \frac 4d \ge 256 C_1 \delta_1$ and we can reformulate \eqref{eqcomparisonomega}-\eqref{eqcomparisonomega2} as 
\begin{align}
  % \max\left\{ \frac{\omega_{b, E_+}^\pm}{\omega_{b, \bar E_-}^\pm}, \frac{\omega_{b, \bar E_-}^\pm}{\omega_{b, E_+}^\pm}\right\}  \le e^{C\frac{\delta_1}{b} },\quad r\ge \frac 2b. \label{eqcomparisonomega} \\
 &\frac 12  e^{- \frac{p-1}{256} \min\{r, b^{-1}\} } \le \frac{\omega_{b, \bar E_-}^\pm}{\omega_{b, E_+}^\pm}e^{\mp \frac{\pi(E_+ - \bar E_-)}{2b}} 
  \le 2e^{\frac{p-1}{256} \min\{r, b^{-1}\} },\quad \forall\, r > 0, \label{eqcomparisonomega3}\\
   &\sup_{\substack{|r-\frac 2b| \le 6\delta_1b^{-1} } } \{ \omega_{b, E}^\pm(r), (\omega_{b, E}^\pm(r))^{-1}, e^{\pm\Re \eta_{b, E}(r)} \} \le e^{\frac{p-1}{256 b}},\quad \forall\, |E-1| \le 2\delta_1. \label{eqcomparisonomega4}
\end{align}

\mbox{}

\textit{Step 2.3. On the exterior region.} 

Note that $\Im E_+ = \Im \bar E_- = \Im \l + bs_c$, we define for $\a \ge 0$ that
\be N_\a = \min \left\{ N \in \NN_{\ge 0}, N \ge \frac{\Im \l}{b} + s_c + \frac \a 2 \right\} \le I_0 + 1 + \frac \a 2,\quad \a \ge 0. \label{eqchoiceNI1}\ee
Consider the system \eqref{eqI0} with source terms
% By inverting $\tilde H_{b, E}$ in \eqref{eqsystemphivarphi}, we define
% \be
%   \left( \begin{array}{c} 
%   \phi^{ext}_j \\ \varphi^{ext}_j
%   \end{array} \right) = 
%   \left( \begin{array}{c} 
%   S^{ext}_j \\ R^{ext}_j
%   \end{array} \right) 
%   + 
%   \calT^{ext} \left( \begin{array}{c} 
%   \phi^{ext}_j \\ \varphi^{ext}_j
%   \end{array} \right) \label{eqIext}
% \ee
% where the source terms are
\be
  \left( \begin{array}{c} 
  S^{ext}_1 \\ R^{ext}_1
  \end{array} \right) = 
  \left( \begin{array}{c} 
  \psi_1^{b, E_+} \\ 0
  \end{array} \right),
  \quad
  \left( \begin{array}{c} 
  S^{ext}_2 \\ R^{ext}_2
  \end{array} \right) = 
  \left( \begin{array}{c} 
  0 \\ \psi_1^{b, \bar E_-} 
  \end{array} \right), \label{eqIextsource}
\ee
and the scalar operators 
\be \calA_j^{ext} = \tilde T_{\frac 4b; b, E_+}^{ext},\quad \calB_j^{ext} = \tilde T_{\frac 4b; b, \bar E_-}^{ext},\quad {\rm for}\,\, j = 1, 2. \label{eqdefcalText}\ee 
Since the linear operator $\calT_j^{ext}$ as in \eqref{eqI0} is independent of $j$, we denote $\calT^{ext} := \calT^{ext}_1 = \calT^{ext}_2$. 

We claim that for $\a \ge 0$,
\be
  \| \mathbf{\calT}^{ext} \|_{\calL\left( X^{\a, N_\a, -}_{r_1, \frac 4b; b, E_+} \times e^{\pm \frac{p-1}{64b}} X^{\a, N_\a, -}_{r_1, \frac 4b; b, \bar E_-} \right)} \lesssim_{\a, \nu} b.\label{eqestcalText}
\ee
and for any $N \ge N_\a$,
\be
   \| \mathbf{\calT}^{ext} \|_{\left( X^{\a, N, -}_{r_1, \frac 4b; b, E_+} \times e^{\pm \frac{p-1}{64b}} X^{\a, N, -}_{r_1, \frac 4b; b, \bar E_-} \right) \to \left( X^{\a, N+1, -}_{r_1, \frac 4b; b, E_+} \times e^{\pm \frac{p-1}{64b}} X^{\a, N+1, -}_{r_1, \frac 4b; b, \bar E_-} \right)} \lesssim_{\a, \nu, N}  1.\label{eqestcalText1}
\ee
 Indeed, with $N_\a$ as \eqref{eqchoiceNI1} and $E = E_+$ or $\bar E_-$ satisfying \eqref{eqcondinv}, 
% we have boundedness of $\tilde T^{\rm ext}_{r_1; b, E}$ in $X^{-2, N, -}_{r_1, \frac 4b; b, E}$. 
the estimates 
\eqref{eqestcalText}-\eqref{eqestcalText1} follow from the boundedness of $\tilde T^{ext}_{r_1; b, E}$ on $X^{\a-2, N, -}_{r_1, \frac 4b; b, E}$ \eqref{eqbddHbnuinvext} and the following estimates of potential: for $(E_1, E_2) \in \{ (E_+, \bar E_-), (\bar E_-, E_+)\}$, 
\bea
 \left\| \left( h_{b, E} + \frac{\nu^2 - \frac 14}{r^2} - W_{1, b}\right) f \right\|_{X^{\a-2, N, -}_{r_1, \frac 4b; b, E_1}} \lesssim_{\nu, N} \|  f  \|_{X^{\a, N, -}_{r_1, \frac 4b; b, E_1}} \label{eqfundI1est1} \\
 \left\|e^{i\frac{br^2}{2}}W_{2,b} \bar f \right\|_{X^{\a-2, N, -}_{r_1, \frac 4b; b, E_1}} \lesssim_N e^{-\frac{p-1}{8b}} \| f \|_{X^{\a, N, -}_{r_1, \frac 4b; b, E_2}}
\label{eqfundI1est2}
\eea
The first estimate \eqref{eqfundI1est1} easily follows from \eqref{eqpotentialest} and \eqref{eqpotentialest2}, and we now prove \eqref{eqfundI1est2}. For simplicity, we only consider the case $(E_1, E_2) = (E_+, \bar E_-)$. 
Define the auxiliary function 
$$\rho(r) = 2\left( \frac{b^2 r^2}{4} - 1 \right)^\frac 12 - \left( \frac{b^2 r^2}{4} - E_+ \right)^\frac 12 - \left( \frac{b^2 r^2}{4} - E_- \right)^\frac 12,$$
since $E_+ + E_- = 2$, $|1 - E_\pm| \le 2\delta_1$, we have the elementary estimate
\[ \left| \pa_r^n \rho \right| \lesssim_n b^{-1} r^{-1-n},\quad \forall\,\, n \ge 0, \quad r \ge \frac 4b.  \]
Recalling the differential operator from Definition \ref{defdiffopspace},  we apply the Leibniz rule and \eqref{eqpotentialest3} to compute
\begin{align*}
   &\left| (D_{-;b, E_+} -i \rho )^n \left( e^{i\frac{br^2}{2}}W_{2,b} \bar f \right)\right| = \left| \sum_{m=0}^n \binom{n}{m} \pa_r^{n-m} D_{--;b,1}(e^{i\frac{br^2}{2}} W_{2,b}) \cdot\overline{D_{-; b, \bar E_-}^m f}  \right| \\
  % &=& \left| \sum_{j, k, l, m} \binom{n}{j \, k\, l \, m} \pa_r^j |Q_b|^{p-3}\cdot D_{-; b, 1}^{k} (e^{i\frac{br^2}{4}} Q_b)\cdot D_{-; b, 1}^{l} (e^{i\frac{br^2}{4}} Q_b)\cdot \overline{D_{-; b, \bar E_-}^m \varphi^{ext}} \right|\\
  \lesssim_n&\,\, e^{-\frac{p-1}{4b}} r^{\a-n-2} \omega_{b, \bar E_-}^-(r)  \| f \|_{X^{\a, N, -}_{r_1, \frac 4b; b, \bar E_-}},\qquad \forall \,\, n \ge 0,\quad r \ge \frac 4b.
\end{align*}
Notice that 
\[ D_{-;b, E_+}^n =\left[ e^{-i\int^r_{\frac 4b} \rho} (D_{-; b, E_+} - i\rho) e^{i\int^r_{\frac 4b} \rho} \right]^n =  e^{-i\int^r_{\frac 4b} \rho} (D_{-; b, E_+} - i\rho)^n e^{i\int^r_{\frac 4b} \rho}.  \]
Hence
\bea
  && \left| (D_{-;b, E_+}^n \left( e^{i\frac{br^2}{2}}W_{2,b} \bar f \right)\right| \nonumber \\
  &\le& \sum_{k = 0}^n \binom{k}{n} \left| e^{-i\int^r_{\frac 4b} \rho} \pa_r^k \left( e^{i\int^r_{\frac 4b} \rho} \right) \right|\cdot \left| (D_{-;b, E_+} -i \rho )^{n-k} \left( e^{i\frac{br^2}{2}}W_{2,b} \bar f \right)\right|\nonumber
  \\
  &\lesssim_n&  b^{-n}e^{-\frac{p-1}{b}} r^{\a-n-2} \omega_{b, \bar E_-}^-(r)  \| f \|_{X^{\a, N, -}_{r_1, \frac 4b; b, \bar E_-}}\nonumber \\
  &\le & e^{-\frac{p-1}{2b}} r^{\a-n-2} \omega_{b,  E_+}^-(r)  \| f \|_{X^{\a, N, -}_{r_1, \frac 4b; b, \bar E_-}}, \quad \forall \,\, n \ge 0,\quad r \ge \frac 4b. \nonumber
\eea
where in the last inequality we used \eqref{eqcomparisonomega3}. Also a simple bound for $[r_1, \frac 4b]$ using again \eqref{eqcomparisonomega3} plus \eqref{eqpotentialest} is 
\bee
  \left|e^{i\frac{br^2}{2}}W_{2,b} \bar f \right| &\lesssim& e^{-\frac{p-1}{b}} \omega_{b, \bar E_-}^-(r) r^{\a-2} \| f \|_{X^{\a, N, -}_{r_1, \frac 4b; b, \bar E_-}} \\
  &\lesssim& e^{-\frac{p-1}{2b}} r^{\a-2} \omega_{b,  E_+}^-(r)  \| f \|_{X^{\a, N, -}_{r_1, \frac 4b; b, \bar E_-}}, \quad r_1 \le r \le \frac 4b. 
\eee
These two estimates combined yield \eqref{eqfundI1est2}. 

% indicates
% \be \left\|e^{i\frac{br^2}{2}}W_{2,b} \bar \varphi^{ext} \right\|_{X^{\a-2, N, -}_{r_1, \frac 4b; b, E_+}} \lesssim_N \| \varphi^{ext} \|_{X^{\a, N, -}_{r_1, \frac 4b; b, \bar E_-}}
% \label{eqfundI1est2}
% \ee
% Similarly, we have 
% \be \left\|e^{i\frac{br^2}{2}}W_{2,b} \bar \phi^{ext} \right\|_{X^{\a-2, N, -}_{r_1, \frac 4b; b, \bar E_-}} \lesssim_N \| \phi^{ext} \|_{X^{\a, N, -}_{r_1, \frac 4b; b, E_+}}
% \label{eqfundI1est3}
% \ee

% With $N_\a$ as \eqref{eqchoiceNI1} and $E = E_+$ or $\bar E_-$ satisfying \eqref{eqcondinv}, 
% % we have boundedness of $\tilde T^{\rm ext}_{r_1; b, E}$ in $X^{-2, N, -}_{r_1, \frac 4b; b, E}$. 
% the estimates 
% \eqref{eqestcalText}-\eqref{eqestcalText1} follow from the boundedness of $\tilde T^{ext}_{r_1; b, E}$ on $X^{\a-2, N, -}_{r_1, \frac 4b; b, E}$ \eqref{eqbddHbnuinvext} and the boundedness of potential  \eqref{eqfundI1est1}-\eqref{eqfundI1est3}.

Hence using \eqref{eqestcalText} and boundedness of $\psi_1^{b, E}$ \eqref{eqestpsi13Xspace}, with $b_1$ small enough depending on $\nu_0$, contraction mapping principle implies the existence of unique solutions $(\phi^{ext}_j, \varphi^{ext}_j)$ for $j = 1, 2$, $\Box = ext$ to \eqref{eqI0} with
\be 
\begin{split}
  \| \phi^{ext}_1 - \psi_1^{b, E_+}\|_{X^{0, N_0, -}_{r_1, \frac 4b; b, E_+}} + \| \varphi^{ext}_2 - \psi_2^{b, \bar E_-}\|_{X^{0, N_0, -}_{r_1, \frac 4b; b, \bar E_-}} &\lesssim b, \\
  \| \phi^{ext}_2 \|_{X^{0, N_0, -}_{r_1, \frac 4b; b, E_+}} + \| \varphi^{ext}_2 \|_{X^{0, N_0, -}_{r_1, \frac 4b; b, \bar E_-}} &\lesssim b e^{-\frac{p-1}{64b}}
  \end{split}
  \label{eqPhiextnondeg}
\ee
Moreover, by iterating \eqref{eqestcalText1} and rewriting in $\Phi^{ext}_j$ variables, we have improved regularity
\be
  \| \Phi^{ext, 1}_j\|_{X^{0, N, -}_{r_1, \frac 4b; b, E_+}} + \| \overline{ \Phi^{ext, 2}_j}\|_{X^{0, N, -}_{r_1, \frac 4b; b, \bar E_-}} \lesssim_N 1,\quad \forall N \ge 0,\,\, j = 1, 2.  
  \label{eqPhiextbdd}
\ee
From Lemma \ref{leminvtildeHext} (4) and going back to $\Phi$ using \eqref{eqPhitophi}, we can evaluate the boundary value of  $\Phi^{ext}_j$ at $r_1$:
\be
  \left( \begin{array}{c}
       \vec\Phi^{ext}_1 (r_1) \\  \vec\Phi^{ext}_2 (r_1)
  \end{array}\right)
  =
  \left( \begin{array}{cccc}
       1 + \gamma^{ext}_{11} & \gamma^{ext}_{12} & \gamma^{ext}_{13} & 
       \gamma^{ext}_{14}  \\
       \gamma^{ext}_{21} & \gamma^{ext}_{22} &  1 + \gamma^{ext}_{23} & \gamma^{ext}_{24}
  \end{array}\right)
  \left( \begin{array}{c}
       \vec\psi_1^{b, E_+}(r_1) \otimes \vec 0 \\
       \vec\psi_3^{b, E_+}(r_1) \otimes \vec 0 \\
       \vec 0 \otimes \overline{\vec\psi_1^{b, \bar E_-}}(r_1)  \\
       \vec 0 \otimes \overline{\vec\psi_3^{b, \bar E_-}}(r_1)
  \end{array}\right)
  \label{eqbdryext}
\ee
where the vector form follows \eqref{eqvectorform} and for $j = 1, 2$, 
\be
\begin{split}
  \gamma^{ext}_{j1} = \int_{\frac 4b}^{r_1} \frac{\psi_3^{b, E_+} }{W_{31;E_+}} F^{ext, +}_{j} ds,&\quad  
  \gamma^{ext}_{j2} = \int_{r_1}^{\frac 4b}\frac{\psi_1^{b, E_+} }{W_{31;E_+}} F^{ext, +}_{j} ds + \calI^-_{N;b, E_+} [F^{ext, +}_{j}]\left( \frac 4b\right) \\
  \overline{\gamma^{ext}_{j3}} = \int_{\frac 4b}^{r_1} \frac{\psi_3^{b, \bar E_-} }{W_{31; \bar E_-}} F^{ext, -}_{j} ds,&\quad 
  \overline{\gamma^{ext}_{j4}} = \int_{r_1}^{\frac 4b}\frac{\psi_1^{b, \bar E_-} }{W_{31;\bar E_-}} F^{ext, -}_{j} ds + \calI^-_{N;b, \bar E_-} [F^{ext, -}_{j}]\left( \frac 4b\right)
\end{split} \label{eqdefgammaext}
\ee
with $F^{ext, \pm}_j$ defined as in \eqref{eqdefFjBoxpm}. 
% \be
% \begin{split}
%  F^{ext, +}_{j} &= \left( h_{b, E_+} + \frac{\nu^2 - \frac 14}{r^2} - W_{1, b} \right) \Phi^{ext, 1}_j - e^{i\frac{br^2}{2}} W_{2,b} \Phi^{ext, 2}_j \\
%  F^{ext, -}_{j} &= \left( h_{b, \bar E_-} + \frac{\nu^2 - \frac 14}{r^2} - W_{1, b} \right) \overline{\Phi^{ext, 2}_j} - e^{i\frac{br^2}{2}} W_{2,b} \overline{\Phi^{ext, 1}_j}.\end{split} \label{eqFextpm}
% \ee
Lemma \ref{leminvtildeHext} (4) also implies the estimate of these coefficients:
    \be 
 (\gamma^{ext}_{jk})_{\substack{1 \le j \le 2\\ 1 \le k \le 4}} = b \left( \begin{array}{cccc}
        O(1) & O(1) & O(e^{-\frac{p-1}{64b}}) & O(e^{-\frac{p-1}{64b}} e^{-2\eta_{b, \bar E_-, \nu}(r_1)} ) \\
        O(e^{-\frac{p-1}{64b}}) & O(e^{-\frac{p-1}{64b}}) &   O(1) & O(e^{-2\eta_{b, \bar E_-, \nu}(r_1)})
  \end{array}\right) \label{eqestgammaext} 
    \ee
% for $j = 1, 2$,
% \be \gamma^{ext}_{jk} 
% = \left| \begin{array}{ll} O(b), & j = 1, 2;\,\, k = 1, 2, 3;\\
%  O(be^{-2\eta_{b, \bar E_-}(r_1)} ), & (j,k) = (1, 4), (2, 4);
% \end{array}\right.
% % = O(b),\,\,\, k = 1, 2, 3;\quad  \gamma^{ext}_{j4} = O(b e^{-2\eta_{b, \bar E_-}(r_1)}). 
% \label{eqestgammaext} 
% \ee
where we have used $e^{-2\eta_{b, E_+}(r_1)} = 1$. 
We stress that $|e^{-2\eta_{b, \bar E_-}(r_1)}| \lesssim 1$  with $\Im \l \le bI_0$, and that $\Phi^{ext}_j$, $\gamma^{ext}_{jk}$, $F^{ext, \pm}_j$ all depend on $b, \l, \nu$, and so are the corresponding notations on $I_{con}$ and $I_{mid}$ constructed below.

% \be
%   \left( \begin{array}{c}
%        \vec\Phi^{ext}_1 (r_1) \\  \vec\Phi^{ext}_2 (r_1)
%   \end{array}\right)
%   =
%   \left( \begin{array}{cccc}
%        1 + O(b) & O(b) & O(b) & O(b e^{-2\eta_{b, \bar E_-}(r_1)})  \\
%        O(b) & O(b) &  1 + O(b) & O(b e^{-2\eta_{b, \bar E_-}(r_1)}) 
%   \end{array}\right)
%   \left( \begin{array}{c}
%        \vec\psi_1^{b, E_+}(r_1) \otimes \vec 0 \\
%        \vec\psi_3^{b, E_+}(r_1) \otimes \vec 0 \\
%        \vec 0 \otimes \overline{\vec\psi_1^{b, \bar E_-}}(r_1)  \\
%        \vec 0 \otimes \overline{\vec\psi_3^{b, \bar E_-}}(r_1)
%   \end{array}\right)
%   \label{eqbdryext}
% \ee
% where the vector form follows \eqref{eqvectorform} and we have used $e^{-2\eta_{b, E_+}(r_1)} = 1$. 

% \bee
% \pa_r^k\phi^{ext}_1(r_1) &=& (1 + \beta_{11}) \pa_r^k \psi_1^{b, E_+}(r_1) + \gamma_{11} \pa_r^k \psi_3^{b, E_+}(r_1)  \\
% \pa_r^k\varphi^{ext}_1(r_1) &=& \beta_{12} \pa_r^k \psi_1^{b, \bar E_-}(r_1) + \gamma_{12} \pa_r^k \psi_3^{b, \bar E_-} (r_1) \\
%   \pa_r^k\phi^{ext}_2(r_1) &=&  \beta_{21} \pa_r^k \psi_1^{b, E_+}(r_1) + \gamma_{21} \pa_r^k \psi_3^{b, E_+}(r_1)  \\
% \pa_r^k\varphi^{ext}_2(r_1) &=& (1 + \beta_{22}) \pa_r^k \psi_1^{b, \bar E_-}(r_1) + \gamma_{22} \pa_r^k \psi_3^{b, \bar E_-} (r_1) 
% \eee
% with $k = 0, 1$ and 
% \bee
%   |\beta_{j1}| + |\beta_{j2}| \lesssim b, && j = 1, 2; \\
%   |\gamma_{j1}| \lesssim b|e^{-2\eta_{b, E_+}(r_1)}|= b, \quad  |\gamma_{j2}| \lesssim b|e^{-2\eta_{b, \bar E_-}(r_1)}|, && j = 1, 2.
% \eee

\mbox{}

\textit{Step 2.4. On the connection region $I_{con}$.} 

On $I_{con}$, we consider the system \eqref{eqI0} with $j \in\{ 1, 2, 3, 4\}$, scalar operators (independent of $j$)
\[ \calA_j^{con} = \tilde T^{mid, G}_{r_2, r_1; b, E_+},\quad \calB_j^{con} = \tilde T^{ext}_{r_2; b, \bar E_-} \circ \mathbbm{1}_{[r_2, r_1]},  \]
and source terms 
\be
\begin{split}
  \left(\begin{array}{c}
       S^{con}_1  \\
       R^{con}_1
  \end{array} \right) = 
   \left(\begin{array}{c}
       \psi_4^{b, E_+}  \\
       0
  \end{array} \right),&\quad
  \left(\begin{array}{c}
       S^{con}_2  \\
       R^{con}_2
  \end{array} \right) = 
   \left(\begin{array}{c}
       \psi_2^{b, E_+}  \\
       0
  \end{array} \right),\\
  \left(\begin{array}{c}
       S^{con}_3  \\
       R^{con}_3
  \end{array} \right) = 
   \left(\begin{array}{c}
          0  \\
       \psi_1^{b, \bar E_-}
  \end{array} \right),&\quad
  \left(\begin{array}{c}
       S^{con}_4  \\
       R^{con}_4
  \end{array} \right) = 
   \left(\begin{array}{c}
       0 \\
       \psi_3^{b,\bar E_-}
  \end{array} \right).
  \end{split} \label{eqIconsource}
\ee

Notice that $r \sim b^{-1}$ on $I_{con}$. Hence from the boundedness \eqref{eqtildeTmidGest4}, \eqref{eqtildeTmidGest2} and \eqref{eqbddHbnuinvext2} plus the $L^\infty$ bound of potentials \eqref{eqpotentialest} and comparison estimate of weights \eqref{eqcomparisonomega4},
we see the linear operators $\calT^{con}_j$ for $j \in \{ 1, 2, 3, 4\}$ are bounded respectively in $e^{(1-\iota_j) \frac{p-1}{64b}} C^0_{\omega_{b, E_+}^{\tilde{\iota}_j}}(I_{con}) \times e^{\iota_j \frac{p-1}{64b}} C^0_{\omega_{b, \bar E_-}^{\tilde{\iota}_j}}(I_{con})$ with $\tilde{\iota}_j = \left|\begin{array}{ll}
    - & j = 1, 3, \\
    + & j = 2, 4,
\end{array}\right.$ and $\iota_j = \left|\begin{array}{ll}
    1 & j = 1, 2, \\
    0 & j = 3, 4,
\end{array}\right.$ and their operator norms are $O(b)$. Hence with $b_1$ small enough depending on $\nu_0$, the contraction mapping principle yields existence and uniqueness of solutions $(\phi^{con}_j, \varphi^{con}_j)$ for $j = 1 , 2 ,3, 4$ satisfying
\bee
   e^{(1-\iota_j) \frac{p-1}{64b}} \| \phi^{con}_j\|_{C^0_{\omega_{b, E_+}^{\tilde{\iota}_j}}(I_{con})} +  e^{\iota_j \frac{p-1}{64b}} \| \varphi^{con}_j\|_{C^0_{\omega_{b, \bar E_-}^{\tilde{\iota}_j}}(I_{con})} \lesssim 1,
   % \quad j = 1, 2; \\
   % e^{(1-\iota_j) \frac{p-1}{64b}} \| \phi^{con}_j\|_{C^0_{\omega_{b, E_+}^+}(I_{con})} +  e^{\iota_j \frac{p-1}{64b}} \| \varphi^{con}_j\|_{C^0_{\omega_{b, \bar E_-}^+}(I_{con})} \lesssim 1,\quad j = 3, 4.
\eee
where $\tilde{\iota}_j, \iota_j$ are defined above. 

From Lemma \ref{leminvtildeHext} and Lemma \ref{leminvtildeHmid}, we also have the boundary values (written in vector form \eqref{eqvectorform})
\be
\begin{split}
  \small \left( \begin{array}{c}
       \vec \Phi^{con}_1 (r_1)\\
       \vec \Phi^{con}_2 (r_1)\\
       \vec \Phi^{con}_3 (r_1)\\
       \vec \Phi^{con}_4 (r_1)
  \end{array}\right) = 
  (I + \Gamma^{con, R})
  \left( \begin{array}{c}
       \vec\psi_4^{b, E_+}(r_1) \otimes \vec 0 \\
       \vec\psi_2^{b, E_+}(r_1) \otimes \vec 0 \\
       \vec 0 \otimes \overline{\vec\psi_1^{b, \bar E_-}}(r_1)  \\
       \vec 0 \otimes \overline{\vec\psi_3^{b, \bar E_-}}(r_1)
  \end{array}\right) = (I + \Gamma^{con, R})A^R
  \left( \begin{array}{c}
       \vec\psi_1^{b, E_+}(r_1) \otimes \vec 0 \\
       \vec\psi_3^{b, E_+}(r_1) \otimes \vec 0 \\
       \vec 0 \otimes \overline{\vec\psi_1^{b, \bar E_-}}(r_1)  \\
       \vec 0 \otimes \overline{\vec\psi_3^{b, \bar E_-}}(r_1)
  \end{array}\right)  \\
 \small  \left( \begin{array}{c}
       \vec \Phi^{con}_1 (r_2)\\
       \vec \Phi^{con}_2 (r_2)\\
       \vec \Phi^{con}_3 (r_2)\\
       \vec \Phi^{con}_4 (r_2)
  \end{array}\right) =
  (I + \Gamma^{con, L})
  \left( \begin{array}{c}
       \vec\psi_4^{b, E_+}(r_2) \otimes \vec 0 \\
       \vec\psi_2^{b, E_+}(r_2) \otimes \vec 0 \\
       \vec 0 \otimes \overline{\vec\psi_1^{b, \bar E_-}}(r_2)  \\
       \vec 0 \otimes \overline{\vec\psi_3^{b, \bar E_-}}(r_2)
  \end{array}\right) = (I + \Gamma^{con, L})A^L
  \left( \begin{array}{c}
       \vec\psi_4^{b, E_+}(r_2) \otimes \vec 0 \\
       \vec\psi_2^{b, E_+}(r_2) \otimes \vec 0 \\
       \vec 0 \otimes \overline{\vec\psi_4^{b, \bar E_-}}(r_2)  \\
       \vec 0 \otimes \overline{\vec\psi_2^{b, \bar E_-}}(r_2)
  \end{array}\right) 
  \end{split}\label{eqbdrycon}
\ee
where $A^R$, $A^L$ are given by the connection formula \eqref{eqconnect}
\be
 A^R = \left( \begin{array}{cccc}
       \frac{e^{\frac{\pi i}{6}}}{2}  & \frac{e^{-\frac{\pi i}{6}}}{2} & 0 & 0 \\
       e^{-\frac{\pi i}{3}} & e^{\frac{\pi i}{3}} & 0 & 0 \\
       0 & 0 & 1  & 0  \\
       0 & 0 & 0 & 1
  \end{array}\right),\quad
 A^L = \left( \begin{array}{cccc}
       1  & 0 & 0 & 0 \\
       0 & 1 & 0 & 0 \\
       0 & 0 & e^{\frac{\pi i}{6}} & \frac{e^{-\frac{\pi i}{3}}}{2}  \\
       0 & 0 & e^{-\frac{\pi i}{6}} & \frac{e^{\frac{\pi i}{3}}}{2}
  \end{array}\right) \label{eqARAL}
\ee
and $\Gamma^{con, R}$, $\Gamma^{con, L}$ are given by
\bea
  \Gamma^{con, R} = \left( \begin{array}{cccc}
       \gamma^{con}_{11} & 0 & \gamma^{con}_{13} & 0 \\
       \gamma^{con}_{21} & 0 & \gamma^{con}_{23} & 0 \\
       \gamma^{con}_{31} & 0 & \gamma^{con}_{33} & 0  \\
       \gamma^{con}_{41} & 0 & \gamma^{con}_{43} & 0
  \end{array}\right), \quad  \Gamma^{con, L} = \left( \begin{array}{cccc}
       0 & \gamma^{con}_{12} & 0 & \gamma^{con}_{14}  \\
       0 & \gamma^{con}_{22} & 0 & \gamma^{con}_{24}  \\
       0 & \gamma^{con}_{32} & 0 & \gamma^{con}_{34}   \\
       0 & \gamma^{con}_{42} & 0 & \gamma^{con}_{44} 
  \end{array}\right) \label{eqGammamatrixcon} \\
  \begin{split}
  \gamma^{con}_{j1} = - \int_{r_2}^{r_1} \frac{\psi_2^{b, E_+}}{W_{42;E_+}} F^{con, +}_{j} ds,&\quad \gamma^{con}_{j2} = -\int_{r_2}^{r_1} \frac{\psi_4^{b, E_+}}{W_{42;E_+}} F^{con, +}_{j} ds, \\
  \overline{\gamma^{con}_{j3}} =  \int_{r_2}^{r_1} \frac{\psi_3^{b, \bar E_-}}{W_{31;\bar E_-}} F^{con, -}_{j} ds,&\quad 
  \overline{\gamma^{con}_{j4}} = \int_{r_2}^{r_1} \frac{\psi_1^{b, \bar E_-}}{W_{31;\bar E_-}} F^{con, -}_{j} ds,
  \end{split} \label{eqdefgammacon}
\eea
for $j = 1, 2, 3, 4$ with $F^{con, \pm}_{j}$ from \eqref{eqdefFjBoxpm}.
% \be
% \begin{split}
%  F^{con, +}_{j} &= \left( h_{b, E_+} + \frac{\nu^2 - \frac 14}{r^2} - W_{1, b} \right) \Phi^{con, 1}_j - e^{i\frac{br^2}{2}} W_{2,b} \Phi^{con, 2}_j \\
%  F^{con, -}_{j} &=\left( h_{b, \bar E_-} + \frac{\nu^2 - \frac 14}{r^2} - W_{1, b} \right) \overline{\Phi^{con, 2}_j} - e^{i\frac{br^2}{2}} W_{2,b} \overline{\Phi^{con, 1}_j}.
% \end{split} \label{eqFconpm}
% \ee
The asymptotics of $\psi_j^{b, E}$ from Proposition \ref{propWKB} implies the estimates of coefficients 
\be 
 \small \Gamma^{con, R} =  \left( \begin{array}{cccc}
       O(b) & 0 & O(be^{-\frac{p-1}{64b}}) & 0 \\
       O(b) & 0 & O(be^{-\frac{p-1}{64b}}) & 0 \\
       O(be^{-\frac{p-1}{64b}}) & 0 & O(b) & 0  \\
       O(be^{-\frac{p-1}{64b}}) & 0 & O(be^{2\eta_{b, \bar E_-}(r_1)}) & 0
  \end{array}\right), \,\,  \Gamma^{con, L} = \left( \begin{array}{cccc}
       0 & O(be^{-2\eta_{b, E_+}(r_2)}) & 0 & O(be^{-\frac{p-1}{64b}})  \\
       0 & O(b) & 0 & O(be^{-\frac{p-1}{64b}})  \\
       0 & O(be^{-\frac{p-1}{64b}}) & 0 & O(b)   \\
       0 & O(be^{-\frac{p-1}{64b}}) & 0 & O(b) 
  \end{array}\right)\label{eqestgammacon}
\ee
We stress that $|e^{2\eta_{b, \bar E_-}(r_1)}| \gtrsim 1$ and $|e^{-2\eta_{b, E_+}(r_2)}| \gtrsim 1$ with $\Im \l \le bI_0$. 

\mbox{}

\textit{Step 2.5. On middle region $I_{mid}$: Construction.} 

On $I_{mid}$, we consider the system \eqref{eqI0} with $j \in\{ 1, 2, 3, 4\}$. The scalar operators are 
\be  \calA_j^{con} = \tilde T_{x_*, r_2; b, E_+}^{mid, \sigma_+(j)},\quad \calB_j^{con} =  \tilde T_{x_*, r_2; b, \bar E_-}^{mid, \sigma_-(j)}, \label{eqdefcalTmid} \ee
where the symbols are 
\be  \sigma_+(j) = \left| \begin{array}{ll}
    G & j =1, 2, 4,  \\
    D & j = 3,
\end{array}\right. \quad \sigma_-(j) = \left| \begin{array}{ll}
    G & j =  2, 3, 4,  \\
    D & j = 1,
\end{array}\right. 
\label{eqcalTmidsigma}
\ee
and the source terms are
\be
\begin{split}
  \left(\begin{array}{c}
       S^{mid}_1  \\
       R^{mid}_1
  \end{array} \right) = 
   \left(\begin{array}{c}
       \psi_4^{b, E_+}  \\
       0
  \end{array} \right),&\quad
  \left(\begin{array}{c}
       S^{mid}_2  \\
       R^{mid}_2
  \end{array} \right) = 
   \left(\begin{array}{c}
       \psi_2^{b, E_+}  \\
       0
  \end{array} \right),\\
  \left(\begin{array}{c}
       S^{mid}_3  \\
       R^{mid}_3
  \end{array} \right) = 
   \left(\begin{array}{c}
          0  \\
       \psi_4^{b, \bar E_-}
  \end{array} \right),&\quad
  \left(\begin{array}{c}
       S^{mid}_4  \\
       R^{mid}_4
  \end{array} \right) = 
   \left(\begin{array}{c}
       0 \\
       \psi_2^{b,\bar E_-}
  \end{array} \right).
  \end{split}\label{eqImidsource}
\ee

Recall the pointwise bound of potential from \eqref{eqpotentialest}, and that $e^{-\frac{p-1}{2}\min\{r, 2b^{-1} \}} \gtrsim e^{-\frac{p-1}{4}r}$ on $I_{mid}$ since $r_2 \le 4b^{-1}$. 
% Now we can estimate the linear operators in \eqref{eqImid}. The boundedness of $\tilde T^{mid}_{x_*, r_2; b, E_+}$ from Lemma \ref{leminvtildeHmid} yields
% \bee
% \left\| \tilde T_{x_*, r_2; b, E}^{mid} \circ \left( h_{b, E} + \frac{\nu^2 - \frac 14}{r^2} - W_{1, b}\right) \right\|_{\calL\left(  C^0_{\omega_{b, E}^-}(I_{mid})\right)} \lesssim x_*^{-1},\quad E = E_+ \,\,{\rm or}\,\, \bar E_-.
% \eee
For the off-diagonal part, we again absorb the quotient of $\frac{\omega_{b, E_1}^\pm}{\omega_{b, E_2}^\pm}$ in \eqref{eqcomparisonomega3} by $W_{2, b}$ for $(E_1, E_2) = (E_+, \bar E_-) \,\,{\rm or}\,\, (\bar E_-, E_+)$ and $k \ge 0$,
% \bee
%   \left\| \tilde T_{x_*, r_2; b, E_1}^{mid} \circ \left( e^{i\frac{br^2}{2}} W_{2, b} \right)  \right\|_{\calL\left(  C^0_{\omega_{b, E_2}^-}(I_{mid}) \to C^0_{\omega_{b, E_1}^-}(I_{mid}) \right)} \lesssim x_*^{-1} e^{-\frac{\pi}{2b}(E_1 - E_2)}.
% \eee
\be
  \left\| e^{i\frac{br^2}{2}} W_{2, b} f \right\|_{C^0_{\omega^\pm_{b, E_1} e^{-\frac{p-1}{8} r } r^{k-2}} (I_{mid})} \lesssim_k e^{\pm\frac{\pi}{2b}(E_1 - E_2)}\|  f \|_{C^0_{\omega^\pm_{b, E_2} r^k} (I_{mid})}. 
  \label{eqpotentialmidest1}
\ee
Also recall from \eqref{eqReetaderiv} that 
\[ \pa_r^2 (\Re \eta_{b, E}) = \frac{b^2r}{4} \Im \left( \frac{b^2 r^2}{4} - E\right)^{-\frac 12} < 0,\quad 0 \le r < r^*_{b, E}, \]
so 
\[ \Re \eta_{b, E}(r) \le \Re \eta_{b, E}(0) + r\pa_r \Re \eta_{b, E}(0) =-\frac{\pi \Re E}{2b} + \Re (\sqrt{E})r,\quad 0 \le r \le r^*_{b, E}.  \]
This implies $e^{-\frac{p-1}{8}r} \le e^{-\frac{\pi(p-1)\Re E}{16\Re \sqrt E b}}e^{-\frac{p-1}{8\Re \sqrt E}\Re \eta_{b, E}}$
% $\le e^{-\frac{p-1}{8b}} e^{-\frac{p-1}{16}\Re\eta_{b, E}}$
and thus
% \be
%   \left\| e^{i\frac{br^2}{2}} W_{2, b} f \right\|_{C^0_{\omega^\pm_{b, E_1} e^{-\frac{p-1}{8\Re \sqrt E_1} \eta_{b, E_1}} r^{-2}} (I_{mid})} \lesssim e^{-\frac{\pi(p-1)\Re E_1}{8\Re \sqrt E_1 b}} e^{\pm\frac{\pi}{2b}(E_1 - E_2)}\|  f \|_{C^0_{\omega^\pm_{b, E_2}} (I_{mid})}.
%   \label{eqpotentialmidest2}
% \ee
\be
  \left\| e^{i\frac{br^2}{2}} W_{2, b} f \right\|_{C^0_{\omega^\pm_{b, E_1} e^{-\frac{p-1}{8\Re \sqrt{E_1}}\Re \eta_{b, E_1}} r^{-2}} (I_{mid})} \lesssim e^{-\frac{\pi(p-1)\Re E_1}{16\Re \sqrt{E_1} b}} e^{\pm\frac{\pi}{2b}(E_1 - E_2)}\|  f \|_{C^0_{\omega^\pm_{b, E_2}} (I_{mid})}.
  \label{eqpotentialmidest2}
\ee

Thus with the boundedness of $\tilde T^{mid,\sigma}_{x_*, r_2;b, E}$ from Lemma \ref{leminvtildeHmid}, we have
\be
\begin{split}
  &\| \calT^{mid}_1 \|_{\calL \left( C^0_{\omega^-_{b, E_+} r^k}(I_{mid}) \times  e^{-\frac{\pi \Re \l}{b}}C^0_{\omega^-_{b, \bar E_-}e^{-\frac{p-1}{8}r} r^k}(I_{mid}) \right)}  \lesssim_k  x_*^{-1},\quad k \ge 0  \\
  &\| \calT^{mid}_2 \|_{\calL \left( C^0_{\omega^+_{b, E_+} } (I_{mid}) \times  e^{\frac{\pi(p-1)\Re E_1}{16\Re \sqrt{E_1} b} + \frac{\pi \Re \l}{b}} C^0_{\omega^+_{b, \bar E_-}e^{-\frac{p-1}{8\Re \sqrt {\bar E_-}}\Re \eta_{b, \bar E_-}} }(I_{mid}) \right)}  \lesssim x_*^{-1}. 
  \end{split}\label{eqcalTmidest}
\ee
and similarly for $j = 3, 4$. Here we used $E_+ - \bar E_- = 2\Re \l$. Thereafter,  
the existence and uniqueness of solutions $(\phi^{mid}_j, \varphi^{mid}_j)$ for $j = 1 , 2 ,3, 4$ follows by contraction mapping principle as long as $x_* \gg 1$ depending on $\nu_0$. The solutions satisfy 
\bea
    \| \phi^{mid}_1\|_{C^0_{\omega_{b, E_+}^-}(I_{mid})} + e^{-\frac{\pi \Re \l}{b} } \| \varphi^{mid}_1\|_{C^0_{\omega_{b, \bar E_-}^- e^{-\frac{p-1}{8}r} }(I_{mid})} \lesssim 1, \label{eqestPhimidb1} \\
    \| \phi^{mid}_2\|_{C^0_{\omega_{b, E_+}^+}(I_{mid})} + 
    e^{\frac{\pi(p-1)\Re {\bar E_-}}{16\Re \sqrt{\bar E_-} b} + \frac{\pi \Re \l}{b}}\| \varphi^{mid}_2\|_{C^0_{\omega_{b, \bar E_-}^+ e^{-\frac{p-1}{8\Re \sqrt {\bar E_-}}\Re \eta_{b, \bar E_-}} }(I_{mid})} \lesssim 1,\label{eqestPhimidb2}\\
    e^{\frac{\pi \Re \l}{b} } \| \phi^{mid}_3\|_{C^0_{\omega_{b, E_+}^- e^{-\frac{p-1}{8}r} }(I_{mid})} +  \| \varphi^{mid}_3\|_{C^0_{\omega_{b, \bar E_-}^-}(I_{mid})} \lesssim 1, \label{eqestPhimidb3}\\
    e^{\frac{\pi(p-1)\Re E_+}{16\Re \sqrt{E_+} b} -\frac{\pi \Re \l}{b} } \| \phi^{mid}_4\|_{C^0_{\omega_{b, E_+}^+ e^{-\frac{p-1}{8\Re \sqrt {E_+}}\Re \eta_{b, E_+}} }(I_{mid})} + \| \varphi^{mid}_4\|_{C^0_{\omega_{b, \bar E_-}^+}(I_{mid})} \lesssim 1. \label{eqestPhimidb4}
\eea

The boundary values at $x_*$ and $r_2$ can be evaluated as
\be\small
\begin{split}
 \left( \begin{array}{c}
       \vec \Phi^{mid}_1 (r_2)\\
       \vec \Phi^{mid}_2 (r_2)\\
       \vec \Phi^{mid}_3 (r_2)\\
       \vec \Phi^{mid}_4 (r_2)
  \end{array}\right) = (I + \Gamma^{mid, R})
  \left( \begin{array}{c}
       \vec\psi_4^{b, E_+}(r_2) \otimes \vec 0 \\
       \vec\psi_2^{b, E_+}(r_2) \otimes \vec 0 \\
       \vec 0 \otimes \overline{\vec\psi_4^{b, \bar E_-}}(r_2)  \\
       \vec 0 \otimes \overline{\vec\psi_2^{b, \bar E_-}}(r_2)
  \end{array}\right)
  % \\
  %  \left( \begin{array}{c}
  %      \vec \Phi^{mid}_1 (x_*)\\
  %      \vec \Phi^{mid}_2 (x_*)\\
  %      \vec \Phi^{mid}_3 (x_*)\\
  %      \vec \Phi^{mid}_4 (x_*)
  % \end{array}\right) = (I + \Gamma^{mid, L})
  % \left( \begin{array}{c}
  %      \vec\psi_4^{b, E_+}(x_*) \otimes \vec 0 \\
  %      \vec\psi_2^{b, E_+}(x_*) \otimes \vec 0 \\
  %      \vec 0 \otimes \overline{\vec\psi_4^{b, \bar E_-}}(x_*)  \\
  %      \vec 0 \otimes \overline{\vec\psi_2^{b, \bar E_-}}(x_*)
  % \end{array}\right) 
  \end{split} \label{eqbdrymid}
  \ee

where 
\be
  \small \Gamma^{mid, R} = \left( \begin{array}{cccc}
       \gamma^{mid}_{11} & 0 & 0 & 0 \\
       \gamma^{mid}_{21} & 0 & \gamma^{mid}_{23} & 0 \\
       0 & 0 & \gamma^{mid}_{33} & 0  \\
       \gamma^{mid}_{41} & 0 & \gamma^{mid}_{43} & 0
  \end{array}\right),
  % \quad \Gamma^{mid, L} = \left( \begin{array}{cccc}
  %           0 & \gamma^{mid}_{12} & -\gamma^{mid}_{13} & \gamma^{mid}_{14} \\
  %      0 & \gamma^{mid}_{22} & 0 & \gamma^{mid}_{24} \\
  %      -\gamma^{mid}_{31} & \gamma^{mid}_{32} & 0 & \gamma^{mid}_{34}  \\
  %      0 & \gamma^{mid}_{42} & 0 & \gamma^{mid}_{44}
  % \end{array}\right) 
  \label{eqGammamatrixmid} \ee
  with 
  \be
  \begin{split}
      \gamma^{mid}_{j1} = - \int_{x_*}^{r_2} \frac{\psi_2^{b, E_+}}{W_{42;E_+}} F^{mid, +}_{j} ds,\quad  \gamma^{mid}_{j2} = - \int_{x_*}^{r_2} \frac{\psi_4^{b, E_+}}{W_{42;E_+}} F^{mid, +}_{j} ds, 
      % \\
       % \overline{\gamma^{mid}_{j3}} = - \int_{x_*}^{r_2} \frac{\psi_2^{b, \bar E_-}}{W_{42;\bar E_-}} F^{mid, -}_{j} ds,\quad   \overline{\gamma^{mid}_{j4}} = - \int_{x_*}^{r_2} \frac{\psi_4^{b, \bar E_-}}{W_{42;\bar E_-}} F^{mid, -}_{j} ds,
  \end{split}
  % \left| \begin{array}{ll}\gamma^{mid}_{j1} = - \int_{x_*}^{r_2} \frac{\psi_2^{b, E_+}}{W_{42;E_+}} F^{mid, +}_{j} ds, & j = 1, 2, 4;\\
  % \overline{\gamma^{mid}_{j3}} =  -\int_{x_*}^{r_2} \frac{\psi_2^{b, \bar E_-}}{W_{42;\bar E_-}} F^{mid, -}_{j} ds, & j = 2, 3, 4. 
  % \end{array}\right.
  \label{eqdefgammamid}
\ee
and $F^{mid, \pm}_j$ as \eqref{eqdefFjBoxpm}.
% \be
% \begin{split}
%  F^{mid, +}_{j} =& \left( h_{b, E_+} + \frac{\nu^2 - \frac 14}{r^2} - W_{1, b} \right) \Phi^{mid, 1}_j - e^{i\frac{br^2}{2}} W_{2,b} \Phi^{mid, 2}_j \\
%  F^{mid, -}_{j} =& \left( h_{b, \bar E_-} + \frac{\nu^2 - \frac 14}{r^2} - W_{1, b} \right) \overline{\Phi^{mid, 2}_j} - e^{i\frac{br^2}{2}} W_{2,b} \overline{\Phi^{mid, 1}_j}.
%  \end{split} \label{eqFmidpm}
% \ee
Lemma \ref{leminvtildeHmid} also implies the estimates of the coefficients 
\bea
   \small \Gamma^{mid, R} = \left( \begin{array}{cccc}
       O(x_*^{-1}) & 0 & 0 & 0 \\
       O(be^{2\eta_{b, E_+}(r_2)}) & 0 & O(be^{2\eta_{b, E_+}(r_2)} e^{-\frac{p-1}{64 b}}) & 0 \\
       0 & 0 & O(x_*^{-1}) & 0  \\
       O(b e^{-\frac{p-1}{64 b}}) & 0 & O(b) & 0
  \end{array}\right). \label{eqestgammamid}
  % \\
%   \small \Gamma^{mid, L} = x_*^{-1}
%    \left( \begin{array}{cccc}
%        0 & O(e^{-2\eta_{b, E_+}(x_*) }) & O(M \tilde \epsilon) & O(M e^{-2\eta_{b, \bar E_-}(x_* )} \tilde \epsilon
%        ) \\
%        0 & O(1) & 0 & O(M^{-1} \tilde \epsilon)  \\
%        O(M^{-1}\tilde \epsilon) & O(M^{-1} e^{-2\eta_{b, E_+}(x_*) }\tilde \epsilon) & 0 & O(e^{-2\eta_{b, \bar E_-}(x_* )})  \\
%        0 & O(M\tilde \epsilon) & 0 & O(1)
%   \end{array}\right)
% \label{eqestgammamidL}
  \eea
% \be
%   \gamma^{mid}_{jk} = \left|  \begin{array}{ll}
%       O(x_*^{-1}) & (j, k) = (1,1), (3,3), \\
%       O(be^{2\eta_{b, E_+}(r_2)}) & (j,k) = (2,1), (2, 3), \\
%       O(b), & (j,k) = (4,1), (4, 3).
%   \end{array}\right. \label{eqestgammamid}
% \ee
% where we denoted 
% \be M = e^{\frac{\pi \Re \l}{b}},\quad \tilde \epsilon = e^{-\frac{p-1}{16}x_*}
% \ee
We remark that here $\gamma^{mid}_{23}$ and $\gamma^{mid}_{41}$ exploited the smallness of $e^{-\frac{\pi(p-1)\Re E}{16\Re \sqrt E b}\pm \frac{\pi \Re \l}{b} } \ll e^{-\frac{p-1}{16 b}} \ll e^{-\frac{p-1}{64 b}} |e^{2\eta_{b, E_+}(r_2)}|$ from \eqref{eqestPhimidb2} and \eqref{eqestPhimidb4} with $E \in \{ E_+, E_-\}$, and that $|e^{2\eta_{b, E_+}(r_2)}| \le 1$.

Besides, \eqref{eqestPhimidb1}-\eqref{eqestPhimidb4} and potential estimates \eqref{eqpotentialmidest1}, \eqref{eqpotentialmidest2} imply the bounds of $F_j^{mid,\pm}$ 

\be
\begin{split}
 \small |F^{mid, +}_j| &\lesssim \left| \begin{array}{ll}
     \omega_{b, E_+}^- r^{-2} & j = 1  \\
     \omega_{b, E_+}^+ r^{-2} & j = 2 \\
      \omega_{b, E_+}^- e^{-\frac{p-1}{8} r} r^{-2}  e^{-\frac{\pi \Re \l}{b}} & j = 3 \\
       \omega_{b, E_+}^+ e^{-\frac{p-1}{8\Re \sqrt{E_+}} \Re\left( \eta_{b, E_+} + \frac{\pi  E_+}{2b}\right) + \frac{\pi \Re \l}{b}} r^{-2} & j = 4
 \end{array}\right. \quad r \in [x_*, r_2],\\
 |F^{mid, -}_j| &\lesssim \left| \begin{array}{ll}
      \omega_{b,  \bar E_-}^- e^{-\frac{p-1}{8} r} r^{-2}  e^{\frac{\pi \Re \l}{b}} & j = 1 \\
       \omega_{b,  \bar E_-}^+ e^{-\frac{p-1}{8\Re \sqrt{\bar E_-}} \Re\left( \eta_{b,  \bar E_-} + \frac{\pi  \bar E_-}{2b}\right) - \frac{\pi \Re \l}{b}} r^{-2} & j = 2 \\
    \omega_{b, \bar E_-}^- r^{-2} & j = 3  \\
     \omega_{b,  \bar E_-}^+ r^{-2} & j = 4 \\
 \end{array}\right. \quad r\in [x_*, r_2].
 \end{split} \label{eqbddFmidpm}
\ee

\mbox{}

\textit{Step 2.6. On middle region $I_{mid}$: analytic solution and convergence as $b \to 0$.}

By choosing $b_1 \ll 1$, we can assume $x_* \le b_1^{-\frac 12} \le b^{-\frac 12}$. 

Notice that the fundamental solutions $\Phi^{mid}_j$ constructed are not analytic w.r.t. $\l$ due to its dependence on $r^*_{b, E}$. We construct another family of solutions $\check \Phi^{mid}_j$ by 
\be
  \left( \begin{array}{c} 
  \check \phi^{mid}_j \\ \check \varphi^{mid}_j
  \end{array} \right) = 
  \left( \begin{array}{c} 
  S^{mid}_j \\ R^{mid}_j
  \end{array} \right) 
  + 
  \check{\mathbf{\calT}}^{mid}_j \left( \begin{array}{c} 
  \check \phi^{mid}_j \\ \check \varphi^{mid}_j
  \end{array} \right) \label{eqImidcheck}
\ee
with linear operators
\be
   \check{\mathbf{\calT}}^{mid}_j \left( \begin{array}{c} 
  \check \phi^{mid}_j \\ \check \varphi^{mid}_j
  \end{array} \right) = 
  \left( \begin{array}{c} 
  \tilde T_{x_*, b^{-\frac 12}; b, E_+}^{mid, \sigma_+(j)} \left[ \left( h_{b, E_+} + \frac{\nu^2 - \frac 14}{r^2} - W_{1, b}\right) \check\phi^{mid}_j - e^{i\frac{br^2}2} W_{2, b}
  \overline{\check{\varphi}^{mid}_j}
  \right] 
  \\ 
  \tilde T_{x_*, b^{-\frac 12}; b, \bar E_-}^{mid, \sigma_-(j)} \left[ \left( h_{b, \bar E_-} + \frac{\nu^2 - \frac 14}{r^2} - W_{1, b}\right) \check\varphi^{mid}_j - e^{i\frac{br^2}2} W_{2, b} \overline{\check \phi^{mid}_j} \right]
  \end{array} \right) \label{eqdefcalTmidcheck}
\ee
with symbols $\sigma_\pm(j)$ as \eqref{eqcalTmidsigma} and source terms as \eqref{eqImidsource}. Substituting $I_{mid}$ by $[x_*, b^{-\frac 12}]$ in the argument above, we obtain same bounds \eqref{eqestPhimidb1}-\eqref{eqestPhimidb4} for $(\check \phi^{mid}_j, \check \varphi^{mid}_j)$ on $[x_*, b^{-\frac 12}]$. Note that from \eqref{eqomegapm}, \eqref{eqnormalcoeff} and \eqref{eqetaconv}, for $r \le b^{-\frac 12}$,
\[ \omega^\pm_{b, E} \sim b^\frac 16 e^{\pm \Re \eta_{b, E}} \sim  |\kappa_{b, E}^\pm|^{-1} e^{\pm \sqrt{E} r}, \qquad \left| \kappa_{b, 1+\l}^\pm(\kappa_{b, 1-\l}^\pm)^{-1} \right| \sim e^{\pm \frac{\pi \Re \l}{b}}, \]
and $|\kappa_{b, 1+\l}^\pm| \sim |\kappa_{b, E_+}^\pm|$,  $|\kappa_{b, 1-\l}^\pm| = | \overline{\kappa_{b, 1-\bar \l}^\pm} | \sim |\kappa_{b, \bar E_-}^\pm|$, 
so the bounds for $\check \varphi^{mid}_j$ can be rewritten as 
\bea
  \| \kappa_{b, 1+\l}^- \check \Phi^{mid, 1}_{1;b,\l,\nu} \|_{C^0_{e^{-\sqrt{1+\l}r } } ([x_*, b^{-\frac 12}])} + \| \kappa_{b, 1+\l}^- \check \Phi^{mid, 2}_{1;b,\l,\nu} \|_{C^0_{e^{-(\sqrt{1-\l} + \frac{p-1}{8})r} } ([x_*, b^{-\frac 12}])} \lesssim 1, \\
  \| \kappa_{b, 1+\l}^+ \check \Phi^{mid, 1}_{2;b,\l,\nu} \|_{C^0_{e^{\sqrt{1+\l}r } } ([x_*, b^{-\frac 12}])} + \| \kappa_{b, 1+\l}^+ \check \Phi^{mid, 2}_{2;b,\l,\nu} \|_{C^0_{e^{(\sqrt{1-\l} - \frac{p-1}{8})r} } ([x_*, b^{-\frac 12}])} \lesssim 1, \\
  \| \kappa_{b, 1-\l}^- \check \Phi^{mid, 1}_{3;b,\l,\nu} \|_{C^0_{e^{-(\sqrt{1+\l} + \frac{p-1}{8})r } } ([x_*, b^{-\frac 12}])} + \| \kappa_{b, 1-\l}^- \check \Phi^{mid, 2}_{3;b,\l,\nu} \|_{C^0_{e^{-\sqrt{1-\l} r} } ([x_*, b^{-\frac 12}])} \lesssim 1, \\
  \| \kappa_{b, 1-\l}^+ \check \Phi^{mid, 1}_{4;b,\l,\nu} \|_{C^0_{e^{(\sqrt{1+\l} - \frac{p-1}{8})r } } ([x_*, b^{-\frac 12}])} + \| \kappa_{b, 1-\l}^+ \check \Phi^{mid, 2}_{4;b,\l,\nu} \|_{C^0_{e^{\sqrt{1-\l}r} } ([x_*, b^{-\frac 12}])} \lesssim 1,
\eea
We also define $\check F^{mid, \pm}_j$ as in \eqref{eqdefFjBoxpm}.

\mbox{}

Now we claim the following properties of $\check \Phi^{mid}_{j;b,\l,\nu}$:

\noindent{\textit{(a) Connection with $\Phi^{mid}_{j;b,\l,\nu}$}}: for $r \in [x_*, b^{-\frac 12}]$, we have
\be
   \left( \begin{array}{c}
        \Phi^{mid}_1 \\
        \Phi^{mid}_2 \\
        \Phi^{mid}_3 \\
        \Phi^{mid}_4 
  \end{array}\right)  = (I + \Omega) \left( \begin{array}{c}
        \check \Phi^{mid}_1 \\
        \check \Phi^{mid}_2 \\
        \check \Phi^{mid}_3 \\
        \check \Phi^{mid}_4 
  \end{array}\right),\quad \Omega = \left(\begin{array}{cccc}
      0 & \omega_{12} & \omega_{13} & \omega_{14}  \\
      0 & \omega_{22} & 0 & \omega_{24} \\
      \omega_{31} & \omega_{32} & 0 & \omega_{34} \\
      0 & \omega_{42} & 0 & \omega_{44}
  \end{array} \right) 
  \label{eqconnectcheck}
\ee
with 
\be
\begin{split}
  \omega_{31} = \int_{b^{-\frac 12}}^{r_2} \frac{\psi_2^{b, E_+}}{W_{42;E_+}}  F^{mid,+}_3 ds, &\quad
 \overline{\omega_{13}} = \int_{b^{-\frac 12}}^{r_2} \frac{\psi_2^{b,\bar E_-}}{W_{42;\bar E_-}}  F^{mid,-}_1 ds,  \\
 \omega_{j2} = -\int_{b^{-\frac 12}}^{r_2} \frac{\psi_4^{b, E_+}}{W_{42;E_+}}  F^{mid,+}_j ds,&\quad
 \overline{\omega_{j4}} = -\int_{b^{-\frac 12}}^{r_2} \frac{\psi_4^{b,\bar E_-}}{W_{42;\bar E_-}}  F^{mid,-}_j ds, \quad j = 1, 2, 3, 4
\end{split} \label{eqdefomegajk}
\ee
satisfying
% {\color{red} Problematic because of the wrong $F^{mid, +}_3$ and $F^{mid, -}_1$ estimates - neglecting $e^{\pm \frac{\pi \Re \l}{b}}$. And these coefficients will be too large for $O(b)$ to control - the current construction not accurate enough.}
\be
  \small \Omega = 
  % \left(\begin{array}{cccc}
  %     0 & O(b^\frac 12 e^{-2 \eta_{b, E_+}(b^{-\frac 12})}) & O(b) & O(b^\frac 32 e^{-2 \eta_{b, \bar E_-}(b^{-\frac 12})})  \\
  %     0 & O(b^\frac 12) & 0 & O(b) \\
  %     O(b) & O(b^\frac 32 e^{-2 \eta_{b, E_+}(b^{-\frac 12})}) & 0 & O(b^\frac 12 e^{-2 \eta_{b, \bar E_-}(b^{-\frac 12})}) \\
  %     0 & O(b) & 0 & O(b^\frac 12)
  % \end{array} \right)  
  b^\frac 12 \left( \begin{array}{cccc}
       0 & O(e^{-2\eta_{b, E_+}(b^{-\frac 12})}) & O(M \tilde \epsilon) & O(M e^{-2\eta_{b, \bar E_-}(b^{-\frac 12})} \tilde \epsilon
       ) \\
       0 & O(1) & 0 & O(M^{-1} \tilde \epsilon)  \\
       O(M^{-1}\tilde \epsilon) & O(M^{-1} e^{-2\eta_{b, E_+}(b^{-\frac 12}) }\tilde \epsilon) & 0 & O(e^{-2\eta_{b, \bar E_-}(b^{-\frac 12})})  \\
       0 & O(M\tilde \epsilon) & 0 & O(1)
  \end{array}\right)
  \label{eqestomegajk}
\ee
% \be
%   \gamma^{mid}_{jk} = \left|  \begin{array}{ll}
%       O(x_*^{-1}) & (j, k) = (1,1), (3,3), \\
%       O(be^{2\eta_{b, E_+}(r_2)}) & (j,k) = (2,1), (2, 3), \\
%       O(b), & (j,k) = (4,1), (4, 3).
%   \end{array}\right. \label{eqestgammamid}
% \ee
where we denoted 
\be M = e^{\frac{\pi \Re \l}{b}},\quad \tilde \epsilon = e^{-\frac{p-1}{16}b^{-\frac 12}}. \label{eqdefMtildeep}
\ee

% \mbox{}\\

% We claim the following asymptotic behavior as $b \to 0$ for $j = 1, 3$:

\noindent{\textit{(b) Asymptotics as $b \to 0$}}: Recall the normalizing coefficient $\kappa^-_{b, E}$ from \eqref{eqnormalcoeff}, and fundamental solutions $\Phi_{j;0,\l,\nu}$ from \eqref{eqb0fund},
    then
    \begin{align}
     \|\kappa_{b, 1+ \l}^- \check \Phi^{mid, }_{1; b, \l, \nu} - \Phi_{1; 0, \l, \nu} \|_{\left(C^1_{e^{-\sqrt{1+\l} r}}([x_*, b^{-\frac 12}]) \times C^1_{e^{-(\sqrt{1-\l}+\frac{p-1}{8}) r}}([x_*, b^{-\frac 12}]) \right) } &\lesssim b^{\frac 16} \label{eqestconv1} \\
     \|\kappa_{b, 1- \l}^-\check \Phi^{mid, }_{3; b, \l, \nu} - \Phi_{2; 0, \l, \nu} \|_{\left(C^1_{e^{-(\sqrt{1+\l} + \frac{p-1}{8}) r}}([x_*, b^{-\frac 12}]) \times C^1_{e^{-\sqrt{1-\l} r}}([x_*, b^{-\frac 12}]) \right) } &\lesssim b^{\frac 16}.
     % \|\kappa_{b, 1\pm \l}^{-1} \Phi^{mid, 2}_{j_\pm; b, \l, \nu} - \Phi^2_{j_\pm; 0, \l, \nu} \|_{\left(C^1_{e^{-\sqrt{1-\l} r}}([x_*, b^{-\frac 12}])\right) } &\lesssim b^{\frac 16}. 
     \label{eqestconv2}
    \end{align}

%  {\color{red} Useless? Since we will divide this coefficient anyway.
 
% \noindent{\textit{(c) Convergence of matching coefficients at $r_2$}}: Define  
%     \bee
%       \gamma_{11; 0, \l, \nu} = -\int_{x_*}^\infty \frac{e^{-\sqrt{1+\l}\cdot s}}{2\sqrt{1+\l}} \left[ \left(\frac{\nu^2 - \frac 14}{s^2} - W_1 \right)\Phi^1_{1;0,\l,\nu} - W_2 \Phi^2_{1;0,\l,\nu}  \right] ds \\
%       \gamma_{33; 0, \l, \nu} = -\int_{x_*}^\infty \frac{e^{-\sqrt{1-\l}\cdot s}}{2\sqrt{1-\l}} \left[ \left(\frac{\nu^2 - \frac 14}{s^2} - W_1 \right)\Phi^2_{3;0,\l,\nu} - W_2 \Phi^1_{3;0,\l,\nu}  \right] ds,
%     \eee
%     then $|\gamma_{jj; 0, \l, \nu}| \lesssim x_*^{-1}$ and 
%     \be
%       |\gamma^{mid}_{jj; b, \l, \nu} - \gamma_{jj; 0, \l, \nu}| \lesssim b^\frac 16 x_*^{-1}.
%       \label{eqalphadiff}
%     \ee}

\mbox{}

During the proof, we omit the subscript $b, \l, \nu$ for simplicity unless necessary. 

For (a), let $\omega_{jk}$ be defined as in \eqref{eqdefomegajk}. We subtract \eqref{eqImidcheck} from \eqref{eqI0} (for $\Box= mid$) with $r \in [x_*, b^{-\frac 12}]$ to obtain
\bee
  \left(I -\check \calT^{mid}_j \right)\left( \begin{array}{c} 
  \triangle \phi^{mid}_j \\ \triangle \varphi^{mid}_j
  \end{array} \right) = \left( \begin{array}{c} 
   \sum_{k = 1}^2 \omega_{jk} \psi_{k}^{b, E_+} \\  \sum_{k' = 3}^4 \overline{\omega_{jk'}}  \psi_{k'}^{b, \bar E_-} 
  \end{array} \right)
\eee
where $\left( \begin{array}{c} 
  \triangle \phi^{mid}_j \\ \triangle \varphi^{mid}_j
  \end{array} \right) = \left( \begin{array}{c} 
   \phi^{mid}_j - \check  \phi^{mid}_j \\ \varphi^{mid}_j - \check \varphi^{mid}_j
  \end{array} \right)$. 
The conjugate linearity of $I- \check T^{mid}_j$ and definition of $(\check \phi^{mid}_j, \check \varphi^{mid}_j)$ from \eqref{eqImidcheck} indicate that $\left( \begin{array}{c} 
  \triangle \phi^{mid}_j \\ \triangle \varphi^{mid}_j
  \end{array} \right)  = \left( \begin{array}{c} 
  \sum_{k=1}^4 \omega_{jk} \check \phi^{mid}_k \\ \sum_{k=1}^4 \overline{\omega_{jk}} \check \varphi^{mid}_j
  \end{array} \right) $ solves the equation. It is also the unique solution from the invertibility of $I- \check T^{mid}_j$. The corresponding estimates come from the bound of $F^{mid, \pm}_{j;b,\l,\nu}$ \eqref{eqbddFmidpm} and integration estimates similar to \eqref{eqintest11}-\eqref{eqintest16}. In particular, we used $e^{-\frac{p-1}{8\Re \sqrt{E}} \Re\left( \eta_{b, E}(b^{-\frac 12}) + \frac{\pi  E}{2b}\right)} \sim e^{-\frac{(p-1)\Re E}{8\Re \sqrt E} b^{-\frac 12}} \le \tilde \epsilon$ from \eqref{eqetaconv}. 
  
  Now we focus on (b) and only consider the case of $j=1$ since $j=3$ will follow similarly. 
  Denote
\[ \triangle \Phi^k = \kappa_{b, 1+\l}^- \check \Phi^{mid, k}_{1;b, \l, \nu} - \Phi^k_{1;0, \l, \nu},  \quad k = 1, 2. \]
Consider the difference equation between \eqref{eqImidcheck} multiplied by $\kappa_{b, 1+\l}^-$ and \eqref{eqb0fund} with the second equation taking complex conjugation
\be
 \left|\begin{array}{l} 
 \triangle \Phi^1 = \kappa_{b, 1+\l}^- \psi_4^{b, E_+}(r) - e^{-\sqrt{1+\l} r} + \kappa_{b, 1+\l}^{-}  \tilde T_{x_*, b^{-\frac 12}; b, E_+}^{mid, G}  \check F^{mid, +}_1 -  T_{x_*; 1 + \l}^{(0), G} F^+_1 \\
  \overline{\triangle \Phi^2} =  \overline{\kappa_{b, 1+\l}^{-}} \tilde T_{x_*, b^{-\frac 12}; b, \bar E_-}^{mid, G}  \check F^{mid, -}_1 -  T_{x_*; 1 - \bar \l}^{(0), G} F^-_1 
  \end{array}\right.\label{eqconveq}
\ee
where $\check F^{mid, \pm}_j$ and $F^\pm_j$ are from \eqref{eqdefFjBoxpm} (with $\Phi^{mid,k}_j$ replaced by $\check \Phi^{mid,k}_j$) and \eqref{eqb0Fpm} respectively.
% \be
% \left|\begin{array}{ll}
%    &\kappa_{b, 1+\l}^- \check \Phi^{mid, 1}_{1;b} - \Phi^1_{1;0} \\
%    =& \kappa_{b, 1+\l}^- \psi_1^{b, E_+}(r) - e^{-\sqrt{1+\l} r} -  T_{x_*; 1 + \l}^{(0), G} \left( \left(\frac{\nu^2 - \frac 14}{r^2} - W_1 \right) \Phi^1_{1;0} - W_2 \Phi^2_{1;0} \right) \\
%   +& \tilde T_{x_*, b^{-\frac 12}; b, E_+}^{mid, G} \left[ \left( h_{b, E_+} + \frac{\nu^2 - \frac 14}{r^2} - W_{1, b}\right) \kappa_{b, 1+\l}^- \check \Phi^{mid, 1}_{1;b} - e^{i\frac{br^2}2} W_{2, b} \kappa_{b, 1+\l}^- \check \Phi^{mid, 2}_{1;b} \right]\\
%   &\kappa_{b, 1+\l}^- \check \Phi^{mid, 2}_{1;b} - \Phi^2_{1;0} \\ 
%   =& -  T_{x_*; 1 - \l}^{(0), D} \left( \left(\frac{\nu^2 - \frac 14}{r^2} - W_1 \right) \Phi^2_{1;0} - W_2 \Phi^1_{1;0} \right) \\
%   +& \overline{\tilde T_{x_*, b^{-\frac 12}; b, \bar E_-}^{mid, D} \left[ \left( h_{b, \bar E_-} + \frac{\nu^2 - \frac 14}{r^2} - W_{1, b}\right) \overline{  \kappa_{b, 1+\l}^- \check \Phi^{mid, 2}_{1;b} }- e^{i\frac{br^2}2} W_{2, b} \overline{  \kappa_{b, 1+\l}^- \check \Phi^{mid, 1}_{1;b} } \right]}
%   \end{array}\right. \label{eqconveq}
% \ee
Now we perform estimates within $r \in [x_*, b^{-\frac 12}]$ for the first equation. Recalling \eqref{eqpsibEderiv5} and $E_+ = 1 + \l + ibs_c$, we have
\be
  | \kappa_{b, 1+\l}^- \psi_4^{b, E_+}(r) - e^{-\sqrt{1+\l} r} | \lesssim b^\frac 12 |e^{-\sqrt{1+\l} r}|,\quad r\in [x_*, b^{-\frac 12}].  \label{eqconvest1}
\ee
For the Duhamel part in the first equation, we decompose as 
\bee
   &&\kappa_{b, 1+\l}^{-}  \tilde T_{x_*, b^{-\frac 12}; b, E_+}^{mid, G}  \check F^{mid, +}_1 -  T_{x_*; 1 + \l}^{(0), G} F^+_1  \\
   &=&  \left[ \tilde T_{x_*, b^{-\frac 12}; b, E_+}^{mid, G}  \left( \kappa_{b, 1+\l}^- \check F^{mid, +}_1 -  F^+_1 
   \right) \right] +  \left[ \left(\tilde T_{x_*, b^{-\frac 12}; b, E_+}^{mid, G} - T_{x_*, 1+\l}^{(0), G} \right)   F^+_1 \right] \\
   &=:& I + II
\eee
For $I$, we recall from 
Proposition \ref{propQbasymp} that $|P_b(r) - Q(r)| \lesssim b^\frac 16 Q(r)$ for $r \le b^{-\frac 12}$ and hence 
\be | W_{1, b}(r) - W_1(r)| +  | e^{i\frac{br^2}{2}} W_{2, b}(r) - W_2(r)|  \lesssim b^\frac 16 |Q(r)|^{p-1}, \quad r \in [0, b^{-\frac 12}],  \label{eqpotentialasymp}\ee
together with  $|h_{b, E}| \lesssim b^2$ for $r \le b^{-\frac 12}$ from \eqref{eqbddh}. Thus the boundedness of linear operator \eqref{eqtildeTmidGest2} implies
\be
\begin{split}
\| I \|_{ C^1_{e^{-\sqrt{1+\l}r }} ([x_*, b^{-\frac 12}]) } \lesssim& x_*^{-1} \| F^+_1 - \kappa_{b, 1+\l}^- \check F^{mid, +}_1 \|_{C^0_{e^{-\sqrt{1+\l}r } r^{-2} } ([x_*, b^{-\frac 12}]) } \\
\lesssim& x_*^{-1} \left(\| \triangle \Phi^1 \|_{C^0_{e^{-\sqrt{1+\l}r}}([x_*, b^{-\frac 12}])} +  \| \triangle \Phi^2 \|_{C^0_{e^{-\sqrt{1-\l}r}}([x_*, b^{-\frac 12}])} \right) \\
+& x_*^{-1} b^{\frac 16} \left(\| \Phi^1_1 \|_{C^0_{e^{-\sqrt{1+\l}r}}([x_*, b^{-\frac 12}])} +  \|  \Phi^2_1 \|_{C^0_{e^{-\sqrt{1-\l}r}}([x_*, b^{-\frac 12}])} \right)
  \end{split} \label{eqconvest1}
\ee
For $II$, the asymptotics of $\pa_r^k \psi_j^{b, E}$ for $ k = 0, 1$, $j = 2, 4$ in \eqref{eqpsibEderiv5} and $\kappa_{b, E}^- \kappa_{b, E}^+ W_{42}=  4E^{\frac 12}$ imply for $k = 0, 1$ that
\[ \frac{  \pa_r^k \psi_4^{b, E_+}(r) \psi_2^{b, E_+}(s)}{W_{42; E_+}} =  \frac{ \pa_r^k e^{-\sqrt{1+\l}(r-s)} }{4\sqrt{1+\l}} \cdot (1 + O(b^\frac 12)), \quad r, s \in [x_*, b^{-\frac 12}] \] 
Therefore
\be
\begin{split}
  \| II \|_{ C^1_{e^{-\sqrt{1+\l}r }} ([x_*, b^{-\frac 12}]) } &\lesssim x_*^{-1} b \|  F^{+}_1 \|_{ C^0_{e^{-\sqrt{1+\l}r } r^{-2} } ([x_*, b^{-\frac 12}]) } + b^{\frac 12} \|   F^{+}_1 \|_{ C^0_{e^{-\sqrt{1+\l}r } r^{-2} } ([b^{-\frac 12},\infty) ) } \\
  &\lesssim b^\frac 12 \left(\| \Phi^1_1 \|_{C^0_{e^{-\sqrt{1+\l}r}}([x_*, \infty))} +  \|  \Phi^2_1 \|_{C^0_{e^{-\sqrt{1-\l}r}}([x_*, \infty) } \right)
  \end{split} \label{eqconvest2}
\ee

Symmetrically, we can derive the counterparts of \eqref{eqconvest1}, \eqref{eqconvest2} for the second equation in \eqref{eqconveq}. Plugging them and controls of source terms \eqref{eqconvest1} and \eqref{eqestPhib0} into \eqref{eqconveq} yields \eqref{eqestconv1}-\eqref{eqestconv2} when $b_1 \ll 1$ and $x_* \gg 1$ by linear contraction principle.

\mbox{}

\textit{Step 2.7. Matching and construction of admissible fundamental solutions.} 

We are in place to match to construct the fundamental solution $\Phi_{k;b, \l, \nu}$ on $(0, \infty)$ with asymptotics identified on $[x_*, \infty)$ and parameters $k = 1, 2$, $0 < b \le b_1$, $|\l| \le \delta_1$, $\Im \l \le bI_0$, and $\nu \le \nu_0$. For notational simplicity, we omit the subscript $\l, \nu$ below. Let
\be
 \tilde \Phi_{k;b} = \left| \begin{array}{ll}
     \Phi^{ext}_{k;b} & r \in I_{ext} \\
     \sum_{j=1}^4 c^{con}_{kj;b} \Phi^{con}_{j;b} & r \in I_{con} \\
     \sum_{j=1}^4 c^{mid}_{kj;b} \Phi^{mid}_{j;b}& r \in I_{mid} 
 \end{array}\right.\quad k = 1, 2, \label{eqtildePhik}
\ee
with the asymptotics matching at $r_1, r_2$:
\bee
  \vec \Phi_{k;b}(r_1+0) = \vec \Phi_{k;b}(r_1-0),\quad  \vec \Phi_{k;b}(r_2+0) = \vec \Phi_{k;b}(r_2-0).
\eee
Besides, by \eqref{eqconnectcheck}, we can write $\tilde \Phi_{k;b}$ under the basis $\check \Phi^{mid}_{j;b}$ on $[x_*, b^{-\frac 12}]$ as 
\be
  \tilde \Phi_{k;b} = \sum_{j=1}^4 \check c^{mid}_{kj;b} \check \Phi^{mid}_{j;b} \quad {\rm for}\,\, r \in [x_*, b^{-\frac 12}];\qquad 
  \check c^{mid}_{kj;b} = c^{mid}_{kj;b} +  \sum_{l=1}^4 c^{mid}_{kl;b} \omega_{lj}.
  \label{eqtildePhik1} 
\ee

\mbox{}

Now we evaluate these coefficients $c^{con}_{jk;b}$, $c^{mid}_{jk;b}$ and $\check c^{mid}_{kj;b}$. First from \eqref{eqestgammacon}, \eqref{eqestgammamid}, we can compute to show the following inversions when $b_1 \ll 1$ and $x_* \gg 1$
\bea \small 
(I+\Gamma^{con, R}_b)^{-1} &=&
\left( \begin{array}{cccc}
       1 + O(b) & 0 & O(be^{-\frac{p-1}{64b}}) & 0 \\
       O(b) & 1 & O(be^{-\frac{p-1}{64b}}) & 0 \\
       O(be^{-\frac{p-1}{64b}}) & 0 & 1 + O(b) & 0  \\
       O(be^{2\eta_{b, \bar E_-}(r_1)}e^{-\frac{p-1}{64b}}) & 0 & O(be^{2\eta_{b, \bar E_-}(r_1)}) & 1
  \end{array}\right),\label{eqestgammaconinv}\\
  \small (I+\Gamma^{mid, R}_b)^{-1} &=& 
  \left( \begin{array}{cccc}
       1 + O(x_*^{-1}) & 0 & 0 & 0 \\
       O(be^{2\eta_{b, E_+}(r_2)}) & 1 &  O(be^{2\eta_{b, E_+}(r_2)}e^{-\frac{p-1}{64b}}) & 0 \\
       0 & 0 & 1+ O(x_*^{-1}) & 0  \\
        O(be^{-\frac{p-1}{64b}}) & 0 & O(b) & 1
        \end{array}\right)\label{eqestgammamidinv}
\eea
A direct computation from the matching condition at $r_1$, and definition of $A^R$ \eqref{eqARAL}, estimates of $\gamma^{ext}_{jk;b}$ \eqref{eqestgammaext} leads to
\bee \small
   \left(\begin{array}{cc}
    c^{con}_{11;b} & c^{con}_{21;b} \\
    c^{con}_{12;b} & c^{con}_{22;b}  \\
    c^{con}_{13;b} & c^{con}_{23;b} \\
    c^{con}_{14;b} & c^{con}_{24;b}
  \end{array}\right) 
  = \left[(I + \Gamma^{con, R}_b)A^R \right]^{-\top}   \left( \begin{array}{cc}
       1 + \gamma^{ext}_{11} & \gamma^{ext}_{21;b} \\
       \gamma^{ext}_{12;b} & \gamma^{ext}_{22;b} \\
       \gamma^{ext}_{13;b} & 1 + \gamma^{ext}_{23;b} \\
       \gamma^{ext}_{14;b} & \gamma^{ext}_{24;b} 
  \end{array}\right) 
 = \left( \begin{array}{cc}
       e^{-\frac{\pi i}{6}} + O(b) & O(be^{-\frac{p-1}{64b}}) \\
       O(b) & O(be^{-\frac{p-1}{64b}}) \\
       O(be^{-\frac{p-1}{64b}}) & 1 + O(b) \\
       O(be^{-\frac{p-1}{64b}}) & O(b) 
  \end{array}\right) 
\eee
where we used $|e^{-2\eta_{b, \bar E_-}(r_1)}| \lesssim 1$ and the cancellation between $e^{2\eta_{b, \bar E_-}(r_1)}$ and $e^{-2\eta_{b, \bar E_-}(r_1)}$. Similarly,  from the matching condition at $r_2$, and definition of $A^L$ \eqref{eqARAL}, estimates of $\Gamma^{con, L}_b$ \eqref{eqestgammacon}, we have
\bee
 \small \left(\begin{array}{cc}
    c^{mid}_{11;b} & c^{mid}_{21;b} \\
    c^{mid}_{12;b} & c^{mid}_{22;b}  \\
    c^{mid}_{13;b} & c^{mid}_{23;b} \\
    c^{mid}_{14;b} & c^{mid}_{24;b}
  \end{array}\right) 
  &=&   (I + \Gamma^{mid, R}_b)^{-\top} \left[(I + \Gamma^{con, L}_b) A^L\right]^{\top}  \left(\begin{array}{cc}
    c^{con}_{11;b} & c^{con}_{21;b} \\
    c^{con}_{12;b} & c^{con}_{22;b}  \\
    c^{con}_{13;b} & c^{con}_{23;b} \\
    c^{con}_{14;b} & c^{con}_{24;b}
  \end{array}\right) \\
  \small &=& \left( \begin{array}{cc}
       e^{-\frac{\pi i}{6}} + O(x_*^{-1}) & O(be^{-\frac{p-1}{64b}}) \\
       O(b e^{-2\eta_{b, E_+}(r_2)}) & O(b e^{-2\eta_{b, E_+}(r_2)}e^{-\frac{p-1}{64b}}) \\
       O(be^{-\frac{p-1}{64b}}) & e^{\frac{\pi i}{6}} + O(x_*^{-1}) \\
       O(be^{-\frac{p-1}{64b}}) & O(b)
  \end{array}\right)
\eee
where we used $|e^{2\eta_{b, E_+}(r_2)}| \lesssim 1$ and the cancellation between $e^{2\eta_{b, E_+}(r_2)}$ in $(I + \Gamma^{mid, R}_b)^{-1}$ and $e^{-2\eta_{b, E_+}(r_2)}$ in $I + \Gamma^{con, L}_b$.

Finally, we compute through \eqref{eqtildePhik1} using the estimates for $\Omega$ \eqref{eqestomegajk} and $c^{mid}_{jk;b}$ as above that
\bea
  \small &&\left(\begin{array}{cc}
   \check c^{mid}_{11;b} & \check c^{mid}_{21;b} \\
   \check c^{mid}_{12;b} & \check c^{mid}_{22;b}  \\
   \check c^{mid}_{13;b} & \check c^{mid}_{23;b} \\
   \check c^{mid}_{14;b} & \check c^{mid}_{24;b}
  \end{array}\right) = (I + \Omega)^\top  \left(\begin{array}{cc}
    c^{mid}_{11;b} & c^{mid}_{21;b} \\
    c^{mid}_{12;b} & c^{mid}_{22;b}  \\
    c^{mid}_{13;b} & c^{mid}_{23;b} \\
    c^{mid}_{14;b} & c^{mid}_{24;b}
  \end{array}\right)   \nonumber
  % \small = \left( \begin{array}{cccc} 
  % 1 & 0 & O(b) & 0 \\
  % O(b^\frac 12 e^{-2\eta_{b, E_+}(b^{-\frac 12})}) & 1 +  O(b^\frac 12) & O(b^\frac 32 e^{-2\eta_{b, E_+}(b^{-\frac 12})}) & O(b) \\
  % O(b) & 0 & 1 & 0 \\
  % O(b^\frac 32 e^{-2\eta_{b, \bar E_-}(b^{-\frac 12})}) & O(b) & O(b^\frac 12 e^{-2\eta_{b, \bar E_-}(b^{-\frac 12})}) & 1 + O(b^\frac 12)
  % \end{array}\right) 
  % \left( \begin{array}{cc}
  %      1 + \tilde \gamma^{mid}_{11;b} + O(b) & O(b) \\
  %      O(b) & O(b) \\
  %      O(b) & 1 + \tilde \gamma^{mid}_{33;b}+ O(b) \\
  %      O(b) & O(b)
  % \end{array}\right) \\
  \\
  \small &=&  
  \left( \begin{array}{cc}
       e^{-\frac{\pi i}{6}} + O(x_*^{-1}) & O(b^\frac 12 M^{-1} \tilde \epsilon) \\
       O(b^\frac 12 e^{-2\eta_{b, E_+}(b^{-\frac 12})}) & O(b^\frac 12 M^{-1} \tilde \epsilon e^{-2\eta_{b, E_+}(b^{-\frac 12})} ) \\
       O(b^\frac 12 M \tilde \epsilon) & e^{\frac{\pi i}{6}}  + O(x_*^{-1}) \\
       O(b^\frac 12M \tilde \epsilon  e^{-2\eta_{b, \bar E_-}(b^{-\frac 12})}) & O(b^\frac 12 e^{-2\eta_{b, \bar E_-}(b^{-\frac 12})})
  \end{array}\right) \label{eqestpalcheckc0}
\eea
Here $M, \tilde \epsilon$ are defined as in \eqref{eqdefMtildeep}, and to obtain the leading errors, we used $\min\{ |e^{-2\eta_{b, E_+}(b^{-\frac 12})}|, |e^{-2\eta_{b, \bar E_-}(b^{-\frac 12})}|\} \gg |e^{-2\eta_{b, E_+}(r_2)}|$ and $e^{-\frac{p-1}{64b}} \ll M^{-1} \tilde \epsilon$ due to \eqref{eqcomparisonomega3}.

We now take 
\be
 \Phi_{1;b, \l, \nu} = \kappa_{b, 1+\l}^-  (\check c^{mid}_{11;b, \l, \nu})^{-1} \tilde \Phi_{1;b, \l, \nu},\quad 
  \Phi_{2;b, \l, \nu} = \kappa_{b, 1-\l}^-  (\check c^{mid}_{23;b, \l, \nu})^{-1} \tilde \Phi_{2;b, \l, \nu} 
\label{eqPhi12adm}
\ee
as the fundamental solutions families in the statement of Proposition \ref{propextfund}. In particular, for $r \in [x_*, b^{-\frac 12}]$,
\be
\begin{split}
   \Phi_{1;b, \l, \nu} =& \kappa_{b, 1+\l}^- \left( \check \Phi^{mid}_{1;b,\l,\nu} + (\check c^{mid}_{11;b, \l, \nu})^{-1} \sum_{k \in \{2, 3, 4 \}} \check c_{1k;b,\l,\nu}^{mid} \check  \Phi^{mid}_{k;b, \l, \nu} \right) \\
   \Phi_{2;b, \l, \nu} =& \kappa_{b, 1-\l}^- \left( \check  \Phi^{mid}_{3;b,\l,\nu} + (\check c^{mid}_{23;b, \l, \nu})^{-1}  \sum_{k \in \{1, 2, 4 \}} \check  c_{2k;b,\l,\nu}^{mid} \check \Phi^{mid}_{k;b, \l, \nu} \right)
   \end{split} \label{eqPhi12asympx*}
\ee

\mbox{}\\

\underline{Step 3. Fundamental solutions for $0<b<b_1$: derivatives with respect to $\l$.}

In this step, we will compute the $\pa_\l$-derivatives of fundamental solutions on different regions constructed above and related coefficients. We still restrict ourselves to $\Re \l > 0$ in this step, and the extension to the whole region $\l \in \Omega_{\delta_1;I_0, b}$ will be discussed in the end.

\mbox{}

\textit{Step 3.1. Local analytic branches of fundamental solutions.}

While $\Phi^{ext}_{j;b,\l,\nu}$ and $\hat \Phi^{mid}_{j;b,\l,\nu}$ are analytic w.r.t. $\l$ and so are their matching coefficients $\hat c^{mid}_{jk;b,\l,\nu}$ \eqref{eqtildePhik}-\eqref{eqtildePhik1} as we will prove in later subsection, the fundamental solution families $\Phi^{con}_{j;b,\l,\nu}$ or $\Phi^{mid}_{j;b,\l,\nu}$ are not analytic since $r^*_{b, E}$ is $\RR$-valued and hence not analytic w.r.t. $E$. 
Nevertheless, we show in this subsection that we can safely neglect the dependence on $E$ from $r^*_{b, E}$ when computing $\pa_\l$-derivatives of fundamental solutions and coefficients in order to obtain derivatives of the final matching cofficients $\hat c^{mid}_{jk;b,\l,\nu}$, by constructing local analytic branches around a fixed $\l_0$. 
% This is realized by fixing $\l_0$ and construct fundamental solution families with $\l$ near $\l_0$ with resolvents not involving $r^*_{b, E_+(\l)}$ and $r^*_{b, \overline{E_-(\l)}}$. 

\mbox{}

For any fixed $\l_0$, we denote $E_{\pm,\l_0}$, $r_{1,\l_0}$, $r_{2,\l_0}$, $I_{ext,\l_0}$, $I_{con,\l_0}$, $I_{mid,\l_0}$ correspondingly. Then consider $|\l - \l_0| \ll 1$ such that 
$|r^*_{b, E_{+, \l}} - r_{1,\l_0}| + | r^*_{b, \overline{E_{-,\l_0}}} - r_{2,\l_0}| \ll 1$
by continuity of $E \mapsto r^*_{b, E}$ from Lemma \ref{lemWKBeta} (1). We now define the fundamental solutions $\Phi^{con,\l_0}_{j;b,\l,\nu}$, $\Phi^{mid,\l_0}_{j;b,\l,\nu}$ for \eqref{eqnu} with parameters $b, \l, \nu$ on $I_{con, \l_0}$, $I_{mid, \l_0}$ respectively as solution of 
\bea
  \left\{ \begin{array}{l}
     \phi^{con,\l_0}_{j;b,\l,\nu} = S^{con}_{j;b,\l,\nu} + \tilde T_{r_{20}, r_{10}; b, E_+}^{con, +} \left[ \left( h_{b, E_+} + \frac{\nu^2 - \frac 14}{r^2} - W_{1, b}\right) \phi^{con,\l_0}_{j;b,\l,\nu} - e^{i\frac{br^2}2} W_{2, b} \bar\varphi^{con,\l_0}_{j;b,\l,\nu} \right]  \\
 \varphi^{con,\l_0}_{j;b,\l,\nu} = R^{con}_{j;b,\l,\nu} + \tilde T_{r_{20}, r_{10}; b, \bar E_-}^{con, -} \left[ \left( h_{b, \bar E_-} + \frac{\nu^2 - \frac 14}{r^2} - W_{1, b}\right) \varphi^{con,\l_0}_{j;b,\l,\nu} - e^{i\frac{br^2}2} W_{2, b} \bar\phi^{con,\l_0}_{j;b,\l,\nu} \right]
 \end{array}\right. \label{eqIcon'} \\
 \left\{ \begin{array}{l}
     \phi^{mid,\l_0}_{j;b,\l,\nu} = S^{mid}_{j;b,\l,\nu} + \tilde T_{x_*, r_{20}; b, E_+}^{mid, +} \left[ \left( h_{b, E_+} + \frac{\nu^2 - \frac 14}{r^2} - W_{1, b}\right) \phi^{mid,\l_0}_{j;b,\l,\nu} - e^{i\frac{br^2}2} W_{2, b} \bar\varphi^{mid,\l_0}_{j;b,\l,\nu} \right]  \\
 \varphi^{mid,\l_0}_{j;b,\l,\nu} = R^{mid}_{j;b,\l,\nu} + \tilde T_{x_*, r_{20}; b, \bar E_-}^{mid, -} \left[ \left( h_{b, \bar E_-} + \frac{\nu^2 - \frac 14}{r^2} - W_{1, b}\right) \varphi^{mid,\l_0}_{j;b,\l,\nu} - e^{i\frac{br^2}2} W_{2, b} \bar\phi^{mid,\l_0}_{j;b,\l,\nu} \right]
 \end{array}\right. \label{eqImid'}
\eea
These systems are the same as \eqref{eqI0} with $\Box \in \{ con, mid\}$ except substituting the $r_1, r_2$ in the subscript of resolvent by $r_{10}, r_{20}$. Thanks to the continuous dependence in $r, E$ of $\psi_j^{b, E}(r)$ with $j = 1, 2, 3, 4$ and potentials $h_{b, E} + \frac{\nu^2-\frac 14}{r^2} - W_{1,b}$ and $e^{\frac {ibr^2}{2}}W_{1,b}$, repeating the proof in Step 2 yields the existence, uniqueness, and the same bounds of $\Phi^{con,\l_0}_{j;b,\l,\nu}$ and $\Phi^{mid,\l_0}_{j;b,\l,\nu}$. Also, we can extend or restrict $\Phi^{ext}_{j; b, \l, \nu}$ from $I_{ext}$ to $I_{ext,\l_0}$, which enjoys the same estimate as before for the same reasoning. 

Thereafter, the boundary data $\gamma^{ext, \l_0}_{jk;b,\l,\nu}$, $\gamma^{con, \l_0}_{jk;b,\l,\nu}$, $\gamma^{mid,\l_0}_{jk;b,\l,\nu}$ and matching coefficients  $c^{con,\l_0}_{jk;b,\l,\nu}$,  $c^{mid,\l_0}_{jk;b,\l,\nu}$ can be computed as before. Regarding the analytic solution $\check \Phi^{mid}_{j;b,\l,\nu}$, we also obtain $\omega^{\l_0}_{jk;b,\l,\nu}$ and $\check c^{mid,\l_0}_{jk;b,\l,\nu}$. We stress that at $\l = \l_0$, all the fundamental solution and coefficients are the same as constructed in Step 2. 

By the uniqueness of solution to ODE $\eqref{eqnu}$, we can rewrite \eqref{eqtildePhik}-\eqref{eqtildePhik1} with $|\l - \l_0| \ll 1$ as 
\be
 \tilde \Phi_{k;b,\l,\nu} = \left| \begin{array}{ll}
     \Phi^{ext}_{k;b,\l,\nu} & r \in I_{ext, \l_0} \\
     \sum_{j=1}^4 c^{con,\l_0}_{kj;b,\l,\nu} \Phi^{con,\l_0}_{j;b,\l,\nu} & r \in I_{con, \l_0} \\
     \sum_{j=1}^4 c^{mid,\l_0}_{kj;b,\l,\nu} \Phi^{mid,\l_0}_{j;b,\l,\nu}& r \in I_{mid, \l_0} \\
     \sum_{j=1}^4 \check  c^{mid,\l_0}_{kj;b,\l,\nu} \check \Phi^{mid}_{j;b,\l,\nu} & r \in [x_*, b^{-\frac 12}]
 \end{array}\right.\quad k = 1, 2, \label{eqtildePhik2}
\ee
In particular, compared with \eqref{eqtildePhik1}, we have 
\[ \check c^{mid}_{kj;b,\l,\nu}  = \check  c^{mid,\l_0}_{kj;b,\l,\nu},
% \quad  \check a_{kj;b,\l,\nu} = \check  a^{\l_0}_{kj;b,\l,\nu}, 
\quad  |\l-\l_0| \ll 1.   \]
So recall from the matching computations in Step 2.6, for $\check c^{mid,\l_0}_{kj;b,\l,\nu}$, we derive
\be
\begin{split}
&\pa_\l^n \left( \overrightarrow{\check c^{mid}_{1\cdot;b,\l,\nu}}, \overrightarrow{\check c^{mid}_{2\cdot;b,\l,\nu}} \right)\bigg|_{\l = \l_0} = \pa_\l^n \left( \overrightarrow{\check c^{mid,\l_0}_{1\cdot;b,\l,\nu}}, \overrightarrow{\check c^{mid,\l_0}_{2\cdot;b,\l,\nu}} \right)\bigg|_{\l = \l_0}\\
  =& \pa_\l^n \bigg\{ (I + \Omega^{\l_0}_{b,\l,\nu})^\top (I + \Gamma^{mid, R, \l_0}_{b,\l,\nu})^{-\top} \left[(I + \Gamma^{con, L,\l_0}_{b,\l,\nu}) A^L\right]^{\top} \\
  & \qquad \cdot \left[(I + \Gamma^{con, R,\l_0}_{b,\l,\nu})A^R \right]^{-\top} 
  \left( \vec e_1 + \overrightarrow{\gamma^{ext,\l_0}_{1\cdot;b,\l,\nu}}, \vec e_3 + \overrightarrow{\gamma^{ext,\l_0}_{2\cdot;b,\l,\nu}} \right)\bigg\}\bigg|_{\l = \l_0}
  \end{split} \label{eqcheckcmid2}
\ee
where the coefficient matrices or vectors are from \eqref{eqbdryext}, \eqref{eqbdrycon}, \eqref{eqbdrymid} and \eqref{eqconnectcheck}.
In the following steps, we will use $\rpa_\l$ to denote the differential on its local analytic branch, for example
\be
  \left(\rpa_\l \Phi^{con}_{j;b,\l,\nu}\right)\big|_{\l = \l_0} =  \left(\pa_\l \Phi^{con,\l_0}_{j;b,\l,\nu}\right)\big|_{\l = \l_0} \label{eqdefrpa}
\ee
and similarly for $\Phi^{mid}_{j;b,\l,\nu}$,  $\gamma^{ext}_{jk;b,\l,\nu}$, $\gamma^{con}_{jk;b,\l,\nu}$, $\gamma^{mid}_{jk;b,\l,\nu}$, $c^{con}_{jk;b,\l,\nu}$,  $c^{mid}_{jk;b,\l,\nu}$ and $\omega_{jk;b,\l,\nu}$. 
Practically, this is equivalent to forget the dependence on $E$ from $r_1$, $r_2$ when differentiating w.r.t. $\l$. 

\mbox{}

\textit{Step 3.2. $\pa_\l$-derivatives for fundamental solutions in $I_{ext}$.}

We claim that the fundamental solution $\Phi^{ext}_j$ with $j = 1 ,2 $ constructed on $I_{ext}$ further satisfy for $0 \le n \le K_0$ and $j = 1, 2$ that
\bea
  \| \pa_\l^n \Phi^{ext, 1}_{j;b,\l,\nu} \|_{X^{2n, L, -}_{r_1, \frac 4b; b, E_+}} + \| \overline{\pa_\l^n \Phi^{ext, 2}_{j;b,\l,\nu}} \|_{X^{2n, L, -}_{r_1, \frac 4b; b, \bar E_-}} \lesssim_{n, L} b^n,\quad \forall L \ge N_n, \label{eqIextpalest}
\eea
where $N_n$ comes from \eqref{eqchoiceNI1}. 

\mbox{}

Indeed, since $\pa_\l \Phi^{ext, 2}_j = \overline{ \pa_{\bar \l} \varphi^{ext}_j}$, we take $\pa_\l^N$, $\pa_{\bar\l}^N$ to the first and second scalar equation of \eqref{eqI0} (with $\Box = ext$) respectively to obtain
\be\left| \begin{array}{l}
 \pa_\l^N \phi^{ext}_{j;b,\l,\nu} = \pa_\l^N S^{ext}_{j;b, \l, \nu} + \sum_{n=0}^N \binom{N}{n} \tilde T^{ext, (n)}_{\frac 4b; b, E_+} \pa_\l^{N-n} F^{ext, +}_{j;b,\l,\nu} \\
  \pa_{\bar \l}^N \varphi^{ext}_{j;b,\l,\nu} = \pa_{\bar \l}^N R^{ext}_{j;b, \l, \nu} + \sum_{n=0}^N (-1)^n \binom{N}{n} \tilde T^{ext, (n)}_{\frac 4b; b, \bar E_-} \pa_{\bar \l}^{N-n} F^{ext, -}_{j;b,\l,\nu}
\end{array}\right. \label{eqpalKIext}
\ee
where $S^{ext}_{j;b, \l, \nu}, R^{ext}_{j;b, \l, \nu}$ are from \eqref{eqIextsource}, $F^{ext,\pm}_{j;b,\l,\nu}$ are as \eqref{eqdefFjBoxpm} and 
\begin{align*}
  &\tilde T^{ext, (n)}_{\frac 4b; b, E} f = \pa_E^n \left(  \tilde T^{ext}_{\frac 4b; b, E} f\right) \\
  =& \left| \begin{array}{ll}
      \pa_E^n\left( \psi_3^{b, E}  \calI^-_{L; b, E} [f](r)\right) +  \int_{r_1}^r \pa_E^n \left(\frac{\psi_1^{b, E}(r) \psi_3^{b, E}(s)}{W_{31;E}} \right) fds  & r \ge \frac 4b \\
        \pa_E^n \left[ \psi_3^{b, E}(r) \left(\calI^-_{L; b, E} [f]\left(\frac 4b\right) + \int^{\frac 4b}_r \psi_1^{b, E} f \frac{ds}{W_{31}} \right)\right] + \int_{r_1}^r \pa_E^n \left(\frac{\psi_1^{b, E}(r) \psi_3^{b, E}(s)}{W_{31;E}} \right)  fds   & r \in \left[ r_1, \frac 4b \right].
  \end{array}\right.
\end{align*}
Here $\calI^-_{L;b,E}$ is defined in Lemma \ref{leminvtildeHext} with $L$ taken large enough depending on decay of $f$. Extracting the highest order terms, we can reformulate \eqref{eqpalKIext} as 
\be (I - \calT^{ext})  \left( \begin{array}{c} 
  \pa_\l^N \phi^{ext}_{j;b,\l,\nu} \\ \pa_\l^N \varphi^{ext}_{j;b,\l,\nu}
  \end{array} \right) = 
  \left( \begin{array}{c} 
  S^{ext, (N)}_{j;b,\l,\nu} \\ R^{ext, (N)}_{j;b,\l,\nu}
  \end{array} \right)  \label{eqpalKIext2}
\ee
with $\calT^{ext}$ same as $N=0$ case from \eqref{eqI0} and \eqref{eqdefcalText}, and the source terms are
\be
\left| \begin{array}{l}
 S_{j;b,\l,\nu}^{ext, (N)} =  \pa_\l^N S^{ext}_{j;b, \l, \nu}  + \sum_{n=1}^{N} \binom{N}{n} \tilde T^{ext, (n)}_{\frac 4b; b, E_+} \pa_\l^{N-n} F^{ext, +}_{j;b,\l,\nu} + T^{ext, (0)}_{\frac 4b; b, E_+} \left([ \pa_\l^N, h_{b, E_+}] \phi^{ext}_{j;b,\l,\nu}\right) \\
 R_{j;b,\l,\nu}^{ext, (N)} =  \pa_{\bar \l}^N R^{ext}_{j;b, \l, \nu}  +  \sum_{n=1}^{N} (-1)^n \binom{N}{n} \tilde T^{ext, (n)}_{\frac 4b; b, \bar E_-} \pa_{\bar \l}^{N-n} F^{ext, -}_{j;b,\l,\nu}  + T^{ext, (0)}_{\frac 4b; b, E_+} \left([ \pa_{\bar \l}^N, h_{b, \bar E_-}] \varphi^{ext}_{j;b,\l,\nu}\right) 
 \end{array}\right.
 \label{eqsourceextN}
\ee
which only involves $\pa_\l^{n} \Phi^{ext}_j$ for $0 \le n \le N-1$. Using the boundedness \eqref{eqestcalText}, we can invert $I- \calT^{ext}$ on the left of \eqref{eqpalKIext2} when $b_1 \ll 1$ depending on $K_0$ and reduce the proof of \eqref{eqIextpalest} to the control of source term
\be
 \|  S_{j;b,\l,\nu}^{ext, (n)} \|_{X^{2n, N_n, -}_{r_1, \frac 4b; b, E_+}} + \| R_{j;b,\l,\nu}^{ext, (n)}  \|_{X^{2n, N_n, -}_{r_1, \frac 4b; b, \bar E_-}} \lesssim_{n} b^n,\quad  j = 1, 2. \label{eqestsourceextN}
\ee
Once we proved \eqref{eqestsourceextN} for some $n$, the higher order derivative control $L \ge N_n$ follows immediately from iterating \eqref{eqestcalText1} as in the proof of $n = 0$ case in Step 2.3.

To show \eqref{eqestsourceextN}, we first consider the boundedness of $\tilde T^{ext, (n)}_{\frac 4b; b, E}$. 
Note that by \eqref{eqDpm},
\bee
 \pa_E^n  \calI^-_{L; b, E} [f](r) = \int_r^\infty \pa_E^n \left[ e^{2i\phi_{b, E}} \left(\pa_r \circ (-2i\phi_{b, E}')^{-1}\right)^L \circ \left(e^{-2i\phi_{b, E}} \psi_1^{b, E} f\right) W_{31;E}^{-1}\right] ds \\
 + \sum_{l=0}^{L-1} \pa_E^n \left[ e^{2i\phi_{b, E}} \left(\pa_r \circ (-2i\phi_{b, E}')^{-1}\right)^l \circ \left(e^{-2i\phi_{b, E}} \psi_1^{b, E} f \right) \cdot (-2i\phi_{b, E}' W_{31;E})^{-1} \right]
\eee
where the phase function $\phi_{b, E} = \frac{E}{2b} \int_2^{\frac{br}{\sqrt E}} \sqrt{\tau^2 - 4} d\tau$ from \eqref{eqdefphi}. Using the apparent bounds
\bee 
  \left|\pa_E^n \pa_r^k \phi_{b, E} \right|\lesssim_{n,k} br^{2-k}, && n \ge 0,\,\, k \ge 0, \,\, r \ge \frac 4b, \\
  |\pa_E^n W_{31;E}| \sim_n b^{\frac 13} && n \ge 0,
\eee
and the estimates of $\pa_E^n \psi_1^{b, E}$ for $r \ge r^*_{b, E}$ \eqref{eqpsibEderiv1}, \eqref{eqpsibEderiv6}, a brute force computation shows 
\bee
  &&\left| D_{--;b, E}^m  \pa_E^n \left[ \frac{e^{2i\phi_{b, E}}}{ -2i\phi_{b, E}' W_{31;E} } \left(\pa_r \circ (-2i\phi_{b, E}')^{-1}\right)^l \circ \left(e^{-2i\phi_{b, E}} \psi_1^{b, E} f \right) \right]  \right| \\
   &\lesssim_{l,m,n}& (br^2)^{n-l-2} r^{-m+\a + 2} |e^{-2\eta_{b, E}}| \| f \|_{X^{\a, N, -}_{r^*_{b, E}, \frac 4b; b, E}},\quad r \ge \frac 4b,
  \eee
namely each $D_{--}$ derivative creates $r^{-1}$ decay, while each $\pa_E$ derivative brings $br^2$ growth. We need
\bee
 N \ge m+l
\eee
for differentiability of $f$. 
Similarly,
\bee
  \left| \int_r^{\infty} \pa_E^n \left[ e^{2i\phi_{b, E}} W_{31;E}^{-1} \left(\pa_r \circ (-2i\phi_{b, E}')^{-1}\right)^L \circ \left(e^{-2i\phi_{b, E}} \psi_1^{b, E} f\right)\right] ds \right| \\
  \lesssim_{L,n} \int_r^\infty (bs^2)^{n-L-1} s^{\a+1} |e^{-2\eta_{b,E}(s)}| ds \cdot \| f \|_{X^{\a, N, -}_{r^*_{b, E}, \frac 4b; b, E}} \\
  \lesssim (br^2)^{n-L-1} r^{\a + 2} |e^{-2\eta_{b,E}(r)}| \cdot \| f \|_{X^{\a, N, -}_{r^*_{b, E}, \frac 4b; b, E}},\quad r \ge \frac 4b. 
\eee
where we require for the integrability over $r$ and differentiability of $f$ that
\bee L \ge n + \frac \a 2 + \max \left\{ \frac{\Im E}{b}, 0 \right\} + 1,\quad N \ge L,  \eee
and the exponential $|e^{-2\eta}|$ is treated similarly to the proof of Lemma \ref{leminvtildeHext}, using monotonicity \eqref{eqReetamono} for $\Im E < 0$ or $|e^{-2\eta(s)}| \sim (br)^{2\frac{\Im E}{b}}$ from \eqref{eqetaRe} for $0 \le \Im E \le bI_0$. Further combined with \eqref{eqD++commutator}, \eqref{eqestfnk} and \eqref{eqgnk}, we have
\begin{align*}
  &\quad\left| D_{--;b, E}^m \int_r^{\infty} \pa_E^n \left[ e^{2i\phi_{b, E}} W_{31;E}^{-1} \left(\pa_r \circ (-2i\phi_{b, E}')^{-1}\right)^L \circ \left(e^{-2i\phi_{b, E}} \psi_1^{b, E} f\right)\right] ds \right| \\
  &\le \sum_{k=1}^m  \sum_{j=0}^k \bigg| f_{m,k;b, E}(r) \int_r^\infty g_{k,j;b,E}(s) \\
  &\quad\qquad \cdot D^j_{--;b, E} \pa_E^n \left[ e^{2i\phi_{b, E}} W_{31;E}^{-1} \left(\pa_r \circ (-2i\phi_{b, E}')^{-1}\right)^L \circ \left(e^{-2i\phi_{b, E}} \psi_1^{b, E} f\right)\right] ds \bigg| \\
  &\lesssim_{L,m,n} \| f \|_{X^{\a, N, -}_{r^*_{b, E}, \frac 4b; b, E}} \sum_{k=1}^m  \sum_{j=0}^k b^k r^{2k-m} \int_r^\infty b^{-k} s^{j-2k} \cdot (bs^2)^{n-L-1} s^{-j+\a+1} |e^{-2\eta_{b, E}(s)}| ds\\
  &\lesssim (br^2)^{n-L-1} r^{-m+\a+2}  |e^{-2\eta_{b, E}(r)}| \cdot \| f \|_{X^{\a, N, -}_{r^*_{b, E}, \frac 4b; b, E}},\quad r \ge \frac 4b.
\end{align*}
Here we required 
\bee  L \ge n + \frac \a2 + \max \left\{ \frac{\Im E}{b}, 0 \right\} + 1 - \frac m2,\quad N \ge L + m.
\eee

The three estimates above yield
\be
  \left| D_{--;b, E}^m \pa_E^n  \calI^-_{L; b, E} [f] (r)\right| \lesssim_{m,n,L} (br)^{-2} (br^2)^n r^{-m+\a} \| f \|_{X^{\a, m+L, -}_{r_1, \frac 4b;b, E}}. \label{eqpaEcalIest}
\ee
for $L \ge n + \frac \a2 + \max \left\{ \frac{\Im E}{b}, 0 \right\} + 1$, namely each $\pa_E$ creates $O(br^2)$ growth. Thereafter, arguing similarly as the proof of Lemma \ref{leminvtildeHext} (3) Step 2-3, we have 
\be
 \left\| \tilde T^{ext, (n)}_{\frac 4b;b,E} \right\|_{X^{\a, N+L_{n,\a,E}, -}_{r_1, \frac 4b;b,E} \to X^{\a+2n, N, -}_{r_1, \frac 4b;b,E} } \lesssim_{n, N} b^{-1+n},
\ee
with $\a \ge -2$, $n \ge 1$, $N \ge 0$ and 
\be  L_{n, \a, E} := \min \left\{ L \in \NN_{\ge 0}: L \ge n + \frac \a 2 + \left\{\frac{\Im E}{b}, 0 \right\} + 1  \right\} \le K_0 + I_0 + 1 =: L^*. \ee
Also notice that the boundedness of $\pa_r^k \pa_E^n h_{b, E}$ \eqref{eqbddh2} and of the potentials \eqref{eqfundI1est1}-\eqref{eqfundI1est2} imply for $\a \in \RR$ and $n \ge 0$ that
\bee
  &&\left\|\pa_{ \l}^{n} F^{ext, +}_{j;b,\l,\nu} \right\|_{X^{\a-2, N, -}_{r_1, \frac 4b; b, E_+}} + \left\|\pa_{\bar \l}^{n} F^{ext, -}_{j;b,\l,\nu} \right\|_{X^{\a-2, N, -}_{r_1, \frac 4b; b, \bar E_-}} \\
  &+&\left\|[\pa_\l^{n+1}, h_{b, E_+}] \phi^{ext}_{j;b,\l,\nu} \right\|_{X^{\a-2, N, -}_{r_1, \frac 4b; b, E_+}} + 
   \left\|[\pa_{\bar\l}^{n+1}, h_{b,\bar E_-}] \varphi^{ext}_{j;b,\l,\nu} \right\|_{X^{\a-2, N, -}_{r_1, \frac 4b; b, \bar E_-}} 
   \\
  &\lesssim& \sum_{k = 0}^n \left(  \| \pa_\l^k \Phi^{ext, 1}_j \|_{X^{\a, N, -}_{r_1, \frac 4b; b, E_+}} + \| \overline{\pa_\l^k \Phi^{ext, 2}_j} \|_{X^{\a, N, -}_{r_1, \frac 4b; b, \bar E_-}}    \right)
\eee
Hence 
\bee
 && {\rm LHS\,\,of\,\,\eqref{eqestsourceextN}} \lesssim b^N + \sum_{n = 0}^{N} b^{-1+n} \sum_{k=0}^{N- \min \{1,n\}} \\ 
 && \qquad \qquad \qquad  \left(  \| \pa_\l^k \Phi^{ext, 1}_j \|_{X^{2N-2n+2, N_n+L^*, -}_{r_1, \frac 4b; b, E_+}} + \| \overline{\pa_\l^k \Phi^{ext, 2}_j} \|_{X^{2N-2n+2, N_n+L^*, -}_{r_1, \frac 4b; b, \bar E_-}}    \right) \\
 &\lesssim& b^N + \sum_{n=0}^{N-1} b^{n+1} \left(  \| \pa_\l^k \Phi^{ext, 1}_j \|_{X^{2N-2n, N_n+L^*, -}_{r_1, \frac 4b; b, E_+}} + \| \overline{\pa_\l^k \Phi^{ext, 2}_j} \|_{X^{2N-2n,  N_n+L^*, -}_{r_1, \frac 4b; b, \bar E_-}}    \right).  \eee
Thus \eqref{eqestsourceextN} and \eqref{eqIextpalest} are proven by induction on $N$.

\mbox{}

\textit{Step 3.3. $\pa_\l$-derivatives for fundamental solutions in $I_{mid}$ and $I_{con}$}

Recall the derivative $\rpa$ on local analytic branch defined in \eqref{eqdefrpa}. We claim the following estimate for $0 \le n \le K_0$ derivatives on fundamental solutions on $I_{mid}$ when $x_* \gg 1$ and $b_1 \ll 1$ depending on $K_0$ and $\nu_0$:
\begin{align}
    \| \rpa_\l^n \left( \kappa_{b, 1+ \l}^- \Phi^{mid, 1}_{1; b, \l, \nu}\right) \|_{C^1_{\omega_{b, E_+}^- r^n}(I_{mid})} + e^{-\frac{\pi \Re \l}{b} } \|\overline{\rpa_\l^n \left( \kappa_{b, 1+ \l}^- \check \Phi^{mid, 2}_{1; b, \l, \nu}\right) } \|_{C^1_{\omega_{b, \bar E_-}^- e^{-\frac{p-1}{8}r} r^n}(I_{mid})} \lesssim |\kappa_{b, 1+ \l}^-| , \label{eqestPhimidb1n} \\
    \| \rpa_\l^n \Phi^{mid,1}_{2;b,\l,\nu} \|_{C^1_{\omega_{b, E_+}^+}(I_{mid})} + 
    e^{\frac{\pi(p-1)\Re {\bar E_-}}{16\Re \sqrt{\bar E_-} b} + \frac{\pi \Re \l}{b}}\|\overline{\rpa_\l^n \Phi^{mid,2}_{2;b,\l,\nu}} \|_{C^1_{\omega_{b, \bar E_-}^+ e^{-\frac{p-1}{8\Re \sqrt {\bar E_-}}\Re \eta_{b, \bar E_-}} }(I_{mid})} \lesssim b^{-n},\label{eqestPhimidb2n}\\
     e^{\frac{\pi \Re \l}{b} } \| \rpa_\l^n \left( \kappa_{b, 1- \l}^- \Phi^{mid, 1}_{3; b, \l, \nu}\right) \|_{C^1_{\omega_{b, E_+}^- e^{-\frac{p-1}{8}r} r^n }(I_{mid})} +  \| \overline{\rpa_\l^n \left( \kappa_{b, 1- \l}^- \Phi^{mid, 2}_{3; b, \l, \nu}\right) } \|_{C^1_{\omega_{b, \bar E_-}^- r^n} (I_{mid})} \lesssim |\kappa_{b, 1- \l}^-| , \label{eqestPhimidb3n}\\
    e^{\frac{\pi(p-1)\Re {E_+}}{16\Re \sqrt{E_+} b} -\frac{\pi \Re \l}{b} } \| \rpa_\l^n \Phi^{mid,1}_{4;b,\l,\nu} \|_{C^1_{\omega_{b, E_+}^+ e^{-\frac{p-1}{8\Re \sqrt {E_+}}\Re \eta_{b, E_+}} }(I_{mid})} + \| \overline{\rpa_\l^n \Phi^{mid,2}_{4;b,\l,\nu}}\|_{C^1_{\omega_{b, \bar E_-}^+}(I_{mid})} \lesssim b^{-n}. \label{eqestPhimidb4n}
\end{align}
and on $I_{con}$ 
\be
\begin{split}
    \| \rpa_\l^n \Phi^{con, 1}_{j;b,\l,\nu} \|_{C^0_{\omega_{b, E_+}^{-}}(I_{con})} + \| \overline{\rpa_\l^n \Phi^{con, 2}_{j;b,\l,\nu}} \|_{C^0_{\omega_{b, \bar E_-}^{-}}(I_{con})} \lesssim_{n} b^{-n},\quad j = 1, 2;\\
    \| \rpa_\l^n \Phi^{con, 1}_{j;b,\l,\nu} \|_{C^0_{\omega_{b, E_+}^{+}}(I_{con})} + \| \overline{\rpa_\l^n \Phi^{con, 2}_{j;b,\l,\nu}} \|_{C^0_{\omega_{b, \bar E_-}^{+}}(I_{con})} \lesssim_{n} b^{-n},\quad j = 3, 4.
\end{split}
   \label{eqIconpalest}
\ee

% In this substep, we slightly abuse the notation to write 
% \[ \pa_\l^n \Phi^{mid}_{j;b,\l,\nu}\big|_{\l = \l_0} := \pa_\l^n \Phi^{mid,\l_0}_{j;b,\l,\nu} \big|_{\l = \l_0},\quad \pa_\l^n \Phi^{con}_{j;b,\l,\nu}\big|_{\l = \l_0} := \pa_\l^n \Phi^{con,\l_0}_{j;b,\l,\nu} \big|_{\l = \l_0}, \]
% and so are derivatives of coefficients with the local analytic branch introduced in Step 3(1). 

\mbox{}

For $I_{mid}$, for $j = 2, 4$ we take $(\pa_\l^N, \pa_{\bar \l}^N)$ the system \eqref{eqImid'} to get 
\be\left| \begin{array}{l}
 \rpa_\l^N \phi^{mid}_{j;b,\l,\nu} = \pa_\l^N S^{mid}_{j;b, \l, \nu} + \sum_{n=0}^N \binom{N}{n} \tilde T^{mid, \sigma_+(j), (n)}_{x_*, r_2; b, E_+} \rpa_\l^{N-n} F^{mid, +}_{j;b,\l,\nu} \\
  \rpa_{\bar \l}^N \varphi^{mid}_{j;b,\l,\nu} = \pa_{\bar \l}^N R^{mid}_{j;b, \l, \nu} + \sum_{n=0}^N (-1)^n \binom{N}{n} \tilde T^{mid, \sigma_-(j), (n)}_{x_*, r_2; b, \bar E_-} \rpa_{\bar \l}^{N-n} F^{mid, -}_{j;b,\l,\nu}
\end{array}\right. \label{eqpalKImid}
\ee
and for $j = 1, 3$, we multiply by $(\kappa_{1\pm_j\l}^-, \overline{\kappa_{1\pm_j\l}^-})$ for $\pm_1 = +, \pm_3 = -$ and differentiate
    \be\left| \begin{array}{l}
 \rpa_\l^N \left( \kappa_{b, 1\pm_j\l}^- \phi^{mid}_{j;b,\l,\nu} \right) = \rpa_\l^N \left( \kappa_{b, 1\pm_j\l}^- S^{mid}_{j;b,\l,\nu} \right) \\
 \qquad \qquad + \sum_{n=0}^N \binom{N}{n} \tilde T^{mid, \sigma_+(j), (n)}_{x_*, r_2; b, E_+} \rpa_\l^{N-n} \left( \kappa_{b, 1\pm_j \l}^- F^{mid, +}_{j;b,\l,\nu} \right) \\
  \rpa_{\bar \l}^N \left( \overline{\kappa_{b, 1\pm_j\l}^-}  \varphi^{mid}_{1;b,\l,\nu} \right) = \rpa_{\bar \l}^N \left( \overline{\kappa_{b, 1\pm_j\l}}^- R^{mid}_{j;b,\l,\nu} \right)\\
 \qquad \qquad + \sum_{n=0}^N (-1)^n \binom{N}{n} \tilde T^{mid, \sigma_-(j), (n)}_{x_*, r_2;b, \bar E_-} \pa_{\bar \l}^{N-n} \left( \overline{\kappa_{b, 1\pm_j\l}^-} F^{mid, -}_{j;b,\l,\nu} \right)
\end{array}\right. \label{eqpalKImid2}
\ee
where
$S^{mid}_{j;b, \l, \nu}, R^{mid}_{j;b, \l, \nu}$ are from \eqref{eqImidsource}, $F^{mid,\pm}_{j;b,\l,\nu}$ are as \eqref{eqdefFjBoxpm}, the symbols $\sigma^\pm(j)$ are as \eqref{eqcalTmidsigma}, 
and the operator families being
\begin{align*}
\left| \begin{array}{l} 
\tilde T^{mid, G, (n)}_{x_*, x^*; b, E}f =  - \int_{r}^{x^*} \pa_{E}^n \left(\frac{\psi_2^{b, E}(r) \psi_4^{b, E}(s)}{W_{42}} \right) fds - \int_{x_*}^{r} \pa_{E}^n \left( \frac{\psi_4^{b, E}(r) \psi_2^{b, E}(s)}{W_{42}} \right)f ds  \\
\tilde T^{mid, D, (n)}_{x_*, x^*; b, E} f=  - \int_{r}^{x^*} \pa_{E}^n \left(\frac{\psi_2^{b, E}(r) \psi_4^{b, E}(s)}{W_{42}} \right) fds + \int_{r}^{x^*}\pa_{E}^n \left( \frac{\psi_4^{b, E}(r) \psi_2^{b, E}(s)}{W_{42}} \right)f ds
\end{array}\right.
\end{align*}

Similar as the case for $I_{ext}$, we can extract the $\rpa_\l^N \phi^{mid}_{j;b,\l,\nu}$ and $ \rpa_{\bar \l}^N \varphi^{mid}_{j;b,\l,\nu}$ term to formulate the system like \eqref{eqpalKIext2} with the linear operator being $I - \calT^{mid}_j$ from \eqref{eqI0} and \eqref{eqdefcalTmid}, and the source of the form \eqref{eqsourceextN} composed by lower order terms. From the boundedness of $\calT^{mid}_j$ as \eqref{eqcalTmidest}, it suffices to check inductively that source term satisfies the target estimate \eqref{eqestPhimidb1n}-\eqref{eqestPhimidb4n}. This is verified by \eqref{eqpsibEderiv2}, \eqref{eqpsibEderiv4}, and the boundedness 
 \bea
      \left\| \tilde T^{mid, G, (n)}_{x_*, x^*; b, E} f \right\|_{C^1_{\omega^+_{b, E} e^{-\frac{p-1}{16} \Re \eta_{b, E}}} (I_{mid})} &\lesssim_{n}& x_*^{-1} b^{-n} \| f \|_{C^0_{\omega^+_{b, E} e^{-\frac{p-1}{16} \Re\eta_{b, E}} r^{-2}} (I_{mid})}\label{eqtildeTmidGest1N} \\
      \left\| \tilde T^{mid, G, (n)}_{x_*, x^*; b, E} f \right\|_{C^1_{\omega^-_{b, E}  r^{k+n} } (I_{mid})} &\lesssim_{k, n} & x_*^{-1} \| f \|_{C^0_{\omega^-_{b, E} r^{k-2}} (I_{mid})} \label{eqtildeTmidGest2N} \\
      % \left\| \tilde T^{mid, G}_{x_*, x^*; b, E} f \right\|_{C^0_{\omega^+_{b, E}e^{-\a \eta_{b, E}} r^{-1}} ([x_*, x^*])} &\lesssim& \| f \|_{C^0_{\omega^+_{b, E}e^{-\a r} r^{-2}} ([x_*, x^*])} \\
      \left\| \tilde T^{mid, D, (n)}_{x_*, x^*; b, E} f \right\|_{C^1_{\omega^-_{b, E} e^{-\frac{p-1}{8}} r^{k+n}}  (I_{mid})} &\lesssim_{k, n} & x_*^{-1} \| f \|_{C^0_{\omega^-_{b, E} e^{-\frac{p-1}{8} r} r^{k-2}} (I_{mid})}, \label{eqtildeTmidDestN} 
    \eea
    plus the $O(r^{-2})$ smallness of $\pa_E^n \left( h_{b, E} + \frac{\nu^2 - \frac 14}{r^2} - W_{1,b} \right)$, and boundedness of $W_{2, b}$ by \eqref{eqpotentialmidest1}, \eqref{eqpotentialmidest2} to control $\rpa_{\l}^n F^{mid, +}_{j;b,\l,\nu}$ and $\rpa_{\bar \l}^n F^{mid, -}_{j;b,\l,\nu}$. 

    To prove \eqref{eqtildeTmidGest1N}-\eqref{eqtildeTmidDestN}, we first notice that when taking $\pa_r$ derivative to the operators, the boundary term cancels and we have
\begin{align*}
\left| \begin{array}{l} 
\pa_r \tilde T^{mid, G, (n)}_{x_*, x^*; b, E}f =  - \int_{r}^{x^*} \pa_r \pa_{E}^n \left(\frac{\psi_2^{b, E}(r) \psi_4^{b, E}(s)}{W_{42}} \right) fds - \int_{x_*}^{r} \pa_r\pa_{E}^n \left( \frac{\psi_4^{b, E}(r) \psi_2^{b, E}(s)}{W_{42}} \right)f ds  \\
\pa_r \tilde T^{mid, D, (n)}_{x_*, x^*; b, E} f=  - \int_{r}^{x^*} \pa_r \pa_{E}^n \left(\frac{\psi_2^{b, E}(r) \psi_4^{b, E}(s)}{W_{42}} \right) fds + \int_{r}^{x^*} \pa_r\pa_{E}^n \left( \frac{\psi_4^{b, E}(r) \psi_2^{b, E}(s)}{W_{42}} \right)f ds.
\end{array}\right.
\end{align*}
Then with $\kappa_{b, E}^- \kappa_{b, E}^+ W_{42}^{-1}=  (2E^{\frac 12})^{-1}$ from \eqref{eqnormalcoeff} and \eqref{eqWronski2}, we bound the integral kernel via \eqref{eqpsibEderiv4}
    \bee
       && \left| \pa_r^k \pa_{E}^n \left( \frac{\psi_4^{b, E}(r) \psi_2^{b, E}(s)}{W_{42}} \right)\right| = \left|  \pa_r^k\pa_E^n \left[ 2E^{-\frac 12} \cdot ( \kappa_{b, E}^- \psi_4^{b, E}(r)) \cdot (\kappa_{b, E}^+ \psi_2^{b, E}(s)) \right]\right| \\
       &\lesssim_{n, k}& \max \{r, s \}^n \cdot \omega_{b, E}^-(s) \omega_{b, E}^+(r) b^{-\frac 13},\qquad {\rm for}\,\, r, s \in [0, r^*_{b, E}],\,\, n \ge 0,\,\,k \in \{ 0, 1\}.
    \eee
    The similar estimate holds for $ \pa_r^k \pa_{E}^n \left( \frac{\psi_4^{b, E}(s) \psi_2^{b, E}(r)}{W_{42}} \right)$. 
    They imply the boundedness \eqref{eqtildeTmidGest1N}-\eqref{eqtildeTmidDestN} in the similar fashion as the proof of Lemma \ref{leminvtildeHmid} (2), using the integral estimates \eqref{eqintest11}-\eqref{eqintest16}.

\mbox{}

For $I_{con}$, similarly, by differentiating the system \eqref{eqIcon'} w.r.t. $(\pa_\l^N, \pa_{\bar \l}^N)$, we obtain 
\be\left| \begin{array}{l}
 \pa_\l^N \phi^{con}_{j;b,\l,\nu} = \pa_\l^N S^{con}_{j;b, \l, \nu} + \sum_{n=0}^N \binom{N}{n} \tilde T^{con, +, (n)}_{r_2, r_1; b, E_+} \pa_\l^{N-n} F^{con, +}_{j;b,\l,\nu} \\
  \pa_{\bar \l}^N \varphi^{con}_{j;b,\l,\nu} = \pa_{\bar \l}^N R^{con}_{j;b, \l, \nu} + \sum_{n=0}^N (-1)^n \binom{N}{n} \tilde T^{con, -, (n)}_{r_2, r_1; b, \bar E_-} \pa_{\bar \l}^{N-n} F^{con, -}_{j;b,\l,\nu}
\end{array}\right. \label{eqpalKIcon}
\ee
where
$S^{con}_{j;b, \l, \nu}, R^{con}_{j;b, \l, \nu}$ are from \eqref{eqIconsource}, $F^{con,\pm}_{j;b,\l,\nu}$ are as \eqref{eqdefFjBoxpm} and 
\begin{align*}
\left| \begin{array}{l} 
\tilde T^{con, +, (n)}_{r_2, r_1; b, E_+} =  - \int_{r}^{r_1} \pa_{E_+}^n \left(\frac{\psi_2^{b, E_+}(r) \psi_1^{b, E_+}(s)}{W_{12;E_+}} \right) fds - \int_{r_2}^{r} \pa_{E_+}^n \left( \frac{\psi_1^{b, E_+}(r) \psi_2^{b, E_+}(s)}{W_{12;E_+}} \right)f ds  \\
\tilde T^{con, -, (n)}_{r_2, r_1; b, \bar E_-} =   \int^{r_1}_r \pa_{\bar E_-}^n \left(\frac{\psi_3^{b, \bar E_-}(r)  \psi_1^{b, \bar E_-}(s)}{W_{31;b,\bar E_-}} \right)f ds + 
\int^r_{r_2} \pa_{\bar E_-}^n \left(\frac{\psi_1^{b, \bar E_-}(r)  \psi_3^{b, \bar E_-}(s)}{W_{31;b,\bar E_-}} \right)f ds. 
\end{array}\right.
\end{align*}
From \eqref{eqpsibEderiv1}-\eqref{eqpsibEderiv3} and noticing that $r \sim b^{-1}$ on $I_{con}$, we easily have the boundedness  of  $\rpa_{\l}^n F^{mid, +}_{j;b,\l,\nu}$, $\rpa_{\bar \l}^n F^{mid, -}_{j;b,\l,\nu}$ and of the operators
\be
\begin{split}
  \left\| \tilde T^{con, +, (n)}_{r_2, r_1; b, E_+} f   \right\|_{C^0_{\omega_{b, E_+}^{\pm}}(I_{con})} \lesssim_n b^{-1-n} \| f \|_{C^0_{\omega_{b, E_+}^{\pm}}(I_{con})}, \quad n \ge 0,\\
  \left\| \tilde T^{con, -, (n)}_{r_2, r_1; b, \bar E_-} f   \right\|_{C^0_{\omega_{b, \bar E_-}^{\pm}}(I_{con})} \lesssim_n b^{-1-n} \| f \|_{C^0_{\omega_{b, \bar E_-}^{\pm}}(I_{con})},\quad n \ge 0.
  \end{split} \label{eqbddinvcon}
\ee
Likewise, that implies \eqref{eqIconpalest} by induction on $n$. 

\mbox{
}

\textit{Step 3.4. $\pa_\l$-derivatives for analytic fundamental solutions on $[x_*, b^{-\frac 12}]$}

Considering the fundamental solution family $\check \Phi^{mid}_{j;b,\l,\nu}$ determined by \eqref{eqImidcheck}. Repeating the computations above implies that $\pa_\l^n \check \Phi^{mid}_{j;b,\l,\nu}$, for $j = 2, 4$ and $\pa_\l^n \left(\kappa_{b, 1+\l}^- \check \Phi^{mid}_{1;b,\l,\nu} \right)$, $\pa_\l^n \left(\kappa_{b, 1-\l}^- \check \Phi^{mid}_{3;b,\l,\nu} \right)$ satisfy \eqref{eqestPhimidb1n}-\eqref{eqestPhimidb4n}  for $0 \le n \le K_0$ with $I_{mid}$ replaced by $[x_*, b^{-\frac 12}]$. Similar to the case of $n = 0$ in Step 2.5, we can replace $\omega_{b, E}^\pm$ by exponential functions and rewrite the bounds more explicitly as
\begin{align}
  \| \pa_\l^n \left( \kappa_{b, 1+\l}^- \check \Phi^{mid, 1}_{1;b,\l,\nu}\right) \|_{C^1_{e^{-\sqrt{1+\l}r } r^n } ([x_*, b^{-\frac 12}])} + \|  \pa_\l^n \left( \kappa_{b, 1+\l}^- \check \Phi^{mid, 2}_{1;b,\l,\nu} \right)\|_{C^1_{e^{-(\sqrt{1-\l} + \frac{p-1}{8})r} r^n } ([x_*, b^{-\frac 12}])} \lesssim 1, \label{eqestpalcheckPhi1} \\
  \| \kappa_{b, 1+\l}^+ \pa_\l^n \check \Phi^{mid, 1}_{2;b,\l,\nu} \|_{C^1_{e^{\sqrt{1+\l}r } } ([x_*, b^{-\frac 12}])} + \| \kappa_{b, 1+\l}^+ \pa_\l^n \check \Phi^{mid, 2}_{2;b,\l,\nu} \|_{C^1_{e^{(\sqrt{1-\l} - \frac{p-1}{8})r} } ([x_*, b^{-\frac 12}])} \lesssim b^{-n},\label{eqestpalcheckPhi2} \\
  \|  \pa_\l^n \left( \kappa_{b, 1-\l}^- \check \Phi^{mid, 1}_{3;b,\l,\nu} \right) \|_{C^1_{e^{-(\sqrt{1+\l} + \frac{p-1}{8})r} r^n } ([x_*, b^{-\frac 12}])} + \| \pa_\l^n \left( \kappa_{b, 1-\l}^- \check \Phi^{mid, 2}_{3;b,\l,\nu}\right) \|_{C^1_{e^{-\sqrt{1-\l} r} r^n } ([x_*, b^{-\frac 12}])} \lesssim 1, \label{eqestpalcheckPhi3} \\
  \| \kappa_{b, 1-\l}^+ \pa_\l^n \check \Phi^{mid, 1}_{4;b,\l,\nu} \|_{C^1_{e^{(\sqrt{1+\l} - \frac{p-1}{8})r } } ([x_*, b^{-\frac 12}])} + \| \kappa_{b, 1-\l}^+ \pa_\l^n \check \Phi^{mid, 2}_{4;b,\l,\nu} \|_{C^1_{e^{\sqrt{1-\l}r} } ([x_*, b^{-\frac 12}])} \lesssim b^{-n}, \label{eqestpalcheckPhi4}
\end{align}

Now we discuss the asymptotics as $b \to 0$. 
 Recall the normalizing coefficient $\kappa^-_{b, E}$ from \eqref{eqnormalcoeff}, and fundamental solutions $\Phi_{j;0,\l,\nu}$ from \eqref{eqb0fund},
    then for $0 \le n \le K_0$, we claim
    \begin{align}
    \left \|\pa_\l^n \left( \kappa_{b, 1+ \l}^- \check \Phi^{mid, }_{1; b, \l, \nu}\right) - \pa_\l^n \Phi_{1; 0, \l, \nu} \right \|_{\left(C^1_{e^{-\sqrt{1+\l} r} r^n}([x_*, b^{-\frac 12}]) \times C^1_{e^{-(\sqrt{1-\l}+\frac{p-1}{8}) r}r^n}([x_*, b^{-\frac 12}]) \right) } &\lesssim b^{\frac 16} \label{eqestconv3} \\
     \left \|\pa_\l^n \left( \kappa_{b, 1- \l}^-\check \Phi^{mid, }_{3; b, \l, \nu} \right)- \pa_\l^n \Phi_{3; 0, \l, \nu} \right \|_{\left(C^1_{e^{-(\sqrt{1+\l} + \frac{p-1}{8}) r}r^n}([x_*, b^{-\frac 12}]) \times C^1_{e^{-\sqrt{1-\l} r}r^n}([x_*, b^{-\frac 12}]) \right) } &\lesssim b^{\frac 16}.
     \label{eqestconv4}
    \end{align}
    The proof is similar to the one for \eqref{eqestconv1}-\eqref{eqestconv2}, and we only sketch the proof for \eqref{eqestconv3}.
    We take complex conjugation for the second equation of \eqref{eqb0fundN} and subtract them from \eqref{eqpalKImid2} to find 
    \be 
      \left| \begin{array}{l}
 \pa_\l^N \triangle \Phi^1 = S_N + \sum_{n=0}^N \binom{N}{n} (S_{N, n}^+ + \tilde S_{N, n}^+) \\
 \overline{\pa_\l^N \triangle \Phi^2} =  \sum_{n=0}^N (-1)^n \binom{N}{n} (S_{N, n}^- + \tilde S_{N, n}^-)
 % \pa_\l^N \left( \kappa_{b, 1+\l}^- \psi_1^{b, E_+}(r) - e^{-\sqrt{1+\l} r}\right) \\
 % \qquad \qquad   + \sum_{n=0}^N \binom{N}{n}  \left[ \tilde T_{x_*, b^{-\frac 12}; b, E_+}^{mid, G, (n)}  \pa_\l^{N-n} \left(  \kappa_{b, 1+\l}^{-} \check F^{mid, +}_1 \right) -  T_{x_*; 1 + \l}^{(n), G} \pa_\l^{N-n} F^+_1  \right] \\
 %  \overline{\pa_{\l}^N \triangle \Phi^2} = \sum_{n=0}^N (-1)^n \binom{N}{n} \left[   \tilde T_{x_*, b^{-\frac 12}; b, E_+}^{mid, D, (n)} \pa_{\bar \l}^{N-n} \left( \overline{\kappa_{b, 1+\l}^-} \check F^{mid, -}_1 \right) -  T_{x_*; 1 + \bar \l}^{(n), D} \pa_{\bar \l}^{N-n} F^-_1 \right]
 \end{array}\right.\label{eqconveqK0}
    \ee
    where
    \[\small
    \left| \begin{array}{l}
         S_N = \pa_\l^N \left( \kappa_{b, 1+\l}^- \psi_1^{b, E_+}(r) - e^{-\sqrt{1+\l} r}\right),   \\
         S_{N, n}^+ =  \tilde T_{x_*, b^{-\frac 12}; b, E_+}^{mid, G, (n)} \left[  \pa_\l^{N-n} \left(  \kappa_{b, 1+\l}^{-} \check F^{mid, +}_1 - F^+_1 \right) \right],\quad
         \tilde S_{N, n}^+ = \left( \tilde T_{x_*, b^{-\frac 12}; b, E_+}^{mid, G, (n)}  - T_{x_*; 1 + \l}^{(n), G} \right) \pa_\l^{N-n} F^+_1,\\
         S_{N, n}^- =  \tilde T_{x_*, b^{-\frac 12}; b, \bar E_-}^{mid, D, (n)} \left[  \pa_{\bar\l}^{N-n} \left(  \overline{\kappa_{b, 1+\l}^-} \check F^{mid, -}_1 - F^-_1 \right) \right],\quad
         \tilde S_{N, n}^- = \left( \tilde T_{x_*, b^{-\frac 12}; b, \bar E_-}^{mid, D, (n)}  - T_{x_*; 1 - \bar \l}^{(n), D} \right) \pa_{\bar\l}^{N-n} F^-_1,
    \end{array}\right.
    \]
%     \be
%  \left| \begin{array}{l}
%  \pa_\l^N \triangle \Phi^1 = \pa_\l^N \left( \kappa_{b, 1+\l}^- \psi_1^{b, E_+}(r) - e^{-\sqrt{1+\l} r}\right) \\
%  \qquad \qquad   + \sum_{n=0}^N \binom{N}{n}  \left[ \tilde T_{x_*, b^{-\frac 12}; b, E_+}^{mid, G, (n)}  \pa_\l^{N-n} \left(  \kappa_{b, 1+\l}^{-} \check F^{mid, +}_1 \right) -  T_{x_*; 1 + \l}^{(n), G} \pa_\l^{N-n} F^+_1  \right] \\
%   \overline{\pa_{\l}^N \triangle \Phi^2} = \sum_{n=0}^N (-1)^n \binom{N}{n} \left[   \tilde T_{x_*, b^{-\frac 12}; b, E_+}^{mid, D, (n)} \pa_{\bar \l}^{N-n} \left( \overline{\kappa_{b, 1+\l}^-} \check F^{mid, -}_1 \right) -  T_{x_*; 1 + \bar \l}^{(n), D} \pa_{\bar \l}^{N-n} F^-_1 \right]
%  \end{array}\right.\label{eqconveqK0}
% \ee
We will estimate each term on the RHS. From \eqref{eqpsibEderiv5}, we have $\| S_N \|_{C^1_{e^{-\sqrt{ (1+\l)r}}r^N}([x_*,b^{-\frac 12}])} \lesssim_N b^\frac 12$. For $S_{N, n}^+$, we compute
\bee
  &&\pa_\l^{n} \left(  \kappa_{b, 1+\l}^{-} \check F^{mid, +}_1 - F^+_1 \right) \\
  &=& 
  \left( \frac{\nu^2 - \frac 14}{r^2} + W_1 \right) \pa_\l^n \triangle \Phi^1 + W_2  \pa_\l^n \triangle \Phi^2 \\
  &+&  \pa_\l^{n}\left(  (h_{b, E} + W_{1, b} - W_1) \kappa_{b, 1+\l}^{-} \check \Phi^{mid, 1}_{1;b, \l, \nu} \right) + (e^{ibr^2}W_{2, b}  - W_2) \pa_\l^n \left(\kappa_{b, 1+\l}^{-} \check \Phi^{mid, 2}_{1;b, \l, \nu} \right) 
\eee
So with \eqref{eqtildeTmidGest2N}, \eqref{eqpotentialasymp} and \eqref{eqbddh2}, we have
\bea
  && \|S_{N, N-n}^+\|_{C^1_{e^{-\sqrt{ (1+\l)r}}r^N}([x_*,b^{-\frac 12}])}\nonumber \\
  &\lesssim_{n, N} & x_*^{-1}\| \pa_\l^{n} \left(  \kappa_{b, 1+\l}^{-} \check F^{mid, +}_1 - F^+_1 \right) \|_{C^0_{e^{-\sqrt{1+\l} r } r^{n-2}} ([x_*, b^{-\frac 12}])} \nonumber \\
  &\lesssim& x_*^{-1}\left(\| \pa_\l^n \triangle \Phi^1 \|_{C^0_{e^{-\sqrt{1+\l} r } r^{n}} ([x_*, b^{-\frac 12}])} + \| \pa_\l^n \triangle \Phi^2 \|_{C^0_{e^{-\sqrt{1-\l} r } r^{n-2}} ([x_*, b^{-\frac 12}])}  \right)\nonumber \\
  &+& x_*^{-1}b^\frac 16 \sum_{k=0}^n  \left(\|\pa_\l^k \check \Phi^{mid, 1}_{1;b, \l, \nu} \|_{C^0_{e^{-\sqrt{1+\l} r } r^{n}} ([x_*, b^{-\frac 12}])} + \|\pa_\l^k \check \Phi^{mid, 2}_{1;b, \l, \nu} \|_{C^0_{e^{-\sqrt{1-\l} r } r^{n}} ([x_*, b^{-\frac 12}])} \right). \nonumber
\eea
For $\tilde S_{N, n}^+$, we again apply \eqref{eqpsibEderiv5} to see 
% and similar for $\pa_{\bar \l}^n \left( \overline{\kappa_{b, 1+\l}^-} \check F^{mid, -}_1 - F^-_1 \right)$. 
% The asymptotics of $\pa_E^n \psi_j^{b, E}$ \eqref{eqpsibEderiv5} implies smallness of source term and
\bee
  \pa_r^k \pa_\l^n \left( \frac{\psi_2^{b, E_+}(r) \psi_1^{b, E_+}(s)}{W_{12;E_+}}-\frac{e^{-\sqrt{1+\l}(r-s) }}{2\sqrt{E_+}} \right) =  O\left(b^\frac 12 \max \{r, s\}^n e^{-\sqrt{1+\l}(r-s)} \right), \quad r, s \in [x_*, b^{-\frac 12}],
\eee
for $k = 0, 1$ and $n \ge 0$. Thus combined with \eqref{eqbddTk} and its proof, we have 
\be
  \left\|\left( \tilde T_{x_*, b^{-\frac 12}; b, E_+}^{mid, \sigma, (n)}  - T_{x_*; 1 + \l}^{(n), \sigma} \right) f  \right\|_{C^1_{e^{-\sqrt{1+\l} r } r^N }([x_*,b^{-\frac 12}]) } 
  % \lesssim x_*^{-\frac 12} b \| f \|_{C^0_{e^{-\sqrt{1+\l} r } r^{n-2} }([x_*,b^{-\frac 12}]) } + b^\frac 12 \| f \|_{C^0_{e^{-\sqrt{1+\l} r } r^{n-2} }([b^{-\frac 12}, \infty)) } \\
  \lesssim_{n,N} b^\frac 12 \| f \|_{C^0_{e^{-\sqrt{1+\l} r } r^{N-n-2} }([x_*, \infty)) }  \label{eqpalNdifference2} 
\ee
for $\sigma = G, D$, and therefore $\| \tilde S_{N, n}^+\|_{C^1_{e^{-\sqrt{ (1+\l)r}}r^N}([x_*,b^{-\frac 12}])} \lesssim_{N, n} b^\frac 12$ using the boundedness of $\pa_\l^n \Phi_{1;0,\l,\nu}$ \eqref{eqestPhib0N}. With similar estimates for the second equation of \eqref{eqconveqK0}, we see when $b_1 \ll 1$ and $x_* \gg 1$ depending on $\nu_0, K_0$, \eqref{eqconveqK0} is a linear contraction leading to \eqref{eqestconv3} with induction on $n$. 

%
%Like the proof of \eqref{eqestconv1} in Step 2(5), \eqref{eqestconv3} follows a linear contraction estimate for the first difference equation of \eqref{eqconveqK0}, with the contraction follows from small of {\color{red} There we used $\calT_j^{mid}$ as operator, but here it seems hard to write in that way? At least I didn't formulate as that... Perhaps I should and that distinguish the source term and contraction part more clearly.}  difference estimate using \eqref{eqpalNdifference1}, \eqref{eqpalNdifference2} plus the boundedness of $ \tilde T_{x_*, b^{-\frac 12}; b, E_+}^{mid, G, (n)}$ \eqref{eqtildeTmidGest2N}, $\pa_\l^n \Phi_{1;0,\l,\nu}$ \eqref{eqestPhib0N}, and $\pa_\l^n \left( \kappa_{b, 1+ \l}^- \check \Phi^{mid, }_{1; b, \l, \nu}\right)$ \eqref{eqestPhimidb1n}. 
%

\mbox{}

\textit{Step 3.5. $\pa_\l$-derivatives for boundary values and matching coefficients.}

Now we are in place to estimate the derivatives of coefficients appearing in \eqref{eqcheckcmid2}.
% Roughtly speaking, we will see each derivative brings at most $O(b^{-1})$ growth.
We still use the notation $\rpa$ from \eqref{eqdefrpa} for differential on the local analytic branch.

 Recall $\gamma^{ext}_{jk;b,\l,\nu}$ from \eqref{eqdefgammaext}, $\gamma^{con}_{jk;b,\l,\nu}$ from \eqref{eqdefgammacon}, 
  $\gamma^{mid}_{jk;b,\l,\nu}$ from \eqref{eqdefgammamid} and 
  $\omega_{jk;b,\l,\nu}$ from \eqref{eqdefomegajk}. The boundedness of $\pa_E^n \psi_j^{b, E}$ \eqref{eqpsibEderiv1}-\eqref{eqpsibEderiv3}, of $\pa_E^n h_{b, E}$ \eqref{eqbddh2} and of solution families $\pa_\l^n \Phi^{ext}_{j;b, \l, \nu}$ \eqref{eqIextpalest}, $\rpa_\l^n \Phi^{con}_{j;b, \l, \nu}$ \eqref{eqIconpalest} and $\rpa_\l^n \Phi^{mid}_{j;b, \l, \nu}$ \eqref{eqestPhimidb1n}-\eqref{eqestPhimidb4n} imply that each derivative $\pa_\l$ creates at most $O(b^{-1})$ growth for each term (for $\gamma^{ext}_{jk;b,\l,\nu}$ we also used 
\eqref{eqpaEcalIest} to control $\pa_E^n \calI^-_{b, E}[f]$). So compared with $n=0$ case \eqref{eqestgammaext}, \eqref{eqestgammacon}, \eqref{eqestgammamid}, \eqref{eqestomegajk}, we have for $n \ge 1$
\begin{align}
&\small (\rpa_\l^n \gamma^{ext}_{jk;b,\l,\nu})_{\substack{1 \le j \le 2\\ 1 \le k \le 4}} = b^{1-n} \left( \begin{array}{cccc}
        O(1) & O(1) & O(e^{-\frac{p-1}{64b}}) & O(e^{-\frac{p-1}{64b}} e^{-2\eta_{b, \bar E_-, \nu}(r_1)} ) \\
        O(e^{-\frac{p-1}{64b}}) & O(e^{-\frac{p-1}{64b}}) &   O(1) & O(e^{-2\eta_{b, \bar E_-, \nu}(r_1)})
  \end{array}\right)\\
% &\small\rpa_\l^n \gamma^{ext}_{jk;b,\l,\nu} = \left| \begin{array}{ll} O(b^{1-n}), & (j, k) \in  \{ 1, 2\} \times \{1, 2, 3 \}\\
%  O(b^{1-n}e^{-2\eta_{b, \bar E_-}(r_1)} ), & (j,k) = (1, 4), (2, 4)
% \end{array}\right. \\
&\small\rpa_\l^n \Gamma^{con, R}_{b, \l, \nu} = b^{-n} \left( \begin{array}{cccc}
       O(b) & 0 & O(be^{-\frac{p-1}{64b}}) & 0 \\
       O(b) & 0 & O(be^{-\frac{p-1}{64b}}) & 0 \\
       O(be^{-\frac{p-1}{64b}}) & 0 & O(b) & 0  \\
       O(be^{-\frac{p-1}{64b}}) & 0 & O(be^{2\eta_{b, \bar E_-}(r_1)}) & 0
  \end{array}\right), \\
  &\small\rpa_\l^n \Gamma^{con, L}_{b, \l, \nu} = b^{-n} \left( \begin{array}{cccc}
       0 & O(be^{-2\eta_{b, E_+}(r_2)}) & 0 & O(be^{-\frac{p-1}{64b}})  \\
       0 & O(b) & 0 & O(be^{-\frac{p-1}{64b}})  \\
       0 & O(be^{-\frac{p-1}{64b}}) & 0 & O(b)   \\
       0 & O(be^{-\frac{p-1}{64b}}) & 0 & O(b) 
  \end{array}\right),\\
  & \small\rpa_\l^n \Gamma^{mid, R}_{b, \l, \nu} = b^{-n} \left( \begin{array}{cccc}
       O(1) & 0 & 0 & 0 \\
       O(be^{2\eta_{b, E_+}(r_2)}) & 0 & O(be^{2\eta_{b, E_+}(r_2)}e^{-\frac{p-1}{64b}}) & 0 \\
       0 & 0 & O(1) & 0  \\
       O(be^{-\frac{p-1}{64b}}) & 0 & O(b) & 0
  \end{array}\right),\\
  & \small\rpa_\l^n \Omega_{b, \l, \nu} = b^{\frac 12-n}
  \left( \begin{array}{cccc}
       0 & O(e^{-2\eta_{b, E_+}(b^{-\frac 12})}) & O(M \tilde \epsilon) & O(M e^{-2\eta_{b, \bar E_-}(b^{-\frac 12})} \tilde \epsilon
       ) \\
       0 & O(1) & 0 & O(M^{-1} \tilde \epsilon)  \\
       O(M^{-1}\tilde \epsilon) & O(M^{-1} e^{-2\eta_{b, E_+}(b^{-\frac 12}) }\tilde \epsilon) & 0 & O(e^{-2\eta_{b, \bar E_-}(b^{-\frac 12})})  \\
       0 & O(M\tilde \epsilon) & 0 & O(1)
  \end{array}\right)
  % \left(\begin{array}{cccc}
  %     0 & O(b^\frac 12 e^{-2\Re \eta_{b, E_+}(b^{-\frac 12})}) & O(b) & O(b^\frac 32 e^{-2\Re \eta_{b, \bar E_-}(b^{-\frac 12})})  \\
  %     0 & O(b^\frac 12) & 0 & O(b) \\
  %     O(b) & O(b^\frac 32 e^{-2\Re \eta_{b, E_+}(b^{-\frac 12})}) & 0 & O(b^\frac 12 e^{-2\Re \eta_{b, \bar E_-}(b^{-\frac 12})}) \\
  %     0 & O(b) & 0 & O(b^\frac 12)
  % \end{array} \right)
% &\rpa_\l^n \gamma_{jk;b, \l, \nu}^{con} =  \left| \begin{array}{ll}
%      O(b^{1-n}e^{-2\eta_{b, E_+}(r_2)})   & (j,k) = (1, 2) \\
%      O(b^{1-n}e^{2\eta_{b, \bar E_-}(r_1)})  & (j, k) = (4, 3) \\
%      O(b^{1-n}) & {\rm otherwise}
%   \end{array}\right. \\
%   &\rpa_\l^n \gamma^{mid}_{jk;b, \l, \nu} = \left|  \begin{array}{ll}
%       O(b^{-n}x_*^{-1}) & (j, k) = (1,1), (3,3) \\
%       O(b^{1-n}e^{2\eta_{b, E_+}(r_2)}) & (j,k) = (2,1), (2, 3) \\
%       O(b^{1-n}), & (j,k) = (4,1), (4, 3)
%   \end{array}\right.  \\
%  &\rpa_\l^n \omega_{jk;b, \l, \nu}  = \left| \begin{array}{ll} 
% O(b^{1-n}e^{-\frac{p-1}{8} b^{-\frac 12}})  & (j, k) = (1, 3), (3, 1), (4, 2), (2, 4)  \\
% O(b^{\frac 12-n}) & (j, k) = (2, 2), (4, 4) \\
% O(b^{\frac 12-n} e^{\frac{\pi (1 \pm \l)}{b} - 2 \sqrt{1\pm \l}b^{-\frac 12}}) & +: (j, k) = (1, 2), \,\, -: (j, k) = (3, 4) \\
% O(b^{1-n} e^{\frac{\pi (1 \pm \l)}{b} - (2 \sqrt{1\pm \l} + \frac{p-1}{8}) b^{-\frac 12}}) & +: (j, k) = (3, 2), \,\, -: (j, k) = (1, 4)
% \end{array}\right.
\end{align}
where $M, \tilde \epsilon$ are as in \eqref{eqdefMtildeep}. 
To compute the derivative of inversions, we first use $\pa_t^N (A(t) A(t)^{-1}) = \delta_{N0}$ to obtain inductive formula 
\be \pa_t^N \left[(A(t))^{-1}\right] = \sum_{n=1}^N A(t)^{-1} \cdot \pa_t^n A(t) \cdot \pa_t^{N-n} \left[(A(t))^{-1}\right], N \ge 1. \label{eqinvmatrix} \ee
Then combine the derivative estimate above with the inversion estimate \eqref{eqestgammaconinv} and \eqref{eqestgammamidinv}, we can inductively show for $n \ge 1$, 
\bea
  \small \rpa_\l^n \left[ (I + \Gamma^{con, R}_{b,\l,\nu})^{-1} \right] = b^{-n} \left( \begin{array}{cccc}
       O(b) & 0 & O(be^{-\frac{p-1}{64b}}) & 0 \\
       O(b) & 0 & O(be^{-\frac{p-1}{64b}}) & 0 \\
       O(be^{-\frac{p-1}{64b}}) & 0 & O(b) & 0  \\
       O(be^{2\eta_{b, \bar E_-}(r_1)}e^{-\frac{p-1}{64b}})  & 0 & O(be^{2\eta_{b, \bar E_-}(r_1)}) & 0
  \end{array}\right);\\
   \small \rpa_\l^n \left[ (I + \Gamma^{mid, R}_{b,\l,\nu})^{-1} \right] = b^{-n} \left( \begin{array}{cccc}
       O(1) & 0 & 0 & 0 \\
       O(be^{2\eta_{b, E_+}(r_2)}) & 0 & O(be^{2\eta_{b, E_+}(r_2)}e^{-\frac{p-1}{64b}}) & 0 \\
       0 & 0 & O(1) & 0  \\
       O(be^{-\frac{p-1}{64b}}) & 0 & O(b) & 0
  \end{array}\right).
\eea

Therefore, we can estimate $\pa_\l^n \check c^{mid}_{jk;b,\l,\nu}$ for $0 \le n \le K_0$ from \eqref{eqcheckcmid2} similar to the computation in Step 2.7
\be 
 \small \pa_\l^n \left(\begin{array}{cc}
   \check c^{mid}_{11;b,\l,\nu} & \check c^{mid}_{21;b,\l,\nu} \\
   \check c^{mid}_{12;b,\l,\nu} & \check c^{mid}_{22;b,\l,\nu}  \\
   \check c^{mid}_{13;b,\l,\nu} & \check c^{mid}_{23;b,\l,\nu} \\
   \check c^{mid}_{14;b,\l,\nu} & \check c^{mid}_{24;b,\l,\nu}
  \end{array}\right) = b^{-n} 
%   \left( \begin{array}{cc}
% O(1) & O(b) \\
% O(b^\frac 12 e^{-2\Re \eta_{b, E_+}(b^{-\frac 12})}) & O(b^\frac 32 e^{-2\Re \eta_{b, E_+}(b^{-\frac 12})})  \\
% O(b) & O(1)  \\
% O(e^{-2\Re \eta_{b, \bar E_-}(b^{-\frac 12})}) & O(b^\frac 12 e^{-2\Re \eta_{b, \bar E_-}(b^{-\frac 12})})
% \end{array}
% \right)
  \left( \begin{array}{cc}
       O(1) & O(b^\frac 12 M^{-1} \tilde \epsilon) \\
       O(b^\frac 12 e^{-2\eta_{b, E_+}(b^{-\frac 12})}) & O(b^\frac 12 M^{-1} \tilde \epsilon e^{-2\eta_{b, E_+}(b^{-\frac 12})} ) \\
       O(b^\frac 12 M \tilde \epsilon) & O(1) \\
       O(b^\frac 12M \tilde \epsilon  e^{-2\eta_{b, \bar E_-}(b^{-\frac 12})}) & O(b^\frac 12 e^{-2\eta_{b, \bar E_-}(b^{-\frac 12})})
  \end{array}\right)
, \label{eqestpalcheckc}
\ee
where $M, \tilde \epsilon$ are as in \eqref{eqdefMtildeep}. We remark that for $n=0$ we have better control \eqref{eqestpalcheckc0}.

\mbox{}

\textit{Step 3.6. Extension to $\l \in \Omega_{\delta_1;I_0,b}$.}

In the previous Steps 3.1-3.5, we have shown analyticity w.r.t. $\l \in \{\Re \l > 0 \} \cap \Omega_{\delta_1;I_0,b}$  and estimates of $\Phi^{ext}_{j;b,\l,\nu}$, $\check \Phi^{mid}_{k;b,\l,\nu}$, $\check c^{mid}_{jk;b,\l,\nu}$ and $\check a_{jj';b,\l,\nu}$ for $j, j' = 1, 2$ and $k = 1, 2, 3, 4$. Apparently, similar to Step 2, the estimates can be proven verbatim for the $\l \in \{\Re \l < 0 \} \cap \Omega_{\delta_1;I_0,b}$.

Since from Step 2, all the functions and quantities above on both branches can be extended continuously to $\{\Re \l = 0 \} \cap \Omega_{\delta_1;I_0,b}$ and coincide, the analyticity on the whole region $\Omega_{\delta_1;I_0,b}$ now follows Morera's theorem. Moreover, the estimates on $\{ \Re \l = 0\}$ can be obtained by continuity from $\{ \Re \l > 0 \}$.

\mbox{}

\underline{Step 4. End of proof.} 
% {\color{red} Remove all the contents about linear independence, and analyticity of $\Psi_3$, $\Psi_4$.}

We have constructed admissible fundamental solution families $\Phi_{1;b,\l,\nu}$, $\Phi_{2;b,\l,\nu}$ from \eqref{eqPhi12adm}. As a solution to the linear ODE system \eqref{eqnu}, they can all be smoothly extended to $r \in (0,\infty)$.  

\mbox{}

\noindent{\textit{Step 4.1. Asymptotic behavior at infinity for $b=0$.}} 

The asymptotics \eqref{eqb0extbdd} and non-degeneracy \eqref{eqnondegb0} of $b=0$ case are derived in Step 1 from \eqref{eqestPhib0} (in particular the extra decay of $\Psi^2_{1;0,\l,\nu}$ and $\Psi^1_{2;0,\l,\nu}$) and \eqref{eqestPhib0N}. Finally, the linear independence of $\{\pa_\l^n \Phi_{j;0,\l,\nu}\}_{\substack{ n \ge 0 \\ 1 \le j \le 2}}$ follows from that of $\{\Phi_{j;0,\l,\nu}\}_{\substack{ 1 \le j \le 2}}$ and $(\HH_{b,\nu} - \l)\pa_\l^n \Phi_{j;b,\l,\nu} = n\pa_\l^{n-1}\Phi_{j;b,\l,\nu}$ by differentiating \eqref{eqnu}.

% For linear independence, we first note $\{\Phi_{j;0,\l,\nu}\}_{j = 1, 2, 3, 4}$ are linear independent due to the contraction estimate \eqref{eqestPhib0}, and that of the family $\{ \pa_\l^n \Phi_{j;0,\l,\nu}\}_{n\ge 0, 1\le j \le 4}$ follows differentiating \eqref{eqnu} to find $(\HH_{b,\nu} - \l)\pa_\l^n \Phi_{j;b,\l,\nu} = n\pa_\l^{n-1}\Phi_{j;b,\l,\nu}$ for $n \ge 1$, $1 \le j \le 4$ and $0 \le b \le b_1$.

\mbox{}

\noindent{\textit{Step 4.2. Asymptotic behavior at infinity for $b > 0$.}} 

Noticing that $\pa_r^n \left[\left( \frac{b^2r^2}{4} - E \right)^\frac 12 - \frac{br}{2}\right] = O(r^{-n-1})$ for $r \ge \frac 4b$, the boundedness \eqref{eqadmosc} can be reduced to showing boundedness in $X^{2k,N,\pm}_{r_0;b,E}$ from \eqref{eqdefXBanach}.
Since $\Phi_{j;b,\l,\nu}$ is parallel to $\Phi^{ext}_{j;b,\l,\nu}$ for $j = 1, 2$ from \eqref{eqPhi12adm} and \eqref{eqtildePhik}, so the boundedness of $\Phi^{ext}_{j;b,\l,\nu}$ \eqref{eqPhiextbdd}, \eqref{eqIextpalest} implies \eqref{eqadmosc}. Moreover, the non-degeneracy and linear independence of $\Phi^{ext}_{1;b,\l,\nu}$ and $\Phi^{ext}_{2;b,\l,\nu}$ follows from the smallness of $\Phi^{ext,1}_{2;b,\l,\nu}$ and $\Phi^{ext,2}_{1;b,\l,\nu}$ from \eqref{eqPhiextnondeg} and \eqref{eqcomparisonomega3}. The linear independence of $\{\pa_\l^n \Phi_{j;b,\l,\nu}\}_{\substack{ n \ge 0 \\ 1 \le j \le 2}}$ follows similarly as $b=0$ case. 

% For their linear independence, we use the qualitative estimates of $\Phi^{ext}_{j;b,\l,\nu}$ \eqref{eqPhiextnondeg} and crossing coefficients \eqref{eqdefacheck} to see there exists $C_1, C_2 \in \CC - \{ 0 \}$ such that 
% \bee
%   \left(\begin{array}{c}
%        \Phi^1_{1;b,\l,\nu}  \\
%        \Phi^2_{1;b,\l,\nu}
%   \end{array} \right) = C_1 
% \eee

% The estimates \eqref{eqnonadmosc}, \eqref{eqnonadmnondeg} of exotic branches $j = 3, 4$ follows directly from \eqref{eqPhi34est}, \eqref{eqPhi34est2}.

% Finally, we check the linear independence of $\{ \pa_\l^n\Phi_{j;b,\l,\nu} \}_{n \ge 0, 1 \le j \le 4}$. for $0 \le b \le b_1$. Similar to the $b=0$ case, this is reduced to the linear independence of $\{\Phi_j\}_{j = 1, 2, 3, 4}$, and further to that of $\{ \Phi_1, \Phi_2\}$ using the asymptotics and non-degeneracy,  \eqref{eqadmosc}, \eqref{eqnonadmosc}, \eqref{eqadmnondeg} and \eqref{eqnonadmnondeg}. This follows \eqref{eqdefacheck} and that $\Phi^{ext}_{1;b,\l,\nu}$, $\Phi^{ext}_{2;b,\l,\nu}$ are linear independent from \eqref{eqPhiextnondeg}.

\mbox{}

\noindent{\textit{Step 4.3. Asymptotics of boundary value at $x_*$.}}

% Firstly, the analyticity of map $r \mapsto \Phi_{j;b,\l,\nu}(r)$ for $j = 1, 2$ follows the analyticity of $\Phi^{ext}_{j;b,\l,\nu}(r)$ \eqref{eqIextpalest} (and so is its extension $\tilde \Phi_{j;b,\l,\nu}$), and the analyticity of coefficients $\kappa_{b, 1 \pm \l}^-$ \eqref{eqnormalcoeff} and $\check a_{jk;b,\l,\nu}$ \eqref{eqestpalchecka}. 

The analyticity of  $r \mapsto \Phi_{j;b,\l,\nu}(r)$ for $j = 1, 2$  was proven in Step 3.6. For the asymptotics of boundary values \eqref{eqbdrymapasymp}, we will only prove the estimate for $\vec \Phi_{1;b,\l,\nu}$ for simplicity, and that for $\vec \Phi_{2;b,\l,\nu}$ follows in almost the same way.

From \eqref{eqPhi12asympx*}, we have
\bee
&&\left|\pa_\l^n \left( \vec\Phi_{1;b,\l,\nu} (x_*)- \vec \Phi_{1;0,\l,\nu}(x_*) \right)\right| \\
&\le& \left| \pa_\l^n \left( \kappa_{b, 1+\l}^- \overrightarrow{\check \Phi^{mid}_{1;b,\l,\nu}} - \overrightarrow{\Phi_{1;0,\l,\nu}} \right) (x_*)\right| \\
&+& \sum_{k \in \{ 2, 3, 4\}} \left| \pa_\l^n \left( \kappa_{b, 1+\l}^- (\check c^{mid}_{11;b,\l,\nu})^{-1} \check c^{mid}_{1k;b,\l,\nu} \overrightarrow{\check \Phi^{mid}_{k;b,\l,\nu}}(x_*)  \right)\right|  =: \sum_{k=1}^4 J_{k;n}
\eee
It suffices to show for $0 \le n \le K_0$ that
\be 
 J_{k;n} \lesssim_n b^\frac 16 e^{-\Re \sqrt{1+\l}x_*} x_*^n,\quad {\rm for}\,\, 1\le k \le 4.  \label{eqJknest}
\ee
The case $k=1$ is included in \eqref{eqestconv3}. For $k = 2$, we use 
% \eqref{eqetaconv} to see
% \bee
%   e^{-2\eta_{b, E}(b^{-\frac 12})} = e^{-2\eta_{b, E}(x_*)} e^{-2\sqrt E(b^{-\frac 12}- x_*)} (1 + O(b^\frac 12)), \quad {\rm for}\,\, |E-1| \le 2\delta_1,
% \eee
% and hence from
the asymptotics of $\eta_{b, E}$ \eqref{eqetaconv}, definition of $\kappa^{\pm}_{b, E}$ \eqref{eqnormalcoeff}, and estimates \eqref{eqestpalcheckPhi2}, \eqref{eqestpalcheckc0} and \eqref{eqestpalcheckc}, we have 
\bee
  J_{2;n} &\lesssim_n& \sum_{m=0}^n \left| \pa_\l^m (\kappa_{b, 1+\l}^+ \overrightarrow{\check \Phi^{mid}_{k;b,\l,\nu}}(x_*)) \right| \cdot \left| \pa_\l^{n-m} \left(\frac{\kappa^-_{b, 1+\l} \check c^{mid}_{12;b,\l,\nu} }{\kappa^+_{b, 1+\l} \check c^{mid}_{11;b,\l,\nu}} \right) \right|\\
  &\lesssim_n& \sum_{m=0}^n b^{-m} \left| e^{\sqrt{1+\l}x_*} \cdot b^{-(n-m)} e^{-\frac{\pi(1+\l)}{b}} \cdot b^\frac 12 e^{\frac{\pi E_+}{b} - 2(\sqrt E_+  + O_\CC(b))b^{-\frac 12}} \right| \\
  &\lesssim_n& b^{\frac 12} e^{-\Re \sqrt {1+\l}x_*} \cdot \left(b^{-n} e^{-2\Re \sqrt{1+\l} (b^{-\frac 12} - x_*) } \right) \lesssim_n b^\frac 12 e^{-\Re \sqrt {1+\l}x_*}.
\eee
% That implies \eqref{eqJknest} for $k = 2$.
For $k = 3, 4$, we further recall $M, \tilde \epsilon$ from \eqref{eqdefMtildeep} and 
notice that $\left|\pa_\l^m \left(\frac{\kappa^\pm_{b, 1+\l}}{\kappa^\pm_{b, 1-\l}}\right)\right| \lesssim_m b^{-m} e^{\pm \frac{\pi \Re \l}{b}} = b^{-m} M^{\pm 1}$
% so with \eqref{eqetaconv}, \eqref{eqestpalcheckc0} and \eqref{eqestpalcheckc} and
 to compute for $0 \le m \le n$, 
\bee 
 \left| \pa_\l^{m} \left(\frac{\kappa^-_{b, 1+\l} \check c^{mid}_{13;b,\l,\nu} }{\kappa^-_{b, 1-\l} \check c^{mid}_{11;b,\l,\nu}} \right) \right|  &\lesssim_m& b^{-m} \tilde \epsilon \lesssim_{n} b^{\frac 16 + n-m} e^{-\Re \left(\sqrt{1+\l} - \sqrt{1-\l}\right) x_*} \\
 \left| \pa_\l^{m} \left(\frac{\kappa^-_{b, 1+\l} \check c^{mid}_{14;b,\l,\nu} }{\kappa^+_{b, 1-\l} \check c^{mid}_{11;b,\l,\nu}} \right) \right| 
 &\lesssim_m& b^{-m}e^{-\frac{2\pi}{b}} \cdot  M \tilde \epsilon e^{2\left(\frac{\pi\Re\bar E_-}{b} - \Re\sqrt{\bar E_-}(b^{-\frac 12}) \right)} \\
 &\lesssim_n& b^{\frac 16 + n-m} e^{-\Re\sqrt{1+\l}x_*}.
\eee
Combined with \eqref{eqestpalcheckPhi2}, a similar computation as for $k=2$ leads to $k = 3, 4$ of  \eqref{eqJknest}. That concludes the proof of \eqref{eqbdrymapasymp}.

% Applying \eqref{eqPhi12asympx*} and estimates \eqref{eqestpalcheckPhi2}, \eqref{eqestconv3}, \eqref{eqestpalcheckc} and \eqref{eqestpalchecka}, we compute for $0 \le n \le K_0$ that
% \bee
% \left|\pa_\l^n \left( \vec\Phi_{1;b,\l,\nu} (x_*)- \vec \Phi_{1;0,\l,\nu}(x_*) \right)\right| \\
% = \left| \pa_\l^n \left( \kappa_{b, 1+\l}^- \overrightarrow{\check \Phi^{mid}_{1;b,\l,\nu}} - \overrightarrow{\Phi_{1;0,\l,\nu}} \right) (x_*) 
% + \sum_{\substack{ k \in \{2, 4 \} \\ j = 1, 2  }} \pa_\l^n \left( \kappa_{b, 1+\l}^- \check a_{1j;b,\l,\nu} \check c^{mid}_{j2;b,\l,\nu} \overrightarrow{\check \Phi^{mid}_{k;b,\l,\nu}}(x_*)  \right)\right| \\
% \lesssim_n b^\frac 16 e^{-\Re \sqrt{1+\l}x_*} x_*^n + b^{-n+\frac 12} \left|\kappa_{b, 1+\l}^- (\kappa_{b, 1+\l}^+)^{-1}\right| e^{-2\Re \eta_{b, E_+}(b^{-\frac 12})} e^{\Re \sqrt{1+\l} x_*} \\
% \lesssim b^\frac 16 e^{-\Re \sqrt{1+\l}x_*} x_*^n.
% \eee
% where the notation $\vec\Phi_{j;b,\l,\nu}$ is as \eqref{eqvectorform}. Note that \eqref{eqnormalcoeff} and \eqref{eqetaconv} implies the control of residual  $ \left|\kappa_{b, 1+\l}^- (\kappa_{b, 1+\l}^+)^{-1}\right| e^{-2\Re \eta_{b, E_+}(b^{-\frac 12})} \sim e^{-2\Re \sqrt{1+\l} b^{-\frac 12}} \ll b^{n+1} e^{-2\Re \sqrt{1+\l} x_*}$.
% The estimate for $\Phi_{2;b,\l,\nu}$ and $\pa_r \Phi_{2;b,\l,\nu}$ follows similarly from \eqref{eqestpalcheckPhi4} and \eqref{eqestconv4}. These estimates concludes \eqref{eqbdrymapasymp}, and implies \eqref{eqbdrymapest} immediately. 

\end{proof}

\subsection{Admissible solutions for high spherical classes}

In this subsection, we construct the admissible solutions to \eqref{eqnu} in high spherical classes and identify their asymptotics. Notations from Definition \ref{defWKBappsoluh}, \ref{defWKBauxh}, and Lemma \ref{leminvtildeHh} will be used  without further citation. 

\begin{proposition}[Construction of exterior fundamental solution for high spherical classes]  
\label{propextfundh} 
For $d \ge 2$, there exists $s_{c;{\rm ext}}^{(2)} > 0$, $0 < b_3 \ll 1$, $0 < \delta_3 \ll 1$ such that for $0 < s_c \le s_{c;{\rm ext}}^{(2)}(d)$, $b = b(d, s_c)$ from Proposition \ref{propQbasymp} and $\l, \nu$ satisfying
\be 0 < b \le b_3,\quad \l \in \Omega_{\delta_3, \frac 12; b},\quad \nu \in \frac 12 \NN \cap [1, \infty), \label{eqrangeblnuh} \ee
    there exist four smooth functions $\Phi_{j;b, \l, \nu}: (0, \infty) \to \CC^2$ with $j = 1, 2$ solving the equation \eqref{eqnu} with parameters in the above range and the following statements hold.
    \begin{enumerate}
        \item Asymptotic behavior at infinity and linear independence: For $j = 1, 2$,  $\Phi_{j;b,\l,\nu}$ satisfies that for $N = 0, 1$
        \bea 
          \sup_{r \ge \frac {4|s_0|}b} \left[ \left| \pa_r^N ( e^{-\frac{ibr^2}{4}} \Phi^1_{j;b,\l,\nu}) \right| +  \left| \pa_r^N ( e^{\frac{ibr^2}{4}} \Phi^2_{j;b,\l,\nu}) \right|  \right] r^{\frac 12 -\frac{\Im \l}{b} - s_c+N} < \infty,
          \label{eqadmosch}
        \eea
        % while $\Phi_{3;b,\l,\nu}$, $\Phi_{4;b,\l,\nu}$ satisfy for $N \ge 0$ that 
        % \bea
        % \begin{split}
        %   \sup_{r \ge \frac 4b} \left[ \left| \pa_r^N ( e^{\frac{ibr^2}{4}} \Phi^1_{3;b,\l,\nu}) \right| +  \left|  \pa_r^N ( e^{\frac{3ibr^2}{4}} \Phi^2_{3;b,\l,\nu}) \right| \right] r^{\frac 12 +\frac{\Im \l}{b} - s_c+N} < \infty, \\
        %   \sup_{r \ge \frac 4b} \left[ \left| \pa_r^N ( e^{-\frac{3ibr^2}{4}} \Phi^1_{4;b,\l,\nu}) \right| +  \left|  \pa_r^N ( e^{-\frac{ibr^2}{4}} \Phi^2_{4;b,\l,\nu}) \right| \right] r^{\frac 12 +\frac{\Im \l}{b} - s_c+N} < \infty,
        %   \end{split}
        %   \label{eqnonadmosch}
        % \eea
        Moreover, they are non-degenerate in the sense that for $N = 0, 1$ and $ (c_1, c_2) \in \CC^2 - \{ \vec 0\}$, 
        \be
        \limsup_{r \to \infty}  \left[ \left|  \pa_r^N \sum_{j=1}^2  ( e^{-\frac{ibr^2}{4}} c_j \Phi^1_{j;b,\l,\nu}) \right| +  \left|  \pa_r^N \sum_{j=1}^2  ( e^{\frac{ibr^2}{4}} c_j \Phi^2_{j;b,\l,\nu}) \right|  \right]  r^{\frac 12 -\frac{\Im \l}{b} - s_c+N} > 0. 
          % |\Phi^1_{1;b,\l,\nu}| \sim_{b, E} |\psi_1^{b, 1+\l+ibs_c, \nu}|, \quad  |\Phi^2_{2;b,\l,\nu}| \sim_{b, E} |\psi_1^{b, 1-\bar \l+ibs_c, \nu}|, \quad r \ge \frac {4|s_0|}b 
          \label{eqadmnondegh}
          % \\
          % |\Phi^1_{3;b,\l,\nu}| \sim |\psi_3^{b, 1+\l+ibs_c}|, \quad  |\Phi^2_{4;b,\l,\nu}| \sim |\psi_3^{b, 1-\bar \l+ibs_c}|, \quad r \ge r_{0;b,\nu,\delta_0} \label{eqnonadmnondeg}
        \ee
        % with $r_{0;b,\nu,\delta_0} \ge \frac 4b$.
        % And $\{ \Phi_{j;b,\l,\nu} \}_{1 \le j \le 4}$ are linear independent functions. 
        \item Boundary value at $x_* = b^{-\frac 12}$: 
  \bea  
  \left| \begin{array}{l}
   \Phi_{1;b,\l,\nu}(x_*) =\left(  \vec e_1 + O_{\RR^2}(b^\frac 12 + \nu^{-1}) \right)   \psi_1^{b, 1+\l+ibs_c, \nu}(x_*), \\
  \pa_r \Phi_{1;b,\l,\nu}(x_*) =- \left( \sqrt{1 + \l + b\nu^2} \vec e_1 + O_{\RR^2}(b^\frac 12 + \nu^{-1}) \right)  \psi_1^{b, 1+\l+ibs_c, \nu}(x_*)
  \end{array}\right.  \label{eqinthighnu1h}
\\
\left| \begin{array}{l}
   \Phi_{2;b,\l,\nu}(x_*) = \left(\vec e_2 + O_{\RR^2}(b^\frac 12 + \nu^{-1}) \right)  \overline{\psi_1^{b, 1 - \bar \l+ibs_c, \nu}(x_*)}, \\
  \pa_r \Phi_{2;b,\l,\nu}(x_*) =-\left( \sqrt{1 - \l + b\nu^2}\vec e_2 + O_{\RR^2}(b^\frac 12 + \nu^{-1}) \right)  \overline{\psi_1^{b, 1 - \bar \l+ibs_c, \nu}(x_*)}
  \end{array}\right. \label{eqinthighnu2h}
\eea
    \end{enumerate}
\end{proposition}
\begin{remark}\label{rmkextfundh}
    Here we restrict our construction for $\Im \l \le \frac 12 b$, which is suffices for Theorem \ref{thmmodestabsmallspec}. It should be possible to extend to $\Im \l \le I_0 b$ for any fixed $I_0 > 0$ by constructing the exterior inversion operator like Lemma \ref{leminvtildeHext} exploiting the quadratic phase oscillation of $\psi_j^{b, E, \nu}$. 
\end{remark}

\begin{proof} The proof is similar to Step 2 of the proof of Proposition \ref{propextfund} using the approximate WKB solutions of high spherical classes and corresponding scalar inversion operators from Section \ref{sec42} to obtain uniform construction. We will use the vector notation \eqref{eqvectorform}.

Recalling $\delta_2$, $b_2(I_0)\big|_{I_0 = \frac 12}$ from Proposition \ref{propWKBh} and $s_c^{(0)'}$ from  Proposition \ref{propQbasympref}. Fixing $d \ge 2$, we similarly restrict our other parameters to satisfy
\be s_{c;{\rm ext}}^{(2)} \le s_c^{(0)'} \ll 1,\quad \delta_3 \le \min \left\{\frac 12 \delta_2, (100d)^{-1}\right\},\quad b_3 \le \min \left\{b_2(\frac 12), \delta_3^\frac 32\right\}. \label{eqdelta1restricth} \ee
We will choose $\delta_3$, $b_3$ smaller during the proof, and finally shrink $s_{c;{\rm ext}}^{(2)}$ so that when  $s_c \le s_{c;{\rm ext}}^{(2)}$, we have $b(d, s_c) \le b_3$ from Proposition \ref{propQbasymp}.  

To begin with, we rewrite \eqref{eqnu} as
    \be
\left| \begin{array}{l}
 \tilde H_{b, E_+, \nu} \phi = \left( h_{b, E_+, \nu}  - W_{1, b}\right) \phi - e^{i\frac{br^2}2} W_{2, b} \bar\varphi \\
 \tilde H_{b, \bar E_-, \nu} \varphi = \left( h_{b, \bar E_-, \nu} - W_{1, b}\right) \varphi - e^{i\frac{br^2}2} W_{2, b} \bar\phi
 \end{array}\right. \label{eqsystemphivarphih}
\ee
with the scalar operator $\tilde \calH_{b, E, \nu}$ \eqref{eqdeftildeHbEh},  $E_\pm = 1 \pm (\l + ibs_c)$ and $\Phi = \left( \begin{array}{c}
     \Phi^1 \\
     \Phi^2 
\end{array} \right) =: \left( \begin{array}{c}
     \phi \\
     \bar \varphi 
\end{array} \right)$. Since $bs_c \le \delta_3$ from \eqref{eqdelta1restricth}, we have
$$|E_\pm - 1| \le \frac 32 \delta_3. $$ 
Assuming $\Re \l \ge 0$ without loss of generality, we define 
\bee r_1 = r^*_{b, E_+, \nu}, \quad r_2 = r^*_{b, \bar E_-, \nu},  \quad
I_{ext} = \left[r_1, \infty \right),\quad I_{con} = [r_2, r_1],\quad I_{mid} = [b^{-\frac 12}, r_2],
\eee
where we used $r_1 \ge r_2$ from the monotonicity \eqref{eqr*monoh}. Besides, we denote $s_{0*}$, $s_{0;+}$, $s_{0; -}$ for $s_0$ \eqref{eqs0def1} related to parameters $(b, 1, \nu)$, $(b, E_+, \nu)$, $(b, \bar E_-, \nu)$ respectively, and similarly for $t_{\pm*}$, $t_{\pm;+}$ and $t_{\pm;-}$. In particular,
\[ s_{0*} = \left(\frac{1 + \sqrt{1 + b(\nu^2 - \frac 14)/4}}{2} \right)^\frac 12 > 1.\]

\mbox{}

\underline{Step 1. Preparatory estimates}

When $0 < s_c \le s_{c;{\rm ext}}^{(2)}$, $0 < b \le b_3$, $|\l| \le \delta_3$, $\Im \l \le \frac 12 b$ and $\nu \ge 1$ with $s_{c;{\rm ext}}^{(2)}, b_3, \delta_3$ satisfying \eqref{eqdelta1restricth}, we claim the following estimates:
\begin{enumerate}
    \item For potentials: with $E \in \{ E_+, \bar E_-\}$, 
    \be  |h_{b, E, \nu} - W_{1, b}| \lesssim r^{-2}, \quad |W_{2, b}| \lesssim e^{-\frac{p-1}{4}\min\{ rs_{0*}^{-1}, b^{-1} \} } \la r \ra^{-2}. \label{eqpotentialesth} \ee
    \item For turning points: 
    \be  r_1 - r_2 \le \frac{12 \delta_3}{b s_{0*}},\quad r_1, r_2 \in \left[ \frac{2s_{0*}}{b} - \frac{6\delta_3}{b s_{0*}} , \frac{2s_{0*}}{b} + \frac{6\delta_3}{b s_{0*}}  \right] \label{eqpositionr1r2} \ee
    \item For weight functions: there exists a universal constant $C_3 > 0$ and a constant $\frakc_{b, \l, \nu} \in \RR$ such that $|\frakc_{b, \l, \nu}| \le C_3 \delta_3$ and
    \begin{align}
       \frac 12  e^{-C_3\delta_3 \min\{rs_{0*}^{-1}, b^{-1}\} } \le \frac{\omega_{b, \bar E_-, \nu}^\pm}{\omega_{b, E_+, \nu}^\pm}e^{\mp \frac{\pi \frakc_{b, \l, \nu}}{2b}} 
  \le 2e^{C_3\delta_3 \min\{rs_{0*}^{-1}, b^{-1}\} },& \,\,\,\forall\, r\ge 0, \label{eqcomparisonomegah}\\
   \sup_{\substack{|r-\frac {2s_{0*}}b| \le 6\frac{\delta_3}{bs_{0*}}} } \{ \omega_{b, E, \nu}^\pm(r), (\omega_{b, E, \nu}^\pm(r))^{-1}, e^{\pm\Re \eta_{b, E, \nu}(r)} \} \le e^{\frac{C_3\delta_3}{b}}, &\,\,\,\forall\, |E-1| \le 2\delta_3. \label{eqcomparisonomega2h}
    \end{align}
% and the monotonicity 
% \be 
% {\rm sgn}\left( \pa_r \left( \Re \eta_{b, E_+, \nu} - \Re \eta_{b, \bar E_-, \nu}\right)\right) = {\rm sgn}(E_+ - \bar E_-) = {\rm sgn}(\Re \l) ,\quad r < \min\{ r_1, r_2\}.  \label{eqcomparisonmonoh}
% \ee
\end{enumerate}
As a corollary, we can choose $\delta_3$ as $\delta_3 = \min \left\{ \frac 12 \delta_2, (100d)^{-1}, (64C_3 d)^{-1} \right\}$ so that \eqref{eqcomparisonomegah}-\eqref{eqcomparisonomega2h} imply 
    \begin{align}
       \frac 12  e^{-\frac{p-1}{256} \min\{rs_{0*}^{-1}, b^{-1}\} } \le \frac{\omega_{b, \bar E_-, \nu}^\pm}{\omega_{b, E_+, \nu}^\pm}e^{\mp \frac{\pi \frakc_{b, \l, \nu}}{2b}} 
  \le 2e^{\frac{p-1}{256} \min\{rs_{0*}^{-1}, b^{-1}\} },& \,\,\,\forall\, r\ge 0, \label{eqcomparisonomega3h}\\
   \sup_{\substack{|r-\frac {2s_{0*}}b| \le 6\frac{\delta_3}{bs_{0*}}} } \{ \omega_{b, E, \nu}^\pm(r), (\omega_{b, E, \nu}^\pm(r))^{-1}, e^{\pm\Re \eta_{b, E, \nu}(r)} \} \le e^{\frac{p-1}{256b}}, &\,\,\,\forall\, |E-1| \le 2\delta_3, \label{eqcomparisonomega4h}
    \end{align}
    and 
    \be |\frakc_{b, \l, \nu}| \le \frac{p-1}{256}. \label{eqestfrakc} \ee

\mbox{}

Now we prove (1)-(3). Indeed, (1) follows directly from \eqref{eqbddhh}, Proposition \ref{propQbasymp} and $s_{0*}> 1$. For (2), we compute that $\sqrt E_+ s_{0;+} = 2^{-\frac 12}\left( E_+ + \sqrt{E_+ + b^2 (\nu^2 - \frac 14)/4} \right)^\frac 12$ and hence 
\be |\sqrt E_+ s_{0;+} - s_{0*}| \le |E_+ - 1| s_{0*}^{-1}.\label{eqestE+-1} \ee
From this and similar estimate for $\sqrt{\bar E_-} s_{0;-}$, the location estimate \eqref{eqr*bE1h} yields \eqref{eqpositionr1r2}. 

Finally, for (3), we first define $\frakc$ as 
\bee
  &&\frakc_{b, \l, \nu} :=-\frac{2b}{\pi} \Re \left( \lim_{r \to 0}(\eta_{b, E_+, \nu}(r) - \eta_{b, \bar E_-, \nu}(r)) \right) \\
  &=& -\Im \lim_{r \to 0} \int_{s_{0*} - 2\delta_3 s_{0*}^{-1}}^{\frac{br}{2}} \left[ \left(w^2 - E_+ - \frac{b^2(4\nu^2 -1)}{16w^2} \right)^\frac 12 - \left(w^2 - \bar E_- - \frac{b^2(4\nu^2 -1)}{16w^2} \right)^\frac 12 \right] dw \\
  &-& \Im \int_{s_{0*} - 2\delta_3 s_{0*}^{-1}}^{\sqrt E_+ s_{0;+}} \left(w^2 - E_+ - \frac{b^2(4\nu^2 -1)}{16w^2} \right)^\frac 12 + \Im \int_{s_{0*} - 2\delta_3 s_{0*}^{-1}}^{\sqrt{\bar E_-} s_{0;-}} \left(w^2 - \bar E_- - \frac{b^2(4\nu^2 -1)}{16w^2} \right)^\frac 12 \\
  &=:& I_1 + I_2 + I_3
\eee
where we used the formula \eqref{eqwquadtilde}. To show the well-definedness and the uniform bound, it suffices to control each integral by $O(\delta_3)$. For $I_2$, it follows from
$w^4 - Ew^2 - \frac{b^2(4\nu^2 -1)}{16}  = (w - \sqrt E s_0)(w + \sqrt E s_0) (w^2 - E t_-)$ and \eqref{eqestE+-1} that
\[ |I_2| \lesssim s_{0*}^\frac 12 \int_{s_{0*} - 2\delta_3 s_{0*}^{-1}}^{\sqrt E_+ s_{0;+}} |w - \sqrt E_+ s_{0;+}|^\frac 12 \lesssim \delta_3^{\frac 32} s_{0*}^{-1}, \]
and the third integral $I_3$ is bounded similarly. For the first integral $I_1$, notice that,
\bea
  &&\left| \left(w^2 - E_+ - \frac{b^2(4\nu^2 -1)}{16w^2} \right)^\frac 12 + \left(w^2 - \bar E_- - \frac{b^2(4\nu^2 -1)}{16w^2} \right)^\frac 12 \right| \nonumber \\
  &\ge& \left| E_+ + \frac{b^2(4\nu^2 -1)}{16w^2} - w^2 \right|^\frac 12 
  = \frac{|w^2 - E_+ t_{-;+}|^{\frac 12} }{w}  \cdot |w - \sqrt E_+ s_{0; +}|^{\frac 12} \cdot  |w + \sqrt E_+ s_{0; +}|^{\frac 12} \nonumber\\
  &\gtrsim& s_{0*}^\frac 12 \cdot |w - (s_{0*} - 2\delta_3 s_{0*}^{-1})|^{\frac 12}, \quad {\rm for} \,\, w \in [0, s_{0*} - 2\delta_3 s_{0*}^{-1}], \label{eqestnumeratorcomp}
\eea
where we exploited $\Re \l \ge 0$ for the first inequality and \eqref{eqestE+-1} for the last. So using the cancellation, we have 
\bee
 I_1 =  -\Im \int_{s_{0*} - 2\delta_3 s_{0*}^{-1}}^{0} \frac{\bar E_- - E_+}{ \left(w^2 - E_+ - \frac{b^2(4\nu^2 -1)}{16w^2} \right)^\frac 12 + \left(w^2 - \bar E_- - \frac{b^2(4\nu^2 -1)}{16w^2} \right)^\frac 12} dw = O(\delta_3).
\eee 

Next, the pointwise estimate \eqref{eqestnumeratorcomp} also leads to 
\bee
   \left|\Re \eta_{b, E_+, \nu}(r) - \Re \eta_{b, \bar E_-, \nu}(r) + \frac{\pi \frakc_{b, \l, \nu}}{2b}\right| \lesssim \delta_3 \min\{ rs_{0*}^{-1}, b^{-1}\}, \quad r \in \left[0, \frac{2s_{0*}}{b} - \frac{4\delta_3}{bs_{0*}}\right]
\eee
by estimating integration of the derivative from $0$ to $r$. Similar to Step 2.1 of proof of Proposition \ref{propextfund}, this inequality holds also for $r \ge \frac{2s_{0*}}{b} - \frac{4\delta_3}{bs_{0*}}$ using \eqref{eqetaabsh}, \eqref{eqetaReh} plus monotonicity \eqref{eqReetamonoh}, and leads to \eqref{eqcomparisonomegah}-\eqref{eqcomparisonomega2h}. 
% The monotonicity \eqref{eqcomparisonmonoh} also follows evaluating the sign of derivative likewise. 

\mbox{}

\underline{Step 2. Constructing fundamental solutions on each region.} 

This step resembles Step 2.2-2.4 of the proof of Proposition \ref{propextfund}. We construct $2 + 4 + 4$ solutions of \eqref{eqsystemphivarphih} on $I_{ext}$, $I_{con}$ and $I_{mid}$ respectively.

Specifically, we invert the scalar operator using Lemma \ref{leminvtildeHh} and plug in approximate WKB solutions $\psi_j^{b, E, \nu}$ from its kernel as source terms, leading to the following system 
\be
   \left( \begin{array}{c} 
  \phi^{\Box}_j \\ \varphi^{\Box}_j
  \end{array} \right) = 
  \left( \begin{array}{c} 
  S^{\Box}_j \\ R^{\Box}_j
  \end{array} \right) 
  + 
\calT^\Box_j
   \left( \begin{array}{c} 
  \phi^{\Box}_j \\ \varphi^{\Box}_j
  \end{array} \right), \quad \calT^\Box_j := \left( \begin{array}{cc} 
  \calA^\Box_j & \\ & \calB^\Box_j
  \end{array} \right)
   \circ \calV, \label{eqIh}
\ee
where the parameters are ranged from 
\[ \Box \in \{ext, con, mid\},\quad j \in \left|\begin{array}{ll}\{ 1, 2\} & \Box = ext, \\ \{ 1, 2, 3, 4\} & \Box \in \{con, mid\}, \end{array}\right.\]
and the potential operator is
\[ \calV   \left( \begin{array}{c} 
  \phi \\ \varphi
  \end{array} \right) :=  \left( \begin{array}{c} 
  \left( h_{b, E_+, \nu}  - W_{1, b}\right) \phi - e^{i\frac{br^2}2} W_{2, b} \bar\varphi \\ 
  \left( h_{b, \bar E_-, \nu} - W_{1, b}\right) \varphi - e^{i\frac{br^2}2} W_{2, b} \bar\phi
  \end{array} \right). \]
  The scalar inversion operators are 
  \bee 
\calA^\Box_j = \left|\begin{array}{ll}
\tilde T^{ext}_{\infty; b, E_+, \nu} & \Box = ext;\\
\tilde T^{mid, G}_{r_2, r_1; b, E_+, \nu}     & \Box = con; \\
\tilde T^{mid, G}_{b^{-\frac 12}, r_2; b, E_+, \nu}  & \Box = mid, \,\,j \neq 3;\\
\tilde T^{mid, D}_{b^{-\frac 12}, r_2; b, E_+, \nu}  & \Box = mid, \,\,j = 3;
\end{array} 
\right. \,\,
\calB^\Box_j = \left|\begin{array}{ll}
\tilde T^{ext}_{\infty; b, \bar E_-, \nu} & \Box = ext;\\
\tilde T^{ext}_{r_2; b, \bar E_-, \nu}\circ \mathbbm{1}_{I_{con}} & \Box = con; \\
\tilde T^{mid, G}_{b^{-\frac 12}, r_2; b, \bar E_-, \nu}  & \Box = mid, \,\,j \neq 1;\\
\tilde T^{mid, D}_{b^{-\frac 12}, r_2; b, \bar E_-, \nu}  & \Box = mid, \,\,j = 1;
\end{array} 
\right. 
  \eee
  and the potentials are 
  \bee
S^\Box_j = \left|\begin{array}{ll}
\psi_1^{b, E_+, \nu} & (\Box, j) = (ext, 1);\\
\psi_4^{b, E_+, \nu} & (\Box, j) \in \{(con, 1), (mid, 1) \}; \\
\psi_2^{b, E_+, \nu} & (\Box, j) \in \{ (con, 2), (mid, 2)\}; \\
   0  & otherwise;
\end{array} 
\right. \,\,
R^\Box_j = \left|\begin{array}{ll}
\psi_1^{b, \bar E_-, \nu} & (\Box, j) \in \{ (ext, 2), (con, 3) \};\\
\psi_3^{b, \bar E_-, \nu} & (\Box, j) = (con, 4); \\
\psi_4^{b, \bar E_-, \nu} & (\Box, j) =  (mid, 3);  \\
\psi_2^{b, \bar E_-, \nu} & (\Box, j) = (mid, 4); \\
   0  & otherwise.
\end{array} 
\right.
  \eee
We claim the following bounds for the linear operators in \eqref{eqIh}: 
\begin{itemize}
    \item $\Box \in \{ext, con\}$: $\| \calT^{\Box}_j \|_{\XX^{\Box}_{j}} \lesssim bs_{0*}^{-2}$, 
    where 
       \[  
    \XX^{\Box}_j = \left| \begin{array}{ll}
        C^0_{\omega_{b, E_+, \nu}^-}(I_{\Box}) \times e^{\frac{p-1}{64b}} (b^{2}|s_{0*}|^{-2}) C^0_{\omega_{b, \bar E_-, \nu}^- r^{-2}}(I_{\Box})  & (\Box, j) \in \{ (ext, 1), (con, 1) \}; \\
        C^0_{\omega_{b, E_+, \nu}^+}(I_{\Box}) \times e^{\frac{p-1}{64b}} C^0_{\omega_{b, \bar E_-, \nu}^+}(I_{\Box})  & (\Box, j) = (con, 2); \\
        e^{\frac{p-1}{64b}} (b^{2}|s_{0*}|^{-2}) C^0_{\omega_{b, E_+, \nu}^- r^{-2}}(I_{\Box}) \times  C^0_{\omega_{b, \bar E_-, \nu}^-}(I_{\Box})  & (\Box, j) \in \{ (ext, 2), (con, 3) \}; \\
        e^{\frac{p-1}{64b}} C^0_{\omega_{b, E_+, \nu}^+}(I_{\Box}) \times C^0_{\omega_{b, \bar E_-, \nu}^+}(I_{\Box})  & (\Box, j) = (con, 4); 
    \end{array}\right.
    \]
    We remark that $b^{2}|s_{0*}|^{-2}\| f \|_{C^0_{\omega_{b, \bar E_-, \nu}^- r^{-2}} (I_{con}) }  \sim \| f \|_{C^0_{\omega_{b, \bar E_-, \nu}^-} (I_{con}) } $ so that \eqref{eqtildeHextesth2} is applicable.
    \item $\Box = mid$: $\| \calT^{mid}_j \|_{\XX^{mid}_j} \lesssim b^\frac 12 s_{0*}^{-1}$, where 
    \[  
    \XX^{mid}_j = \left| \begin{array}{ll}
        C^0_{\omega_{b, E_+, \nu}^-}(I_{mid}) \times e^{-\frac{\pi \frakc_{b, \l, \nu}}{2b}} C^0_{\omega_{b, \bar E_-, \nu}^- e^{-\frac{p-1}{16s_{0*}}r} }(I_{mid})  & j = 1; \\
        C^0_{\omega_{b, E_+, \nu}^+}(I_{mid}) \times e^{\frac{\pi \frakc_{b, \l, \nu}}{2b}} C^0_{\omega_{b, \bar E_-, \nu}^+ e^{-\frac{p-1}{16s_{0*}}r} }(I_{mid})  & j = 2; \\
         e^{\frac{\pi \frakc_{b, \l, \nu}}{2b}}  C^0_{\omega_{b, E_+, \nu}^- e^{-\frac{p-1}{16s_{0*}}r} }(I_{mid}) \times C^0_{\omega_{b, \bar E_-, \nu}^-}(I_{mid})  & j = 3; \\
         e^{-\frac{\pi \frakc_{b, \l, \nu}}{2b}}  C^0_{\omega_{b, E_+, \nu}^+ e^{-\frac{p-1}{16s_{0*}}r} }(I_{mid}) \times C^0_{\omega_{b, \bar E_-, \nu}^+}(I_{mid}) & j = 4.\\
    \end{array}\right.
    \]
\end{itemize}
Then when $b_3 \ll 1$, \eqref{eqIh} is a linear contraction and generates a unique solution for each case.
Indeed, these estimates can be easily verified using \eqref{eqtildeHextesth1} for $\Box = ext$, \eqref{eqtildeHextesth2}-\eqref{eqtildeTmidGesth} for $\Box = con$ and \eqref{eqtildeTmidGesth}, \eqref{eqtildeTmidGesth2}, \eqref{eqtildeTmidDesth} for $\Box = mid$. Here the additional $r^{-2}$ decay is provided by \eqref{eqpotentialesth}. The exponential decay of the coupling potential $W_{2, b}$ provides three smallness: (1) covering the quotient between $\omega_{b, E_+, \nu}^\pm$ and $\omega_{b, \bar E_-, \nu}^\pm$ thanks to \eqref{eqcomparisonomega3h}, \eqref{eqcomparisonomega4h} and \eqref{eqestfrakc}; (2) on $I_{ext}$ and $I_{con}$, the coefficient $e^{-\frac{p-1}{64b}}$ to the component without source term; (3) on $I_{mid}$, the exponential decaying weight for the component without source term, and moreover, the additional $e^{-\frac{\a r}{2|s_0|}}$ decay when applying \eqref{eqtildeTmidGesth2}. 

In particular, due to boundedness of $\tilde T^{ext}_{r_1;b, E, \nu}$ on $\dot C^1_{\omega_{b, E, \nu}^- (br)^{-3}} ([\frac{4|s_0|}{b},\infty))$, we see that $\Phi^{ext}_1$ and $\Phi^{ext}_2$ satisfy the boundedness \eqref{eqadmosch} and non-degeneracy \eqref{eqadmnondegh} because $\varphi^{ext}_2$ and $\phi^{ext}_1$ have additional $r^{-2}$ decay. 

\mbox{}

Next, we identify the boundary values for these solutions. We denote  $\Phi^{\Box}_j = \left( \phi^{\Box}_j, \overline{\varphi^{\Box}_j} \right)^\top$ and recall the vector notation \eqref{eqvectorform}. For notational simplicity, we also denote 
\be \tb = b |s_{0*}|^{-2},\quad \epsilon = e^{-\frac{p-1}{64b}}. \ee
Then Lemma \ref{leminvtildeHh} (2) immediately implies the following: 
\begin{itemize}
    \item For $\vec \Phi^{ext}_j(r_1)$: Under the form \eqref{eqbdryext}, the coefficients are
    \be 
 (\gamma^{ext}_{jk})_{\substack{1 \le j \le 2\\ 1 \le k \le 4}} = \tb \left( \begin{array}{cccc}
        O(1) & O(1) & O(\epsilon) & O(\epsilon e^{-2\eta_{b, \bar E_-, \nu}(r_1)} ) \\
        O(\epsilon) & O(\epsilon) &   O(1) & O(e^{-2\eta_{b, \bar E_-, \nu}(r_1)})
  \end{array}\right) \label{eqestgammaexth} 
    \ee
%     \be \gamma^{ext}_{jk} 
% = \left| \begin{array}{ll} O(b|s_{0*}|^{-2}), & j \in\{ 1, 2\},\, k = 2;\\
%  O(b|s_{0*}|^{-2}e^{-2\eta_{b, \bar E_-, \nu}(r_1)} ), & j \in\{ 1, 2\},\,k = 4;\\
%  0 & j \in\{ 1, 2\},\, k \in \{1, 3\}.
% \end{array}\right.
% \label{eqestgammaexth} 
% \ee
    \item For $\vec \Phi^{con}_j(r_1)$ and $\vec \Phi^{con}_j(r_2)$: Under the form \eqref{eqbdrycon}, \eqref{eqARAL} and \eqref{eqGammamatrixcon}, the coefficients are 
     \be 
 (\gamma^{con}_{jk})_{\substack{1 \le j \le 4\\ 1 \le k \le 4}} = \tb \left( \begin{array}{cccc}
        O(1) & O(e^{-2\eta_{b, E_+, \nu}(r_2)}) & O(\epsilon) & O(\epsilon) \\
        O(1) &  O(1) & O(\epsilon) & O(\epsilon) \\
      O(\epsilon) & O(\epsilon) & O(1) &  O(1) \\
      O(\epsilon) & O(\epsilon) & O(e^{2\eta_{b, \bar E_-, \nu}(r_1)}) & O(1)
  \end{array}\right) \label{eqestgammaconh} 
    \ee
%     \be
%     \gamma^{con}_{jk} 
% = \left| \begin{array}{ll} 
%  O(b|s_{0*}|^{-2}e^{-2\eta_{b, E_+, \nu}(r_2)}), & (j, k) = (1, 2);\\
%  O(b|s_{0*}|^{-2}e^{2\eta_{b, \bar E_-, \nu}(r_1)} ), & (j, k) = (4, 3);\\
%  O(b|s_{0*}|^{-2}), & otheriwise.
% \end{array}\right.
% \label{eqestgammaconh} 
    % \ee
    \item For $\vec \Phi^{mid}_j(r_2)$: Under the form \eqref{eqbdrymid} and \eqref{eqGammamatrixmid}, the coefficients are 
    \be 
 \Gamma^{mid, R} = \tb \left( \begin{array}{cccc}
        O(\tb^\frac 12) & 0 & 0 & 0 \\
        O(\tb e^{2\eta_{b, E_+, \nu}(r_2)})  & 0 &  O(\epsilon \tb e^{2\eta_{b, E_+, \nu}(r_2)}) & 0 \\
      0 & 0 & O(\tb^\frac 12) & 0 \\
      O(\epsilon \tb) & 0  & O(\tb) & 0
  \end{array}\right) \label{eqestgammamidRh}  
    \ee
%      \be
%     \gamma^{mid}_{jk} 
% = \left| \begin{array}{ll} 
%  O(b^\frac 12|s_{0*}|^{-1}), & (j, k) \in \{ (1, 1), (3, 3) \};\\
%  O(b|s_{0*}|^{-2}e^{2\eta_{b, E_+, \nu}(r_2)} ), & (j, k) \in \{ (2, 1),  (2, 3) \};\\
%  O(b|s_{0*}|^{-2}), &  (j, k) \in \{ (4, 1),  (4, 3) \}.
% \end{array}\right.
% \label{eqestgammamidRh} 
%     \ee
    Here for $(j, k) \in \{(2,3), (4, 1)\}$, we used $e^{-\frac{(p-1)r_2}{32|s_{0*}|}} \le e^{\pm \frac{\pi \frakc_{b, \l, \nu}}{2b}}e^{2 \eta_{b, E_+, \nu}(r_2)} \epsilon$ due to \eqref{eqpositionr1r2}, \eqref{eqcomparisonomega3h} and \eqref{eqcomparisonomega4h}. 
    % {\color{red} Additional $e^{-\frac{(p-1)r_2}{32|s_{0*}|}}$ smallness for $\gamma^{mid}_{23}$ and $\gamma^{mid}_{41}$. 
    % }
    \item For $\vec \Phi^{mid}_j(b^{-\frac 12})$: we have
    \be
 \left( \begin{array}{c}
       \vec \Phi^{mid}_1 (b^{-\frac 12})\\
       \vec \Phi^{mid}_2 (b^{-\frac 12})\\
       \vec \Phi^{mid}_3 (b^{-\frac 12})\\
       \vec \Phi^{mid}_4 (b^{-\frac 12})
  \end{array}\right) = 
(I + \Gamma^{mid, L})
  \left( \begin{array}{c}
       \vec\psi_4^{b, E_+, \nu}(b^{-\frac 12}) \otimes \vec 0 \\
       \vec\psi_2^{b, E_+, \nu}(b^{-\frac 12}) \otimes \vec 0 \\
       \vec 0 \otimes \overline{\vec\psi_4^{b, \bar E_-, \nu}}(b^{-\frac 12})  \\
       \vec 0 \otimes \overline{\vec\psi_2^{b, \bar E_-, \nu}}(b^{-\frac 12})
  \end{array}\right) \label{eqbdrymidh}
  \ee
  where 
  \be \small \Gamma^{mid, L} 
  %   =  \left( \begin{array}{cccc}
  %      0 & \gamma^{mid}_{12} & \gamma^{mid}_{13} & \gamma^{mid}_{14} \\
  %      0 & \gamma^{mid}_{22} & 0 & \gamma^{mid}_{24} \\
  %      \gamma^{mid}_{31} & \gamma^{mid}_{32} & 0 & \gamma^{mid}_{34}  \\
  %      0 & \gamma^{mid}_{42} & 0 & \gamma^{mid}_{44}
  % \end{array}\right) 
  = \tb^\frac 12
   \left( \begin{array}{cccc}
       0 & O(e^{-2\eta_{b, E_+, \nu}(b^{-\frac 12}) }) & O(M \tilde \epsilon^2) & O(M e^{-2\eta_{b, \bar E_-, \nu}(b^{-\frac 12} )} \tilde \epsilon^2
       ) \\
       0 & O(1) & 0 & O(M^{-1} \tilde \epsilon)  \\
       O(M^{-1}\tilde \epsilon^2) & O(M^{-1} e^{-2\eta_{b, E_+, \nu}(b^{-\frac 12}) }\tilde \epsilon^2 ) & 0 & O(e^{-2\eta_{b, \bar E_-, \nu}(b^{-\frac 12} )})  \\
       0 & O(M\tilde \epsilon) & 0 & O(1)
  \end{array}\right)
  \label{eqestgammamidLh} \ee
  Here we denoted
  \[ M = e^{\frac{\pi \frakc_{b,\l,\nu}}{2b}},\quad \tilde \epsilon = e^{-\frac{p-1}{32 \sqrt{\tb}}}. \]
\end{itemize}

\mbox{}

\underline{Step 3. Linear matching and concluding the proof.}

Now we do similar computation as in Step 2.6 of the proof of Proposition \ref{propextfund} to match these solutions and construct $\Phi_{j;b,\l,\nu}$ for $j = 1, 2$. 

Let
\be 
 \Phi_{k } = \left| \begin{array}{ll}
     \Phi^{ext}_{k } & r \in I_{ext} \\
     \sum_{j=1}^4 c^{con}_{kj } \Phi^{con}_{j } & r \in I_{con} \\
     \sum_{j=1}^4 c^{mid}_{kj } \Phi^{mid}_{j }& r \in I_{mid} 
 \end{array}\right.\quad k = 1, 2, \label{eqtildePhikh}
\ee
with the asymptotics matching at $r_1, r_2$:
\bee
  \vec{{\Phi}}_{k }(r_1+0) = \vec{{\Phi}}_{k }(r_1-0),\quad \vec{{\Phi}}_{k }(r_2+0) = \vec{{\Phi}}_{k }(r_2-0).
\eee
Then the boundedness \eqref{eqadmosch} and non-degeneracy \eqref{eqadmnondegh} follow from those of $\Phi^{ext}_k$ proven in Step 2. 

Now we evaluate the matching coefficients and boundary value to verify \eqref{eqinthighnu1h} and \eqref{eqinthighnu2h}. 
Similar computation as in Step 2.6 of the proof of Proposition \ref{propextfund} yields 
\bee \small
   \left(\begin{array}{cc}
    c^{con}_{11 } & c^{con}_{21 } \\
    c^{con}_{12 } & c^{con}_{22 }  \\
    c^{con}_{13 } & c^{con}_{23 } \\
    c^{con}_{14 } & c^{con}_{24 }
  \end{array}\right) 
 &=& \left( \begin{array}{cc}
       e^{-\frac{\pi i}{6}} + O(\tb) & O(\epsilon \tb) \\
       O(\tb) & O(\epsilon \tb) \\
       O(\epsilon \tb) & 1 + O(\tb) \\
       O(\epsilon \tb ) & O(\tb) 
  \end{array}\right) \\
  \small \left(\begin{array}{cc}
    c^{mid}_{11 } & c^{mid}_{21 } \\
    c^{mid}_{12 } & c^{mid}_{22 }  \\
    c^{mid}_{13 } & c^{mid}_{23 } \\
    c^{mid}_{14 } & c^{mid}_{24 }
  \end{array}\right) 
 &=& \left( \begin{array}{cc}
       e^{-\frac{\pi i}{6}} + O(\tb^\frac 12) & O(\epsilon \tb) \\
       O(\tb e^{-2\eta_{b, E_+, \nu}(r_2)}) & O(\epsilon \tb e^{-2\eta_{b, E_+, \nu}(r_2)}) \\
       O(\epsilon \tb) & e^{\frac{\pi i}{6}} + O(\tb^\frac 12)  \\
       O(\epsilon \tb) & O(\tb) 
  \end{array}\right)
\eee
Hence we can evaluate the boundary value using \eqref{eqbdrymidh} and \eqref{eqestgammamidLh} that
\be
 \small \left( \vec{{\Phi}}_1(b^{-\frac 12}), \vec{{\Phi}}_2(b^{-\frac 12}) \right) =  \calC  \left( \begin{array}{c}
       \vec\psi_4^{b, E_+, \nu}(b^{-\frac 12}) \otimes \vec 0 \\
       \vec\psi_2^{b, E_+, \nu}(b^{-\frac 12}) \otimes \vec 0 \\
       \vec 0 \otimes \overline{\vec\psi_4^{b, \bar E_-, \nu}}(b^{-\frac 12})  \\
       \vec 0 \otimes \overline{\vec\psi_2^{b, \bar E_-, \nu}}(b^{-\frac 12})
  \end{array}\right), \label{eqbdrymidLh}
  \ee
  with the matrix coefficient as 
  \begin{align*}
   \small &\calC = \left(\begin{array}{cccc}
   c^{mid}_{11 } & c^{mid}_{12 } & c^{mid}_{13 } & c^{mid}_{14 } \\
   c^{mid}_{21 } & c^{mid}_{22 } & c^{mid}_{23 } & c^{mid}_{24 } 
  \end{array}\right) (I + \Gamma^{mid, L}) \\
   \small= & \left(\begin{array}{cccc}
    e^{-\frac{\pi i}{6}} + O(\tb^\frac 12) & O(\tb^\frac 12 e^{-2\eta_{b, E_+, \nu}(b^{-\frac 12}) }) & O(\tb^\frac 12 M \tilde \epsilon^2) & O(\tb^\frac 12 M \tilde \epsilon^2 e^{-2\eta_{b, \bar E_-, \nu}(b^{-\frac 12}) }) \\
    O(\tb^\frac 12 M^{-1} \tilde \epsilon^2) & O(\tb^\frac 12 M^{-1} \tilde \epsilon^2 e^{-2\eta_{b, E_+, \nu}(b^{-\frac 12}) }) & e^{\frac{\pi i}{6}} + O(\tb^\frac 12) & O(\tb^\frac 12 e^{-2\eta_{b, \bar E_-, \nu}(b^{-\frac 12}) })
  \end{array}\right)
\end{align*}
where we used that when $b \le b_3 \ll 1$, we have $\epsilon \ll \min \{ M, M^{-1}\} \tilde \epsilon^2$ from \eqref{eqestfrakc} and $e^{-2\eta_{b, E, \nu}(b^{-\frac 12})} = e^{2\int_{b^{-\frac 12}}^{r^*_{b, E, \nu}} \pa_r \Re \eta_{b, E, \nu}(\tau)d\tau} \ge e^{-(2\tb)^{-1}} \gg \epsilon^{-1}$ for $E \in \{ E_+, \bar E_-\}$ using \eqref{eqetaRehlowerbdd}.

Finally, we notice that \eqref{eqcomparisonomega3h} implying $M^{-1} \tilde \epsilon^2 \omega_{b, E_+, \nu}^- \lesssim \omega_{b, \bar E_-, \nu}^-$ and $M \tilde \epsilon^2 \omega_{b, \bar E_-, \nu}^- \lesssim \omega_{b, E_+, \nu}^-$, and that $\sqrt E (1 - s^2 + \a s^{-2})^{\frac 12}\big|_{r = b^{-\frac 12}} = (E + b\nu^2 + O(b))^\frac 12$. Plugging the asymptotics of $\vec \psi_4^{b, E, \nu}$ and $\vec \psi_2^{b, E, \nu}$ from Proposition \ref{propWKBh} (4) into \eqref{eqbdrymidLh} yields \eqref{eqinthighnu1h} and \eqref{eqinthighnu2h}. 
\end{proof}

\subsection{Inadmissible solutions}

To show the admissibility of fundamental solutions in Proposition \ref{propextfund} and Proposition \ref{propextfundh}, we need the existence of the other two branches of fundamental solutions. Since less information of asymptotics is required, we apply the same strategy in Proposition \ref{propextfund} for all spherical classes and all $|\l| \ll 1$.

\begin{proposition}[Construction of exterior inadmissible solution] 
\label{propextfundin}
 For $d \ge 1$, $I_0 > 0$, let $\delta_0, b_0(I_0)$ from proposition \ref{propWKB}. There exists $s_{c;{\rm ext}}^{(3)}(d, I_0) > 0$ such that for $0 < s_c \le s_{c;{\rm ext}}^{(3)}(d, I_0)$, $b = b(d, s_c)$ from Proposition \ref{propQbasymp} and $\l, \nu$ satisfying
    \[ 0 \le b \le b_0,\quad  \l \in \Omega_{\delta_0, I_0; b}, \quad \nu \in \frac 12 \NN_{\ge 0}, \]
    there exist smooth functions $\Phi_{j;b, \l, \nu}: (0, \infty) \to \CC^2$ with $j = 3, 4$ solving the equation \eqref{eqnu} with parameters in the above range. They satisfy the following properties.
    \begin{enumerate}
    \item Asymptotic behavior at infinity for $b = 0$: For $k \ge 0$, $N = 0, 1$ and $j = 3, 4$, 
    \be
      \sup_{r \ge 1} \left[\left| e^{-\sqrt{1+\l}r}  \pa_r^N \Phi^1_{j;0,\l,\nu}  \right| + \left|  e^{-\sqrt{1-\l}r} \pa_r^N \Phi^2_{j;0,\l,\nu}  \right|\right] < \infty, \label{eqasympb0in}
    \ee
    with non-degeneracy for $N = 0, 1$ and $(c_3, c_4) \in \CC^2 - \{ \vec 0\}$,
    \be
 \limsup_{r \to \infty}  \left[ \left| e^{-\sqrt{1+\l}r} \pa_r^N \sum_{j=3}^4  ( c_j \Phi^1_{j;0,\l,\nu} )\right| +  \left| e^{-\sqrt{1-\l}r} \pa_r^N \sum_{j=3}^4 ( c_j \Phi^2_{j;b,\l,\nu}) \right|  \right]  > 0. 
  \label{eqnondegb0in}
\ee
%     \be
%   |\Phi^1_{3;0,\l,\nu}| \sim |e^{\sqrt{1+\l}r}|, 
%   \quad 
%   |\Phi^2_{4;0,\l,\nu}| \sim |e^{\sqrt{1-\l}r}|,\quad r \ge r_{0;0,\nu,\delta_0} 
%   \label{eqnondegb0in}
% \ee
% where $r_{0;0,\nu,\delta_0} \ge 1$. 
% {\color{red} In particular, $\{ \pa_\l^n\Phi_{j;0,\l,\nu} \}_{n \ge 0, 1 \le j \le 4}$ are linear independent functions. }
        \item Asymptotic behavior at infinity and linear independence for $ b > 0$: For $b>0$ and $N \ge 0$,
        \bea
        \begin{split}
          \sup_{r \ge \frac 4b} \left[ \left| \pa_r^N ( e^{\frac{ibr^2}{4}} \Phi^1_{3;b,\l,\nu}) \right| +  \left|  \pa_r^N ( e^{\frac{3ibr^2}{4}} \Phi^2_{3;b,\l,\nu}) \right| \right] r^{\frac 12 +\frac{\Im \l}{b} + s_c+N} < \infty, \\
          \sup_{r \ge \frac 4b} \left[ \left| \pa_r^N ( e^{-\frac{3ibr^2}{4}} \Phi^1_{4;b,\l,\nu}) \right| +  \left|  \pa_r^N ( e^{-\frac{ibr^2}{4}} \Phi^2_{4;b,\l,\nu}) \right| \right] r^{\frac 12 +\frac{\Im \l}{b} + s_c+N} < \infty,
          \end{split}
          \label{eqnonadmosc}
        \eea
        with non-degeneracy for $N \ge 0$ and  $(c_3, c_4) \in \CC^2 - \{ \vec 0\}$,
         \be
        \liminf_{r \to \infty}  \left[ \left|  \pa_r^N \sum_{j=3}^4  ( e^{-\frac{ibr^2}{4}} c_j \Phi^1_{j;b,\l,\nu}) \right| +  \left|  \pa_r^N \sum_{j=3}^4  ( e^{\frac{ibr^2}{4}} c_j \Phi^2_{j;b,\l,\nu}) \right|  \right]  r^{\frac 12 +\frac{\Im \l}{b} + s_c-N} > 0. 
          \label{eqnonadmnondeg}
        \ee
        % \bea
        %   % |\Phi^1_{1;b,\l,\nu}| \sim_{b, E} |\psi_1^{b, 1+\l+ibs_c}|, \quad  |\Phi^2_{2;b,\l,\nu}| \sim_{b, E} |\psi_1^{b, 1-\bar \l+ibs_c}|, \quad r \ge \frac 4b \label{eqadmnondeg}\\
        %   |\Phi^1_{3;b,\l,\nu}| \sim |\psi_3^{b, 1+\l+ibs_c}|, \quad  |\Phi^2_{4;b,\l,\nu}| \sim |\psi_3^{b, 1-\bar \l+ibs_c}|, \quad r \ge r_{0;b,\nu,\delta_0} \label{eqnonadmnondeg}
        % \eea
        % with $r_{0;b,\nu,\delta_0} \ge \frac 4b$.
        % {\color{red}And $\{ \pa_\l^n\Phi_{j;b,\l,\nu} \}_{n \ge 0, 1 \le j \le 4}$ are linear independent functions. }
        \item Linear independence: $\{\Phi_{j;b,\l,\nu}\}_{1 \le j \le 4}$ are linearly independent when $\{\Phi_{j;b,\l,\nu}\}_{j = 1, 2}$ are also well-defined from Proposition \ref{propextfund} or Proposition \ref{propextfundh}. 
    \end{enumerate}
\end{proposition}

\begin{remark}[Exterior solutions for 1D even case]\label{rmkchoice1Dext} For the case $d = 1, l = 0$ where $\nu = -\frac 12$, the equation \eqref{eqnu} is the same as $d = 1, l = 1$ case where $\nu = \frac 12$. So we define exterior solutions as 
\be
    \Phi_{j;b,\l,-\frac 12} := \Phi_{j;b,\l,\frac 12}, \quad {\rm for}\,\, 1 \le j \le 4;
\ee
with $j = 1, 2$ from Proposition \ref{propextfund} and $j = 3, 4$ from Proposition \ref{propextfundin}. 
\end{remark}

\begin{proof}The proof of (1) resembles Step 1 of the proof of Proposition \ref{propextfund}, and (2) resembles Step 2.2 of that plus Lemma \ref{leminvtildeHext} for scalar inversion of $\tilde H_{b, E}$ with different oscillatory behavior. Finally, (3) is a corollary from the non-degeneracy properties. We will apply the same vector notation \eqref{eqvectorform}. 

In contrast to Proposition \ref{propextfund}, the smallness of contractions comes from $r_0 \gg 1$ rather than smallness of  $b$ and $|E-1|$. Therefore, it suffices to guarantee Proposition \ref{propQbasymp}, Proposition \ref{propQbasympref} and Proposition \ref{propWKB} are applicable, which leads to the choice $s_{c;{\rm ext}}^{(3)} \le \min\{s_c^{(0)}(d),  s_c^{(0)'}(d)\}$ and small enough such that $b(d, s_c) \le b_0(I_0)$ for all $s_c < s_{c;{\rm ext}}^{(3)}$. 
The constants $s_c^{(0)}(d),  s_c^{(0)'}(d), b_0(I_0)$ are from the propositions listed above. 

\mbox{}

\underline{Step 1. Fundamental solutions for $b = 0$ case.}

Recall the inversion operator $T_{x_*, E}^{(0), G}$ for $\pa_r^2 - E$ from \eqref{eqdefTb0inv}, and similar to \eqref{eqbddTk}, it is easy to see $T_{x_*, E}^{(0), G}$ is bounded in $C^0_{r^{-2} e^{\sqrt E r}}([x_*, \infty]) \to C^1_{r^{-1} e^{\sqrt E r}}([x_*, \infty])$ and in $C^0_{r^{m} e^{(\sqrt E - \a) r}}([x_*, \infty]) \to C^1_{r^{m} e^{(\sqrt E - \a) r}}([x_*, \infty])$ for $m \ge -2, 0 < \a < 2\Re \sqrt{E}$, with the bounds independent of $x_* \ge 1$. 

Then we can invert $\calH_{0, \nu}$ in \eqref{eqnu} and define $\Phi_{j ;0, \l, \nu} = \left( \begin{array}{c}
     \Phi^1_j  \\
     \Phi^2_j 
\end{array} \right)$ for $j = 3, 4$ as the solution of the integral equation
\bee
\left| \begin{array}{l}
    \Phi^1_j = S_j + T_{x_*; 1+\l}^{(0), G} F^+_j
    \\
    \Phi^2_j = R_j + T_{x_*; 1-\l}^{(0), G} \overline{F^-_j}
\end{array}\right.
\eee
where $x_* = x_*(d, \nu)$ will be picked large enough, and
\[\left( \begin{array}{c} 
  S_3 \\ R_3
  \end{array} \right) = 
  \left( \begin{array}{c} 
  e^{\sqrt{1+\l} r}  \\ 0
  \end{array} \right),\quad 
  \left( \begin{array}{c} 
  S_4 \\ R_4
  \end{array} \right) = 
  \left( \begin{array}{c} 
  0 \\ e^{\sqrt{1-\l} r}   
  \end{array} \right) \]
  and $F^\pm_j$ are as \eqref{eqb0Fpm}. With the decay of potentials in $F^\pm_j$, the contraction mapping principle implies 
\be\begin{split}
    \| \Phi_{3;0, \l, \nu}^1 - e^{\sqrt{1+\l}r}  \|_{C^1_{e^{\sqrt{1+\l}{r}}}([x_*, \infty)) } + 
 \| \Phi_{3;0, \l, \nu}^2 \|_{C^1_{e^{(\sqrt{1-\l} - \frac{p-1}{8}){r}}}([x_*, \infty)) } &\lesssim_d (\nu^2 + 1) x_*^{-1},\\
\| \Phi_{4;0, \l, \nu}^1 \|_{C^1_{e^{(\sqrt{1+\l}-\frac{p-1}{8} ){r}}}([x_*, \infty)) } + \| \Phi_{4;0, \l, \nu}^2 - e^{\sqrt{1-\l}r} \|_{C^1_{e^{\sqrt{1-\l}{r}}}([x_*, \infty)) } &\lesssim_d (\nu^2 + 1) x_*^{-1},
\end{split} \label{eqestPhib0in}
\ee
where $x_*(d, \nu) \gg \nu^2 + 1$. That implies \eqref{eqasympb0in}-\eqref{eqnondegb0in}.

\mbox{}

\underline{Step 2. Fundamental solutions for $b > 0$.}

%In this step, we construct two other fundamental solutions $\Phi^{ext}_{3;b,\l,\nu}, \Phi^{ext}_{4;b,\l,\nu}$ of \eqref{eqnu}, or more specifically \eqref{eqsystemphivarphi}, near infinity. They exhibit different oscillatory behavior from $\Phi^{ext}_{1;b,\l,\nu}, \Phi^{ext}_{2;b,\l,\ni}$ and hence these four solutions generate the fundamental solution families for \eqref{eqnu}. The proof will resemble arguments in the proof of Lemma \ref{leminvtildeHext} and Step 2(2).  

\mbox{}

\noindent{\textit{Step 2.1. Other inversions of $\tilde H_{b, E}$ in exterior region.}}

Let $r_0 \ge \frac 4b$. We consider the exterior region $[r_0,\infty)$. 

Firstly, we define $\tilde T^{ext,+}_{b,E}$ as another inversion of $\tilde H_{b, E}$ similarly as $\tilde T^{ext}_{\infty;b,E}$ in Lemma \ref{leminvtildeHext}. For $N$ satisfying
\be N > -\frac{\Im E}{b} + 1, \quad N \ge 0 \label{eqcondinv+} \ee
we define $\tilde T^{ext,+}_{b,E}$ for $ f \in X^{-2, N, +}_{r_0 ;b, E}$ (defined in \eqref{eqdefXBanach}) as 
\be \label{eqdeftildeHbnu+}
  \tilde T_{b, E}^{ext,+} f = -
      \psi_1^{b, E} \calI^+_{N; b, E} [f](r) + \psi_3^{b, E} \int_r^\infty \psi_1^{b, E} f \frac{ds}{W_{31}} 
\ee
where $W_{31} = \calW (\psi_3^{b, E}, \psi_1^{b, E}) = i\frac{b^\frac 13 E^\frac 16}{2^\frac 43 \pi}$ from \eqref{eqWronski2} and the integral operator being
\be
\begin{split}
  \calI^+_{N; b, E} [f](r) &= \int_r^\infty \left(D_{++} \circ (2i\phi')^{-1}\right)^N (\psi_3^{b, E} f) \frac{ds}{W_{31}} \\
  &+ \sum_{l = 0}^{N-1} \left(D_{++} \circ (2i\phi')^{-1}\right)^l (\psi_3^{b, E} f) \cdot (2i\phi' W_{31})^{-1}.
  \end{split} \label{eqdefcalIn}
\ee
Here $D_{++} = D_{++;b, E}$ is from \eqref{eqDpm}. 

We claim that $\calI^+_{N;b, E}$ is well-defined and independent of $N$ satisfying \eqref{eqcondinv+}, and $\tilde T^{ext, +}_{b,E}$ is well-defined inversion of $\tilde H_{b, E}$ with the following bound
 \be
    \| \tilde T_{b, E}^{ext,+} f \|_{X^{-2, N, +}_{r_0;b, E}} +  \| \tilde T_{b, E}^{ext,+} f \|_{X^{0, N+1, +}_{r_0;b, E}} \lesssim_{b, E, N} \| f \|_{X^{-2, N, +}_{r_0;b, E}}, \label{eqbddHbnuinvext+}
    \ee
    where the constant is independent of $r_0 \ge \frac 4b$.
Indeed, we first can prove the property of $\calI^+_{N;b, E}[f]$ and its estimate
\[  |D_{--}^{j} \tilde \calI^+_N[f](r)| \lesssim_{j} |e^{2\eta(r)}| \cdot \| f \|_{\tilde X^{-2, N}_{r_0;b, E}} \cdot \left| \begin{array}{ll}
    b^{-1} r^{-j-2} & 0 \le j \le N \\
    r^{-N-1} & j = N+1
\end{array}\right. \quad r \ge r_0,\]
    as \eqref{eqestD--IN} and \eqref{eqestD--IN+1} in the proof of Lemma \ref{leminvtildeHext}. Next, noticing  that $\left|\int_r^\infty \psi_1^{b, E} f\frac{ds}{W_{31}}\right| \lesssim b^{-1} r^{-2} \| f \|_{X^{-2, N, +}_{r_0 ;b, E}}$, we can easily show  
    \[ \left|\pa_r^j \int_r^\infty \psi_1^{b, E} f\frac{ds}{W_{31}}\right| \lesssim b^{-1} r^{-j-2} \| f \|_{X^{-2, N, +}_{r_0 ;b, E}},\quad 0 \le j \le N+1, \quad r\ge r_0. \]
The bounds \eqref{eqbddHbnuinvext+} now follow these two estimates, the estimate of $\psi_j^{b, E}$ in Proposition \ref{propWKB} and Leibniz rule.

\mbox{}

Next, we define the function space
\be
\tilde X_{r_0; b, E}^{\a, N} = \left\{ f \in C^N([r_0, \infty)) : \| f \|_{\tilde X_{r_0; b, E}^{\a, N}} = \sum_{k = 0}^N \sup_{r \ge r_0} \left(\frac{r^{-\a + k} |D_{-3; b, E}^k f|}{\omega_{b, E}^+}\right) < \infty \right\}; \label{eqtildeBanach}
\ee
where the differential operator $D_{-3;b, E}$ and weight function $\omega_{b, E}^+$ are as \eqref{eqDpm}, \eqref{eqomegapm}. Now for $N$ satisfying \eqref{eqcondinv+}, we define $\tilde T^{ext, -}_{b, E}$ for $f \in \tilde X_{r_0;b, E}^{-2, N}$ as 
\be
 \tilde T^{ext, -}_{b, E} f = -\psi_1^{b, E} \tilde \calI^{-}_{N; b, E} [f](r) + \psi_3^{b, E} \int_r^\infty \psi_1^{b, E} f \frac{ds}{W_{31}},
\ee
with $W_{31}$ from \eqref{eqWronski2} and the integral operator $\tilde \calI^-_{N;b, E}$ as 
\be
\begin{split}
  \tilde \calI^-_{N; b, E} [f](r) &= \int_r^\infty \left(D_{--} \circ (-2i\phi')^{-1}\right)^N (\psi_3^{b, E} f) \frac{ds}{W_{31}} \\
  &+ \sum_{l = 0}^{N-1} \left(D_{--} \circ (-2i\phi')^{-1}\right)^l (\psi_3^{b, E} f) \cdot (-2i\phi' W_{31})^{-1}.
  \end{split} \label{eqdefcalIntilde}
\ee
Like the previous case of $\tilde T^{ext, +}_{b, E}$, we claim that $ \tilde T^{ext, -}_{b, E}$ is a well-defined inversion of $\tilde H_{b, E}$ with the following bound 
\be
      \| \tilde T_{b, E}^{ext,-} f \|_{\tilde X^{-2, N}_{r_0;b, E}} + \| \tilde T_{b, E}^{ext,-} f \|_{\tilde X^{0, N+1}_{r_0;b, E}} \lesssim_{b, E, N} \| f \|_{\tilde X^{-2, N}_{r_0;b, E}} \label{eqbddHbnuinvexttilde}
    \ee
    where the constant is independent of $r_0 \ge \frac 4b$. 
Indeed, the estimate of $\tilde \calI^-_{N;b, E}[f]$ and thereafter the first term in $\tilde T^{ext, -}_{b, E} f$ is controlled similarly as above. 
% \[  |D_{--}^{j} \tilde \calI^-_N[f](r)| \lesssim_{b, E, N} 
%     r^{-j} |e^{2\eta(r)}| \cdot \| f \|_{\tilde X^{-2, N}_{r_0, \frac 4b}},\quad r \ge r_0.  \]
%     similar to the proof of \eqref{eqestD--N}, \eqref{eqestD--IN+1} using boundedness of $D_{-3;b, E}^j f$ and $D_{+;b, E}^k \psi_3^{b, E}$, and the boundedness of first term in $\tilde T^{ext, -}_{b, E} f$ follows. 
    For the second term, since 
    \[ |D_{-4;b, E}^j (\psi_1^{b, E} f) r^j| \lesssim r^j \sum_{k = 0}^j |D_{-;b,E}^k \psi_1^{b, E}| \cdot |D_{-3;b, E}^{j-k} f | \lesssim b^{-\frac 23} r^{-3-j} \| f \|_{ \tilde X^{-2, N}_{r_0;b, E} } ,\quad r \ge r_0,   \]
    namely $\psi_1^{b, E} f$ satisfies \eqref{equpupasas}, 
    the identity \eqref{eqD++commutator} and estimates of $f_{n, k}$ \eqref{eqestfnk}, $g_{n, k}$ \eqref{eqgnk} imply that for $1 \le j \le N$, 
    \bee
     \left| D_{-4;b, E}^j \int_r^\infty \psi_1^{b, E} f \frac{ds}{W_{31}}\right| \lesssim_j b^{-1} r^{-2-j} \| f \|_{ \tilde X^{-2, N}_{r_0;b, E} },\quad r \ge r_0,\eee
     and for $j = N+1$ 
     \bee
     &&D_{-4;b, E}^{N+1} \int_r^\infty \psi_1^{b, E} f \frac{ds}{W_{31}} + 4i\phi_{b, E}'  D_{-4;b, E}^{N} \int_r^\infty \psi_1^{b, E} f \frac{ds}{W_{31}} \\
     &=& \sum_{k=1}^N \sum_{j=0}^k \left( \pa_r f_{n,k} \int_r^\infty D^j_{-4;b, E}(\psi_1^{b, E} f) g_{k,j} \frac{ds}{W_{31}} - f_{n,k} D^j_{-4;b, E}(\psi_1^{b, E} f) g_{k,j} W_{31}^{-1} \right)
    \eee
    so that 
    \[ \left| D_{-4;b, E}^{N+1} \int_r^\infty \psi_1^{b, E} f \frac{ds}{W_{31}}\right| \lesssim_N r^{-N-1}   \| f \|_{ \tilde X^{-2, N}_{r_0;b, E} },\quad r \ge r_0. \]
    These estimates imply the estimate for the second term in $\tilde T^{ext, -}_{b, E} f$, and conclude the proof of \eqref{eqbddHbnuinvexttilde}.

\mbox{}

{\noindent{\textit{Step 2.2. Construction of exotic fundamental solution.}}}

As in Step 2.2 of the proof of Proposition \ref{propextfund}, we proceed by inverting $\tilde H_{b, E}$ in \eqref{eqnu} to write the integral equation for $j = 3, 4$
\be
  \left| \begin{array}{l}
     \phi^{ext}_j = S^{ext}_j + \tilde T_{r_0; b, E_+}^{ext, \tilde \sigma(j)} \left[ \left( h_{b, E_+} + \frac{\nu^2 - \frac 14}{r^2} - W_{1, b}\right) \phi^{ext}_j - e^{i\frac{br^2}2} W_{2, b} \bar\varphi^{ext}_j \right]  \\
 \varphi^{ext}_j = R^{ext}_j + \tilde T_{r_0; b, \bar E_-}^{ext, -\tilde \sigma(j)} \left[ \left( h_{b, \bar E_-} + \frac{\nu^2 - \frac 14}{r^2} - W_{1, b}\right) \varphi^{ext}_j - e^{i\frac{br^2}2} W_{2, b} \bar\phi^{ext}_j \right]
 \end{array}\right. \label{eqIextexotic} 
\ee
with $r_0 \ge \frac 4b$, 
$\tilde \sigma(3) = +$, $ \tilde \sigma(4) = -$, and 
source terms
\bee
  \left(\begin{array}{c}
       S^{ext}_3  \\
       R^{ext}_3
  \end{array} \right) = 
   \left(\begin{array}{c}
       \psi_3^{b, E_+}  \\
       0
  \end{array} \right),\quad
  \left(\begin{array}{c}
       S^{ext}_4  \\
       R^{ext}_4
  \end{array} \right) = 
   \left(\begin{array}{c}
       0 \\
       \psi_3^{b, \bar E_-} 
  \end{array} \right),\label{eqIextexoticsource}
\eee
In particular, $S^{ext}_3 \in X^{0, N, +}_{\frac 4b;b, E_+}$ and $R^{ext}_4 \in X^{0, N, +}_{\frac 4b;b, \bar E_-}$ for any $N \ge 0$. 

With the decay of potentials \eqref{eqpotentialest}, \eqref{eqpotentialest2} and \eqref{eqpotentialest3}, 
% $|\pa_r^n W_{k, b}| \lesssim_{b, n} r^{-2-n}$ for $k = 1, 2$ and $r \ge \frac 4b$ from \eqref{eqQbasymp8} and \eqref{eqUabs2}, 
we have the following rough bounds similar to 
 \eqref{eqfundI1est1},  \eqref{eqfundI1est2},
\bee
 \left\| \left( h_{b, E} + \frac{\nu^2 - \frac 14}{r^2} - W_{1, b}\right)  \right\|_{\calL \left( X^{k, N, +}_{r_0; b, E} \to X^{k-2, N, +}_{r_0; b, E} \right) \cap \calL \left( \tilde X^{k, N}_{r_0; b, E} \to \tilde X^{k-2, N}_{r_0; b, E} \right) } &\lesssim_{b, E, N,\nu}& 1\\
 \left\|f \mapsto e^{i\frac{br^2}{2}}W_{2,b} \bar f \right\|_{\calL\left( X^{k, N, +}_{r_0; b, E_1} \to \tilde X^{k-2, N}_{r_0; b, E_2}  \right) \cap \calL \left( \tilde X^{k, N}_{r_0; b, E_1} \to X^{k-2, N, +}_{r_0; b, E_2} \right) } &\lesssim_{b, \l, N}& 1
 % \| f \|_{\tilde X^{0, N}_{r_0; b, E_2}}, \\
  % \left\|e^{i\frac{br^2}{2}}W_{2,b} \bar f \right\|_{\tilde X^{-2, N}_{r_0; b, E_1}} &\lesssim_{b, \l, N}& \| f \|_{X^{0, N, +}_{r_0; b, E_2}} 
\eee
with $(E_1, E_2) = (E_+, \bar E_-), (\bar E_-, E_+)$, any $r_0 \ge \frac 4b$, $k \ge 0$ and $N \ge 0$. Moreover, the constants are independent of $r_0$.  

Therefore, for any $0 < b \le b_0$, $|\l| \le \delta_0$, $\nu \ge 0$ fixed, we can take a $N = N_{b, \delta_0}$ satisfying \eqref{eqcondinv+} with $\Im E = \Im E_+ = \Im \bar E_- = \Im \l + bs_c$ for any $\l$ above. Then the above potential estimates and \eqref{eqbddHbnuinvext+}, \eqref{eqbddHbnuinvexttilde} imply that \eqref{eqIextexotic} is a contraction mapping in $X^{0, N_{b, \delta_0}, +}_{r_0; b, E_+} \times r_0^{-2} \tilde X^{-2, N_{b, \delta_0}}_{r_0; b, \bar E_-}$ for $j = 3$ and  $\tilde X^{0, N_{b, \delta_0}}_{r_0; b, E_+} \times r_0^{-2} X^{-2, N_{b, \delta_0}, +}_{r_0; b, \bar E_-}$ for $j = 4$, with $r_0 = r_{0;b,\nu, \delta_0}$ sufficient large depending on $b, \nu$ and uniform for $|\l| \le \delta_0$. Hence
\be 
\begin{split}
  \| \phi^{ext}_3 - S_3^{ext}\|_{X^{0, N_{b, \delta_0}, +}_{r_0; b, E_+ }} + r_0^{-2} \| \varphi^{ext}_3\|_{ \tilde X^{-2, N_{b, \delta_0}}_{r_0; b, \bar E_-}} &\ll \| S_3^{ext} \|_{X^{0, N_{b, \delta_0}, +}_{r_0; b, E_+}}, \\
  r_0^{-2} \| \phi^{ext}_4 \|_{\tilde X^{-2, N_{b, \delta_0}}_{r_0; b, E_+}} + \| \varphi^{ext}_4 - R_4^{ext}\|_{X^{0, N_{b, \delta_0}, +}_{r_0; b, \bar E_-}} &\ll \| R_4^{ext} \|_{X^{0, N_{b, \delta_0}, +}_{r_0; b, \bar E_-}}, \end{split}
  \label{eqPhiextnondegin}
\ee
Then by iterating the smoothing estimate in \eqref{eqbddHbnuinvext+}, \eqref{eqbddHbnuinvexttilde}, we can improve $N$ from $N_{b, \delta_0}$ to any finite $N$. To sum up, we have constructed two fundamental solutions $\Phi^{ext}_{3;b, \l,\nu}$, $\Phi^{ext}_{4;b, \l,\nu}$ on $[r_{0;b,\nu,\delta_0}, \infty)$ satisfying
\be
\begin{split}
  \| \Phi^{ext, 1}_{3;b,\l,\nu} \|_{X^{0,N,+}_{r_0;b,E_+}} +  \| \overline{\Phi^{ext, 2}_{3;b,\l,\nu}} \|_{\tilde X^{-2,N}_{r_0;b,\bar E_-}} <& \infty,\quad \forall N \ge 0,\\
  \| \Phi^{ext, 1}_{4;b,\l,\nu} \|_{\tilde X^{-2,N}_{r_0;b,E_+}} +  \| \overline{\Phi^{ext, 2}_{4;b,\l,\nu}}  \|_{X^{0,N,+}_{r_0;b,\bar E_-}} <& \infty,\quad \forall N \ge 0,
  \end{split} \label{eqPhi34est}
\ee
which implies the boundedness \eqref{eqnonadmosc} using \eqref{eqetaconv} to show $\omega_{b, E}^+ \sim_{b, E} r^{-\frac 12 - \frac{\Im E}{b}}$ as $r\to \infty$, and \eqref{eqphasematcheibx2} to go from $D_{\pm l;b,E}^n$ to $(\pa_r \pm \frac{ilbr}{2})^n$.

Moreover, the smallness in \eqref{eqPhiextnondegin} and asymptotics of $\psi_3^{b, E}$ from Proposition \ref{propWKB} (5) indicate that for some $r_0' \gg r_0$, we have for $k \in \{ 0, 1\}$, 
\be 
 |\Phi^{ext, 1}_{3;b,\l,\nu}| \sim |\psi_3^{b, E_+}| \sim_{b, \l} r^{-\frac 12 - \frac{\Im \l}{b} - s_c}, \quad |\Phi^{ext, 2}_{4;b,\l,\nu}| \sim | \psi_3^{b, \bar E_-}| \sim_{b, \l} r^{-\frac 12 - \frac{\Im \l}{b} - s_c},\quad \forall \,\, r \ge r_0'.\label{eqPhi34est2}
\ee
Combined with the extra $r^{-2}$ decay for $\Phi^{ext, 2}_{3;b,\l,\nu}$ and $\Phi^{ext, 1}_{4;b,\l,\nu}$ from \eqref{eqPhiextnondegin}, we obtain the non-degeneracy  \eqref{eqnonadmnondeg} with $N = 0$. For $N \ge 1$, we compute 
\bee
&&\pa_r^N\sum_{j=3}^4  ( e^{-\frac{ibr^2}{4}} c_j \Phi^1_{j;b,\l,\nu}) \\
&=& c_3 
e^{-\frac{ibr^2}{2}} \left(\pa_r - ibr\right)^N \left( e^{i\frac{br^2}{4}} \Phi^1_{3;b,\l,\nu}\right) + c_4 e^{\frac{ibr^2}{2}} \left(\pa_r + ibr\right)^N \left( e^{-i\frac{3br^2}{4}} \Phi^1_{4;b,\l,\nu}\right) \\
&=& (br)^N e^{-\frac{ibr^2}{4}} \left( (-i)^N c_3   \Phi^1_{3;b,\l,\nu} + i^N c_4   \Phi^1_{4;b,\l,\nu}  \right) + O(r^{-\frac 12 - \frac{\Im \l}{b} - s_c + N-2})
\eee
where the residuals are controlled by \eqref{eqnonadmosc}. With a similar computation for the second term in \eqref{eqnonadmnondeg}, we can evaluate for $N \ge 1$ that
\bee
 {\rm LHS\,\,of\,\,\eqref{eqnonadmnondeg}} = \liminf_{r \to \infty} \left[ \left|  \sum_{j=3}^4  ( e^{-\frac{ibr^2}{4}} c_{j,N} \Phi^1_{j;b,\l,\nu}) \right| +  \left| \sum_{j=3}^4  ( e^{\frac{ibr^2}{4}} c_{j,N} \Phi^2_{j;b,\l,\nu}) \right|  \right]  r^{\frac 12 +\frac{\Im \l}{b} - s_c} > 0
\eee
where $(c_{3,N},c_{4,N}) = ((-ib)^N c_3, (ib)^N c_4) \neq (0, 0)$ and the positivity comes from $N=0$ case. That concludes \eqref{eqnonadmnondeg} for all $N \ge 0$. 

% Moreover, with $r_0 = r_{0;b,\nu,\delta_0}$ is independent of $\l$, we can differentiate \eqref{eqIextexotic} w.r.t. $\l$ and $\bar \l$, and similar to Step 3(2), the existence of solution $(\pa_\l \phi^{ext}_j, \pa_{\bar \l} \varphi_j^{ext})$ follows inverting the same linear operator within $X^{2,N_b+2,+}_{r_0;b,E_+} \times X^{2,N_b+2,+}_{r_0;b,\bar E_-}$. This can be done similarly with $r_0$ taken large enough if necessary. In conclusion, $\Phi^{ext}_{j;b,\l,\nu}$ for $j = 3, 4$ is analytic w.r.t. $\l$.  

\mbox{}

\underline{Step 3. Linear independence.}

We are going to show the equation for $(c_j)_{j=1}^4$
\be \sum_{j=1}^4 c_j \left( \begin{array}{c}
    \Phi^1_{j;b,\l,\nu}(r)  \\
    \Phi^2_{j;b,\l,\nu}(r) 
\end{array} \right) \equiv 0,\quad \forall\,\,r > 0,\label{eqlinindepeq}
\ee
has only zero solution when $b, \l, \nu$ allows all four fundamental solutions to exist according to Proposition \ref{propextfund} and Proposition \ref{propextfundin}, or Proposition \ref{propextfundh} and Proposition \ref{propextfundin}. 

\textit{Case 1. $b = 0$.} Notice that \eqref{eqlinindepeq} implies
\[ 0 = \lim_{r \to \infty} \left| \left( \begin{array}{cc}
    e^{-\sqrt{1+\l}r} &  \\
     &  e^{-\sqrt{1-\l}r}
\end{array} \right) \sum_{j=1}^4 c_j \left( \begin{array}{c}
    \Phi^1_{j;b,\l,\nu}(r)  \\
    \Phi^2_{j;b,\l,\nu}(r) 
\end{array} \right) \right|. \]
From the fast decay of $\Phi_1$, $\Phi_2$ \eqref{eqb0extbdd} and non-degeneracy of $\Phi_3$, $\Phi_4$ \eqref{eqnondegb0in}, this implies $c_3 = c_4 = 0$. Then $c_1 = c_2 = 0$ follows considering 
\[ 0 = \lim_{r \to \infty} \left| \left( \begin{array}{cc}
    e^{\sqrt{1+\l}r} &  \\
     &  e^{\sqrt{1-\l}r}
\end{array} \right) \sum_{j=1}^2 c_j \left( \begin{array}{c}
    \Phi^1_{j;b,\l,\nu}(r)  \\
    \Phi^2_{j;b,\l,\nu}(r) 
\end{array} \right) \right|. \]
and applying the non-degeneracy \eqref{eqnondegb0}. 

\mbox{}

\textit{Case 2. $b > 0$ with $\{\Phi_{j;b,\l,\nu}\}_{j=1}^2$ from Proposition \ref{propextfund}.} Pick an $N \in (\frac{\Im \l}{b} - s_c, \infty) \cap \ZZ_{\ge 0}$ and consider 
\be 0 = \lim_{r \to \infty} \left| \left( \begin{array}{cc}
    \pa_r^N \circ e^\frac{-ibr^2}{4} &  \\
     &  \pa_r^N \circ e^\frac{ibr^2}{4}
\end{array} \right) \sum_{j=1}^4 c_j \left( \begin{array}{c}
    \Phi^1_{j;b,\l,\nu}(r)  \\
    \Phi^2_{j;b,\l,\nu}(r) 
\end{array} \right) \right| r^{\frac 12 + \frac{\Im \l}{b} - s_c- N}. \label{eqnondegtest} \ee
Then the boundedness of $\Phi_1$, $\Phi_2$ \eqref{eqadmosc} and non-degeneracy of $\Phi_3$, $\Phi_4$ \eqref{eqnonadmnondeg} yield that $c_3 = c_4 = 0$. The linear independence of $\Phi_1$ and $\Phi_2$ from Proposition \ref{propextfund} (2) concludes $c_1 = c_2 = 0$. 

\mbox{}

\textit{Case 3. $b > 0$ with $\{\Phi_{j;b,\l,\nu}\}_{j=1}^2$ from Proposition \ref{propextfundh}.} Since $\Im \l \le \frac 12 b$ as required in Proposition \ref{propextfundh}, we can consider \eqref{eqnondegtest} with $N = 1$, and that $c_3 = c_4 = 0$ and $c_1 = c_2 = 0$ follow from \eqref{eqadmosch}, \eqref{eqnonadmnondeg} and \eqref{eqadmnondegh} similarly.

% This is immediate conclusion of the leading order near infinity by contraction: for $b = 0$ see \eqref{eqestPhib0} and \eqref{eqestPhib0in}, for $b > 0$ see \eqref{eqPhiextnondeg} and \eqref{eqPhiextnondegin} noticing that $\Phi_{1;b, \l, \nu}$ and $\Phi_{2;b, \l, \nu}$ are linear independent superposition of $\Phi^{ext}_{1; b, \l, \nu}$ and $\Phi^{ext}_{2; b, \l, \nu}$ from \eqref{eqtildePhik}, \eqref{eqdefacheck} and \eqref{eqPhi12adm}).  {\color{red} What about  $\Phi_{1;b, \l, \nu}$ and $\Phi_{2;b, \l, \nu}$ from Proposition \ref{propextfundh}?}
\end{proof}

\part{Application} \label{part2}

\section{Existence of bifurcated eigenmodes}\label{sec6}

In this section, we construct two eigenpairs of $\calH_b$ bifurcated from the radial generalized functions $\xi_2, \xi_3$ of $\calH_0$ \eqref{eqxi2xi3}. Similar to the construction of $Q_b$ in \cite{MR4250747}, the main strategy is to find admissible interior and exterior solution families, and match asymptotics to determine the parameter. 

We will exploit the scalar inversion operators in Subsection \ref{sec31} and Subsection \ref{sec41}, with the asymptotics of $Q_b$ from Proposition \ref{propQbasymp}, Proposition \ref{propQbasympref}. 
% In particular, the refined asymptotics of $P_b$ Proposition \ref{propQbasympref} is essential in evaluating the leading asymptotics of bifurcated eigenvalue. 
In particular, we recall the following notations: 
\begin{itemize}
    \item For interior analysis: Banach spaces $Z_{\pm, \a; x_*}$, $\tilde Z_{\pm, \a; x_*}$ from \eqref{eqdefZpm}, potential $V_{\pm,b}$ from \eqref{eqdefVpmb}; functions $A, E, D, A_b, \tilde{\frakE}$ and linear operators $L_{+;x_*}^{-1}$, $L_{+,b;x_*}^{-1}$, $L_-^{-1}$, $L_{-,b}^{-1}$,  from Lemma \ref{lemLpm} and Lemma \ref{lemLpmb}. 
    \item For exterior analysis: Banach space $X^{\a, N, -}_{r_0, r_1; b, E}$ and weight function $\omega^\pm_{b, E}$ from Definition \ref{defdiffopspace}.
    \item Self-similar profile with quadratic phase: $P_b = Q_b e^{i\frac{br^2}{4}}$. 
\end{itemize}

\mbox{}

For imaginary eigenvalue $\l \in i\RR$, we can assume the eigenfunction has symmetry
\[ Z^1 = \bar Z^2
\quad {\rm or\,\,equivalently} \quad \Phi^1 = \bar \Phi^2 =: \phi, \]
with $\Phi$ \eqref{eqZPhi}. 
Then we can rewrite the eigen equation \eqref{eqnu} with $l = 0$ as a scalar equation 
\be
  \left( \pa_r^2 - 1 -i\upsilon + \frac{b^2 r^2}{4} - \frac{(d-1)(d-3)}{4r^2} + W_{1,b} \right)\phi + e^{i\frac{br^2}{2}} W_{2, b} \bar \phi = 0,\label{eqeigenrad1}
\ee
where we denote 
\[ \upsilon = -i \l + bs_c \in \RR. \]
Equivalently, for
\[ \Xi = r^{-\frac{d-1}{2}} \phi, \]
the equation \eqref{eqeigenrad1} becomes 
\be
   \left( \Delta - 1 -i\upsilon + \frac{b^2 r^2}{4} + W_{1,b} \right)\Xi + e^{i\frac{br^2}{2}} W_{2, b} \bar \Xi = 0.\label{eqeigenrad2}
\ee

\mbox{}

The main results of this section are the followings. We first construct the $\CC$-valued function $\rho_b$ as bifurcation of $\rho$, and then the bifurcated eigenpairs of $\calH_b$.

\begin{lemma}[Construction of $\rho_b$] \label{lemrhob} For $d \ge 1$, there exists $s_{c;{\rm eig}}^{(1)}(d) \ll 1$ such that  for $0 <s_c \le s_{c;{\rm eig}}^{(1)}(d)$ with $ b = b(s_c, d) \ll 1$ from Proposition \ref{propQbasymp}, there exists a positive real value
\be \upsilon_{\rho_b} = \frac{4\pi \kappa_Q^2}{\int_0^\infty Q^2 r^{d+1} dr} b^{-3}e^{-\frac{\pi}{b}} (1 + o_{s_c \to 0}(1)) \label{equpsilonasymp} \ee
with $\kappa_Q$ from Lemma \ref{lemQasymp}, and $\rho_b \in C^\infty_{rad}(\RR^d; \CC)$ such that
\be \left(\Delta - 1 - i \upsilon_{\rho_b} + \frac{b^2 r^2}{4} + W_{1, b} \right) \rho_b + e^{i\frac{br^2}{2}} W_{2, b} \bar \rho_b = - r^2  P_b. \label{eqrhob}
\ee
and the following estimates hold with $x_* = 10|\log b|$
\bea
 \Re \rho_b = L_{+, b; x_*}^{-1} \left( r^2 \Re P_b \right) + O_{\tilde Z_{+, 0;x_*}} (e^{-2x_*} x_*^\frac 52),\quad r \in [0, x_*], \label{eqasympXibReint} \\
 \Im \rho_b = \frac{\| x Q \|_{L^2}^2}{2 |\SS^{d-1}|} \upsilon_{\rho_b} \tilde{\frakE} + O_{\tilde Z_{-, 4;x_*}}(\upsilon_{\rho_b}) + O_{\tilde Z_{+, 0;x_*}}(\upsilon_{\rho_b} b^{\frac 13}),\quad r \in [0, x_*].\label{eqasympXibImint} \\
  \left\| r^{\frac{d-1}{2}}\rho_b \right\|_{X^{5, N, -}_{\frac 2b, \frac{2+\sqrt b}{b}; b, E}} \lesssim_N b^{2 - \frac 16} e^{-\frac{\pi}{2b}}, \quad \forall N \ge 6 \label{eqasympXibext}\\
  \left\| r^{\frac{d-1}{2}}\Re \rho_b \right\|_{C^0_{\omega_{b, 1}^- r^3}([x_*, \frac 2b])} \lesssim b^{-\frac 16} e^{-\frac{\pi}{2b}}, \label{eqasympXibRemid} \\
  \left\| r^{\frac{d-1}{2}}\Im \rho_b \right\|_{C^0_{\omega_{b, 1}^+} ([x_*, \frac 2b]) } \lesssim b^{-3-\frac 16} e^{-\frac{\pi}{2b}}. \label{eqasympXibImmid} 
\eea
\end{lemma}

\begin{proposition}[Construction of bifurcated eigenmodes] \label{propbifeigen}For $d \ge 1$, there exists $s_{c;{\rm eig}}^{(2)}(d) \ll 1$ such that  for $0 <s_c \le s_{c;{\rm eig}}^{(2)}(d)$ with $ b = b(s_c, d) \ll 1$ from Proposition \ref{propQbasymp}, there exist two positive real values
\be \upsilon_{2,b} = 2b + O(b^{-2} e^{-\frac \pi b}) ,\quad \upsilon_{3,b} = \upsilon_{\rho_b} + O(b^{-2} e^{-\frac \pi b}), \ee
and $Z_{2,b}$, $Z_{3, b}  \in C^\infty_{rad}(\RR^d; \CC^2)$ such that 
\[ \left(\calH_b - i\upsilon_{k,b} + ibs_c\right) Z_{k,b} = 0,\quad {\rm for}\,\, j = 2, 3.  \]
Moreover, for $k = 2, 3$, $Z_{k,b} = (\mathfrak{z}_{k, b}, \overline{\mathfrak{z}}_{k, b})^\top$ and 
\be
  \left\| \mathfrak{z}_{k, b} r^{\frac{d-1}{2}} e^{i\frac{br^2}{4}} \right\|_{X^{7, N, -}_{\frac 2b, \frac{2+\sqrt{b}}{b}; b, 1 - i \upsilon_{k,b} } } < \infty,\quad \forall N \ge 8.  \label{eqdecayZkb}
\ee
\end{proposition}

Their proofs will be given in the following two subsections. As preparation, we define the $\RR$-linear multiplication operator related to \eqref{eqeigenrad1} as 
\be
  \calV_{b, E} f:=  \left(  h_{b, E} - \frac{(d-1)(d-3)}{4r^2} + W_{1, b} \right)f + e^{i\frac{br^2}{2}} W_{2,b}\bar f. \label{eqdefcalV}
\ee

\begin{lemma} For $d \ge 1$,  $0 < s_c < \min \{ s_c^{(0)}(d), s_c^{(0)'}(d)\}$ from Proposition \ref{propQbasymp} and Proposition \ref{propQbasympref}, let $E = 1 + i \upsilon$ with $\upsilon \in \RR$. We have 
\begin{enumerate}
    \item If $|\upsilon| \le 10b$, then 
    \be
   \left\| \calV_{b, E}  f \right\|_{X^{\a-2, N, -}_{\frac 2b, \frac{2 + \sqrt b}b ;b, E}} \lesssim_{\a, N}   \left\|   f \right\|_{X^{\a, N, -}_{\frac 2b, \frac{2 + \sqrt b}b ;b, E}},\quad \forall \a \in \RR,\,\, N \ge 0; \label{eqbddcalVext}
   \ee
   and 
   \be
\begin{split}
|\Re (\calV_{b, E} g)| &\lesssim \la r\ra^{-2} |\Re g| + b^2 |\Im g|, \qquad r \le \frac 2b, \\
|\Im (\calV_{b, E} g)| &\lesssim \la r\ra^{-2} |\Im g| + b^2 |\Re g|,\qquad r \le \frac 2b.
\end{split}\label{eqbddcalVext2} 
 \ee
 \item For $\upsilon = 0$, 
\be
 \begin{split}
  \left| \Im(e^{i\frac{br^2}{2}} W_{2, b})\right| &\lesssim e^{-(p-3)r}e^{-\frac{\pi}{b}},\qquad r \le b^{-\frac 12}; \\
     | \Re ( \calV_{b, 1} f)| &\lesssim \la r\ra^{-2} \left( |\Re f| + \la r\ra^{-1} e^{-2S_b} |\Im f| \right),\quad r \le \frac 2b; \\
    | \Im ( \calV_{b, 1} f)| &\lesssim \la r\ra^{-2} \left( |\Im f| + \la r\ra^{-1} e^{-2S_b} |\Re f| \right),\quad r \le \frac 2b.
 \end{split} \label{eqtildecalVest}
\ee
where $S_b(r) = \int_{\min\{ r, \frac 2b \}}^{\frac 2b} \left( 1 - \frac{b^2 s^2}{4} \right)^\frac 12$. 
\end{enumerate}
\end{lemma}

\begin{proof}
The first estimate \eqref{eqbddcalVext2} can be derived in the same way as \eqref{eqfundI1est1} and  \eqref{eqfundI1est2} in the Step 2.2 of the proof of Proposition \ref{propextfund}, noticing that here $E_+ = \bar E_- = E$, and the derivative estimates for $Q_b$ \eqref{eqUabs2}, \eqref{eqQbasymp9} can be extended to $r \ge 2b^{-1} + b^{-\frac 12}$. 

For \eqref{eqbddcalVext2} and \eqref{eqtildecalVest}, with $f = f_1 + i f_2$ and $G_b = e^{i\frac{br^2}{2}} W_{2, b}$, we compute
\bee
 \Re ( \calV_{b, E} f) &=& \left(  \Re h_{b, E} - \frac{(d-1)(d-3)}{4r^2} + W_{1, b} + \Re G_b \right) f_1 - \left(\Im h_{b, E} - \Im G_b \right) f_2 \\
 \Im (\calV_{b, E} f) &=& \left( \Re h_{b, E} - \frac{(d-1)(d-3)}{4r^2} + W_{1, b} - \Re G_b \right) f_2 + \left(\Im h_{b, E} + \Im G_b\right) f_1.
\eee
From \eqref{eqQbasymp1}, \eqref{eqQbasymp5}, \eqref{eqQbasymp6} and \eqref{eqQbsharp}, 
\bee |W_{1, b}(r)| + |\Re G_b(r)| \lesssim |P_b(r)|^{p-1} \sim \la r\ra^{-\frac{(d-1)(p-1)}{2}} e^{(p-1)(S_b(r) - \frac{\pi}{2b})}\lesssim \la r\ra^{-3}, \\
|\Im G_b(r)| \lesssim |\Re P_b(r)|^{p-2} |\Im P_b(r)| \sim \la r\ra^{-\frac{(d-1)(p-1)}{2}}e^{(p-1)(S_b(r) - \frac{\pi}{2b}) }e^{-2S_b} \lesssim \la r\ra^{-3} e^{-2S_b},
\eee
These bounds combined with \eqref{eqbddh} and $\Im h_{b, 1} = 0$ from Proposition \ref{propWKB} (4) yield \eqref{eqbddcalVext2} and \eqref{eqtildecalVest}. 
\end{proof}

\subsection{Construction of $\rho_b$}

\begin{proof}[Proof of Lemma \ref{lemrhob}]

With $x_* = 10 |\log b|$, we will construct families of interior solutions on $[0, x_*]$ and exterior solutions on $[x_*, \infty)$ with parameters $(\upsilon, \a, \mu_\Re, \mu_\Im)$.
Finally we construct $\rho_b$ and determine $\upsilon_{\rho_b}$ via matching asymptotics through Brouwer's fixed point argument. 

As for the parameter $s_{c;{\rm eig}}^{(1)}$, we first assume its smallness such that 
\be
\begin{split}
&s_{c;{\rm eig}}^{(1)} \le \min \{ s_c^{(0)}(d), s_c^{(0)'}(d), s_{c;{\rm int}}^{(2)}(d) \}, \\
& b(s_c, d) \le \min \{ b_{\rm int},  b_0(I_0 = 10) \},\quad  \forall\, \, s_c \in (0, s_{c;{\rm eig}}^{(1)}), 
\end{split}
\label{eqsceig1}\ee
to apply Proposition \ref{propQbasymp}, Lemma \ref{lemLpmb}, Proposition \ref{propQbasympref}, Proposition \ref{propWKB}, Lemma \ref{leminvtildeHext} and Lemma \ref{leminvtildeHmid}. We might futher shrink it to make $b(d,s_c)$ small enough for linear inversion and Brouwer's fixed point argument later. 

For simplicity, we might omit the parameter $b$ when no confusion occurs, and for a scalar function $f$, we denote its vector version as $\vec f := (f, \pa_r f)^\top$.

\mbox{}

\underline{1. Interior solution on $[0, x_*]$.} 

Let
\[  \rho_b = \Sigma + i\Theta.  \]

Then \eqref{eqrhob} is equivalent to
\be
\left| \begin{array}{l}
      -L_{+, b} \Sigma + G \Theta + \upsilon \Theta = - r^2 \Re P_b \\
      -L_{-, b} \Theta + G \Sigma - \upsilon \Sigma = - r^2 \Im P_b
\end{array}
\right. \label{eqrhobint}
\ee
where the potential
\[ G = \Im \left( e^{i\frac{br^2}{2}}W_{2, b} \right). \]

We claim for 
\be |\upsilon| \le b^{-3-\frac 16} e^{-\frac \pi b}, \quad |\a| \le e^{-2x_*} x_*^{\frac 52},  \label{eqrangeupsilona} \ee
there exists a solution $(\Sigma_{\upsilon, \a}, \Theta_{\upsilon, \a})$ solving \eqref{eqrhobint} on $[0, x_*]$, which satisfies \eqref{eqasympXibReint}-\eqref{eqasympXibImint} and have boundary value
\bea
   \vec \Sigma_{\upsilon, \a}(x_*) &=& \mathfrak{c}_{x_*} \vec D(x_*) + \a \vec A(x_*) + O\left( x_*^{1 - \frac{d-1}{2}} e^{-x_*}\right) \label{eqrhobbdryint1}\\
   \vec \Theta_{\upsilon, \a}(x_*) &=& \frac{\| xQ\|_{L^2}^2}{2|\SS^{d-1}|}\upsilon  \vec \frakE(x_*) + O\left( b^{-3} e^{-\frac \pi b} x_*^{-2-\frac{d-1}{2}} e^{x_*} \right) \label{eqrhobbdryint2}
\eea
with 
\be \mathfrak{c}_{x_*} = O(x_*^3)  \label{eqestfrakcx*}
\ee
and the constant bounds for $O(...)$ terms are all independent of $\a, \upsilon$ within \eqref{eqrangeupsilona}.

In the following proof of this claim, we will apply the Banach spaces $Z_{\pm, \beta;x_*}$ and $\tilde Z_{\pm, \beta ;x_*}$ for $\beta \ge 0$ from \eqref{eqdefZpm}. For simplicity, we usually omit $\upsilon, \a, x_*$ subscripts.

\mbox{}

Invert $L_{\pm, b}$ and add the fundamental solution $A_b$ from Lemma \ref{lemLpmb}, so that \eqref{eqrhobint} turns into
\bee
  \left| \begin{array}{l}
       \Sigma = L_{+, b;x_*}^{-1} \left( (G + \upsilon) \Theta\right) - L_{+, b;x_*}^{-1} \left(  - r^2 \Re P_b \right) + \a A_b\\
       \Theta = L_{-, b}^{-1} \left( (G - \upsilon) \Sigma \right) - L_{-, b}^{-1} \left( - r^2 \Im P_b\right)
\end{array}
\right.
\eee
Plug the second equation into the first, we obtain
\bee
 \Sigma &=& L_{+, b;x_*}^{-1} \left( (G + \upsilon) L_{-, b}^{-1} \left( (G - \upsilon) \Sigma \right) \right) \\
 &+& L_{+, b;x_*}^{-1} \left( (G + \upsilon)L_{-, b}^{-1}  \left(  r^2 \Im P_b\right)  \right) - L_{+, b;x_*}^{-1} \left(  - r^2 \Re P_b \right) + \a A_b
\eee
We will show under the range of parameters\eqref{eqrangeupsilona}, this is a contraction in the ball of radius $M_0 \sim 1$ in $Z_{-, 3}$. 
Applying bounds in Lemma \ref{lemLpmb} and
 \eqref{eqtildecalVest}, we notice that 
\bee
 &&\left\| L_{+, b; x_*}^{-1} \circ (G+\upsilon) \circ L_{-, b}^{-1} \right\|_{Z_{+, 2} \to \tilde Z_{-, 3}} \\
 &\le &
 \left\| L_{-, b}^{-1} \right\|_{Z_{+, 2} \to Z_{+, 3}}
 \left\|  G+\upsilon \right\|_{Z_{+, 3} \to Z_{-, 2}}
 \left\| L_{+, b; x_*}^{-1} \right\|_{Z_{-, 2} \to \tilde Z_{-, 3}} \\
 &\lesssim& bs_c e^{(5-p)_+ x_*} x_* + |\upsilon| e^{2x_*} x_* \le e^{-\frac \pi b} e^{4x_*} x_* \\
 && \| G - \upsilon \|_{Z_{-, 3} \to Z_{+, 2}} \lesssim bs_c + |\upsilon|  \le b^{-4} e^{-\frac \pi b}
\eee
and thus
\bee
 \left\|L_{+, b;x_*}^{-1}  \left(  - r^2 \Re P_b \right) \right\|_{\tilde Z_{-, 3}} \lesssim 1, &&  \| \a A_b \|_{\tilde Z_{-, 3}} \lesssim |\a| e^{2x_*} x_*^{-3}   \le x_*^{-\frac 12} \\
 \left\| L_{+, b;x_*}^{-1} \left( (G + \upsilon)L_{-, b}^{-1}  \left( r^2 \Im P_b\right)  \right) \right\|_{\tilde Z_{-, 3}} &\lesssim& e^{-\frac{2\pi}b} e^{4x_*} x_*  \\
\left\| L_{+, b;x_*}^{-1} \left( (G + \upsilon) L_{-, b}^{-1} \left( (G - \upsilon) \Sigma \right) \right)\right\|_{\tilde Z_{-, 3}} 
&\lesssim& b^{-4} e^{-\frac{2\pi} b} e^{4x_*} x_* \| \Sigma\|_{Z_{-, 3}}
\eee
Hence the contraction is verified when $b \ll 1$, leading to 
\be
  \Sigma = \Sigma_0 + \a A_b + O_{\tilde Z_{-, 3}}\left( b^{-4} e^{-\frac{2\pi}{b}} e^{4x_*} x_* \right) \label{eqasymprhobSigma}
\ee
with $\Sigma_0 := -L_{+, b;x_*}^{-1} \left(  - r^2 \Re P_b \right)$ and $\| \Sigma_0\|_{\tilde Z_{-, 3}} \lesssim 1$. Recall the definition of $L_{+, b;x_*}^{-1}$ \eqref{eqdefLpmb} and $L_{+;x_*}^{-1}$ \eqref{eqdefLpminv}, we have $L_{+, b;x_*}^{-1} f = L_{+;x_*}^{-1}(f - V_{+, b} L_{+, b;x_*}^{-1}f)$ and therefore evaluate the boundary value of $\Sigma_0$ as
\bee
  \vec \Sigma_0(x_*) = \vec D(x_*) \int_0^{x_*} (r^2 \Re P_b - V_{+,b} \Sigma_0)(s)A(s) s^{d-1} ds =: \mathfrak{c}_{x_*} \vec D(x_*). 
\eee
Evaluating the integral with estimates of $\Re P_b$, $A$, $\Sigma_0$ and $V_{+,b}$ \eqref{eqestVpmb} yields the bound \eqref{eqestfrakcx*}. This combined with the estimate of $A_b$ \eqref{eqestAb} yields \eqref{eqrhobbdryint1} when $b \ll 1$.

\mbox{}

For $\Theta$, we use \eqref{eqasymprhobSigma} and Lemma \ref{lemLpmb} to estimate
\bee
\| L_{-, b}^{-1} \Sigma_0 \|_{\tilde Z_{+, 0}} \lesssim \| \Sigma_0 \|_{Z_{+, -2}} \lesssim 1, &\Rightarrow  &  \left|\left(  V_{-, b} L_{-, b}^{-1} \Sigma_0, Q \right)_{L^2(B_{x_*})}\right| \lesssim b^\frac 13, \\
\left \| L_{-, b}^{-1} \Sigma_0 + \frac{\left( (1- V_{-, b} L_{-, b}^{-1}) \Sigma_0, Q \right)_{L^2(B_{x_*})}}{|\SS^{d-1}|} \tilde{\frakE} \right\|_{\tilde Z_{- ,4}} &\lesssim& \| \Sigma_0 \|_{Z_{-, 3}} \lesssim 1, \\
\| L_{-, b}^{-1} (G \Sigma_0) \|_{\tilde Z_{+, 0}} &\lesssim& \| G \Sigma_0 \|_{Z_{+, -2}} \lesssim e^{-\frac \pi b} \\
\| \a L_{-, b}^{-1} ((G-\upsilon) A_b) \|_{\tilde Z_{+, 2}}&\lesssim& |\a|(bs_c e^{(3-p)_+x_*} + |\upsilon|) \lesssim e^{-\frac \pi b}\\
% e^{-\frac \pi b} e^{(3-p)_+ x_*} \\
\left\| L_{-, b}^{-1} \left( (G-\upsilon) \left(\Sigma - \Sigma_0 - \a A_b \right) \right) \right\|_{\tilde Z_{+, 3}} &\lesssim& b^{-4} e^{-\frac \pi b} \cdot b^{-4} e^{-\frac{2\pi}{b}} e^{4x_*} x_* \\
\left\| L_{-, b}^{-1} \left( - r^2 \Im P_b\right)  \right\|_{\tilde Z_{+, 3}} &\lesssim& e^{-\frac \pi b}, 
\eee
To evaluate $(\Sigma_0, Q)_{L^2(B_{x_*})}$, 
% recall from \eqref{eqdefQ1Q2rho} that $\rho = L_+^{-1} (|x|^2 Q)$ and then 
% \bee
% \left( \rho, Q\right)_{L^2(\RR^d)} = \left(|x|^2 Q, -\frac 12\Lambda Q\right)_{L^2(\RR^d)} = \frac{1} 2 \| xQ\|_{L^2(\RR^d)}^2,
% \eee
we estimate
\bea
  \Sigma_0 - \rho &=&  L_{+;x_*}^{-1} \left( r^2 (\Re P_b - Q) \right) + \left( L_{+;x_*}^{-1} - L_+^{-1} \right) (r^2 Q)  - L_{+; x_*}^{-1} V_{+,b} \Sigma_0 \nonumber\\
  &=& L_{+;x_*}^{-1} \left(O_{Z_{-,2;x_*}}(b^{\frac 13})\right) + \left( |x|^2 Q, D \right)_{L^2(B^c_{x_*})} A \nonumber\\
  &=& O_{Z_{-,3;x_*}}(b^{\frac 13}) + O_{Z_{+,0;x_*}}(x_*^2 e^{-2x_*}) \label{eqSigma0rho}
  \eea
  and then with \eqref{eqrhobQ}, 
  \bee
  ( \Sigma_0, Q)_{L^2(B_{x_*})} = ( \rho, Q)_{L^2(\RR^d)} + ( \Sigma_0 - \rho, Q)_{L^2(B_{x_*})} + ( \rho, Q)_{L^2(B_{x_*}^c)} 
  =  \frac{1} 2 \| xQ\|_{L^2(\RR^d)}^2 + O(b^\frac 13)
\eee
using $e^{-2x_*} = b^{20}$. These bounds lead to
\bea
\Theta &=&-\upsilon L_{-,b}^{-1} \Sigma_0 + L_{-,b}^{-1} (G\Sigma_0) + \a L_{-, b}^{-1} \left( (G-\upsilon) A_b \right)  \nonumber \\
&+& L_{-, b}^{-1} \left( (G-\upsilon) \left(\Sigma - \Sigma_0 - \a A_b \right) \right) - L_{-, b}^{-1} \left( - r^2 \Im P_b\right)  \nonumber\\
&=& \upsilon \left[ \frac{\| xQ\|_{L^2}^2}{2|\SS^{d-1}|} + O(b^\frac 13) \right] \tilde{\frakE} + O_{\tilde Z_{-, 4}}(\upsilon) + O_{\tilde Z_{+, 3}}(e^{-\frac \pi b})\label{eqasymprhobTheta}
\eea
which implies \eqref{eqrhobbdryint2}. 

\mbox{}

\underline{2. Exterior solution on $[x_*, \infty)$.}

Let 
$$\xi := r^{\frac{d-1}{2}}\rho_b, $$
it satisfies 
\[ \left(  \pa_r^2 - \frac{(d-1)(d-3)}{4r^2} + \frac{b^2 r^2}{4} - 1 - i\upsilon + W_{1, b}\right) \xi + e^{i\frac{br^2}{2}} W_{2, b} \bar \xi = - r^{2+\frac{d-1}{2}} P_b, \]
namely 
\be
  (\tilde H_{b, E} + \calV_{b, E}) \xi=   -r^{2+\frac{d-1}{2}} P_b.  \label{eqrhobext}
\ee
with 
\[ E = 1 + i\upsilon, \]
$\tilde H_{b, E}$ from \eqref{eqdeftildeHbE} and $\calV_{b, E}$ from \eqref{eqdefcalV}. The  homogeneous version is 
\be  (\tilde H_{b, E} + \calV_{b, E}) \xi= 0.  \label{eqrhobexthomo} \ee

For  
$$|\upsilon| \le b^{-3-\frac 16} e^{-\frac \pi b},$$ namely the range in \eqref{eqrangeupsilona},
we will construct a solution $\xi^*_\upsilon$ of \eqref{eqrhobext} and two solutions $\xi^\Re_\upsilon$, $\xi^\Im_\upsilon$ for the homogeneous equation \eqref{eqrhobexthomo} on $[x_*, \infty)$, satisfying the estimates
\be
\left| \begin{array}{l}
    \| \xi^*_\upsilon \|_{X^{5, N, -}_{\frac 2b, \frac{2+\sqrt b}{b}; b, E}} \lesssim_{N} b^{2-\frac 16} e^{-\frac \pi{2b}}, \qquad \sum_{\sigma \in \{\Re, \Im\}} \| \xi^\sigma_\upsilon \|_{X^{5, N, -}_{\frac 2b, \frac{2+\sqrt b}{b}; b, E}} \lesssim b^5,\quad \forall\, N \ge 6;\\
    \| \Re \xi^*_\upsilon\|_{C^0_{\omega^-_{b, 1}r^3 }([x_*, \frac 2b])} \lesssim b^{-\frac 16}e^{-\frac \pi{2b}}, \qquad  \| \Im \xi^*_\upsilon\|_{C^0_{\omega^+_{b, 1}}([x_*, \frac 2b])} \lesssim b^{-3-\frac 16}e^{-\frac \pi{2b}},  \\
     \| \Re \xi^\Re_\upsilon\|_{C^0_{\omega^-_{b, 1}}([x_*, \frac 2b])} +  \| \Im \xi^\Im_\upsilon\|_{C^0_{\omega^-_{b, 1}}([x_*, \frac 2b])} +
     \| \Im \xi^\Re_\upsilon\|_{C^0_{\omega^+_{b, 1}}([x_*, \frac 2b])} +  \| \Re \xi^\Im_\upsilon\|_{C^0_{\omega^+_{b, 1}}([x_*, \frac 2b])} \lesssim 1,
     \end{array}\right.
     \label{eqxiextestcompo}
\ee
with the boundary value at $x_*$ being for $k \in \{0, 1\}$, 
\be
\left| \begin{array}{l}
 \pa_r^k \xi^*_\upsilon (x_*) =\pa_r^k \psi_2^{b, 1} (x_*) \left( \Re \kappa^*_0 (1 + O_\RR(b^3)) + i \Im \kappa^*_0 (1 + O_\RR(b^\frac 14)) \right) \\
 \pa_r^k \xi^\Re_\upsilon (x_*) = \pa_r^k \psi_4^{b, 1}(x_*) \left (1 + O_\RR(x_*^{-2}) + i O_\RR(x_*^{-1} e^{-2S_b(x_*)})\right) \\
 \pa_r^k \xi^\Im_\upsilon (x_*) = \pa_r^k \psi_4^{b, 1}(x_*) \left (i + iO_\RR(x_*^{-2}) +  O_\RR(x_*^{-1} e^{-2S_b(x_*)})\right) 
\end{array}\right.
\label{eqxiextbdryx*}
\ee
with 
\be \Re \kappa^*_0 \sim x_*^2 b^{-\frac 16}e^{2S_b(x_*) - \frac \pi{2b}},\quad \Im \kappa^*_0 \sim  b^{-3-\frac 16}e^{-\frac \pi{2b}} \label{eqdefkappa*upsilon}
\ee 
independent of $\upsilon$. 

We will work on the following two regions
\be I_{ext} = \left[ r_0, \infty \right),\quad I_{mid} = \left[x_*, r_0\right],\quad {\rm where\,\,} r_0 = 2b^{-1}.  \ee

\mbox{}

\textit{2.1. Construction on $I_{ext}$.}

Denote 
\[  r_1 = 2b^{-1} + b^{-\frac 12}.\]
Recall that \eqref{eqr*bE1} implies $r_1 > r^*_{b, E}$ when $|\Im E| = |\upsilon| \le b^{-4} e^{-\frac \pi b} \ll b$. So by Lemma \ref{leminvtildeHext} and \eqref{eqbddcalVext}, we have for any $\a \ge 0$ and $N \ge \a + 1$ that 
\bee \left\| \tilde \calT^{ext}_{r_0; b, E} \calV_{b, E} \right\|_{\calL\left(X^{\a, N, -}_{r_0, r_1; b, E} \right)} \lesssim_{\a, N} b,\quad \left\| \tilde \calT^{ext}_{r_0; b, E} \calV_{b, E} \right\|_{ \calL\left(X^{\a, N, -}_{r_0, r_1; b, E} \to X^{\a, N+1, -}_{r_0, r_1; b, E} \right)} \lesssim_{\a, N} 1.
% \label{eqinvboundextrhob} 
\eee
 Also with \eqref{eqQbasymp9} from Proposition \ref{propQbasympref}, for any $\a > 4+s_c - \frac{\upsilon}b$ and $N \ge \a + 1$,  
 \[  \left\|  \tilde \calT^{ext}_{r_0; b, E} \left( r^{2+\frac{d-1}{2}} P_b\right) \right\|_{ X^{\a, N, -}_{r_0, r_1; b, E}}  \lesssim b \left\| r^{2+\frac{d-1}{2}} P_b \right\|_{ X^{\a-2, N, -}_{r_0, r_1; b, E}} \lesssim_{\a, N} b^{-3-\frac 16+\a}  e^{-\frac{\pi}{2b}}.  \]
 where we used $b^{-s_c + \frac \upsilon b} \sim 1$. 

Hence we fix $\a = 5$, $N = 6$, and define  $\xi^{ext, *}_\upsilon$ solving \eqref{eqrhobext} and $\xi^{ext, \Re}_\upsilon$, $\xi^{ext, \Im}_\upsilon$ solving \eqref{eqrhobexthomo} on $I_{ext}$ as
\bee
  \xi^{ext, *}_\upsilon &=& \tilde \calT^{ext}_{r_0; b, E}  \left( -r^{2+\frac{d-1}{2}}P_b \right) - \tilde \calT^{ext}_{r_0; b, E} \calV_{b, E} \xi^{ext, *}_\upsilon \\
  \xi^{ext, \Re}_\upsilon &=& - \tilde \calT^{ext}_{r_0; b, E} \calV_{b, E} \xi^{ext, \Re}_\upsilon + \psi_1^{b, E} \\
  \xi^{ext, \Im}_\upsilon &=& - \tilde \calT^{ext}_{r_0; b, E} \calV_{b, E} \xi^{ext, \Im}_\upsilon + i\psi_1^{b, E},
\eee
using the inversion of $\tilde H_{b, E}$ from Lemma \ref{leminvtildeHext} and fundamental solution $\psi_1^{b, E}$ from Definition \ref{defWKBappsolu}. From the above two estimates and \eqref{eqestpsi13Xspace}, they are well-defined by contraction mapping principle in $X^{5, N, -}_{r_0, r_1;b, E}$, and further iterate the above estimates w.r.t. $N$ implies
\be
  \| \xi^{ext, *}_\upsilon \|_{X^{5, N, -}_{r_0, r_1; b, E}} \lesssim_{N} b^{2-\frac 16} e^{-\frac \pi{2b}}, \quad \sum_{\sigma \in \{\Re, \Im\}} \| \xi^{ext, \sigma}_\upsilon \|_{X^{5, N, -}_{r_0, r_1; b, E}} \lesssim b^5,\quad \forall\, N \ge 6. \label{eqxiextest}
\ee
The boundary values are given by Lemma \ref{leminvtildeHext} (4) as
\be
\left| \begin{array}{l}
  \vec  \xi^{ext, *}_\upsilon (r_0) = \gamma^{ext, *}_\upsilon \vec \psi_3^{b, E} (r_0), \\ 
  \vec  \xi^{ext, \Re}_\upsilon (r_0) =  \vec \psi_1^{b, E} (r_0) + \gamma^{ext, \Re}_\upsilon \vec \psi_3^{b, E} (r_0), \\
   \vec  \xi^{ext, \Im}_\upsilon (r_0) = i \vec \psi_1^{b, E} (r_0) + \gamma^{ext, \Im}_\upsilon \vec \psi_3^{b, E} (r_0),
   \end{array}\right. \label{eqxiextbdry1}
\ee
% \be
% \left| \begin{array}{l}
%   \vec  \xi^{ext, *}_\upsilon (r_0) =\beta^{ext, *}_\upsilon \vec \psi_1^{b, E} (r_0) + \gamma^{ext, *}_\upsilon \vec \psi_3^{b, E} (r_0), \\ 
%   \vec  \xi^{ext, \Re}_\upsilon (r_0) = \left( 1 + \beta^{ext, \Re}_\upsilon \right) \vec \psi_1^{b, E} (r_0) + \gamma^{ext, \Re}_\upsilon \vec \psi_3^{b, E} (r_0), \\
%    \vec  \xi^{ext, \Im}_\upsilon (r_0) = \left( i + \beta^{ext, \Im}_\upsilon\right) \vec \psi_1^{b, E} (r_0) + \gamma^{ext, \Im}_\upsilon \vec \psi_3^{b, E} (r_0),
%    \end{array}\right. \label{eqxiextbdry1}
% \ee
with 
\be
  % \beta^{ext, *}_\upsilon, 
  \gamma^{ext, *}_\upsilon = O\left(b^{-3 + \frac 1{12}} e^{-\frac{\pi}{2b}} \right),\quad
  % \beta^{ext, \Re}_\upsilon, \beta^{ext, \Im}_\upsilon, 
  \gamma^{ext, \Re}_\upsilon, \gamma^{ext, \Im}_\upsilon = O(b^\frac 54).\label{eqxiextbdry2}
\ee
We emphasize that the oscillation of $\psi_1^{b, E}$ on $[r_1, \frac 4b]$ leads to the additional $b^{\frac 14}$ smallness. This is essential to guarantee that the leading order of $\upsilon_{\rho_b}$ is determined by behavior of $P_b$ on $I_{mid}$.

\mbox{}

\textit{2.2. Construction on $I_{mid}$.} 

Since 
$$\tilde H_{b, E} + \calV_{b, E} = \tilde H_{b, 1} + \calV_{b, 1} - i\upsilon,$$
we define $\xi^{mid, *}_\upsilon$ solving \eqref{eqrhobext} and $\xi^{mid, \Re, j}_\upsilon$, $\xi^{mid, \Im, j}_\upsilon$ with $j = 2, 4$ solving \eqref{eqrhobexthomo} on $I_{mid}$ through the following integral equations
\bee
  \xi^{mid, *}_\upsilon &=& \tilde \calT^{mid, G}_{x_*, r_0; b, 1}  \left( -r^{2+\frac{d-1}{2}}P_b \right) - \tilde \calT^{mid, G}_{x_*, r_0; b, 1} (\calV_{b, 1} - i\upsilon) \xi^{mid, *}_\upsilon  \\
  \xi^{mid, \Re, j}_\upsilon &=& - \tilde \calT^{mid, G}_{x_*, r_0; b, 1} (\calV_{b, 1} - i\upsilon) \xi^{mid, \Re, j}_\upsilon + \psi_j^{b, 1},\quad j = 2, 4, \\
  \xi^{mid, \Im, j}_\upsilon &=& - \tilde \calT^{mid, G}_{x_*, r_0; b, 1} (\calV_{b, 1} - i\upsilon) \xi^{mid, \Im, j}_\upsilon + i\psi_j^{b, 1},\quad j = 2, 4,
\eee
using Lemma \ref{leminvtildeHmid}. We claim that $\xi^{mid, *}_\upsilon$, $\xi^{mid, \Re, 4}_\upsilon$, $\xi^{mid, \Re, 2}_\upsilon$, $\xi^{mid, \Im, 4}_\upsilon$ and $\xi^{mid, \Im, 2}_\upsilon$ are uniquely determined by linear contraction in $\| \cdot  \|_{\calZ_3}$, $\| \cdot \|_{\calZ_0}$, $C^0_{\omega_{b, 1}^+}(I_{mid})$, $\| i\cdot \|_{\calZ_0}$, $C^0_{\omega_{b, 1}^+}(I_{mid})$ respectively with
\be
 \| f \|_{\calZ_k} = \| \Re f \|_{C^0_{\omega_{b, 1}^- r^k}(I_{mid})} + b^k \| \Im f \|_{C^0_{\omega_{b, 1}^+}(I_{mid})}, \quad k \ge 0; \label{eqdefcalZk}
\ee

For $\xi^{mid, *}_\upsilon$, notice that $\calT^{mid, G}_{x_*, r_0; b, 1}$ is $\RR$-valued integral operator (Lemma \ref{leminvtildeHmid} (1)), we apply \eqref{eqtildeTmidGest4}, \eqref{eqtildeTmidGest3} plus \eqref{eqQbasymp5}-\eqref{eqQbasymp6} and \eqref{eqQbsharp}  to obtain
\bee
  \left\| \tilde \calT^{mid, G}_{x_*, r_0; b, 1}  \left( r^{2+\frac{d-1}{2}} P_b \right)  \right\|_{\calZ_3} &\lesssim& \|r^{2+\frac{d-1}{2}} \Re P_b \|_{C^0_{\omega_{b, 1}^- r^2}(I_{mid})} + \|r^{2+\frac{d-1}{2}} \Im P_b \|_{C^0_{\omega_{b, 1}^+ r^2}(I_{mid})}\\
  &\lesssim& b^{-\frac 16} e^{-\frac{\pi}{2b}} 
\eee
From \eqref{eqbddcalVext2} and
\be e^{-2S_b(x_*)} = e^{2x_*} e^{-\frac \pi b} (1 + O(b)) \sim b^{-20} e^{-\frac \pi b}  \quad \Rightarrow \quad |\upsilon| \le b^3 r^{-3} e^{-2S_b} \quad {\rm on}\,\,I_{mid}, \label{equpsilonmid}\ee 
we derive
\bee
 \left| \Re \left( ( \calV_{b, 1} - i\upsilon)  \xi^{mid}_\upsilon\right)  \right| &\lesssim& \omega_{b, 1}^- (r^{k-2} + b^{-k}r^{-3}e^{-4S_b}) \| \xi^{mid}_\upsilon\|_{\calZ_k} \lesssim \omega_{b, 1}^- r^{k-2}  \| \xi^{mid}_\upsilon\|_{\calZ_k}\\
 \left| \Im \left( ( \calV_{b, 1} - i\upsilon)  \xi^{mid}_\upsilon\right)  \right|&\lesssim& w_{b, 1}^+ (r^{k-3} + b^{-k}r^{-2})  \| \xi^{mid}_\upsilon\|_{\calZ_k}
 \lesssim b^{-k} w_{b, 1}^+ r^{-2}  \| \xi^{mid}_\upsilon\|_{\calZ_k}.
\eee
So with \eqref{eqtildeTmidGest1} and \eqref{eqtildeTmidGest2}, we have
\[ \left\| \tilde \calT^{mid, G}_{x_*, r_0; b, 1} 
 \circ  (\calV_{b, 1} - i\upsilon) \right\|_{\calL(\calZ_k)} \lesssim x_*^{-1}. \]
That concludes the proof of contraction for $\xi^{mid, *}_\upsilon$ in $\calZ_3$, and of $\xi^{mid, \Re, 4}_\upsilon$, $\xi^{mid, \Im, 4}_\upsilon$ in $\calZ_0$, $i \calZ_0$ exploiting $\| \psi_4^{b, 1} \|_{C^0_{\omega_{b, 1}^-}(I_{mid})} \sim 1$. The contraction for $\xi^{mid, \Re, 2}_\upsilon$ and $\xi^{mid, \Im, 2}_\upsilon$ are easier using \eqref{eqtildeTmidGest1}, $|(\calV_{b, 1} - i\upsilon) \xi| \lesssim r^{-2} |\xi|$ and $\| \psi_2^{b, 1} \|_{C^0_{\omega_{b, 1}^+}(I_{mid})} \sim 1$. 
To sum up, we have 
\be \| \xi^{mid, *}_\upsilon \|_{\calZ_3} \lesssim b^{-\frac 16} e^{-\frac \pi{2b}},\quad \| \xi^{mid,\Re,4}_{\upsilon} \|_{\calZ_0} +\| i\xi^{mid,\Im,4}_{\upsilon} \|_{\calZ_0} \lesssim 1,\quad \sum_{\sigma \in \{ \Re, \Im \}} \| \xi^{mid, \sigma, 2}_\upsilon \|_{C^0_{\omega_{b, 1}^+ }(I_{mid})} \lesssim 1.
\label{eqximidest} \ee
% {\color{purple} Since for $\xi^{mid, \sigma, j}_\upsilon$, the source term is only posed on $\Re$ or $\Im$ side, and \eqref{eqtildecalVest} presents one additional order decay for the crossing part, we can gain this $r^{-1}$ on the other side.}

Moreover, since for $\sigma \in \{ \Re, \Im\}$
\bee
\xi^{mid, *}_\upsilon - \xi^{mid, *}_0  &=& \left( 1 +  \tilde \calT^{mid, G}_{x_*, r_0; b, 1} \calV_{b, 1}\right)^{-1} \left(i\upsilon \xi^{mid, *}_\upsilon \right), \\
   \xi^{mid, \sigma, j}_\upsilon - \xi^{mid, \sigma, j}_0 &=& \left( 1 +  \tilde \calT^{mid, G}_{x_*, r_0; b, 1} \calV_{b, 1}\right)^{-1} \left(i\upsilon \xi^{mid, \sigma, j}_\upsilon \right),\quad j = 2, 4, 
\eee
with the additional smallness of $\upsilon$ \eqref{equpsilonmid}, we also have the following difference estimate
\be
 % \left\| \xi^{mid, \sigma, 4}_\upsilon - \xi^{mid, \sigma, 4}_0 \right\|_{\calZ_0} + \left\| i(\xi^{mid, \sigma, 4}_\upsilon - \xi^{mid, \sigma, 4}_0 )\right\|_{\calZ_0}
  \| \xi^{mid, *}_\upsilon - \xi^{mid, *}_0 \|_{\calZ_3} + \sum_{\sigma \in \{ \Re, \Im\}}\left\| \xi^{mid, \sigma, 2}_\upsilon - \xi^{mid, \sigma, 2}_0 \right\|_{C^0_{\omega_{b, 1}^+ }(I_{mid})}  \lesssim b^3.
 \label{eqximiddiff}
\ee

The boundary values are given by Lemma \ref{leminvtildeHmid} as 
\bea
\left|\begin{array}{l}
  \vec  \xi^{mid, *}_\upsilon (r_0) = \gamma^{mid, *}_\upsilon \vec \psi_4^{b, 1} (r_0), \\
  \vec  \xi^{mid, \Re, j}_\upsilon (r_0) = \vec \psi_j^{b, 1} (r_0) + \gamma^{mid, \Re, j}_\upsilon \vec \psi_4^{b, 1} (r_0), \\
  \vec  \xi^{mid, \Im, j}_\upsilon (r_0) = i\vec \psi_j^{b, 1} (r_0) + \gamma^{mid, \Im, j}_\upsilon \vec \psi_4^{b, 1} (r_0), 
  \end{array}\right.  \label{eqximidbdry1}
  \\
  \left|\begin{array}{l}
  \vec  \xi^{mid, *}_\upsilon (x_*) = \kappa^{mid, *}_\upsilon \vec \psi_2^{b, 1}(x_*), \\ 
  \vec  \xi^{mid, \Re, j}_\upsilon (x_*) = \vec \psi_j^{b, 1} (x_*) + \kappa^{mid, \Re, j}_\upsilon \vec \psi_2^{b, 1} (x_*),  \\
  \vec  \xi^{mid, \Im, j}_\upsilon (x_*) = \vec \psi_j^{b, 1} (x_*) + \kappa^{mid, \Im, j}_\upsilon \vec \psi_2^{b, 1} (x_*).
  \end{array}\right. 
  \label{eqximidbdry2}
\eea
with 
\be
\left| \begin{array}{l}
 \gamma^{mid, *}_\upsilon = -\int_{x_*}^{r_0} \psi_2^{b, 1} \left( -r^{2 + \frac{d-1}{2}} P_b + r^{-2} O_{\calZ_3}(b^{-\frac 16} e^{-\frac{\pi}{2b}}) \right) \frac{dr}{W_{42;1}} \\
\quad  \qquad  = \int_{x_*}^{r_0} \psi_2^{b, 1} r^{2+\frac{d-1}{2}} \Re P_b \frac{dr}{W_{42;1}} + O_\RR\left(b^{-2-\frac 16} e^{-\frac{\pi}{2b}} \right)
   \\
  \quad \qquad  + i  \int_{x_*}^{r_0} \psi_2^{b, 1} r^{2+\frac{d-1}{2}} \Im P_b \frac{dr}{W_{42;1}} + i O_\RR\left(x_*^{-1} b^{-3+\frac 16} e^{-\frac{\pi}{2b}} \right) \\
   \gamma^{mid, \Re, 4}_\upsilon = -\int_{x_*}^{r_0} \psi_2^{b, 1} r^{-2} O_{\calZ_0}(1) \frac{dr}{W_{42;1}} = O_\RR (x_*^{-1}) + iO_\RR\left(b^{1+\frac 13}\right) \\
   \gamma^{mid, \Im, 4}_\upsilon = -\int_{x_*}^{r_0} \psi_2^{b, 1} ir^{-2} O_{\calZ_0}(1) \frac{dr}{W_{42;1}} = iO_\RR (x_*^{-1}) + O_\RR\left(b^{1+\frac 13}\right) \\
   \gamma^{mid, \Re, 2}_\upsilon, \gamma^{mid, \Im, 2}_\upsilon = -\int_{x_*}^{r_0} \psi_2^{b, 1} r^{-2} O_{C^0_{\omega_{b, 1}^+}(I_{mid}) }(1) \frac{dr}{W_{42;1}} = O_{\CC}(b^{1 + \frac 13}).
\end{array}\right.  \label{eqximidbdry3}
\ee
and 
\be
\left|
\begin{array}{l}
\kappa^{mid, *}_\upsilon = -\int_{x_*}^{r_0} \psi_4^{b, 1} \left( -r^{2 + \frac{d-1}{2}} P_b + r^{-2} O_{\calZ_3}(b^{-\frac 16} e^{-\frac{\pi}{2b}}) \right) \frac{dr}{W_{42;1}} \\
\quad  \qquad   = \int_{x_*}^{r_0} \psi_4^{b, 1} r^{2+\frac{d-1}{2}} \Re P_b \frac{dr}{W_{42;1}} + O_\RR\left(x_* b^{ -\frac 16} e^{2S_b(x_*)-\frac{\pi}{2b}} \right) \\
 \quad  \qquad  + i  \int_{x_*}^{r_0} \psi_4^{b, 1} r^{2+\frac{d-1}{2}} \Im P_b \frac{dr}{W_{42;1}} + i O_\RR\left(x_*^{-1} b^{-3-\frac 16} e^{-\frac{\pi}{2b}} \right) \\
   \kappa^{mid, \Re, 4}_\upsilon = -\int_{x_*}^{r_0} \psi_4^{b, 1} r^{-2} O_{\calZ_0}(1) \frac{dr}{W_{42;1}} = O_\RR (x_*^{-2}e^{2S_b(x_*)}) + iO_\RR\left(x_*^{-1}\right)\\
   \kappa^{mid, \Im, 4}_\upsilon = -\int_{x_*}^{r_0} \psi_4^{b, 1} ir^{-2} O_{\calZ_0}(1) \frac{dr}{W_{42;1}} = iO_\RR (x_*^{-2}e^{2S_b(x_*)}) + O_\RR\left(x_*^{-1}\right) \\
   \kappa^{mid, \Re, 2}_\upsilon, \kappa^{mid, \Im, 2}_\upsilon = -\int_{x_*}^{r_0} \psi_4^{b, 1} r^{-2} O_{C^0_{\omega_{b, 1}^+}(I_{mid}) }(1) \frac{dr}{W_{42;1}} = O_{\CC}(x_*^{-1}).
   \end{array}\right. \label{eqximidbdry4}
\ee
Moreover, we can use the asymptotics of $\Re P_b$, $\Im P_b$ on $[b^{-\frac 12}, r_0]$ from Proposition \ref{propQbasympref} \eqref{eqQbsharp}, asymptotics of $\Re P_b$ and boundedness of $\Im P_b$ on $[x_*, b^{-\frac 12}]$ from Proposition \ref{propQbasymp} \eqref{eqQbasymp5}-\eqref{eqQbasymp6} to evaluate
\be
\left| \begin{array}{l}
\int_{x_*}^{r_0} \psi_2^{b, 1} r^{2+\frac{d-1}{2}} \Re P_b \frac{dr}{W_{42;1}} =  \varrho_b b^\frac 13 \int_{b^{-\frac 12}}^{r_0} r^2 \psi_2^{b, 1} \psi_4^{b, 1} (1 + O(r^{-1})) \frac{dr}{W_{42;1}} +\int_{x_*}^{b^{-\frac 12}} r^2 O(b^{-\frac 16} e^{-\frac \pi{2b}})dr \\  
\qquad \qquad \sim b^{-\frac 12}e^{-\frac{\pi}{2b}} \int_{b^{-\frac 12}}^{r_0} \omega_{b, 1}^- \omega_{b, 1}^+ r^2 dr \sim b^{-3-\frac 16}e^{-\frac{\pi}{2b}}, \\
\int_{x_*}^{r_0} \psi_2^{b, 1} r^{2+\frac{d-1}{2}} \Im P_b \frac{dr}{W_{42;1}} \sim b^{-\frac 12}e^{-\frac{\pi}{2b}} \int_{b^{-\frac 12}}^{r_0} (\omega_{b, 1}^+)^2 r^2 dr \lesssim b^{-3+\frac 16} e^{-\frac \pi{2b}},\\
\int_{x_*}^{r_0} \psi_4^{b, 1} r^{2+\frac{d-1}{2}} \Re P_b \frac{dr}{W_{42;1}} \sim b^{-\frac 12}e^{-\frac{\pi}{2b}} \int_{x_*}^{r_0} (\omega_{b, 1}^-)^2 r^2 dr \sim x_*^2 b^{-\frac 16}e^{2S_b(x_*)-\frac{\pi}{2b}}, \\
\int_{x_*}^{r_0} \psi_4^{b, 1} r^{2+\frac{d-1}{2}} \Im P_b \frac{dr}{W_{42;1}} = \frac 12 \varrho_b b^\frac 13 \int_{b^{-\frac 12}}^{r_0} r^2 \psi_2^{b, 1} \psi_4^{b, 1} (1 + O(b^\frac 14)) \frac{dr}{W_{42;1}}+\int_{x_*}^{b^{-\frac 12}} r^2 O(b^{-\frac 16} e^{-\frac \pi{2b}})dr   \\  
\qquad \qquad \sim b^{-\frac 12}e^{-\frac{\pi}{2b}} \int_{b^{-\frac 12}}^{r_0} \omega_{b, 1}^- \omega_{b, 1}^+ r^2 dr \sim b^{-3-\frac 16}e^{-\frac{\pi}{2b}},
\end{array}\right.\label{eqximidbdry5}
\ee
and the difference estimate \eqref{eqximiddiff} and smallness of $\upsilon$ \eqref{equpsilonmid} imply
\be  
\left| \begin{array}{l}
 \kappa^{mid, \sigma, 2}_{\upsilon} - \kappa^{mid, \sigma, 2}_0  =  \int_{x_*}^{r_0} \psi_4^{b, 1} \left( \calV_{b, 1} (\xi^{mid, \sigma, 2}_\upsilon -\xi^{mid, \sigma, 2}_0 ) - i\upsilon \xi^{mid, \sigma, 2}_\upsilon \right) \frac{dr}{W_{42;1}} \\
 \qquad \qquad \qquad \qquad  = O_\CC( b^3 x_*^{-1}),\qquad \qquad \qquad 
{\rm for}\,\, \sigma \in \{ \Re, \Im \}.\\
 \kappa^{mid, *}_{\upsilon} - \kappa^{mid, *}_0 = \int_{x_*}^{r_0}  \psi_4^{b, 1} \left( \calV_{b, 1} (\xi^{mid, *}_\upsilon -\xi^{mid, *}_0) - i\upsilon \xi^{mid, *}_\upsilon \right) \frac{dr}{W_{42;1}} \\
\qquad \qquad \qquad \,\,\,  = O_\RR\left(x_* b^{3 -\frac 16} e^{2S_b(x_*)-\frac{\pi}{2b}} \right) + iO_\RR\left( x_*^{-1} b^{-\frac 16} e^{-\frac \pi{2b}} \right)
\end{array}\right.
\label{eqximidbdry6}
\ee

\mbox{}

\textit{2.3. Linear matching at $r_0$: inhomogeneous solution.}

 Now we construct $\xi^*_\upsilon$ solving \eqref{eqrhobext} on $[x_*, \infty)$.  Let 
\bee
\xi^*_\upsilon := \left| \begin{array}{ll}
    \xi^{ext, *}_\upsilon + d^{*;\Re}_\upsilon\xi^{ext, \Re}_\upsilon+ d^{*;\Im}_\upsilon \xi^{ext, \Im}_\upsilon  & r \in I_{ext} \\
    \xi^{mid, *}_\upsilon + c^{*;\Re}_\upsilon \xi^{mid, \Re, 2}_\upsilon + c^{*;\Im}_\upsilon \xi^{mid, \Im, 2}_\upsilon & r \in I_{mid}
\end{array}\right. 
\eee
and the real matching coefficients $c^{*;\Re}_\upsilon, c^{*;\Im}_\upsilon, d^{*;\Re}_\upsilon, d^{*;\Im}_\upsilon$ are determined by the matching conditions
\be
  \vec{\xi}^*_\upsilon(r_0 + 0) = \vec{\xi}^*_\upsilon(r_0 - 0).  \label{eqxiextmatch}
\ee
Then $\xi^*_\upsilon$ will solve \eqref{eqrhobext} on $[x_*, \infty)$ by uniqueness of the ODE solution. 
Recall \eqref{eqpsibEderiv1}-\eqref{eqpsibEderiv2} indicate for $1 \le j \le 4$, 
\bee \vec \psi_j^{b, 1+i\upsilon}(r_0) &=& \vec \psi_j^{b, 1}(r_0) + O_{\CC^2}(b^{-1}\upsilon) \\
&=&\vec \psi_j^{b, 1}(r_0) +  O_\CC(b^{-\frac 43}\upsilon) \vec \psi_4^{b, 1}(r_0) +O_\CC(b^{-\frac 43}\upsilon)  \vec \psi_2^{b, 1}(r_0), \eee
where the second equivalence follows $\vec\psi_4^{b, 1}(r_0),\vec\psi_2^{b, 1}(r_0)$ are of size $O(1)$ and $\vec \psi_4^{b, 1} \wedge \vec \psi_2^{b, 1} = \calW(\psi_4^{b, 1}, \psi_2^{b, 1}) \sim b^{\frac 13}$ from \eqref{eqWronski2}. Together with the boundary values \eqref{eqxiextbdry1}, \eqref{eqxiextbdry2}, \eqref{eqximidbdry1}, \eqref{eqximidbdry3},
and the connection formula \eqref{eqconnect}, we can parametrize the matching condition \eqref{eqxiextmatch} under the basis $\{ \vec\psi_4^{b, 1}(r_0),i \vec\psi_4^{b, 1}(r_0),\vec\psi_2^{b, 1}(r_0), i\vec\psi_2^{b, 1}(r_0)\}$ to obtain $\RR^4$-valued algebraic equations
\bee
\left[ \left( \begin{array}{cccc}
     \frac {\sqrt 3}2 & \frac 12 &  &   \\
    -\frac 12 & \frac {\sqrt 3}2 & &  \\
     \frac 14 & - \frac{\sqrt 3}{4} & -1 & \\
    \frac{\sqrt 3}{4} & \frac 14  & &  -1
\end{array} \right) + O(b^\frac 54) \right]
\left( \begin{array}{c}
    d^{*;\Re}_\upsilon \\
    d^{*;\Im}_\upsilon \\
    c^{*;\Re}_\upsilon \\
    c^{*;\Im}_\upsilon 
\end{array} \right) = 
\left( \begin{array}{c}
    \Re \gamma^{mid, *}_\upsilon   \\
    \Im \gamma^{mid, *}_\upsilon  \\
    0\\
    0
\end{array} \right) + O\left(b^{-3+\frac 1{12}}e^{-\frac{\pi}{2b}}\right)
\eee
Inverting the matrix yields 
\bee
 \left( \begin{array}{c}
    d^{*;\Re}_\upsilon \\
    d^{*;\Im}_\upsilon \\
    c^{*;\Re}_\upsilon \\
    c^{*;\Im}_\upsilon 
\end{array} \right) = 
 \left( \begin{array}{cccc}
      \frac {\sqrt 3}2 & -\frac 12 &  &   \\
    \frac 12 & \frac {\sqrt 3}2 & &  \\
      & -\frac 12 & -1 & \\
     \frac 12 &  & & -1
\end{array} \right) \left( \begin{array}{c}
    \Re \gamma^{mid, *}_\upsilon   \\
    0  \\
    0\\
    0
\end{array} \right) + O\left(b^{-3+\frac 1{12}}e^{-\frac{\pi}{2b}}\right).
\eee
Here we also used \eqref{eqximidbdry3}, \eqref{eqximidbdry5} to include $O(b^\frac 54) \cdot \Re \gamma^{mid, *}_\upsilon = O(b^{-2+ \frac 1{12}}e^{-\frac{\pi}{2b}})$ and $O(1) \cdot \Im \gamma^{mid, *}_\upsilon = O(b^{-3 + \frac 16}e^{-\frac \pi{2b}})$ in the residual. Since $\Re \gamma^{mid,*}_\upsilon \sim b^{-3-\frac 16}e^{-\frac\pi{2b}}$, we obtain the $O(b^{-3-\frac 16}e^{-\frac\pi{2b}})$ boundedness of coefficients $d^{*;\Re}_\upsilon, d^{*;\Im}_\upsilon, c^{*;\Re}_\upsilon, c^{*;\Im}_\upsilon$. This plus \eqref{eqxiextest}, \eqref{eqximidest} implies the estimate of $\xi^*_\upsilon$ in \eqref{eqxiextestcompo}.

Now consider the boundary value at $x_*$, we have
\bee
 \vec \xi^*_\upsilon (x_*) = \kappa^*_\upsilon \vec \psi_2^{b, 1} (x_*) \eee
 where from \eqref{eqximidbdry2},
 \bee
 \begin{split}
 \kappa^*_\upsilon &=  \kappa^{mid, *}_\upsilon + c^{*;\Re}_\upsilon (1 + \kappa^{mid, \Re, 2}_\upsilon) + c^{*;\Im}_\upsilon  \left( i+ \kappa^{mid, \Im, 2}_\upsilon\right) \\
 &= \kappa^{mid, *}_\upsilon + \frac 12 \Re \gamma^{mid, *}_\upsilon \left( i+ \kappa^{mid, \Im, 2}_\upsilon\right) +  O_\CC (b^{-3 + \frac 1{12}}e^{-\frac \pi{2b}})
 \end{split}
\eee
Moreover, applying \eqref{eqximidbdry3}, \eqref{eqximidbdry4}, \eqref{eqximidbdry5} and \eqref{eqximidbdry6}, we evaluate
\be
 \left|\begin{array}{l}
    \Re \kappa^*_0 = \int_{x_*}^{r_0} \psi_4^{b, 1} r^{2+\frac{d-1}{2}} \Re P_b \frac{dr}{W_{42;1}} +O_\RR(x_* b^{-\frac 16}e^{2S_b(x_*)-\frac{\pi}{2b}} ) \sim x_*^2 b^{-\frac 16}e^{2S_b(x_*)-\frac{\pi}{2b}}   \\
    \Im \kappa^*_0 = \int_{x_*}^{r_0} \left[ \psi_4^{b, 1}\Im P_b + \frac 12 \psi_2^{b, 1} \Re P_b \right] r^{2+\frac{d-1}{2}}  \frac{dr}{W_{42;1}}
    % + \frac 12(1 + \Im \kappa^{mid, \Im, 2}_0) \int_{x_*}^{r_0} \psi_2^{b, 1} r^{2+\frac{d-1}{2}} \Re P_b \frac{dr}{W_{42;1}}
    + O_\RR\left(x_*^{-1} b^{-3-\frac 16} e^{-\frac{\pi}{2b}} \right) \\
    \qquad = \varrho_b b^\frac 13 \int_{b^{-\frac 12}}^{r_0} r^2 \psi_2^{b, 1} \psi_4^{b, 1} \frac{dr}{W_{42;1}} + O_\RR\left(x_*^{-1} b^{-3-\frac 16} e^{-\frac{\pi}{2b}} \right) \\
    \qquad = \pi \varrho_b b^{-3+\frac 13} (1 + O(x_*^{-1})) \sim b^{-3 - \frac 16} e^{-\frac \pi{2b}}.
 \end{array}\right.  \label{eqkappa*upsilonasymp}
\ee
Here the integral in $\Im \kappa_0^*$ is estimated using \eqref{eqQbasymp5}-\eqref{eqQbasymp6}, \eqref{eqQbsharp} and the asymptotics from \eqref{eqWKBasymp3}, \eqref{eqWKBasymp4L} and \eqref{eqetaabs}:
\bee
  &&\int_{b^{-\frac 12}}^{r_0} r^2 \psi_2^{b, 1} \psi_4^{b, 1} \frac{dr}{W_{42;1}}\\
  &=& \int_{b^{-\frac 12}}^{\frac 2b - b^{-\frac 13}} r^2 \frac{e^{\frac{\pi i}{6}}(1 + O_\RR(b|2-br|^{-\frac 32}))}{4\pi (e^\frac{\pi i}{4} (\sqrt{2b^{-1}})^\frac 23 (1-b^2 r^2/4)^\frac 12}\cdot \frac{2^\frac 43 \pi dr}{b^\frac 13} 
 + \int_{\frac 2b - b^{-\frac 13}}^{\frac 2b} r^2 O_\RR(1) \cdot b^{-\frac 13}dr \\
  &=& \int_0^{\frac 2b} \frac{r^2 dr}{2(1-b^2 r^2/4)^\frac 12} + O_\RR(b^{-2-\frac 23}) = \pi b^{-3} (1 + O_\RR(b^{\frac 13}))
\eee

Besides, combining \eqref{eqximidbdry6} with $\Re \gamma^{mid,*}_\upsilon = \Re \gamma^{mid,*}_0 ( 1+ O_\RR(b))$ from \eqref{eqximidbdry2} and \eqref{eqximidbdry5}, we can evaluate $\kappa^*_\upsilon$ with $|\upsilon| \le b^{-3-\frac 16}e^{-\frac \pi b}$ by
\be \Re \kappa^*_\upsilon = \Re \kappa^*_0 \left(1 + O_\RR(b^3) \right),\quad \Im \kappa^*_\upsilon = \Im \kappa^*_0 \left(1 + O_\RR(b^\frac 14) \right).  
\label{equpsilonindep}
\ee
That implies the boundary value evaluation \eqref{eqxiextbdryx*} for $\xi^*_\upsilon$.

\mbox{}

\textit{2.4. Linear matching at $r_0$: homogeneous solution.}

For $\sigma = \Re, \Im$, Let 
\bee
\xi^\sigma_\upsilon =  
\left| \begin{array}{ll}
    d^{\sigma;\Re}_\upsilon \xi^{ext, \Re}_\upsilon + d^{\sigma;\Im}_\upsilon \xi^{ext, \Im}_\upsilon &   r \in I_{ext}\\
    \xi^{mid, \sigma, 4}_\upsilon + c^{\sigma,\Re}_\upsilon \xi^{mid, \Re, 2}_\upsilon + c^{\sigma,\Im}_\upsilon \xi^{mid, \Im, 2}_\upsilon   &  r \in I_{mid}
\end{array}\right.
\eee
with the real-valued matching coefficients $\{d^{\sigma;\iota}_\upsilon, c^{\sigma;\iota}_\upsilon \}_{\sigma, \iota \in \{ \Re, \Im\}}$ determined by matching $\vec \xi^\sigma_\upsilon$ at $r_0$. For $\xi^\Re_\upsilon$, similar to the matching for $\xi^*_\upsilon$, the matching condition at $r_0$ can be parametrized under the basis $\{ \vec\psi_4^{b, 1}(r_0),i \vec\psi_4^{b, 1}(r_0),\vec\psi_2^{b, 1}(r_0), i\vec\psi_2^{b, 1}(r_0)\}$ to obtain
\bee
\left[ \left( \begin{array}{cccc}
     \frac {\sqrt 3}2 & \frac 12 &  &   \\
    -\frac 12 & \frac {\sqrt 3}2 & &  \\
     \frac 14 & - \frac{\sqrt 3}{4} & -1 & \\
    \frac{\sqrt 3}{4} & \frac 14  & &  -1
\end{array} \right) + O(b^\frac 54) \right]
\left( \begin{array}{c}
    d^{\Re;\Re}_\upsilon \\
    d^{\Re;\Im}_\upsilon \\
    c^{\Re;\Re}_\upsilon \\
    c^{\Re;\Im}_\upsilon 
\end{array} \right) = 
\left( \begin{array}{c}
    1 + \Re \gamma^{mid, \Re, 4}_\upsilon   \\
    \Im \gamma^{mid, \Re, 4}_\upsilon  \\
    0\\
    0
\end{array} \right)
\eee
Invert the coefficient matrix and apply $\Im \gamma^{mid, \Re, 4}_\upsilon = O(b^\frac 43)$ from \eqref{eqximidbdry3},
\bee
% \left| \begin{array}{l}
%     d^{\Re;\Re}_\upsilon = \frac{\sqrt{3}}{2} \left(1 + \Re \gamma^{mid, \Re, 4}_\upsilon \right) + O(b^\frac 54) \\
%     d^{\Re;\Im}_\upsilon = \frac 12 \left(1 + \Re \gamma^{mid, \Re, 4}_\upsilon \right) + O(b^\frac 54) \\
%     c^{\Re;\Re}_\upsilon = O(b^\frac 54) \\
%     c^{\Re;\Im}_\upsilon = \frac 12  \left(1 + \Re \gamma^{mid, \Re, 4}_\upsilon \right) + O(b^\frac 54)
% \end{array}\right. 
 \left( \begin{array}{c}
    d^{\Re;\Re}_\upsilon \\
    d^{\Re;\Im}_\upsilon \\
    c^{\Re;\Re}_\upsilon \\
    c^{\Re;\Im}_\upsilon 
\end{array} \right) = 
 \left( \begin{array}{cccc}
      \frac {\sqrt 3}2 & -\frac 12 &  &   \\
    \frac 12 & \frac {\sqrt 3}2 & &  \\
      & -\frac 12 & -1 & \\
     \frac 12 &  & & -1
\end{array} \right) \left( \begin{array}{c}
    1 + \Re \gamma^{mid, \Re, 4}_\upsilon  \\
    0  \\
    0\\
    0
\end{array} \right) + O\left(b^{\frac 54}\right) = O(1).
\eee
As in the discussion for $\xi^*_\upsilon$, this implies the estimate for $\xi^\Re_\upsilon$ in \eqref{eqxiextestcompo}.
The boundary value at $x_*$ is 
\bee
  \vec \xi^\Re_\upsilon(x_*) = \vec \psi_4^{b, 1} (x_*) + \kappa^\Re_\upsilon \vec \psi_2^{b, 1}(x_*)
\eee
where 
\bee
  \kappa^\Re_\upsilon = \kappa^{mid, \Re, 4}_\upsilon + c_\Re^\Re  \kappa^{mid, \Re, 2}_\upsilon + c^{\Re;\Im}_\upsilon \kappa^{mid, \Im, 2}_\upsilon 
  % \\
  % &=& \kappa^{mid, \Re, 4}_\upsilon + \frac 12 \left(1 + \Re \gamma^{mid, \Re, 4}_\upsilon \right) \kappa^{mid, \Im, 2}_\upsilon + O_\CC(b^\frac 54 x_*^{-1})\\
  = O_\RR(x_*^{-2} e^{2S_b(x_*)}) + iO_\RR(x_*^{-1}),
\eee
which yields  \eqref{eqxiextbdryx*} using $\vec \psi_4^{b, 1}(x_*) = \vec \psi_2^{b, 1} e^{2S_b(x_*)}(1 + O_\RR(b))$ from Proposition \ref{propWKB} (7). 

Similarly, for $\xi^\Im_\upsilon$ we also have 
\bee
   \left( \begin{array}{c}
    d^{\Im;\Re}_\upsilon \\
    d^{\Im;\Im}_\upsilon \\
    c^{\Im;\Re}_\upsilon \\
    c^{\Im;\Im}_\upsilon 
\end{array} \right) = 
 \left( \begin{array}{cccc}
      \frac {\sqrt 3}2 & -\frac 12 &  &   \\
    \frac 12 & \frac {\sqrt 3}2 & &  \\
      & -\frac 12 & -1 & \\
     \frac 12 &  & & -1
\end{array} \right) \left( \begin{array}{c}
    0  \\
    1 + \Im \gamma^{mid, \Im, 4}_\upsilon  \\
    0\\
    0
\end{array} \right) + O\left(b^{\frac 54}\right).
\eee
and 
\bee
  \vec \xi^\Im_\upsilon(x_*) = i\vec \psi_4^{b, 1} (x_*) + \kappa^\Im_\upsilon \vec \psi_2^{b, 1}(x_*)
\eee
with 
\bee
  \kappa^\Im_\upsilon 
  = \kappa^{mid, \Im, 4}_\upsilon + c^{\Im;\Re}_\upsilon  \kappa^{mid, \Re, 2}_\upsilon + c^{\Im;\Im}_\upsilon \kappa^{mid, \Im, 2}_\upsilon 
  % - \frac 12 \left(1 + \Im \gamma^{mid, \Im, 4}_\upsilon \right) \kappa^{mid, \Re, 2}_\upsilon + O_\CC(b^\frac 54 x_*^{-1})
  = i O_\RR(x_*^{-2} e^{2S_b(x_*)}) + O_\RR(x_*^{-1}).
\eee
That concludes the proof of \eqref{eqxiextestcompo} and \eqref{eqxiextbdryx*}.

% Moreover, the difference estimate \eqref{eqximiddiff} implies
% \[ \Im \kappa^\Re_\upsilon = \Im \kappa^\Re_0 + O_\RR(b^3x_*^{-1}),\quad \Re \kappa^\Im_\upsilon = \Re \kappa^\Im_0 + O_\RR(b^3x_*^{-1}).   \]
% }

\mbox{}

\underline{3. Nonlinear matching at $x_*$.}

From Step 1 and 2, we have constructed two solution families on \eqref{eqrhobext}: $(\cdot)^{\frac{d-1}{2}} (\Sigma_{\upsilon,\a} + i \Theta_{\upsilon, \a})$ on $[0, x_*]$ for $\a, \upsilon$ in the range \eqref{eqrangeupsilona}, and $\xi^*_\upsilon + \mu_\Re \xi^\Re_\upsilon + \mu_\Im \xi^\Im_\upsilon$ on $[x_*,\infty)$ for $\upsilon$ in \eqref{eqrangeupsilona} and $\mu_\Re, \mu_\Im \in \RR$. In this final step, we look for $(\upsilon, \a, \mu_\Re, \mu_\Im) \in \RR^4$ satisfying
\be
  |\upsilon| \le x_*^{\frac 12} b^{-3}e^{-\frac \pi{b}}, \quad | \a | \le e^{-2x_*} x_*^{\frac 52}, \quad |\mu_\Re| \le b^{-\frac 16} x_*^{\frac 72} e^{-\frac \pi{2b}},\quad |\mu_\Im| \le b^{-3-\frac 16}e^{-\frac{3\pi}{2b}+2x_*}x_*^{-\frac 12},
  \label{equamuri}
\ee
such that the two solutions above match at $x_*$, namely
\be \left| \begin{array}{l} 
x_*^{\frac{d-1}{2}}  \left( \Sigma_{\upsilon, \a} + i\Theta_{\upsilon, \a} \right) (x_*) = \left( \xi^*_\upsilon + \mu_\Re \xi^\Re_\upsilon + \mu_\Im \xi^\Im_\upsilon \right) (x_*) \\ 
\pa_r \left[ (\cdot)^{\frac{d-1}{2}}  \left( \Sigma_{\upsilon, \a} + i\Theta_{\upsilon, \a} \right)\right] (x_*) = \pa_r \left( \xi^*_\upsilon + \mu_\Re \xi^\Re_\upsilon + \mu_\Im \xi^\Im_\upsilon \right) (x_*).
\end{array} \right. \label{eqmatchrhob}
\ee
We define $r^{\frac{d-1}{2}} \rho_b$ to be this global solution. 

Recall  the asymptotics of $\vec A, \vec \frakE, \vec D$ from Lemma \ref{lemLpm} (5), and those of $\vec \psi_2^{b, 1},  \vec \psi_4^{b, 1}$ from Proposition \ref{propWKB} (5) and \eqref{eqetaconv} that with $\kappa_\psi = 2^{-\frac 76} \pi^{-\frac 12}$, for $k = 0, 1$, 
\be
 \pa_r^k \psi_4^{b, 1}(x_*) = (-1)^{k} \kappa_\psi b^{\frac 16}e^{\frac \pi {2b}-x_*}\left(1 + O_\RR(b) \right), \quad \pa_r^k \psi_2^{b, 1}(x_*) = \kappa_\psi b^{\frac 16}e^{-\frac \pi {2b}+x_*}\left(1 + O_\RR(b) \right).
 \label{eqpsix+1}
\ee
Using \eqref{eqrhobbdryint1}, \eqref{eqrhobbdryint2} and \eqref{eqxiextbdryx*}, we can write the matching condition \eqref{eqmatchrhob} with \eqref{equamuri} as
\bee
  \left| \begin{array}{l}
        \mathfrak{c}_{x_*} (2\kappa_A)^{-1} e^{-x_*}\left(1 + c_D x_*^{-1}\right) + \a \kappa_A e^{x_*} 
        =  \Re \kappa_0^* \kappa_\psi b^\frac 16 e^{-\frac \pi{2b}+x_*}  + \mu_\Re \kappa_\psi b^\frac 16 e^{\frac \pi{2b}-x_*}  + O\left(e^{-x_*} x_*^{\frac 32} \right) \\
         -\mathfrak{c}_{x_*} (2\kappa_A)^{-1} e^{-x_*}\left(1 + c_D x_*^{-1}\right) + \a \kappa_A e^{x_*} 
        =  \Re \kappa_0^* \kappa_\psi b^\frac 16 e^{-\frac \pi{2b}+x_*} -\mu_\Re \kappa_\psi b^\frac 16 e^{\frac \pi{2b}-x_*} + O\left(e^{-x_*} x_*^{\frac 32} \right) \\
        \upsilon\frac{ \| xQ\|_{L^2}^2}{4 \kappa_Q|\SS^{d-1}|} e^{x_*} \left(1 + c_\frakE x_*^{-1} \right) = 
        \Im \kappa_0^* \kappa_\psi b^\frac 16 e^{-\frac \pi{2b}+x_*} + \mu_\Im \kappa_\psi b^\frac 16 e^{\frac \pi{2b}-x_*} +  O\left( b^{-3} e^{-\frac\pi b} x_*^{-\frac 32}  e^{x_*} \right)  \\
        \upsilon\frac{ \| xQ\|_{L^2}^2}{4 \kappa_Q|\SS^{d-1}|} e^{x_*} \left(1 + c_\frakE x_*^{-1} \right)  = 
        \Im \kappa_0^* \kappa_\psi b^\frac 16 e^{-\frac \pi{2b}+x_*} - \mu_\Im \kappa_\psi b^\frac 16 e^{\frac \pi{2b}-x_*} +  O\left( b^{-3} e^{-\frac\pi b} x_*^{-\frac 32}  e^{x_*} \right)
  \end{array}\right.
\eee
Firstly we match the leading order and normalize the coefficients. Let 
\[ 
 \left| \begin{array}{l}
      \upsilon = \frac{4 \kappa_Q|\SS^{d-1}|} { \| xQ\|_{L^2}^2}\cdot \left( \Im \kappa_0^* \kappa_\psi b^\frac 16 e^{-\frac \pi{2b}} +  b^{-3} e^{-\frac \pi b} x_*^{-1} \tilde \upsilon \right) \\
      \a = \kappa_A^{-1} x_*^2 e^{-2x_*} \tilde \a \\
      \mu_\Re = \mathfrak{c}_{x_*} (2\kappa_A)^{-1} \kappa_\psi^{-1} b^{-\frac 16} e^{-\frac \pi {2b}} + \kappa_\psi^{-1} b^{-\frac 16} e^{-\frac \pi{2b}}x_*^2 \tilde \mu_\Re \\
      \mu_\Im = \kappa_\psi^{-1} b^{-3-\frac 16}e^{-\frac {3\pi}{2b}+2x_*} x_*^{-1} \tilde \mu_\Im
 \end{array}\right.
\]
and we further require 
\[ |\tilde \upsilon|, |\tilde \a|, |\tilde \mu_\Re|, |\tilde \mu_\Im| \le x_*^{\frac 14} \]
so that \eqref{equamuri} is still satisfied. Plugging these new variables and dividing the first two and last two equations by $e^{-x_*}x_*^2$, $b^{-3}e^{-\frac\pi b}x_*^{-1} e^{x_*}$ respectively, the matching condition now turns into
\bee
  \left( \begin{array}{cccc}
     1 & -1 & &   \\
     1 & 1 & &    \\
     & & 1 & -1   \\
     & & 1 & 1  
  \end{array} \right) 
   \left( \begin{array}{c}
       \tilde \a \\
       \tilde \mu_\Re \\
     \tilde \upsilon\\
       \tilde \mu_\Im
  \end{array} \right) 
  = \left( \begin{array}{c}
     -\mathfrak{c}_{x_*} (2\kappa_A)^{-1} c_D x_*^{-3} + \Re \kappa^*_0 \kappa_\psi b^\frac 16 e^{-\frac \pi{2b}  + 2x_*} x_*^{-2}  \\
     \mathfrak{c}_{x_*} (2\kappa_A)^{-1} c_D x_*^{-3} + \Re \kappa^*_0 \kappa_\psi b^\frac 16 e^{-\frac \pi{2b}  + 2x_*} x_*^{-2}  \\
     - \Im \kappa^*_0 \kappa_\psi b^{3+\frac 16}e^{\frac \pi{2b}} c_\frakE  \\
     - \Im \kappa^*_0 \kappa_\psi b^{3+\frac 16}e^{\frac \pi{2b}} c_\frakE
  \end{array} \right) + O(x_*^{-\frac 12}),
\eee
where the constant bound of $O(x_*^{-\frac 12})$ is independent of $\tilde \upsilon, \tilde \a, \tilde \mu_\Re, \tilde \mu_\Im$, and the vector on the RHS is of size $O(1)$ from \eqref{eqestfrakcx*} and \eqref{eqdefkappa*upsilon}. Therefore, inverting the matrix coefficient on LHS, the Brouwer's fixed point theorem implies the existence of solution $|(\tilde \upsilon, \tilde \a, \tilde \mu_\Re, \tilde \mu_\Im)| \le C$ with $C$ determined by the leading order vector on the RHS. 

In particular, we have 
\be 
\begin{split}
&\upsilon = \frac{4 \kappa_Q \kappa_\psi|\SS^{d-1}|} { \| xQ\|_{L^2}^2} \Im \kappa_0^* b^\frac 16 e^{-\frac \pi{2b}} (1 + O(x_*^{-1})) =\frac{2^\frac 56 \pi^\frac 12 \kappa_Q}{\int_0^\infty Q^2 r^{d+1} dr} \varrho_b b^{-3+\frac 12}e^{-\frac{\pi}{2b}} (1 + O(x_*^{-1})), \\
&| \a | \lesssim e^{-2x_*} x_*^2, \quad |\mu_\Re| \lesssim b^{-\frac 16} x_*^3 e^{-\frac \pi{2b}},\quad |\mu_\Im| \lesssim b^{-3-\frac 16}e^{-\frac{3\pi}{2b}+2x_*}x_*^{-1}.\end{split} \label{eqrhobpara} 
\ee
This verifies \eqref{equpsilonasymp} with the asymptotic of $\varrho_b$ from \eqref{eqvarrhob}. The interior bounds \eqref{eqasympXibReint} and \eqref{eqasympXibImint} come from \eqref{eqasymprhobSigma} and \eqref{eqasymprhobTheta} with the range of $\a$ above, and the exterior bounds \eqref{eqasympXibext}, \eqref{eqasympXibRemid} and \eqref{eqasympXibImmid} will follow \eqref{eqxiextestcompo} and the above bound of $\mu_\Re$, $\mu_\Im$. 
\end{proof}

\mbox{}

\subsection{Construction of bifurcated eigenmodes. }

\begin{proof}[Proof of Proposition \ref{propbifeigen}]

We look for $\upsilon \in \RR$ and $Z$ solving $(\calH_b - i \upsilon + ibs_c) Z = 0$ with $Z$ having the symmetry $Z_1 = \bar Z_2$. In particular, we assume 
\be \upsilon \notin \{ bs_c,  -2b + bs_c \} \label{equpsilonnondeg} \ee
so that this gives a different eigen pair from those for $\xi_{0, b}$ or $\xi_{1, b}$. We choose $s_{c;{\rm eig}}^{(2)} \le s_{c;{\rm eig}}^{(1)}$ from Lemma \ref{lemrhob}, in particular satisfying \eqref{eqsceig1}, and may further shrink it for smallness of $b(d, s_c)$. 

\mbox{}

\underline{1. Ansatz and formulation of equation.}

Recall $\rho_b$ from Lemma \ref{lemrhob} and $\xi_{j,b}$ for $j = 0, 1, 2$ from Section \ref{sec22}. Denote 
\[ \xi_{3,b} = \frac 18 \left(\begin{array}{c} e^{-i\frac{br^2}{4}}\rho_b \\e^{i\frac{br^2}{4}} \bar{\rho}_b \end{array}\right) \]
Then \eqref{eqrhob} indicates
\[\left( \calH_b - i \upsilon_{\rho_b} + i bs_c\right) \xi_{3,b} = -i \xi_{2,b}.  \]

Let $\mathfrak{u} = \upsilon - bs_c$. 
Take the ansatz 
\be 
Z = \xi_{0,b} - \mathfrak{u} \xi_{1,b} + \mathfrak{u}(2b+\mathfrak{u}) \xi_{2,b} + (2b-\mathfrak{u})(2b+\mathfrak{u}) \mathfrak{u} \cdot \xi_{3,b} + \mathfrak{Z} 
\label{eqeigenansatz}
\ee
then with the algebraic relations \eqref{eqQbalgrel} and above, the eigen equation  reads as
\bee
  (\calH_b - i\mathfrak{u})\mathfrak{Z} = i(2b+\mathfrak{u}) \mathfrak{u} \left[ \frac{s_c}{2} \eta_b + (\mathfrak{u} - \upsilon_{\rho_b} + bs_c)(2b- \mathfrak{u}) \xi_{3,b} \right]
\eee
Let $\mathfrak{Z} =:  \left( \begin{array}{c} \mathfrak{h} \\ \overline{\mathfrak{h}} \end{array} \right) $, we obtain the $
\CC$-valued scalar equation
\be
 \left(\Delta_b - 1 - i (\mathfrak{u} + bs_c) + W_{1, b} \right) \mathfrak{h} + W_{2, b} \overline{\mathfrak{h}} = i(2b+\mathfrak{u})\mathfrak{u} \left[ \frac{s_c}{2} Q_b -\frac 18 (\mathfrak{u} - \upsilon_{\rho_b} + bs_c)(\mathfrak{u} - 2b)e^{-i\frac{br^2}{4}} \rho_b \right]. \label{eqhh}
\ee
Let 
\[ \left| \begin{array}{l}
     \omega_1(\upsilon) = \frac{s_c}{2}(2b+\mathfrak{u})\mathfrak{u} =\frac{s_c}{2} (2b-bs_c + \upsilon) (\upsilon - bs_c), \\
     \omega_2(\upsilon) = -\frac{1}{4s_c} (\mathfrak{u} - \upsilon_{\rho_b} + bs_c)(\mathfrak{u} - 2b) = -\frac{1}{4s_c} (\upsilon - \upsilon_{\rho_b})(\upsilon - 2b - bs_c),
\end{array}\right. \]
and $\omega_1(\upsilon) \neq 0$ by \eqref{equpsilonnondeg}, so let
 \[  h = e^{i\frac{br^2}{4}}  \omega_1(\upsilon)^{-1} \mathfrak{h}, \]
then the equation becomes
\be
 \left(\Delta + \frac{b^2 r^2}{4} - 1 - i \upsilon + W_{1, b} \right)  h + e^{i\frac{br^2}{2}}W_{2, b} \bar h = i\left( P_b + \omega_2(\upsilon)\rho_b \right). \label{eqhhh}
\ee
We now solve \eqref{eqhhh} for $h$ and $\upsilon$ using a similar strategy as the proof of Lemma \ref{lemrhob}.

\mbox{}

\underline{2. Interior solution.}

Denote 
 \[ h = \Sigma + i\Theta. \]
 Then \eqref{eqhhh} is equivalent to
 \be
\left| \begin{array}{l}
     -L_{+, b} \Sigma + (G +\upsilon) \Theta = -\left(  \Im P_b + \omega_2 \Im \rho_b \right) \\
     -L_{-, b} \Theta + (G - \upsilon) \Sigma =   \Re P_b +  \omega_2 \Re \rho_b
\end{array}\right. \label{eqSigmaTheta2}
 \ee
where the potential
\[ G = \Im \left( e^{i\frac{br^2}{2}}W_{2, b} \right). \]
% with the pointwise bound from Proposition \ref{propQbasymp}
% \[ |G| \lesssim b s_c e^{-(p-3)r} \la r \ra^{-\frac{(d-1)(p-1)}{2}}. \]
% \mbox{}

Let 
\be
  x_+ = \frac 1{10}|\log b|,\quad x_- = 10|\log b|. 
\ee
Recall $Q_1$ and $\rho$ from \eqref{eqdefQ1Q2rho}. From \eqref{eqsolitondecay} and \eqref{eqLpmest1}, we have their exponential decay 
\be
\la r \ra^{-1}|Q_1|  + \la r \ra^{-3} |\rho| \lesssim \la r \ra^{-\frac{d-1}{2}}e^{-r}.\label{eqestQ1Q2rho}
\ee
Using the algebraic identity \eqref{eqrhobQ} and controlling the residuals by \eqref{eqestQ1Q2rho}, \eqref{eqSigma0rho}, \eqref{eqasympXibReint}, 
and Proposition \ref{propQbasymp}, we find
\bee
 \left( \Re \rho_b, Q\right)_{L^2(B_{x_\pm})} = \frac 12 \| xQ\|_{L^2(\RR^d)} + O(b^\frac 13+ x_\pm^4 e^{-2x_\pm} ) = \frac 12 \| xQ\|_{L^2(\RR^d)} + O(b^\frac 1{10}),\\
 \left( \Re P_b, Q\right)_{L^2(B_{x_\pm})} = \| Q \|_{L^2(\RR^d)}^2 + O(b^\frac 13+ x_\pm^4 e^{-2x_\pm} )  = \| Q \|_{L^2(\RR^d)}^2 + O(b^\frac 1{10}).
\eee
and from Lemma \ref{lemLpmb}, 
\[ \left( V_{-, b} L_{-, b}^{-1} \Re \rho_b,  Q \right)_{L^2(B_{x_\pm})},  \left( V_{-, b} L_{-, b}^{-1} \Re P_b,  Q \right)_{L^2(B_{x_\pm})} = O(b^\frac 13)  \]
Hence we define 
\bee
C_\pm := -\frac{ \left(\Re P_b - V_{-, b} L_{-, b}^{-1} \Re P_b, Q \right)_{L^2(B_{x_\pm})} }{\left(  \Re \rho_b - V_{-, b} L_{-, b}^{-1} \Re \rho_b, Q \right)_{L^2(B_{x_\pm})}} = -\frac{2\| Q \|_{L^2(\RR^d)}^2}{ \| xQ \|_{L^2(\RR^d)}^2} + O(b^\frac 1{10}) \sim -1
\eee
and further denote $\upsilon^*_\pm$ as
\be \begin{split} \upsilon^*_\pm &= \frac{(\upsilon_{\rho_b} + 2b + bs_c) \pm \sqrt{(2b+bs_c-\upsilon_{\rho_b})^2 - 16 s_c C_\pm} }{2} \\ &= \left| \begin{array}{ll}
    2b+O(b^{-2}e^{-\frac \pi b}) & {\rm for} \,\, \upsilon_+^*  \\
    \upsilon_{\rho_b} + O(b^{-2}e^{-\frac \pi b}) &  {\rm for} \,\, \upsilon_-^*
\end{array}\right. 
\end{split} \label{eqmupm*}
 \ee
such that $\omega_2(\upsilon^*_\pm) = C_\pm$, or equivalently
\be \left( (1 - V_{-,b} L_{-, b}^{-1})  (\Re P_b +  \omega_2(\upsilon^*_\pm) \Re \rho_b), Q \right)_{L^2(B_{x_\pm})} = 0.  \label{eqomega2*choice}\ee
For simplicity, we denote
\[  \omega^*_{2, \pm} = \omega_2(\upsilon^*_{\pm}). \]

\mbox{}

Let 
\[ \tilde \upsilon_\pm = \upsilon - \upsilon^*_\pm. \]
For $\upsilon \approx \upsilon^*_\pm$, we introduce two variables $\gamma_\pm \in \RR$ to characterize the invertibility of $L_{-, b}$
\be
       \gamma_\pm =  |\SS^{d-1}|^{-1} \left( (1 -V_{-,b} L_{-, b}^{-1} )   (\Re P_b   + \omega_2(\upsilon^*_\pm  + \tilde \upsilon_\pm)\Re \rho_b), Q \right)_{L^2(B_{x_\pm})}, \label{eqdefgammapm}
\ee
so that 
\be  \omega_2(\upsilon) - \omega^*_{2,\pm} = \frac{ |\SS^{d-1}| \gamma_\pm}{(\Re P_b - V_{-,b} L_{-, b}^{-1} \Re \rho_b, Q)_{L^2(B_{x_\pm})}}. \label{eqomega2diff}
\ee
Besides, for any fixed $|\gamma_\pm|\le 1$, the variable
$\tilde \upsilon_\pm$ can be solved from the quadratic equation \eqref{eqdefgammapm} as
\be \tilde \upsilon_\pm (\gamma_\pm) = \frac{\mp \Delta_\pm \pm \sqrt{\Delta_\pm^2 -\frac{16s_c  |\SS^{d-1}| \gamma_\pm}{((1- V_{-,b} L_{-, b}^{-1}) \Re \rho_b, Q)_{L^2(B_{x_\pm})}} }
 }{2} \sim \mp \gamma_\pm b^{-2} e^{-\frac \pi b}
 % = \pm \frac{4s_c \gamma_\pm \left( 1 + O(b^{-2} s_c \gamma_\pm) \right)}{(\Re \rho_b, Q)_{L^2(B_{x_\pm}) }\Delta_\pm }, 
 \label{eqtildeupsilonpm} \ee
 where $\Delta_\pm = \sqrt{(2b+bs_c-\upsilon_{\rho_b})^2 - 16 s_c C_\pm} = 2b + O(b^{-3}e^{-\frac \pi b})$.

We will prove in the rest of this step that for 
\be
 \left| \begin{array}{l}
 |\gamma_+| \le x_+^\frac 72 e^{-2x_+}, \\
 |\a_+| \le b x_+^{\frac 92} e^{-2x_+},
 \end{array}\right., \quad {\rm or} \quad 
 \left| \begin{array}{l}
 |\gamma_-| \le x_-^\frac 72 e^{-2x_-}, \\
 |\a_-| \le x_-^{\frac 12} b^{-4} e^{-\frac\pi b},
 \end{array}\right.,
\label{eqgammaarange+}
\ee
we can find solution $(\Sigma_{\pm; \gamma_\pm, \a_\pm}, \Theta_{\pm; \gamma_\pm, \a_\pm}) \in (C^\infty_{rad} (B_{x_{\pm}}))^2$ solving \eqref{eqSigmaTheta2} on $B_{x_\pm}$ with boundary value at $x_\pm$ as
\be
\left| \begin{array}{l}
 \vec \Sigma_{-;\gamma_-, \a_-}(x_-) = \a_- \vec A (x_-) + O\left(b^{-3 - \frac 23} e^{-\frac\pi b} x_-^{\frac 12 -\frac{d-1}{2}} e^{x_-}\right), \\ 
 \vec \Theta_{-;\gamma_-, \a_-}(x_-) = x_-^4 \kappa_{f_{0, -}}  \vec Q (x_-) + \gamma_- \vec \frakE(x_-) + O\left( x_-^{\frac{15}2 - \frac{d-1}{2}} e^{-3x_-}  \right)
 \end{array}\right.
  \label{eqST-intbdry}
\ee
and 
\be
\left| \begin{array}{l}
 \vec \Sigma_{+;\gamma_+, \a_+}(x_+) =- 2b x_+^5 \kappa_{F_{+,0}} \vec D(x_+) + \a_+ \vec A (x_+) + O\left(b x_+^{\frac{7}{2}-\frac{d-1}{2}} e^{-x_+} \right), \\ 
 \vec \Theta_{+;\gamma_+, \a_+}(x_+) = x_+^4 \kappa_{f_{0, +}}  \vec Q (x_+) + \gamma_+ \vec \frakE(x_+) + O\left( x_+^{\frac{15}2 - \frac{d-1}{2}} e^{-3x_+}  \right)
 \end{array}\right.
  \label{eqST-intbdry2}
\ee
where 
\be \kappa_{f_{0, \pm}}, \kappa_{F_{+,0}} = O(1)  \label{eqfF01pm} \ee
are independent of $\gamma_\pm$ and $\a_\pm$. 

\mbox{}

\textit{2.1. Construction of $(\Sigma_{-;\gamma_-, \a_-}, \Theta_{-;\gamma_-, \a_-})$.}

Apply Lemma \ref{lemLpmb} to invert $L_{\pm, b}$ on $B_{x_-}$ in \eqref{eqSigmaTheta2} and add multiple of $A_b$ to obtain
\be
  \left| \begin{array}{l}
       \Sigma = L_{+, b;x_-}^{-1} \left( (G + \upsilon) \Theta\right) + L_{+, b;x_-}^{-1} \left(  \Im P_b + \omega_2 \Im \rho_b \right) + \a_- A_b\\
       \Theta = L_{-, b}^{-1} \left( (G - \upsilon) \Sigma \right) -  L_{-, b}^{-1} \left( \Re P_b +  \omega_2 \Re \rho_b \right) 
\end{array}
\right..\label{eqSTint1}
\ee
We claim that with the parameter $(\gamma_-, \a_-)$ within \eqref{eqgammaarange+}, this is a linear contraction in the following space
\bee 
\| (\Sigma, \Theta)\|_{\mathbb{Z}_-} := \left( x_-^{\frac 12} b^{-4}e^{-\frac{\pi}{b}}\right)^{-1} \| \Sigma\|_{Z_{+, 0; x_-}} + \|\Theta\|_{Z_{-, 4; x_-}}
\eee
Indeed, first notice that by \eqref{eqomega2diff} and \eqref{eqtildeupsilonpm}, 
\[ |\tilde \upsilon_-| \lesssim b^{-2}e^{-\frac \pi b}x_-^\frac 72 e^{-2x_-}, \,\, |\omega_2(\upsilon) - \omega^*_{2, -}| \lesssim   x_-^\frac 72 e^{-2x_-}\,\, \Rightarrow \,\, \upsilon \sim b^{-3}e^{-\frac{\pi}{2b}},\,\, \omega_2 (\upsilon) \sim -1.\]
Further from $x_- = 10|\log b|$ and the boundedness in Lemma \ref{lemLpmb}, Lemma \ref{lemrhob} and Proposition \ref{propQbasymp}, we have the control of source terms 
\bee
  \| \a_- A_b \|_{\tilde Z_{+, 0;x_-}} \lesssim x_-^\frac 12 b^{-4} e^{-\frac \pi b}, \quad 
  \| L_{+, b;x_-}^{-1} \left(  \Im P_b + \omega_2 (\upsilon)\Im \rho_b \right) \|_{\tilde Z_{+, 0;x_-}} \lesssim x_- b^{-3} e^{-\frac \pi b} \\ 
  \| L_{-, b}^{-1} \left(  \Re P_b + \omega^*_{2, -} \Re \rho_b \right) \|_{\tilde Z_{-, 4;x_-}} \lesssim 1 \\
  \| (\omega_2(\upsilon) - \omega_{2,-}^*) L_{-,b}^{-1} \Re \rho_b \|_{\tilde Z_{-, 4;x_-}} \lesssim x_-^{\frac 72} e^{-2x_-} \cdot x_-^{-4} e^{2x_-} = x_-^{-\frac 12}, 
  \eee
and the linear contraction estimate
  \bee
  \| L_{+, b;x_-}^{-1} \left( (G+\upsilon)\Theta \right) \|_{\tilde Z_{+,0;x_-}} \lesssim ( bs_c + |\upsilon|) \| \Theta \|_{Z_{-,4;x_-}} \lesssim b^{-3}e^{-\frac \pi b} \| \Theta \|_{Z_{-,4;x_-}}, \\
  \| L_{-, b}^{-1} ((G-\upsilon)\Sigma) \|_{\tilde Z_{-,4;x_-}} \lesssim (bs_c + |\upsilon|) e^{2x_-} \| \Sigma \|_{Z_{+,0;x_-}} \lesssim b^2 \| \Sigma \|_{Z_{+,0;x_-}},
\eee
namely 
\[ \left\|\left(   L_{+, b;x_-}^{-1} \left( (G+\upsilon)\Theta \right), L_{-, b}^{-1} ((G-\upsilon)\Sigma)\right) \right\|_{\ZZ_-} \lesssim x_-^{-\frac 12} b \| (\Sigma, \Theta) \|_{\ZZ_-}. \]
Therefore, 
\be
   \Sigma_{-;\gamma_-, \a_-} = \a_- A_b + O_{\tilde Z_{+,0;x_-}}(x_- b^{-3} e^{-\frac \pi b}), \qquad \Theta_{-;\gamma_-, \a_-} = F_{-,\gamma_-} + O_{\tilde Z_{-,4;x_-}}(b^{-3} e^{-\frac \pi b}), \label{eqSigmaTheta-int}
\ee
where we denote
\be
\begin{split}
 f_{0,\pm} = -L_{-, b}^{-1} \left( \Re P_b +  \omega_{2,\pm}^* \Re \rho_b \right),\quad f_1 = -L_{-, b}^{-1} \Re \rho_b,  \\
 F_{\pm, \gamma_\pm} = -L_{-, b}^{-1} \left( \Re P_b + \omega_2(\upsilon^*_\pm + \tilde \upsilon_\pm) \Re \rho_b \right) = f_{0,\pm}+ \left(\omega_2(\upsilon) - \omega_{2,\pm}^* \right) f_1.
 \end{split} \label{eqdeff0f1F}
\ee
The smoothness of $\Sigma_{-;\gamma_-, \a_-}, \Theta_{-;\gamma_-, \a_-}$ on $B_{x_-}$ comes from using \eqref{eqSigmaTheta2} to improve regularity near the origin and the smoothness of ODE solution on $r \in (0, x_-)$. 

\mbox{}

Now we evaluate the boundary value of $F_{\pm,\gamma_\pm}$ at $x_\pm$. From the bound of $\Re \rho_b$ \eqref{eqasympXibReint} and choice of $\omega^*_{2,\pm}$ \eqref{eqomega2*choice}, we define and compute
\bee
\kappa_{f_{0, \pm}} &:=& x_\pm^{-4} \int_0^{x_\pm} \left[ \int_0^s  \left[ \left(1 - V_{-, b} L_{-, b}^{-1}\right) \left( \Re P_b +  \omega_{2,\pm}^* \Re \rho_b \right) \right](\tau) Q(\tau) \tau^{d-1} d\tau \right] \frac{ds}{Q^2 s^{d-1}}\\
&=& x_\pm^{-4}  \int_0^1  \left[ \int_0^s  \left[ \left(1 - V_{-, b} L_{-, b}^{-1}\right) \left( \Re P_b +  \omega_{2,\pm}^* \Re \rho_b \right) \right](\tau) Q(\tau) \tau^{d-1} d\tau \right] \frac{ds}{Q^2 s^{d-1}}\\
&-& x_\pm^{-4}  \int_1^{x_\pm}  \left[ \int_s^{x_\pm}  \left[ \left(1 - V_{-, b} L_{-, b}^{-1}\right) \left( \Re P_b +  \omega_{2,\pm}^* \Re \rho_b \right) \right](\tau) Q(\tau) \tau^{d-1} d\tau \right] \frac{ds}{Q^2 s^{d-1}} \\
&=& O(1) 
\eee
so that 
\be f_{0, \pm} (x_\pm) = x_\pm^4 \kappa_{f_{0, \pm}} \vec Q(x_\pm). \label{eqSTbdryint1}
 \ee
Also from Lemma \ref{lemLpmb} and \eqref{eqomega2diff}
\be (\omega_2(\upsilon) - \omega_{2, \pm}^*) f_1 = \gamma_\pm \tilde{\frakE} + O_{\tilde Z_{-, 4; x_\pm}}(\gamma_\pm).    \label{eqSTbdryint2}
\ee
Combined with \eqref{eqSigmaTheta-int} and \eqref{eqestAb}, we obtain \eqref{eqST-intbdry}. 

\mbox{}

\textit{2.2. Construction of $(\Sigma_{+;\gamma_+, \a_+}, \Theta_{+;\gamma_+, \a_+})$.} 

Similarly, we first invert $L_{\pm, b}$ on $B_{x_+}$ for \eqref{eqSigmaTheta2} to obtain the integral equation similar to \eqref{eqSTint1}, then further plug the $\Theta$ equation into the $\Sigma$ one to get
\be
\left|\begin{array}{l}
 \Sigma =
  - L_{+, b;x_+}^{-1} \circ (G + \upsilon) L_{-, b}^{-1} \left( \Re P_b +  \omega_2 \Re \rho_b \right) \\
  \qquad +  L_{+, b;x_+}^{-1} \circ (G + \upsilon) L_{-, b}^{-1}  \circ (G-\upsilon) \Sigma + L_{+, b;x_+}^{-1} \left(   \Im P_b + \omega_2 \Im \rho_b \right) + \a_+ A_b \\
   \Theta = L_{-, b}^{-1} \left( (G - \upsilon) \Sigma \right) -  L_{-, b}^{-1} \left( \Re P_b +  \omega_2 \Re \rho_b \right) 
\end{array}\right. \label{eqSTint2}
\ee

We claim that with $(\gamma_+, \a_+)$ within \eqref{eqgammaarange+}, the $\Sigma$ equation is a linear contraction in $Z_{-, 5;x_+}$. Indeed, by \eqref{eqomega2diff} and \eqref{eqtildeupsilonpm}, 
\be \upsilon =  2b + O(b^{-2}e^{-\frac \pi b}),\quad \omega_2(\upsilon) \sim -1.  \label{equpsilonasymp+} \ee
Then  Lemma \ref{lemLpmb}, Lemma \ref{lemrhob}, Proposition \ref{propQbasymp} and $x_+ = \frac{|\log b|}{10}$ imply
\bee
  \| L_{+, b;x_+}^{-1} \left(  \Im P_b + \omega_2 (\upsilon)\Im \rho_b \right) \|_{\tilde Z_{-, 5;x_+}} &\lesssim& b^{-3} e^{-\frac \pi b} e^{2x_+}, \quad \| \a_+ A_b \|_{\tilde Z_{-, 5;x_+}} \lesssim bx_+^{-\frac 12}, \\ 
  \| L_{-, b}^{-1} \left(  \Re P_b + \omega^*_{2, +} \Re \rho_b \right) \|_{\tilde Z_{-, 4;x_+}} &\lesssim& 1 \\
  \| (\omega_2(\upsilon) - \omega_{2,-}^*) L_{-,b}^{-1} \Re \rho_b \|_{\tilde Z_{-, 4;x_-}} &\lesssim& x_+^{\frac 72} e^{-2x_+} \cdot x_+^{-4} e^{2x_+} = x_+^{-\frac 12},  \\
   \|  L_{+,b;x_+}^{-1} \circ G \|_{Z_{-, 4;x_+} \to \tilde Z_{-, 5;x_+}}  &\lesssim& b^{-\frac 15}e^{-\frac \pi b},
   \quad \| \upsilon L_{+,b;x_+}^{-1}  \|_{Z_{-, 4;x_+} \to \tilde Z_{-, 5;x_+}} \lesssim b,  \\
  \| L_{-,b}^{-1} \circ (G-\upsilon) \|_{Z_{-, 5;x_+} \to \tilde Z_{-, 4;x_+}} &\lesssim& x_+^{-4} e^{2x_+} b \le b^{\frac 45}.
\eee
Consequently,
\bee
  \|   - L_{+, b;x_+}^{-1} \circ (G + \upsilon) L_{-, b}^{-1} \left( \Re P_b +  \omega_2 \Re \rho_b \right) + L_{+, b;x_+}^{-1} \left(   \Im P_b + \omega_2 \Im \rho_b \right) + \a_+ A_b \|_{\tilde Z_{-,5;x_+}} \lesssim b, \\
  \|  L_{+, b;x_+}^{-1} \circ (G + \upsilon) L_{-, b}^{-1}  \circ (G-\upsilon) \Sigma  \|_{\tilde Z_{-,5;x_+}} \lesssim b^{\frac 95} \| \Sigma  \|_{Z_{-,5;x_+}}.
\eee
That verifies the contraction of $\Sigma$ in $Z_{-,5;x_+}$. With the above estimate for \eqref{eqSTint2} and \eqref{eqestAb}, we have (recalling $F_{+, \gamma_+}$ from \eqref{eqdeff0f1F})
\be \Sigma_{+;\gamma_+,\a_+} =  \a_+ A + \upsilon L_{+,b;x_+}^{-1} F_{+,\gamma_+} + O_{\tilde Z_{-,5;x_+}}(b^\frac 43 x_+^{-\frac 12}),\quad \Theta_{+;\gamma_+,\a_+} = F_{+,\gamma_+} + O_{\tilde Z_{-,4;x_+}}(b^{\frac 95}). \label{eqSigmaTheta+int}
\ee
The smoothness of $(\Sigma_{+;\gamma_+,\a_+}, \Theta_{+;\gamma_+,\a_+})$ in $B_{x_+}$ follows similarly by improving regularity via \eqref{eqSigmaTheta2} and smoothness of ODE solution. 

Next, we evaluate the boundary value of $\upsilon L_{+,b;x_+}^{-1} F_{+,\gamma_+}$. From the definition of $L_{+,b;x_+}$ in Lemma \ref{lemLpmb}, we have the functional identity 
\[ L_{+,b;x_+}^{-1} = L_{+;x_+}^{-1} \left( 1 - V_{+,b} L_{+,b;x_+}^{-1}\right). \]
So using \eqref{eqSTbdryint2}, we have
\[ L_{+,b;x_+}^{-1} F_{+,\gamma_+} = L_{+;x_+}^{-1}(f_{0, +} + \gamma_+ \tilde{\frakE}) + O_{\tilde Z_{-,5;x_+}}( |\gamma_+| + b^\frac 13) \]
Define 
\be \kappa_{F_{+, 0}} = x_+^{-5} \int_0^{x_+} f_{0,+} A s^{d-1} ds = O(1),\quad \kappa_{f_{1, +}} = e^{-2x_+} \int_{0}^{x_+} \tilde{\frakE} A s^{d-1} ds \sim 1.  \ee
Then 
\[ \overrightarrow{ L_{+;x_+}^{-1}(f_{0, +} + \gamma_+ \tilde{\frakE})(x_+)} = -\left( x_+^5 \kappa_{F_{+,0}} + \gamma_+  e^{2x_+} \kappa_{f_{1, +}}\right) \vec D(x_+) = -\left( x_+^5 \kappa_{F_{+,0}} + O(x_+^\frac 72)\right)\] 
The asymptotics \eqref{eqST-intbdry2} follows this and \eqref{eqSTbdryint1}, \eqref{eqSTbdryint2}, \eqref{equpsilonasymp+}, \eqref{eqSigmaTheta+int}. 

\mbox{}

\underline{3. Exterior solution.}

Let 
$$\xi := r^{\frac{d-1}{2}} h. $$ 
Then \eqref{eqhhh} turns into 
\[\left( \pa_r^2 - \frac{(d-1)(d-3)}{4r^2} +\frac{b^2 r^2}{4} - 1 - i\upsilon + W_{1, b} \right) \xi + e^{i\frac{br^2}{2}} W_{2, b} \bar \xi = i r^{\frac{d-1}{2}} \left( P_b + \omega_2(\upsilon) \rho_b \right) \]
namely 
\be \left( \tilde H_{b, E} + \calV_{b, E}\right) \xi = i r^{\frac{d-1}{2}} \left( P_b + \omega_2(\upsilon) \rho_b \right)
\label{eqgpmext}
\ee
with $E = 1 + i\upsilon$. 

Recall the definition of $\upsilon^*_\pm$, $\gamma_\pm$ and $\tilde \upsilon_\pm(\gamma_\pm)$ from \eqref{eqmupm*}, \eqref{eqdefgammapm} and \eqref{eqtildeupsilonpm}. For 
\be |\gamma_\pm | \le x_\pm^\frac 72 e^{-2x_\pm}, \label{eqgammapmrange}
\ee
we will construct solutions $\xi^*_{\pm, \gamma_\pm}$ of \eqref{eqgpmext} and solutions $\xi^\Re_{\pm, \gamma_\pm}$, $\xi^\Im_{\pm, \gamma_\pm}$ of the homogeneous version for \eqref{eqgpmext} on $[x_\pm, \infty)$ with $\upsilon = \upsilon^*_\pm + \tilde \upsilon_\pm(\gamma_\pm)$. They are smooth on $[x_\pm, \infty)$ and satisfy 
\be  \| \xi^*_{\pm, \gamma_\pm} \|_{X^{7, N, -}_{\frac 2b, \frac{2+\sqrt b}{b}; b, E}} \lesssim_{N} b^{3-\frac 16} e^{-\frac \pi{2b}},\quad \sum_{\sigma \in \{\Re, \Im\}} \| \xi^\sigma_{\pm, \gamma_\pm}\|_{X^{7, N, -}_{\frac 2b, \frac{2+\sqrt b}{b}; b, E}} \lesssim b^7,\quad \forall\, N \ge 8, \label{eqxipmextest} \ee
and the boundary values are 
\begin{align}
&\left| \begin{array}{l}
 \vec \xi^*_{-,\gamma_-}(x_-) =\left( \kappa^*_{-,0} + O_\RR(b^{-4+\frac 1{12}} e^{-\frac \pi b}) + i O_\RR(x_-^3 b^{1-\frac 16}e^{2S_b(x_-) - \frac \pi{2b}} ) \right) \vec \psi_2^{b, 1} (x_-),\\  
\xi^\Re_{-,\gamma_-}(x_-) = \vec \psi_4^{b, 1} (x_-)  \left (1 + O_\RR(x_-^{-2}) + i O_\RR(x_-^{-1} e^{-2S_b(x_-)})\right) ,\\
\xi^\Im_{-,\gamma_-}(x_-) =\vec \psi_4^{b, 1} (x_-)  \left (i + iO_\RR(x_-^{-2}) + O_\RR(x_-^{-1} e^{-2S_b(x_-)})\right) 
\end{array} \right. \label{eqxi-bdryx-}\\
&\left| \begin{array}{l}
 \vec \xi^*_{+,\gamma_+}(x_+) = \left( \kappa^*_{+,0} + O_\CC(b^{3-\frac 16} e^{-\frac{\pi}{2b}} x_+^4 e^{2S_b(x_+)} ) \right) \vec \psi_2^{b, E_{+,0}} (x_+),\\  
\xi^\Re_{+,\gamma_+}(x_+) =e^{i\Im\eta_{b, E_{+,0}}(x_+)}  \vec \psi_4^{b, E_{+,0}} (x_+) \left( 1 +O_\RR(x_+^{-2}) + iO_\RR(bx_+^{-1})  \right) \\
% \qquad \qquad \,\,\, + \left(  O_\RR \left(x_+^{-2}e^{2S_b(x_+)}\right) + iO_\RR  \left(bx_+^{-1}e^{2S_b(x_+)}\right) \right)e^{-i\Im\eta_{b, E_{+,0}}(x_+)}  \vec \psi_2^{b, E_{+,0}}(x_+),\\
\xi^\Im_{+,\gamma_+}(x_+) =  e^{i\Im\eta_{b, E_{+,0}}(x_+)} \vec \psi_4^{b, E_{+,0}} (x_+)\left( i +iO_\RR(x_+^{-2}) + O_\RR(bx_+^{-1})  \right) 
% \\
% \qquad \qquad\,\,\,+ \left( iO_\RR \left(x_+^{-2}e^{2S_b(x_+)}\right) + O_\RR  \left(bx_+^{-1}e^{2S_b(x_+)}\right) \right)e^{-i\Im\eta_{b, E_{+,0}}(x_+)}  \vec \psi_2^{b, E_{+,0}}(x_+)
\end{array}\right. \label{eqxi+bdryx+}
\end{align}
where $E_{+, 0} = 1 + i \upsilon^*_+$ and 
\bea
  \kappa^*_{-, 0} &=& O_\RR\left(b^{-4 - \frac 16} e^{-\frac \pi{2b}}\right) + iO_\RR\left(x_-^{3} b^{-\frac 16} e^{2S_b(x_-) - \frac \pi {2b}}\right) \label{eqxi-bdryx-2}\\
   e^{i\Im\eta_{b, E_{+,0}}(x_+)}\kappa^*_{+,0} &=&  O_\RR\left(b^{\frac 56} x_+^4 e^{2S_b(x_+)-\frac \pi {2b}}\right)
 +  iO_\RR \left(b^{-\frac 16}  x_+^3 e^{2S_b(x_+)-\frac \pi {2b}}\right) \label{eqxi+bdryx+2}
\eea
are independent of $\gamma_\pm$ within the range of \eqref{eqgammapmrange}.

As in the proof of Lemma \ref{lemrhob}, we define the exterior and middle regions as.
\be I_{ext} = \left[ r_0, \infty \right),\quad I_{mid, \pm,} = \left[x_\pm, r_0\right],\quad {\rm where\,\,} r_0 = 2b^{-1}.  \ee

\mbox{}

\textit{3.1. Construction on $I_{ext}$.} 

This part is almost the same as in Step 2.1 of Proof of Lemma \ref{lemrhob}. We let $r_1 = 2b^{-1} + b^{-\frac 12}$, and define $\xi^{ext, *}_{\pm,\gamma_\pm}$ solving \eqref{eqgpmext} and $ \xi^{ext, \Re}_{\pm,\gamma_\pm}$, $ \xi^{ext, \Im}_{\pm,\gamma_\pm}$ solving the homogeneous version on $I_{ext}$ as 
\bee
  \xi^{ext, *}_{\pm,\gamma_\pm} &=& \tilde \calT^{ext}_{r_0; b, E} \left[ i r^{\frac{d-1}{2}} \left( P_b + \omega_2 \rho_b \right) \right]  - \tilde \calT^{ext}_{r_0; b, E} \calV_{b, E} \xi^{ext, *}_{\pm,\gamma_\pm} \\
  \xi^{ext, \Re}_{\pm,\gamma_\pm} &=& - \tilde \calT^{ext}_{r_0; b, E} \calV_{b, E} \xi^{ext, \Re}_{\pm,\gamma_\pm} + \psi_1^{b, E} \\
  \xi^{ext, \Im}_{\pm,\gamma_\pm}  &=& - \tilde \calT^{ext}_{r_0; b, E} \calV_{b, E} \xi^{ext, \Im}_{\pm,\gamma_\pm} + i\psi_1^{b, E},
\eee
From Lemma \ref{leminvtildeHext} and \eqref{eqasympXibext}, and $\upsilon > bs_c$, $|\omega_2(\upsilon)| \sim 1$ with $\gamma_\pm$ satisfying \eqref{eqgammapmrange}, we have for $\a = 7$, $N \ge 8$, 
 \[  \left\|  \tilde \calT^{ext}_{r_0; b, E} \left[ i r^{\frac{d-1}{2}} \left( P_b + \omega_2(\upsilon) \rho_b \right) \right] \right\|_{ X^{7, N, -}_{r_0, r_1; b, E}}\lesssim_N b^{3-\frac 16}  e^{-\frac{\pi}{2b}}.  \]
 Then Lemma \ref{leminvtildeHext} and \eqref{eqbddcalVext} implies contraction in $X^{7, N, -}_{r_0, r_1;b, E}$ for $N = 8$ and thereafter
\be
  \| \xi^{ext, *}_{\pm,\gamma_\pm} \|_{X^{7, N, -}_{r_0, r_1; b, E}} \lesssim_{N} b^{3-\frac 16} e^{-\frac \pi{2b}}, \quad \sum_{\sigma \in \{\Re, \Im\}} \| \xi^{ext, \sigma}_{\pm,\gamma_\pm} \|_{X^{7, N, -}_{r_0, r_1; b, E}} \lesssim b^7,\quad \forall\, N \ge 8. \label{eqxiextestpm}
\ee
The boundary values are given by Lemma \ref{leminvtildeHext} as
\be
\left| \begin{array}{l}
  \vec  \xi^{ext, *}_{\pm,\gamma_\pm} (r_0) =
  % \beta^{ext, *}_{\pm,\gamma_\pm} \vec \psi_1^{b, E} (r_0) +
  \gamma^{ext, *}_{\pm,\gamma_\pm} \vec \psi_3^{b, E} (r_0), \\ 
  \vec  \xi^{ext, \Re}_{\pm,\gamma_\pm} (r_0) = 
  % \left( 1 + \beta^{ext, \Re}_{\pm,\gamma_\pm} \right) 
  \vec \psi_1^{b, E} (r_0) + \gamma^{ext, \Re}_{\pm,\gamma_\pm} \vec \psi_3^{b, E} (r_0), \\
   \vec  \xi^{ext, \Im}_{\pm,\gamma_\pm} (r_0) =
   % \left( i + \beta^{ext, \Im}_{\pm,\gamma_\pm}\right)
   i\vec \psi_1^{b, E} (r_0) + \gamma^{ext, \Im}_{\pm,\gamma_\pm} \vec \psi_3^{b, E} (r_0),
   \end{array}\right. \label{eqxiextbdry3}
\ee
with 
\be
  % \beta^{ext, *}_{\pm,\gamma_\pm},
  \gamma^{ext, *}_{\pm,\gamma_\pm} = O\left(b^{-4 + \frac 1{12}} e^{-\frac{\pi}{2b}} \right),\quad 
  % \beta^{ext, \Re}_{\pm,\gamma_\pm}, \beta^{ext, \Im}_{\pm,\gamma_\pm},
  \gamma^{ext, \Re}_{\pm,\gamma_\pm}, \gamma^{ext, \Im}_{\pm,\gamma_\pm} = O(b^\frac 54).\label{eqxiextbdry4}
\ee

\mbox{}

\textit{3.2. Construction on $I_{mid, -}$ and linear matching: "-" case.}

Since now $|\upsilon| \lesssim b^{-3}e^{-\frac \pi{b}} \ll b^3 r^{-3} e^{-2S_b(r)}$ on $I_{mid, -}$, this part is almost the same as in Step 2.2-2.4 of Proof of Lemma \ref{lemrhob}, which we omit. The outcome \eqref{eqxipmextest} (for "-" case), \eqref{eqxi-bdryx-} and \eqref{eqxi-bdryx-2} are counterparts of \eqref{eqxiextestcompo}, \eqref{eqxiextbdryx*} and \eqref{eqdefkappa*upsilon}.

\mbox{}

\textit{3.3. Construction on $I_{mid,+}$ and linear matching: "+" case.}

Notice that for any fixed $\gamma_+$  satisfying \eqref{eqgammapmrange}, 
\[ \tilde H_{b, E} + \calV_{b, E} = \tilde H_{b, E_{+,0}} + \calV_{b, E_{+,0}} - i\tilde \upsilon_+(\gamma_+) \]
where 
\[ E = E_{+, \gamma_+} = 1 + i (\upsilon^*_+ + \tilde \upsilon_+(\gamma_+)),\quad E_{+, 0} = 1 + i\upsilon^*_+.\]
We define $\xi^{mid, *}_{+, \gamma_+}$ solving \eqref{eqgpmext} and $\xi^{mid, \sigma, j}_{+, \gamma_+}$ for $j = 2, 4$, $\sigma \in \{ \Re, \Im \}$ solving its homogeneous on $I_{mid,+}$ by
\bee
  \xi^{mid, *}_{+, \gamma_+} &=&  \tilde \calT^{mid, G}_{x_+, r_0; b, E_{+, 0}}  \left[ i r^{\frac{d-1}{2}} \left( P_b + \omega_2 \rho_b \right) \right] -  \tilde \calT^{mid, G}_{x_+, r_0; b, E_{+, 0}} \left( \calV_{b, E_{+, 0}} - i \tilde \upsilon_+ \right) \xi^{mid, *}_{+, \gamma_+} \\
    \xi^{mid, \Re, 4}_{+, \gamma_+} &=& -  \tilde \calT^{mid, G}_{x_+, r_0; b, E_{+, 0}} \left( \calV_{b, E_{+, 0}} - i \tilde \upsilon_+ \right) \xi^{mid, \Re, 4}_{+, \gamma_+} + e^{i\Im \eta_{b, E_{+, 0}}(x_+)} \psi_4^{b, E_{+, 0}}, \\
     \xi^{mid, \Re, 2}_{+, \gamma_+} &=& -  \tilde \calT^{mid, G}_{x_+, r_0; b, E_{+, 0}} \left( \calV_{b, E_{+, 0}} - i \tilde \upsilon_+ \right)\xi^{mid, \Re, 2}_{+, \gamma_+} + \psi_2^{b, E_{+, 0}}, \\
  \xi^{mid, \Im, 4}_{+, \gamma_+} &=& -  \tilde \calT^{mid, G}_{x_+, r_0; b, E_{+, 0}} \left( \calV_{b, E_{+, 0}} - i \tilde \upsilon_+ \right) \xi^{mid, \Im, 4}_{+, \gamma_+} + i  e^{i\Im \eta_{b, E_{+, 0}}(x_+)}  \psi_4^{b, E_{+, 0}},\\
   \xi^{mid, \Im, 2}_{+, \gamma_+} &=& - \tilde \calT^{mid, G}_{x_+, r_0; b, E_{+, 0}} \left( \calV_{b, E_{+, 0}} - i \tilde \upsilon_+ \right) \xi^{mid, \Im, 2}_{+, \gamma_+} + i \psi_2^{b, E_{+, 0}}.
\eee
The phase correction is introduced to make $\xi^{mid, \sigma, 4}_{+, \gamma_+}$ has distinguishable real and imaginary parts near $r = x_+$. 
 
Without distinguishing real and imaginary parts, we can apply the bound in Lemma \ref{leminvtildeHmid}, Lemma \ref{lemrhob} and $|(\calV_{b, E_{+, 0}} - i\tilde \upsilon_+) f | \lesssim r^{-2}|f|$ from \eqref{eqbddcalVext2} and \eqref{eqtildeupsilonpm}, to obtain the existence of solutions by contraction in $C^0_{\omega_{b, E_{+, 0}}^- r^4}(I_{mid,+})$, $C^0_{\omega_{b, E_{+, 0}}^-}(I_{mid,+})$, $C^0_{\omega_{b, E_{+, 0}}^+}(I_{mid,+})$, $C^0_{\omega_{b, E_{+, 0}}^-}(I_{mid,+})$, $C^0_{\omega_{b, E_{+, 0}}^+}(I_{mid,+})$, respectively.  
We claim the following refined estimates
\bea
  \|\Im \xi^{mid, *}_{+, \gamma_+} \|_{C^0_{\omega_{b, 1}^- r^4}(I_{mid,+})} \lesssim b^{-\frac 16 }e^{-\frac{\pi}{2b}}, \quad \| \Re \xi^{mid, *}_{+, \gamma_+} \|_{C^0_{\omega_{b, 1}^- r^5}(I_{mid,+})} \lesssim b^{1-\frac 16 }e^{-\frac{\pi}{2b}}, \label{eqreg+*mid}\\
  \|\Re \xi^{mid, \Re, 4}_{+, \gamma_+} \|_{C^0_{\omega_{b, 1}^- }(I_{mid,+})} \lesssim 1, \quad \| \Im \xi^{mid, \Re, 4}_{+, \gamma_+} \|_{C^0_{\omega_{b, 1}^- r}(I_{mid,+})} \lesssim b. \label{eqreg+1Remid}
\\
  \|\Im \xi^{mid, \Im, 4}_{+, \gamma_+} \|_{C^0_{\omega_{b, 1}^- }(I_{mid,+})} \lesssim 1, \quad \| \Re \xi^{mid, \Im, 4}_{+, \gamma_+} \|_{C^0_{\omega_{b, 1}^- r}(I_{mid,+})} \lesssim b. \label{eqreg+1Immid} \\
  \| \xi^{mid, \Re, 2}_{+, \gamma_+} \|_{C^0_{\omega_{b, 1}^+}(I_{mid, +})} + \| \xi^{mid, \Im, 2}_{+, \gamma_+} \|_{C^0_{\omega_{b, 1}^+}(I_{mid, +})} \lesssim 1. \label{eqreg+2Immid}
\eea
and the difference estimate
\bea
  \| \xi^{mid, *}_{+, \gamma_+} - \xi^{mid, *}_{+, 0} \|_{C^0_{\omega_{b, 1}^- r^4}(I_{mid, +})} \lesssim b^{3-\frac 16}e^{-\frac \pi {2b}} \label{eqg+diff} 
\eea
Indeed, from \eqref{eqetaRe} and \eqref{eqomegapm}, we have $\omega_{b, 1}^\pm \sim \omega_{b, E_{+,0}}^\pm$ on $I_{mid,+}$; so with the contraction estimate above, it remains to check the improved estimates of $\Re \xi^{mid, *}_{+, \gamma_+}$, $\Re \xi^{mid, \Im, 4}_{+, \gamma_+}$ and $\Im \xi^{mid, \Re, 4}_{+, \gamma_+}$ on $\left[ x_+, (2b)^{-1} \right]$. Note that the difference estimate \eqref{eqg+diff} easily follows from the smallness $|\tilde \upsilon_+| \lesssim b^{-2} e^{-\frac \pi b}$ by \eqref{eqtildeupsilonpm} and \eqref{eqgammaarange+}.

We will derive linear contraction estimate for $\Re \xi^{mid, *}_{+, \gamma_+} \big|_{ [x_+, (2b)^{-1}]}$ in $C^0_{\omega_{b, 1}^- r^5}([x_+, (2b)^{-1}])$ and treat $\xi^{mid, *}_{+, \gamma_+} \big|_{[(2b)^{-1}, 2b^{-1}]}$ and $\Im \xi^{mid, *}_{+, \gamma_+} \big|_{ [x_+, (2b)^{-1}]}$ as source terms. Denote 
\[S = ir^\frac{d-1}{2}\left (P_b + \omega_2 \rho_b\right) - (\calV_{b, E_{+, 0}} - i\tilde \upsilon_+) \xi^{mid, *}_{+, \gamma_+}.\] 
We expand $\tilde \calT^{mid, G}_{x_+, r_0; b, E_{+, 0}}$ and compute for $r \in [x_+, (2b)^{-1}]$
\bea
 \Re \xi^{mid, *}_{+, \gamma_+}(r) &=&  \Re \left[ - \psi_4^{b,  E_{+,0}} \int_{x_+}^r \psi_2^{b,  E_{+,0}} S \frac{ds}{W_{42; E_{+,0}}} -    \psi_2^{b,  E_{+,0}} \int_r^{r_0}    \psi_4^{b,  E_{+,0}} S \frac{ds}{  W_{42;  E_{+,0}}} \right] \nonumber\\
 &=& \Re \Bigg[ -\frac{e^{S_b(r)}}{\left(1 - b^2r^2 /4 \right)^{\frac 14}} \int_{x_+}^r \frac{e^{-S_b(s)} e^{i\Im(\eta_{b, E_{+,0}}(s) - \eta_{b, E_{+,0}}(r))}}{2\left(1 - b^2s^2 /4 \right)^{\frac 14}} \left( 1 + O_\CC(b)\right) Sds \nonumber\\
  &-&\frac{e^{-S_b(r)}}{\left(1 - b^2r^2 /4 \right)^{\frac 14}} \int_r^{(2b)^{-1}} \frac{e^{S_b(s)} e^{i\Im(\eta_{b, E_{+,0}}(r) - \eta_{b, E_{+,0}}(s))}}{2\left(1 - b^2s^2 /4 \right)^{\frac 14}} \left( 1 + O_\CC(b)\right) Sds \Bigg] \nonumber \\
 &+& O_\RR\left( \omega_{b, 1}^+(r) \int_{(2b)^{-1}}^{r_0} \omega_{b, 1}^-(s) |S| b^{-\frac 13} ds \right) \label{eqRezetamid*+contract}
  % 
 % - \frac{e^{-S_b(r)}}{2\left(1 - \frac{b^2 r^2}{4}\right)^{\frac 14}} \int_{r}^{\frac 2b} \frac{e^{S_b(s)} e^{-i\Im(\eta(s) - \eta(r))}}{\left(1 - \frac{b^2 s^2}{4}\right)^{\frac 14}} \left( 1 + O_\CC(\la b^{-\frac 23}|bs-2| \ra^{-\frac 32})\right) Sds \Bigg]
\eea
where we exploited the asymptotics of $\psi_j^{b, E}$ from Proposition \ref{propWKB} (5), and used \eqref{eqetaRe} to evaluate $\Re \eta_{b, E_{+,0}}(r) = -S_b(r) (1 + O_\RR(b))$ for $r \le (2b)^{-1}$. 

To show \eqref{eqRezetamid*+contract} is a linear contraction for $\Re \xi^{mid, *}_{+, \gamma_+} \big|_{ [x_+, (2b)^{-1}]}$, we start with controlling the $O_\CC(b)$ part and the last term. Since $\|S\|_{C^0_{\omega_{b, 1}^- r^3}(I_{mid, +})} \lesssim b^{-\frac 16}{e^{-\frac \pi{2b}}}$ from Lemma \ref{lemrhob} and the contraction estimate for $\xi^{mid, *}_{+, \gamma_+}$ above, 
we have for $r \in [x_+, (2b)^{-1}]$,
\bee
&&\left| \omega_{b, 1}^+(r) \int_{(2b)^{-1}}^{r_0} \omega_{b, 1}^-(s) |S| b^{-\frac 13} ds \right|
  \lesssim \omega_{b, 1}^+(r) e^{-2S_b(\frac{1}{2b})} b^{-4-\frac 16} e^{-\frac{\pi}{2b}} \lesssim  \omega_{b, 1}^-(r) r^5 b^{1-\frac 16} e^{-\frac{\pi}{2b}}\\
&&\left| -\frac{e^{S_b(r)}}{\left(1 - b^2r^2 /4 \right)^{\frac 14}} \int_{x_+}^r \frac{e^{-S_b(s)} }{2\left(1 - b^2s^2 /4 \right)^{\frac 14}}  Sds \right|  
  + \left| \frac{e^{-S_b(r)}}{\left(1 - b^2r^2 /4 \right)^{\frac 14}} \int_r^{\frac{1}{2b}} \frac{e^{S_b(s)} }{2\left(1 - b^2s^2 /4 \right)^{\frac 14}} Sds \right|\\
  &\lesssim& \omega_{b, 1}^- r^4 \left(  b^{-\frac 16} e^{-\frac{\pi}{2b}} + \|\Re \xi^{mid, *}_{+, \gamma_+} \|_{C^0_{\omega_{b, 1}^- r^5}([x_+, (2b)^{-1}])}  \right)
\eee

Next, for $r, s \in [x_+, (2b)^{-1}]$, 
\bee
\left|\Im (\eta_{b, E_{+, 0}} (r) - \eta_{b, E_{+,0}}(s)) \right|
= \left| \frac 2b \int_{\frac{bs}{2}}^{\frac{br}{2}} \Im \left[ (E_{+, 0}-\tau^2)^\frac 12 \right] d\tau \right| \lesssim b|r-s|.
\eee
Then for $r \in [x_+, (2b)^{-1}]$, with \eqref{eqasympXibReint}, \eqref{eqasympXibImint} and \eqref{eqbddcalVext2}, \eqref{eqtildeupsilonpm}, we compute
\bee
  &&|\Re S(r)| + |\Im (\eta_{b, E_{+, 0}} (r) - \eta_{b, E_{+,0}}(s))| |\Im S(r)|\\
   &\lesssim& \left(b^{-3}e^{-\frac{\pi}{2b}}e^{-S_b(r)} + r^{-2} |\Re \xi^{mid, *}_{+, \gamma_+}(r)| + b^2 \cdot e^{-\frac{\pi}{2b}} e^{S_b(r)}r^4\right) \\
   &+& b|r-s| \cdot \left( e^{-\frac{\pi}{2b}} e^{S_b(r)}r^3 + b^2 |\Re \xi^{mid, *}_{+, \gamma_+}(r)| + r^{-2} \cdot e^{-\frac{\pi}{2b}} e^{S_b(r)}r^4 \right) \\
  &\lesssim& b e^{-\frac{\pi}{2b}} e^{S_b(r)} r^4 + b^{\frac 16}e^{S_b(r)} r^3 \|\Re \xi^{mid, *}_{+, \gamma_+}\|_{C^0_{\omega_{b, 1}^- r^5}([x_+, (2b)^{-1}])} 
\eee
and hence
\bee
 &&\left| \Re \left[ -\frac{e^{S_b(r)}}{\left(1 - b^2 r^2/4\right)^{\frac 14}} \int_{x_+}^r \frac{e^{-S_b(s)} e^{i\Im(\eta_{b, E_{+,0}}(s) - \eta_{b, E_{+,0}}(r))}}{2\left(1 -b^2 s^2/4\right)^{\frac 14}} Sds \right] \right| \\
 &\lesssim& \left| e^{S_b(r)} \int_{x_+}^r  e^{-S_b(s)} \left( |\Re S(r)| + |\Im (\eta_{b, E_{+,0}}(s)-\eta_{b, E_{+,0}}(r))| |\Im S(r)| \right) ds  \right| \\
 &\lesssim& \omega_{b, 1}^- r^5 \left( b^{1 - \frac 16} e^{-\frac{\pi}{2b}} + r^{-1} \|\Re \xi^{mid, *}_{+, \gamma_+}\|_{C^0_{\omega_{b, 1}^- r^5}([x_+, (2b)^{-1}])}  \right) 
 \eee
and similarly
\bee
 &&\left| \Re \left[   - \frac{e^{-S_b(r)}}{\left(1 - b^2 r^2/4\right)^{\frac 14}} \int_{r}^{\frac{1}{2b}} \frac{e^{S_b(s)} e^{-i\Im(\eta_{b, E_{+,0}}(s) - \eta_{b, E_{+,0}}(r))}}{2\left(1 - b^2 s^2/2\right)^{\frac 14}}S ds  \right] \right| \\
 % &\lesssim&  \left| e^{-S_b(r)} \int_{r}^{\frac{1}{2b}}  e^{S_b(s)} \left( |\Re S| + |\Im (\eta_{b, E_{+,0}}(s)-\eta_{b, E_{+,0}}(r))| |\Im S| \right) ds  \right| \\
 &\lesssim&  \omega_{b, 1}^- r^4 \left( b^{1 - \frac 16} e^{-\frac{\pi}{2b}} + r^{-1} \|\Re \xi^{mid, *}_{+, \gamma_+}\|_{C^0_{\omega_{b, 1}^- r^5}([x_+, (2b)^{-1}])}  \right) 
\eee

 Plugging these estimates into \eqref{eqRezetamid*+contract}, we obtain
\bee
 \| \Re \xi^{mid, *}_{+, \gamma_+}\|_{C^0_{\omega_{b, 1}^- r^5}([x_+, (2b)^{-1}])} \lesssim b^{1 - \frac 16} e^{-\frac{\pi}{2b}} + x_+^{-1} \| \Re \xi^{mid, *}_{+, \gamma_+}\|_{C^0_{\omega_{b, 1}^- r^5}([x_+, (2b)^{-1}])},
\eee
which implies  \eqref{eqreg+*mid}. For \eqref{eqreg+1Immid} (and \eqref{eqreg+1Remid}), we notice that the asymptotics of $\psi_4^{b, E}$ from Proposition \ref{propWKB} indicates that $i e^{i\Im \eta_{b, E}(x_+)}  \psi_4^{b, E}$ satisfies \eqref{eqreg+1Immid}, and then we can similarly obtain the linear contraction part to conclude the bound for $\Re \xi^{mid, \Im, 4}_{+, \gamma_+}$. So is the estimate for $\Im \xi^{mid, \Re, 4}_{+, \gamma_+}$ in \eqref{eqreg+1Remid}.  

\mbox{}

Thereafter, we can evaluate the boundary data for $j = 2, 4$ and $\sigma \in \{ \Re, \Im\}$ via Lemma \ref{leminvtildeHmid} and the estimates \eqref{eqreg+*mid}-\eqref{eqg+diff}:
\bee
 \vec \xi^{mid,*}_{+,\gamma_+}(r_0) &=& O_\CC (b^{-4 - \frac 16} e^{-\frac\pi{2b}}) \vec \psi_4^{b, E_{+, 0}}(r_0),\\
 \vec \xi^{mid,*}_{+,\gamma_+}(x_+) &=& \left(\kappa^{mid,*}_{+,0} + O_\CC(b^{3-\frac 16} e^{-\frac{\pi}{2b}} x_+^4 e^{2S_b(x_+)} ) \right) \vec \psi_2^{b, E_{+, 0}}(x_+) \\
 \vec \xi^{mid, \sigma, j}_{+,\gamma_+}(r_0) &=& e^{\frac{j-2}{2}i\Im\eta_{b, E_{+,0}}(x_+)} \vec \psi_j^{b, E_{+,0}}(r_0) + \e^{mid, \sigma, j}_{+, \gamma_+} \vec \psi_4^{b, E_{+,0}}(r_0), \\
 \vec \xi^{mid, \sigma, j}_{+,\gamma_+}(x_+) &=& e^{\frac{j-2}{2}i\Im\eta_{b, E_{+,0}}(x_+)} \vec \psi_j^{b, E_{+,0}}(x_+) + \kappa^{mid, \sigma, j}_{+, \gamma_+} \vec \psi_2^{b, E_{+,0}}(x_+). 
\eee
where
\bea
  e^{i\Im\eta_{b, E_{+,0}}(x_+)}\kappa^{mid,*}_{+,0} &=& -\int_{x_+}^{r_0} e^{i\Im\eta_{b, E_{+,0}}(x_+)} \psi_4^{b, E_{+,0}}\left( ir^{\frac{d-1}{2}} (P_b + \omega_{2,+}^* \rho_b) - \calV_{b, E_{+,0}}\xi^{mid, *}_{+,0} \right) \frac{dr}{W_{42;E_{+,0}}}\nonumber \\
  &=& O\left(b^{1-\frac 16} e^{-\frac{\pi}{2b}} x_+^4 e^{2S_b(x_+)}\right)
 +  iO\left(b^{-\frac 16} e^{-\frac{\pi}{2b}} x_+^3 e^{2S_b(x_+)}\right) \nonumber \\
  \e^{mid, \sigma, j}_{+, \gamma_+} &=& \left| \begin{array}{ll}
      O_\CC (x_+^{-1})  & j = 4, \, \sigma \in \{ \Re, \Im\} \\
      O_\CC (b^{1+ \frac 13})  & j = 2,\, \sigma \in \{ \Re, \Im\}
  \end{array}\right.,\nonumber \\
 e^{i\Im \eta_{b, E_{+,0}}(x_+)}  \kappa^{mid, \sigma, j}_{+,\gamma_+} 
 &=& \left| \begin{array}{ll}
  O_\CC(x_+^{-1}) & j = 2, \,\sigma \in \{ \Re, \Im\},\\
      O \left(x_+^{-2}e^{2S_b(x_+)}\right) + iO  \left(bx_+^{-1}e^{2S_b(x_+)}\right) & j = 4,\, \sigma = \Re,  \\
       i O \left(x_+^{-2}e^{2S_b(x_+)}\right) + O  \left(bx_+^{-1}e^{2S_b(x_+)}\right)  & j = 4,\,  \sigma = \Im,
  \end{array}\right..\nonumber
\eea

\mbox{}

Finally, we match the solution on $I_{ext}$ with those on $I_{mid,+}$ at $r_0$. Define 
\bee
  &&\xi^*_{+,\gamma_+} = \left| \begin{array}{ll}
    \xi^{ext, *}_{+,\gamma_+} + d^{*;\Re}_{+,\gamma_+} \xi^{ext, \Re}_{+,\gamma_+}+ d^{*;\Im}_{+,\gamma_+}\xi^{ext, \Im}_{+,\gamma_+}  & r \in I_{ext} \\
    \xi^{mid, *}_{+,\gamma_+} + c^{*;\Re}_{+,\gamma_+} \xi^{mid, \Re, 2}_{+,\gamma_+} + c^{*;\Im}_{+,\gamma_+} \xi^{mid, \Im, 2}_{+,\gamma_+} & r \in I_{mid,-}
\end{array}\right. \\
&&\xi^\sigma_{+,\gamma_+} =  
\left| \begin{array}{ll}
    d^{\sigma;\Re}_{+,\gamma_+} \xi^{ext, \Re}_{+,\gamma_+} + d^{\sigma;\Im}_{+,\gamma_+} \xi^{ext, \Im}_{+,\gamma_+} &   r \in I_{ext}\\
    \xi^{mid, \sigma, 4}_{+,\gamma_+} + c^{\sigma;\Re}_{+,\gamma_+} \xi^{mid, \Re, 2}_{+,\gamma_+} + c^{\sigma;\Im}_{+,\gamma_+} \xi^{mid, \Im, 2}_{+,\gamma_+}   &  r \in I_{mid,-}
\end{array}\right.\quad {\rm for}\,\,\sigma \in \{\Re, \Im\}.
\eee
Similar to Step 2.3-2.4 of the proof of Lemma \ref{lemrhob}, the matching condition can be parametrized under the basis $\{ \vec \psi_4^{b, E_{+,0}}, i\vec \psi_4^{b, E_{+,0}},\vec \psi_2^{b, E_{+,0}}, i\vec \psi_2^{b, E_{+,0}}\}$, and inverting the matrix leading to 
\[ d^{*;\sigma}_{+,\gamma_+}, c^{*;\sigma}_{+,\gamma_+} = O(b^{-4-\frac 16}e^{-\frac \pi{2b}}),\quad d^{\Re;\sigma}_{+,\gamma_+}, d^{\Im;\sigma}_{+,\gamma_+},c^{\Re;\sigma}_{+,\gamma_+}, c^{\Im;\sigma}_{+,\gamma_+} = O(1),\quad {\rm for}\,\,\sigma\in \{ \Re, \Im\}.  \]
Hence the related estimates in \eqref{eqxipmextest} follow these coefficient estimates and \eqref{eqxiextestpm}. 

Further notice from Propostion \ref{propWKB} (5) that for $k = 0, 1$, 
\be
\begin{split} 
e^{i\Im\eta_{b, 1 + i\upsilon^*_+}(x_+)} \pa_r^k \psi_4^{b, 1 + i\upsilon^*_+} (x_+) =& (-1)^{k+1} \kappa_\psi b^\frac 16 e^{\frac\pi{2b} - x_+} (1+ c_{\psi_4, k;b} b); \\ 
e^{-i\Im\eta_{b, 1 + i\upsilon^*_+}(x_+)} \pa_r^k \psi_2^{b, 1 + i\upsilon^*_+} (x_+) =& \kappa_\psi b^\frac 16 e^{-\frac\pi{2b} + x_+} (1+ O_\CC(b)),
\end{split} \label{eqpsix+2}
\ee
where we denote $\kappa_\psi = 2^{-\frac 76} \pi^{-\frac 12}$ and the residual of $\pa_r^k \psi_4^{b, 1 + i\upsilon^*_+}(x_+)$ by $c_{\psi_4, k; b}b $ with $c_{\psi_4,k;b} = O_\CC(1)$ for $k = 0, 1$.
Combined this with the above boundary value evaluation for $\vec \xi^{mid,*}_{+, \gamma_+}, \vec \xi^{mid, \sigma, j}_{+, \gamma_+}$ at $x_+$, we obtain \eqref{eqxi+bdryx+} and \eqref{eqxi+bdryx+2}.

\mbox{}

\underline{4. Nonlinear matching.}

Now we look for $(\gamma_\pm, \a_\pm, \mu_{\pm, \Re}, \mu_{\pm, \Im}) \in \RR^4$ such that 
\be \left| \begin{array}{l} 
x_\pm^{\frac{d-1}{2}}  \left( \Sigma_{\pm; \gamma_\pm, \a_\pm} + i\Theta_{\pm; \gamma_\pm, \a_\pm} \right) (x_\pm) = \left( \xi^*_{\pm,\gamma_\pm} + \mu_{\pm, \Re} \xi^\Re_{\pm,\gamma_\pm} + \mu_{\pm, \Im} \xi^\Im_{\pm,\gamma_\pm} \right) (x_\pm) \\ 
\pa_r \left[ (\cdot)^{\frac{d-1}{2}}  \left( \Sigma_{\pm; \gamma_\pm, \a_\pm} + i\Theta_{\pm; \gamma_\pm, \a_\pm} \right)\right] (x_\pm) = \pa_r \left( \xi^*_{\pm,\gamma_\pm} + \mu_{\pm, \Re} \xi^\Re_{\pm,\gamma_\pm} + \mu_{\pm, \Im} \xi^\Im_{\pm,\gamma_\pm} \right) (x_\pm).
\end{array} \right. \label{eqmatchxipm}
\ee
Again, recall the asymptotics of $\vec A, \vec \frakE, \vec D$ from Lemma \ref{lemLpm} (1), of $\vec \psi_2^{b, 1},  \vec \psi_4^{b, 1}$ from \eqref{eqpsix+1} and of $\vec \psi_2^{b, 1 +\upsilon^*_+}, \vec \psi_4^{b, 1 +\upsilon^*_+}$ from \eqref{eqpsix+2}. 

\mbox{}

\textit{4.1. "-" case.}  Within the range
\bee
|\a_-|\le x_-^\frac 12 b^{-4}e^{-\frac \pi b}\quad |\gamma_-| \le  x_-^{\frac 72} e^{-2x_-} \quad 
 |\mu_{-,\Re}| \le  b^{-4-\frac 16} e^{-\frac{3\pi}{2b} + 2x_-} x_-^{-\frac 12} \quad
|\mu_{-,\Im}| \le b^{-\frac 16} e^{-\frac{\pi}{2b}} x_-^{\frac 92},
\eee
we apply \eqref{eqST-intbdry}, \eqref{eqxi-bdryx-} and \eqref{eqxi-bdryx-2} to write the matching condition \eqref{eqmatchxipm} as
\bee \small
  \left| \begin{array}{l}
        \a_- \kappa_A e^{x_-} (1 + c_A x_-^{-1})
        =  \Re \kappa^*_{-,0} \kappa_\psi b^\frac 16 e^{-\frac \pi{2b}+x_-}  + \mu_{-,\Re} \kappa_\psi b^\frac 16 e^{\frac \pi{2b}-x_-}  + O\left(x_-^{-\frac 32} e^{x_-} b^{-4} e^{-\frac \pi b} \right) \\
         \a_- \kappa_A e^{x_-} (1 + c_A x_-^{-1})
        = \Re \kappa^*_{-,0} \kappa_\psi b^\frac 16 e^{-\frac \pi{2b}+x_-}  - \mu_{-,\Re} \kappa_\psi b^\frac 16 e^{\frac \pi{2b}-x_-}  + O\left(x_-^{-\frac 32} e^{x_-} b^{-4} e^{-\frac \pi b} \right) \\
        x_-^4 \kappa_{f_{0, -}} \kappa_Q e^{-x_-} (1 + c_Q x_-^{-1}) + \gamma_- (2\kappa_Q)^{-1} e^{x_-} = 
        \Im \kappa^*_{-, 0} \kappa_\psi b^\frac 16 e^{-\frac \pi{2b}+x_-} + \mu_{-,\Im} \kappa_\psi b^\frac 16 e^{\frac \pi{2b}-x_-} +  O\left( x_-^{\frac 52} e^{-x_-} \right)  \\
        -x_-^4 \kappa_{f_{0, -}} \kappa_Q e^{-x_-} (1 + c_Q x_-^{-1}) + \gamma_- (2\kappa_Q)^{-1} e^{x_-}  = 
        \Im \kappa^*_{-, 0} \kappa_\psi b^\frac 16 e^{-\frac \pi{2b}+x_-} - \mu_{-,\Im} \kappa_\psi b^\frac 16 e^{\frac \pi{2b}-x_-} +  O\left( x_-^{\frac 52} e^{-x_-} \right)
  \end{array}\right.
\eee
Hence similar to the proof of Lemma \ref{lemrhob}, we take 
\bee
 \left| \begin{array}{l}
     \a_- = \kappa_A^{-1} \left( \Re \kappa^*_{-, 0} \kappa_{\tilde \psi} b^\frac 16 e^{-\frac{\pi}{2b}} + b^{-4} e^{-\frac \pi b} x_-^{-1} \tilde \a_- \right)  \\
    \gamma_- = (2\kappa_Q) x_-^3 e^{-2x_-} \tilde \gamma_- \\
    \mu_{-,\Re} = \kappa_\psi ^{-1} b^{-4-\frac 16} e^{-\frac {3\pi}{2b} + 2x_-} x_-^{-1} \tilde \mu_{-,\Re}\\
    \mu_{-,\Im} = ( \kappa_\psi b^{\frac 16} e^{\frac{\pi}{2b}})^{-1} \left(x_-^4 \kappa_{f_{0,-}} \kappa_Q + x_-^{3} \tilde \mu_{-,\Im}\right)
 \end{array}\right.
\eee
and can find solution $(\tilde \a_-, \tilde \gamma_-, \tilde \mu_{-, \Re}, \tilde \mu_{-, \Im}) \in B^{\RR^4}_{x_-^\frac 14}$ via a Brouwer's fixed point theorem. That provides a first smooth solution of \eqref{eqhhh} with $\upsilon = \upsilon^*_- + \tilde \upsilon_-(\gamma_-) = \upsilon_{\rho_b} + O(b^{-2} e^{-\frac \pi b})$. 

\mbox{}

\textit{4.2. "+" case.}
Now let
\bee
 |\a_+| \le bx_+^\frac 92 e^{-2x_+},\quad |\gamma_+| \le x_+^\frac 72 e^{-2x_+},\quad |\mu_{+,\Re}| \le b^{1-\frac 16} x_+^\frac{11}{2} e^{-\frac{\pi}{2b}} ,\quad |\mu_{+,\Im}| \le  b^{-\frac 16} x_+^\frac 92 e^{-\frac{\pi}{2b}}.
\eee
Via \eqref{eqST-intbdry2}, \eqref{eqxi+bdryx+} and \eqref{eqxi+bdryx+2}, the matching condition \eqref{eqmatchxipm} reads
\bee
  \left| \begin{array}{l}
         -2b x_+^5 \kappa_{F_{+,0}}  (-2\kappa_A)^{-1} e^{-x_+}(1 + c_D x_+^{-1})   + \a_+ \kappa_A e^{x_+} \\
        \quad=  \Re \tilde \kappa^*_{+,0} \kappa_\psi b^\frac 16 e^{-\frac \pi{2b}+x_+}  + \mu_{+,\Re} \kappa_\psi b^\frac 16 e^{\frac \pi{2b}-x_+}  +  \mu_{+, \Im} \kappa_\psi \Im c_{\psi_4, 0; b} b^{1+\frac 16} e^{\frac \pi{2b}-x_+} +  O\left( b x_+^{\frac 72} e^{-x_+} \right)  \\
          2b x_+^5 \kappa_{F_{+,0}} (-2\kappa_A)^{-1} e^{-x_+}(1 + c_D x_+^{-1})  + \a_+ \kappa_A e^{x_+} \\
        \quad=  \Re \tilde \kappa^*_{+,0} \kappa_\psi b^\frac 16 e^{-\frac \pi{2b}+x_+}  - \mu_{+,\Re} \kappa_\psi b^\frac 16 e^{\frac \pi{2b}-x_+}  - \mu_{+, \Im} \kappa_\psi \Im c_{\psi_4, 1; b} b^{1+\frac 16} e^{\frac \pi{2b}-x_+} +  O\left( b x_+^{\frac 72} e^{-x_+} \right)  \\
         x_+^4 \kappa_{f_{0, +}} \kappa_Q e^{-x_+} (1 + c_Q x_+^{-1}) + \gamma_+ (2\kappa_Q)^{-1} e^{x_+} \\
        \quad= \Im \tilde \kappa^*_{+,0} \kappa_\psi b^\frac 16 e^{-\frac \pi{2b}+x_+} + \mu_{+,\Im} \kappa_\psi b^\frac 16 e^{\frac \pi{2b}-x_+} +  O\left( x_+^{\frac 52} e^{-x_+} \right)  \\
         -x_+^4 \kappa_{f_{0, +}} \kappa_Q e^{-x_+} (1 + c_Q x_+^{-1}) + \gamma_+ (2\kappa_Q)^{-1} e^{x_+}  \\
        \quad= \Im \tilde \kappa^*_{+,0}  \kappa_\psi b^\frac 16 e^{-\frac \pi{2b}+x_+} - \mu_{+,\Im} \kappa_\psi b^\frac 16 e^{\frac \pi{2b}-x_+} +  O\left( x_+^{\frac 52} e^{-x_+} \right)
  \end{array}\right.
\eee
where $c_{\psi_4, 0; b}, c_{\psi_4,1; b} = O_\CC(1)$ are from \eqref{eqpsix+2}, and $\tilde \kappa^*_{+,0} = e^{i\Im \eta_{b, 1+i\upsilon^*_+}(x_+)} \kappa^*_{+,0}$ with boundedness \eqref{eqxi+bdryx+2}. Therefore we take 
\bee
 \left| \begin{array}{l}
     \a_+ = \kappa_A^{-1} b x_+^4 e^{-2x_+} \tilde \a_+ \\
    \gamma_+ = (2\kappa_Q) x_+^3 e^{-2x_+} \tilde \gamma_+ \\
    \mu_{+,\Re} = ( \kappa_\psi b^{\frac 16} e^{\frac{\pi}{2b}})^{-1} \left(-2b x_+^5 \kappa_{F_{+,0}} (-2\kappa_A)^{-1} + b x_+^4 \tilde \mu_{+, \Re} \right)\\
    \mu_{+,\Im} = ( \kappa_\psi b^{\frac 16} e^{\frac{\pi}{2b}})^{-1} \left(x_+^4 \kappa_{f_{0,+}} \kappa_Q + x_-^3 \tilde \mu_{+,\Im}\right)
 \end{array}\right.
\eee
and the equation of $(\tilde \a_+, \tilde \gamma_+, \tilde \mu_{+, \Re}, \tilde \mu_{+, \Im})\in [-x_+^\frac 14, x_+^\frac 14]^4$ becomes 
\bee
  &&\left( \begin{array}{cccc}
     1 & -1 & &  \\
     1 & 1 &  &  \\
     & & -1 & -1 \\
     & & -1 & 1 \\
  \end{array}\right)
  \left( \begin{array}{c}
       \tilde \a_+  \\
        \tilde \mu_{+,\Re} \\
        \tilde \gamma_+ \\
        \tilde \mu_{+,\Im}
  \end{array}\right) \\
  &=& 
  \left( \begin{array}{c}
    -\kappa_A^{-1}  \kappa_{F_{+,0}}c_D + \Re \tilde \kappa^*_{+, 0} \kappa_\psi b^{-1 + \frac 16}e^{-\frac \pi{2b} +2x_+} x_+^{-4} + \Im c_{\psi_4, 0;b} \kappa_{f_{0, +}} \kappa_Q \\
     \kappa_A^{-1}  \kappa_{F_{+,0}}c_D + \Re \tilde \kappa^*_{+, 0} \kappa_\psi b^{-1 + \frac 16}e^{-\frac \pi{2b} +2x_+} x_+^{-4} - \Im c_{\psi_4, 1;b} \kappa_{f_{0, +}} \kappa_Q\\
     -\kappa_{f_{0, +}} \kappa_Q c_Q +  \Im \tilde \kappa^*_{+, 0} \kappa_{\psi} b^\frac 16 e^{-\frac{\pi}{2b} + 2x_+} x_+^{-3} \\
     \kappa_{f_{0, +}} \kappa_Q c_Q +  \Im \tilde \kappa^*_{+, 0} \kappa_{\psi} b^\frac 16 e^{-\frac{\pi}{2b} + 2x_+} x_+^{-3}
  \end{array}\right)
  + O(x_+^{-\frac 12}).
\eee
% Since the coefficient matrix on LHS still has bounded inverse (thanks to $\kappa_{f_{+, 1}} = O(1)$ from \eqref{eqfF01pm}), 
The existence of solution follows Brouwer's fixed point theorem. This provides a second smooth solution of \eqref{eqhhh} with $\upsilon = \upsilon^*_+ + \tilde \upsilon_+(\gamma_+) = 2b + O(b^{-2}e^{-\frac \pi b})$. 

Finally, the smoothness and decay \eqref{eqdecayZkb} of these eigenfunctions follow its construction \eqref{eqeigenansatz} and such smoothness and decay for each component. In particular, the decay for $\xi_{0,b}$, $\xi_{1, b}$ and $\xi_{2, b}$ comes from Proposition \ref{propQbasympref}, for $\xi_{3, b}$ from Lemma \ref{lemrhob} and for $\mathfrak{Z}$ from \eqref{eqxipmextest}. 

\end{proof}

\section{Uniqueness of bifurcated eigenmodes}\label{sec7}

In this section, we will conclude the proof of Theorem \ref{thmmodestabsmallspec} via Jost function argument. To define the Jost function and analyze its properties, we need the construction of the fundamental solutions for \eqref{eqnu} and that of the bifurcated eigenpairs. Our analysis will distinguish low/high spherical classes, motivating the introduction of the following parameters:
\begin{itemize}
    \item Low spherical classes: To apply Proposition \ref{propintfund}, Proposition \ref{propextfund}, Proposition \ref{propextfundin} and Proposition \ref{propbifeigen},  we define for $d \ge 1$ and $\nu_0 \ge 0$ the parameters
    \be
    \left| \begin{array}{l}
    I_0 = K_0 = 10, \\
         s_{c;{\rm low}}^*(d, \nu_0) = \min \{ s_c^{(0)}(d), s_{c;{\rm int}}^{(2)}(d), s_{c;{\rm ext}}^{(1)}(d, I_0, \nu_0, K_0), s_{c;{\rm ext}}^{(3)}(d, I_0), s_{c;{\rm eig}}^{(1)}(d) \},\\
         b_{\rm low}(d, \nu_0) = \min\{ b_{\rm int}(d), b_1(d, I_0, \nu_0, K_0), b_0(d, I_0), (x_*(d, I_0, \nu_0, K_0))^{-2}  \}, \\
         \delta_{\rm low} (d, \nu_0) = \min\{\delta_{\rm int}(d), \delta_1 (d, I_0, \nu_0, K_0), \delta_0(d) \}.
    \end{array}\right. \label{eqdefsclow}
    \ee
    We stress that the dependence of $\nu_0$ comes from the parameters $s_{c;{\rm ext}}^{(1)}, b_1, \delta_1,x_*$ in Proposition \ref{propextfund}. For $d = 1$ we can take $\nu_0 = \frac 12$, and the fundamental solution families for $\nu = -\frac 12$ case ($l = 0$) are determined in Remark \ref{rmkchoice1Dint} and Remark \ref{rmkchoice1Dext}.
    \item High spherical classes: To apply Proposition \ref{propintfund}, Proposition \ref{propextfundh} and Proposition \ref{propextfundin}, we fix define for $d \ge 2$ the parameters
    \be
    \left| \begin{array}{l}
    I_0 = \frac 12, \\
         s_{c;{\rm high}}^*(d) = \min \{ s_c^{(0)}(d), s_{c;{\rm int}}^{(2)}(d), s_{c;{\rm ext}}^{(2)}(d), s_{c;{\rm ext}}^{(3)}(d, I_0) \},\\
         b_{\rm high}(d) = \min\{ b_{\rm int}(d), b_3(d), b_0(d, I_0)  \}, \\
         \delta_{\rm high} (d)= \min\{\delta_{\rm int}(d), \delta_3 (d), \delta_0(d) \}.
    \end{array}\right. \label{eqdefschigh}
    \ee
\end{itemize}

\subsection{Jost function and unstable spectrum}\label{sec53}

In this subsection, we will define the Jost function $\Wfr_\nu(b,\upsilon)$ between interior solutions and exterior solutions, characterize its zeros with multiplicity and connect with the unstable spectrum. 

\begin{lemma}[Definition and basic properties of Jost function]\label{lemJostWronskian}
Let $d \ge 1$, $\nu_0 \ge 0$, and $I_0, K_0, s_{c;{\rm low}}^*, b_{\rm low}, \delta_{\rm low}$ defined in \eqref{eqdefsclow}. Then for 
\[ 0 \le s_c \le  s_{c;{\rm low}}^*,\quad  0 \le b \le b_{\rm low}, \quad \l \in \Omega_{\delta_{\rm low}; I_0, b},\quad  \nu \le \nu_0,\]
we recall $\{\Psi_{k;b,\l,\nu}\}_{1 \le k \le 4}$ from Proposition \ref{propintfund} and $\Phi_{1;b,\l,\nu}, \Phi_{2;b,\l,\nu}$ from Proposition \ref{propextfund} as fundamental solutions of \eqref{eqnu}. Define
\be F_{b, \l, \nu} = [\Psi_{1;b,\l,\nu}, \Psi_{2;b,\l,\nu} ],\quad G_{b, \l, \nu} = [\Phi_{1;b,\l,\nu}, \Phi_{2;b,\l,\nu} ] \label{eqJostdef1} \ee
as $2 \times 2$ matrix solutions of \eqref{eqnu}, and define the Jost function 
\be
    \Wfr_\nu(b, \l, r) = \det \left( \begin{array}{cc}
       F_{b,\l,\nu}(r) & G_{b,\l,\nu}(r)\\
       \pa_r F_{b,\l,\nu}(r) & \pa_r G_{b,\l,\nu}(r) 
    \end{array}\right).\label{eqJostdef2}
\ee
Then we have
\begin{enumerate}
    \item Wronskian properties: $\Wfr_\nu(b, \l, r) = \Wfr_\nu(b,\l)$ is independent of $r$. 
    % {\color{blue} No need the following?} $\Wfr_\nu(b,\l) = 0$ iff there exist $\vec \a, \vec \beta \in \CC^2 - \{ \vec 0 \}$, such that 
% \[ F_{b,\l,\nu}(r) \vec \a + G_{b, \l, \nu}(r) \vec \beta = 0,\quad \forall r \in (0, \infty). \]
    \item Analyticity and continuity: $\Wfr_\nu(b,\l)$ is analytic w.r.t. $\l$ and satisfies 
    \be
      \sup_{\l  \in \Omega_{\delta_{\rm low},I_0;b}} \left| \pa_\l^k \Wfr_{\nu}(b, \l) - \pa_\l^k \Wfr_{\nu}(0, \l) \right| \lesssim_{I_0, K_0, \nu_0} b^\frac 16,\quad \forall\, 0 \le k \le K_0.
    \ee
    \item Formula via connection coefficients:
    Define the connection coefficients $\iota_{jk;b,\nu}(\l)$ for $j = 1, 2$ and $k = 1, 2, 3, 4$ by 
    \be
      \Phi_{j;b,\l,\nu} = \sum_{k = 1}^4 \Psi_{k;b,\l,\nu} \iota_{jk;b,\nu}(\l),\quad j = 1, 2.  \label{eqdefiotacoeff}
    \ee
    Then $\iota_{jk;b,\nu}$ is analytic w.r.t. $\l$, and 
    \bea
       \Wfr_\nu(b,\l) = \det \left( \begin{array}{cc}
          \iota_{13;b,\nu}(\l)  & \iota_{23;b,\nu}(\l) \\
          \iota_{14;b,\nu}(\l)  & \iota_{24;b,\nu}(\l)
       \end{array} \right). \label{eqrepJW}
    \eea
\end{enumerate}
\end{lemma}
\begin{proof}
    (1) Since \eqref{eqnu} can be reforumlated as the first order system 
    \be  \pa_r \left( \begin{array}{cccc}
         \Phi^1  \\
         \Phi^2 \\
         \pa_r \Phi^1 \\
         \pa_r \Phi^2 
    \end{array} \right) = \left( \begin{array}{cccc}
         0& 0&1 &0   \\
         0& 0& 0& 1 \\
         * & * &0 & 0 \\
         * & * &0 &0 
    \end{array} \right)
    \left( \begin{array}{cccc}
         \Phi^1  \\
         \Phi^2 \\
         \pa_r \Phi^1 \\
         \pa_r \Phi^2 
    \end{array} \right). \label{eqnufirstorder}
    \ee
    Since the matrix has zero trace, its Wronskian $\Wfr_\nu(b, \l, r)$ is independent of $r$. 
    % and vanishes iff $\Psi_1, \Psi_2, \Phi_1, \Phi_2$ are linear dependent. The non-vanishing of $\vec \a$, $\vec \beta$ follows from the linear independence of  $\Psi_1$ and $\Psi_2$ by Proposition \ref{propintfund} and of $\Phi_1$ and $\Phi_2$ by Proposition \ref{propextfund}.

    (2) From (1), $\Wfr_\nu(b,\l) = \Wfr_\nu(b,\l,x_*)$ with $x_*$ from Proposition \ref{propextfund}. The corresponding analyticity of each component and the continuity estimates \eqref{eqPsicontb}, \eqref{eqbdrymapasymp} yield (2). 

    (3) Since $\{\Psi_{k;b,\l,\nu}\}_{1 \le k \le 4}$ 
    spans the solution spaces of \eqref{eqnu}, we can well-define $\iota_{jk;b,\nu}(\l)$. Their analyticity comes from that of $\Psi_{k;b,\l,\nu}(r)$ for $1 \le k \le 4$ and $\Phi_{j;b,\l,\nu}(r)$ for $j = 1, 2$. 

    For \eqref{eqrepJW}, we first compute the Wronskian for $\{\Psi_{j;b,\l,\nu}\}_{1 \le j \le 4}$. Denote vector form of solution to \eqref{eqnu} as $\vec \Psi = (\Psi^1, \Psi^2, \pa_r \Psi^1, \pa_r \Psi^2)^\top \in \CC^4$. We claim that 
    \be
       \det\left( \begin{array}{cccc} \vec \Psi_{1;b,\l,\nu} & \vec \Psi_{2;b,\l,\nu} & \vec \Psi_{3;b,\l,\nu} & \vec \Psi_{4;b,\l,\nu} \end{array}\right)(r) = 1,\quad \forall r > 0. \label{eqintWron}
    \ee
    Indeed, the Wronskian property indicates its independence of $r$. So use the asymptotics from \eqref{eqPsiasympest1}-\eqref{eqPsiasympest3} to compute
    \bee
     &&{\rm LHS\,\,of\,\,}\eqref{eqintWron} = \lim_{r\to 0} \det\left( \begin{array}{cccc} r^{-\nu} \vec \Psi_{1;b,\l,\nu} & r^{-\nu} \vec \Psi_{2;b,\l,\nu} & r^{\nu} \vec \Psi_{3;b,\l,\nu} & r^{\nu} \vec \Psi_{4;b,\l,\nu} \end{array}\right)(r) \\
      &=&\left| \begin{array}{ll}
      \lim_{r\to 0} \det \left( \begin{array}{cc} 
      r^\frac 12 I_2 + O(r^\frac 52) & (2\nu)^{-1} r^\frac 12 I_2 + o(r^\frac 12) \\
      (\nu+\frac 12) r^{-\frac 12} I_2 + O(r^\frac 32) & \frac{1/2 - \nu}{2\nu}  r^{-\frac 12} I_2 + o(r^{-\frac 12}) 
      \end{array} \right) & \nu > 0 \\
      \lim_{r\to 0} \det \left( \begin{array}{cc} 
      r^\frac 12 I_2 + O(r^\frac 52) &  r^\frac 12 A(r) + O(r^\frac 32) \\
      \frac 12 r^{-\frac 12} I_2 + O(r^\frac 32) &  \frac 12 r^{-\frac 12} (A(r) - 2I_2) + O(r^{\frac 12}) 
      \end{array} \right)  & \nu = 0
      % \\
      % \lim_{r\to 0} \det \left( \begin{array}{cc} 
      % (2|\nu|)^{-1} r^\frac 12 I_2 + o(r^\frac 12) & r^\frac 12 I_2 + O(r^\frac 52) \\
      % \frac{1/2 - |\nu|}{2|\nu|}  r^{-\frac 12} I_2 + o(r^{-\frac 12}) & (|\nu|+\frac 12) r^{-\frac 12} I_2 + O(r^\frac 32) 
      % \end{array} \right) & \nu = -\frac 12
      \end{array}\right.\\
      &=& 1
    \eee
    where $I_2 = \left( \begin{array}{cc}
    1 & 0 \\
    0 & 1
    \end{array}\right)$ and $A(r) = \left( \begin{array}{cc}
    -\ln(\sqrt{1+\l}r) + \ln 2 - \gamma & 0 \\
    0 & -\ln(\sqrt{1-\l}r) + \ln 2 - \gamma
    \end{array}\right)$ with $\gamma$ the Euler's constant. For the case $\nu = -\frac 12$, namely $d = 1, l = 1$, we easily seen 
    \[ \det \left( \{ \vec \Psi_{j;b,\l,-\frac12} \}_{1 \le j \le 4} \right) =  \det \left( \{ \vec \Psi_{j;b,\l,\frac12} \}_{1 \le j \le 4} \right) = 1 \] 
    from the choice of $\Psi_{j;b,\l,-\frac12}$ in Remark \ref{rmkchoice1Dint}. 
    
    Now \eqref{eqrepJW} follows the observation 
    \bee
     \Wfr_{\nu}(b,\l) = \det \left[  \left( \begin{array}{cccc}1 &0 & \iota_{11;b,\nu}(\l)  & \iota_{21;b,\nu}(\l)  \\
     0 & 1 & \iota_{12;b,\nu}(\l) & \iota_{22;b,\nu}(\l)  \\
     0 & 0& \iota_{13;b,\nu}(\l)  & \iota_{23;b,\nu}(\l)  \\
     0 & 0& \iota_{14;b,\nu}(\l)  & \iota_{24;b,\nu}(\l) 
      \end{array}\right)
     \left(\vec \Psi_{1;b,\l,\nu} \,\, \vec \Psi_{2;b,\l,\nu} \,\, \vec \Psi_{3;b,\l,\nu} \,\, \vec \Psi_{4;b,\l,\nu}\right) \right]
    \eee
\end{proof}

Next, we can use the zeros of Jost function to characterize the unstable spectrum with the corresponding multiplicity. This connection is established through the following two lemmas. We mention that similar statement was introduced in \cite{MR1852922} without proof, and
the case without multitplicity was proven in \cite[Lemma 5.17]{MR2219305}. 

Firstly, we connect the vanishing order of Jost function to the multiplicity of matching for admissible fundamental solutions.

\begin{lemma}[Vanishing order of Jost function]\label{lemvanishJW} Under the assumption of Lemma \ref{lemJostWronskian}, 
we have  
\be 
  \min_{n\ge 0} \left\{ \pa_\l^n \Wfr_\nu(b, \l) \neq 0 \right\}  = \dim \left( {\rm span} \{ \pa_\l^n \Psi_{j;b,\l,\nu} \}_{\substack{n \ge 0 \\ 1 \le j \le 2}} \cap  {\rm span} \{ \pa_\l^n \Phi_{j;b,\l,\nu} \}_{\substack{n \ge 0 \\ 1 \le j \le 2}}\right).   \label{eqorderJW}
\ee

   % \begin{enumerate}
   %     \item $\l \in \sigma(\calH_b  \big|_{\dot H^\sigma_l(\RR^d)} )$ iff $\Wfr_\nu(b, \l) = 0$; \\
   %     \item If $\Wfr_\nu(b, \l) = 0$ and $\pa_\l \Wfr_\nu(b, \l) \neq 0$, then 
   %     \be\cup_{k \ge 0} \ker (\calH_b \big|_{\dot H^\sigma_l(\RR^d)} - \l)^k = \ker (\calH_b \big|_{\dot H^\sigma_l(\RR^d)} - \l),\quad \dim  \ker (\calH_b \big|_{\dot H^\sigma_l(\RR^d)} - \l) = 1.\ee
   % \end{enumerate}
   % In particular, $\l \in \sigma(\calH \big|_{\dot H^\sigma_\nu(\RR^d)})$ iff $\Wfr_\nu(b, \l) = 0$. {\color{red} Bad notation here, especially range of $\l$. Make it clearer?}
\end{lemma}

\begin{proof} For notational simplicity, we will omit the parameters $b, \nu, \l$ in the proof. We will also use the linear independence of $ \{ \pa_\l^n \Psi_{j;b,\l,\nu} \}_{\substack{n \ge 0\\ 1 \le j \le 4}}$ and of $ \{ \pa_\l^n \Phi_{j;b,\l,\nu} \}_{\substack{ n \ge 0 \\ 1 \le j \le 2}}$ from Proposition \ref{propintfund} and Proposition \ref{propextfund} without restating it.

 Let $\vec \a_{n,j}  = \left( \begin{array}{c} 
 \pa_\l^n\iota_{j3} \\  \pa_\l^n\iota_{j4}
\end{array}\right)$ for $j = 1, 2$. 
Also denote for $N \ge 1$ that
$$X_N =  {\rm span} \{ \pa_\l^n \Phi_{j;b,\l,\nu} \}_{\substack{0 \le n < N \\ 1 \le j \le 2}},\quad Y_N =  {\rm span} \{ \pa_\l^n \Psi_{j;b,\l,\nu} \}_{\substack{0 \le n < N \\ 1 \le j \le 2}},\quad Z_N = X_N \cap Y_N,$$ 
and
\[ h_N = \left| \begin{array}{ll}
    \dim Z_1 & N = 1,\\
    \dim Z_N - \dim Z_{N-1} & N \ge 2,
\end{array}\right. \]

\mbox{}

\underline{Step 1. Structure of $Z_N$.}

We claim that there exists $0 \le N_1, N_2 \le \infty$, such that 
 \[ h_N = \left| \begin{array}{ll}
    2 & 1 \le N \le N_1, \\
    1 & N_1 < N \le N_1 +  N_2,\\
    0 & N > N_1 + N_2,
\end{array}\right. \]
and the coefficients $\vec \a_{n,j}$ satisfy 
\begin{enumerate}
    \item $\vec \a_{n,j} = \vec 0$ for every $0 \le n \le N_1 - 1$, $j = 1, 2$.
\item If $N_2 = 0$ and $N_1 < \infty$, then $\sum_{j=1}^2 v_j \vec \a_{N_1, j} \neq \vec 0$ for any $(v_1, v_2)^\top \in \CC^2 - \{ \vec 0 \}$.
\item If $N_2 > 0$, there exists $\{\vec v_n = (v_{n, 1}, v_{n, 2})^\top \}_{n = 0}^{N_2 - 1} \subset \CC^2$ with $\vec v_0 \neq \vec 0$ such that 
\be
\sum_{n=0}^N \frac{N_1!}{(N_1 + N - n)!} \sum_{j = 1}^2 v_{n, j} \vec \a_{N_1 + N - n, j} = \vec 0,\quad \forall 0\le N \le N_2 - 1. \label{eqalphabetarep}
\ee
In addition, if $N_1, N_2 < \infty$, for any $\vec v_{N_2} = (v_{N_2, 1}, v_{N_2, 2})^\top \in \CC^2$,
 \be \sum_{n = 0}^{N_2} \frac{N_1!}{(N_1 + N_2 - n)!} \sum_{j = 1}^2 v_{n, j} \vec \a_{N_1 + N_2 - n, j}  \neq \vec 0. \label{eqalphanondeg} \ee
\end{enumerate}

Indeed, we first observe that via differentiating \eqref{eqnu},
\be (\HH_{b, \nu} - \l) \pa_\l^{N+1}F_\l = (N+1) \pa_\l^N F_\l,\label{eqdiffl}
\ee
for any $N \ge 0$ and any $F_\l \in \{ \Phi_{j;b,\l,\nu} \}_{j = 1, 2} \cup \{\Psi_{k;b,\l,\nu} \}_{1 \le k \le 4}$.
So for $N \ge 2$, 
\be G_N \in \calZ_N - \calZ_{N-1} \quad \Rightarrow \quad (\HH_{b, \nu} - \l) G_N \in \calZ_{N-1} - \calZ_{N-2}\quad{\rm for\,\,}\calZ = X \,\,{\rm or}\,\, Y. \label{eqXYNiter}\ee
Also for any $N \ge 0$, we have 
\bee (X_{N+1} - X_N) \cap Y_N = X_N \cap (Y_{N+1} - Y_N) = \emptyset \eee
by checking whether the vector vanishes after applying $(\HH_{b, \nu} - \l)^N$ using \eqref{eqdiffl}.  
So we have
\be (X_N - X_{N-1}) \cap Y_N = Z_N - Z_{N-1},  \label{eqZNnondeg}\ee
and hence \eqref{eqXYNiter} implies $(\HH_{b, \nu} - \l ) (Z_N - Z_{N-1}) \subset Z_{N-1} - Z_{N-2}$ and hence the monotonicity $h_{N} \le h_{N-1}$. The existence of $N_1, N_2$ now follows the first bound $h_1 = \dim Z_1 \le \dim X_1 = 2$. 

For (1), notice that $X_N = Y_N = Z_N$ for $1 \le N \le N_1$. Differentiating \eqref{eqdefiotacoeff}, we have
\be
   \pa_\l^N \Phi_j = \sum_{k=1}^4 \sum_{n=0}^N \pa_\l^n \Psi_k  \cdot \binom{N}{n} \pa_\l^{N-n} \iota_{jk},\quad N \ge 0, \,\, j = 1, 2. \label{eqdefiotacoeff2}
\ee
So $\pa_\l^N \Phi_j \in Y_{N+1}$ for $N \in [0, N_1-1]$ indicates the vanishing of coefficient before $\pa_\l^n \Psi_3$, $\pa_\l^n \Psi_4$, resulting in (1).

For (2), $X_{N_1} = Y_{N_1} = Z_{N_1}$ implies that
\bee 
Z_{N_1 + 1} = X_{N_1} \cup \left( {\rm span} \{ \pa_\l^{N_1} \Phi_{j;b,\l,\nu} \}_{\substack{ 1 \le j \le 2}} \cap {\rm span} \{ \pa_\l^n \Psi_{j;b,\l,\nu} \}_{\substack{ 0 \le n < N_1 +1  \\ 1 \le j \le 2}} \right).
\eee
Then $N_2 = 0$ indicates $Z_{N_1 + 1} = Z_{N_1} = X_{N_1}$ and thus  $ {\rm span} \{ \pa_\l^{N_1} \Phi_{j;b,\l,\nu} \}_{\substack{ 1 \le j \le 2}} \cap {\rm span} \{ \pa_\l^n \Psi_{j;b,\l,\nu} \}_{\substack{ 0 \le n < N_1 +1  \\ 1 \le j \le 2}}= \{ 0 \}$, which implies (2) using \eqref{eqdefiotacoeff2} and (1). 

For (3), we prove \eqref{eqalphabetarep} by induction on $N$. Firstly, let $F_0 \in Z_{N_1+1} - Z_{N_1} \subset X_{N_1 + 1} - X_{N_1}$ from \eqref{eqZNnondeg}, which we denote by
$$F_0 =\vec v_0^\top \left( \begin{array}{c}\pa_\l^{N_1} \Phi_1 \\ \pa_\l^{N_1} \Phi_2 \end{array}\right) + \sum_{j = 1}^2 \sum_{n = 0}^{N_1-1} w_{jn} \pa_\l^n \Phi_j.$$
with $\vec v_0 \neq \vec 0$. 
Then from \eqref{eqdefiotacoeff2}, $X_{N_1} = Y_{N_1}$ and (1), we can write
\bee
  F_0 \in \sum_{j=1}^2 v_{0, j} \vec \a_{N_1, j} \cdot \left( \begin{array}{c} \Psi_3 \\ \Psi_4 \end{array}\right)  + Y_{N_1 + 1}
\eee
That $F_0 \in Y_{N_1 + 1}$ yields the $N=0$ case. In particular, we can take 
$$G_0 =\vec v_0^\top \left( \begin{array}{c}\pa_\l^{N_1} \Phi_1 \\ \pa_\l^{N_1} \Phi_2 \end{array}\right) \in Z_{N_1 + 1} - Z_{N_1}. $$
Now for $0 \le N \le N_2 - 2$, suppose \eqref{eqalphabetarep} holds for $0 \le N' \le N$ with $\{\vec v_n\}_{n=0}^N \subset \CC^2$ well-defined and 
\be G_{N'} = \sum_{n = 0}^{N'} \frac{N_1!}{(N_1 + N' - n)!} \vec v_n^\top \left( \begin{array}{c}\pa_\l^{N_1+N'-n} \Phi_1 \\ \pa_\l^{N_1+N'-n} \Phi_2 \end{array}\right) \in Z_{N_1 + N' + 1} - Z_{N_1 + N'},\quad 0 \le N' \le N.  \label{eqdefGN}\ee
To show \eqref{eqalphabetarep} for $N+1$ case and find $\vec v_{N+1}$ and $G_{N+1}$ of the same form, we can take $F_{N+1} \in Z_{N_1 + N + 2} - Z_{N_1 + N + 1}$. Similarly by \eqref{eqZNnondeg}, we write as
\[ F_{N+1} = \sum_{n=0}^{N+1} \vec w_n^\top \left( \begin{array}{c}\pa_\l^{N_1+N+1-n} \Phi_1 \\ \pa_\l^{N_1+N+1-n} \Phi_2 \end{array}\right) + \sum_{m=0}^{N_1-1} \vec x_m^\top \left( \begin{array}{c}\pa_\l^{m} \Phi_1 \\ \pa_\l^{m} \Phi_2 \end{array}\right) \]
with $\vec w_0 \neq \vec 0$. Then \eqref{eqdiffl} implies that $(\HH_{b, \nu} - \l)F_{N+1} \in Z_{N_1 + N + 1} - Z_{N_1 + N}$. Due to $h_{N_1 + N'} = 1$ for $0 \le N' \le N$,  there exists $\{ c_{N'} \}_{0 \le N' \le N} \subset \CC$ with $c_0 \neq 0$ such that  $$(\HH_{b, \nu} - \l)F_{N+1} - \sum_{N'=0}^{N} c_{N-N'} G_{N'} \in Z_{N_1}.$$ 
By normalizing $\vec w_0$, we can assume $c_0 = 1$. Compute
\bee
  (\HH_{b, \nu} - \l)F_{N+1} -  \sum_{N'=0}^{N} c_{N-N'} G_{N'} 
  =\sum_{n=0}^{N} \vec y_n^\top \left( \begin{array}{c}\pa_\l^{N_1+N-n} \Phi_1 \\ \pa_\l^{N_1+N-n} \Phi_2
  \end{array} \right) 
+ \sum_{m=0}^{N_1-1} (\vec x_m')^\top \left( \begin{array}{c}\pa_\l^{m} \Phi_1 \\ \pa_\l^{m} \Phi_2
  \end{array} \right) 
\eee
where 
\[ \vec y_n = (N_1 + N + 1 - n) \vec w_n - \frac{N_1!}{(N_1 + N - n)!} 
 \sum_{m=0}^n c_m \vec v_{n-m}. \]
So $(\HH_{b, \nu} - \l)F_{N+1} - \sum_{N'=0}^N c_{N-N'} G_{N'} \in X_{N_1}$ implies 
\[ \vec y_n = \vec 0 \quad \Rightarrow \quad  \vec w_n  = \frac{N_1!}{(N_1 + N + 1-n)!} \sum_{m=0}^n c_m \vec v_{n-m},\quad \forall \, 0 \le n \le N. \]
Using \eqref{eqdefiotacoeff2}, $X_{N_1} = Y_{N_1}$ and (1), we also have 
\bee F_{N+1} \in \sum_{j=1}^2 \sum_{n=0}^{N+1} w_{n,j} \sum_{l=0}^{N+1-n} \binom{N_1 + N + 1 - n}{N_1+l}  \vec \a_{N_1+l,j} \cdot \left( \begin{array}{c} \pa_\l^{N + 1 - n - l} \Psi_3 \\ \pa_\l^{N + 1 - n - l} \Psi_4 \end{array}\right) + Y_{N_1 + N + 2}
\eee
Since $F_{N+1} \in Y_{N_1 + N + 2}$, the vanishing of coefficient of $(\Psi_3, \Psi_4)^\top$ leads to 
\bee
  0 &=&  \sum_{n=0}^{N+1} \sum_{j=1}^2 w_{n,j} \vec \a_{N_1 + N+1-n, j} 
   =   \sum_{n=0}^{N+1} \frac{N_1!}{(N_1 + N+1-n)!} \sum_{j=1}^2 \sum_{m=0}^n c_m  v_{n-m, j} \vec \a_{N_1 + N+1-n,j} \\
   &=& \sum_{m = 0}^{N+1} c_m \sum_{n' = 0}^{N+1-m} \frac{N_1!}{(N_1 + N+1-m-n')!} \sum_{j=1}^2  v_{n', j} \vec \a_{N_1 + N+1-m-n', j}.
\eee
Noticing that for each $m\ge 1$, the sum over $n'$ cancels due to \eqref{eqalphabetarep} with $N' = N + 1 - m \le N$, the identity above becomes \eqref{eqalphabetarep} for $N+1$ thanks to $c_0 = 1$. To conclude the induction step, we take $\vec v_{n+1} = \vec w_{N+1}$ and $G_{N+1}$ defined as in \eqref{eqdefGN} will satisfy  $G_{N+1} = F_{N+1} - \sum_{m=0}^{N_1-1} \vec x_m^\top \left( \begin{array}{c}\pa_\l^{m} \Phi_1 \\ \pa_\l^{m} \Phi_2 \end{array}\right) \in Z_{N_1 + N + 2} - Z_{N_1 + N + 1}$. 

Finally, we prove the non-degeneracy \eqref{eqalphanondeg} when $N_1, N_2 < \infty$. With the $\{ 
v_n \}_{n=0}^{N_2-1} \subset \CC^2$ constructed above and any $\vec v_{N_2} \in \CC^2$, consider
\[ G_{N_2} = \sum_{n = 0}^{N_2} \frac{N_1!}{(N_1 + N_2 - n)!} \vec v_n^\top \left( \begin{array}{c}\pa_\l^{N_1+N_2-n} \Phi_1 \\ \pa_\l^{N_1+N_2-n} \Phi_2 \end{array}\right).  \]
Then $G_{N_2} \in X_{N_1 + N_2 + 1} - X_{N_1 + N_2}$ since $\vec v_0 \neq \vec 0$. Expand $F_{N_2}$ using \eqref{eqdefiotacoeff2}, $X_{N_1} = Y_{N_1}$ and \eqref{eqalphabetarep}, 
\begin{align*}
 G_{N_2} &\in \sum_{j=1}^2 \sum_{n=0}^{N_2} \frac{N_1!}{(N_1 + N_2 - n)!}  v_{n,j} \sum_{l=0}^{N_2-n} \binom{N_1 + N_2 - n}{N_1 + l}  \vec \a_{N_1 + l,j}\cdot \left( \begin{array}{c} \pa_\l^{N_2 - n - l} \Psi_3 \\ \pa_\l^{N_2 - n - l} \Psi_4 \end{array}\right) + Y_{N_1 + N_2 + 1} \\
 &= \sum_{L=0}^{N_2} \frac{1}{(N_2 - L)!}
 \left( \begin{array}{c} \pa_\l^{N_2 - L} \Psi_3 \\ \pa_\l^{N_2 - L} \Psi_4 \end{array}\right) 
 \cdot \sum_{n=0}^{L} \frac{N_1!}{(N_1 + L-n)!} \sum_{j=1}^2  v_{n, j} \vec \a_{N_1 + L-n,j} + Y_{N_1 + N_2 + 1} \\
 &= \left( \begin{array}{c}  \Psi_3 \\ \Psi_4 \end{array}\right) \cdot \sum_{n=0}^{N_2} \frac{N_1!}{(N_1 + N_2-n)!} \sum_{j=1}^2 v_{n, j} \vec \a_{N_1 + N_2 - n,j} + Y_{N_1 + N_2 + 1} 
\end{align*}
From the definition of $N_2$ and \eqref{eqZNnondeg}, $G_{N_2} \notin Y_{N_1 + N_2 + 1}$, which yields \eqref{eqalphanondeg}.

\mbox{}

\underline{Step 2. End of proof.}

From Step 1, \eqref{eqorderJW} is equivalent to 
\bea
\begin{split}\pa_\l^N \Wfr_\nu(b, \l) = 0,&\quad  \forall\, 0 \le N < 2N_1 + N_2;\\
\quad \pa_\l^{2N_1 + N_2} \Wfr_\nu(b, \l) \neq 0.& \quad \end{split} \label{eqorderJW1} 
\eea
We only prove the most difficult case $N_1 < \infty$ and $0 < N_2 < \infty$, and the other cases ($N_1 < \infty$ and $N_2 = 0$;  $N_1 = \infty$;  $N_1 < \infty$ and $N_2 = \infty$) follows in a similar and simpler way. 

By differentiating \eqref{eqrepJW}, we have
\bee
  \pa_\l^N \Wfr_\nu(b, \l) = \sum_{n=0}^N \binom{N}{n} \vec \a_{N-n, 1} \wedge \vec \a_{n, 2}. 
\eee
From Step 1.1, we have the vanishing of $\pa_\l^N \Wfr$ for $0 \le N < 2N_1$ and 
\be \pa_\l^N \Wfr_\nu (b, \l) = \sum_{n=0}^{N-2N_1} \binom{N}{N_1 + n} \vec \a_{N - N_1 - n, 1} \wedge \vec \a_{N_1 + n, 2},\quad N \ge 2N_1.  \label{eqrepJW2} \ee

Now recall \eqref{eqalphabetarep}-\eqref{eqalphanondeg} and $\vec v_n = (v_{n, 1}, v_{n, 2})$ for $0 \le n \le N_2 - 1$. Due to $\vec v_0 \neq \vec 0$ and linearity of these inequalities, we assume $v_{0, 2} = -1$ without loss of generality. Define $\{ \omega_n \}_{n=0}^{N_2-1}$ inductively by
\[ \omega_0 = 1; \quad \omega_{n+1} = \sum_{k=0}^n \omega_k v_{n+1-k, 2} \quad {\rm for}\,\, 0 \le n \le N_2 - 2; \]
so that for any $1 \le N \le N_2 - 1$, 
\bee 
  \sum_{n = 0}^{N} \omega_{n} v_{N-n, 2} = \sum_{n = 0}^{N-1} \omega_n v_{N-n, 2} - \omega_{N} = 0
\eee
Denote 
$$\vec{\tilde v}_0 = \vec v_0;\quad  \vec{\tilde v}_n = \sum_{n = 0}^{N} \omega_{n} \vec v_{N-n} = (\tilde v_{n, 1}, 0), \quad {\rm for}\,\, 1 \le n \le N_2 - 1; $$
Then for $1 \le N_* \le N_2 - 1$ fixed, we multiply the $N$-th equation of \eqref{eqalphabetarep} by $\omega_{N_*-N}$ for $0 \le N \le N_*$ and sum them up, leading to
\bee
  0 = \sum_{N=0}^{N_*} \omega_{N_* - N} \sum_{n=0}^{N} \frac{N_1!}{(N_1 + N - n)!} \sum_{j = 1}^2  v_{n, j} \vec \a_{N_1 + N - n, j} 
  %\\
  % = \sum_{m = 0}^{N_*} \frac{N_1!}{(N_1 + m)!} \left( \sum_{N = m}^{N^*} \omega_{N_* - N} \vec v_{N - m} \right)^\top \vec \a_{N_1 + m, j} 
  = \sum_{n=0}^{N_*} \frac{N_1!}{(N_1+N_* - n)!} \sum_{j = 1}^2 \tilde v_{n, j} \vec \a_{N_1 + N^* - n, j}
\eee
Similarly, we view \eqref{eqalphanondeg} as the $N_2$-th equation of \eqref{eqalphabetarep}, and take $\omega_{N_2} = 0$. Convolution of \eqref{eqalphabetarep} with $\{ \omega_n\}_{n = 0}^{N_2}$ implies 
\bee
 0 \neq \sum_{N=0}^{N_2} \omega_{N_2 - N} \sum_{n=0}^{N}\frac{N_1!}{(N_1 + N - n)!} \sum_{j = 1}^2  v_{n, j} \vec \a_{N_1 + N - n, j} = \sum_{n=0}^{N_2} \frac{N_1!}{(N_1+N_2 - n)!} \sum_{j = 1}^2 \tilde v_{n, j}\vec \a_{N_1 + N_2 - n, j}
\eee
with $\vec{\tilde v}_{N_2} = \sum_{n = 0}^{N_2} \omega_{n} \vec v_{N_2-n}$. With $\tilde v_{N_2}$ taken arbitrarily, we indeed have proven that  we can replace $\vec{v}_n$ by $\vec{\tilde v}_n$ in \eqref{eqalphabetarep}-\eqref{eqalphanondeg}. 

Now with $\tilde v_{n, 2} = -\delta_{n,0}$ for $0 \le n \le N_2 - 1$, the above two relations imply
\bee
  \vec \a_{N_1 + N, 2} &=& \sum_{n=0}^N \frac{(N_1 + N)!}{(N_1 + N - n)!} \tilde v_{n,1}\vec \a_{N_1 + N - n, 1},\quad 0 \le N \le N_2 - 1;  \\
  \vec \beta &:=& \vec \a_{N_1 + N_2, 2} - \sum_{n=0}^{N_2 - 1} \frac{(N_1 + N_2)!}{(N_1 + N_2 - n)!} \tilde v_{n,1}\vec \a_{N_1 + N_2 - n, 1} \notin {\rm span}\{\vec \a_{N_1, j}\}_{j = 1, 2}.
\eee
Plugging these relations into \eqref{eqrepJW2}, for $2N_1 \le N \le 2N_1 + N_2 - 1$ we have
\bee
 \pa_\l^N \Wfr_\nu(b,\l) &=& \sum_{n=0}^{N-2N_1} \binom{N}{N_1 + n} \vec \a_{N - N_1 - n, 1} \wedge \left(\sum_{m=0}^n \frac{(N_1 + n)!}{(N_1 + n - m)!} \tilde v_{m,1}\vec \a_{N_1 + n-m, 1} \right)\\
 &=& N! \sum_{m=0}^{N-2N_1} \tilde v_{m, 1} \sum_{n = m}^{N - 2N_1} \frac{\vec a_{N - N_1 - n, 1}}{(N - N_1 - n)!} \wedge \frac{\vec \a_{N_1 + n - m, 1} }{(N_1 + n - m)!}  = 0
\eee
from the anti-symmetry of wedge product; and for $N = 2N_1 + N_2$, 
\bee
   &&\pa_\l^{N} \Wfr_\nu(b,\l) \\
   &=& \sum_{n=0}^{N_2-1} \binom{2N_1 + N_2}{N_1 + n} \vec \a_{N_1 + N_2 - n, 1} \wedge \left(\sum_{m=0}^n \frac{(N_1 + n)!}{(N_1 + n - m)!} \tilde v_{m,1}\vec \a_{N_1 + n-m, 1} \right) \\
&+&  \binom{2N_1 + N_2}{N_1 + N_2} \vec \a_{N_1, 1} \wedge \left( \vec \beta +  \sum_{m=0}^{N_2 - 1} \frac{(N_1 + N_2)!}{(N_1 + N_2 - m)!} \tilde v_{m,1}\vec \a_{N_1 + N_2 -m, 1}\right)\\ 
 &=& (2N_1 + N_2)! \sum_{m=0}^{N_2 - 1} \tilde v_{m, 1} \sum_{n=m}^{N_2} \frac{\vec a_{N_1 + N_2 - n, 1}}{(N_1 + N_2 - n)!} \wedge \frac{\vec \a_{N_1 + n - m, 1} }{(N_1 + n - m)!} + \binom{2N_1 + N_2}{N_1 + N_2} \vec \a_{N_1, 1} \wedge  \vec \beta \\
 &\neq& 0.
\eee
That verifies \eqref{eqorderJW1} and hence concludes the proof. 
\end{proof}

Our next lemma deals with the regularity of solutions to connect with the generalized eigenspace of $\calH_b$.  As preparation, we first define the following projection from each spherical class to its radial part with a quadratic phase: for $d \ge 1$, $0 \le s < \frac d2$, $l \ge 0$ and $b \ge 0$, define
\be \PP_{d,l,b} : (\dot H^s_l(\RR^d))^2 \to (L^1_{loc}(\RR_{>0}))^2,\quad \left( \begin{array}{c}
    f^1(|x|) Y_{d, l,m}(x/|x|)   \\
    f^2(|x|) Y_{d, l,m}(x/|x|) 
\end{array} \right)  \mapsto  r^{\frac{d-1}{2}} \left( \begin{array}{c}
    e^{i\frac{br^2}{4}} f^1(r)   \\
    e^{-i\frac{br^2}{4}} f^2(r) 
\end{array} \right) .\label{eqprojsphrad} \ee
Notice that the change of variable \eqref{eqZPhi} indicates that for smooth function in spherical class $l \ge 0$, we have $\PP_{d,l,b} \calH_b = \HH_{b, \nu} \PP_{d, l, b}$ where $\nu = l + \frac{d-2}{2}$.

\begin{lemma}[Regularity of fundamental solutions]\label{lemregfund}
\mbox{} 

\begin{enumerate}
    \item Low spherical classes: Under the assumption of Lemma \ref{lemJostWronskian},  for $s_c = 0$, $\nu = l + \frac{d-2}{2} \le \nu_0$, $|\l| \le \delta_{\rm low}$, and $1 \le k \le K_0$,
      \be
    \PP_{d,l,0} \ker (\calH_0\big|_{(L^2_l(\RR^d))^2} - \l)^k   
     =  {\rm span} \{ \pa_\l^n \Psi_{j;0,\l,\nu} \}_{\substack{0 \le n < k \\ 1 \le j \le 2}} \cap  {\rm span} \{ \pa_\l^n \Phi_{j;0,\l,\nu} \}_{\substack{0 \le n < k \\ 1 \le j \le 2}};\label{eqchargeneigenb0}
     \ee
     and for $0 < s_c \le s_{c;{\rm low}}^*$, $\sigma \in (s_c, \frac 12)$, $0 < b(s_c, d) \le b_{\rm low}$, $\nu = l + \frac{d-2}{2} \le \nu_0$, $\l \in   \Omega_{\delta_{\rm low};\sigma - s_c,b}$ and $1 \le k \le K_0$,
     \bea
     \PP_{d,l,b} \ker (\calH_b\big|_{(\dot H^\sigma_l(\RR^d))^2} - \l)^k &\supset& {\rm span} \{  \Psi_{j;b,\l,\nu} \}_{\substack{1 \le j \le 2}} \cap  {\rm span} \{ \Phi_{j;b,\l,\nu} \}_{\substack{ 1 \le j \le 2}}  \label{eqchargeneigenb1} \\
    \PP_{d,l,b} \ker (\calH_b\big|_{(\dot H^\sigma_l(\RR^d))^2} - \l)^k &\subset & {\rm span} \{ \pa_\l^n \Psi_{j;b,\l,\nu} \}_{\substack{0 \le n < k \\ 1 \le j \le 2}} \cap  {\rm span} \{ \pa_\l^n \Phi_{j;b,\l,\nu} \}_{\substack{0 \le n < k \\ 1 \le j \le 2}} \label{eqchargeneigenb2}
      \eea 
      \item High spherical classes: For $d \ge 2$, let $(I_0, s_{c;{\rm high}}^*, b_{\rm high}, \delta_{\rm high})$ as in \eqref{eqdefschigh}, $\nu_{\rm int}(d)$ from Proposition \ref{propintfund} (4) and $\Phi_{j;b,\l,\nu}$ from Proposition \ref{propextfundh}. Then \eqref{eqchargeneigenb1}-\eqref{eqchargeneigenb2} hold for $0 < s_c \le s_{c;{\rm high}}^*$, $\sigma \in (s_c, \frac 12)$, $0 < b(s_c, d) \le b_{\rm high}$, $\nu = l + \frac{d-2}{2} \ge \max \{\nu_{\rm int}(d), 1\}$, $\l \in   \Omega_{\delta_{\rm high};\sigma - s_c,b}$ and $k = 1$.
\end{enumerate}
\end{lemma}

\begin{proof}
  \underline{1. Proof of (1).} 
  
  \underline{\textit{1.1. Reduction of proof.}} We claim that \eqref{eqchargeneigenb0}, \eqref{eqchargeneigenb1} and \eqref{eqchargeneigenb2} are boiled down to the following estimates where $s(0) = 0$ and $s(b) = \sigma$ for $1 \le k \le K_0$:
   \begin{align}
    {\rm span} \{ \pa_\l^n \Psi_{j;0,\l,\nu} \}_{\substack{0 \le n < k \\ 1 \le j \le 2}} \cap  {\rm span} \{  \pa_\l^n \Phi_{j;0,\l,\nu} \}_{\substack{0 \le n < k \\ 1 \le j \le 2}}  \subset \PP_{d,l,0} \left( L^2_l(\RR^d) \cap \dot H^{2}_l(B_1^{\RR_d}) \right)^2;\label{eqcharreg1} \\
      {\rm span} \{  \Psi_{j;b,\l,\nu} \}_{\substack{1 \le j \le 2}} \cap  {\rm span} \{ \Phi_{j;b,\l,\nu} \}_{\substack{ 1 \le j \le 2}}  \subset \PP_{d,l,b} \left( \dot H^{\sigma}_l(\RR^d) \cap \dot H_l^{2}(B_1^{\RR_d}) \right)^2, \quad b \in (0, b_{\rm low}];\label{eqcharreg2} \\
      \Psi_{j;b,\l,\nu}, \Phi_{j;b,\l,\nu} \notin \PP_{d,l,b} \left( \dot H^{s(b)}_l(\RR^d) \cap \dot H^{2}_l(B_1^{\RR_d}) \right)^2, \quad j \in \{ 3, 4\},\quad  b \in [0, b_{\rm low}].\label{eqcharreg3}
   \end{align}

Indeed, we first notice that the change of variable \eqref{eqZPhi} indicates $\PP_{d, l, b}\ker (\calH_b\big|_{(\dot H^s_l(\RR^d))^2}  - \l)^k \subset  \ker(\HH_{b, \nu} \big|_{L^1_{loc}((0,\infty))}  - \l)^k$ for any $0 \le s < \frac d2$ and $k \ge 1$. Further exploiting the improvement of local regularity from the eigenequation \eqref{eqeigenHb} and the smoothness of ODE solution away from the origin, we indeed have
    \be \ker (\calH_b\big|_{(\dot H^s_l(\RR^d))^2}  - \l)^k = \left( \PP_{d, l, b}^{-1} \ker(\HH_{b, \nu} \big|_{(C^\infty_{loc}(\RR_+))^2}  - \l)^k\right)\cap \left( \dot H_l^{s}(\RR^d) \cap \dot H_l^{2}(B_1^{\RR_d}) \right)^2. \label{eqinteregfund}
    \ee
By differentiating \eqref{eqnu} w.r.t. $\l$, we have 
   \be (\HH_{b,\nu}\big|_{(C^\infty_{loc}(\RR_+))^2} - \l)\pa_\l^n F = n \pa_\l^{n-1} F,\qquad  {\rm for}\,\, F \in \{\Psi_{j;b,\l,\nu}, \Phi_{j;b,\l,\nu} \}_{1 \le j \le 2},\quad n \ge 1. \label{eqbwbwbw} \ee
Thus \eqref{eqcharreg1} and \eqref{eqcharreg2} imply the $\supset$ direction of \eqref{eqchargeneigenb0} and \eqref{eqchargeneigenb1}. 

Now it suffices to show the $\subset$ direction of \eqref{eqchargeneigenb0} and \eqref{eqchargeneigenb2}, which we prove by induction on $k$. For $k = 1$, this follows from  \eqref{eqcharreg3} and 
\be \ker(\HH_{b, \nu} \big|_{(C^\infty_{loc}(\RR_+))^2} - \l) 
   = {\rm span} \{ \Psi_{j;b,\l,\nu} \}_{\substack{ 1 \le j \le 4}} \cap  {\rm span} \{ \Phi_{j;b,\l,\nu} \}_{\substack{ 1 \le j \le 4}} \label{eqkernalHHbnu}
   \ee
since $\{ \Psi_{j;b,\l,\nu} \}_{\substack{ 1 \le j \le 4}}$ and $\{ \Phi_{j;b,\l,\nu} \}_{\substack{ 1 \le j \le 4}}$ are linear independent solution basis for the ODE system $(\HH_{b, \nu} - \l) \Phi= 0$. Next, supposing \eqref{eqchargeneigenb0} and \eqref{eqchargeneigenb2} hold for $1 \le k \le K_0 - 1$, for any $F \in  \ker (\calH_b\big|_{(\dot H^{s(b)}_l(\RR^d))^2}  - \l)^{k+1}$, we have 
\[ (\calH_b - \l) F \in  \ker (\calH_b\big|_{(\dot H^{s(b)}_l(\RR^d))^2}  - \l)^{k} \subset \PP_{b, l, b}^{-1} \left({\rm span} \{ \pa_\l^n \Psi_{j;b,\l,\nu} \}_{\substack{0 \le n < k \\ 1 \le j \le 2}} \cap  {\rm span} \{ \pa_\l^n \Phi_{j;b,\l,\nu} \}_{\substack{0 \le n < k \\ 1 \le j \le 2}}  \right). \]
From \eqref{eqbwbwbw} and \eqref{eqkernalHHbnu}, we have 
\[ F \in  \PP_{b, l, b}^{-1} \left({\rm span} \{ \pa_\l^n \Psi_{j;b,\l,\nu}, \Psi_{j+2;b,\l,\nu} \}_{\substack{0 \le n < k+1 \\ 1 \le j \le 2}} \cap  {\rm span} \{ \pa_\l^n \Phi_{j;b,\l,\nu}, \Phi_{j+2;b,\l,\nu} \}_{\substack{0 \le n < k+2 \\ 1 \le j \le 2}}  \right). \]
Excluding $\Psi_{j+2;b,\l,\nu}$ and $\Phi_{j+2;b,\l,\nu}$ via \eqref{eqcharreg3}-\eqref{eqinteregfund} leads to \eqref{eqchargeneigenb0} and \eqref{eqchargeneigenb2} for $k+1$. That concludes the induction.

\mbox{}

\underline{\textit{1.2. Proof of \eqref{eqcharreg1}, \eqref{eqcharreg2} and \eqref{eqcharreg3}.}}

   \textit{Proof of \eqref{eqcharreg1}.} We first show $\left( \begin{array}{c}r^{-\frac{d-1}{2}} e^{-i\frac{br^2}{4}} \pa_\l^n \Psi_{j;b,\l,\nu}^1 Y_{d,l,m} \\ r^{-\frac{d-1}{2}} e^{i\frac{br^2}{4}} \pa_\l^n \Psi_{j;b,\l,\nu}^2 Y_{d,l,m} \end{array}\right) \in (W^{2, \infty}(B_1^{\RR^d}))^2$ for $j = 1, 2$, and any $l \ge 0$, $1 \le m \le N_{d,l}$, $n \ge 0$ and $b \in [0, b_{\rm low}]$, $|\l| \le \delta_{\rm low}$. Due to the smoothness of $r^l Y_{d,l,m}$ and $e^{\pm i\frac{br^2}{4}}$, it suffices to show $ r^{-\frac{d-1}{2} - l} \pa_\l^n \Psi_{j;b,\l,\nu} \in W^{2, \infty}(B_1^{\RR^d})$. Indeed, from Proposition \ref{propintfund} we have
   \[ r^{-\frac{d-1}{2} - l} \pa_\l^n \Psi_{j;b,\l,\nu} = \delta_{n,0} \vec e_j+ O(r^2),\quad \pa_r \left(r^{-\frac{d-1}{2} - l} \pa_\l^n \Psi_{j;b,\l,\nu}\right) = O(r),\quad r \to 0; \]
   and \eqref{eqbwbwbw} further implies $\pa_r^2 \left(r^{-\frac{d-1}{2} - l} \pa_\l^n \Psi_{j;b,\l,\nu}\right) = O(1)$ as $r \to 0$. These three bounds implies the desired regularity. 

   Next, from the exponential decay of $\Phi_{j;b,\l,\nu}$ in Proposition \ref{propextfund}(1), we easily see $r^{-\frac{d-1}{2}} \pa_\l^n \Phi_{j;0,\l,\nu} Y_{d,l,m} \in (L^2(\RR^d - B_1^{\RR^d}))^2$ for $j = 1, 2$, and any $l \ge 0$, $1 \le m \le N_{d,l}$, $n \ge 0$, $|\l|\le \delta_{\rm low}$. The regularity \eqref{eqcharreg1} follows applying the above estimates in $B_1$ and $\RR^d - B_1$ respectively. 

   \mbox{}
   
   \textit{Proof of \eqref{eqcharreg2}.} Let $j = 1$ or $2$. The interior bound for $\Psi_{j;b,\l,\nu}$ has been proven above. So it suffices to show 
   $U := \left( \begin{array}{c} (1-\chi(r))r^{-\frac{d-1}{2}} e^{-i\frac{br^2}{4}}  \Psi_{j;b,\l,\nu}^1 Y_{d,l,m} \\ (1-\chi(r)) r^{-\frac{d-1}{2}} e^{i\frac{br^2}{4}} \Psi_{j;b,\l,\nu}^2 Y_{d,l,m} \end{array}\right) \in (\dot H^{\sigma}(\RR^2))^2$ for $j = 1, 2$, $l = \nu - \frac{d-2}{2} \le \nu_0 - \frac{d-2}{2}$, $1 \le m \le N_{d, l}$ and $0 < b \le b_{\rm low}$, $\l \in   \Omega_{\delta_{\rm low};\sigma - s_c, b}$, where $\chi$ is the smooth cutoff as in \eqref{eqdefchiR}.
   % $\chi \in C^\infty_c(B_1^{\RR^d})$ and $\chi\big|_{B_{1/2}^{\RR^d}} = 1$ is a smooth cutoff. 
   From Proposition \ref{propextfund}(2) and smoothness of $Y_{d,l,m}$, we see 
    \bee
             \left|\pa_{x_i}^n U \right| \lesssim_{b, \l} \la x \ra^{-\frac d2 + \frac{\Im \l}{b} + s_c - n}, \quad \forall \, x\in \RR^d,\quad n = 0, 1,\,\,1 \le i \le d.
         \eee
    Thus we have $U \in L^\infty \cap L^{d(\frac d2 - \frac{\Im \l}b - s_c)^{-1} +} \cap \dot W^{1, \infty} \cap \dot W^{1,d(\frac d2 + 1 - \frac{\Im \l}b - s_c)^{-1} +} \subset \dot H^\sigma$. Here the inclusion follows from Sobolev embedding and $\frac{\Im \l}{b} - s_c < \sigma< \min \{ \frac d2, 1\}$. 
    
\mbox{}

\textit{Proof of \eqref{eqcharreg3}.} Let $j = 3$ or $4$. For interior solution, we show $\left( \begin{array}{c}r^{-\frac{d-1}{2}} e^{-i\frac{br^2}{4}} \Psi_{j;b,\l,\nu}^1 Y_{d,l,m} \\ r^{-\frac{d-1}{2}} e^{i\frac{br^2}{4}}  \Psi_{j;b,\l,\nu}^2 Y_{d,l,m} \end{array}\right) \notin (H^2(B_1^{\RR_d}))^2$ for $0 \le b \le b_{\rm low}$ by their singular asymptotics at $0$ from Proposition \eqref{propintfund}(1). More specifically, for $l \ge 2$, we have $r^{-\frac{d-1}{2}}\Psi_j^{j-2} \sim r^{-(d-2+l)} \notin L^2(\RR^d)$; and for $l = 0$ or $1$, we can use the explicit form of $Y_{d,0,1} = 1$ and $Y_{d,1,m} = r^{-1}x_m$, the asymptotics in Proposition \eqref{propintfund}(1) and the eigenequation $(\HH_{b, \nu} - \l)\Psi_{j;b,\l,\nu} = 0$ to compute $\pa_{x_1}^2 \left( r^{-\frac{d-1}{2}}  e^{(-1)^j i\frac{br^2}{4}} 
 \Psi_{j;b,\l,\nu}^{j-2} Y_{d, l,m}\right) \notin L^2(B_1^{\RR^d})$.

%      {\color{purple}  
%      Computation: 
% \begin{enumerate}
% \item $l = 0$ case: Note that $N_{d, 0} = 1$,   $Y_{d,0}^m = 1$. We can further omit the smooth part $e^{\pm ibr^2/4}$ and compute 
% \bee &&\pa_{x_i}^2 \left( r^{-\frac{d-1}{2}}\Psi_j^{j-2} \right) = \left( \frac{x_i^2}{r^2} \pa_r^2 + \left(1- \frac{x_i^2}{r^2}\right) r^{-1} \pa_r \right)\left( r^{-\frac{d-1}{2}}\Psi_j^{j-2} \right) \\ 
% &=& \left| \begin{array}{ll}
%     \left( 2\nu(2\nu + 1) \frac{x_i^2}{r^2} - 2\nu (1- \frac{x_i^2}{r^2}) \right)r^{-2\nu - 2} + o(r^{-2\nu-2}) &  \nu > 0 \\
%     \left(\frac{x_i^2}{r^2} r^{-2} +( 1- \frac{x_i^2}{r^2})(-r^{-2})\right) + O(r^{-1}) & \nu = 0
% \end{array}\right. \notin L^2(B_1^{\RR_d})
% \eee
% \item $l = 1$ case: Note that $Y_{d, 1}^m = \frac{x_m}{r}$ for $ 1 \le m \le N_{d, 1} =  d$. Also note that $\nu \ge 1$. So similarly drop $e^{\pm ibr^2/4}$ and compute 
% \bee
%   &&\pa_{x_i} \left( r^{-\frac{d-1}{2}}\Psi_j^{j-2} \frac{x_m}{r} \right) 
%   = -d x_i x_m r^{-d-2} + \delta_{i,m} r^{-d} + o(r^{-d}) \notin L^2(B_1^{\RR_d})
% \eee
% \end{enumerate}}

 For exterior solutions, $r^{-\frac{d-1}{2}} \Phi_{j;0,\l,\nu}^{j-2} Y_{d, l, m} \notin L^2(\RR^d)$ directly comes from their exponential growth as in Proposition \ref{propintfund}(1), and we will prove for $0 < b \le b_{\rm low}$, $\l \in   \Omega_{\delta_{\rm low};\sigma - s_c, b}$ that $r^{-\frac{d-1}{2}} e^{(-1)^j i\frac{br^2}{4}} \Phi_{j;b,\l,\nu}^{j-2} Y_{d, l, m} \notin \dot H^\sigma(\RR^d)$. Indeed, from Proposition \ref{propextfundin}(2) and the eigenequation $(\HH_{b, \nu} - \l)\Phi_{j;b,\l,\nu} = 0$ we obtain
         \bee
             \left|\pa_r^n \left(r^{-\frac{d-1}{2}} e^{(-1)^j i\frac{br^2}{4}} \Phi_{j;b,\l,\nu}^{j-2} \right) \right| \sim_{b, \l} r^{-\frac d2 - \frac{\Im \l}{b} - s_c + n},\quad r \gg \frac 4b, \quad n = 0, 1, 2.
         \eee
         For each $Y_{d,l,m}$, we can find a small open set $\calO \subset \SS^{d-1}$ such that $\inf_{\omega \in \calO}|Y_{d,l}^m(\omega)| > 0$. By taking a further smaller open set $\calO' \subset \calO$ and $ c \ll 1$, one can estimate for $u = r^{-\frac{d-1}{2}} e^{(-1)^j i\frac{br^2}{4}} \Phi_{j;b,\l,\nu}^{j-2}$ that
         \bee
           \int_{B_{c^2 r^{-1}}(2cr^{-1}\omega)} \frac{|u(r\omega+y) - u(r\omega)|^2}{|y|^{d+2\sigma}} dy \gtrsim r^{-d + 2 ( \sigma - s_c - \frac{\Im \l}{b})},\quad {\rm for\,\,} \omega \in \calO',\,\, r \gg \frac 4b.
         \eee
         where the constant is independent of $\omega, r$.
         Since $\frac{\Im \l}{b} < \sigma - s_c$, integrating the above estimate yields 
         $$\| u \|_{\dot H^\sigma}^2 \ge \int_{(0,\infty) \times \calO'} \int_{B_{c^2r^{-1}}(2cr^{-1}\omega)} \frac{|u(x+y) - u(x)|^2}{|y|^{d+2\sigma}} dydx = \infty. $$
\mbox{}

  \underline{2. Proof of (2).} The proof is almost identical as above for $k = 1$ and $b > 0$, since $\{\Phi_{j;b,\l,\nu}\}_{j=1,2}$ in  Proposition \ref{propextfund} has the same boundedness as in Proposition \ref{propextfundh}, and $\{\Phi_{j;b,\l,\nu}\}_{j=3,4}$, $\{\Psi_{j;b,\l,\nu}\}_{1 \le j \le 4}$ are still from Proposition \ref{propextfundin}, Proposition \ref{propintfund} respectively.
\end{proof}

We end this subsection with the following lemma, which is a substitution of Rouch\'e's Theorem to show continuity of zeros of continuous families of analytic functions without contour integration. 

\begin{lemma}[Uniqueness of zeros for degenerate bifurcation of analytic function]\label{lemuniqzero}
  Suppose for $b_0 > 0, \delta_0 > 0$, $I_0 > 0$ and $k_0 \in \NN_{\ge 0}$, a family of analytic functions 
  \[ f_0 \in C^{\omega_\CC}(B_{\delta_0}) \quad {\rm and} \quad f_b \in C^{\omega_\CC}(\Omega_{\delta; I_0, b})\quad {\rm for}\,\, 0 < b \le b_0, \] 
  satisfy the following conditions:
  \begin{enumerate}
      \item Existence of zeros for $f_b$: There exist $k_0$ functions $z_j: [0, b_0] \to B_{\delta_0}$ for $1 \le j \le k_0$, such that for each $j$,
      \[ z_j(0) = 0,\quad |z_j(b)| \le \frac 13 {bI_0}\quad {\rm for}\,\, 0 < b \le b_0 \] 
      and $f_b(z_j(b)) = 0$ for $0 \le b \le b_0$. 
      \item Uniqueness of zeros of $f_0$: $g_0(z) := f_0(z) \cdot z^{-k_0}$ is analytic and non-vanishing on $B_{\delta_0}$. 
      \item Uniform boundedness and continuity w.r.t. $b$: \[ \sup_{0 < b \le b_0} \| f_b\|_{C^{k_0+1}(\Omega_{\delta;I_0,b})} < \infty,\qquad \| f_b - f_0\|_{C^{k_0+1}(\Omega_{\delta;I_0,b})} = o_b(1). \]
      \item Equidistribution except at most one pair\footnote{This condition is technical and can likely be relaxed.}: If $k_0 \ge 2$, then $z_j(b) \neq z_k(b)$ for any $j\neq k$ and $b > 0$; moreover, there exists $C_* > 1$ such that 
      \[
       |z_j(b) - z_{j'}(b)| \in [C_*^{-1}b, C_* b],\quad \forall\,1 \le j < j' \le k_0, \,\,(j, j') \neq (1, 2).
      \]
  \end{enumerate}
  Then there exists $0 < \delta_1 \le \delta_0$, $0 < b_1 \le b_0$ such that for $0 < b \le b_1$, 
  $g_b(z) := f_b(z) \cdot \left( \prod_{j=1}^{k_0} z - z_j(b) \right)^{-1}$ is analytic and non-vanishing on $\Omega_{\delta_1;I_0,b}$. 
  \end{lemma}

  \begin{proof}
    If $k_0 = 0$, then the result simply follows $|f_b(z)| \sim 1$ for $z \in B_{\delta_0/2}$ from condition (2) and the continuity (3). 
    
    Now we consider $k_0 \ge 1$. The core is to use Lagrangian interpolation polynomial to decompose $f_b$ so as to obtain good estimate of $g_b - g_0$. For notational simplicity, we might write $z_j(b)$ simply as $z_j$.

    \mbox{}

    \underline{1. Decomposition of $g_b$.}

    We claim the following decomposition of $g_b$ for $0 < b \le b_0$: for each $1 \le j \le k_0$,
    \be
    g_b(z) = \frac{\pa_z^{k_0}f_b(z_j)}{k_0!} - \sum_{\substack{1 \le k \le  k_0 \\ k \neq j}} \frac{h_{j;b}(z_k)}{(z-z_k)  \cdot  \prod_{\substack{1 \le l \le k_0 \\ l \neq k}}(z_k- z_l)} + \frac{h_{j;b}(z)}{\prod_{1 \le k \le k_0}(z-z_k)}\label{eqgbrep}
    \ee
    where 
    \bee
     h_{j;b}(z) = \frac{1}{(k_0 + 1)!}\int_{z_j}^z \pa_z^{k_0+1} f_b(w) (z-w)^{k_0} dw.
    \eee

Indeed, the Taylor expansion at $z_j(b)$ indicates that
$$P_{j;b}(z) := \sum_{n=0}^{k_0} \frac{\pa_z^n f_b(z_j)}{n!}(z-z_j)^n = f_b(z) - h_{j;b}(z). $$
Notice that $f_b(z_k) = 0$ for $1 \le k \le k_0$ indicate that $P_{j;b}$ is the Lagrangian interpolation polynomial at $(z_k, -h_{j;b}(z_k))$ with highest order coefficient $\frac{\pa_z^{k_0}f_b(z_j)}{k_0!}$, so 
\[ P_{j;b}(z) =  \frac{\pa_z^{k_0}f_b(z_j)}{k_0!} \prod_{k=1}^{k_0}(z-z_k) -\sum_{\substack{1 \le k \le k_0 \\ k \neq j} } h_{j;b}(z_k) \cdot \prod_{\substack{1 \le l \le k_0 \\ l \neq k}} \frac{z - z_l}{z_k - z_l}. \]
Plugging this back to $f_b = P_{j;b} - h_{j;b}$ and $g_b = f_b / (\prod_{k=1}^{k_0}(z-z_k))$ implies \eqref{eqgbrep}.

Besides, a direct Taylor expansion at $z = 0$ implies 
\be
 g_0(z) = \frac{\pa_z^{k_0}f_0(0)}{k_0!} + \frac{z^{-k_0}}{(k_0 + 1)!}\int_{0}^z \pa_z^{k_0+1} f_0(w) (z-w)^{k_0} dw. 
 \label{eqg0rep}
\ee
From condition (2), $A_0 :=  \frac{|\pa_z^{k_0}f_0(0)|}{k_0!} \neq 0$. 

    \mbox{}

    \underline{2. Estimates.}

    Using the representation formula for $g_b$ \eqref{eqgbrep} and for $g_0$ \eqref{eqg0rep}, for each $ 1\le j \le k_0$, 
    \bee
    |g_b(z) - g_0(z)| &\le& \frac{1}{k_0!} \left| \pa_z^{k_0} f_b(z_j) - \pa_z^{k_0} f_b(0) \right| + \left| \frac{z^{-k_0}}{(k_0 + 1)!}\int_{0}^z \pa_z^{k_0+1} f_0(w) (z-w)^{k_0} dw \right| \\
    &+& \left| - \sum_{\substack{1 \le k \le  k_0 \\ k \neq j}} \frac{h_{j;b}(z_k)}{(z-z_k)  \cdot  \prod_{\substack{1 \le l \le k_0 \\ l \neq k}}(z_k- z_l)} + \frac{h_{j;b}(z)}{\prod_{1 \le k \le k_0}(z-z_k)} \right| 
    \eee
    Using boundedness of $\| f_0 \|_{C^{k_0 + 1}(B_{\delta_0/2})}$ and continuity condition (3), the first two terms can be made arbitrarily small when taking $\delta_1$, $b_1$ small enough, uniformly for $0 < b \le b_1$, $z \in \Omega_{\delta_1; I_0, b}$ and $1 \le j \le k_0$. So to guarantee the non-vanishing of $g_b$ on $\Omega_{\delta_1;I_0,b}$,  it suffices to show 
    \be  \sup_{0 < b \le b_1} \sup_{z \in \Omega_{\delta_1; I_0, b}} \min_{1 \le j \le k_0} \left| - \sum_{\substack{1 \le k \le  k_0 \\ k \neq j}} \frac{h_{j;b}(z_k)}{(z-z_k)  \cdot  \prod_{\substack{1 \le l \le k_0 \\ l \neq k}}(z_k- z_l)} + \frac{h_{j;b}(z)}{\prod_{1 \le k \le k_0}(z-z_k)} \right| \le \frac{A_0}{2} \label{eqdegbifres}  \ee
    Depending on the separation of $|z_1 - z_2|$, we identify two cases of $b$. 

    % In this step, we look for $\delta_1$, $b_1$ small enough such that 
    % \be  \sup_{0 < b \le b_1} \sup_{z \in \Omega_{\delta_1; I_0, b}} \left| - \sum_{\substack{1 \le k \le  k_0 \\ k \neq j}} \frac{h_{j;b}(z_k)}{(z-z_k)  \cdot  \prod_{\substack{1 \le l \le k_0 \\ l \neq k}}(z_k- z_l)} + \frac{h_{j;b}(z)}{\prod_{1 \le k \le k_0}(z-z_k)} \right| \le \frac{A_0}{2} \label{eqdegbifres}  \ee
    % This will concludes the proof, since the other errors in $g_b - g_0$ can be easily controlled using boundedness of $\| f_0 \|_{C^{k_0 + 1}(B_{\delta_0/2})}$ and continuity condition (3) when taking $\delta_1$, $b_1$ small enough. 

    \textit{Case 1. $k_0 = 1$ or $k_0 \ge 2$ with $|z_1 - z_2| \ge \frac{b}{4C_*}$.}

    Denote for $\delta \le \delta_0$
    \[ D_{\delta, b} := \Omega_{\delta;I_0, b} -\left(  \cup_{1 \le k \le k_0} B_{\frac{b}{8C_*}}(z_j(b)) \right). \]
    Recall condition (1) that $|z_k| \lesssim b$ for every $k$. For $z \in D_{\delta, b}$, we have $|z-z_k| \sim \max \{b, |z|\}$ for every $k$. Let $j = 1$, with $|h_{1;b}(z)| \lesssim |z-z_1|^{k_0+1}$ from the boundedness condition (3), we see 
    \bee
      \sum_{\substack{2 \le k \le  k_0}} \left|  \frac{h_{1;b}(z_k)}{(z-z_k)  \cdot  \prod_{\substack{1 \le l \le k_0 \\ l \neq k}}(z_k- z_l)} \right| + \left|\frac{h_{1;b}(z)}{\prod_{1 \le k \le k_0}(z-z_k)} \right|  \lesssim b + |z|,\quad \forall\, z \in D_{\delta, b}.
    \eee
    For $z \in B_{\frac{b}{8C_*}}(z_j(b))$ for $1 \le j \le k_0$, we notice that $|z-z_k| \sim b$ for $k \neq j$ and $|h_{j;b}(z)| \lesssim |z-z_j|^{k_0 + 1}$, so 
    \bee
     \sum_{\substack{1 \le k \le  k_0 \\ k \neq j}}  \left| \frac{h_{j;b}(z_k)}{(z-z_k)  \cdot  \prod_{\substack{1 \le l \le k_0 \\ l \neq k}}(z_k- z_l)} \right| + \left|\frac{h_{j;b}(z)}{\prod_{1 \le k \le k_0}(z-z_k)} \right| \lesssim b+ \frac{|z-z_j|^{k_0}}{b^{k_0-1}} \sim b.
    \eee
    These bounds imply \eqref{eqdegbifres} with $\delta_1$, $b_1$ small enough depending on $A_0$. 

    \textit{Case 2. $k_0 \ge 2$ and $|z_1 - z_2| \le \frac{b}{4C_*}$.}

    We first notice the elementary difference estimates
    \be |h_{j;b}(z) - h_{j;b}(z')| \lesssim |z-z'| \left( |z-z'|^{k_0} + |z-z_j|^{k_0} \right),\quad  \label{eqyeyeye} \ee
    Now let 
     \[ D_{\delta, b}' := \Omega_{\delta;I_0, b} -\left(  \cup_{2 \le k \le k_0} B_{\frac{b}{2C_*}}(z_j(b)) \right). \]
     For $z \in D_{\delta, b}'$, we take $j = 1$ and estimate similarly as in Case 1 using the cancellation $|h_{1;b}(z_2)| (z_2 - z_1)^{-1} \lesssim |z_1 - z_2|^{k_0} \lesssim b^{k_0}$. For $z \in  B_{\frac{b}{2C_*}}(z_j(b)) $ with $j \ge 3$, we compute
     \bee
       \sum_{\substack{3 \le k \le  k_0 \\ k \neq j}}  \left| \frac{h_{j;b}(z_k)}{(z-z_k)  \cdot  \prod_{\substack{1 \le l \le k_0 \\ l \neq k}}(z_k- z_l)} \right| + \left|\frac{h_{j;b}(z)}{\prod_{1 \le k \le k_0}(z-z_k)} \right| \lesssim b. \\
       \left|\sum_{k = 1}^2 \frac{h_{j;b}(z_k)}{(z-z_k)  \cdot  \prod_{\substack{1 \le l \le k_0 \\ l \neq k}}(z_k- z_l)} \right| = \frac{1}{|z_1 - z_2|} \left| \sum_{k=1}^2 \frac{h_{j;b}(z_k) (-1)^k}{(z-z_k) \prod_{l \ge 3}(z_k-z_l)} \right| \lesssim b.
     \eee
     where we exploited \eqref{eqyeyeye}. 
     Finally, for $z \in B_{\frac{b}{2C_*}}(z_2(b))$ and $|z-z_1| \le |z-z_2|$, we have $|z_1 - z_2| \le 2|z-z_2|$ and can take $j = 1$ to compute
     \bee
       \sum_{k \ge 3} \left| \frac{h_{1;b}(z_k)}{(z-z_k)  \cdot  \prod_{\substack{1 \le l \le k_0 \\ l \neq k}}(z_k- z_l)} \right| +
       \left| \frac{h_{1;b}(z)}{ \prod_{\substack{1 \le l \le k_0}} (z - z_l) }\right| + \left| \frac{h_{1;b}(z_2)}{(z-z_2) \prod_{\substack{1 \le l \le k_0 \\ l \neq 2}} (z_2 - z_l) } \right| \lesssim b.
     \eee
     The case of $z \in B_{\frac{b}{2C_*}}(z_2(b))$ and $|z-z_1| \ge |z-z_2|$ follows by the same computation with $j = 2$. 
     These bounds verifies \eqref{eqdegbifres} in case 2 when $\delta_1, b_1$ are taken small enough and hence conclude the proof.  
\end{proof}

\subsection{Proof of Theorem \ref{thmmodestabsmallspec}}
 
\begin{proof}[Proof of Theorem \ref{thmmodestabsmallspec}]\label{sec72}
Since Fourier transform preserves spherical harmonics \cite[Theorem 3.10]{MR0304972}, we have 
\[
\| f \|_{\dot H^\sigma(\RR^d)}^2 = \left| \begin{array}{ll}
    \sum_{l = 0}^\infty \sum_{m = 1}^{N_{d, l}} \| f_{l,m}(\cdot) Y_{d, l, m}(\cdot/|\cdot|)  \|_{\dot H^\sigma(\RR^d)}^2 & d \ge 2 \\
    \sum_{l = 0}^1 \| f_{l,1}(\cdot) Y_{1, l, 1}(\cdot/|\cdot|)  \|_{\dot H^\sigma(\RR)}^2   & d = 1
\end{array} \right.
\]
% This can be proven by applying the spherical harmonics decomposition on weighted $L^2$ space to its Fourier transform, and notice that the Fourier transform preserves spherical harmonic class and the spherical harmonic function (both parameters $l$ and $m$) \cite[Theorem 3.10]{MR0304972}.
Therefore, for $d \ge 1$, using  the explicit eigenfunctions from \eqref{eqeigencalHb1}-\eqref{eqeigencalHb3}, we can rewrite \eqref{eqcharacspecHb} and \eqref{eqRieszchar} as 
\be
 \cup_{k \ge 1} \ker (\calH_b\big|_{(\dot H^\sigma_l(\RR^d))^2} - \l)^k = \left| \begin{array}{ll}
{\rm span}\{ \xi_{0, b} \}, & l = 0, \l = 0; \\
{\rm span}\{ \xi_{0, b} + 2b\xi_{1, b} \}, & l = 0, \l = -2bi; \\
{\rm span}\{ \zeta_{0, j, b} \}_{1 \le j \le d},  & l = 1, \l = -bi; \\
 \{ 0 \} & {\rm otherwise;}
 \end{array}\right. \label{eqeigenresult}
\ee
where $l \ge 0, \l \in \Omega_{\delta, \sigma - s_c; b}$. 
Here the vector functions $\xi_{0,b}, \xi_{1, b}, \zeta_{0, j, b}$ are from \eqref{eqdefxi01b} and belongs to $\dot H^\sigma(\RR^d)$ for any $\sigma > s_c$ due to \eqref{eqQbasymp8} (assuming $s_c^*(d) \le s_c^{(0)'}(d)$ to apply Proposition \ref{propQbasympref}). 

We will first consider high spherical classes for $d \ge 2$ and $l \ge l_*(d) \gg 1$. Then for the low spherical classes (including $d=1$ case), we apply the Jost functions argument to bifurcate from the spectrum of $\calH_0$ near zero. We will frequently use the parameter $\nu = l + \frac{d-2}{2}$ as \eqref{eqdefnu} in replacement of $l$.

\mbox{}

\underline{1. High spherical classes.} 

If $d \ge 2$, we recall $(s_{c;{\rm high}}^*, b_{\rm high}, \delta_{\rm high})$ from \eqref{eqdefschigh} and $\nu_{\rm int}(d)$ from Proposition \ref{propintfund} (4). For 
\[ 0 < s_c \le s_{c;{\rm high}}^*, \quad 0 < b(d, s_c) \le b_{\rm high}, \quad \l \in \Omega_{\delta_{\rm high}; \frac 12, b},\quad \nu \ge \max \{ \nu_{\rm int}(d), 1\}, \]
we recall the fundamental solutions $\check \Psi_{j;b,\l,\nu}$ and $\Phi_{j;b,\l,\nu}$ of \eqref{eqnu} from Proposition \ref{propintfund} (4) and Proposition \ref{propextfundh}. 

   Similar to Lemma \ref{lemJostWronskian}, 
   we first define their Jost function as
   \be
    \hat \Wfr_\nu(b, \l, r) = \det \left( \begin{array}{cccc}
       \check \Psi_{1;b,\l,\nu}(r) & \check \Psi_{2;b,\l,\nu}(r) &  \Phi_{1;b,\l,\nu}(r) & \Phi_{2;b,\l,\nu}(r) \\
       \pa_r \check \Psi_{1;b,\l,\nu}(r) & \pa_r \check \Psi_{2;b,\l,\nu}(r) & \pa_r \Phi_{1;b,\l,\nu}(r) & \pa_r \Phi_{2;b,\l,\nu}(r) 
    \end{array}\right).\label{eqJostdef3}
\ee
    As a Wronskian function of \eqref{eqnu}, it is also independent of $r$ from the first-order formulation \eqref{eqnufirstorder}. Evaluating at $r = b^{-\frac 12}$ using the asymptotics \eqref{eqinthighnu1}, \eqref{eqinthighnu2}, \eqref{eqinthighnu1h} and \eqref{eqinthighnu2h}, we see 
   \bee
     \check \Wfr_{\nu}(b, \l) = C_{b, \l,\nu}^* \cdot \det\left[ \left( \begin{array}{cccc} 
    1 & & 1 & \\
    & 1 & & 1 \\
    1 & & -1 & \\
    & 1 & & -1
     \end{array}\right) + O(b^\frac 12 + \nu^{-1})  \right].
   \eee
   where $C_{b, \l, \nu}^* =  \left( \ti_{\nu, 1+\l} \ti_{\nu, 1-\l} \psi_1^{b, 1+\l+ibs_c, \nu} \overline{\psi_1^{b, 1-\bar \l + ibs_c, \nu}} \right)(b^{-\frac 12}) \cdot \sqrt{(1 + b\nu^2)^2 - \l^2} \neq 0$. Hence there exist $\check \nu_{\rm int}(d) \ge \nu_{\rm int}(d) \gg 1$ and $\check b_{\rm high}(d) \le b_{\rm high}(d) \ll 1$, such that 
   \[\check \Wfr_\nu(b, \l) \neq 0,\quad {\rm when}\,\,\, 0 < b(d, s_c) \le \check b_{\rm high},\,\,\,\l \in \Omega_{\delta_{\rm high}; \frac 12, b},\,\,\,\nu \ge \check \nu_{\rm int}. \]
   With \eqref{eqlindepcheckPsi}, the non-vanishing of Wronskian indicates that  
   \[ {\rm span}\{ \Psi_{j;b, \l,\nu} \}_{1 \le j \le 2} \cap {\rm span}\{ \Phi_{j;b, \l,\nu} \}_{1 \le j \le 2} = \{ 0 \}.  \] 
   Lemma \ref{lemregfund} (2) now implies $\ker (\calH_b\big|_{(\dot H^\sigma_l(\RR^d))^2} - \l) = \{ 0 \}$ for 
   \bee 0 < s_c \le s_{c;{\rm high}}^*, \quad 0 < b(d, s_c) \le \check b_{\rm high}, \quad \l \in \Omega_{\delta_{\rm high}; \frac 12, b},\quad \nu \ge  \check  \nu_{\rm int}, \quad \sigma \in \left(s_c, \frac 12\right). \eee
Namely \eqref{eqeigenresult} holds with the parameters in the above range. 

\mbox{}

\underline{2. Low spherical classes.} 

Let $(s_{c;{\rm low}}^*, b_{\rm low}, \delta_{\rm low})$ be from \eqref{eqdefsclow} with $d=1$ or $d \ge 2$ and $\nu_0 = \check \nu_{\rm int}(d)$ from the former step. For 
\[ 0 \le s_c \le  s_{c;{\rm low}}^*,\quad  0 \le b(d, s_c) \le b_{\rm low}, \quad \l \in \Omega_{\delta_{\rm low}; I_0, b},\quad  \nu \le \nu_0 = \check \nu_{\rm int},\]
we have the Jost function $\Wfr_{\nu}(b,\l)$ from Lemma \ref{lemJostWronskian}. We also recall the projection $\PP_{d, \nu, b}$ from \eqref{eqprojsphrad}.

From Proposition \ref{propspecH0}, $0 \in \sigma_{\rm disc}(\calH_0)$, so there exists $\bar \delta > 0$ such that $B_{\bar \delta}^{\CC} \cap \sigma(\calH_0) = \{ 0 \}$. Combined this with Proposition \ref{propspecH0}(2), Lemma \ref{lemvanishJW}, and \eqref{eqchargeneigenb0} from Lemma \ref{lemregfund} (1), we obtain the following characterization of zeros of $\Wfr_\nu(0,\cdot)$:
\begin{enumerate}
    \item \textit{ $\pa_\l^k \Wfr_{\frac{d-2}{2}}(0, 0) = \pa_\l^{k'}\Wfr_{\frac{d}{2}}(0, 0) = 0$
for $k = 0, 1, 2, 3$ and $k' = 0, 1$;}
\item \textit{$\Wfr_{\frac{d-2}{2}}(0,\l) / \l^4$, $\Wfr_{\frac{d}{2}}(0,\l) / \l^2$ and $\Wfr_\nu(b,\l)$ for $d \ge 2$, $2 \le l \le \nu_0-\frac{d-2}{2}$ does not vanish in $|\l| \le \min\{ \delta_{\rm low}, \bar \delta\}$. }
\end{enumerate}

For $0 < s_c  \le  s_{c;{\rm low}}^*$ and $0 < b(d, s_c) \le b_{\rm low}$, recall the radial eigenfunctions $Z_{0,b}:=\xi_{0, b}$, $Z_{1, b}:= \xi_{0, b} + 2b\xi_{1, b}$ from \eqref{eqeigencalHb1} and $Z_{2, b}$, $Z_{3, b}$ from Proposition \ref{propbifeigen}, and the corresponding eigenvalues $\l_{0, b} = 0$, $\l_{1, b} = -2bi$ from \eqref{eqeigencalHb1} and $\l_{2, b} = i(\upsilon_{2, b}-bs_c) = 2bi + iO_\RR(b^{-1}s_c)$, $\l_{3, b} =i(\upsilon_{3,b}- bs_c) = i\upsilon_{3,b} + iO_\RR(b^{-1}s_c)$ with $\upsilon_{3, b} \sim b^{-2} s_c$ from Proposition \ref{propbifeigen}. Due to the smoothness of $Z_{k,b}$ near $0$ and decay of $\pa_r^n (e^{i\frac{br^2}{4}} Z_{k,b}^1)$ at infinity from 
\eqref{eqdecayZkb}, we know these radial functions come from the admissible interior and exterior fundamental solution of $\HH_{b, \frac{d-2}{2}} - \l_{k,b}$, namely
\be \PP_{d,\frac{d-2}{2},b} Z_{k,b} \in {\rm span} \{  \Psi_{j;b,\l_{k, b},\frac{d-2}{2}} \}_{\substack{1 \le j \le 2}} \cap  {\rm span} \{ \Phi_{j;b,\l_{k, b},\frac{d-2}{2}}  \}_{\substack{ 1 \le j \le 2}},\quad \forall k = 0, 1, 2, 3.  \label{eqexterioreigenfunc} \ee
Similarly, the eigenpairs in $l=1$ spherical class $(\tilde \l_{0, b}, \tilde Z_{0, j', b}) = (-bi, \zeta_{0,j',b})$, $(\tilde \l_{1, b}, \tilde Z_{1, j', b}) = (bi, \zeta_{0, j', b} - 2b\zeta_{1,j',b})$ for $1 \le j' \le d$ \eqref{eqeigencalHb3} and their regularity from \eqref{eqQbasymp8} indicate that $\PP_{d,\frac{d}{2},b} \tilde Z_{k',j',b} = \PP_{d,\frac{d}{2},b} \tilde Z_{k',1,b}$ for any $j'$ and 
\[ \PP_{d,\frac{d}{2},b} \tilde Z_{k',1,b} \in {\rm span} \{  \Psi_{j;b,\tilde \l_{k', b},\frac{d}{2}} \}_{\substack{1 \le j \le 2}} \cap  {\rm span} \{ \Phi_{j;b,\tilde \l_{k', b},\frac{d}{2}}  \}_{\substack{ 1 \le j \le 2}},\quad \forall k' = 0, 1. \]
Therefore, Lemma \ref{lemvanishJW} implies 

     \, (3) \textit{ $\Wfr_{\frac{d-2}{2}}(b, \l_{k,b}) = \Wfr_{\frac d2}(b, \tilde \l_{k',b}) = 0$ for $k = 0, 1, 2, 3$ and $k' = 0, 1$. }
     
Denote $\l_{k, 0} = \tilde \l_{k',0} = 0$. Then the analytic bifurcation family $\Wfr_{\frac{d-2}{2}}(b,\cdot)$ with zeros $\{\l_{k, b}\}_{0 \le k \le 3}$,  $\Wfr_{\frac{d}{2}}(b,\cdot)$ with zeros $\{\tilde \l_{k', b}\}_{k' = 0, 1}$, and  $\Wfr_{\nu} (b,\cdot)$ for $\frac{d+2}{2}\le\nu\le \nu_0$ with no known zeros, satisfy the condition (1), (2) and (4) in Lemma \ref{lemuniqzero} from (1)-(3) above, particularly the only non-equidistributed pair being $\l_{0, b}, \l_{3, b}$. The uniform continuity in Lemma \ref{lemJostWronskian}(3) implies condition (3) in Lemma \ref{lemuniqzero}. So we can apply Lemma \ref{lemuniqzero} to see 

\, (4) \textit{There exists $\check b_{\rm low} \le b_{\rm low}$, $\check \delta_{\rm low} \le \min \{ \delta_{\rm low}, \bar \delta\}$, such that for
$0 < b(d, s_c) \le \check b_{\rm low}$, the functions $\Wfr_{\frac{d-2}{2}}(b,\l) /  (\prod_{k=0}^3 (\l -  \l_{k, b}))$, $\Wfr_{\frac{d}{2}}(b,\l) / /  (\prod_{k'=0}^1 (\l - \tilde \l_{k', b}))$ and $\Wfr_\nu(b,\l)$ for $d \ge 2$, $2 \le l \le \nu_0-\frac{d-2}{2}$ do not vanish in $\l \in   \Omega_{\check \delta_{\rm low};10,b}$.}

 Further by Lemma \ref{lemvanishJW}, this implies 

\, (4') \textit{For any $0 \le k \le 3$, $0 \le k' \le 1$, and $0 < b \le b_{\rm low}$, 
\bee
  {\rm span} \{ \pa_\l^n \Psi_{j;b,\l_{k,b},\frac{d-2}{2}} \}_{\substack{n \ge 0 \\ 1 \le j \le 2}} \cap  {\rm span} \{ \pa_\l^n \Phi_{j;b,\l_{k,b},\frac{d-2}{2}} \}_{\substack{n \ge 0 \\ 1 \le j \le 2}} = {\rm span} \{ \PP_{d,\frac{d-2}{2},b} Z_{k,b} \}; \\
  {\rm span} \{ \pa_\l^n \Psi_{j;b,\tilde \l_{k',b},\frac{d}{2}} \}_{\substack{n \ge 0 \\ 1 \le j \le 2}} \cap  {\rm span} \{ \pa_\l^n \Phi_{j;b,\tilde \l_{k',b},\frac{d}{2}} \}_{\substack{n \ge 0 \\ 1 \le j \le 2}} = {\rm span} \{ \PP_{d,\frac{d}{2},b} Z_{k',1,b} \}; 
\eee
}

Now we pick 
\be \epsilon^*(s_c) := b^{-1}\Im \l_{3,b} \sim b^{-3} s_c \ll \frac 12. \label{eqchoiceepsilon*} \ee
For $\sigma \in (s_c, s_c + \epsilon^*(s_c))$, we have
$$\Im \l_{2,b} > \Im \l_{3, b} > b(\sigma - s_c) > \Im \l_{0, b} > \Im \l_{1, b},\quad \Im \tilde \l_{1,b} > b(\sigma - s_c) > \Im \tilde \l_{0, b}.$$
Thus with (4') above, \eqref{eqchargeneigenb1}-\eqref{eqchargeneigenb2} and that
\be
\begin{split}& \ker (\calH_b\big|_{(\dot H^\sigma_l(\RR^d))^2} - \l) =  \ker (\calH_b\big|_{(\dot H^\sigma_l(\RR^d))^2} - \l)^2 \\
\Rightarrow\,\,& \ker (\calH_b\big|_{(\dot H^\sigma_l(\RR^d))^2} - \l) = \cup_{k \ge 1}  \ker (\calH_b\big|_{(\dot H^\sigma_l(\RR^d))^2} - \l)^k, \end{split}\label{eqiterationgeneralizedeig}
\ee
we derive \eqref{eqeigenresult} for 
  \bee 0 < s_c \le s_{c;{\rm low}}^*, \,\, \quad 0 < b(d, s_c) \le \check b_{\rm low}, \,\, \sigma \in (s_c, s_c + \epsilon^*(s_c)), \,\, \l \in   \Omega_{\check \delta_{\rm low}; \sigma - s_c, b},\,\, \nu \le \check \nu_{\rm int}.  \eee

\mbox{}

\underline{3. End of proof of Theorem \ref{thmmodestabsmallspec}} 

We set $\delta := \min \{ \delta_{\rm high}, \check \delta_{\rm low}\}$, and choose $s_c^*(d) \le \min \{ s_{c;{\rm high}}^*, s_{c;{\rm low}}^*\}$ such that $\sup_{0 < s_c \le s_c^*(d)}b(d, s_c) \le \min \{ \check b_{\rm high}, \check b_{\rm low}\}$. Lastly, we choose $\epsilon^*(s_c)$ as in \eqref{eqchoiceepsilon*}. Then the results from the previous two steps imply \eqref{eqeigenresult} for 
\[ 0 < s_c \le s_c^*(d), \quad \sigma \in (s_c, s_c + \epsilon^*(d)), \quad \l \in  \Omega_{\delta;\sigma -s_c, b},\quad l \ge 0.  \]
That concludes the proof of Theorem \ref{thmmodestabsmallspec}.

\mbox{}

\underline{4. Structure of the spectrum}

Finally, we verify the statements in Comment 2 of Theorem \ref{thmmodestabsmallspec}. The existence of bifurcated eigenmodes are proven in Proposition \ref{propbifeigen}, with the estimate of eigenfunction follows \eqref{eqexterioreigenfunc} and \eqref{eqadmosc}. For the uniqueness part, we first notice that the asymptotics in Proposition \ref{propintfund}, Proposition \ref{propextfund} and Proposition \ref{propextfundin} plus interior elliptic regularity indicates that for $\l \in   \Omega_{\delta;10,b}$, 
\bee
  \PP_{d,l,b}^{-1} \left( {\rm span} \{  \Psi_{j;b,\l,\nu} \}_{\substack{1 \le j \le 2}} \cap  {\rm span} \{ \Phi_{j;b,\l,\nu} \}_{\substack{ 1 \le j \le 2}}\right) \subset  C^\infty \cap  \dot H^{11}_{l};\\
  \Psi_{j';b,\l,\nu}, \Phi_{j';b,\l,\nu} \notin C^\infty \cap  \dot H^{11}_{l},\quad  {\rm for}\,\, j' \in \{ 3, 4\}.
\eee
Then the uniqueness statement follows from statements (4), (4') above and the argument \eqref{eqiterationgeneralizedeig}.
\end{proof}

\appendix

\section{Properties of ground state and related linear operators}\label{appA1}

We prove Lemma \ref{lemQasymp}, Lemma \ref{lemcontgs} and Lemma \ref{lemLpm}

\mbox{}

\begin{proof}[Proof of Lemma \ref{lemQasymp}]
\underline{1. Almost sharp decay by comparison theorem.} 
Firstly, we know $|Q(r)| \lesssim e^{-\a r}$ for some $\a > 0$ from \cite[Proposition B.7]{MR2233925}.

Let $\varphi_{\delta,\a} = r^{-\frac{d-1}{2} + \a} e^{-\delta r}$ with $0 < \delta \le 1$ and $\a \ge 0$. Compute 
$$(\delta^2 - \Delta) \varphi_{\delta, \a} = \left(\a \delta r^{-1} + \frac{(d-1 - 2\a)(d-3 + 2\a)}{4r^2}\right) \varphi_{\delta, \a}. $$
Then for $a > 0$, and $ 0 < \delta  < 1$, 
\[ \left(\frac{\delta^2 + 1}{2} - \Delta\right) (a\varphi_{\delta, 0} - Q) = \left(\frac{1-\delta^2}{2} + \frac{(d-1)(d-3)}{4r^2}\right) a\varphi_{\delta, 0} + \left( \frac{1-\delta^2}{2} - Q^{p-1}\right)Q  \]
From the decay of $Q$ and positivity, there exists $r_\delta \gg 1$ such that $\left(\frac{\delta^2 + 1}{2} - \Delta\right) (a\varphi_{\delta, 0} - Q)  > 0$ on $B_{r_\delta}^c$ for any $a > 0$. Choosing $a \gg 1$ such that $(a\varphi_{\delta, 0} - Q)(r_\delta) > 0$, the comparison principle indicates $Q < a \varphi_{\delta, 0}$ on $B_{r_\delta}^c$, namely $Q \lesssim r^{-\frac{d-1}{2}} e^{-\delta r}$ for any $0 < \delta < 1$. 

Next, compare $a \varphi_{1, \a}$ with $Q$, 
\[  (1 - \Delta)(a \varphi_{1, \a} - Q) = \left(\a \delta r^{-1} + \frac{(d-1 - 2\a)(d-3 + 2\a)}{4r^2}\right) \varphi_{\delta, \a} - Q^p. \]
Since $Q^p \lesssim e^{-(p-\e) r}$ for any $\e > 0$, for any $\a > 0$, we can take $r_\a \gg 1$ and $a \gg 1$ such that the comparison principle holds on $B_{r_\a}^c$. To conclude, we have proven 
\be  |Q(r)| \lesssim_\e \la r \ra^{-\frac{d-1}{2} + \e} e^{-r},\quad \forall \, \e > 0. \label{eqQalmostdecay} \ee

\mbox{}

\underline{2. ODE argument.}
Write the nonlinear ODE for $\tilde Q(r) := r^{\frac {d-1}{2}}Q(r)$ 
\[ \left(-\pa_r^2 + \frac{(d-1)(d-3)}{4r^2} + 1\right) \tilde Q + r^{-\frac{(d-1)(p-1)}{2}} \tilde Q^p  = 0.\]
By Duhamel formula, with $F(r) = \frac{(d-1)(d-3)}{4r^2} \tilde Q + r^{-\frac{(d-1)(p-1)}{2}} \tilde Q^p$, for any $r_0 > 0$, there exists $a_\pm(r_0) \in \RR$ such that 
\[ \tilde Q(r) = \sum_{\pm} a_\pm(r_0) e^{\pm r} - e^r \int_{r_0}^r e^{-s} F(s) \frac{ds}{2} + e^{-r} \int_{r_0}^r e^{s} F(s) \frac{ds}{2}. \]
Take $\e = \frac 12$ in \eqref{eqQalmostdecay}, we know both integrals above converge at infinity, and the decay of $\tilde Q$ implies $a_+ - \int_{r_0}^\infty e^{-s} F(s) ds = 0$. So we rewrite 
\be \tilde Q(r) = \left( a_-(r_0) + \int_{r_0}^\infty e^s F(s)\frac{ds}{2} \right) e^{-r} + e^r \int_r^\infty e^{-s} F(s) \frac{ds}{2} - e^{-r} \int_r^\infty e^s F(s) \frac{ds}{2}. \label{eqQduhameldecay}  \ee
Thanks to the $r^{-2}$ decay, the a priori bound \eqref{eqQalmostdecay} now implies the sharp decay $|\tilde Q(r)| \lesssim e^{-r}$. 

Now we can prove \eqref{eqsolitondecay}, which can be reduced to $|\pa_r^k \tilde Q| \lesssim e^{-r}$ for any $k \ge 0$ at infinity. The $k=0$ case has been proven above, and $k = 1$ case follows by differentiating the above equation. For $k \ge 2$, we use the ODE for $\tilde Q$.

\mbox{}

\underline{3. Asymptotic expansion.}

Denote 
$$\kappa_Q = a_-(r_0) + \int_{r_0}^\infty e^s F(s)\frac{ds}{2} $$
from \eqref{eqQduhameldecay}. 
The sharp bound \eqref{eqsolitondecay} and \eqref{eqQduhameldecay} indicate that $R(r) := \tilde Q(r) - \kappa_Q e^{-r} = O(r^{-1} e^{-r})$. Subtracting $\kappa_Q e^{-r}$ from \eqref{eqQduhameldecay},  we obtain
\[  R(r) = \sum_\pm (\pm 1) e^{\pm r} \int_r^\infty e^{\mp s} \left( \frac{(d-1)(d-3)\kappa_Q e^{-s}}{4s^2} + G(s) \right)\frac{ds}{2} \]
where $G(r) = F(r) - \frac{(d-1)(d-3)\kappa_Q e^{-r}}{4r^2} = O(r^{-3}e^{-r})$. So the $\kappa_Q$ gives a leading order $c_Q r^{-1} e^{-r}$ term for $R(r)$ and the residual are bounded by $r^{-2} e^{-r}$. That concludes \eqref{eqQdecay1} for $k=0$, and the case $k=1$ again follows differentiating the equation of $R(r)$. 
\end{proof}

\mbox{}

\begin{proof}[Proof of Lemma \ref{lemcontgs}]
  Consider $F(f, p) :=  f - (-\Delta + 1)^{-1} |f|^{p-1} f$, we claim $F$ is a $C^1$ map from $L^2_{rad} \cap L^\infty_{rad}$ onto itself. Indeed, the Fr\'echet derivatives formally are 
  \[ \pa_f F = 1 - p(-\Delta + 1)^{-1} |f|^{p-1},\quad \pa_p F = -(-\Delta + 1)^{-1}(|f|^{p-1} f \cdot \log (|f|) ).  \]
  Notice that $f \in L^\infty$ implies $|f|^{p-1}f \log(|f|) \in L^\infty$, and $(-\Delta + 1)^{-1}(L^2 \cap L^q) \subset H^2 \cap W^{2, q} \hookrightarrow L^2 \cap L^\infty$ with $q \in (\frac d2,\infty)$, so they are both bounded linear operator on $L^2 \cap L^\infty$. To verify they are actual derivatives, we estimate
  \bee
   &&F(f+h, p) - F(f, p) - \pa_f F(f, p) h \\
   &=& - (-\Delta + 1)^{-1} \left( |f+h|^{p-1} (f+h) - |f|^{p-1} f - p|f|^{p-1}h \right) 
   = O_{L^2 \cap L^\infty}(\| h \|_{L^2 \cap L^\infty}^{\min\{p, 2\}}) \\
   &&F(f, p+\delta) - F(f, p) - \pa_p F(f, p) \delta \\
   &=& -(-\Delta +1)^{-1}\left[  (|f|^{\delta} - 1 - \delta \log(|f|)) |f|^{p-1} f \right]  = O_{L^2 \cap L^\infty} ( \delta^2 (\log \delta)^2)
  \eee
  where we used elementary estimates 
  \bee \left| |f+h|^{p-1}(f+h) - |f|^{p-1}f - p |f|^{p-1}h \right| \lesssim |h|^{\min \{ p, 2\} } |f|^{\max \{ p-2, 0 \} }, \\
  \left| a^\delta - 1 - \delta \log a \right| \lesssim_M \delta^2 \left( (\log \delta)^2 + (\log a)^2 \right),\quad 0 < \delta \le a \le M.
  \eee
  Next, we show $\pa_f F\big|_{(f, p) = (Q, p_0) } = (-\Delta + 1)^{-1} L_+$ is invertible on $L^2 \cap L^\infty$ with $L_+$ defined in \eqref{eqdefLpm}. Here $Q := Q_{p_0}$ for simplicity. Indeed, since $L_+^{-1}$ is bounded on $L^2_{rad}$ \cite[Lemma 2.1]{chang2008spectra}, from the formulation 
  \[L_+^{-1} f =  (-\Delta + 1)^{-1} (pQ^p L_+^{-1} f) +  (-\Delta + 1)^{-1} f, \]
  we can iteratively improve the regularity and show $L_+^{-1} \in \calL(L^2_{rad} \cap L^\infty_{rad})$. Therefore  
  \[ \left((-\Delta + 1)^{-1} L_+  \right)^{-1} = L_+^{-1} (-\Delta + 1) = 1 + p_0 L_+^{-1} \circ ( Q^{p_0} \cdot) \in \calL(L^2 \cap L^\infty). \]
  The implicit function theorem now indicates the existence of $Q_p$ for $|p - p_0| \ll 1$ such that $F(Q_p, p) = 0$ with the estimate \eqref{eqcontinuitysoliton} This $Q_p$ is the ground state of \eqref{eqNLS} with parameter $(d, p)$ thanks to its uniqueness, and that concludes the difference estimate in $L^2 \cap L^\infty$. For high derivatives, it suffices to improve the regularity by iteration using the difference equation $(1 - \Delta)(Q_p - Q_{p_0}) = Q_p^p - Q_{p_0}^{p_0}$, with the help of radial decreasing and non-vanishing of $Q_p$, $Q_{p_0}$. 
\end{proof}

\mbox{}

\begin{proof}[Proof of Lemma \ref{lemLpm}]
    (1) The construction of $A, D$ and their asymptotics are from \cite[Lemma 4.1]{MR4250747}, with the second leading order term near infinity extracted similarly as in the proof of Lemma \ref{lemQasymp}. For $\frakE$, its definition \eqref{eqdefE} directly implies the equation $L_-\frakE = 0$, the asymptotics and Wronskian. 

    (2) The definition of $L_{+;x_*}^{-1}$ and $L_-^{-1}$ comes from Duhamel formula, so the inversion property is straightforward. The bounds \eqref{eqLpmest1}-\eqref{eqLpmest3} follow the asymptotics of $A, D, Q$ above, and we only prove the last one here. For $r \le 2$, we can easily derive $\left| \pa_r^k \left( L_-^{-1} f + |\SS^{d-1}|^{-1}(f, Q)_{L^2(B_{x_*}^{\RR^d})} \frakE \right)\right| \lesssim_\a \| f \|_{Z_{-,\a;x_*}}$ for $k=0,1$ from \eqref{eqsolitondecay}, and for $r \ge 2$, $\tilde{\frakE} = \frakE$, so 
    \bee
      &&L_-^{-1} f (r)+ (f, Q)_{L^2(B_{x_*}^{\RR^d})} \tilde{\frakE} (r)\\
      &=& -Q \int_0^1 \left[\int_0^s f(\tau) Q(\tau) \tau^{d-1} d\tau \right] \frac{ds}{Q^2(s) s^{d-1}} + Q \int_1^r \left[\int_s^{x_*} f(\tau) Q(\tau) \tau^{d-1} d\tau \right] \frac{ds}{Q^2(s) s^{d-1}}\\
      && \pa_r \left( L_-^{-1} f(r) + (f, Q)_{L^2(B_{x_*}^{\RR^d})} \tilde{\frakE}(r)\right) \\
      &=& -Q' \int_0^1 \left[\int_0^s f(\tau) Q(\tau) \tau^{d-1} d\tau \right] \frac{ds}{Q^2(s) s^{d-1}} + Q^{-1}(r) r^{-(d-1)} \int_r^{x_*} f(\tau) Q(\tau) \tau^{d-1} d\tau \\
      &+& Q' \int_1^r \left[\int_s^{x_*} f(\tau) Q(\tau) \tau^{d-1} d\tau \right] \frac{ds}{Q^2(s) s^{d-1}}.
    \eee
    So with \eqref{eqsolitondecay} and $\left|\int_{s}^{x_*} fQ \tau^{d-1} d\tau \right| 
    % \lesssim \int_s^{x_*} \tau^\a e^{-2\tau} d\tau \| f \|_{Z_{-,\a;x_*}} 
    \lesssim_\a s^{\a} e^{-2s} \| f \|_{Z_{-,\a;x_*}}$ for $s \ge 1$, we can obtain the pointwise estimate on $[2, x_*]$ and concludes  \eqref{eqLpmest3}.
\end{proof}

\section{Algebra of $D_{\pm l}$}

Recall the phase function $\phi_{b, E}$ and differential operator $D_{\pm l; b, E}$ from \eqref{eqdefphi}-\eqref{eqDpm}.

The subscripts $b, E$ will be omitted in the proof below.

\begin{lemma}\label{lemfnk}
  Let $0 < b \le \frac 14$, $|E-1| \le \frac 14$ and $l \ge 1$. There exist smooth functions $\{ f_{n, k; \pm l, b, E}\}_{n \ge 1, 1 \le k \le n}$ such that supposing smooth functions $\{ I_{n; \pm l}\}_{0 \le n \le N}$ satisfy
\be D_{\pm l;b, E} I_{n; \pm l} = \pm il\phi_{b, E}' I_{n+1; \pm l}, \quad \forall 0 \le n \le N-1, \label{eqalgrec} \ee
then
\[ D_{\pm l;b, E}^n I_{0; \pm l} = \sum_{k = 1}^n f_{n, k; \pm l, b, E} I_{k; \pm l}, \quad \forall 1 \le n \le N. \]
Moreover, $f_{n, k; \pm l, b, E}$ satisfies
\be |\pa_r^m f_{n, k; \pm l, b, E}(r)| \lesssim_{m, n, k, l} b^k r^{\frac k2} \left|r-\frac{2\sqrt{E}}{b}\right|^{\frac{3k}{2} - n -m},\quad r > b^{-1},\,\, m \ge 0.\label{eqestfnk} \ee
\end{lemma}

\begin{proof}
Firstly, we claim that $f_{n, k; \pm l}$ are determined inductively
	\bee f_{1,1;\pm l} = \pm il\phi', \quad f_{n,0;\pm l} = f_{n, n+1; \pm l} = 0,\quad f_{n, k; \pm l} = \pa_r f_{n-1, k; \pm l} \pm il\phi' f_{n-1, k-1; \pm l}. \eee
 Indeed, $f_{1, 1; \pm l}$ and the choice of $f_{n, 0; \pm l}, f_{n, n+1; \pm l}$ are obvious. Then it suffices to compute
    \bee
    &&D_{\pm l}^{n+1} I_{0; \pm l} = \sum_{k = 1}^n D_{\pm l} \left(f_{n, k; \pm l} I_{k; \pm l}\right)\\
    &=& \sum_{k = 1}^n \left( \pa_r f_{n, k; \pm l} \cdot I_{k; \pm l} + f_{n, k; \pm l} \cdot  (\pm il\phi') I_{k+1; \pm l}\right) =: \sum_{k = 1}^{n+1} f_{n+1, k;\pm l} I_{k; \pm l}.
    \eee
  Therefore, we can inductively show that for all $n \ge 1$, $1 \le k \le n$, 
  \[ f_{n,k;\pm l} \in {\rm span} \left\{ \prod_{j = 1}^k \pa_r^{m_j} \phi' \right\}_{\substack{\sum_{j=1}^{k} m_j = n-k \\ m_j \ge 0\,\,\forall j } }.  \]
  The estimate \eqref{eqestfnk} now comes from the elementary estimate: for $m \ge 0$, 
  \[ |\pa_r^m \phi'| \lesssim_m br^\frac 12 \left|r-\frac{2\sqrt{E}}{b}\right|^{\frac 12-m},\quad \forall \, r > b^{-1}. \] 
\end{proof}
%     Now we apply this algebraic recurrence structure to integrals and $\calD_N[f]$.

% \mbox{}

\begin{lemma}\label{lemgnk}
    Let $0 < b \le \frac 14$, $|E-1| \le \frac 14$ and $l \ge 1$. Then there exist smooth functions $\{ g_{n, k; \pm l, b, E}\}_{n \ge 0, 0 \le k \le n}$ such that 
    \be (D_{\pm l; b, E} \circ (\pm il\phi_{b, E}')^{-1})^n G = \sum_{k=0}^n g_{n,k; \pm l, b, E}  D_{\pm l; b, E}^k G,\quad \forall n \ge 0. \label{eqlemb22}\ee
    Moreover, $g_{n, k; \pm l}$ satisfies
     \be  |\pa_r^m g_{n, k; \pm l, b, E}(r)| \lesssim_{m,n,k,l} b^{-n} r^{- \frac n2}\left|r-\frac{2\sqrt{E}}{b}\right|^{k-\frac 32 n-m},\quad r > b^{-1},\,\,m \ge 0. \label{eqgnk} \ee
    As a corollary, on $[\frac 4b, \infty)$, we also have
    \bea D_{\pm l;b, E}^n \int_r^\infty G = \sum_{k=1}^n f_{n, k; \pm l, b, E}(r)  \sum_{j=0}^k \int_r^\infty D^j_{\pm l;b, E} G(s) \cdot g_{k,j;\pm l, b, E}(s)ds\label{eqD++commutator} \eea
    provided that $G \in C^n([\frac 4b, \infty))$ satisfies
    \be D_{\pm l;b, E}^j G \cdot r^j \in L^1\left( \left[\frac 4b, \infty\right)\right),\quad 0 \le j \le n. \label{equpupasas} \ee
    % {\color{purple} The overall representation \eqref{eqD++commutator} will only be used when defining the non-admissible branch. For the admissible case (namely Lemma \ref{leminvtildeHext}), we iterate only once \eqref{eqD++commutator1} for the oscillating part, and we take plane $\pa_r$ for the non-oscillatory part.}
\end{lemma}
\begin{proof}
    Similar as Lemma \ref{lemfnk}, we derive from \eqref{eqlemb22} the recurrence definition of  
	$g_{n, k; \pm l}$
	\bee  && g_{0,0;\pm l} = 1,\quad g_{n, -1;\pm l} = g_{n, n+1; \pm l} = 0,\\ 
 && g_{n, k; \pm l} = (\pm il\phi')^{-1} g_{n-1,k-1; \pm l} + \pa_r((\pm il\phi')^{-1} g_{n-1,k; \pm l}), \eee
 and the components of $g_{n,k;\pm l}$ for $n \ge 1$, $0 \le k \le n$,
\[  g_{n,k;\pm l} \in {\rm span} \left\{ \prod_{j = 1}^n \pa_r^{m_j} \left[(\phi')^{-1}\right] \right\}_{\substack{\sum_{j=1}^{n} m_j = n-k \\ m_j \ge 0\,\,\forall j } }. 
\] 
    The estimate for $g_{n, k; \pm l}$ \eqref{eqgnk} again follows the elementary estimate for $m \ge 0$
    \[ \left| \pa_r^m \left[ (\phi')^{-1} \right] \right| \lesssim_m b^{-1} r^{-\frac 12} \left|r-\frac{2\sqrt{E}}{b}\right|^{-\frac 12 - m},\quad \forall \, r > b^{-1}.   \]
    
    the above induction formula like \eqref{eqestfnk}, here using $\left|\pa_r^m (\phi')^{-1} \right| \lesssim_m b^{-1} r^{-1-m}$ for all $m \ge 0$.

    Now we show \eqref{eqD++commutator}. 
    Via integration by parts
	\[ \int_r^\infty (D_{\pm l} \circ (\pm il\phi')^{-1}) G ds = -\frac{G(r)}{\pm il\phi'} + \int_r^\infty Gds, \]
	we have 
	\be D_{\pm l} \int_r^\infty G = (\pm il\phi') \int_r^\infty \left(D_{\pm l} \circ (\pm il\phi')^{-1} \right) G(s) ds.\label{eqD++commutator1}  \ee
	Then let $I_j = \int_r^\infty (D_{\pm l} \circ (\pm il\phi')^{-1})^jG ds$ for $0 \le j \le n$, which are well-defined by the integrability assumption \eqref{equpupasas} and \eqref{eqlemb22}. Since the recursive relation \eqref{eqalgrec} holds for $\{ I_j\}_{0 \le j \le n}$, so we have
	\be D_{\pm l}^n \int_r^\infty G ds = \sum_{k =1}^n f_{n,k;\pm l}(r) \int_r^\infty  (D_{\pm l} \circ(\pm il\phi')^{-1})^k G ds. \label{eqlemb21} \ee
    Plugging this in \eqref{eqlemb22} yields \eqref{eqD++commutator}.
\end{proof}

 % For notational simplicity, in the applications, we will omit the $\pm l$ part of subscript of $f_{n,k}$, $g_{n, k}$ and $I_k$ when no ambiguity occurs. 

\section{Asymptotics of $Q_b$} \label{appQb}

In this section, we prove those two propositions for asymptotics of $Q_b$ Proposition \ref{propQbasymp} and Proposition \ref{propQbasympref}. The proof of the former is based on \cite{MR4250747}, while the latter applies the WKB approximate fundamental solutions constructed in Subsection \ref{sec41}.

\mbox{}

\begin{proof}[Proof of Proposition \ref{propQbasymp}] The existence of $b=b(d, s_c)$ and $Q_b$ is contained in the statement of \cite[Theorem 1]{MR4250747}, and the asymptotics of $s_c$ follows \cite[(1.10)-(1.11)]{MR4250747}. For the estimates, \eqref{eqQbasymp1}-\eqref{eqQbasymp6} can be derived from the construction process in \cite{MR4250747} although not explicitly stated, while the high regularity estimates \eqref{eqQbasymp7}-\eqref{eqQbasympint2} follow from an application of the profile equation \eqref{eqselfsimilar}. We now prove them separately. 

\mbox{}

\underline{1. Proof of \eqref{eqQbasymp1}-\eqref{eqQbasymp6}.} 

The $r \ge b^{-2}$ part asymptotics of \eqref{eqQbasymp1} is explicitly included in \cite[Proposition 2.1]{MR4250747}. 

The $r \le b^{-\frac 12}$ asymptotics \eqref{eqQbasymp5}-\eqref{eqQbasymp6} come from \cite[Proposition 4.1]{MR4250747}. More specifically, we record equation \cite[(4.14)]{MR4250747} 
\[ P_{\rm int} = (Q + \gamma A + \phi_+) + i (b\sigma B + \phi_-). \]
Then \eqref{eqQbasymp5}-\eqref{eqQbasymp6} follow from $Q_b = e^{-i\frac{br^2}{4}} P_{\rm int}$ in \cite[Theorem 2]{MR4250747}, $A$ in \cite[Lemma 4.1]{MR4250747}, $B$ in \cite[Lemma 4.2]{MR4250747}, $\gamma$ in the range \cite[(1.14)]{MR4250747}, $\sigma = s_c$, and the residual estimates $\|(\phi_+, \phi_-) \|_K \le 1, \|(\phi_+', \phi_-')\|_K \le 1$ shown at the end of the proof of \cite[Proposition 4.1]{MR4250747} where the $K$-norm is defined as 
\[ \|(\phi_+, \phi_-) \|_K := \max \left\{ b^{-\frac 13}\| \phi_+/Q \|_{L^\infty([0, b^{-\frac 12}])}; b^{-\frac 54}\sigma^{-1} \| \phi_- \cdot (1+r)^{d-1} Q \|_{L^\infty([0, b^{-\frac 12}])} \right\}. \]

For asymptotics on $r \in [b^{-\frac 12}, b^{-2}]$ \eqref{eqQbasymp2}, recall the change of variables in \cite[Proposition 3.1 Step 1]{MR4250747} and \cite[Theorem 2]{MR4250747}
\bee
\left| \begin{array}{l}
     Q_b(r) = e^{-i\frac{br^2}{4}}e^{i\theta}P_{ext}(r),\\
     P_{ext}(r) = e^{i\theta_{ext}}  r^{-\frac{d-1}{2}} U(r),\\
  W(\tau) = U(r),\quad \tau = 2-br,\\
  X(\zeta(\tau)) = W(\tau) \sqrt{\zeta'(\tau)}, \\
  Y(s) = X(\zeta),\quad \zeta = b^{\frac 23}s,
\end{array} \right. \quad \Rightarrow \quad  Q_b(r) = e^{-i\frac{br^2}{4}} e^{i(\theta_{ext} + \theta)} r^{-\frac{d-1}{2}} \frac{ Y\left(b^{-\frac 23}\zeta(2-br)\right)}{\zeta'(2-br)^{\frac 12}}
\eee
where $\theta$ is from \cite[Theorem 2]{MR4250747} satisfying the estimate \cite[(1.15)]{MR4250747} 
\[ |\theta| \le b^{\frac 16} e^{-\frac{\pi}{b}} e^{\frac{2}{\sqrt b}}, \] 
and 
$\zeta$ is given by equation  \cite[(3.13)]{MR4250747}
\be \zeta(\tau) = \left| \begin{array}{ll}
    \left( \frac 32 \int_0^\tau \frac 12 \sqrt{\tau'(2-\tau')}  d\tau' \right)^\frac 23  & \tau \in [0, 2] \\
    - \left( \frac 32 \int_\tau^0 \frac 12 \sqrt{(-\tau')(2-\tau')}  d\tau' \right)^\frac 23 & \tau \in (-\infty, 0]
\end{array}\right.. \label{eqZetadef} \ee
In particular,
\[ \zeta(\tau) \sim \left| \begin{array}{ll}
   \tau  & \tau\in [-1, 2], \\
   -|\tau|^{\frac 43}  & \tau \le -1;
\end{array}\right.\qquad \zeta'(\tau) \sim \la \tau \ra^{\frac 13},\quad \tau \le 2. \]
Also \cite[Lemma 3.2]{MR4250747} implies 
\[ Y(s) = \beta_I \mathbf{B}(s) \left(1 + O(b^\frac 14) \right),\quad Y'(s) = \beta_I \mathbf{B}'(s)\left(1 + O(b^\frac 14) \right), \]
% \[\quad Y'(s) = e^{i\arg \beta_I} \frac{1}{\sqrt{2\pi}} \rho b^\frac 13 \mathbf{B}'(s) \left(1 + O(b^\frac 14) \right) \]
with 
\[ |\beta_I| = \frac{1}{\sqrt{2\pi}} \rho b^\frac 13 \left(1 + O(b|\ln b|) \right),\quad \rho \sim b^{-\frac 12} e^{-\frac{\pi}{2b}}, \] 
and
\[ |\mathbf{B}(s)| \sim \la s \ra^{-\frac 14} w(s),\quad  |\mathbf{B}'(s)| \lesssim \la s \ra^{\frac 14} w(s),
\]
where $w(s) := e^{\frac 23 \max\{ s, 0 \}^\frac 32}$ for $ s\in \RR$.
Also note that $s_c \ll (\ln b^{-1})^{-1}$ so $r^{s_c} \sim 1$ on this region $r \in [b^{-\frac 12}, b^{-2}]$. We now derive asymptotics of $Q_b$ in different subintervals.

(1) When $r \in [\frac 4b, b^{-2}]$, we have $2-br \le -2$, then $s \sim b^{-\frac 23} |2-br|^{\frac 43}$ and $\zeta'(2-br) \sim |2-br|^\frac 13$. So 
\[ |Q_b| \sim \rho b^\frac 13 r^{-\frac{d-1}{2}} \left| b^{-\frac 23}|2-br|^\frac 43\right|^{-\frac 14} \cdot (|2-br|^\frac 13)^{-\frac 12} \sim r^{-\frac d2}b^{-\frac 12} e^{-\frac{\pi}{2b}}.  \]
(2) When $r \in \left[\frac 2b + b^{-\frac 13}, \frac 4b \right]$, we have $2 - br \in [-2, -b^{\frac 23}]$ and thus $s  \sim b^{-\frac 23}(2-br) \le -1$, $\zeta'(2-br) \sim 1$. Hence 
\[ |Q_b| \sim \rho b^\frac 13 r^{-\frac{d-1}{2}} \left| b^{-\frac 23}(2-br)\right|^{-\frac 14} \sim r^{-\frac{d-1}2}e^{-\frac{\pi}{2b}}|2-br|^{-\frac 14}. 
% \sim r^{-\frac d2+s_c} b^{-\frac 12} e^{-\frac{\pi}{2b}}.
\]
(3) When $\left| r-\frac 2b\right| \le b^{-\frac 13}$, we see $|s| \lesssim 1$, $\zeta'(2-br) \sim 1$. Then 
\[ |Q_b| \sim \rho b^\frac 13 r^{-\frac{d-1}{2}} \sim r^{-\frac{d-1}{2}} b^{-\frac 16}e^{-\frac{\pi}{2b}}. \]
(4) When $r \in \left[b^{-\frac 12}, \frac 2b - b^{-\frac 13} \right]$, similar to case (2), $s \sim b^{-\frac 23} (2-br) \ge 1$ and $\zeta'(2-br) \sim 1$. Note that 
\[ \frac 23 \left( b^{-\frac 23} \zeta\right)^\frac 32 = b^{-1} \int_0^{2-br} \frac 12\sqrt{\tau'(4-\tau')} d\tau' = \int_r^{\frac 2b} \left(1-\frac{b^2s^2}{4}\right)^\frac 12 ds = S_b(r) \]
and that $\int_0^\frac 2b \left(1-\frac{b^2s^2}{4}\right)^\frac 12 ds = \frac{\pi}{2b}$, we obtain
\[ |Q_b| \sim \rho b^\frac 13 r^{-\frac{d-1}{2}} e^{\frac 23\left( b^{-\frac 32} \zeta\right)^\frac 32} \left| b^{-\frac 23}(2-br)\right|^{-\frac 14}  \sim r^{-\frac{d-1}{2}} e^{-\frac \pi{2b} + S_b(r)}  |2-br|^{-\frac 14}  \]
Compared with \eqref{eqQbasymp2} concludes the proof for $k=0$ case. The $k=1$ case follows similarly, where the extra $\la br \ra$ factor arises from $\pa_r(e^{\frac{ibr^2}{4}})$ or $|\pa_r Y(s)| \lesssim |Y(s)| \cdot \la s \ra^\frac12 |\pa_r s|$. 

\mbox{}

\underline{2. Proof of \eqref{eqQbasymp7}-\eqref{eqQbasympint2}.}    

These estimates follow from improving regularity of their low regularity counterparts \eqref{eqQbasymp1}-\eqref{eqQbasymp6} through the profile equation \eqref{eqselfsimilar}. More specifically, we will use elliptic regularity for the bounds near the origin and use the ODE away from the origin. In particular, the non-vanishing of $Q_b$ from \eqref{eqQbasymp1}-\eqref{eqQbasymp2} is crucial for estimating the nonlinear term. For simplicity, we only prove the most complicated case \eqref{eqQbasympint1}. The other two estimates \eqref{eqQbasymp7} and \eqref{eqQbasympint2} will be similar and simpler. 

From \eqref{eqselfsimilar}, $P_b = e^{i\frac{br^2}{4}}$ satisfies 
\be \left(\Delta + \frac{b^2 r^2}4 - 1 - ibs_c \right) P_b + |P_b|^{p-1} P_b = 0. \ee
Subtracting the equation for mass-critical ground state \eqref{eqgroundstatemasscritical}, we obtain 
\be (1 - \Delta)(P_b - Q) =\left(\frac{b^2 r^2}{4} -ibs_c\right) P_b + \left( |P_b|^{p-1} P_b - Q^{p_0} \right) \label{eqPbQdiff}
\ee
where $p_0 = \frac 4d + 2$.
We will show for any $n \ge 0$, with $R_n := 1+2^{-n}$,
\bea
    \left| \pa_r^{n}(P_b - Q)\right| &\lesssim_n& b^{\frac 13} r^{-\frac{d-1}{2}} e^{-r}\quad {\rm for }\,\, r \in [1, b^{-\frac 13}];\label{eqPbQdiffest1} \\
    \| P_b - Q \|_{H^{2n}(B^{\RR^d}_{R_n})} &\lesssim_n &b^\frac 13.\label{eqPbQdiffest2} 
\eea
Clearly they can imply \eqref{eqQbasympint1}. 

Firstly, we claim the following nonlinear estimates for $n \ge 0$: for $r \le b^{-\frac 13}$, 
\begin{align}
\begin{split}
 &\left|\pa_r^n \left(|P_b|^{p-1} P_b - |P_b|^{p_0 - 1} P_b \right)\right| \\
 \lesssim_n& (p - p_0) |\ln P_b| \sum_{m = 0}^n  |P_b|^{p_0 - 1 - 2m}  \sum_{\substack{\gamma \ge 0, \a \in \ZZ_{\ge 0}^{2m}\\ \gamma + \sum_i \a_i = n}} |\pa_r^\gamma P_b| \cdot \prod_{i=1}^{2m} |\pa_r^{\a_i} P_b|
 \end{split} \label{eqPbnlpoint1} \\
 \begin{split}
 &\left|\pa_r^n \left(|P_b|^{p_0-1} P_b - Q^{p_0} \right)\right| \\
 \lesssim_n&  \sum_{m = 0}^n  |P_b|^{p_0 - 3 - 2m}  \sum_{\substack{\gamma \ge 0, \a \in \ZZ_{\ge 0}^{2m}\\ \gamma + \sum_i \a_i = n}}  |\pa_r^\gamma (P_b - Q)| \cdot \prod_{i=1}^{2m}\left( |\pa_r^{\a_i} P_b| + |\pa_r^{\a_i} Q|\right) 
 \end{split} \label{eqPbnlpoint2}
\end{align}
and 
\bea
\begin{split}
 &\left\||P_b|^{p-1} P_b - |P_b|^{p_0 - 1} P_b \right\|_{\calK_n}\\
 \lesssim_n& (p-p_0) \| \ln P_b \|_{L^\infty(B_{R_n})} \| P_b \|_{L^\infty(B_{R_n})}^{p_0 - 1} \sum_{m=0}^n  \| P_b^{-1} \|_{L^\infty(B_{R_n})}^{2m}\cdot \| P_b \|_{\calK_n}^{2m+1} 
 \end{split} \label{eqPbnlpoint3}\\
 \begin{split}
 &\left\| |P_b|^{p_0-1} P_b -Q^{p_0}  \right\|_{\calK_n}\\
 \lesssim_n&  \| P_b \|_{L^\infty(B_{R_n})}^{p_0 - 3} \sum_{m=0}^n  \| P_b^{-1} \|_{L^\infty(B_{R_n})}^{2m}\cdot \left( \| P_b \|_{\calK_n} + \| Q \|_{\calK_n}\right)^{2m} \| P_b - Q\|_{\calK_n} 
 \end{split} \label{eqPbnlpoint4}
\eea
where $\| f \|_{\calK_n} := \| f \|_{(L^\infty \cap H^n)(B_{R_n})}$. 

Indeed, we first compute the difference as 
\bee
  |P_b|^{p-1} P_b - |P_b|^{p_0 - 1} P_b = \int_{p_0 - 1}^{p-1} \ln (|P_b|) |P_b|^{\tau} P_b d\tau = \int_{p_0 - 1}^{p-1} f_\tau (|P_b|^2) P_b d\tau
\eee
where $f(a) := \frac 12 (\ln a) \cdot a^{\frac \tau 2}$. Then \eqref{eqPbnlpoint1} follow from the Fa\'a di Bruno's formula \eqref{eqFaadiBruno}, Leibniz rule and that $|P_b(r)|^{\tau - p_0} \sim 1$ for $r \le b^{-\frac 13}$ and $0 < \tau - p_0 \le p - p_0 \sim s_c$ (from the asymptotics \eqref{eqbasymp} and \eqref{eqQbasymp5}). Similarly, we compute 
\bee
 &&|P_b|^{p_0-1} P_b -Q^{p_0}  = (|P_b|^{p_0-1}- Q^{p_0 - 1}) P_b + Q^{p_0 - 1} (P_b - Q)  \\
 &=& \frac{p_0 - 1}{2} \int_{0}^1 (\theta |P_b|^2 + (1-\theta) Q^2)^{\frac{p_0 - 3}{2}} d\theta\cdot (|P_b|^2 - Q^2)P_b + Q^{p_0-1} (P_b - Q).
\eee
Noticing that 
\[ |P_b|^2 - Q^2 = \Re (P_b - Q) \cdot \Re (P_b + Q) + \Im (P_b - Q) \cdot \Im (P_b + Q), \]
the pointwise estimate \eqref{eqPbnlpoint1} comes from the Fa\'a di Bruno's formula \eqref{eqFaadiBruno}, Leibniz rule, and that $|Q| \sim |P_b| \sim \la r \ra^{-\frac{d-1}{2}}e^{-r}$ for $r \le b^{-\frac 13}$ (using \eqref{eqQbasymp5} and \eqref{eqsolitondecay}). Since $\calK_n$ is a Banach algebra from an application of Gagliardo-Nirenberg interpolation (see for example \cite[Comments on Chapter 9, 3.C, P.313]{zbMATH05633610}), the pointwise bounds \eqref{eqPbnlpoint1}-\eqref{eqPbnlpoint2} easily implies the norm estimates \eqref{eqPbnlpoint3}-\eqref{eqPbnlpoint4}.

\mbox{}

Finally we are in place to prove \eqref{eqPbQdiffest1} and \eqref{eqPbQdiffest2}. 

\textit{Proof of  \eqref{eqPbQdiffest1}.} We differentiate \eqref{eqPbQdiff} to obtain for $k \ge 0$
 \be
\pa_r^{2+k} (P_b - Q) = \pa_r^{k} \circ \left( 1 - \frac{d-1}{r} \pa_r \right)(P_b - Q) - \pa_r^{k} \left[ \left(\frac{b^2 r^2}{4} -ibs_c\right) P_b + \left( |P_b|^{p-1} P_b - Q^{p_0} \right) \right]. \label{eqbanonbao}
\ee
Then \eqref{eqPbQdiffest1} will follow from an induction on $n$, with $n=0, 1$ proven in \eqref{eqQbasymp5}. 
The only non-trivial bound for this induction is the nonlinearity: Supposing  \eqref{eqPbQdiffest1} holds for $0 \le n \le N$, then from \eqref{eqPbnlpoint1}-\eqref{eqPbnlpoint2} and \eqref{eqQbasymp5}, we have 
\[ \left| \pa_r^N \left( |P_b|^{p-1} P_b - Q^{p_0} \right)\right| \lesssim_N (s_c \cdot |\ln b|  b^{-\frac 13} +  b^\frac 13)\left(r^{-\frac{d-1}{2}} e^{-r}\right)^{p_0} \lesssim b^\frac 13  r^{-\frac{d-1}{2}} e^{-r} \]
for $r \in [ 1, b^{-\frac 13}]$. That concludes the induction step and hence \eqref{eqPbQdiffest1} for all $n \ge 0$. 

\textit{Proof of  \eqref{eqPbQdiffest2}.} From a standard interior elliptic regularity estimate, \eqref{eqPbQdiff} implies 
\bee
  \| P_b - Q \|_{H^{2(n+1)}(B_{R_{n+1}})} \lesssim_n  \| P_b - Q \|_{H^{2n}(B_{R_{n}})} + \left\| \left(\frac{b^2 r^2}{4} -ibs_c\right) P_b + \left( |P_b|^{p-1} P_b - Q^{p_0} \right) \right\|_{H^{2n}(B_{R_{n}})}. 
\eee
So  \eqref{eqPbQdiffest2} follows from again an induction on $n \ge 0$ with $n = 0, 1$ verified through \eqref{eqQbasymp5} and the nonlinearity of induction step is bounded using \eqref{eqPbnlpoint3}-\eqref{eqPbnlpoint4}. 

\end{proof}

\mbox{}

\begin{proof}[Proof of Proposition \ref{propQbasympref}]
Let $Q_b = e^{-i\frac{br^2}{4}}r^{-\frac{d-1}{2}} U(r)$, then $U$ satisfies
    \[ U'' + \left( \frac{b^2 r^2}{4} - 1 - \frac{(d-1)(d-3)}{4r^2} - ibs_c \right)U + r^{-\frac 12 (d-1)(p-1)}|U|^{p-1}U = 0, \]
    namely 
    \be 
 \tilde H_{b, E} U = P_{b, E} U + N(U), \label{eqQbnonlinear}
    \ee
    with $E = 1 + ibs_c$, $\tilde H_{b, E}$ from \eqref{eqdeftildeHbE} and 
        \[ P_{b, E} = -h_{b, E} + \frac{(d-1)(d-3)}{4r^2},\quad N(U) = -r^{-\frac 12 (d-1)(p-1)}|U|^{p-1}U. \]

We recall that the construction of $Q_b$ in \cite[Theorem 2]{MR4250747} is based on matching the interior solutions on $[0, b^{-\frac 12}]$ from Proposition \cite[Proposition 4.1]{MR4250747} with the exterior solutions on $[b^{-\frac 12}, \infty)$ satisfying
\cite[Proposition 2.1, Proposition 3.1]{MR4250747}. Our proof here will reconstruct the exterior solution on $[b^{-\frac 12}, \infty)$ using the approximate WKB solutions from Proposition \ref{propWKB} so as to obtain better estimates.

More specifically, we will construct a family of exterior solutions $U_\varrho$ for every $\varrho \sim b^{-\frac 12} e^{-\frac \pi{2b}}$ satisfying the exterior decay
 \be \left| \begin{array}{l}
        |U_\varrho| \sim b^{-\frac 16} e^{-\frac{\pi}{2b}} r^{-\frac 12 + s_c}\\
         \left|\left(\pa_r - \frac{ibr}{2} \right) U_\varrho \right| \lesssim (br)^{-1} |U(r)|  
     \end{array}\right. \quad r \ge b^{-2}, \label{eqUasymp}\ee
and boundary value at $r_* = b^{-\frac 12}$ 
\be
\left| \begin{array}{l}
  \Re U_\varrho (r_*) = \frac{\varrho b^\frac 12 e^{\frac\pi {2b}}}{2^{1 + \frac{1}{6}} \sqrt \pi } e^{-r_*} \left(1 + O(r_*^{-1}) \right) \\
  \Re U_\varrho' (r_*) = -\frac{\varrho b^\frac 12 e^{\frac\pi {2b}}}{2^{1 + \frac{1}{6}} \sqrt \pi} e^{-r_*} \left(1 + O(r_*^{-1}) \right) \\
  \Im U_\varrho (r_*) = \frac{\varrho b^\frac 12 e^{-\frac\pi {2b}}}{2^{2 + \frac{1}{6}} \sqrt \pi} e^{r_*} \left(1 + O(r_*^{-1}) \right)\\
  \Im U_\varrho' (r_*) = \frac{\varrho b^\frac 12 e^{-\frac\pi {2b}}}{2^{2 + \frac{1}{6}} \sqrt \pi} e^{r_*} \left(1 + O(r_*^{-1}) \right).
  \end{array}\right. \label{eqUbdryr*}
\ee
% {\color{red} Should we change the construction to obtain sharp asymptotics up to $10 |\log b|$? }Notice that the interior solution constructed in \cite[Proposition 4.1]{MR4250747} can be evaluated at $r_*$ in the similar manner as at $r = b^{-\frac 12}$ in \cite[(4.9)-(4.13)]{MR4250747}: 
% \bee
% \left| \begin{array}{l}
%   \Re P_{\rm int}(r_*)  = \kappa_Q r_*^{-\frac{d-1}{2}} e^{-r_*} \left(1 + O(r_*^{-1}) \right) + \kappa_A \gamma r_*^{-\frac{d-1}{2}} e^{r_*} \left(1 + O(r_*^{-1}) \right)  \\
%   \Re P_{\rm int}' (r_*) = -\frac{\varrho b^\frac 12 e^{\frac\pi {2b}}}{2^{1 + \frac{1}{6}} \sqrt \pi} r_*^{-\frac{d-1}{2}} e^{-r_*} \left(1 + O(r_*^{-1}) \right) + \kappa_A \gamma r_*^{-\frac{d-1}{2}} e^{r_*} \left(1 + O(r_*^{-1}) \right)  \\
%   \Im P_{\rm int} (r_*) = \kappa_B bs_c  \frac{\varrho b^\frac 12 e^{-\frac\pi {2b}}}{2^{2 + \frac{1}{6}} \sqrt \pi} r_*^{-\frac{d-1}{2}} e^{r_*} \left(1 + O(r_*^{-1}) \right)\\
%   \Im P_{\rm int}' (r_*) = \frac{\varrho b^\frac 12 e^{-\frac\pi {2b}}}{2^{2 + \frac{1}{6}} \sqrt \pi} r_*^{-\frac{d-1}{2}} e^{r_*} \left(1 + O(r_*^{-1}) \right).
%   \end{array}\right. 
% \eee
% with the error size $O(r_*^{-1})$. 
Then the matching argument at $r_*$ as in \cite[Theorem 2]{MR4250747} leads to $\varrho_b$ satisfying \eqref{eqvarrhob} and 
\be Q_b = e^{i\tilde \theta_b} e^{-i\frac{br^2}{4}} r^{-\frac{d-1}2} U_{\varrho_b},\qquad {\rm with}\quad   |\tilde \theta_b| \lesssim b^\frac 14 e^{-\frac \pi b + 2b^{-\frac 12}}. \label{eqQbandU} \ee
% We remark that the smallness of $\tilde \theta_b$ follows from the better residual estimates $O(r_*^{-1})$ in \eqref{eqUbdryr*} compared with \cite[Theorem 2]{MR4250747}. 

Therefore, the proof is reduced to constructing $U_\varrho$ satisfying \eqref{eqUasymp}-\eqref{eqUbdryr*}, and proving estimates which implies \eqref{eqQbasymp8}-\eqref{eqUabs2} and \eqref{eqQbsharp}.

\mbox{}

\underline{1. Exterior solution.} Let $r_0 = \frac 2b$ and $I_{ext} = [r_0, \infty)$. With the fundamental solutions $\psi_1^{b, E}$, $\psi_3^{b, E}$ of $\tilde H_{b, E}$, we define the exterior solution on $I_{ext}$ via the Duhamel formula
\be
U = c_3 \psi_3^{b, E} + \psi_3^{b, E} \int_r^\infty \psi_1^{b, E}  (PU + N(U))\frac{dr}{W_{31;E}} - \psi_1^{b, E} \int_r^\infty \psi_3^{b, E}  (PU + N(U))\frac{dr}{W_{31;E}},\label{eqQbDuhamel1}
\ee
Take $c_3 = \varrho b^\frac 13 e^{\frac{\pi i}{6} + i \theta_1}$ with any $\varrho \sim b^{-\frac 12} e^{-\frac \pi{2b}}$ and any $|\theta_1| \le \pi$. We claim that \eqref{eqQbDuhamel1} is a contraction mapping in a ball of radius $2|c_3|\sim b^{-\frac 16} e^{-\frac \pi{2b}}$ in $C^0_{\omega_{b, E}^-}(I_{ext})$. Indeed, recall that 
$$\omega_{b, E}^\pm \sim \left\la b^{-\frac 23}(b^2 r^2 - 4E) \right \ra^{-\frac 14} e^{\pm\Re \eta_{b, E}} \lesssim \min \left\{b^\frac 16 (2-br)^{-\frac 14} (br)^{-\frac 14 \mp s_c}, (br)^{-\frac 12 \mp s_c} \right\}.$$
So we compute
\begin{align}
&\left| \int_{r}^\infty \psi_j^{b, E} (P_{b, E}U + N(U))\frac{ds}{W_{31;E}}\right| \nonumber\\
\lesssim& \int_{r}^\infty s^{-2} (bs)^{-\frac 12 + \mu_j s_c} (2-bs)^{-\frac 12} \| U \|_{C^0_{\omega_{b, E}^-}(I_{ext})}  \left( 1 + b^{-2} \| U \|_{C^0_{\omega_{b, E}^-}(I_{ext})}^{p-1}  \right) ds \nonumber \\
\sim& b (br)^{-2+\mu_j s_c}  \| U \|_{C^0_{\omega_{b, E}^-}(I_{ext})}
\label{eqQbintbdd}\end{align}
which implies the onto estimate with smallness $O(b)$. The contraction estimate can be derived similarly with contraction rate $O(b)$. In particular, the Duhamel form indicates 
\be
 \vec U(r_0) = \tilde c_3 \vec \psi_1^{b, E}(r_0) + \beta \vec \psi_3^{b, E}(r_0) \label{eqUbdryr0}
\ee
where $\vec f = (f, \pa_r f)^\top$ and the coefficients being
\be
\begin{split}
\tilde c_3 &= c_3 - \int_{r_0}^\infty \psi_3^{b, E}(P_{b, E} U + N(u)) ds = c_3 + O_\CC(b^{1-\frac 16} e^{-\frac \pi{2b}}), \\
\beta &= \int_{r_0}^\infty \psi_1^{b, E} (P_{b, E} U + N(U)) \frac{dr}{W_{31;E}} = O_\CC(b^{1-\frac 16}e^{-\frac \pi{2b}})
\end{split} \label{eqdeftildec3beta}
\ee
using the estimate \eqref{eqQbintbdd}. Thanks to the phase rotation symmetry of $U$, we can take $U_\varrho$ with $c_3=  \varrho b^\frac 13 e^{\frac{\pi i}{6} + i \theta_1(\varrho)}$ with a fixed $\theta_1(\varrho) = O(b)$ such that 
\be 
  c_3^* := e^{-\frac{\pi i}{6}} \tilde c_3 + e^{\frac{\pi i}{6}} \beta = \varrho b^\frac 13 (1 + O_\RR(b)) \in \RR. \label{eqRRtildec3}
\ee
We now verify \eqref{eqUasymp}. The contraction directly implies the pointwise bound of $|U_\varrho|$, and the derivative bound in \eqref{eqUasymp} follows  the derivative estimates of $\psi_3^{b, E}$ \eqref{eqpsibderiv2} and the extra $(br)^{-2}$ decay of the Duhamel terms \eqref{eqQbintbdd}. 
% Combining \eqref{eqUasymp} with the estimates of $\psi_1^{b, E}$, $\psi_3^{b, E}$, $h_{b, E}$, $W_{31;E} = \calW(\psi_3^{b, E}, \psi_1^{b, E})$ from Proposition \ref{propWKB}, we have for $r \ge b^{-2}$ 
% \be \left| \int_{r}^\infty \psi_j^{b, E} (PU + N(U))\frac{ds}{W_{31;E}}\right| \lesssim b^{-\frac 16} e^{-\frac{\pi}{2b}} \int_{r}^\infty s^{-3 + \mu_j s_c} (2-bs)^{-\frac 12}  ds \sim r^{-2+\mu_j s_c}  b^{-\frac 16} e^{-\frac{\pi}{2b}},\label{eqQbintbdd}\ee
%     for $j = 1, 3$ with $\mu_1 = 2, \mu_2 = 0$. The convergence of these integral imply the following Duhamel formulation of $U$ via variation of constants
%   \bea
%      U = \sum_{j \in \{ 1, 3\}} c_j \psi_j^{b, E} + \psi_3^{b, E} \int_r^\infty \psi_1^{b, E}  (PU + N(U))\frac{dr}{W_{31;E}} - \psi_1^{b, E} \int_r^\infty \psi_3^{b, E}  (PU + N(U))\frac{dr}{W_{31;E}},\label{eqQbDuhamel1}\\
%      U' = \sum_{j \in \{ 1, 3\}} c_j (\psi_j^{b, E})' + (\psi_3^{b, E})' \int_r^\infty \psi_1^{b, E}  (PU + N(U))\frac{dr}{W_{31;E}} - (\psi_1^{b, E})' \int_r^\infty \psi_3^{b, E}  (PU + N(U))\frac{dr}{W_{31;E}},\label{eqQbduhamel2}
%     \eea
%     with some constant $c_1, c_3$. Moreover, using \eqref{eqpsibderiv2} and \eqref{eqQbintbdd}, we obtain $\left|\left(\pa_r - \frac{ibr}{2} \right) U \right| \sim c_1 b^{\frac 76}r^{\frac 12 -s_c} + O(r^{-\frac 32 + s_c})$ for $r \ge b^{-2}$, so \eqref{eqUasymp} forces $c_1 = 0$; then $|U| \sim c_3 b^\frac 13 r^{-\frac 32 + s_c} + O(r^{-\frac 52 + s_c})$ as $r \ge b^{-2}$ and therefore  $|c_3| \sim b^{-\frac 16} e^{-\frac \pi{2b}}$ by \eqref{eqUasymp}.

\mbox{}

\underline{2. Intermediate solution.} Let $x_* = r_* = b^{-\frac 12}$ and $I_{mid} = [x_*, r_0]$. Now we extend $U$ to $I_{mid}$ by solving \eqref{eqQbnonlinear} with prescribed boundary condition on $r_0 = \frac 2b$. Rewrite \eqref{eqQbnonlinear} with real and imaginary parts decoupled as
\be \left| \begin{array}{l}
     \tilde H_{b, 1} \Sigma  - P_{b, 1} \Sigma + bs_c \Theta - N_1(\Sigma + i\Theta) = 0  \\
     \tilde H_{b, 1} \Theta  - P_{b, 1} \Theta - bs_c \Sigma - N_2(\Sigma + i\Theta) = 0
\end{array}\right.
 \label{eqQbDuhamel4}
\ee
where we exploited $\tilde H_{b, E} - P_{b, E} = \tilde H_{b, 1} - P_{b, 1} - ibs_c$ and denote $U = \Sigma + i\Theta$, $N(U) = N_1(U) + iN_2(U)$.
Then we apply the inversion in Lemma \ref{leminvtildeHmid} to obtain
\be
\left| \begin{array}{l}
\Sigma = \sum_{j \in\{ 2, 4 \}} d_j^\Re \psi_j^{b, 1} + \tilde T^{mid, G}_{x_*, r_0;b, 1} \left( P_{b, 1} \Sigma - bs_c \Theta + N_1(\Sigma + i\Theta) \right) \\
\Theta = \sum_{j \in\{ 2, 4 \}} d_j^\Im \psi_j^{b, 1} + \tilde T^{mid, D}_{x_*, r_0;b, 1} \left( P_{b, 1} \Theta + bs_c \Sigma + N_2(\Sigma + i\Theta) \right).
\end{array}\right. \label{eqUDuhamelmid1}
\ee
For boundary condition at $r_0 = \frac 2b$, we choose according to \eqref{eqUbdryr0} that
\bee
  \sum_{j \in \{ 2, 4 \} } (d_j^\Re + id_j^\Im) \vec \psi_j^{b, 1}(r_0) + \a(\Sigma, \Theta) \vec \psi_4^{b, 1}(r_0) 
  = \tilde c_3 \vec \psi_1^{b, E}(r_0) + \beta \vec \psi_3^{b, E}(r_0)  
\eee
where $\tilde c_3$, $\beta$ from \eqref{eqdeftildec3beta}, and the nonlinear $\RR$-valued functional $a = a(\Sigma, \Theta)$ is
\bee
\a (\Sigma, \Theta) = \int_{x_*}^{r_0} \psi_2^{b, 1} \left( P_{b, 1} \Sigma - bs_c \Theta + N_1(\Sigma + i\Theta) \right) \frac{dr}{W_{42;1}},
\eee
% Similar to the Step 2(3) of the proof of Lemma \ref{lemrhob}, 
Next, we use  \eqref{eqpsibEderiv1}-\eqref{eqpsibEderiv2} and connection formula \eqref{eqconnect} decompose $\vec \psi_1^{b, E}(r_0)$, $\vec \psi_3^{b, E}(r_0)$ with respect to the $\CC^2$-basis $\vec \psi_2^{b, 1}(r_0)$, $\vec \psi_4^{b, 1}(r_0)$,
\[ \left( \begin{array}{c}
\vec \psi_1^{b, E}(r_0) \\ \vec \psi_3^{b, E}(r_0)
\end{array}\right) =  \left[ \left( \begin{array}{cccc}
\frac{e^{\frac{\pi i}{3}}}{2} & e^{-\frac{\pi i}{6}} \\ \frac{e^{-\frac{\pi i}{3}}}{2} & e^{\frac{\pi i}{6}}
\end{array}\right) +\calR_b \right]
\left( \begin{array}{c}
\vec \psi_2^{b, 1}(r_0) \\ \vec \psi_4^{b, 1}(r_0)
\end{array}\right), \quad \calR_b = O_{\CC^{2\times 2}}(b^{-\frac 13} s_c) 
\]
So we parametrize the boundary condition as
\bea
\left( \begin{array}{c}
d_2^\Re + i d_2^\Im \\ d_4^\Re + i d_4^\Im
\end{array}\right) 
 = 
 \left( \begin{array}{c}
\bar d_2^\Re + i \bar d_2^\Im \\ \bar d_4^\Re + i \bar d_4^\Im
\end{array}\right) -
\left( \begin{array}{c}
0 \\ \a(\Sigma, \Theta) 
\end{array}\right)
\label{eqUDuhamelmid2}
\eea
where the $(\Sigma, \Theta)$-independent part is
\be
  \left( \begin{array}{c}
\bar d_2^\Re + i \bar d_2^\Im \\ \bar d_4^\Re + i \bar d_4^\Im
\end{array}\right) =  \left[ \left( \begin{array}{cccc}
\frac{e^{\frac{\pi i}{3}}}{2} & \frac{e^{-\frac{\pi i}{3}}}{2} \\ e^{-\frac{\pi i}{6}}  & e^{\frac{\pi i}{6}}
\end{array}\right) + \calR_b \right]
 \left( \begin{array}{c}
 \tilde c_3  \\ \beta
\end{array}\right) = \left( \begin{array}{c}
  \frac 12 \varrho b^\frac 13 (i + O_\CC(b))  \\ c_3^* (1 + O_{\CC} (b^{-\frac 13} s_c)) 
\end{array}\right) \label{eqdefbardj}
\ee
using  \eqref{eqRRtildec3}. 

Plugging \eqref{eqUDuhamelmid2} into  \eqref{eqUDuhamelmid1}, we claim that \eqref{eqUDuhamelmid1} defines a contraction mapping for $(\Sigma, \Theta)$ in a ball of radius $ \sim b^{-\frac 16}e^{-\frac \pi{2b}}$ in $X_b:=C^0_{\omega_{b, 1}^-}(I_{mid}) \times C^0_{\omega_{b, 1}^+}(I_{mid})$. Indeed, first recall that 
\[|\psi_2^{b, 1}| \lesssim \omega_{b, 1}^+,\quad |\psi_4^{b, 1}| \lesssim \omega_{b, 1}^-,\quad \omega_{b, 1}^+ \le \omega_{b, 1}^- \le 2 e^{\frac \pi b - 2b^{-\frac 12}} \omega_{b, 1}^+,\quad {\rm for}\,\, r \in I_{mid}.  \]
So the fixed source terms satisfy 
\bee
  \|\sum_{j \in \{ 2, 4 \}} \bar d_j^\Re \psi_j^{b, 1} \|_{C^0_{\omega_{b, 1}^-}(I_{mid})} &\lesssim& |\bar d_4^\Re| + |\bar d_2^\Re| = \varrho b^\frac 13(1 + O(b)) \lesssim b^{-\frac 16}e^{-\frac \pi{2b}}, \\
  \|\sum_{j \in \{ 2, 4 \}} \bar d_j^\Im \psi_j^{b, 1} \|_{C^0_{\omega_{b, 1}^+}(I_{mid})} &\lesssim&  e^{\frac \pi b - 2b^{-\frac 12}}  |\bar d_4^\Im| + |\bar d_2^\Im| = \frac 12 \varrho b^\frac 13(1 + O(b)) \lesssim b^{-\frac 16}e^{-\frac \pi{2b}}.
\eee
Next, with $\|(\Sigma, \Theta)\|_{X_b} \lesssim b^{-\frac 16} e^{-\frac \pi{2b}}$, we have $|N_1(\Sigma + i\Theta)| \lesssim r^{-2} |\Sigma|$, $|N_2(\Sigma + i\Theta)| \lesssim r^{-2} |\Theta|$ on $I_{mid}$. Combined with $|P_{b, 1}| \lesssim r^{-2}$ on $I_{mid}$, we obtain
\bee
 |P_{b, 1} \Sigma - bs_c \Theta + N_1(\Sigma + i\Theta)|& \lesssim&  r^{-2} \omega_{b, 1}^- \|\Sigma\|_{C^0_{\omega_{b, 1}^-}(I_{mid})}   + bs_c \omega_{b, 1}^+ \|\Theta\|_{C^0_{\omega_{b, 1}^+}(I_{mid})} \\
  &\lesssim& r^{-2} \omega_{b, 1}^- \| (\Sigma, \Theta)\|_{X_b}, \\
  |P_{b, 1} \Theta + bs_c \Sigma + N_2(\Sigma + i\Theta)| &\lesssim& r^{-2} \omega_{b, 1}^+ \|\Theta\|_{C^0_{\omega_{b, 1}^+}(I_{mid})}   + bs_c \omega_{b, 1}^- \|\Sigma\|_{C^0_{\omega_{b, 1}^-}(I_{mid})} \\
  &\lesssim& r^{-2} \omega_{b, 1}^+ \| (\Sigma, \Theta)\|_{X_b}.
\eee
This implies the onto estimates applying the boundedness of the inversion operators \eqref{eqtildeTmidGest1}, \eqref{eqtildeTmidDest}, 
\bee
  |\a (\Sigma, \Theta)| &\lesssim& x_*^{-1} \| (\Sigma, \Theta)\|_{X_b}, \\
  \left\|  \tilde T^{mid, G}_{x_*, r_0;b, 1} \left( P_{b, 1} \Sigma - bs_c \Theta + N_1(\Sigma + i\Theta) \right) \right\|_{C^0_{\omega_{b, 1}^-}(I_{mid})} &\lesssim& x_*^{-1}  \| (\Sigma, \Theta)\|_{X_b},\\
\left\| \tilde T^{mid, D}_{x_*, r_0;b, 1} \left( P_{b, 1} \Theta + bs_c \Sigma + N_2(\Sigma + i\Theta) \right) \right\|_{C^0_{\omega_{b, 1}^+}(I_{mid})}& \lesssim& x_*^{-1}  \| (\Sigma, \Theta)\|_{X_b},
  \eee
   and the contraction estimate with contraction rate $O(x_*^{-1})$ follows similarly. 
   
   Now we check the boundary values \eqref{eqUbdryr*} at $r_* = b^{-\frac 12}$. We will show for any $r \in I_{mid}$, 
   \be
     \left| \begin{array}{l}
     \Re \vec U(r) = \varrho b^\frac 13 \vec \psi_4^{b, 1}(r) + \omega_{b, 1}^-(r) \cdot O_{\RR^2} (r^{-1})   \\
     \Im \vec U(r) = \frac 12 \varrho b^\frac 13 \vec \psi_2^{b, 1}(r) +  \omega_{b, 1}^+(r) \cdot O_{\RR^2} (r^{-1})  
      \end{array} \right. \label{eqUbdryr*2}
   \ee
   The case $r = r_* = b^{-\frac 14}$ implies \eqref{eqUbdryr*} using the asymptotics of $\psi_2^{b, 1}$, $\psi_4^{b, 1}$ from Proposition \ref{propWKB} (5). Indeed, from \eqref{eqUDuhamelmid1}, \eqref{eqUDuhamelmid2}, we have 
\bee
 \vec \Sigma (r)= \left( \bar d_4^{\Re} + \int_r^{r_0}   \psi_2^{b, 1} \left( P_{b, 1} \Sigma - bs_c \Theta + N_1(\Sigma + i\Theta) \right) \frac{dr}{W_{42;1}}
 \right) \vec \psi_4^{b, 1}(r) \\
 + \left( \bar d_2^\Re -   \int_r^{r_0}  \psi_4^{b, 1}   \left( P_{b, 1} \Sigma - bs_c \Theta + N_1(\Sigma + i\Theta) \right) \frac{dr}{W_{42;1}} \right) \vec \psi_2^{b, 1}(r) \\
  \vec \Theta (r) = \left( \bar d_4^{\Im} + \int_r^{r_0}   \psi_2^{b, 1} \left( P_{b, 1} \Theta + bs_c \Sigma + N_2(\Sigma + i\Theta)  \right) \frac{dr}{W_{42;1}}
 \right) \vec \psi_4^{b, 1}(r) \\
 + \left( \bar d_2^\Im -   \int_r^{r_0}  \psi_4^{b, 1}   \left( P_{b, 1} \Theta + bs_c \Sigma + N_2(\Sigma + i\Theta) \right) \frac{dr}{W_{42;1}} \right) \vec \psi_2^{b, 1}(r) 
\eee
where we exploited the cancellation between $\a(\Sigma, \Theta) \psi_4^{b, 1}$ with $\tilde \calT^{mid, G}_{x_*, r_0;b, 1}$. Therefore \eqref{eqUbdryr*2} follows from the boundedness of $\Sigma, \Theta$ and the evaluation of coefficients \eqref{eqdefbardj}. 

\mbox{}

\underline{3. Conclusion of proof.} We have constructed the family of solutions $U_\varrho$ satisfying \eqref{eqUasymp} and \eqref{eqUbdryr*}. As stated at the beginning of the proof, the matching argument from \cite[Theorem 2]{MR4250747} implies \eqref{eqQbandU}, from which we will check \eqref{eqQbasymp8}-\eqref{eqUabs2} and \eqref{eqQbsharp}. 

The asymptotics \eqref{eqQbsharp} can be read directly from the Duhamel formula on $I_{ext}$ \eqref{eqQbDuhamel1} with $\theta_b = \theta_1(\varrho_b) + \tilde \theta_b$, and the bound \eqref{eqUbdryr*2} with the perturbation $e^{i\tilde \theta_b}$ small enough to be absorbed in the residuals.

% Next, we show how \eqref{eqQbsharp} implies the lower bound \eqref{eqQbasymp21}. The case with $|r - 2b^{-1}| \gg M b^{-\frac 13}$ follows from the asymptotics of $\psi_1^{b, E}$, $\psi_4^{b, 1}$, $\psi_2^{b, 1}$ in Proposition \ref{propWKB} (5). For $| r - 2b^{-1}| \lesssim b^{-\frac 13}$, we exploit that $|e^{\frac{2\pi i}{3}}\mu^\frac 32 \zeta| \lesssim 1$, and therefore $|\psi_1^{b, 1}| = |\psi_4^{b, 1} + \frac i2\psi_2^{b, 1}| \sim 1$ from the nonvanishing property Proposition \ref{propWKB} (4), combined with and $|\psi_1^{b, E} - \psi_1^{b, 1}| \lesssim s_c$ from \eqref{eqpsibEderiv1}. 

Finally,  we prove the derivative estimates on $I_{ext}$ \eqref{eqQbasymp8}-\eqref{eqUabs2}.  Recalling Lemma \ref{leminvtildeHext}, we can rewrite \eqref{eqQbDuhamel1} as
\be U = \tilde c_3 \psi_1^{b, E} + \tilde T^{ext}_{r_0;b, E} \left( P_{b, E} U + N(U) \right) \label{eqQbDuhamel3} \ee
We define $r_1 = \frac 2b + b^{-\frac 12}$ and claim the nonlinear estimates for any $N \ge 0$.
\be 
  \| \pa_r^N (|U|^2) \|_{C^0_{(\omega_{b, E}^-)^2}([r_1, \infty))} \lesssim_N \| U \|_{X^{0,N;-}_{r_0, r_1;b, E}}^2. \label{eqnonestmodU}
\ee
Then from the derivatives estimate of $h_{b, E}$ \eqref{eqbddh}, Leibniz rule and Fa\'a di Bruno's formula \eqref{eqFaadiBruno}, this implies 
\be
  \| P_{b, E} U + N(U) \|_{X^{-2,N;-}_{r_0, r_1;b, E}} \lesssim_{N} \| U \|_{X^{0,N;-}_{r_0, r_1;b, E}} \left(1 + b^{-2} \| U \|_{X^{0,N;-}_{r_0, r_1;b, E}}^{p-1} \right),
\ee
and thereafter 
\[ \| U \|_{X^{0,N;-}_{r_0, r_1;b, E} } \lesssim_N b^{-\frac 16}e^{-\frac \pi{2b}},\quad \forall\,\, N \ge 0, \]
by induction on $N \ge 0$ using \eqref{eqbddHbnuinvext} and \eqref{eqpsibderiv1}. From the definition of the norm and $U$, this implies \eqref{eqQbasymp9} and \eqref{eqUabs2} through \eqref{eqnonestmodU}. For \eqref{eqQbasymp8}, notice that 
\[ D_{-;b, E}^n \left(e^{i\frac{br^2}{4}}Q_b \right) = e^{i\phi_{b, E}} \pa_r^n \left( e^{-i\phi_{b, E} + i\frac{br^2}{4} } Q_b \right)  \]
where $\phi_{b, E}$ is as \eqref{eqdefphi} and that for every $n \ge 1$, 
\[\left| \pa_r^n (-i\phi_{b, E} + i br^2/4)\right|  =  \frac 12 \left|\pa_r^{n-1} \left( \sqrt{b^2 r^2 - 4E} - br \right)\right| \lesssim_n (br)^{-n},\quad {\rm for}\,\, r \ge \frac 4b.\]
Then \eqref{eqQbasymp8} follows \eqref{eqQbasymp9} a simple computation by Leibniz rule and Fa\'a di Bruno's formula \eqref{eqFaadiBruno}.

Now we prove \eqref{eqnonestmodU}, which concludes the derivative estimates \eqref{eqQbasymp8}-\eqref{eqUabs2}. Indeed, the $N=0$ case is straightforward from 
 $\| U \|_{C^0_{\omega_{b, E}^-}(I_{ext})} = \| U \|_{X^{0,0,-}_{r_0,r_1;b, E}}$. So it suffices to bound $\pa_r^n |U|^2$ on $[r_1, \infty)$. Notice that $\phi_{b, 1}$ (defined as \eqref{eqdefphi}) is $\RR$-valued and that for any $n \ge 1$, 
\bee \left| \pa_r^{n} (\phi_{b, E} - \phi_{b, 1}) \right| &=& \left|\pa_r^n \left[\frac{bs_c}{\left( \frac{b^2 r^2}{4} - E \right)^{\frac 12} + \left( \frac{b^2 r^2}{4} - 1 \right)^{\frac 12}} \right]\right| \\
&\lesssim_n& s_c r^{-\frac 12} |r-2b^{-1}|^{\frac 12 - n},\quad {\rm for}\,\, r \ge r_1.  \eee
    Thus we compute
     \bee
       \pa_r^n |U|^2 &=& \pa_r^n (\overline{e^{-i\phi_{b, 1}} U}\cdot e^{-i\phi_{b, 1}} U) = \sum_{k=0}^n \binom{n}{k}\overline{\left( e^{i\phi_{b, 1}} \pa_r^k e^{-i\phi_{b, 1}} \right) U }\cdot \left( e^{i\phi_{b, 1}} \pa_r^{n-k} e^{-i\phi_{b, 1}} \right) U \\
       &=& \sum_{k=0}^n \binom{n}{k}\overline{e^{-i\triangle \phi } D_-^k (e^{i\triangle \phi } U) }\cdot e^{-i\triangle \phi } D_-^{n-k} (e^{i\triangle \phi } U).
     \eee
     where $\triangle \phi := \phi_{b, E} - \phi_{b, 1}$, and \eqref{eqnonestmodU} is an application of Leibniz rule and Fa\'a di Bruno's formula \eqref{eqFaadiBruno}.

\end{proof}

\bibliographystyle{plain}
\bibliography{Bib-1}

\end{document}